\documentclass[12pt,twoside]{amsart}
\usepackage[margin=3cm]{geometry}
\usepackage{setspace}
\usepackage[colorlinks=false]{hyperref}
\usepackage[english]{babel}
\usepackage{graphicx,titling}
\usepackage{float}
\usepackage{amsmath,amsfonts,amssymb,amsthm}
\usepackage{lipsum}
\usepackage[T1]{fontenc}
\usepackage{fourier}
\usepackage{color}
\usepackage[latin1]{inputenc}
\usepackage{esint}
\usepackage{caption}
\usepackage{ccicons}

\makeatletter
\def\blfootnote{\xdef\@thefnmark{}\@footnotetext}
\makeatother

\newcommand\ccnote{
	\blfootnote{\copyright\,\, Marco Usula}
	\blfootnote{\ccLogo\, \ccAttribution\,\, Licensed under a \href{https://creativecommons.org/licenses/by/4.0/}{Creative Commons Attribution License (CC-BY)}.}
}

\usepackage[export]{adjustbox}
\numberwithin{equation}{section}

\renewcommand{\leq}{\leqslant}

\renewcommand{\geq}{\geqslant}
\renewcommand{\mathbb}{\varmathbb}
\usepackage{fancyhdr}
\usepackage[all]{xy}
\numberwithin{equation}{section}
\numberwithin{figure}{section}
\theoremstyle{plain}
\newtheorem*{thm*}{\protect\theoremname}
\theoremstyle{plain}
\newtheorem{thm}{\protect\theoremname}[section]
\theoremstyle{definition}
\newtheorem{defn}[thm]{\protect\definitionname}
\theoremstyle{remark}
\newtheorem{rem}[thm]{\protect\remarkname}
\theoremstyle{plain}
\newtheorem{lem}[thm]{\protect\lemmaname}
\theoremstyle{plain}
\newtheorem{prop}[thm]{\protect\propositionname}
\theoremstyle{definition}
\newtheorem{problem}[thm]{\protect\problemname}
\theoremstyle{plain}
\newtheorem{cor}[thm]{\protect\corollaryname}
\theoremstyle{plain}
\newtheorem*{cor*}{\protect\corollaryname}

\newcommand{\eunion}{\mathbin{\overline{\cup}}}

\DeclareMathOperator{\spec}{spec}
\DeclareMathOperator{\Op}{\normalfont \text{Op}}
\DeclareMathOperator{\Diff}{Diff}
\DeclareMathOperator{\phg}{phg}

\DeclareMathOperator{\GL}{GL}

\DeclareMathOperator{\id}{id}

\DeclareMathOperator{\of}{\normalfont \textbf{of}}
\DeclareMathOperator{\lf}{\normalfont \textbf{lf}}
\DeclareMathOperator{\rf}{\normalfont \textbf{rf}}
\DeclareMathOperator{\ff}{\normalfont \textbf{ff}}
\DeclareMathOperator{\iif}{\normalfont \textbf{if}}
\DeclareMathOperator{\ef}{\normalfont \textbf{ef}}

\DeclareMathOperator{\po}{\normalfont{\text{P}}}
\DeclareMathOperator{\tr}{\normalfont{\text{tr}}}
\DeclareMathOperator{\inte}{\normalfont{\text{int}}}

\DeclareMathOperator{\J}{J}

\usepackage{graphicx}
\usepackage{amsmath}
\usepackage{tikz}
\usepackage{tikz-3dplot}
\usetikzlibrary{math,calc,arrows}

\makeatother

\providecommand{\corollaryname}{Corollary}
\providecommand{\definitionname}{Definition}
\providecommand{\lemmaname}{Lemma}
\providecommand{\problemname}{Problem}
\providecommand{\propositionname}{Proposition}
\providecommand{\remarkname}{Remark}
\providecommand{\theoremname}{Theorem}

\pgfmathsetmacro{\myscale}{0.6}

\address{Marco Usula, University of Oxford, Mathematical Institute}
\email{marco.usula@maths.ox.ac.uk}

\begin{document}

\thispagestyle{empty}

\begin{minipage}{0.28\textwidth}
\begin{figure}[H]
%\centering
\includegraphics[width=2.5cm,height=2.5cm,left]{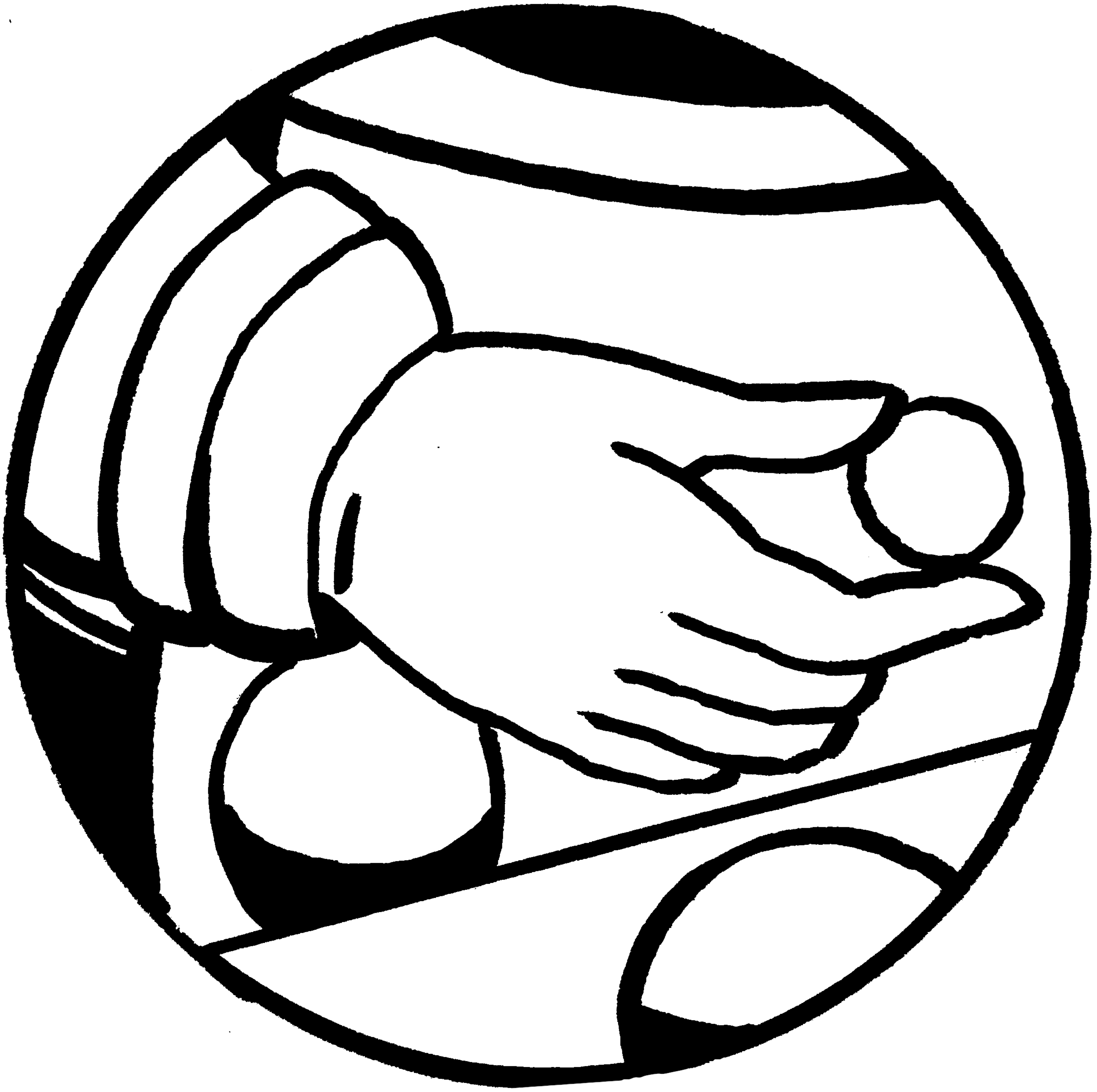}
\end{figure}
\end{minipage}
\begin{minipage}{0.7\textwidth} 
\begin{flushright}
%% The following metadata, in particular
%% the Paper No. and the DOI will be inserted by the journal
Ars Inveniendi Analytica (2026), Paper No. 4, 153 pp.
\\
DOI 10.15781/9f6mnb58
\\
ISSN: 2769-8505
\end{flushright}
\end{minipage}

\ccnote

\vspace{1cm}

%%      -------------------------------------------------------------------------------
%%      -------------------------- TITLE ----------------------------
%%      -------------------------------------------------------------------------------
%% Authors, please put here the full title of the article

\begin{center}
\begin{huge}
\textit{Boundary Value Problems for $0$-Elliptic Operators}

%\textit{some titles take two lines}

\end{huge}
\end{center}

\vspace{1cm}

%%      -------------------------------------------------------------------------------
%%      -------------------------- AUTHORS AND AFFILIATIONS ----------------------------
%%      -------------------------------------------------------------------------------
%% Authors, please put here your full names and affiliations

%\begin{minipage}[t]{.28\textwidth}
\begin{center}
{\large{\bf{Marco Usula}}} \\
\vskip0.15cm
\footnotesize{University of Oxford}
\end{center}
%\end{minipage}
%\hfill
%\noindent
%\begin{minipage}[t]{.28\textwidth}
%\begin{center}
%{\large{\bf{Joe Violet}}} \\
%\vskip0.15cm
%\footnotesize{Institute for Mathematics}
%\end{center}
%\end{minipage}
%\hfill
%\noindent
%\begin{minipage}[t]{.28\textwidth}
%\begin{center}
%{\large{\bf{Fran\c{c}ois Plum}}} \\
%\vskip0.15cm
%\footnotesize{Univerisity of Maryland} 
%\end{center}
%\end{minipage}

\vspace{1cm}

%%% Please replace "James Mustard" below 
%%% with the name of the managing editor for your submission.
%%% If you are unsure about their identity
%%% please ask an editor-in-chief about.

\begin{center}
\noindent \em{Communicated by Rafe Mazzeo}
\end{center}
\vspace{1cm}

%%      -------------------------------------------------------------------------------
%%      -------------------------- BEGIN ABSTRACT ----------------------------
%%      -------------------------------------------------------------------------------
%% Authors, please put here the ABSTRACT and KEYBOARDS

\noindent \textbf{Abstract.} \textit{Let $X$ be a manifold with boundary, and let $L$ be a $0$-elliptic
	operator on $X$ which is semi-Fredholm essentially surjective with
	infinite-dimensional kernel. Examples include Hodge Laplacians and
	Dirac operators on conformally compact manifolds. We construct left
	and right parametrices for $L$ when supplemented with appropriate
	elliptic boundary conditions. The construction relies on a new calculus
	of pseudodifferential operators on functions over both $X$ and $\partial X$,
	which we call the ``symbolic $0$-calculus''. This new calculus
	supplements the ordinary $0$-calculus of Mazzeo--Melrose, enabling
	it to handle boundary value problems. In the original $0$-calculus,
	operators are characterized as polyhomogeneous right densities on
	a blow-up of $X^{2}$. By contrast, operators in the symbolic $0$-calculus
	are characterized (locally near each point of the boundary of the
	diagonal) as quantizations of polyhomogeneous symbols on appropriate
	blown-up model spaces.}
\vskip0.3cm

\vspace{0.5cm}

%%      -------------------------------------------------------------------------------
%%      -------------------------- BEGIN ARTICLE ----------------------------
%%      -------------------------------------------------------------------------------
%% Authors, copy the body of your paper here

\tableofcontents

\section{\label{sec:Introduction}Introduction}

Let $Y$ be a closed manifold. If $L\in\Diff^{m}\left(Y\right)$ is
an elliptic operator, then the induced map $L:C^{\infty}\left(Y\right)\to C^{\infty}\left(Y\right)$
is Fredholm. This can be proved easily using pseudodifferential theory:
if $L$ is elliptic, then it has a parametrix $Q\in\Psi^{-m}\left(Y\right)$,
that is, a pseudodifferential inverse for $L$ modulo smoothing operators
in $\Psi^{-\infty}\left(Y\right)$. This implies Fredholmness, because
smoothing operators on a closed manifold are compact.

On a compact manifold with boundary $X$, elliptic operators need
not be Fredholm. For example, the Laplacian $\Delta$ relative to
a smooth metric on $X$ has infinite-dimensional kernel. To recover
Fredholmness, one needs to impose boundary conditions, such as Dirichlet
or Neumann boundary conditions for $\Delta$. In the general setting,
given an elliptic operator $L$, one aims to supplement the equation
$Lu=v$ with an equation prescribing the value of some pseudodifferential
operator on $\partial X$ acting on the ``boundary value'' of $u$.
Under an appropriate ellipticity hypothesis on the boundary condition,
the supplemented map is Fredholm.

This paper provides a similar result for a natural class of differential
operators on $X$, degenerate along the boundary $\partial X$, called
\emph{$0$-differential }operators. Denote by $\mathcal{V}_{0}\left(X\right)$
the Lie algebra of vector fields on $X$ vanishing on the boundary,
called \emph{$0$-vector fields}; a \emph{$0$-differential operator
}is a $C^{\infty}\left(X\right)$ linear combination of compositions
of elements of $\mathcal{V}_{0}\left(X\right)$. The degeneracy of
$0$-vector fields along the boundary implies that a $0$-differential
operator of order $m$ cannot be elliptic in the standard sense: the
principal symbol vanishes to $m$-th order along the boundary. Instead,
we have a natural notion of \emph{$0$-ellipticity}. It turns out
that $\mathcal{V}_{0}\left(X\right)$ is the module of sections of
a vector bundle $^{0}TX$, called the \emph{$0$-tangent bundle};
moreover, the inclusion $\mathcal{V}_{0}\left(X\right)\to\mathcal{V}\left(X\right)$
induces a bundle map $^{0}TX\to TX$ which is an isomorphism when
restricted to the interior $X^{\circ}$. If $L\in\Diff_{0}^{m}\left(X\right)$,
the principal symbol of $L$ extends from $T^{*}X^{\circ}$ to a smooth
function on $^{0}T^{*}X$, fibrewise homogeneous of degree $m$; the
operator is\emph{ }$0$-elliptic if this function is invertible away
from the zero section. For example, if $L\in\Diff^{m}\left(X\right)$
is elliptic in the standard sense, and $x$ is a boundary defining
function\footnote{A boundary defining function for $X$ is a smooth function $x:X\to[0,+\infty)$
	such that $x^{-1}\left(0\right)=\partial X$ and $dx$ is nowhere
	zero along $\partial X$.} for $X$, then $x^{m}L$ is $0$-elliptic. Other examples are the
usual geometric elliptic operators associated to \emph{conformally
	compact metrics }(cf.\ \S\ref{subsubsec:Conformally-compact-geometry})\emph{
}on $X$.

Let $L$ be a $0$-elliptic $0$-differential
operator in $\Diff_{0}^{m}\left(X\right)$. Then $L$ is bounded between appropriate weighted Sobolev
and Hölder spaces associated to $\mathcal{V}_{0}\left(X\right)$.
To discuss its Fredholm properties, for simplicity in this introduction
we avoid regularity issues and work on weighted spaces of \emph{conormal
	functions} (cf.\ \S\ref{subsec:Conormality,-polyhomogeneity,projective-tensor-products})
on $X$: given $\delta\in\mathbb{R}$, an auxiliary boundary defining
function $x$ and an auxiliary positive density $\omega$ on $X$,
we define $x^{\delta}H_{b}^{\infty}\left(X\right)$ as the space of
distributions $u$ on $X$ for which the weak derivatives $V_{1}\cdots V_{k}u$
are all in $L^{2}\left(X,x^{-2\delta-1}\omega\right)$ for every finite
set of vector fields $V_{1},...,V_{k}$ tangent to the boundary. These
Fréchet spaces should be thought of as natural weighted analogues
of the space of smooth functions on a closed manifold.

For every $\delta\in\mathbb{R}$, $L$ is bounded as a map from $x^{\delta}H_{b}^{\infty}\left(X\right)$ to $x^{\delta}H_{b}^{\infty}\left(X\right)$.
However, $0$-ellipticity alone is not sufficient to guarantee Fredholmness.
More precisely, as discussed by Mazzeo--Melrose in \cite{MazzeoMelroseResolvent}
and Mazzeo in \cite{MazzeoPhD, MazzeoEdgeI}, the Fredholmness of
$L$ on $x^{\delta}H_{b}^{\infty}\left(X\right)$ depends on the invertibility
of another family of model operators, called the \emph{Bessel family
}of $L$. This is an ``operator-valued principal symbol'' $\eta\mapsto\hat{N}_{\eta}\left(L\right)$,
parametrized by $\eta\in T^{*}\partial X$: for every $p\in\partial X$
and $\eta\in T_{p}^{*}\partial X$, $\hat{N}_{\eta}\left(L\right)$
is an operator on functions over $N_{p}^{+}\partial X$, where $N^{+}\partial X$
is the inward-pointing normal bundle of $\partial X$ in $X$. We
say that a weight $\delta$ is \emph{injective }(resp.\ \emph{surjective})
for $L$ if $\hat{N}_{\eta}\left(L\right)$ is injective (resp.\ surjective)
on $x^{\delta}H_{b}^{\infty}\left(N_{p}^{+}\partial X\right)$, for
every $p\in\partial X$ and $\eta\in T_{p}^{*}\partial X\backslash0$.
If $\delta$ is injective (resp.\ surjective) for $L$, then $L$
is semi-Fredholm essentially injective (resp.\ essentially surjective)
as a map $x^{\delta}H_{b}^{\infty}\left(X\right)\to x^{\delta}H_{b}^{\infty}\left(X\right)$.
In general, every weight sufficiently large (resp.\ small) is surjective
(resp.\ injective), but it is entirely possible that no invertible
weights exist (cf.\ \cite{UsulaNahmPoles} and \S3.8 of \cite{UsulaPhD}
for concrete examples).

Let $\delta$ be a surjective, not injective weight for $L$. Then
we can solve the equation $Lu=v$ in $x^{\delta}H_{b}^{\infty}\left(X\right)$
for essentially every $v$; however, the solution is not essentially
unique. As already hinted, this paper deals with the problem of supplementing
the equation $Lu=v$ with appropriate boundary conditions on $u$,
in order to obtain essential existence and uniqueness of the solution.
We now explain more in detail our notion of boundary condition.

It is proved in \cite{MazzeoEdgeI} that, under the hypothesis of
\emph{constancy of indicial roots}\footnote{This hypothesis is satisfied in most geometrically interesting situations,
	cf.\ Remark \ref{rem:rescaled-x^mL} and the discussion in \S\ref{subsubsec:Conformally-compact-geometry}.} (a hypothesis which we shall assume as well), a solution of $Lw=0$
admits a \emph{polyhomogeneous expansion} at $X$ (cf.\ Definition
\ref{def:polyhomogeneous-functions}). More precisely, denote by $\spec_{b}\left(L\right)$
the set of indicial roots of $L$ (cf.\ Definition \ref{def:indicial-root}),
and assume for simplicity that the set \[\left\{ \mu\in\spec_{b}\left(L\right):\Re\left(\mu\right)>\delta\right\} \]
does not contain indicial roots differing by integers. Choose also
an auxiliary vector field $V$ on $X$ transverse to $\partial X$
and inward-pointing, so that its flow starting at $\partial X$ determines
a collar $[0,\varepsilon)\times\partial X$. Calling $x$ the projection
$[0,\varepsilon)\times\partial X\to[0,\varepsilon)$, $w$ admits
an expansion as $x\to0$ of the form
\[
w\sim\sum_{\begin{smallmatrix}\mu\in\spec_{b}\left(L\right)\\
		\Re\left(\mu\right)>\delta
\end{smallmatrix}}\sum_{k=0}^{\infty}\sum_{l=0}^{M_{\mu}}w_{\mu,k,l}x^{\mu+k}\left(\log x\right)^{l}
\]
where $w_{\mu,k,l}\in C^{\infty}\left(\partial X\right)$ and $M_{\mu}+1$
is the multiplicity\emph{ }of $\mu$ (cf.\ Definition \ref{def:indicial-root}).
Solving formally the equation $Lw=0$ term by term, one realizes that
the coefficients $w_{\mu,k,l}$ are not independent of each other:
specifically, each $w_{\mu,k,l'}$ with $k>0$ can be expressed as
a differential linear combination of the ``leading coefficients''
$w_{\mu,0,l}$ with $l\in\left\{ 0,...,M_{\mu}\right\} $. Moreover,
once an injective weight $\overline{\delta}>\delta$ is chosen, the
``critical coefficients'' \[\left\{ w_{\mu,0,l}:\mu\in\spec_{b}\left(L\right),\ \Re\left(\mu\right)\in\left(\delta,\overline{\delta}\right),\ l\leq M_{\mu}\right\} \]
determine $w$ up to finite-dimensional subspaces.

It turns out that the collection of critical coefficients above can
be neatly interpreted as a section of a natural smooth vector bundle
$\boldsymbol{E}_{L}\to\partial X$ (cf.\ \S\ref{subsubsec:Log-homogeneous-functions})
dependent only on $\delta,\overline{\delta}$ and $\spec_{b}\left(L\right)$.
The \emph{trace map} for $L$ associated to the weights $\delta,\overline{\delta}$
(and the auxiliary choices of the boundary collar and a positive density
on $X$), is the operator
\[
\boldsymbol{A}_{L}:x^{\delta}H_{b}^{\infty}\left(X\right)\to C^{\infty}\left(\partial X;\boldsymbol{E}_{L}\right)
\]
mapping $u\in x^{\delta}H_{b}^{\infty}\left(X\right)$ to the critical
coefficients in the expansion of $w=P_{1}u$, where $P_{1}$ is the
orthogonal projector (with respect to the singular density $x^{-2\delta-1}\omega$)
into the $x^{\delta}H_{b}^{\infty}$ kernel of $L$. A \emph{boundary
	condition} for $L$ (relative to $\delta,\overline{\delta}$) is then,
loosely speaking, a pseudodifferential operator $\boldsymbol{Q}\in\Psi^{\bullet}\left(\partial X;\boldsymbol{E}_{L},\boldsymbol{W}\right)$
where $\boldsymbol{W}$ is another vector bundle.

Given a boundary condition $\boldsymbol{Q}\in\Psi^{\bullet}\left(\partial X;\boldsymbol{E}_{L},\boldsymbol{W}\right)$,
consider the supplemented operator
\[
\mathcal{L}:=\left(\begin{matrix}L\\
	\boldsymbol{Q}\boldsymbol{A}_{L}
\end{matrix}\right):x^{\delta}H_{b}^{\infty}\left(X\right)\to x^{\delta}H_{b}^{\infty}\left(X\right)\oplus C^{\infty}\left(\partial X;\boldsymbol{W}\right).
\]
The notion of \emph{ellipticity} of $\boldsymbol{Q}$ is due to Mazzeo--Vertman
\cite{MazzeoEdgeII}; as in \cite{MazzeoEdgeI}, this condition is
essentially a notion of invertibility of an operator-valued principal
symbol for $\mathcal{L}$. Indeed, $\mathcal{L}$ comes equipped with
a ``supplemented Bessel family'' $\eta\mapsto\hat{N}_{\eta}\left(\mathcal{L}\right)$,
dependent on $\eta\in T^{*}\partial X\backslash0$, taking the form
\[
\hat{N}_{\eta}\left(\mathcal{L}\right)=\left(\begin{matrix}\hat{N}_{\eta}\left(L\right)\\
	\sigma_{\eta}\left(\boldsymbol{Q}\right)\hat{N}_{\eta}\left(\boldsymbol{A}_{L}\right)
\end{matrix}\right).
\]
Here, calling $p=\pi\left(\eta\right)\in\partial X$:
\begin{enumerate}
	\item $\hat{N}_{\eta}\left(L\right)$ is the Bessel operator of $L$ at
	$\eta$;
	\item $\sigma_{\eta}\left(\boldsymbol{Q}\right)$ is the principal symbol
	of $\boldsymbol{Q}$;
	\item $\hat{N}_{\eta}\left(\boldsymbol{A}_{L}\right)$ is the \emph{Bessel
		trace map}, sending a function $u\in x^{\delta}L_{b}^{2}\left(N_{p}^{+}\partial X\right)$
	to the trace of $\hat{N}_{\eta}\left(P_{1}\right)u$, the projection
	of $u$ onto the kernel of $\hat{N}_{\eta}\left(L\right)$.
\end{enumerate}
The boundary condition $\boldsymbol{Q}$ is elliptic for $L$ if $\hat{N}_{\eta}\left(\mathcal{L}\right)$
is invertible as a map $x^{\delta}H_{b}^{\infty}\left(N_{p}^{+}\partial X\right)\to x^{\delta}H_{b}^{\infty}\left(N_{p}^{+}\partial X\right)\oplus\boldsymbol{W}_{p}$
for every $\eta\in T_{p}^{*}\partial X\backslash0$ and every $p\in\partial X$.

One of the main results of this paper is the following Fredholm theorem:
\begin{thm*}
	(Cf.\ Theorem \ref{thm:main-theorem-fredholm} for a more precise
	formulation) Let $L\in\Diff_{0}^{m}\left(X\right)$ be a $0$-elliptic
	operator with constant indicial roots; let $\delta$ be a surjective,
	not injective weight for $L$, and let $\overline{\delta}$ be an
	injective weight. If $\boldsymbol{Q}\in\Psi^{\bullet}\left(\partial X;\boldsymbol{E}_{L},\boldsymbol{W}\right)$
	is an elliptic boundary condition for $L$ relative to $\delta,\overline{\delta}$,
	then the map
	\[
	\mathcal{L}:=\left(\begin{matrix}L\\
		\boldsymbol{Q}\boldsymbol{A}_{L}
	\end{matrix}\right):x^{\delta}H_{b}^{\infty}\left(X\right)\to x^{\delta}H_{b}^{\infty}\left(X\right)\oplus C^{\infty}\left(\partial X;\boldsymbol{W}\right)
	\]
	is Fredholm.
\end{thm*}
This Fredholm theorem can be applied to prove the well-posedness of
several natural geometric boundary value problems. For example, we
can apply it to the usual Dirichlet, Neumann, Robin boundary value
problems for a genuinely elliptic operator $L\in\Diff^{m}\left(X\right)$,
by passing to the rescaled $0$-elliptic operator $x^{m}L$. Other
examples include boundary value problems for geometric elliptic operators
in conformally compact geometry. To exemplify the theory, in this
paper we shall discuss an example in detail: the Dirichlet boundary
value problem for the Laplacian on an asymptotically hyperbolic manifold
(cf.\ \S\ref{sec:Recap-on-the-0-calc}, \S\ref{sec:Boundary-value-problems},
\S\ref{sec:Parametrix-construction} and \S\ref{sec:sharp-regularity-of-traces}).
A sequel to this paper, currently in preparation,
will be more geared towards examples and applications.

Besides the direct applications of the Fredholm theorem above, the
technology developed to prove this theorem can be useful in other
contexts. As an example, in \S7 we shall use the mapping properties
of the trace map $\boldsymbol{A}_{L}$ to prove a sharp regularity
result for the coefficients $w_{\mu,k,l}$ in the expansion of a \emph{weak
}solution $w\in x^{\delta}L_{b}^{2}\left(X\right)$ of $Lw=0$. In
this case, $w$ is only polyhomogeneous in a distributional sense
(cf.\ \S\ref{subsec:Weak-polyhomogeneous-expansions}), and the
coefficients $w_{\mu,k,l}$ live in appropriate Sobolev spaces with
monotonically decreasing regularity as $k$ or $l$ increase. These
regularity results (Theorem \ref{thm:sharp-regularity-of-traces}
and Corollary \ref{cor:Log-loss-of-regularity}) complete and correct
the analysis of this problem, initiated by Mazzeo in \S7 of \cite{MazzeoEdgeI}.

Let us outline the techniques developed in this paper. As in the closed
and the fully $0$-elliptic ($0$-ellipticity + invertible Bessel
family) case, the proof of the Fredholm theorem consists in constructing
a \emph{parametrix }for $\mathcal{L}$, i.e.\ an inverse for $\mathcal{L}$
modulo \emph{compact }operators. The hard work consists in developing
a pseudodifferential calculus well-suited to contain such a parametrix.
Mazzeo--Melrose \cite{MazzeoMelroseResolvent, MazzeoPhD} and Mazzeo
\cite{MazzeoEdgeI} constructed a parametrix for fully $0$-elliptic
operators in the \emph{large $0$-calculus} $\Psi_{0}^{m,\mathcal{E}}\left(X\right)$;
elements of this calculus are characterized geometrically as distributional
right densities on $X^{2}$, lifting to \emph{polyhomogeneous }right
densities on the blow-up of $X^{2}$ at $\partial\Delta$ with a pseudodifferential
singularity along the lifted diagonal $\Delta_{0}$. We shall recall
the main aspects of this calculus in \S\ref{sec:Recap-on-the-0-calc}.
Three crucial aspects of the parametrix construction are:
\begin{enumerate}
	\item \emph{inversion of the Bessel family}: given a $0$-elliptic operator
	$L\in\Diff_{0}^{m}\left(X\right)$, with constant indicial roots and
	invertible Bessel family $\hat{N}\left(L\right)$, one needs to be
	able to find an operator in $\Psi_{0}^{-m,\bullet}\left(X\right)$
	whose Bessel family coincides with $\hat{N}\left(L\right)^{-1}$;
	\item \emph{composition theorem}: given two operators $A\in\Psi_{0}^{m,\mathcal{E}}\left(X\right)$
	and $B\in\Psi_{0}^{m',\mathcal{F}}\left(X\right)$, one needs to know
	when the composition $AB$ is well-defined, and on which instance
	of $\Psi_{0}^{\bullet,\bullet}\left(X\right)$ it lives;
	\item \emph{mapping properties}: given an operator $A\in\Psi_{0}^{m,\mathcal{E}}\left(X\right)$,
	one needs to know whether $A$ is bounded on appropriate weighted
	Sobolev and Hölder spaces associated to $\mathcal{V}_{0}\left(X\right)$,
	and whether it is compact.
\end{enumerate}
The first two steps are quite challenging. The second step uses a
very elegant and powerful technique due to Richard Melrose, which
we refer to as the \emph{pull-back / push-forward technique }(cf.\ \S\ref{subsec:The-pull-back-push-forward-technique}).

To address the semi-Fredholm essentially surjective case, we need
to introduce new classes of operators bridging between the manifolds $X$ and $\partial X$.
We denote by $\Op\left(X,\partial X\right)$ (resp.\ $\Op\left(\partial X,X\right)$)
the class of operators of \emph{trace type} $\dot{C}^{\infty}\left(X\right)\to C^{-\infty}\left(\partial X\right)$
(resp.\ \emph{Poisson type} $C^{\infty}\left(\partial X\right)\to C^{-\infty}\left(X\right)$).
For example, the trace map $\boldsymbol{A}_{L}$ discussed above is
an operator of trace type (with values in sections of $\boldsymbol{E}_{L}\to\partial X$).

From an abstract standpoint, we expect a parametrix for the supplemented operator
$\mathcal{L}=\left(\begin{matrix}L & \boldsymbol{Q}\boldsymbol{A}_{L}\end{matrix}\right)^{T}$
 to take the form $\mathcal{G}=\left(\begin{matrix}G & \boldsymbol{B}\end{matrix}\right)$,
where $G$ is in the large $0$-calculus, and $\boldsymbol{B}$ is
an operator of Poisson type acting on sections of $\boldsymbol{W}\to\partial X$.
Proving that $\mathcal{G}$ is a parametrix for $\mathcal{L}$ then
amounts to proving that the remainders $R_{1}$ and $\mathcal{R}_{2}$,
defined by the formulas
\begin{align*}
	\mathcal{G}\mathcal{L} & =GL+\boldsymbol{B}\boldsymbol{Q}\boldsymbol{A}_{L}=\id_{X}+R_{1}\\
	\mathcal{L}\mathcal{G} & =\left(\begin{matrix}LG & L\boldsymbol{B}\\
		\boldsymbol{Q}\boldsymbol{A}_{L}G & \boldsymbol{Q}\boldsymbol{A}_{L}\boldsymbol{B}
	\end{matrix}\right)=\left(\begin{matrix}\id_{X}\\
		& \id_{\boldsymbol{W}}
	\end{matrix}\right)+\mathcal{R}_{2},
\end{align*}
are \emph{compact }between the relevant spaces. Following the general
philosophy of \cite{MelroseMendozaTotallyCharacteristic, MazzeoEdgeI}
and many other works stemming from the Melrose School, an effective
solution to this problem requires:
\begin{enumerate}
	\item constructing two \emph{sub-classes} of operators of trace
	and Poisson type, equipped with \emph{Bessel} \emph{family maps} analogous
	to the homonymous map in the $0$-calculus; these classes should be
	modeled in order to contain $\boldsymbol{A}_{L}$ and $\boldsymbol{B}$;
	\item constructing an appropriate \emph{boundary calculus} of pseudodifferential
	operators on the boundary, modeled in order to contain the boundary
	conditions;
	\item proving \emph{composition theorems}, relating these classes to each
	other and to the $0$-calculus;
	\item proving \emph{mapping properties} on appropriate Sobolev / Hölder
	spaces.
\end{enumerate}
The paper \cite{MazzeoEdgeII} by Mazzeo--Vertman, which is
the main inspiration for this work, attempts an approach following
closely the ideas of \cite{MazzeoEdgeI}. More precisely, Mazzeo and Vertman
introduce classes of\emph{ $0$-trace }(resp.\ $0$\emph{-Poisson})\emph{
}operators, trace type (resp.\ Poisson type) operators with Schwartz
kernels lifting to polyhomogeneous right densities on the blow-ups
$\left[\partial X\times X;\partial\Delta\right]$ (resp.\ $\left[X\times\partial X;\partial\Delta\right]$);
the precise definitions are recalled in \S\ref{subsec:-Poisson-and--trace}.
We will refer to this as the ``Schwartz kernel strategy''. Unfortunately,
this approach does not seem sufficient to solve the whole problem.
More precisely, various composition and mapping properties required
are not treated in \cite{MazzeoEdgeII}, and some results need corrections.
In \S2 of his PhD thesis \cite{UsulaPhD}, the author of the present
paper implemented the missing pieces and required corrections to \cite{MazzeoEdgeII},
following the Schwartz kernel strategy. This however led to a still
incomplete solution, requiring many ad-hoc arguments, and valid only
in a very special case: the one where the critical strip $\left\{ \mu\in\spec_{b}\left(L\right):\Re\left(\mu\right)\in\left(\delta,\overline{\delta}\right)\right\} $
contains only one, simple indicial root.

In the author's opinion, the Schwartz kernel strategy is not well-suited
to fully solve the general case. When adopting this approach, various
technical difficulties arise (for example, it is quite tricky to understand
how $0$-Poisson and $0$-trace operators compose with pseudodifferential
operators on the boundary, due to the different nature of their singularities).
However, technical issues aside, there is a much more fundamental
reason pointing beyond the Schwartz kernel strategy, which we now
outline.

As we already noted, a crucial step of the parametrix construction
requires inverting explicitly the Bessel map
\[
\hat{N}_{\eta}\left(\mathcal{L}\right)=\left(\begin{matrix}\hat{N}_{\eta}\left(L\right)\\
	\sigma_{\eta}\left(\boldsymbol{Q}\right)\hat{N}_{\eta}\left(\boldsymbol{A}_{L}\right)
\end{matrix}\right):x^{\delta}H_{b}^{\infty}\left(N_{p}^{+}\partial X\right)\to x^{\delta}H_{b}^{\infty}\left(N_{p}^{+}\partial X\right)\oplus\boldsymbol{W}_{p}
\]
for every $p\in\partial X$ and every $\eta\in T_{p}^{*}\partial X\backslash0$.
The maps $\hat{N}_{\eta}\left(L\right)$ and $\hat{N}_{\eta}\left(\boldsymbol{A}_{L}\right)$
can be explicitly computed, and the analysis of their behavior under
a positive rescaling of $\eta$ is fundamental. Calling $\lambda_{t}$
the ``dilation by $t\in\mathbb{R}^{+}$'' operator on $N_{p}^{+}\partial X$,
the map $\hat{N}_{\eta}\left(L\right)$ satisfies the simple homogeneity
condition
\[
\hat{N}_{t\eta}\left(L\right)=\lambda_{t}^{*}\circ\hat{N}_{\eta}\left(L\right)\circ\lambda_{t^{-1}}^{*}.
\]
However, the Bessel trace map $\hat{N}_{\eta}\left(\boldsymbol{A}_{L}\right)$
satisfies a \emph{twisted }homogeneity condition: there is a natural
endomorphism $\boldsymbol{\mathfrak{s}}_{L}:\boldsymbol{E}_{L}\to\boldsymbol{E}_{L}$,
whose eigenvalues are the indicial roots of $L$ with real part in
$\left(\delta,\overline{\delta}\right)$, such that
\[
\hat{N}_{t\eta}\left(\boldsymbol{A}_{L}\right)=t^{\boldsymbol{\mathfrak{s}}_{L}}\hat{N}_{\eta}\left(\boldsymbol{A}_{L}\right)\circ\lambda_{t^{-1}}^{*}.
\]
In order for the remainder $R_{1}=\mathcal{G}\mathcal{L}-\id$ to
be compact, we expect $\boldsymbol{B}\boldsymbol{Q}\boldsymbol{A}_{L}$
to be an element of the large $0$-calculus satisfying $\hat{N}_{\eta}\left(\boldsymbol{B}\boldsymbol{Q}\boldsymbol{A}_{L}\right)=\hat{N}_{\eta}\left(P_{1}\right)$,
where $P_{1}=I-GL$; similarly, in order for the remainder $\mathcal{R}_{2}=\mathcal{L}\mathcal{G}-\id$
to be compact, we expect $L\boldsymbol{B}$ (resp.\ $\boldsymbol{Q}\boldsymbol{A}_{L}G$)
to be a compact Poisson type (resp.\ trace type) operator, while
$\boldsymbol{Q}\boldsymbol{A}_{L}\boldsymbol{B}$ is expected to be
a pseudodifferential operator on the boundary satisfying $\sigma_{\eta}\left(\boldsymbol{Q}\boldsymbol{A}_{L}\boldsymbol{B}\right)=\id_{\boldsymbol{W}_{p}}$.
This suggests that $\sigma_{\eta}\left(\boldsymbol{Q}\right)$ and
$\hat{N}_{\eta}\left(\boldsymbol{B}\right)$ should satisfy compatible
twisted homogeneity properties, of the form
\begin{align*}
	\sigma_{t\eta}\left(\boldsymbol{Q}\right) & =t^{\boldsymbol{\mathfrak{m}}}\sigma_{\eta}\left(\boldsymbol{Q}\right)t^{-\boldsymbol{\mathfrak{s}}_{L}}\\
	\hat{N}_{t\eta}\left(\boldsymbol{B}\right) & =\lambda_{t}^{*}\circ\hat{N}_{\eta}\left(\boldsymbol{B}\right)t^{-\boldsymbol{\mathfrak{m}}}
\end{align*}
for some endomorphism $\boldsymbol{\mathfrak{m}}:\boldsymbol{W}\to\boldsymbol{W}$.
This observation points towards a full calculus of operators of trace,
Poisson and boundary type, whose model families satisfy \emph{twisted}
homogeneity conditions. In the Schwartz kernel strategy, these conditions
do not emerge naturally\footnote{One can avoid addressing this problem when the critical strip contains
	a single, simple indicial root $\mu$: in this case, $\boldsymbol{E}_{L}$
	is a trivial line bundle, and $\boldsymbol{\mathfrak{s}}_{L}$ is
	simply multiplication by $\mu$. This is indeed the only case where
	the theory developed in \cite{MazzeoEdgeII, UsulaPhD} is sufficient.}.

To address the issues raised above, in this paper we depart from the
Schwartz kernel strategy. The main change of perspective is that,
instead of characterizing (as in \cite{MazzeoEdgeI, MazzeoEdgeII})
our operators as polyhomogeneous right densities on appropriate blow-ups
of the double spaces $X^{2}$, $\partial X\times X$ and $X\times\partial X$,
we characterize them (locally near any point of the boundary of the
diagonal $\partial\Delta$) as \emph{quantizations of polyhomogeneous
	symbols} on appropriate blown-up spaces.

The fundamental shift in our symbolic approach is moving the analysis
of the singularity from the ``physical'' model space to the dual
``frequency'' space. We exemplify the idea by discussing the trace
case briefly. Following the Schwartz kernel approach, a $0$-trace
operator $A$ behaves as follows: given $p\in\partial X$ and coordinates
$\left(x,y\right)$ for $X$ centered at $p$, the Schwartz kernel
of $A$ coincides near $\left(p,p\right)$ with a distributional right
density, written in the induced coordinates $\left(y,\tilde{x},\tilde{y}\right)$
for $\partial X\times X$ as $A\left(y;\tilde{x},y-\tilde{y}\right)d\tilde{x}d\tilde{y}$;
the distribution $A\left(y;\tilde{x},Y\right)$, defined on the physical
model space $\mathbb{R}^{n}\times[0,+\infty)\times\mathbb{R}^{n}$
with coordinates $\left(y,\tilde{x},Y\right)$, lifts to a polyhomogeneous
function on the blow-up of the model space at $\left\{ \tilde{x}=0,Y=0\right\} $.

By contrast, in our approach, a \emph{symbolic }$0$-trace operator
$A$ is a trace-type operator $A$ whose Schwartz kernel coincides
near $\left(p,p\right)$ with a distributional right density of the
form
\[
A\left(y;\tilde{x},y-\tilde{y}\right)d\tilde{x}d\tilde{y}=\left[\frac{1}{\left(2\pi\right)^{n}}\int e^{i\left(y-\tilde{y}\right)\eta}a\left(y;\tilde{x},\eta\right)d\eta\right]d\tilde{x}d\tilde{y};
\]
here it is the \emph{symbol} $a\left(y;\tilde{x},\eta\right)$ which
is required to lift to a\emph{ }polyhomogeneous distribution on a
blow-up of the dual frequency model space $\mathbb{R}^{n}\times[0,+\infty)\times\mathbb{R}^{n}$
with coordinates $\left(y,\tilde{x},\eta\right)$. More precisely,
since the singularity of $A\left(y;\tilde{x},Y\right)$ in the Schwartz
kernel approach arises at the locus $\left\{ \tilde{x}=0,Y=0\right\} $,
the formula above suggests that the singularity of $a\left(y;\tilde{x},\eta\right)$
should arise at the ``dual'' locus $\left\{ \tilde{x}=0,\left|\eta\right|=\infty\right\} $
in the frequency model space. This is exactly how we define our class
of \emph{$0$-trace} symbols (cf.\ \S\ref{subsubsec:Symbolic-0-trace-operators}).
We can argue similarly for Poisson type (cf.\ \S\ref{subsubsec:Symbolic-0-Poisson-operators})
and interior type (cf.\ \S\ref{subsubsec:0-interior-and-0b-interior})
operators.

An important advantage of the symbolic approach over the Schwartz
kernel approach, is a better control over the Bessel family map. For
example, in the $0$-trace case outlined above, the Bessel family
of $A$ at $p$ is simply \emph{the leading order homogeneous term
	in the expansion of the $0$-trace symbol $a\left(y;\tilde{x},\eta\right)$,
	at the front face of the blow-up of the model frequency space}. This
definition makes explicit the nature of the Bessel family map as a
``principal symbol'' map. In particular, in the symbolic approach,
the Bessel family map is \emph{trivially surjective}; on the other
hand, understanding the range of the Bessel map in the Schwartz kernel
approach is a very delicate task (\cite{Hintz0calculus} explains
this aspect very well). Moreover, in our approach, it is easy to adapt
the definition to allow for operators with \emph{twisted }homogeneous
Bessel maps: for example, in the trace case, it is sufficient to pre-multiply
our (vector-valued) $0$-trace symbol $a\left(y;\tilde{x},\eta\right)$
by a factor $\left\langle \eta\right\rangle ^{\mathfrak{a}\left(y\right)}$,
where $\left\langle \eta\right\rangle =\sqrt{1+\left|\eta\right|^{2}}$
and $\mathfrak{a}\left(y\right)$ is matrix-valued. We refer to \S\ref{subsec:Definitions-and-relation-with-0-calculus}
for more details.

We call the calculus developed in this paper the \emph{symbolic $0$-calculus}.
It consists of several pieces:
\begin{enumerate}
	\item classes $\hat{\Psi}_{0\tr}^{-\infty,\mathcal{E}}\left(X,\partial X\right)$
	and $\hat{\Psi}_{0\po}^{-\infty,\mathcal{E}}\left(\partial X,X\right)$
	of \emph{symbolic $0$-trace }and\emph{ $0$-Poisson operators} (cf.\ \S\ref{subsubsec:Symbolic-0-trace-operators}
	and \S\ref{subsubsec:Symbolic-0-Poisson-operators}), filtered by
	index families indicating the asymptotic behavior of the \emph{symbols}
	representing these operators;
	\item variants $\hat{\Psi}_{0\tr}^{-\infty,\mathcal{E},\boldsymbol{\mathfrak{s}}}\left(X;\partial X,\boldsymbol{E}\right)$
	and $\hat{\Psi}_{0\po}^{-\infty,\mathcal{E},\boldsymbol{\mathfrak{s}}}\left(\partial X,\boldsymbol{E};X\right)$
	of the classes above, which we call (the classes of) \emph{twisted}\textbf{\emph{
	}}\emph{symbolic $0$-trace and $0$-Poisson operators} (cf.\ \S\ref{subsec:Twisted--trace-and-poisson}):
	here $\boldsymbol{E}\to\partial X$ is a vector bundle on $\partial X$,
	and $\boldsymbol{\mathfrak{s}}:\boldsymbol{E}\to\boldsymbol{E}$ is
	a smooth endomorphism with constant eigenvalues;
	\item two classes $\hat{\Psi}_{0}^{-\infty,\mathcal{E}}\left(X\right)$
	and $\hat{\Psi}_{0b}^{-\infty,\mathcal{E}}\left(X\right)$ (note the
	hats!) of \emph{$0$-interior }and \emph{$0b$-interior }operators,
	acting on functions over $X$ (cf.\ \S\ref{subsubsec:0-interior-and-0b-interior}),
	again filtered by index families indicating the asymptotic behavior
	of the symbols representing these operators; these classes are closely
	related to the $0$-calculus and the extended $0$-calculus of Lauter
	\cite{Lauter} (cf.\ Corollary \ref{cor:global-relation-symbolic0b-physical0b});
	\item a calculus $\Psi_{\phg}^{\bullet,\left(\boldsymbol{\mathfrak{s}},\boldsymbol{\mathfrak{t}}\right)}\left(\partial X;\boldsymbol{E},\boldsymbol{F}\right)$
	of \emph{twisted pseudodifferential operators} on $\partial X$, acting
	between sections of vector bundles $\boldsymbol{E},\boldsymbol{F}\to\partial X$
	equipped with smooth endomorphisms $\boldsymbol{\mathfrak{s}},\boldsymbol{\mathfrak{t}}$
	with constant eigenvalues.
\end{enumerate}
As hinted above, the classes $\hat{\Psi}_{0\tr}^{-\infty,\mathcal{E},\boldsymbol{\mathfrak{s}}}\left(X;\partial X,\boldsymbol{E}\right)$,
$\hat{\Psi}_{0\po}^{-\infty,\mathcal{E},\boldsymbol{\mathfrak{s}}}\left(\partial X,\boldsymbol{E};X\right)$
(resp.\ $\Psi_{\phg}^{\bullet,\left(\boldsymbol{\mathfrak{s}},\boldsymbol{\mathfrak{t}}\right)}\left(\partial X;\boldsymbol{E},\boldsymbol{F}\right)$)
are designed so that the Bessel family (resp.\ principal symbol)
of their elements satisfies appropriate twisted homogeneity conditions.
The boundary calculus $\Psi_{\phg}^{\bullet,\left(\boldsymbol{\mathfrak{s}},\boldsymbol{\mathfrak{t}}\right)}\left(\partial X;\boldsymbol{E},\boldsymbol{F}\right)$
developed here is a particularly simple sub-calculus of the one developed
in \cite{KrainerMendozaVariableOrders} by Krainer--Mendoza, with
the aim of provide elliptic boundary conditions for elliptic \emph{edge}
operators with \emph{variable} indicial roots. We restrict to the
case of constant indicial roots, and this allows us to work with twisting
endomorphisms with constant eigenvalues; this simplifies some technical
aspects of the theory. While we expect the theory developed in this
paper to generalize with relative ease to the edge case with constant
indicial roots, a full generalization to variable indicial roots probably
requires substantial work along the lines of \cite{KrainerMendozaVariableOrders}.

This paper is one of many works aiming to discuss elliptic boundary
value problems on singular spaces. Besides the already mentioned \cite{MazzeoEdgeII, UsulaPhD}
and \cite{KrainerMendozaVariableOrders}, we mention in particular\cite{Schulze1991PseudoDifferentialOO, schultze1994pseudo, schulze2002operators, krainer2016boundary, KrainerMendozaSurvey}
.

The paper is organized as follows. \S\ref{sec:Recap-on-the-0-calc}
is a recap on the $0$-calculus and the Schwartz kernel strategy in
general. In \S\ref{sec:Boundary-value-problems}, we explain more
in detail our formulation of boundary value problems, and we discuss
the model problem. \S\ref{sec:The-symbolic-0-calculus} and \S\ref{sec:Adjoints-compositions-mapping}
form the technical core of the paper: in \S\ref{sec:The-symbolic-0-calculus},
we develop the symbolic $0$-calculus, and we prove its relation with
the $0$-calculus and the extended $0$-calculus, while in \S\ref{sec:Adjoints-compositions-mapping}
we prove composition and mapping properties. In \S\ref{sec:Parametrix-construction},
we construct the parametrix and we prove the Fredholm theorem and
some corollaries. Finally, in \S\ref{sec:sharp-regularity-of-traces}
we use the theory to prove our sharp regularity result for weak expansions
of solutions of homogeneous $0$-elliptic equations.

\subsection*{Acknowledgements}

The author is grateful to Rafe Mazzeo and Joel Fine for their valuable
discussions and insights that contributed to this work. A significant
portion of this research was conducted during the author\textquoteright s
visit to Stanford University, and he wishes to extend his sincere gratitude to the Department of Mathematics at Stanford for its generous
support and hospitality. The author is also grateful to the anonymous referee
for their very helpful suggestions.

This work was supported at different stages by the EoS grant 40007524 ``Beyond Symplectic
Geometry'', and by the Wiener--Anspach Foundation.

\subsection*{Conventions}

Throughout the paper, the symbol $\mathbb{N}$ (resp.\ $\mathbb{N}^{+}$)
denotes the set of non-negative (resp.\ positive) integers.

\section{\label{sec:Recap-on-the-0-calc}A primer on \texorpdfstring{$0$}{0}-pseudodifferential theory}

In this section, we recall the main objects in the theory of $0$-pseudodifferential
operators. We will recall:
\begin{enumerate}
	\item the \emph{$0$-calculus} introduced by Mazzeo--Melrose in \cite{MazzeoMelroseResolvent, MazzeoPhD, MazzeoEdgeI},
	and the \emph{extended $0$-calculus }introduced by Lauter in\emph{
	}\cite{Lauter} (cf.\ also \cite{Hintz0calculus});
	\item the classes of \emph{$0$-trace} and \emph{$0$-Poisson} operators
	introduced by Mazzeo--Vertman in \cite{MazzeoEdgeII} (cf.\ also
	\cite{UsulaPhD}).
\end{enumerate}
We will also recall the\emph{ $b$-calculus }introduced by Melrose
in\emph{ }\cite{MelroseAPS, MelroseMendozaTotallyCharacteristic},
although this class plays a minor role in what follows. We will freely
make use of the language of compact manifolds with corners and $b$-maps,
due to Melrose and presented in \cite{MelroseCorners}. Some concepts
(such as boundary conormality, polyhomogeneity and the pull-back /
push-forward Theorems) which are particularly important for this work
will be recalled in the Appendix.

\subsection{\label{subsec:-calculus,--calculus,-extended}\texorpdfstring{$0$}{0}-calculus, \texorpdfstring{$b$}{b}-calculus, extended \texorpdfstring{$0$}{0}-calculus}

Fix a compact manifold with boundary $X^{n+1}$. Denote by $\dot{C}^{\infty}\left(X\right)$
the Fréchet space of smooth functions on $X$ vanishing to infinite
order along the boundary. Call $\mathcal{D}_{X}^{1}$ the bundle of
$1$-densities on $X$, and denote by $C^{-\infty}\left(X\right)$
the topological dual of $\dot{C}^{\infty}\left(X;\mathcal{D}_{X}^{1}\right)$
equipped with the strong topology. We denote by $\Op\left(X\right)$
the class of continuous linear operators $\dot{C}^{\infty}\left(X\right)\to C^{-\infty}\left(X\right)$.
Its elements will be called \emph{operators of interior type} (as
opposed to operators of \emph{trace} / \emph{Poisson} / \emph{boundary}
type, which we shall define later). By the Schwartz Kernel Theorem,
we can identify these operators with their Schwartz kernels, which
are distributional right densities on $X^{2}$ i.e.\ elements $C^{-\infty}\left(X^{2};\pi_{R}^{*}\mathcal{D}_{X}^{1}\right)$;
here $\pi_{L},\pi_{R}:X^{2}\to X$ are the left and right projections.
The \emph{$0$-calculus, $b$-calculus }and \emph{extended $0$-calculus}
are all subclasses of $\Op\left(X\right)$ characterized in terms
of the lifts of their Schwartz kernels to appropriate blow-ups of
$X^{2}$.

Let's start with the $0$-calculus. The double space $X^{2}$ is a
compact manifold with corners, with two collective boundary hyperfaces,
namely the \emph{left face }$\lf=\partial X\times X$ and the \emph{right
	face} $\rf=X\times\partial X$. We call $\Delta$ the diagonal of
$X^{2}$. The \emph{$0$-stretched double space }is the blow-up
\[
X_{0}^{2}=\left[X^{2};\partial\Delta\right].
\]
This space has three boundary hyperfaces, which we call $\lf,\rf,\ff_{0}$:
the left and right faces $\lf,\rf$ are simply the lifts of the left
and right faces of $X^{2}$, while the $0$-front face $\ff_{0}$
is the new face obtained from the blow-up. We will use the symbols
$r_{\lf},r_{\rf},r_{\ff_{0}}$ to denote boundary defining functions
for the corresponding faces of $X_{0}^{2}$ (recall that a \emph{boundary
	defining function }for a boundary hyperface $H$ of a compact manifold
with corners $Y$ is a smooth non-negative function $r_{H}$ on $Y$
such that $r_{H}^{-1}\left(0\right)=H$ and $dr_{H}$ is non-vanishing
along $H$). As explained in detail in \cite{MazzeoEdgeI, MazzeoPhD},
the $0$-front face $\ff_{0}$ is the total space of a smooth fiber
bundle $\ff_{0}\to\partial\Delta\equiv\partial X$, with typical fiber
the closed quarter $\left(n+1\right)$-dimensional sphere 
\[
S_{2}^{n+1}=\left\{ \left(\theta_{0},\theta,\theta_{n+1}\right)\in\mathbb{R}^{n+2}:\theta_{0}^{2}+\left|\theta\right|^{2}+\theta_{n+1}^{2}=1,\theta_{0}\geq0,\theta_{n+1}\geq0\right\} .
\]
We denote by $\beta_{0}:X_{0}^{2}\to X^{2}$ the blow-down map, and
by $\beta_{0,L},\beta_{0,R}:X_{0}^{2}\to X$ the lifts of the left
and right projections $\pi_{L},\pi_{R}:X^{2}\to X$. Another important
submanifold of $X_{0}^{2}$ is the \emph{lifted diagonal} $\Delta_{0}$,
defined as the topological closure of $\beta_{0}^{-1}\left(\Delta\backslash\partial\Delta\right)$
in $X_{0}^{2}$. Unlike $\Delta$ in $X^{2}$, $\Delta_{0}$ is an
\emph{interior $p$-submanifold} of $X_{0}^{2}$ (cf.\ \S6 of \cite{MelroseCorners}),
roughly meaning that $\Delta_{0}$ admits a collar neighborhood inside
$X_{0}^{2}$.
\begin{defn}
	(Small $0$-calculus) Let $m\in\mathbb{R}$. We denote by $\Psi_{0}^{m}\left(X\right)$
	the space of operators $A\in\Op\left(X\right)$ whose Schwartz kernels
	$K_{A}$ lift via $\beta_{0}$ to elements
	\[
	\kappa_{A}\in\mathcal{A}_{\phg}^{\left(\infty,\infty,0\right)}I^{m}\left(X_{0}^{2},\Delta_{0};r_{\ff_{0}}^{-n-1}\beta_{0,R}^{*}\mathcal{D}_{X}^{1}\right).
	\]
	This means that $\kappa_{A}$ is a section of $r_{\ff_{0}}^{-n-1}\beta_{0,R}^{*}\mathcal{D}_{X}^{1}$
	vanishing to infinite order at the faces $\lf,\rf$, smooth across
	$\ff_{0}$, and with an interior conormal singularity of order $m$
	along $\Delta_{0}$.
\end{defn}

\begin{rem}
	We refer to \S6 of \cite{MelroseCorners} for the notion of conormality
	along an interior $p$-submanifold. This notion is however not very
	important in this paper.
\end{rem}

\begin{defn}
	(Large residual $0$-calculus) Let $\mathcal{E}=\left(\mathcal{E}_{\lf},\mathcal{E}_{\rf},\mathcal{E}_{\ff_{0}}\right)$
	be a triple of index sets. We denote by $\Psi_{0}^{-\infty,\mathcal{E}}\left(X\right)$
	the space of operators $A\in\Op\left(X\right)$ whose Schwartz kernels
	$K_{A}$ lift via $\beta_{0}$ to elements
	\[
	\kappa_{A}\in\mathcal{A}_{\phg}^{\mathcal{E}}\left(X_{0}^{2};r_{\ff_{0}}^{-n-1}\beta_{0,R}^{*}\mathcal{D}_{X}^{1}\right).
	\]
\end{defn}

\begin{rem}
	Some authors, for example \cite{Hintz0calculus, Lauter}, prefer to
	work with operators acting on ``half-$0$-densities''. By the Schwartz
	Kernel Theorem, these operators are naturally described as half-$0$-densities
	on $X_{0}^{2}$, and this choice causes a discrepancy in the index
	set convention. Our index set convention is compatible with that of
	\cite{MazzeoEdgeI}.
\end{rem}

To introduce the large residual $b$-calculus and extended $0$-calculus,
we need to introduce two more blow-ups, namely the \emph{$b$-stretched
}double space
\[
X_{b}^{2}=\left[X^{2};\lf\cap\rf\right]
\]
and the $0b$\emph{-stretched }double space
\[
X_{0b}^{2}=\left[X_{0}^{2};\lf\cap\rf\right].
\]
We denote by $\beta_{b}$ (resp.\ $\beta_{0b}$) the blow-down maps
to $X^{2}$, and by $\beta_{b,L},\beta_{b,R}$ (resp.\ $\beta_{0b,L},\beta_{0b,R}$)
the lifts of the canonical left and right projections $\pi_{L},\pi_{R}:X^{2}\to X$.
Similarly to $X_{0}^{2}$, $X_{b}^{2}$ has three boundary hyperfaces:
the left and right faces $\lf,\rf$, obtained by lifting the left
and right faces of $X^{2}$, and the \emph{$b$-front face} $\ff_{b}$
created by the blow-up. The space $X_{0b}^{2}$ can be described equivalently
either as the simple blow-up of $X_{0}^{2}$ at the corner $\lf\cap\rf$,
or as the simple blow-up of $X_{b}^{2}$ at the locus $\beta_{b}^{-1}\left(\partial\Delta\right)$.
As such, we obtain four boundary hyperfaces $\lf,\rf,\ff_{b},\ff_{0}$,
all lifted from either $X_{0}^{2}$ or $X_{b}^{2}$ indifferently
via the partial blow-down maps $X_{0b}^{2}\to X_{0}^{2}$ and $X_{0b}^{2}\to X_{b}^{2}$.
The $0$-front face $\ff_{0}$ of $X_{0b}^{2}$ is again a smooth
fiber bundle over $\partial\Delta\equiv\partial X$, with typical
fiber the blow-up of the quarter $n+1$-dimensional sphere $S_{2}^{n+1}$
at its corner, i.e.\ the manifold
\[
\left[S_{2}^{n+1};\left\{ \theta_{0}=0,\theta_{n+1}=0\right\} \right].
\]

\begin{defn}
	(Large residual extended $0$-calculus) Let $\mathcal{E}=\left(\mathcal{E}_{\lf},\mathcal{E}_{\rf},\mathcal{E}_{\ff_{b}},\mathcal{E}_{\ff_{0}}\right)$
	be a quadruple of index sets. We define
	\[
	\Psi_{0b}^{-\infty,\mathcal{E}}\left(X\right)=\mathcal{A}_{\phg}^{\mathcal{E}}\left(X_{0b}^{2};r_{\ff_{0}}^{-n-1}r_{\ff_{b}}^{-1}\beta_{0b,R}^{*}\mathcal{D}_{X}^{1}\right).
	\]
\end{defn}

\begin{defn}
	(Large residual $b$-calculus) Let $\mathcal{E}=\left(\mathcal{E}_{\lf},\mathcal{E}_{\rf},\mathcal{E}_{\ff_{b}}\right)$
	be a triple of index sets. We define
	\[
	\Psi_{b}^{-\infty,\mathcal{E}}\left(X\right)=\mathcal{A}_{\phg}^{\mathcal{E}}\left(X_{b}^{2};r_{\ff_{b}}^{-1}\beta_{b,R}^{*}\mathcal{D}_{X}^{1}\right).
	\]
\end{defn}

We also need a class of simpler operators, for which we do not need
to blow $X^{2}$ up.
\begin{defn}
	(Very residual operators) Given an index family $\mathcal{E}=\left(\mathcal{E}_{\lf},\mathcal{E}_{\rf}\right)$
	for $X^{2}$, we denote by $\Psi^{-\infty,\mathcal{E}}\left(X\right)$
	the space $\mathcal{A}_{\phg}^{\mathcal{E}}\left(X^{2};\pi_{R}^{*}\mathcal{D}_{X}^{1}\right)$.
\end{defn}

The following result summarizes various inclusions between the classes
mentioned so far. These results are known from the sources cited above,
and are anyway straightforward consequences of the pull-back Theorem.
\begin{lem}
	We have the following inclusions:
	\begin{align*}
		\Psi^{-\infty,\left(\mathcal{E}_{\lf},\mathcal{E}_{\rf}\right)}\left(X\right) & \subseteq\Psi_{0}^{-\infty,\left(\mathcal{E}_{\lf},\mathcal{E}_{\rf},\mathcal{E}_{\lf}+\mathcal{E}_{\rf}+n+1\right)}\left(X\right)\\
		\Psi^{-\infty,\left(\mathcal{E}_{\lf},\mathcal{E}_{\rf}\right)}\left(X\right) & \subseteq\Psi_{b}^{-\infty,\left(\mathcal{E}_{\lf},\mathcal{E}_{\rf},\mathcal{E}_{\lf}+\mathcal{E}_{\rf}+1\right)}\left(X\right)\\
		\Psi_{0}^{-\infty,\left(\mathcal{E}_{\lf},\mathcal{E}_{\rf},\mathcal{E}_{\ff_{0}}\right)}\left(X\right) & \subseteq\Psi_{0b}^{-\infty,\left(\mathcal{E}_{\lf},\mathcal{E}_{\rf},\mathcal{E}_{\lf}+\mathcal{E}_{\rf}+1,\mathcal{E}_{\ff_{0}}\right)}\left(X\right)\\
		\Psi_{b}^{-\infty,\left(\mathcal{E}_{\lf},\mathcal{E}_{\rf},\mathcal{E}_{\ff_{b}}\right)}\left(X\right) & \subseteq\Psi_{0b}^{-\infty,\left(\mathcal{E}_{\lf},\mathcal{E}_{\rf},\mathcal{E}_{\ff_{b}},\mathcal{E}_{\ff_{b}}+n\right)}\left(X\right).
	\end{align*}
\end{lem}

\subsection{\label{subsec:-Poisson-and--trace}\texorpdfstring{$0$}{0}-Poisson and \texorpdfstring{$0$}{0}-trace operators}

Let's now pass to $0$-Poisson and $0$-trace operators. Analogously
to $\Op\left(X\right)$, we denote by $\Op\left(X,\partial X\right)$
the class of continuous linear maps \[\dot{C}^{\infty}\left(X\right)\to C^{-\infty}\left(\partial X\right),\]
which we call \emph{operators of trace type}. By the Schwartz Kernel
Theorem, these operators can be identified with elements of $C^{-\infty}\left(\partial X\times X;\pi_{R}^{*}\mathcal{D}_{X}^{1}\right)$.
Similarly, we denote by $\Op\left(\partial X,X\right)$ the class
of continuous linear maps $C^{\infty}\left(\partial X\right)\to C^{-\infty}\left(X\right)$,
which we call \emph{operators of} \emph{Poisson type}. These operators
are identified with right densities in $X\times\partial X$.

As above, we now define appropriate blow-ups of the double spaces
$\partial X\times X$ and $X\times\partial X$. Namely, we define
\begin{align*}
	\left(\partial X\times X\right)_{0} & =\left[\partial X\times X;\partial\Delta\right]\\
	\left(X\times\partial X\right)_{0} & =\left[X\times\partial X;\partial\Delta\right].
\end{align*}
Observe that we have canonical identifications $\left(\partial X\times X\right)_{0}\simeq\lf\left(X_{0}^{2}\right)$
and $\left(X\times\partial X\right)_{0}\simeq\rf\left(X_{0}^{2}\right)$.
These spaces have two boundary hyperfaces each: the \emph{original
	face} $\of$ is in both cases the lift of $\partial X\times\partial X$,
or equivalently the intersection $\lf\left(X_{0}^{2}\right)\cap\rf\left(X_{0}^{2}\right)$;
the \emph{front face} $\ff$ is the new face obtained from the blow-up,
which for $\left(\partial X\times X\right)_{0}$ coincides with $\lf\left(X_{0}^{2}\right)\cap\ff_{0}\left(X_{0}^{2}\right)$
while for $\left(X\times\partial X\right)_{0}$ coincides with $\rf\left(X_{0}^{2}\right)\cap\ff_{0}\left(X_{0}^{2}\right)$.
The front face $\ff\left(\left(\partial X\times X\right)_{0}\right)$
(resp.\ $\ff\left(\left(X\times\partial X\right)_{0}\right)$) is
again the total space of a smooth fiber bundle $\ff\to\partial\Delta$,
with typical fiber equal to the left (resp.\ right) face of the quarter
sphere $S_{2}^{n+1}$; of course, this face is canonically diffeomorphic
to the $n$-dimensional closed unit ball. We denote by $\beta_{\tr}:\left(\partial X\times X\right)_{0}\to\partial X\times X$
and $\beta_{\po}:\left(X\times\partial X\right)_{0}\to X\times\partial X$
the blow-down maps, and we call $\beta_{\tr,L},\beta_{\tr,R}$ and
$\beta_{\po,L},\beta_{\po,R}$ the lifted canonical projections.
\begin{defn}
	$ $
	\begin{enumerate}
		\item ($0$-trace operators) Given index sets $\mathcal{E}=\left(\mathcal{E}_{\of},\mathcal{E}_{\ff}\right)$,
		we define \[\Psi_{0\tr}^{-\infty,\mathcal{E}}\left(X,\partial X\right)=\mathcal{A}_{\phg}^{\left(\mathcal{E}_{\of},\mathcal{E}_{\ff}\right)}\left(\left(\partial X\times X\right)_{0};r_{\ff}^{-n-1}\beta_{\tr,R}^{*}\mathcal{D}_{X}^{1}\right);\]
		\item ($0$-Poisson operators) Given index sets $\mathcal{E}=\left(\mathcal{E}_{\of},\mathcal{E}_{\ff}\right)$,
		we define \[\Psi_{0\po}^{-\infty,\mathcal{E}}\left(\partial X,X\right)=\mathcal{A}_{\phg}^{\left(\mathcal{E}_{\of},\mathcal{E}_{\ff}\right)}\left(\left(X\times\partial X\right)_{0};r_{\ff}^{-n}\beta_{\po,R}^{*}\mathcal{D}_{\partial X}^{1}\right).\]
	\end{enumerate}
\end{defn}

As above, we also have classes of residual trace and Poisson operators
for which we don't need to blow the double space up:
\begin{defn}
	\leavevmode
	\begin{enumerate}
		\item (Residual trace operators) Given an index set $\mathcal{E}$, we define
		\[
		\Psi_{\tr}^{-\infty,\mathcal{E}}\left(X,\partial X\right)
		=
		\mathcal{A}_{\phg}^{\mathcal{E}}\left(\partial X\times X;\pi_{R}^{*}\mathcal{D}_{X}^{1}\right);
		\]
		
		\item (Residual Poisson operators) Given an index set $\mathcal{E}$, we define
		\[
		\Psi_{\po}^{-\infty,\mathcal{E}}\left(\partial X,X\right)
		=
		\mathcal{A}_{\phg}^{\mathcal{E}}\left(X\times\partial X;\pi_{R}^{*}\mathcal{D}_{\partial X}^{1}\right).
		\]
	\end{enumerate}
\end{defn}

Again, the pull-back Theorem easily implies the following
\begin{lem}
	We have inclusions
	\begin{align*}
		\Psi_{\tr}^{-\infty,\mathcal{E}_{\of}}\left(X,\partial X\right) & \subseteq\Psi_{0\tr}^{-\infty,\left(\mathcal{E}_{\of},\mathcal{E}_{\of}+n+1\right)}\left(X,\partial X\right)\\
		\Psi_{\po}^{-\infty,\mathcal{E}_{\of}}\left(\partial X,X\right) & \subseteq\Psi_{0\po}^{-\infty,\left(\mathcal{E}_{\of},\mathcal{E}_{\of}+n\right)}\left(\partial X,X\right).
	\end{align*}
\end{lem}

\subsection{\label{subsec:The-local-perspective}The local perspective}

To get a firm grasp on the classes of operators defined above, we
need to understand the local structure of their Schwartz kernels.
This local perspective will be important later, when we will develop
our alternative ``symbolic'' calculus.

\subsubsection{\label{subsubsec:local-physical-0-calculus-extended-0-calculus}$0$-calculus,
	$b$-calculus and extended $0$-calculus}

Recall the following notations taken from \cite{MelroseCorners}:
\begin{enumerate}
	\item we denote by $\mathbb{R}_{k}^{n}$ the model corner of dimension $n$
	and depth $k$, i.e.\ the product $[0,\infty)^{k}\times\mathbb{R}^{n-k}$;
	\item we denote by $\overline{\mathbb{R}^{n}}$ the radial compactification
	of $\mathbb{R}^{n}$ (cf.\ \S1.9 of \cite{MelroseMicrolocal});
	\item we denote by $\overline{\mathbb{R}_{k}^{n}}$ the portion of $\overline{\mathbb{R}^{n}}$
	defined by $x_{1}\geq0,...,x_{k}\geq0$.
\end{enumerate}
\begin{rem}
	The space $\dot{C}^{\infty}\left(\overline{\mathbb{R}^{n}}\right)$
	coincides with the space of Schwartz functions $\mathcal{S}\left(\mathbb{R}^{n}\right)$,
	and dually the space $C^{-\infty}\left(\overline{\mathbb{R}^{n}}\right)$
	coincides with the space $\mathcal{S}'\left(\mathbb{R}^{n}\right)$
	of tempered distributions.
\end{rem}

Let $P\in\Op\left(X\right)$ be an operator of interior type. Given
a point $p\in\partial X$ and coordinates $\left(x,y\right)$, the
Schwartz kernel of $P$ can be described in the induced coordinates
$\left(x,y,\tilde{x},\tilde{y}\right)$ on $X^{2}$ near $\left(p,p\right)$
as a right density of the form $P\left(x,y,\tilde{x},\tilde{y}\right)d\tilde{x}d\tilde{y}$,
where $P\left(x,y,\tilde{x},\tilde{y}\right)$ is a distribution defined
in a neighborhood of the origin in $\mathbb{R}_{1}^{n+1}\times\mathbb{R}_{1}^{n+1}$.
By performing the change of coordinates $Y=y-\tilde{y}$ and rearranging
the variables, we can equivalently write the Schwartz kernel as $P\left(y;x,\tilde{x},y-\tilde{y}\right)d\tilde{x}d\tilde{y}$,
where now $P\left(y;x,\tilde{x},Y\right)$ is a distribution defined
in a neighborhood of the origin in $\mathbb{R}^{n}\times\mathbb{R}_{1}^{1}\times\mathbb{R}_{1}^{1}\times\mathbb{R}^{n}$.

We are interested in characterizing the elements of $\Psi_{0}^{-\infty,\bullet}\left(X\right),\Psi_{0b}^{-\infty,\bullet}\left(X\right)$,
locally near points $\left(p,p\right)\in\partial\Delta$, in coordinates.
In order to do that, we first define local versions of these classes
in $\mathbb{R}_{1}^{n+1}$. Consider first the model space
\[
M^{2}=\overline{\mathbb{R}_{1}^{1}}\times\overline{\mathbb{R}_{1}^{1}}\times\overline{\mathbb{R}^{n}}
\]
with coordinates $x,\tilde{x},Y$. $M^{2}$ has five boundary hyperfaces,
namely the \emph{left }and \emph{right faces}
\[
\lf=\left\{ x=0\right\} ,\rf=\left\{ \tilde{x}=0\right\} 
\]
and the \emph{infinity} \emph{faces}
\[
\iif_{Y}=\left\{ \left|Y\right|=\infty\right\} ,\iif_{x}=\left\{ x=\infty\right\} ,\iif_{\tilde{x}}=\left\{ \tilde{x}=\infty\right\} .
\]
We think of $\overline{\mathbb{R}^{n}}\times M^{2}$, with coordinates
$y,x,\tilde{x},Y$, as a compactified local model for $X^{2}$ near
a point $\left(p,p\right)$ of $\partial\Delta$. Now, consider the
three blow-ups,
\begin{align*}
	M_{0}^{2} & =\left[M^{2};\lf\cap\rf\cap\left\{ Y=0\right\} \right]\\
	M_{b}^{2} & =\left[M^{2};\lf\cap\rf\right]\\
	M_{0b}^{2} & =\left[M_{0}^{2};\lf\cap\rf\right].
\end{align*}
These spaces inherit from $M^{2}$ the five boundary hyperfaces described
above, as lifts of these faces to the blow-ups. Additionally, $M_{0}^{2}$
and $M_{b}^{2}$ have each an extra boundary hyperface, created from
the blow-up. We call these faces $\ff_{0}$ in $M_{0}^{2}$, and $\ff_{b}$
in $M_{b}^{2}$. We order the faces of $M_{0}^{2}$ as $\lf$, $\rf$,
$\ff_{0}$, $\iif_{Y}$, $\iif_{x}$, $\iif_{\tilde{x}}$, and the
faces of $M_{b}^{2}$ as $\lf$, $\rf$, $\ff_{b}$, $\iif_{Y}$,
$\iif_{x}$, $\iif_{\tilde{x}}$. Finally, these faces lift to faces
of $M_{0b}^{2}$, and we order its seven faces as $\lf$, $\rf$,
$\ff_{b}$, $\ff_{0}$, $\iif_{Y}$, $\iif_{x}$, $\iif_{\tilde{x}}$.
The spaces $M^{2},M_{b}^{2}$ are represented in Figure \ref{fig:M2&M2b}\begin{figure}
	\centering
	\caption{$M^2$ and $M^2_b$}
	\label{fig:M2&M2b}
	
	\begin{minipage}{.49\textwidth}
		\centering
		\tikzmath{\L = 5;\R = \L /3;}
		\tdplotsetmaincoords{60}{160}
		\begin{tikzpicture}[tdplot_main_coords, scale = \myscale]
			
			%%% axes
			%\coordinate (O) at (0,0,0) ;
			%\coordinate (X) at (1,0,0) ;
			%\coordinate (Y) at (0,1,0) ;
			%\coordinate (Z) at (0,0,1) ;
			%\draw[->] (O) -- (X) node {$x$};
			%\draw[->] (O) -- (Y) node {$y$};
			%\draw[->] (O) -- (Z) node {$z$};
			
			%%% points
			\coordinate (A) at (-\L,0,0) ;
			\coordinate (A1) at (-\L+\R,0,0) ;
			\coordinate (A2) at (-\L,\R,0) ;
			\coordinate (A3) at (-\L,0,\R) ;
			
			\coordinate (B) at (\L,0,0) ;
			\coordinate (B1) at (\L-\R,0,0) ;
			\coordinate (B2) at (\L,\R,0) ;
			\coordinate (B3) at (\L,0,\R) ;
			
			\coordinate (C) at (-\L,\L,0) ;
			\coordinate (D) at (\L,\L,0) ;
			\coordinate (E) at (-\L,0,\L) ;
			\coordinate (F) at (\L,0,\L) ;
			\coordinate (G) at (-\L,\L,\L);
			\coordinate (H) at (\L,\L,\L);
			
			%%% segments
			\draw 
			(B) -- (F) 
			(B) -- (D) 
			(B) -- (A) 
			(A) -- (C) 
			(A) -- (E)
			;
			
			%%% dashed segments and arcs
			\draw[]
			(F) -- (E)
			(D) -- (C)
			(D) -- (H)
			(H) -- (F)
			(C) -- (G)
			(G) -- (E)
			(G) -- (H)
			;
			
			%%% text
			\node at (0,0,{\L/2}) {$\text{\LARGE lf}$};
			\node at (0,{\L/2},0) {$\text{\LARGE rf}$};
			
		\end{tikzpicture}
		
	\end{minipage}
	\begin{minipage}{.49\textwidth}
		\centering
		\tikzmath{\L = 5; \R = \L/6; \r= \L/6; \u= asin(\r/\R);}
		\tdplotsetmaincoords{60}{160}
		\begin{tikzpicture}[tdplot_main_coords, scale = \myscale]
			
			%%% points
			\coordinate (O) at (0,0,0) ;
			\coordinate (A) at (-\L,0,0) ;
			\coordinate (B) at (\L,0,0) ;
			\coordinate (C) at (-\L,\L,0) ;
			\coordinate (D) at (\L,\L,0) ;
			\coordinate (E) at (-\L,0,\L) ;
			\coordinate (F) at (\L,0,\L) ;
			\coordinate (G) at (-\L,\L,\L);
			\coordinate (H) at (\L,\L,\L);
			
			\coordinate (A1) at (-\L,\r,0) ;
			\coordinate (A2) at (-\L,0,\r) ;
			
			\coordinate (B1) at (\L,\r,0) ;
			\coordinate (B2) at (\L,0,\r) ;
			
			\coordinate (OA1) at ({-\R*cos(\u)},\r,0);
			\coordinate (OA2) at ({-\R*cos(\u)},0,\r);
			\coordinate (OB1) at ({\R*cos(\u)},\r,0);
			\coordinate (OB2) at ({\R*cos(\u)},0,\r);
			
			%%% segments
			\draw 
			(A1) -- (OA1) 
			(A2) -- (OA2)
			(B1) -- (OB1) 
			(B2) -- (OB2)
			(A1) -- (C)
			(A2) -- (E)
			(B1) -- (D)
			(B2) -- (F)
			;
			
			%%% dashed segments
			\draw[]
			(C) -- (D)
			(E) -- (F)
			(G) -- (H)
			(E) -- (G)
			(G) -- (C)
			(F) -- (H)
			(H) -- (D)
			;
			
			%%% arcs
			\begin{scope}[canvas is yz plane at x=0]
				\draw (A1) arc (0:90:\r);
				%\draw (OA1) arc (0:90:\r);
				\draw (B1) arc (0:90:\r);
				%\draw (OB1) arc (0:90:\r);
			\end{scope}
			\begin{scope}[canvas is xz plane at y=0]
				\draw (OB2) arc (\u:{180-\u}:\R);
			\end{scope}
			\begin{scope}[canvas is xy plane at z=0]
				\draw (OB1) arc (\u:{180-\u}:\R);
			\end{scope}
			
			%%% text
			\node at (0,0,{\L/2}) {$\text{\LARGE lf}$};
			\node at (0,{\L/2},0) {$\text{\LARGE rf}$};
			%\node at (0,{\L/6},{\L/6}) {$\text{\LARGE ff}_0$};
			%\node at ({\L/1.7},{\L/15},{\L/15}) {$\text{\LARGE ff}_b$};
			\node at ({0},{\L/15},{\L/15}) {$\text{\LARGE ff}_b$};
			
		\end{tikzpicture}
		
	\end{minipage}
	
\end{figure}, and the spaces $M_{0}^{2},M_{0b}^{2}$ are represented in Figure
\ref{fig:M20&M20b}\begin{figure}
	\centering
	\caption{$M^2_0$ and $M^2_{0b}$}
	\label{fig:M20&M20b}
	
	\begin{minipage}{.49\textwidth}
		\centering
		\tikzmath{\L = 5;\R = \L /3; \u = 0;}
		\tdplotsetmaincoords{60}{160}
		\begin{tikzpicture}[tdplot_main_coords, scale = \myscale]
			
			%%% axes
			%\coordinate (O) at (0,0,0) ;
			%\coordinate (X) at (1,0,0) ;
			%\coordinate (Y) at (0,1,0) ;
			%\coordinate (Z) at (0,0,1) ;
			%\draw[->] (O) -- (X) node {$x$};
			%\draw[->] (O) -- (Y) node {$y$};
			%\draw[->] (O) -- (Z) node {$z$};
			
			%%% points
			\coordinate (A) at (-\L,0,0) ;
			\coordinate (A1) at (-\L+\R,0,0) ;
			\coordinate (A4) at (-\L,{\R*sin(\u)},0) ;
			\coordinate (A5) at (-\L,0,{\R*sin(\u)}) ;
			\coordinate (A10) at ({-\R*cos(\u)},{\R*sin(\u)},0) ;
			\coordinate (A11) at ({-\R*cos(\u)},0,{\R*sin(\u)}) ;
			
			\coordinate (B) at (\L,0,0) ;
			\coordinate (B3) at (\L,0,\R) ;
			\coordinate (B4) at (\L,{\R*sin(\u)},0) ;
			\coordinate (B5) at (\L,0,{\R*sin(\u)}) ;
			\coordinate (B10) at ({\R*cos(\u)},{\R*sin(\u)},0) ;
			\coordinate (B11) at ({\R*cos(\u)},0,{\R*sin(\u)}) ;
			
			\coordinate (C) at (-\L,\L,0) ;
			\coordinate (D) at (\L,\L,0) ;
			\coordinate (E) at (-\L,0,\L) ;
			\coordinate (F) at (\L,0,\L) ;
			\coordinate (G) at (-\L,\L,\L);
			\coordinate (H) at (\L,\L,\L);
			
			%%% segments
			\draw 
			(B5) -- (F) 
			(B4) -- (D)
			(A4) -- (C) 
			(A5) -- (E)
			(B4) -- (B10)
			(B5) -- (B11)
			(A4) -- (A10)
			(A5) -- (A11)
			;
			
			%%% dashed segments and arcs
			\draw[]
			(F) -- (E)
			(D) -- (C)
			(D) -- (H)
			(H) -- (F)
			(C) -- (G)
			(G) -- (E)
			(G) -- (H)
			;
			%\begin{scope}[canvas is zy plane at x=0]
			%	\draw[] (E) arc(0:90:{3*\R});
			%	\draw[] (F) arc(0:90:{3*\R});
			%\end{scope}
			
			%%% arcs
			\begin{scope}[canvas is xy plane at z=0]
			\end{scope}
			\begin{scope}[canvas is xz plane at y=0]
			\end{scope}
			\begin{scope}[canvas is zy plane at x=0]
				\draw (B5) arc(0:90:{\R*sin(\u)});
				\draw (A5) arc(0:90:{\R*sin(\u)});
				\draw (B11) arc(0:90:{\R*sin(\u)});
				\draw (A11) arc(0:90:{\R*sin(\u)});
			\end{scope}
			\begin{scope}[canvas is xy plane at z=0]
				\draw ({\R*cos(\u)},{\R*sin(\u)}) arc(\u:{180-\u}:\R);
			\end{scope}
			\begin{scope}[canvas is xz plane at y=0]
				\draw ({\R*cos(\u)},{\R*sin(\u)}) arc(\u:{180-\u}:\R);
			\end{scope}

			%%% text
			\node at (0,0,{\L/2}) {$\text{\LARGE lf}$};
			\node at (0,{\L/2},0) {$\text{\LARGE rf}$};
			\node at (0,{\L/6},{\L/6}) {$\text{\LARGE ff}_0$};
			
			%%% coordinates
			%\draw[->,red] ({\L/6},{\L/6},{\L/6}) -- node[below, midway] {$\eta$} (-{\L/6},{\L/6},{\L/6});
			%\draw[->,red] ({\L/6},0,{2*\L/6}) -- node[right, midway] {$\tilde x$} (1,0,4);
			%\draw[->,red] ({\L/6},{2*\L/6},0) -- node[right, midway] {$x$} (1,4,0);
			
			%\begin{scope}[shift={(-{1.2*\L/3},0,0)}, canvas is zy plane at x=0]
			%	\draw[->,blue!50!green!50] ({\R*sin(\u)},0) arc(0:70:{\R*sin(\u)}) node[right, midway] {$s$};
			%\end{scope}
			
		\end{tikzpicture}
		
	\end{minipage}
	\begin{minipage}{.49\textwidth}
		\centering
		\tikzmath{\L = 5;\R = \L /3; \u = 180/6;}
		\tdplotsetmaincoords{60}{160}
		\begin{tikzpicture}[tdplot_main_coords, scale = \myscale]
			
			%%% axes
			%\coordinate (O) at (0,0,0) ;
			%\coordinate (X) at (1,0,0) ;
			%\coordinate (Y) at (0,1,0) ;
			%\coordinate (Z) at (0,0,1) ;
			%\draw[->] (O) -- (X) node {$x$};
			%\draw[->] (O) -- (Y) node {$y$};
			%\draw[->] (O) -- (Z) node {$z$};
			
			%%% points
			\coordinate (A) at (-\L,0,0) ;
			\coordinate (A1) at (-\L+\R,0,0) ;
			\coordinate (A4) at (-\L,{\R*sin(\u)},0) ;
			\coordinate (A5) at (-\L,0,{\R*sin(\u)}) ;
			\coordinate (A10) at ({-\R*cos(\u)},{\R*sin(\u)},0) ;
			\coordinate (A11) at ({-\R*cos(\u)},0,{\R*sin(\u)}) ;
			
			\coordinate (B) at (\L,0,0) ;
			\coordinate (B3) at (\L,0,\R) ;
			\coordinate (B4) at (\L,{\R*sin(\u)},0) ;
			\coordinate (B5) at (\L,0,{\R*sin(\u)}) ;
			\coordinate (B10) at ({\R*cos(\u)},{\R*sin(\u)},0) ;
			\coordinate (B11) at ({\R*cos(\u)},0,{\R*sin(\u)}) ;
			
			\coordinate (C) at (-\L,\L,0) ;
			\coordinate (D) at (\L,\L,0) ;
			\coordinate (E) at (-\L,0,\L) ;
			\coordinate (F) at (\L,0,\L) ;
			\coordinate (G) at (-\L,\L,\L);
			\coordinate (H) at (\L,\L,\L);
			
			%%% segments
			\draw 
			(B5) -- (F) 
			(B4) -- (D)
			(A4) -- (C) 
			(A5) -- (E)
			(B4) -- (B10)
			(B5) -- (B11)
			(A4) -- (A10)
			(A5) -- (A11)
			;
			
			%%% dashed segments and arcs
			\draw[]
			(F) -- (E)
			(D) -- (C)
			(D) -- (H)
			(H) -- (F)
			(C) -- (G)
			(G) -- (E)
			(G) -- (H)
			;
			%\begin{scope}[canvas is zy plane at x=0]
			%	\draw[] (E) arc(0:90:{3*\R});
			%	\draw[] (F) arc(0:90:{3*\R});
			%\end{scope}
			
			%%% arcs
			\begin{scope}[canvas is xy plane at z=0]
			\end{scope}
			\begin{scope}[canvas is xz plane at y=0]
			\end{scope}
			\begin{scope}[canvas is zy plane at x=0]
				\draw (B5) arc(0:90:{\R*sin(\u)});
				\draw (A5) arc(0:90:{\R*sin(\u)});
				\draw (B11) arc(0:90:{\R*sin(\u)});
				\draw (A11) arc(0:90:{\R*sin(\u)});
			\end{scope}
			\begin{scope}[canvas is xy plane at z=0]
				\draw ({\R*cos(\u)},{\R*sin(\u)}) arc(\u:{180-\u}:\R);
			\end{scope}
			\begin{scope}[canvas is xz plane at y=0]
				\draw ({\R*cos(\u)},{\R*sin(\u)}) arc(\u:{180-\u}:\R);
			\end{scope}

			%%% text
			\node at (0,0,{\L/2}) {$\text{\LARGE lf}$};
			\node at (0,{\L/2},0) {$\text{\LARGE rf}$};
			\node at (0,{\L/6},{\L/6}) {$\text{\LARGE ff}_0$};
			\node at ({\L/1.7},{\L/15},{\L/15}) {$\text{\LARGE ff}_b$};
			\node at ({-\L/1.7},{\L/15},{\L/15}) {$\text{\LARGE ff}_b$};
			
			%%% coordinates
			%\draw[->,red] ({\L/6},{\L/6},{\L/6}) -- node[below, midway] {$\eta$} (-{\L/6},{\L/6},{\L/6});
			%\draw[->,red] ({\L/6},0,{2*\L/6}) -- node[right, midway] {$\tilde x$} (1,0,4);
			%\draw[->,red] ({\L/6},{2*\L/6},0) -- node[right, midway] {$x$} (1,4,0);
			
			%\begin{scope}[shift={(-{1.2*\L/3},0,0)}, canvas is zy plane at x=0]
			%	\draw[->,blue!50!green!50] ({\R*sin(\u)},0) arc(0:70:{\R*sin(\u)}) node[right, midway] {$s$};
			%\end{scope}
			
		\end{tikzpicture}
		
	\end{minipage}
	
\end{figure} (the infinity faces are not indicated in order to not weigh down
the pictures). Of course, the spaces $\overline{\mathbb{R}^{n}}\times M_{0}^{2}$,
$\overline{\mathbb{R}^{n}}\times M_{b}^{2}$ and $\overline{\mathbb{R}^{n}}\times M_{0b}^{2}$
should be considered as compactified local models for $X_{0}^{2}$,
$X_{b}^{2}$ and $X_{0b}^{2}$, in a neighborhood of the lift of $\left(p,p\right)$.

Denote by $\Op\left(\mathbb{R}_{1}^{n+1}\right)$ the space of continuous
linear operators \[\dot{C}^{\infty}\left(\overline{\mathbb{R}_{1}^{1}}\times\overline{\mathbb{R}^{n}}\right)\to C^{-\infty}\left(\overline{\mathbb{R}_{1}^{1}}\times\overline{\mathbb{R}^{n}}\right).\]
An extendible distribution $K\left(y;x,\tilde{x},Y\right)$ on $\overline{\mathbb{R}^{n}}\times M^{2}$
determines an element of $\Op\left(\mathbb{R}_{1}^{n+1}\right)$,
with Schwartz kernel $K\left(y;x,\tilde{x},y-\tilde{y}\right)d\tilde{x}d\tilde{y}$.
The local version of the (large, residual) $0$-calculus is then the
following:
\begin{defn}
	(Local $0$-calculus) We denote by $\Psi_{0,\mathcal{S}}^{-\infty,\mathcal{E}}\left(\mathbb{R}_{1}^{n+1}\right)$
	the subspace of $\Op\left(\mathbb{R}_{1}^{n+1}\right)$ consisting
	of operators of the form $K\left(y;x,\tilde{x},y-\tilde{y}\right)d\tilde{x}d\tilde{y}$,
	with\footnote{The symbol $\hat{\otimes}$ indicates the completed projective tensor
		product of Fréchet spaces; the properties of $\hat{\otimes}$ which
		are important for us are summarized in the Appendix. We refer to \cite{Treves}
		for details.} \[K\left(y;x,\tilde{x},Y\right)\in\mathcal{S}\left(\mathbb{R}^{n}\right)\hat{\otimes}\mathcal{A}_{\phg}^{\left(\mathcal{E}_{\lf},\mathcal{E}_{\rf},\mathcal{E}_{\ff_{0}}-n-1,\infty,\infty,\infty\right)}\left(M_{0}^{2}\right).\]
\end{defn}

\begin{rem}
	Essentially by definition, the space $\mathcal{S}\left(\mathbb{R}^{n}\right)$
	is precisely the space of smooth functions on $\overline{\mathbb{R}^{n}}$
	vanishing to infinite order along its boundary at infinity; in other
	words, $\mathcal{S}\left(\mathbb{R}^{n}\right)=\dot{C}^{\infty}\left(\overline{\mathbb{R}^{n}}\right)=\mathcal{A}_{\phg}^{\infty}\left(\overline{\mathbb{R}^{n}}\right)$.
	Therefore, from the discussion in \S\ref{subsec:Conormality,-polyhomogeneity,projective-tensor-products},
	the space \[\mathcal{S}\left(\mathbb{R}^{n}\right)\hat{\otimes}\mathcal{A}_{\phg}^{\mathcal{E}}\left(M_{0}^{2}\right)\]
	coincides with the space of polyhomogeneous functions on $\overline{\mathbb{R}^{n}}\times M_{0}^{2}$
	with index set $\infty$ at the face $\partial\overline{\mathbb{R}^{n}}\times M_{0}^{2}$,
	and with index sets $\mathcal{E}$ at the other boundary hyperfaces
	of $\overline{\mathbb{R}^{n}}\times M_{0}^{2}$.
\end{rem}

In the previous definition, the subscript $\mathcal{S}$ stands for
``Schwartz'' and refers to the fact that $K\left(y;x,\tilde{x},Y\right)$
vanishes to infinite order as $\left|y\right|\to\infty$. We similarly
define
\begin{defn}
	(Local extended $0$-calculus) We define $\Psi_{0b,\mathcal{S}}^{-\infty,\mathcal{E}}\left(\mathbb{R}_{1}^{n+1}\right)$
	as the subclass of $\Op\left(\mathbb{R}_{1}^{n+1}\right)$ consisting
	of right densities $K\left(y;x,\tilde{x},y-\tilde{y}\right)d\tilde{x}d\tilde{y}$
	with $K\left(y;x,\tilde{x},Y\right)$ lifting to an element of $\mathcal{S}\left(\mathbb{R}^{n}\right)\hat{\otimes}\mathcal{A}_{\phg}^{\left(\mathcal{E}_{\lf},\mathcal{E}_{\rf},\mathcal{E}_{\ff_{b}}-1,\mathcal{E}_{\ff_{0}}-n-1,\infty,\infty,\infty\right)}\left(M_{0b}^{2}\right)$.
\end{defn}

The following two lemmas are then easy consequences of the definitions:
\begin{lem}
	\label{lem:Local-characterization-0-calculus}(Local characterization
	of the $0$-calculus) Let $\mathcal{E}=\left(\mathcal{E}_{\lf},\mathcal{E}_{\rf},\mathcal{E}_{\ff_{0}}\right)$
	be a triple of index sets. If $P\in\Psi_{0}^{-\infty,\mathcal{E}}\left(X\right)$,
	then:
	\begin{enumerate}
		\item for every open neighborhood $U$ of $\partial\Delta$ in $X^{2}$,
		the restriction of the Schwartz kernel $K_{P}$ to $X^{2}\backslash U$
		coincides with the Schwartz kernel of a very residual operator in
		$\Psi^{-\infty,\left(\mathcal{E}_{\lf},\mathcal{E}_{\rf}\right)}\left(X\right)$;
		\item for every point $p\in\partial X$ and every choice of coordinates
		$\left(x,y\right)$ for $X$ centered at $p$, there is an operator
		$K\left(y;x,\tilde{x},y-\tilde{y}\right)d\tilde{x}d\tilde{y}\in\Psi_{0,\mathcal{S}}^{-\infty,\mathcal{E}}\left(\mathbb{R}_{1}^{n+1}\right)$
		such that, in the chosen coordinates, we have
		\[
		K_{P}\equiv K\left(y;x,\tilde{x},y-\tilde{y}\right)d\tilde{x}d\tilde{y}
		\]
		in a neighborhood of the origin.
	\end{enumerate}
	Conversely, if an operator $P:\dot{C}^{\infty}\left(X\right)\to C^{-\infty}\left(X\right)$
	satisfies Point 1, and for every $p\in\partial X$ there exist coordinates
	$\left(x,y\right)$ centered at $p$ such that the condition in Point
	2 is satisfied, then $P\in\Psi_{0}^{-\infty,\mathcal{E}}\left(X\right)$.
\end{lem}

\begin{lem}
	\label{lem:local-characterization-0b-calculus}(Local characterization
	of the extended $0$-calculus) Let $\mathcal{E}=\left(\mathcal{E}_{\lf},\mathcal{E}_{\rf},\mathcal{E}_{\ff_{b}},\mathcal{E}_{\ff_{0}}\right)$
	be a quadruple of index sets. If $P\in\Psi_{0b}^{-\infty,\mathcal{E}}\left(X\right)$,
	then:
	\begin{enumerate}
		\item for every open neighborhood $U$ of $\partial\Delta$ in $X^{2}$,
		the restriction of the Schwartz kernel $K_{P}$ to $X^{2}\backslash U$
		coincides with the Schwartz kernel of an operator in $\Psi_{b}^{-\infty,\left(\mathcal{E}_{\lf},\mathcal{E}_{\rf},\mathcal{E}_{\ff_{b}}\right)}\left(X\right)$;
		\item for every point $p\in\partial X$ and every choice of coordinates
		$\left(x,y\right)$ for $X$ centered at $p$, there exists an operator
		$K\left(y;x,\tilde{x},y-\tilde{y}\right)d\tilde{x}d\tilde{y}\in\Psi_{0b,\mathcal{S}}^{-\infty,\mathcal{E}}\left(\mathbb{R}_{1}^{n+1}\right)$
		such that in the chosen coordinates we have
		\[
		K_{P}\equiv K\left(y;x,\tilde{x},y-\tilde{y}\right)d\tilde{x}d\tilde{y}
		\]
		near the origin.
	\end{enumerate}
	Conversely, if an operator $P:\dot{C}^{\infty}\left(X\right)\to C^{-\infty}\left(X\right)$
	satisfies Point 1, and for every $p\in\partial X$ there exist coordinates
	$\left(x,y\right)$ for $X$ centered at $p$ such that the condition
	in Point 2 is satisfied, then $P\in\Psi_{0b}^{-\infty,\mathcal{E}}\left(X\right)$.
\end{lem}

For completeness, we also formulate the local versions of the classes
of very residual operators and the $b$-calculus:
\begin{defn}
	(Local $b$-calculus and very residual operators)
	We define
	\[
	\Psi_{\mathcal{S}}^{-\infty,
		\left(\mathcal{E}_{\lf},\mathcal{E}_{\rf}\right)}
	\left(\mathbb{R}_{1}^{n+1}\right)
	\]
	and, respectively,
	\[
	\Psi_{b,\mathcal{S}}^{-\infty,
		\left(\mathcal{E}_{\lf},\mathcal{E}_{\rf},
		\mathcal{E}_{\ff_{b}}\right)}
	\left(\mathbb{R}_{1}^{n+1}\right)
	\]
	as the spaces of operators
	$K\left(y;x,\tilde{x},y-\tilde{y}\right)
	d\tilde{x}d\tilde{y}$
	satisfying, respectively,
	\[
	K\left(y;x,\tilde{x},Y\right)
	\in
	\mathcal{S}\left(\mathbb{R}^{n}\right)
	\hat{\otimes}
	\mathcal{A}_{\phg}^{
		\left(\mathcal{E}_{\lf},\mathcal{E}_{\rf},
		\infty,\infty,\infty\right)}
	\left(M^{2}\right)
	\]
	and
	\[
	K\left(y;x,\tilde{x},Y\right)
	\in
	\mathcal{S}\left(\mathbb{R}^{n}\right)
	\hat{\otimes}
	\mathcal{A}_{\phg}^{
		\left(\mathcal{E}_{\lf},\mathcal{E}_{\rf},
		\mathcal{E}_{\ff_{b}}-1,
		\infty,\infty,\infty\right)}
	\left(M_{b}^{2}\right).
	\]
\end{defn}

It is easy to see that if $P\in\Psi^{-\infty,\mathcal{E}}\left(X\right)$
(resp.\ $P\in\Psi_{b}^{-\infty,\mathcal{E}}\left(X\right)$), then
for every $p\in\partial X$ and some (hence every) choice of coordinates
$x,y$ for $X$ centered at $p$, there is an operator $K\left(y;x,\tilde{x},y-\tilde{y}\right)d\tilde{x}d\tilde{y}$
in $\Psi_{\mathcal{S}}^{-\infty,\mathcal{E}}\left(\mathbb{R}_{1}^{n+1}\right)$
(resp.\ $\Psi_{b,\mathcal{S}}^{-\infty,\mathcal{E}}\left(\mathbb{R}_{1}^{n+1}\right)$)
such that in the chosen coordinates we have $K_{P}\equiv K\left(y;x,\tilde{x},y-\tilde{y}\right)d\tilde{x}d\tilde{y}$
near the origin.

\subsubsection{\label{subsubsec:local-physical-0-trace-and-0-Poisson}$0$-trace
	and $0$-Poisson operators}

We get similar local characterizations for $0$-Poisson and $0$-trace
operators. Define
\[
P^{2}=\overline{\mathbb{R}_{1}^{1}}\times\overline{\mathbb{R}^{n}}
\]
with coordinates $x,Y$, which we identify with the right face of
the model space $M^{2}$ introduced previously. We introduce the blow-up
\[
P_{0}^{2}=\left[P^{2};\left\{ x=0,Y=0\right\} \right],
\]
which we can canonically identify with the right face of $M_{0}^{2}$.
$P_{0}^{2}$ has again four boundary hyperfaces, namely $\of$ obtained
as the lift of $x=0$, $\ff$ obtained from the blow-up, $\iif_{Y}$
obtained as the lift of $\left|Y\right|=\infty$, and $\iif_{x}=\left\{ x=\infty\right\} $.
The space $P_{0}^{2}$ is represented in Figure \ref{fig:P20}\begin{figure}
	\centering
	\caption{$P^2_0$}
	\label{fig:P20}
	
	\tikzmath{\L = 4;\R = \L /3;}
	\begin{tikzpicture}
		
		%%% axes
		%\coordinate (O) at (0,0) ;
		%\coordinate (X) at (1,0) ;
		%\coordinate (Y) at (0,1) ;
		%\draw[->] (O) -- (X) node {$x$};
		%\draw[->] (O) -- (Y) node {$y$};
		
		%%% points
		\coordinate (O) at (0,0);
		\coordinate (A) at (-\L, 0);
		\coordinate (B) at (\L, 0);
		\coordinate (C) at (-\L, \L);
		\coordinate (D) at (\L, \L);
		\coordinate (A1) at ({-\L+\R}, 0);
		\coordinate (A2) at (-\L, \R);
		\coordinate (A3) at (-\R,0);
		\coordinate (B1) at ({\L-\R}, 0);
		\coordinate (B2) at (\L, \R);
		\coordinate (B3) at (\R,0);
		
		%%% dashed segments
		\draw (C) -- (D);
		
		%%% segments
		\draw (A) -- (A3);
		\draw (B) -- (B3);
		\draw (A) -- (C);
		\draw (B) -- (D);
		
		%%% arcs
		\draw (B3) arc(0:180:\R);
		
		%%% coordinates
		%\draw[->,red] (-{\L/6},{\L/2}) -- node[below, midway] {$Y$} ({\L/6},{\L/2});
		%\draw[->,red] (0,{\L/4 + 2*\L/6}) -- node[right, midway] {$x$} (0,{\L/4 + 4*\L/6});
		%\draw[->,blue] ({(-\R - \L/15)*cos(10)}, {sin(10)}) arc(170:10:{\R + \L/15}) node[above, midway] {$u$} ;
		
		%%% annotations
		\node at ({-\L+1.7*\L/6},{\L/12}) {${\text{\Large of}}$};
		\node at ({\L-1.7*\L/6},{\L/12}) {${\text{\Large of}}$};
		\node at (0,{\R + \L/12}) {${\text{\Large ff}}$};
		\node at ({-\L+0.6*\L/6},{\L*2/3}) {${\text{\Large if}_Y}$};
		\node at ({\L-0.6*\L/6},{\L*2/3}) {${\text{\Large if}_Y}$};
		\node at ({0},{\L-0.6*\L/6}) {${\text{\Large if}_x}$};

	\end{tikzpicture}
\end{figure}. For every point $p\in\partial X$, a choice of coordinates $\left(x,y\right)$
for $X$ centered at $p$ determines a local diffeomorphism between
a neighborhood of $\left(p,p\right)$ in $X\times\partial X$ and
a neighborhood of the origin in $\overline{\mathbb{R}^{n}}\times P^{2}$:
this diffeomorphism lifts to a neighborhood of the fiber $\ff_{p}$
in $\left(X\times\partial X\right)_{0}$, and a neighborhood of $\left\{ 0\right\} \times\ff$
in $\overline{\mathbb{R}^{n}}\times P_{0}^{2}$.

Similarly, define
\[
T^{2}=\overline{\mathbb{R}_{1}^{1}}\times\overline{\mathbb{R}^{n}}
\]
with coordinates $\tilde{x},Y$, which we identify with the left face
of $M^{2}$. We introduce the blow-up
\[
T_{0}^{2}=\left[T^{2};\left\{ \tilde{x}=0,Y=0\right\} \right],
\]
which we can canonically identify with the left face of $M_{0}^{2}$.
Again $T_{0}^{2}$ is canonically diffeomorphic to $P_{0}^{2}$, and
we call its faces $\of,\ff,\iif_{Y},\iif_{\tilde{x}}$.

Denote by $\Op\left(\mathbb{R}^{n},\mathbb{R}_{1}^{n+1}\right)$ the
class of operators \[\mathcal{S}\left(\mathbb{R}^{n}\right)\to C^{-\infty}\left(\overline{\mathbb{R}_{1}^{1}}\times\overline{\mathbb{R}^{n}}\right),\]
and by $\Op\left(\mathbb{R}_{1}^{n+1},\mathbb{R}^{n}\right)$ the
class of operators \[\dot{C}^{\infty}\left(\overline{\mathbb{R}_{1}^{1}}\times\overline{\mathbb{R}^{n}}\right)\to\mathcal{S}'\left(\mathbb{R}^{n}\right).\]
\begin{defn}
	(Local $0$-trace operators) We denote by $\Psi_{0\tr,\mathcal{S}}^{-\infty,\mathcal{E}}\left(\mathbb{R}_{1}^{n+1},\mathbb{R}^{n}\right)$
	the subclass of $\Op\left(\mathbb{R}_{1}^{n+1},\mathbb{R}^{n}\right)$
	consisting of right densities $K\left(y;\tilde{x},y-\tilde{y}\right)d\tilde{x}d\tilde{y}$
	with $K\left(y;\tilde{x},Y\right)$ lifting to $\mathcal{S}\left(\mathbb{R}^{n}\right)\hat{\otimes}\mathcal{A}_{\phg}^{\left(\mathcal{E}_{\of},\mathcal{E}_{\ff}-n-1,\infty,\infty\right)}\left(T_{0}^{2}\right)$.
\end{defn}

\begin{defn}
	(Local $0$-Poisson operators) We denote by $\Psi_{0\po,\mathcal{S}}^{-\infty,\mathcal{E}}\left(\mathbb{R}^{n},\mathbb{R}_{1}^{n+1}\right)$
	the subclass of $\Op\left(\mathbb{R}^{n},\mathbb{R}_{1}^{n+1}\right)$
	consisting of right densities $K\left(y;x,y-\tilde{y}\right)d\tilde{y}$,
	with $K\left(y;x,Y\right)$ lifting to $\mathcal{S}\left(\mathbb{R}^{n}\right)\hat{\otimes}\mathcal{A}_{\phg}^{\left(\mathcal{E}_{\of},\mathcal{E}_{\ff}-n,\infty,\infty\right)}\left(P_{0}^{2}\right)$.
\end{defn}

Analogously to the interior cases, the definitions imply easily the
following characterizations:
\begin{lem}
	\label{lem:local-characterization-0-trace}(Local characterization
	of $0$-trace operators) Let $\mathcal{E}=\left(\mathcal{E}_{\of},\mathcal{E}_{\ff}\right)$
	be a pair of index sets. If $A\in\Psi_{0\tr}^{-\infty,\mathcal{E}}\left(X,\partial X\right)$
	, then:
	\begin{enumerate}
		\item for every open neighborhood $U$ of $\partial\Delta$ in $X\times\partial X$,
		the restriction of the Schwartz kernel $K_{A}$ to $\left(\partial X\times X\right)\backslash U$
		coincides with the Schwartz kernel of a residual trace operator in
		$\Psi_{\tr}^{-\infty,\mathcal{E}_{\of}}\left(X,\partial X\right)$;
		\item for every point $p\in\partial X$ and every choice of coordinates
		$\left(x,y\right)$ for $X$ centered at $p$, there is an operator
		$K\left(y;\tilde{x},y-\tilde{y}\right)d\tilde{x}d\tilde{y}\in\Psi_{0\tr,\mathcal{S}}^{-\infty,\mathcal{E}}\left(\mathbb{R}_{1}^{n+1},\mathbb{R}^{n}\right)$
		such that in the chosen coordinates we have
		\[
		K_{A}\equiv K\left(y;\tilde{x},y-\tilde{y}\right)d\tilde{x}d\tilde{y}
		\]
		in a neighborhood of the origin.
	\end{enumerate}
	Conversely, if an operator $A:\dot{C}^{\infty}\left(X\right)\to C^{-\infty}\left(\partial X\right)$
	satisfies Point 1, and for every $p\in\partial X$ there exist coordinates
	$\left(x,y\right)$ centered at $p$ such that the condition in Point
	2 is satisfied, then $A\in\Psi_{0\tr}^{-\infty,\mathcal{E}}\left(X,\partial X\right)$.
\end{lem}

\begin{lem}
	\label{lem:local-characterization-0-poisson}(Local characterization
	of $0$-Poisson operators) Let $\mathcal{E}=\left(\mathcal{E}_{\of},\mathcal{E}_{\ff}\right)$
	be a pair of index sets. If $B\in\Psi_{0\po}^{-\infty,\mathcal{E}}\left(\partial X,X\right)$
	, then:
	\begin{enumerate}
		\item for every open neighborhood $U$ of $\partial\Delta$ in $X\times\partial X$,
		the restriction of the Schwartz kernel $K_{B}$ to $\left(X\times\partial X\right)\backslash U$
		coincides with the Schwartz kernel of a residual Poisson operator
		in $\Psi_{\po}^{-\infty,\mathcal{E}_{\of}}\left(\partial X,X\right)$;
		\item for every point $p\in\partial X$ and every choice of coordinates
		$\left(x,y\right)$ for $X$ centered at $p$, there is an operator
		$K\left(y;x,y-\tilde{y}\right)d\tilde{y}\in\Psi_{0\po,\mathcal{S}}^{-\infty,\mathcal{E}}\left(\mathbb{R}^{n},\mathbb{R}_{1}^{n+1}\right)$
		such that in the chosen coordinates we have
		\[
		K_{B}\equiv K\left(y;x,y-\tilde{y}\right)d\tilde{y}
		\]
		in a neighborhood of the origin.
	\end{enumerate}
	Conversely, if an operator $B:C^{\infty}\left(\partial X\right)\to C^{-\infty}\left(X\right)$
	satisfies Point 1, and for every $p\in\partial X$ there exist coordinates
	$\left(x,y\right)$ centered at $p$ such that the condition in Point
	2 is satisfied, then $B\in\Psi_{0\po}^{-\infty,\mathcal{E}}\left(\partial X,X\right)$.
\end{lem}

Finally, we formulate the local versions of the classes of residual
Poisson and trace operators for completeness.
\begin{defn}
	(Local residual trace and Poisson operators)
	\begin{enumerate}
		\item The space of residual Poisson operators
		$\Psi_{\po,\mathcal{S}}^{-\infty,\mathcal{E}_{\of}}
		\left(\mathbb{R}^{n},\mathbb{R}_{1}^{n+1}\right)$
		consists of operators represented by right densities
		\[
		K\left(y;x,y-\tilde{y}\right)d\tilde{y},
		\]
		where
		\[
		K\left(y;x,Y\right)
		\in\mathcal{S}\left(\mathbb{R}^{n}\right)
		\hat{\otimes}
		\mathcal{A}_{\phg}^{\left(
			\mathcal{E}_{\of},\infty,\infty\right)}
		\left(P^{2}\right).
		\]
		
		\item The space of residual trace operators
		$\Psi_{\tr,\mathcal{S}}^{-\infty,\mathcal{E}_{\of}}
		\left(\mathbb{R}_{1}^{n+1},\mathbb{R}^{n}\right)$
		consists of operators represented by right densities
		\[
		K\left(y;\tilde{x},y-\tilde{y}\right)
		d\tilde{x}d\tilde{y},
		\]
		where
		\[
		K\left(y;\tilde{x},Y\right)
		\in\mathcal{S}\left(\mathbb{R}^{n}\right)
		\hat{\otimes}
		\mathcal{A}_{\phg}^{\left(
			\mathcal{E}_{\of},\infty,\infty\right)}
		\left(T^{2}\right).
		\]
	\end{enumerate}
\end{defn}

Again, it is easy to formulate a local characterization for elements
of $\Psi_{\tr}^{-\infty,\mathcal{E}}\left(X,\partial X\right)$ and
$\Psi_{\po}^{-\infty,\mathcal{E}}\left(\partial X,X\right)$ in terms
of their local counterparts, along the lines of the characterizations
given above.

\subsection{\label{subsec:Model-families}Model families}

Every operator in the classes $\Psi_{0}^{-\infty,\bullet}\left(X\right)$,
$\Psi_{0b}^{-\infty,\bullet}\left(X\right)$, $\Psi_{0\tr}^{-\infty,\bullet}\left(X,\partial X\right)$,
$\Psi_{0\po}^{-\infty,\bullet}\left(\partial X,X\right)$ comes equipped
with a family of model operators, smoothly parametrized by a point
in $\partial X$. This family is called the\emph{ normal family} of
the operator: it is essentially the leading term in the asymptotic
expansion of the lifted Schwartz kernel at the front face. Using the
Fourier transform, the normal family can be conveniently re-packaged
as a family of model operators on a simpler model space, smoothly
parametrized by a nonzero covector in $\partial X$. This family will
be called the \emph{Bessel family} in accordance with \cite{MazzeoEdgeI}
(other sources such as \cite{Hintz0calculus} refer to it as the \emph{transformed
	normal operator}).\emph{ }We will now recall their definitions.

Recall our convention for the leading set of an index set (cf.\ Definition
\ref{def:leading-set-index-set}). Given an index set $\mathcal{I}$,
its leading set $\left[\mathcal{I}\right]$ is the set of pairs $\left(\alpha,l\right)\in\mathcal{I}$
such that $\inf\Re\left(\mathcal{I}\right)\leq\Re\left(\alpha\right)<\inf\Re\left(\mathcal{I}\right)+1$.
In what follows, we will often use the expression $\left[\mathcal{I}\right]=\mu$,
where $\mu\in\mathbb{C}$; this means that $\mathcal{I}$ contains
the pair $\left(\mu,0\right)$, and that every other pair in $\mathcal{I}$
is of the form $\left(\alpha,l\right)$ with $\Re\left(\alpha\right)\geq\Re\left(\mu\right)+1$.
For simplicity, in this sub-section we restrict ourselves to operator
whose index set $\mathcal{I}$ at $\ff_{0}$ satisfies $\left[\mathcal{I}\right]=0$.

\subsubsection{\label{subsubsec:The-normal-family}The normal family}

The invariant definitions are straightforward:
\begin{defn}
	(Normal family in the $0$-calculus and extended $0$-calculus) Let
	\[\mathcal{E}=\left(\mathcal{E}_{\lf},\mathcal{E}_{\rf},\mathcal{E}_{\ff_{b}},\mathcal{E}_{\ff_{0}}\right),\]
	and let \[\mathcal{E}'=\left(\mathcal{E}_{\lf},\mathcal{E}_{\rf},\mathcal{E}_{\ff_{0}}\right).\]
	Assume that $\left[\mathcal{E}_{\ff_{0}}\right]=0$. Given $P\in\Psi_{0b}^{-\infty,\mathcal{E}}\left(X\right)$
	(resp.\ $P\in\Psi_{0}^{-\infty,\mathcal{E}'}\left(X\right)$), its
	\emph{normal family} is the restriction
	\[
	N\left(P\right)=\kappa_{P|\ff_{0}},
	\]
	where $\kappa_{P}$ is the lift of the Schwartz kernel $K_{P}$ of
	$P$ from $X^{2}$ to $X_{0}^{2}$ (resp.\ $X_{0b}^{2}$). Given
	$p\in\partial X$, the \emph{normal operator} of $P$ at $p$ is the
	restriction of $\kappa_{P}$ to the fiber $\ff_{0|\left(p,p\right)}$.
\end{defn}

\begin{defn}
	(Normal family for $0$-trace and $0$-Poisson operators) Let $\mathcal{E}=\left(\mathcal{E}_{\of},\mathcal{E}_{\ff}\right)$,
	and assume that $\left[\mathcal{E}_{\ff}\right]=0$. Given $P\in\Psi_{0\tr}^{-\infty,\mathcal{E}}\left(X,\partial X\right)$
	(resp.\ $P\in\Psi_{0\po}^{-\infty,\mathcal{E}}\left(\partial X,X\right)$),
	its \emph{normal family} is the restriction
	\[
	N\left(P\right)=\kappa_{A|\ff}
	\]
	where $\kappa_{P}$ is the lift of the Schwartz kernel $K_{P}$ of
	$P$ from $\partial X\times X$ (resp.\ $X\times\partial X$) to
	$\left(\partial X\times X\right)_{0}$ (resp.\ $\left(X\times\partial X\right)_{0}$).
	Given $p\in\partial X$, the \emph{normal operator} of $P$ at $p$
	is the restriction of $\kappa_{P}$ to the fiber $\ff_{|\left(p,p\right)}$.
\end{defn}

These definitions are manifestly invariant; however, they do not help
recognizing $N_{p}\left(P\right)$ as an operator. In order to explain
this interpretation, we resort to the local characterizations given
above. We refer to \cite{MazzeoPhD, MazzeoEdgeI, MazzeoEdgeII, UsulaPhD}
for more details.

For every point $p\in\partial X$, denote by $X_{p}$ the inward-pointing
closed half of $T_{p}X$. A choice of coordinates $\left(x,y\right)$
for $X$ centered at $p$ determines global linear coordinates on
$X_{p}$, which we denote by $\left(x,y\right)$ again with slight
abuse of notation, and which identify $X_{p}$ with the closed half-space
$\mathbb{R}_{1}^{n+1}$. Define polar coordinates around $x=\tilde{x}=0,y=\tilde{y}$,
namely
\begin{align*}
	r & =\sqrt{x^{2}+\tilde{x}^{2}+\left|y-\tilde{y}\right|^{2}}\\
	\theta_{0} & =\frac{x}{r},\theta=\frac{y-\tilde{y}}{r},\theta_{n+1}=\frac{\tilde{x}}{r}.
\end{align*}
These coordinates parametrize $X_{0}^{2}$ near $\ff_{0|\left(p,p\right)}$;
in particular, they determine a diffeomorphism between a neighborhood
of $\ff_{0|\left(p,p\right)}$ in $X_{0}^{2}$ and a neighborhood
of the locus $r=0$ in the local model $\mathbb{R}^{n}\times\mathbb{R}_{1}^{1}\times S_{2}^{n+1}$
with coordinates $y,r,\theta_{0},\theta,\theta_{n+1}$. Here, as above,
\[
S_{2}^{n+1}=\left\{ \left(\theta_{0},\theta,\theta_{n+1}\right)\in\mathbb{R}^{n+2}:\theta_{0}^{2}+\left|\theta\right|^{2}+\theta_{n+1}^{2}=1,\theta_{0}\geq0,\theta_{n+1}\geq0\right\} .
\]
In these coordinates, $\ff_{0|\left(p,p\right)}$ corresponds to the
locus $\left\{ y=0,r=0\right\} $ and is therefore identified with
$S_{2}^{n+1}$ and parametrized by the coordinates $\left(\theta_{0},\theta,\theta_{n+1}\right)$.
The left face is instead defined by the equation $\theta_{0}=0$,
while the right face is defined by the equation $\theta_{n+1}=0$.

If $P\in\Psi_{0}^{-\infty,\mathcal{E}}\left(X\right)$, we can write
its Schwartz kernel as $K\left(y;x,\tilde{x},y-\tilde{y}\right)d\tilde{x}d\tilde{y}$  in coordinates near $\left(p,p\right)$,
and by definition of $\Psi_{0}^{-\infty,\mathcal{E}}\left(X\right)$
we can find a polyhomogeneous function $K'\left(y;r,\theta_{0},\theta,\theta_{n+1}\right)$
on $\mathbb{R}^{n}\times\mathbb{R}_{1}^{1}\times S_{2}^{n+1}$ such
that
\[
K\left(y;x,\tilde{x},y-\tilde{y}\right)d\tilde{x}d\tilde{y}\equiv K'\left(y;r,\frac{x}{r},\frac{y-\tilde{y}}{r},\frac{\tilde{x}}{r}\right)\frac{d\tilde{x}d\tilde{y}}{\tilde{x}^{n+1}}
\]
near the origin. More precisely, since the singular factor $\tilde{x}^{-n-1}$
has index set $-n-1$ both at the right face and at the front face,
the function $K'\left(y;r,\theta_{0},\theta,\theta_{n+1}\right)$
is polyhomogeneous on $\mathbb{R}^{n}\times\mathbb{R}_{1}^{1}\times S_{2}^{n+1}$
with index sets:
\begin{enumerate}
	\item $\mathcal{E}_{\lf}$ at $\left\{ \theta_{0}=0\right\} $;
	\item $\mathcal{E}_{\rf}+n+1$ at $\left\{ \theta_{n+1}=0\right\} $;
	\item $\mathcal{E}_{\ff_{0}}$ at $\left\{ r=0\right\} $.
\end{enumerate}
In particular, from $\left[\mathcal{E}_{\ff_{0}}\right]=0$, we know
that the evaluation \[K_{0}'\left(\theta_{0},\theta,\theta_{n+1}\right):=K'\left(0;0,\theta_{0},\theta,\theta_{n+1}\right)\]
at $y=0$ and $r=0$ is well defined and in $\mathcal{A}_{\phg}^{\left(\mathcal{E}_{\lf},\mathcal{E}_{\rf}+n+1\right)}\left(S_{2}^{n+1}\right)$.
The function $K_{0}'\left(\theta_{0},\theta,\theta_{n+1}\right)$
is precisely the expression in the coordinates $\left(\theta_{0},\theta,\theta_{n+1}\right)$
of the restriction of $\kappa_{P}$ to $\ff_{0|\left(p,p\right)}$
(i.e.\ $N_{p}\left(P\right)$) with the singular density factor removed.
We interpret $N_{p}\left(P\right)$ as an operator on functions over
$\mathbb{R}_{1}^{n+1}$, as follows:
\begin{align*}
	N_{p}\left(P\right) & =K'_{0}\left(\frac{x}{r},\frac{y-\tilde{y}}{r},\frac{\tilde{x}}{r}\right)\frac{d\tilde{x}d\tilde{y}}{\tilde{x}^{n+1}}.
\end{align*}
It is proved in \cite{MazzeoPhD} that this definition does not depend
on the choice of coordinates, and determines an operator on functions
over $X_{p}$, invariant under positive dilations and translations
by vectors in $T_{p}\partial X$.

Instead of using polar coordinates, it is convenient to introduce
``left'' projective coordinates
\[
s=\frac{x}{\tilde{x}},u=\frac{y-\tilde{y}}{\tilde{x}}
\]
for $X_{0}^{2}$ defined near $\left(p,p\right)$ and away from the
right face. These coordinates restrict to coordinates on $\ff_{0|\left(p,p\right)}$
away from the intersection $\ff_{0|\left(p,p\right)}\cap\rf$. More
precisely, we have on $\ff_{0|\left(p,p\right)}\cap\rf$
\[
s=\frac{\theta_{0}}{\theta_{n+1}},u=\frac{\theta}{\theta_{n+1}}.
\]
This change of coordinates extends to a diffeomorphism
\begin{align*}
	S_{2}^{n+1} & \to\overline{\mathbb{R}_{1}^{n+1}}\\
	\left(\theta_{0},\theta,\theta_{n+1}\right) & \mapsto\left(\frac{\theta_{0}}{\theta_{n+1}},\frac{\theta}{\theta_{n+1}}\right),
\end{align*}
where $\overline{\mathbb{R}_{1}^{n+1}}$ is the radial compactification
of the half-space $\mathbb{R}_{1}^{n+1}$. Thus, the coordinates $\left(s,u\right)$
parametrize $\ff_{0|\left(p,p\right)}$ as the radial compactification
$\overline{\mathbb{R}_{1}^{n+1}}$. It follows that we can write
\[
N_{p}\left(P\right)=K_{0}''\left(\frac{x}{\tilde{x}},\frac{y-\tilde{y}}{\tilde{x}}\right)\frac{d\tilde{x}d\tilde{y}}{\tilde{x}^{n+1}},
\]
where $K_{0}''\left(s,u\right)\in\mathcal{A}_{\phg}^{\left(\mathcal{E}_{\lf},\mathcal{E}_{\rf}+n+1\right)}\left(\overline{\mathbb{R}_{1}^{n+1}}\right)$.

If $P\in\Psi_{0b}^{-\infty,\left(\mathcal{E}_{\lf},\mathcal{E}_{\rf},\mathcal{E}_{\ff_{b}},\mathcal{E}_{\ff_{0}}\right)}\left(X\right)$,
the construction is exactly as above, except that the function $K'_{0}\left(\theta_{0},\theta,\theta_{n+1}\right)$
is polyhomogeneous on $S_{2}^{n+1}$ blown up at the corner, and correspondingly
$K_{0}''\left(s,u\right)$ is polyhomogeneous on $\overline{\mathbb{R}_{1}^{1}}\times\overline{\mathbb{R}^{n}}$.
The index sets are $\mathcal{E}_{\lf}$ at $\left\{ s=0\right\} $,
$\mathcal{E}_{\rf}+n+1$ at $\left\{ s=\infty\right\} $, and $\mathcal{E}_{\ff_{b}}+n$
at $\left\{ \left|u\right|=\infty\right\} $.

In the $0$-trace and $0$-Poisson cases, the definitions are similar.
Given $A\in\Psi_{0\tr}^{-\infty,\mathcal{E}}\left(X,\partial X\right)$
with $\left[\mathcal{E}_{\ff}\right]=0$, choose a point $p\in\partial X$
and coordinates $\left(x,y\right)$ for $X$ centered at $p$. Using
polar coordinates, we can write the Schwartz kernel of $A$ near $\left(p,p\right)$
as
\[
K\left(y;\tilde{x},y-\tilde{y}\right)d\tilde{x}d\tilde{y}=K'\left(y;r,\frac{y-\tilde{y}}{r},\frac{\tilde{x}}{r}\right)\frac{d\tilde{x}d\tilde{y}}{\tilde{x}^{n+1}}
\]
for some function $K'\left(y;r,\theta,\theta_{n+1}\right)$ polyhomogeneous
on $\mathbb{R}^{n}\times\mathbb{R}_{1}^{1}\times\left\{ S_{2}^{n+1}\cap\left\{ \theta_{0}=0\right\} \right\} $
with index sets $\mathcal{E}_{\ff}$ at $\left\{ r=0\right\} $ and
$\mathcal{E}_{\of}+n+1$ at $\left\{ \theta_{n+1}=0\right\} $. Of
course, the intersection $S_{2}^{n+1}\cap\left\{ \theta_{0}=0\right\} $
is canonically identified with the upper half-sphere
\[
S_{1}^{n}=\left\{ \left(\theta,\theta_{n+1}\right)\in\mathbb{R}^{n+1}:\left|\theta\right|^{2}+\theta_{n+1}^{2}=1,\quad\theta_{n+1}^{2}\geq0\right\} .
\]
Since $\left[\mathcal{E}_{\ff}\right]=0$, the restriction $K'_{0}\left(\theta\right):=K'\left(0;0,\theta,\sqrt{1-\left|\theta\right|^{2}}\right)$
is a well-defined polyhomogeneous function in $\mathcal{A}_{\phg}^{\mathcal{E}_{\of}+n+1}\left(S_{1}^{n}\right)$.
Again, the projective coordinate
\[
u=\frac{y-\tilde{y}}{\tilde{x}}=\frac{\theta}{\theta_{n+1}}=\frac{\theta}{\sqrt{1-\left|\theta\right|^{2}}}
\]
allows us to identify $S_{1}^{n}$ with the radial compactification
$\overline{\mathbb{R}^{n+1}}$, and we can write
\begin{align*}
	N_{p}\left(A\right) & =K'\left(0;0,\frac{y-\tilde{y}}{r},\frac{\tilde{x}}{r}\right)\frac{d\tilde{x}d\tilde{y}}{\tilde{x}^{n+1}}\\
	& =K'_{0}\left(\frac{y-\tilde{y}}{r}\right)\frac{d\tilde{x}d\tilde{y}}{\tilde{x}^{n+1}}\\
	& =K_{0}''\left(\frac{y-\tilde{y}}{\tilde{x}}\right)\frac{d\tilde{x}d\tilde{y}}{\tilde{x}^{n+1}}
\end{align*}
with $K_{0}''\left(u\right)\in\mathcal{A}_{\phg}^{\mathcal{E}_{\of}+n+1}\left(\overline{\mathbb{R}^{n}}\right)$.
This allows us to interpret $N_{p}\left(A\right)$ as a translation
and dilation invariant operator from functions on $X_{p}$ to functions
on $\partial X_{p}=T_{p}\partial X$.

Similarly, given $B\in\Psi_{0\po}^{-\infty,\mathcal{E}}\left(\partial X,X\right)$
with $\left[\mathcal{E}_{\ff}\right]=0$, choose a point $p\in\partial X$
and coordinates $\left(x,y\right)$ for $X$ centered at $p$. Write
again the Schwartz kernel of $B$ near $\left(p,p\right)$ in polar
coordinates as
\[
K\left(y;x,y-\tilde{y}\right)=K'\left(y;r,\frac{x}{r},\frac{y-\tilde{y}}{r}\right)x^{-n}d\tilde{y},
\]
where $K'\left(y;r,\theta_{0},\theta\right)$ is polyhomogeneous on
$\mathbb{R}^{n}\times\mathbb{R}_{1}^{1}\times\left\{ S_{2}^{n+1}\cap\left\{ \theta_{n+1}=0\right\} \right\} $
with index sets $\mathcal{E}_{\ff}$ at $\left\{ r=0\right\} $ and
$\mathcal{E}_{\of}+n$ at $\left\{ \theta_{0}=0\right\} $. Since
$\left[\mathcal{E}_{\ff}\right]=0$ we have a well-defined restriction
$K'_{0}\left(\theta\right):=K'\left(0;0,\sqrt{1-\left|\theta\right|^{2}},\theta\right)\in\mathcal{A}_{\phg}^{\mathcal{E}_{\of}+n}\left(S_{1}^{n}\right)$,
where now $S_{1}^{n}$ is the upper half-sphere
\[
S_{1}^{n}=\left\{ \left(\theta_{0},\theta\right)\in\mathbb{R}^{n+1}:\theta_{0}^{2}+\left|\theta\right|^{2}=1,\quad\theta_{0}\geq0\right\} .
\]
Introducing ``right'' projective coordinates
\[
v=\frac{y-\tilde{y}}{x},
\]
we get a diffeomorphism $S_{1}^{n}\to\overline{\mathbb{R}^{n}}$ given
by $\theta\mapsto\frac{\theta}{\sqrt{1-\left|\theta\right|^{2}}}$,
and we have
\begin{align*}
	N_{p}\left(B\right) & =K'\left(0;0,\frac{x}{r},\frac{y-\tilde{y}}{r}\right)\frac{d\tilde{y}}{x^{n}}\\
	& =K'_{0}\left(\frac{y-\tilde{y}}{r}\right)\frac{d\tilde{y}}{x^{n}}\\
	& =K''_{0}\left(\frac{y-\tilde{y}}{x}\right)\frac{d\tilde{y}}{x^{n}}.
\end{align*}
This allows us to interpret $N_{p}\left(B\right)$ as a dilation and
translation invariant operator from functions on $\partial X_{p}=T_{p}\partial X$
to functions on $X_{p}$.

\subsubsection{\label{subsubsec:The-Bessel-family-physical}The Bessel family}

The Bessel family of an operator in any of the classes $\Psi_{0}^{-\infty,\bullet}\left(X\right)$,
$\Psi_{0b}^{-\infty,\bullet}\left(X\right)$, $\Psi_{0\tr}^{-\infty,\bullet}\left(X,\partial X\right)$,
$\Psi_{0\po}^{-\infty,\bullet}\left(\partial X,X\right)$, is a convenient
equivalent way to express the normal family we just described. Unlike
the normal family, the Bessel family is not completely invariant;
rather, it depends on the auxiliary choice of a vector field $V$
on $X$ defined along $\partial X$, transverse to $\partial X$ and
inward-pointing. Let's fix such a vector field. Then, for every point
$p\in\partial X$, the vector $V_{p}$ represents a preferred point
in the interior of the model space $X_{p}$ described above. This
identifies $X_{p}$ with the product $T_{p}\partial X\times\mathbb{R}_{1}^{1}$.
Moreover, $V$ determines a global smooth section of the inward-pointing
normal bundle $N^{+}\partial X$, which then gets identified with
the trivial fiber bundle $\partial X\times\mathbb{R}_{1}^{1}$.

Consider an operator $P$ in $\Psi_{0}^{-\infty,\left(\mathcal{E}_{\lf},\mathcal{E}_{\rf},\mathcal{E}_{\ff_{0}}\right)}\left(X\right)$
or in $\Psi_{0b}^{-\infty,\left(\mathcal{E}_{\lf},\mathcal{E}_{\rf},\mathcal{E}_{\ff_{b}},\mathcal{E}_{\ff_{0}}\right)}\left(X\right)$,
with $\left[\mathcal{E}_{\ff_{0}}\right]=0$, and a point $p\in\partial X$.
We have seen above that the normal operator $N_{p}\left(P\right)$
acts on functions over $X_{p}$, and is invariant under translations
by elements of $T_{p}\partial X$ and positive dilations. Thanks to
the decomposition $X_{p}=T_{p}\partial X\times\mathbb{R}_{1}^{1}$,
we can conjugate $P$ by the Fourier transform on functions over $T_{p}\partial X$,
and the result can be interpreted as an operator on functions\footnote{Strictly speaking, the invariant Fourier transform on a vector space
	$V$ transforms functions over $V$ to density-valued functions over
	the dual $V^{*}$; however, since we are conjugating an operator,
	we can invariantly factor out a common linear density in the domain
	and in the codomain, and the result is invariant under this choice.} over $T_{p}^{*}\partial X\times\mathbb{R}_{1}^{1}$. The translation
invariance of $N_{p}\left(P\right)$ implies that the result depends
only parametrically on the covector. More precisely, the conjugation
is a smooth family $\eta\mapsto\hat{N}_{\eta}\left(P\right)$, where
$\eta$ ranges in $T_{p}^{*}\partial X\backslash0$ and $\hat{N}_{\eta}\left(P\right)$
is an operator on functions over $\mathbb{R}_{1}^{1}$. We call $\hat{N}_{\eta}\left(P\right)$
the \emph{Bessel operator} of $P$ at $\eta$. A different choice
of the vector field $V$ only changes $\hat{N}_{\eta}\left(P\right)$
by a smooth, nowhere zero factor.

Let's express $\hat{N}_{\eta}\left(P\right)$ in coordinates. Let
$x,y$ be coordinates for $X$ centered at $p$ and compatible with
$V$, in the sense that $Vx\equiv1$ along the domain of definition
of the chart. As explained above, we have
\[
N_{p}\left(P\right)=k\left(\frac{x}{\tilde{x}},\frac{y-\tilde{y}}{\tilde{x}}\right)\frac{d\tilde{x}d\tilde{y}}{\tilde{x}^{n+1}},
\]
where
\[
k\left(s,u\right)\in\begin{cases}
	\mathcal{A}_{\phg}^{\left(\mathcal{E}_{\lf},\mathcal{E}_{\rf}+n+1\right)}\left(\overline{\mathbb{R}_{1}^{n+1}}\right) & \text{ if \ensuremath{P\in\Psi_{0}^{-\infty,\left(\mathcal{E}_{\lf},\mathcal{E}_{\rf},\mathcal{E}_{\ff_{0}}\right)}\left(X\right)}}\\
	\mathcal{A}_{\phg}^{\left(\mathcal{E}_{\lf},\mathcal{E}_{\rf}+n+1,\mathcal{E}_{\ff_{b}}+n\right)}\left(\overline{\mathbb{R}_{1}^{1}}\times\overline{\mathbb{R}^{n}}\right) & \text{ if \ensuremath{P\in\Psi_{0b}^{-\infty,\left(\mathcal{E}_{\lf},\mathcal{E}_{\rf},\mathcal{E}_{\ff_{b}},\mathcal{E}_{\ff_{0}}\right)}\left(X\right)}}
\end{cases}.
\]
A computation shows that the Bessel family is then
\[
\hat{N}_{\eta}\left(P\right)=\hat{k}\left(\frac{x}{\tilde{x}},\tilde{x}\eta\right)\frac{d\tilde{x}}{\tilde{x}},
\]
where $\hat{k}\left(s,\eta\right)$ is the Fourier transform of $k\left(s,u\right)$
in the $u$ variables. This expression implies immediately that $\hat{N}_{\eta}\left(P\right)$
is not dilation invariant anymore. Rather, it is ``homogeneous of
degree $0$'' in $\eta$ in the following sense. Given $t>0$, denote
by $\lambda_{t}$ the ``dilation by $t$'' action on $N_{p}^{+}\partial X$.
Then we have
\[
\lambda_{t}^{*}\circ\hat{N}_{\eta}\left(P\right)\circ\lambda_{t^{-1}}^{*}=\hat{N}_{t\eta}\left(P\right).
\]
We refer to \cite{MazzeoPhD, MazzeoEdgeI, Lauter, Hintz0calculus}
for more details.

For $0$-trace and $0$-Poisson operators, the construction is similar.
If $A\in\Psi_{0\tr}^{-\infty,\mathcal{E}}\left(X,\partial X\right)$
with $\left[\mathcal{E}_{\ff}\right]=0$, as explained above we can
write the normal operator $A$ at $p$ with respect to coordinates
$\left(x,y\right)$ with $Vx\equiv1$ near $p$, as
\[
N_{p}\left(A\right)=k\left(\frac{y-\tilde{y}}{\tilde{x}}\right)\frac{d\tilde{x}d\tilde{y}}{\tilde{x}^{n+1}}
\]
for some $k\left(u\right)\in\mathcal{A}_{\phg}^{\mathcal{E}_{\of}+n+1}\left(\overline{\mathbb{R}^{n}}\right)$.
The Bessel family at $p$ in these coordinates is
\[
\hat{N}_{\eta}\left(A\right)=\hat{k}\left(\tilde{x}\eta\right)\frac{d\tilde{x}}{\tilde{x}}.
\]
Invariantly, $\eta\mapsto\hat{N}_{\eta}\left(A\right)$ can be interpreted
as a family of operators from functions over $N_{p}^{+}\partial X$
to $\mathbb{C}$, smoothly parametrized by $\eta\in T_{p}^{*}\partial X\backslash0$.
Again, the dependence on the choice of the vector field $V$ is only
up to a nowhere zero phase factor. This time, the dilation invariance
of $N_{p}\left(A\right)$ translates into the homogeneity property
\[
\hat{N}_{\eta}\left(A\right)\circ\lambda_{t^{-1}}^{*}=\hat{N}_{t\eta}\left(A\right).
\]
Similarly, if $B\in\Psi_{0\po}^{-\infty,\mathcal{E}}\left(X,\partial X\right)$
with $\left[\mathcal{E}_{\ff}\right]=0$, we can write its normal
operator $N_{p}\left(B\right)$ in coordinates as
\[
N_{p}\left(B\right)=k\left(\frac{y-\tilde{y}}{x}\right)\frac{d\tilde{y}}{x^{n}}
\]
for some $k\left(v\right)\in\mathcal{A}_{\phg}^{\mathcal{E}_{\of}+n}\left(\overline{\mathbb{R}^{n}}\right)$,
and the corresponding Bessel family at $p$ is simply
\[
\hat{N}_{\eta}\left(B\right)=\hat{k}\left(x\eta\right).
\]
We can think of $\eta\mapsto\hat{N}_{\eta}\left(B\right)$ invariantly
as a family of operators from $\mathbb{C}$ to functions over $N_{p}^{+}\partial X$,
smoothly parametrized by $\eta\in T_{p}^{*}\partial X\backslash0$.
The dilation invariance of $N_{p}\left(B\right)$ translates into
the homogeneity property
\[
\lambda_{t}^{*}\circ\hat{N}_{\eta}\left(B\right)=\hat{N}_{t\eta}\left(B\right).
\]
We refer to \S2.6.2 and \S2.7.2 of \cite{UsulaPhD} for more details.

\subsection{\label{subsec:0-elliptic-operators-and-parametrices}Fully \texorpdfstring{$0$}{0}-elliptic
	operators and parametrices}

Mazzeo and Melrose developed the $0$-calculus
to construct parametrices for the class of
\emph{fully $0$-elliptic $0$-differential operators}.
\begin{defn}
	($0$-differential operators) A \emph{$0$-differential operator}
	on $X$ is a differential operator $L$ which, in coordinates $\left(x,y\right)$
	near any point $p\in\partial X$, can be written in the form
	\[
	L=\sum_{j+\left|\alpha\right|\leq m}L_{j,\alpha}\left(x,y\right)\left(x\partial_{x}\right)^{j}\left(x\partial_{y}\right)^{\alpha}
	\]
	where $j\in\mathbb{N},\alpha\in\mathbb{N}^{n}$, and $L_{j,\alpha}\left(x,y\right)$
	are locally defined smooth functions. We write $\Diff_{0}^{m}\left(X\right)$ for
	the space of $0$-differential operators of order at most $m$.
\end{defn}

Invariantly, a $0$-differential operator can be described as a $C^{\infty}\left(X\right)$
linear combination of compositions of vector fields on $X$ which
vanish along the boundary. These vector fields are called\emph{ $0$-vector
	fields}, and they form a Lie sub-algebra $\mathcal{V}_{0}\left(X\right)$
of the Lie algebra of vector fields on $X$. The object $\mathcal{V}_{0}\left(X\right)$
is central in the theory. Via the Serre--Swan Theorem, $\mathcal{V}_{0}\left(X\right)$
determines uniquely a smooth vector bundle $^{0}TX\to X$, the \emph{$0$-tangent
	bundle}, whose module of sections is $\mathcal{V}_{0}\left(X\right)$.
The inclusion $\mathcal{V}_{0}\left(X\right)\to\mathcal{V}\left(X\right)$
into the module of vector fields on $X$ determines a bundle map $\#:{}^{0}TX\to TX$,
which is an isomorphism at each point $p\in X^{\circ}$ since $0$-vector
fields span $T_{p}X^{\circ}$ pointwise. However, $\#$ vanishes identically
along $\partial X$. The dual of $^{0}TX$ is denoted by $^{0}T^{*}X$,
and can be thought of as the bundle whose smooth sections are the
singular $1$-forms of the form $x^{-1}\omega$, where $\omega$ is
a smooth $1$-form on $X$ and $x$ is a boundary defining function
on $X$.

By construction, $0$-differential operators of order $m\geq1$ can
never be elliptic in the standard sense: due to the degeneracy of
$0$-vector fields along $\partial X$, the principal symbol $\sigma\left(L\right)$
of an operator $L\in\Diff_{0}^{m}\left(X\right)$ (thought of as a
smooth function on $T^{*}X$, fibrewise homogeneous of degree $m$)
always vanishes to order $m$ along $\partial X$. The correct notion
of ellipticity involves extending $\sigma\left(L\right)$ from the
interior (i.e.\ from $T^{*}X^{\circ}\equiv{^{0}T^{*}X_{|X^{\circ}}}$)
to a smooth function $^{0}\sigma\left(L\right)$ on $^{0}T^{*}X$,
fibrewise homogeneous of degree $m$.
\begin{defn}
	The \emph{principal $0$-symbol} of an operator $L\in\Diff_{0}^{m}\left(X\right)$
	is the unique smooth extension to $^{0}T^{*}X$ of the principal symbol
	$\sigma\left(L\right)$ restricted to $T^{*}X^{\circ}$. $L$ is called
	\emph{$0$-elliptic} if $^{0}\sigma\left(L\right)$ is nowhere zero
	on $^{0}T^{*}X\backslash0$.
\end{defn}

The small $0$-calculus $\Psi_{0}^{\bullet}\left(X\right)$ recalled
above was developed precisely with the aim of inverting $0$-elliptic
operators modulo smoothing operators. Indeed, it is easy to check
that $\Diff_{0}^{m}\left(X\right)\subseteq\Psi_{0}^{m}\left(X\right)$,
and using a variant of the standard Hadamard parametrix construction,
it is not hard (with the small $0$-calculus at hand) to prove the
following
\begin{thm}
	(Mazzeo--Melrose) Let $L\in\Diff_{0}^{m}\left(X\right)$ be $0$-elliptic.
	Then there exists an operator $Q\in\Psi_{0}^{-m}\left(X\right)$ such
	that $LQ\equiv QL\equiv I\mod\Psi_{0}^{-\infty}\left(X\right)$.
\end{thm}

The result above is sufficient for properties such as elliptic regularity
on appropriate Sobolev spaces associated to $\mathcal{V}_{0}\left(X\right)$.
More precisely, given $k\in\mathbb{N}$ and $\delta\in\mathbb{R}$,
denote by $x^{\delta}H_{0}^{k}\left(X\right)$ the space of functions
$u\in x^{\delta}L_{b}^{2}\left(X\right)$\footnote{The space $L_{b}^{2}\left(X\right)$ is the space of functions on
	$X^{\circ}$ with finite $L^{2}$ norm with respect to a singular
	positive density of the form $x^{-1}\omega$, where $\omega$ is a
	smooth positive density on $X$ and $x$ is a boundary defining function.
	The choice of $L_{b}^{2}\left(X\right)$ as our base space for $0$-Sobolev
	spaces is purely conventional: with this choice, $o\left(x^{\delta}\right)$
	conormal functions are in $x^{\delta}L_{b}^{2}\left(X\right)$.} for which $V_{1}\cdots V_{l}u\in x^{\delta}L_{b}^{2}\left(X\right)$
for every $V_{1},...,V_{l}\in\mathcal{V}_{0}\left(X\right)$ and $l\leq k$.
As usual, $x^{\delta}H_{0}^{k}\left(X\right)$ is a Banach space with
a norm induced by a choice of a boundary defining function $x$, a
smooth positive density $\omega$ on $X$, and a finite subset of
$\mathcal{V}_{0}\left(X\right)$ which generates $\mathcal{V}_{0}\left(X\right)$
pointwise. The induced Banach topology is intrinsic of $X$. If $Q\in\Psi_{0}^{m}\left(X\right)$
with $m\in\mathbb{Z}$, then $Q$ induces for every $\delta\in\mathbb{R}$
a continuous linear map
\[
Q:x^{\delta}H_{0}^{k+m}\left(X\right)\to x^{\delta}H_{0}^{k}\left(X\right).
\]
Consequently, if $L\in\Diff_{0}^{m}\left(X\right)$ is $0$-elliptic,
for every $\delta$ and $k$ we have $0$-elliptic estimates of the
form
\[
\left|\left|u\right|\right|_{x^{\delta}H_{0}^{k+m}}\leq C\left(\left|\left|u\right|\right|_{x^{\delta}L_{b}^{2}}+\left|\left|Lu\right|\right|_{x^{\delta}H_{0}^{k}}\right).
\]
However, this is not enough to guarantee that the induced map
\[
L:x^{\delta}H_{0}^{k+m}\left(X\right)\to x^{\delta}H_{0}^{k}\left(X\right)
\]
is \emph{Fredholm}. This is essentially due to the fact that the residual
operators in $\Psi_{0}^{-\infty}\left(X\right)$ are in general \emph{non-compact}
as operators on $x^{\delta}L_{b}^{2}\left(X\right)$.

In order to invert $L$ modulo compact remainders, one needs to be
able to invert every Bessel operator of $L$ on the appropriate weighted
space. Since $L\in\Psi_{0}^{m}\left(X\right)$, one can define $N\left(L\right)$
as in \S\ref{subsec:Model-families} as the restriction of $\kappa_{L}=\beta_{0}^{*}K_{L}$
to the $0$-front face $\ff_{0}$ of $X_{0}^{2}$. Its interpretation
as a family of operators is very easy to understand: if in local coordinates
$\left(x,y\right)$ centered at a point $p\in\partial X$ $L$ takes
the local form
\[
L=\sum_{j+\left|\alpha\right|\leq m}L_{j,\alpha}\left(x,y\right)\left(x\partial_{x}\right)^{j}\left(x\partial_{y}\right)^{\alpha},
\]
then the normal operator at $p$ in the induced global linear coordinates
on $X_{p}$ (still denoted by $\left(x,y\right)$ with slight abuse
of notation) is simply
\[
N_{p}\left(L\right)=\sum_{j+\left|\alpha\right|\leq m}L_{j,\alpha}\left(0,0\right)\left(x\partial_{x}\right)^{j}\left(x\partial_{y}\right)^{\alpha}.
\]
Choosing a vector field $V$ on $X$ transverse to $\partial X$ and
inward-pointing, and assuming that $Vx\equiv1$ near $p$, the Bessel
operator of $L$ (with respect to the auxiliary choice of $V$) at
a covector $\eta\in T_{p}^{*}\partial X$ is then
\[
\hat{N}_{\eta}\left(L\right)=\sum_{j+\left|\alpha\right|\leq m}L_{j,\alpha}\left(0,0\right)\left(x\partial_{x}\right)^{j}\left(ix\eta\right)^{\alpha}.
\]
We remark that $\hat{N}_{\eta}\left(L\right)$ should be interpreted
as an operator on functions over $N_{p}^{+}\partial X$. The evaluation
of $\hat{N}_{\eta}\left(L\right)$ at $\eta=0$ determines a simpler
operator which depends only on $p$, which is called the \emph{indicial
}operator of $L$ at $p$:
\[
I_{p}\left(L\right)=\sum_{j\leq m}L_{j,0}\left(0,0\right)\left(x\partial_{x}\right)^{j}.
\]
$I_{p}\left(L\right)$ can be invariantly interpreted as a \emph{dilation
	invariant} operator on functions over $N_{p}^{+}\partial X$. Its
fibrewise Mellin transform yields a \emph{polynomial}
\[
I_{p}^{M}\left(L\right)=\sum_{j\leq m}L_{j,0}\left(0,0\right)z^{j},
\]
which is called the \emph{indicial polynomial} of $L$ at $p$.
\begin{defn}
	\label{def:indicial-root}An \emph{indicial root} for $L$ at $p$
	is a root of $I_{p}^{M}\left(L\right)$. Equivalently, it is a complex
	number $\mu$ for which $x^{\mu}$ is in the kernel of $I_{p}\left(L\right)$.
	The \emph{multiplicity} of $\mu$ is the algebraic multiplicity of
	$\mu$ as a root of $I_{p}^{M}\left(L\right)$. Equivalently, the
	multiplicity of $\mu$ is $M_{\mu}+1$, where $M_{\mu}$ is the maximal
	integer such that $x^{\mu}\left(\log x\right)^{M_{\mu}}$ is in the
	kernel of $I_{p}\left(L\right)$. An indicial root is \emph{simple}
	if it has multiplicity $1$. Two distinct indicial roots $\mu_{1},\mu_{2}$
	are \emph{resonant} if they differ by an integer.
\end{defn}

In general, the indicial roots of $L$ depend on the point $p\in\partial X$.
However, this does not happen for most geometrically interesting operators,
so from now on \emph{we shall assume that the $0$-elliptic operators
	considered in this paper have constant indicial roots}. We denote
by $\spec_{b}\left(L\right)$ the set of indicial roots of $L$, and
by $\widetilde{\spec}_{b}\left(L\right)$ the set of pairs $\left(\mu,M_{\mu}\right)$
where $\mu\in\spec_{b}\left(L\right)$.
\begin{rem}
	\label{rem:rescaled-x^mL}The assumption of constant indicial roots
	is quite common in the literature, because most geometrically relevant
	examples satisfy this property. For example, geometric operators on
	conformally compact manifolds have constant indicial roots (cf.\ \S\ref{subsubsec:Conformally-compact-geometry}).
	Another important class of examples consists of operators in $x^{m}\Diff^{m}\left(X\right)$.
	If $\tilde{L}\in\Diff^{m}\left(X\right)$ is elliptic in the standard
	sense, then $L=x^{m}\tilde{L}$ is $0$-elliptic, and has constant
	and simple indicial roots $0,...,m-1$.
\end{rem}

A weight $\delta\in\mathbb{R}$ is called \emph{non-indicial} if $\delta\not\in\Re\left(\spec_{b}\left(L\right)\right)$.
It is proved in \cite{MazzeoEdgeI} that, if $L\in\Diff_{0}^{m}\left(X\right)$
is $0$-elliptic and $\delta$ is non-indicial, then the Bessel operators
\begin{equation}
	\hat{N}_{\eta}\left(L\right):x^{\delta}H_{0}^{m}\left(N_{p}^{+}\partial X\right)\to x^{\delta}L_{b}^{2}\left(N_{p}^{+}\partial X\right)\label{eq:bessel_on_weighted}
\end{equation}
are Fredholm for every $\eta\in T^{*}\partial X\backslash0$. Moreover,
the kernel and orthogonal complement of the range of $\hat{N}_{\eta}\left(L\right)$
consist of polyhomogeneous functions on the radial compactification
$\overline{N_{p}^{+}\partial X}=\left[0,+\infty\right]$, with index
set $\mathcal{E}$ with $\Re\left(\mathcal{E}\right)>\delta$ at $0$,
and infinite order of vanishing (i.e.\ Schwartz decay) at $\infty$.
\begin{defn}
	\label{def:injective-surjective-weights}A weight $\delta\in\mathbb{R}$
	is said to be \emph{injective }(resp.\ \emph{surjective}, \emph{invertible})
	for $L$ if it is non-indicial and the map \ref{eq:bessel_on_weighted}
	is injective (resp.\ surjective, invertible).
\end{defn}

If $\delta$ is an invertible weight, then one obtains good parametrices
modulo \emph{compact} operators on $x^{\delta}L_{b}^{2}\left(X\right)$.
\begin{thm}
	\label{thm:(Mazzeo=002013Melrose)Good_parametrices_fully_0-elliptic}(Mazzeo--Melrose)
	Let $L\in\Diff_{0}^{m}\left(X\right)$ be a $0$-elliptic operator with constant
	indicial roots, and let $\delta\in\mathbb{R}$ be an invertible weight
	for $L$. Then there exist operators \[A,B\in\Psi_{0}^{-m}\left(X\right)+\Psi_{0}^{-\infty,\mathcal{H}}\left(X\right),\]
	with $\mathcal{H}=\left(\mathcal{H}_{\lf},\mathcal{H}_{\rf},\mathcal{H}_{\ff_{0}}\right)$
	and
	\begin{align*}
		\Re\left(\mathcal{H}_{\lf}\right) & >\delta\\
		\Re\left(\mathcal{H}_{\rf}\right) & >-\delta-1\\
		\left[\mathcal{H}_{\ff_{0}}\right] & =0,
	\end{align*}
	such that:
	\begin{enumerate}
		\item $LA=I-R_{A}$, with $R_{A}\in\Psi^{-\infty,\left(\infty,\mathcal{H}_{\rf}\right)}\left(X\right)$;
		\item $BL=I-R_{B}$, with $R_{B}\in\Psi^{-\infty,\left(\mathcal{H}_{\lf},\infty\right)}\left(X\right)$.
	\end{enumerate}
\end{thm}

The main ingredients for the proof are:
\begin{enumerate}
	\item a detailed description of the inverse of the Bessel operators $\hat{N}_{\eta}\left(L\right)$
	as maps \[x^{\delta}H_{0}^{m}\left(N_{p}^{+}\partial X\right)\to x^{\delta}L_{b}^{2}\left(N_{p}^{+}\partial X\right);\]
	\item a \emph{composition theorem} for the $0$-calculus.
\end{enumerate}
This theorem immediately implies that if $L$ and $\delta$ are as
in the previous theorem, then $L$ is Fredholm as a map $x^{\delta}H_{0}^{k+m}\left(X\right)\to x^{\delta}H_{0}^{k}\left(X\right)$,
because the right and left parametrices $A,B\in\Psi_{0}^{-m}\left(X\right)+\Psi_{0}^{-\infty,\mathcal{H}}\left(X\right)$
are bounded as maps $x^{\delta}H_{0}^{k}\left(X\right)\to x^{\delta}H_{0}^{k+m}\left(X\right)$
and the very residual remainders $R_{A},R_{B}$ are \emph{compact}
as operators $x^{\delta}H_{0}^{k}\left(X\right)\to x^{\delta}H_{0}^{k}\left(X\right)$
for every $k$. We refer to \cite{MazzeoEdgeI, Lauter, Hintz0calculus}
for details.

\subsection{\label{subsec:The-guiding-example-=0000A72}The guiding example}

The guiding example which we shall consider in this paper is the \emph{Laplacian
	on a conformally compact manifold}.

\subsubsection{\label{subsubsec:Conformally-compact-geometry}Conformally compact
	geometry}

Conformally compact geometry is a vast area of research, pioneered
by Fefferman--Graham \cite{FeffermanGraham}, Mazzeo \cite{MazzeoPhD, MazzeoHodge}, Mazzeo and Melrose \cite{MazzeoMelroseResolvent} and others \cite{GrahamLee, LeeFredholm, Biquard_Metriques, BiquardAdS}.
A \emph{conformally compact metric} (or $0$-metric) on $X$ is a
bundle metric $g$ on $^{0}TX$. The name is due to the fact that,
calling $g^{\circ}$ the restriction of $g$ to $X^{\circ}$ and thinking
of $g^{\circ}$ as a metric on $X^{\circ}$, we obtain that for some
(hence every) boundary defining function $x$ on $X$, the conformal
rescaling $x^{2}g^{\circ}$ extends from $X^{\circ}$ to a smooth
metric on $X$. When seen as a metric on $X^{\circ}$, $g^{\circ}$
is complete with bounded geometry and asymptotically negative sectional
curvature. If $\left|dx\right|_{x^{2}g}\equiv1$ along $\partial X$,
then the sectional curvature is asymptotic to $-1$, and we say that
$g$ is \emph{asymptotically hyperbolic.}

The present work is relevant to conformally compact geometry because
\emph{natural geometric operators associated to a conformally compact
	metric are always $0$-differential}. To be more precise, let $E\to X^{\circ}$
be a geometric vector bundle on $\left(X^{\circ},g^{\circ}\right)$,
that is, a bundle associated to the $\mathrm{O}\left(n+1\right)$
principal bundle of orthogonal frames of $\left(TX^{\circ},g^{\circ}\right)$
via a representation of $\mathrm{O}\left(n+1\right)$. Then one can
use the same representation to define the ``$0$-version'' $^{0}E\to X$
of this bundle, associated to the $\mathrm{O}\left(n+1\right)$ principal
bundle of orthogonal frames of $\left(^{0}TX,g\right)$. Since $TX^{\circ}$
is canonically isomorphic to $\left(^{0}TX\right)_{|X^{\circ}}$,
we can always identify canonically $^{0}E_{|X^{\circ}}$ with $E$.
Now, if $E,F\to X^{\circ}$ are two such geometric bundles, and $L\in\Diff^{m}\left(X^{\circ};E,F\right)$
is a geometric operator, then $L$ always extends to a $0$-differential
operator $L\in\Diff_{0}^{m}\left(X;{^{0}E},{^{0}F}\right)$.

Given $p\in\partial X$, denote by $X_{p}$ the closed inward-pointing
half of $T_{p}X$. Then a $0$-metric $g$ on $X$ induces on $X_{p}$
a natural model $0$-metric $g_{p}$, invariant under positive dilations
and translations by vectors in $\partial X_{p}=T_{p}\partial X$.
If $g$ is asymptotically hyperbolic, $\left(X_{p}^{\circ},g_{p}^{\circ}\right)$
is globally isometric to the hyperbolic $n+1$-space. More explicitly,
given a boundary defining function $x$ and denoted by $h_{0}$ the
metric $\left(x^{2}g\right)_{|\partial X}$ on $\partial X$, choose
normal coordinates $y$ for $\left(\partial X,h_{0}\right)$ centered
at $p$. We call $\left(x,y\right)$ \emph{normal half-space coordinates}
for $X$ at $p$. In the global linear coordinates on $T_{p}X$ induced
by $\left(x,y\right)$, which we still call $\left(x,y\right)$ with
slight abuse of notation, $X_{p}$ is identified with the closed half-space
$\mathbb{R}_{1}^{n+1}$ and the $0$-metric $g_{p}$ is precisely
the hyperbolic metric.
\[
g_{p}=\frac{dx^{2}+dy^{2}}{x^{2}}.
\]
If $^{0}E\to X$ is a geometric vector bundle associated to $g$ and
a representation of $\mathrm{O}\left(n+1\right)$, the fiber $^{0}E_{p}$
(or more precisely, the trivial bundle $X_{p}\times{^{0}E_{p}}$)
can be canonically identified with the corresponding (trivial) geometric
vector bundle associated to $\left(X_{p},g_{p}\right)$.

Let now $L\in\Diff_{0}^{m}\left(X;{^{0}E},{^{0}F}\right)$ be a geometric
$0$-differential operator. In this case the normal operator $N_{p}\left(L\right)$,
which is by definition an element of $\Diff_{0}^{m}\left(X_{p};{^{0}E_{p}},{^{0}F_{p}}\right)$
(cf.\ \S\ref{subsec:0-elliptic-operators-and-parametrices}), is
precisely the corresponding geometric $0$-differential operator on
$\left(X_{p},g_{p}\right)$. This observation guarantees that geometric
operators on asymptotically hyperbolic manifolds always have constant
indicial roots, and reduces the analysis of the Bessel family to a
simple problem in hyperbolic space.

\subsubsection{The scalar Laplacian}

Fix an oriented asymptotically hyperbolic manifold $\left(X^{n+1},g\right)$.
We choose an auxiliary metric $h_{0}$ in the \emph{conformal infinity}
of $g$, the conformal class on $\partial X$ defined as
\[
\mathfrak{c}_{\infty}\left(g\right)=\left\{ \left(x^{2}g\right)_{|\partial X}:\text{\ensuremath{x} boundary defining function}\right\} .
\]
We also choose a boundary defining function $x$ such that $\left(x^{2}g\right)_{|\partial X}=h_{0}$,
and such that $\left|dx^{2}\right|_{x^{2}g}\equiv1$ near the boundary;
this condition determines $x$ uniquely near $\partial X$. To trivialize
the normal bundle of $\partial X$, we use the vector field $V=\partial_{x}$,
where $\partial_{x}$ is the gradient of $x$ with respect to $x^{2}g$.
Finally, as a reference positive density, we choose $\omega=x^{n+1}\omega_{g}$,
where $\omega_{g}$ is the Riemannian volume $0$-form associated
to $g$. In this section and in \S\ref{subsec:The-guiding-example-=0000A73},
\S\ref{subsec:The-guiding-example-=0000A76} and \S\ref{subsec:The-guiding-example-=0000A77},
we will use the triple $x,\partial_{x},\omega$ to compute asymptotic
expansions, induce inner products on weighted $L^{2}$ spaces, and
compute Bessel operators.

We denote by $\Delta=d^{*}d$ the Laplacian on $\left(X,g\right)$.
This is a $0$-elliptic element of $\Diff_{0}^{2}\left(X\right)$.
We will now compute its indicial roots and its normal and Bessel families.
Given $p\in\partial X$, choose normal coordinates $y$ for $\left(\partial X,h_{0}\right)$
centered at $p$. The model tangent half-space $\left(X_{p},g_{p}\right)$
is then identified, using the induced half-space coordinates on $X_{p}$,
with the upper half-space model of $\mathrm{H}^{n+1}$. The normal
operator $N_{p}\left(\Delta\right)$ is then the Laplacian of $\left(X_{p},g_{p}\right)$,
namely
\[
N_{p}\left(\Delta\right)=-\left(x\partial_{x}\right)^{2}+n\left(x\partial_{x}\right)-x^{2}\sum_{j=1}^{n}\partial_{y^{j}}^{2}.
\]
The Bessel family $\eta\mapsto\hat{N}_{\eta}\left(\Delta\right)$
at $T_{p}^{*}\partial X\equiv\mathbb{R}^{n}$ is obtained via conjugation
of $N_{p}\left(\Delta\right)$ by the Fourier transform in the $y$
variables:
\[
\hat{N}_{\eta}\left(\Delta\right)=-\left(x\partial_{x}\right)^{2}+n\left(x\partial_{x}\right)+\left|\eta\right|^{2}x^{2}.
\]
We obtain the indicial operator at $p$ by evaluating $\hat{N}_{\eta}\left(\Delta\right)$
at $\eta=0$:
\[
I_{p}\left(\Delta\right)=-\left(x\partial_{x}\right)^{2}+n\left(x\partial_{x}\right).
\]
It follows that the indicial roots of $\Delta$ at $p$ are $0$ and
$n$, both simple.

\subsubsection{Analysis of the Bessel family}

To understand the Fredholm properties of the operator $\Delta$ on weighted $0$-Sobolev
spaces, we need to study the kernel of $\hat{N}_{\eta}\left(\Delta\right)$.

The space of tempered solutions of the equation $\hat{N}_{\eta}\left(\Delta\right)u=0$
is spanned by the two fundamental solutions
\[
x^{\frac{n}{2}}I_{\frac{n}{2}}\left(\left|\eta\right|x\right),\quad x^{\frac{n}{2}}K_{\frac{n}{2}}\left(\left|\eta\right|x\right),
\]
where $I_{\nu}\left(z\right)$ and $K_{\nu}\left(z\right)$ are the
modified Bessel functions of the first and second kind, respectively.
The functions above have the following asymptotic properties:
\begin{align*}
	x^{\frac{n}{2}}I_{\frac{n}{2}}\left(\left|\eta\right|x\right) & \sim\begin{cases}
		C_{1}x^{n} & x\to0^{+}\\
		C_{2}x^{\frac{n-1}{2}}e^{\left|\eta\right|x} & x\to+\infty
	\end{cases}\\
	x^{\frac{n}{2}}K_{\frac{n}{2}}\left(\left|\eta\right|x\right) & \sim\begin{cases}
		C_{3} & x\to0^{+}\\
		C_{4}x^{\frac{n-1}{2}}e^{-\left|\eta\right|x} & x\to+\infty
	\end{cases}
\end{align*}
for some non-zero constants $C_{1},C_{2},C_{3},C_{4}$. These properties
tell us that $x^{\frac{n}{2}}I_{\frac{n}{2}}\left(\left|\eta\right|x\right)$
cannot be in \emph{any} space $x^{\delta}L_{b}^{2}\left(N_{p}^{+}\partial X\right)$,
for any $\delta\in\mathbb{R}$, because of its exponential growth
as $x\to+\infty$. On the other hand, since $x^{\frac{n}{2}}K_{\frac{n}{2}}\left(\left|\eta\right|x\right)$
decays exponentially as $x\to+\infty$, we obtain that the kernel
of $\hat{N}_{\eta}\left(\Delta\right)$ on $x^{\delta}L_{b}^{2}\left(N_{p}^{+}\partial X\right)$
is non-trivial if and only if $\delta<0$. In other words, \emph{a
	weight $\delta\in\mathbb{R}\backslash\left\{ 0,n\right\} $ is injective
	for $\Delta$ if and only if $\delta>0$}.

To understand which weights are surjective for $\Delta$, we use the
following general fact: given $L\in\Diff_{0}^{m}\left(X\right)$,
a weight $\delta$ is surjective for $L$ if and only if it is injective
for the formal adjoint $L^{\dagger}$ with respect to an inner product
on $x^{\delta}L_{b}^{2}\left(X\right)$. In our case, we obtain
\[
\Delta^{\dagger}=x^{-\left(n-2\delta\right)}\Delta x^{n-2\delta}
\]
because $\Delta$ is formally self-adjoint with respect to the inner
product on $x^{\frac{n}{2}}L_{b}^{2}\left(X\right)$.
Since $\hat{N}_{\eta}\left(\Delta^{\dagger}\right)=x^{-\left(n-2\delta\right)}\hat{N}_{\eta}\left(\Delta\right)x^{n-2\delta}$,
it then follows that \emph{a weight $\delta\in\mathbb{R}\backslash\left\{ 0,n\right\} $
	is surjective for $\Delta$ if and only if $\delta<n$}.

From the previous discussion, it follows that the map
\[
\Delta:x^{\delta}H_{0}^{k+2}\left(X\right)\to x^{\delta}H_{0}^{k}\left(X\right)
\]
is:
\begin{enumerate}
	\item semi-Fredholm essentially surjective if and only if $\delta<n$, $\delta\not=0$;
	\item semi-Fredholm essentially injective if and only if $\delta>0$, $\delta\not=n$.
\end{enumerate}
In particular, the map above is Fredholm for $\delta\in\left(0,n\right)$.

The analysis of this example continues in \S\ref{subsec:The-guiding-example-=0000A73}.

\section{\label{sec:Boundary-value-problems}Boundary value problems}

In this section, we will sketch our approach to boundary value problems,
and we will solve in detail the model problem. We adopt a strategy
similar to the one introduced by Mazzeo and Vertman in \cite{MazzeoEdgeII}.
A careful analysis of the model problem motivates the need for the
``symbolic $0$-calculus'', developed in the rest of the paper with
the aim of constructing parametrices for elliptic boundary value problems.

\subsection{\label{subsec:Boundary-conditions}Boundary conditions}

\subsubsection{\label{subsubsec:Semi-Fredholmness}Semi-Fredholmness}

Theorem \ref{thm:(Mazzeo=002013Melrose)Good_parametrices_fully_0-elliptic}
provides very satisfactory left and right parametrices for a $0$-elliptic
operator $L$ with constant indicial roots on $x^{\delta}L_{b}^{2}\left(X\right)$,
when the weight $\delta$ is invertible. However, in natural geometric
problems, one very commonly has to deal with $0$-elliptic operators
on weighted spaces where the weight $\delta$ is only injective or
surjective. In this case, $L$ is only \emph{semi-Fredholm}. The following
fundamental result is due to Mazzeo, see Theorem 6.1 of \cite{MazzeoEdgeI}.
Defining $L^{\dagger}$ as in the statement below, this theorem follows
easily from Theorem \ref{thm:(Mazzeo=002013Melrose)Good_parametrices_fully_0-elliptic}
applied to $L^{\dagger}L$ (resp.\ $LL^{\dagger}$) in the essentially
injective (resp.\ surjective) case.
\begin{thm}
	\label{thm:(Mazzeo)-generalized-projectors-and-gen-inverse}(Mazzeo)
	Let $L\in\Diff_{0}^{m}\left(X\right)$ be $0$-elliptic with constant
	indicial roots, and let $\delta\in\mathbb{R}$ be an injective (resp.\ surjective)
	weight for $L$. Then $L$ is semi-Fredholm essentially injective
	(resp.\ essentially surjective) as an operator $x^{\delta}H_{0}^{k+m}\left(X\right)\to x^{\delta}H_{0}^{k}\left(X\right)$.
	Moreover, given a choice of an $x^{\delta}L_{b}^{2}$ inner product
	induced by a singular density of the form $x^{-2\delta-1}\omega$,
	with $\omega$ a smooth positive density on $X$, the $x^{\delta}L_{b}^{2}$
	generalized inverse $G$ and orthogonal projectors $P_{1},P_{2}$
	onto the kernel and orthogonal complement of the range of $L$, are
	all in the $0$-calculus. More specifically:
	\begin{enumerate}
		\item $G\in\Psi_{0}^{-m}\left(X\right)+\Psi_{0}^{-\infty,\mathcal{H}}\left(X\right)$,
		where $\mathcal{H}=\left(\mathcal{H}_{\lf},\mathcal{H}_{\rf},\mathcal{H}_{\ff_{0}}\right)$
		is as in Theorem \ref{thm:(Mazzeo=002013Melrose)Good_parametrices_fully_0-elliptic};
		\item $P_{1}\in\Psi_{0}^{-\infty,\mathcal{E}}\left(X\right)$ where $\mathcal{E}=\left(\mathcal{E}_{\lf},\mathcal{E}_{\rf},\mathcal{E}_{\ff_{0}}\right)$
		and:
		\begin{enumerate}
			\item $\mathcal{E}_{\lf}$ is the extended union of the index sets generated
			by the pairs $\left(\mu,M_{\mu}\right)$ in $\widetilde{\spec}_{b}\left(L\right)$
			with $\Re\left(\mu\right)>\delta$;
			\item $\mathcal{E}_{\rf}=\overline{\mathcal{E}_{\lf}}-2\delta-1$;
			\item $\mathcal{E}_{\ff_{0}}=\mathbb{N}\cup\mathcal{I}$ where $\mathcal{I}$
			is an index set with $\Re\left(\mathcal{I}\right)>n$;
		\end{enumerate}
		\item $P_{2}\in\Psi_{0}^{-\infty,\mathcal{F}}\left(X\right)$ where $\mathcal{F}=\left(\mathcal{F}_{\lf},\mathcal{F}_{\rf},\mathcal{F}_{\ff_{0}}\right)$
		and:
		\begin{enumerate}
			\item $\mathcal{F}_{\lf}$ is the extended union of the simple index sets
			generated by the pairs $\left(\mu,M_{\mu}\right)$ in $\widetilde{\spec}_{b}\left(L^{\dagger}\right)$
			with $\Re\left(\mu\right)>\delta$, where $L^{\dagger}$ is the $x^{\delta}L_{b}^{2}$
			formal adjoint of $L$;
			\item $\mathcal{F}_{\rf}=\overline{\mathcal{F}_{\lf}}-2\delta-1$;
			\item $\mathcal{F}_{\ff_{0}}=\mathbb{N}\cup\mathcal{I}$ where $\mathcal{I}$
			is an index set with $\Re\left(\mathcal{I}\right)>n$.
		\end{enumerate}
	\end{enumerate}
	If $\delta$ is essentially injective (resp.\ essentially surjective)
	then $P_{1}$ (resp.\ $P_{2}$) is very residual in $\Psi^{-\infty,\left(\mathcal{E}_{\lf},\mathcal{E}_{\rf}\right)}\left(X\right)$
	(resp.\ $\Psi^{-\infty,\left(\mathcal{F}_{\lf},\mathcal{F}_{\rf}\right)}\left(X\right)$).
	Finally, the range of $L:x^{\delta}H_{0}^{k+m}\left(X\right)\to x^{\delta}H_{0}^{k}\left(X\right)$
	is $x^{\delta}L_{b}^{2}$ orthogonally complemented by the $x^{\delta}L_{b}^{2}$
	kernel of $L^{\dagger}$, which consists of a finite-dimensional subspace
	of $\mathcal{A}_{\phg}^{\mathcal{F}_{\lf}}\left(X\right)$.
\end{thm}

If $\delta$ is not invertible, but still injective or surjective,
then one can still hope to supplement $L$ with boundary conditions
in order to get Fredholmness. The formulation of elliptic boundary
conditions for semi-Fredholm $0$-elliptic operators is due to Mazzeo
and Vertman, cf.\ \cite{MazzeoEdgeII}. In this paper, we present
a slightly different formulation which we sketch now. For simplicity,
in this section we will ignore regularity issues and work on weighted
spaces of conormal functions $x^{\delta}H_{b}^{\infty}\left(X\right)$
(this space consists of functions on $x^{\delta}L_{b}^{2}\left(X\right)$
for which every multiple $b$-derivative is still in $x^{\delta}L_{b}^{2}\left(X\right)$).

\subsubsection{\label{subsubsec:Critical-indicial-roots}Critical indicial roots}

Fix once and for all the following auxiliary data:
\begin{enumerate}
	\item an auxiliary vector field $V$ on $X$ transverse to $\partial X$
	and inward-pointing;
	\item a boundary defining function $x$ compatible with $V$, in the sense
	that $Vx\equiv1$ near the boundary;
	\item a smooth positive density $\omega$ on $X$.
\end{enumerate}
As previously explained, the choices of $x$ and $\omega$ allow us
to define inner products on the weighted spaces $x^{\delta}L_{b}^{2}\left(X\right)$.
If $L\in\Diff_{0}^{m}\left(X\right)$ is $0$-elliptic with constant
indicial roots, and $\delta$ is an injective or surjective weight,
we denote by $G,P_{1},P_{2}$ the generalized inverse and orthogonal
projectors for $L$ relative to the chosen inner product on $x^{\delta}L_{b}^{2}\left(X\right)$.
The choice of $V$ trivializes the inward-pointing normal bundle $N^{+}\partial X\equiv\partial X\times\mathbb{R}_{1}^{1}$,
and its flow starting at $\partial X$ determines a collar neighborhood
of $\partial X$ in $X$ with respect to which we can compute asymptotic
expansions of polyhomogeneous functions on $X$. The global coordinate
$N^{+}\partial X\to\mathbb{R}_{1}^{1}$ induced by $V$ will be denoted
with $x$ as well, with slight abuse of notation. 

Let $\delta$ be a surjective, not injective weight for $L$. We are
interested in solving the equation $Lu=v$ for $u,v\in x^{\delta}H_{b}^{\infty}\left(X\right)$.
This problem can be solved for essentially every $v$ (i.e.\ for
all $v$ orthogonal to the kernel of $L^{\dagger}$, which by the
previous theorem span a finite-dimensional subspace of $\mathcal{A}_{\phg}^{\mathcal{F}_{\lf}}\left(X\right)$),
but the solution is not essentially unique since the kernel of $L$
is infinite-dimensional. Assuming for simplicity that $v$ is in the
range of $L$, a particular\emph{ }solution is provided by $Gv$,
the unique solution orthogonal to the $x^{\delta}L_{b}^{2}$ kernel
of $L$. In general, every solution can be decomposed orthogonally
as $u=P_{1}u+Gv$. Now, although $Gv$ is generically only in $x^{\delta}H_{b}^{\infty}\left(X\right)$,
$P_{1}u$ is in fact \emph{polyhomogeneous} with index set $\mathcal{E}_{\lf}$:
this is proved in \cite{MazzeoEdgeI}\footnote{If $u$ is an $x^{\delta}L_{b}^{2}$ solution of $Lu=v$ with $v$
	only in $x^{\delta}L_{b}^{2}\left(X\right)$, then $P_{1}u$ is smooth
	in the interior by $0$-elliptic regularity, but it only has a \emph{weak}
	polyhomogeneous expansion, as explained in \cite{MazzeoEdgeI}.}. Moreover, if one writes down the expansion of $w:=P_{1}u$ with
respect to $V$
\[
w\sim\sum_{\left(\alpha,l\right)\in\mathcal{E}_{\lf}}w_{\alpha,l}x^{\alpha}\left(\log x\right)^{l},
\]
then one quickly realizes that only some of the coefficients $w_{\alpha,l}\in C^{\infty}\left(\partial X\right)$
are formally free: specifically, only the coefficients $w_{\mu,l}$
with $\mu\in\spec_{b}\left(L\right)$ and $\mu\leq M_{\mu}$ are,
while the others are obtained as linear differential combinations
of those coefficients.

Choose now an \emph{injective} weight $\overline{\delta}$. Note that
this implies that $\overline{\delta}>\delta$. If $Lw=0$ and all
the coefficients $w_{\mu,l}$ with $\mu\in\spec_{b}\left(L\right)$,
$\Re\left(\mu\right)\in\left(\delta,\overline{\delta}\right)$ and
$l\leq M_{\mu}$ vanish, then $w$ lies in the $x^{\overline{\delta}}L_{b}^{2}$
kernel of $L$, which is finite-dimensional since $\overline{\delta}$
is an injective weight.
\begin{defn}
	The \emph{critical strip} for $L$ (relative to $\delta,\overline{\delta}$)
	is the set $\left\{ z\in\mathbb{C}:\Re\left(z\right)\in\left(\delta,\overline{\delta}\right)\right\} $.
	An indicial root for $L$ is said to be \emph{critical} (relative
	to $\delta,\overline{\delta}$) if it lies in the critical strip.
\end{defn}

From the discussion above, it is clear that in order to characterize
an $x^{\delta}H_{b}^{\infty}$ solution $u$ of $Lu=v$ up to finite-dimensional
subspaces, it is sufficient to prescribe the terms $w_{\mu,l}$ of
the expansion of $w=P_{1}u$ for which $\mu$ is a critical indicial
root.

\subsubsection{\label{subsubsec:Log-homogeneous-functions}Log-homogeneous functions}

It is quite convenient to interpret the collection $\left\{ w_{\mu,l}\right\} $
with $\mu$ critical, as a section of a natural vector bundle over
$\partial X$. Let us construct this bundle precisely.
\begin{defn}
	Let $\left(\alpha,M\right)\in\mathbb{C}\times\mathbb{N}$. A \emph{log-homogeneous
		function on $\mathbb{R}_{1}^{1}$ of degree $\leq\left(\alpha,M\right)$}
	is a smooth function on $\mathbb{R}^{+}$ which is in the kernel of
	the operator $\left(x\partial_{x}-\alpha\right)^{M+1}$. Equivalently,
	a log-homogeneous function of degree $\leq\left(\alpha,M\right)$
	is a linear combination of the $M+1$ functions $x^{\alpha}\left(\log x\right)^{j}$
	for $0\leq j\leq M$.
\end{defn}

\begin{rem}
	Although a log-homogeneous function on $\mathbb{R}_{1}^{1}$ is defined
	only in the interior, it determines canonically an extendible distribution
	on $\mathbb{R}_{1}^{1}$.
\end{rem}

One can easily check that being log-homogeneous of degree $\leq\left(\alpha,M\right)$
is a dilation invariant property. For this reason, we can talk about
\emph{fibrewise }log-homogeneous\emph{ }functions on $N^{+}\partial X$
of degree $\leq\left(\alpha,M\right)$. These functions form a $C^{\infty}\left(\partial X\right)$
module, which determines a smooth vector bundle $E_{\alpha,M}\to\partial X$
of rank $M+1$. Specifically, a choice of a vector field $V$ on $X$
defined along $\partial X$, transverse to $\partial X$ and inward-pointing,
determines a global trivialization $N^{+}\partial X\simeq\partial X\times\mathbb{R}_{1}^{1}$.
Calling $x$ the coordinate on the $\mathbb{R}_{1}^{1}$ factor, a
global frame of $E_{\alpha,M}$ is given precisely by the functions
$x^{\alpha},...,x^{\alpha}\left(\log x\right)^{M}$.

We can see $E_{\alpha,M}$ as an associated vector bundle as follows.
First of all, we can see the fiber bundle $N^{+}\partial X\backslash0$
as a principal $\mathbb{R}^{+}$ bundle, with fibrewise action given
by dilations. Now, the infinitesimal action induces a canonical vector
field on $N^{+}\partial X\backslash0$ tangent to the fibers, which
extends to a smooth canonical $0$-vector field on $N^{+}\partial X$.
We shall call this canonical $0$-vector field $x\partial_{x}$ without
reference to a specific trivialization choice, because if we trivialize
$N^{+}\partial X\simeq\partial X\times\mathbb{R}_{1}^{1}$ as above,
the canonical $0$-vector field is precisely $x\partial_{x}$.

It is immediate to see that $x\partial_{x}$ preserves log-homogeneity.
More precisely, $x\partial_{x}$ induces an\emph{ endomorphism }of
$E_{\alpha,M}$ with a single constant eigenvalue $\alpha$. For this
reason, we denote by $\mathfrak{N}_{M}:E_{\alpha,M}\to E_{\alpha,M}$
the endomorphism induced by $x\partial_{x}-\alpha$. We omit $\alpha$
from the notation for the following reason: with respect to a trivialization
$N^{+}\partial X\simeq\partial X\times\mathbb{R}_{1}^{1}$ as above,
the matrix expression of $\mathfrak{N}_{M}$ in the frame $x^{\alpha},...,x^{\alpha}\left(\log x\right)^{M}$
is
\[
\mathfrak{N}_{M}=\left(\begin{matrix}0 & 1\\
	& 0 & \ddots\\
	&  & \ddots & M\\
	&  &  & 0
\end{matrix}\right).
\]
Now, let $f$ be a smooth section of $E_{\alpha,M}$. Write $f$ in
terms of the frame $x^{\alpha},...,x^{\alpha}\left(\log x\right)^{M}$
as
\[
f\left(x\right)=\sum_{j=0}^{M}f_{j}x^{\alpha}\left(\log x\right)^{j}.
\]
Calling again $\lambda_{t}$ the dilation by $t$ action on $\mathbb{R}^{+}$,
we have
\begin{align*}
	\left(\lambda_{t}^{*}f\right)\left(x\right) & =f\left(tx\right)\\
	& =\sum_{k=0}^{M}f_{k}t^{\alpha}x^{\alpha}\sum_{j=0}^{k}\binom{k}{j}\left(\log x\right)^{j}\left(\log t\right)^{k-j}\\
	& =\sum_{j=0}^{M}\sum_{k=j}^{M}\left[t^{\alpha}\binom{k}{j}\left(\log t\right)^{k-j}\right]f_{k}x^{\alpha}\left(\log x\right)^{j}.
\end{align*}
Let us now compute the upper triangular matrix $t^{\alpha+\mathfrak{N}_{M}}$.
We have
\begin{align*}
	t^{\alpha+\mathfrak{N}_{M}} & =t^{\alpha}t^{\mathfrak{N}_{M}}\\
	& =t^{\alpha}\exp\left(\left(\log t\right)\mathfrak{N}_{M}\right)\\
	& =\sum_{j=0}^{M}t^{\alpha}\left(\log t\right)^{j}\frac{\mathfrak{N}_{M}^{j}}{j!};
\end{align*}
we see that, numbering the rows and columns from $0$ to $M$, the
$\left(j,k\right)$-th entry of $t^{\alpha+\mathfrak{N}_{M}}$ for
$k\geq j$ is precisely $t^{\alpha}\binom{k}{j}\left(\log t\right)^{k-j}$.
Therefore, we can write in matrix form
\[
\left(\lambda_{t}^{*}f\right)\left(x\right)=\left(\begin{matrix}x^{\alpha} & \cdots & x^{\alpha}\left(\log x\right)^{M}\end{matrix}\right)t^{\alpha+\mathfrak{N}_{M}}\left(\begin{matrix}f_{0}\\
	\vdots\\
	f_{M}
\end{matrix}\right).
\]
This shows that $E_{\alpha,M}$ is the vector bundle associated to
$N^{+}\partial X\backslash0$ via the representation $\rho_{\alpha,M}$
of $\mathbb{R}^{+}$ on $\GL\left(M+1,\mathbb{C}\right)$ given by
\begin{align*}
	\rho_{\alpha,M}:\mathbb{R}^{+} & \to\GL\left(M+1,\mathbb{C}\right)\\
	t & \mapsto t^{-\left(\alpha+\mathfrak{N}_{M}\right)}.
\end{align*}

\subsubsection{\label{subsubsec:The-trace-map}The trace map \& boundary conditions}

We now formalize the notion of \emph{trace map} for $L$ relative
to $\delta,\overline{\delta}$ (and the auxiliary choices of $x,\omega,V$).
\begin{defn}
	\label{def:extended-multiplicity}Let $\mu$ be an indicial root of
	$L$ with $\Re\left(\mu\right)>\delta$. Define \[\tilde{M}_{\mu}=\max\left\{ l\in\mathbb{N}:\left(\mu,l\right)\in\mathcal{E}_{\lf}\right\}.\]
	We call $\tilde{M}_{\mu}+1$ the \emph{extended multiplicity }of $\mu$
	(relative to $\delta$).
\end{defn}

\begin{rem}
	By definition, the index set $\mathcal{E}_{\lf}$ is the extended
	union of the simple index sets generated by the pairs $\left(\mu,M_{\mu}\right)\in\widetilde{\spec}_{b}\left(L\right)$,
	with $\Re\left(\mu\right)>\delta$. If there are no resonant indicial
	roots in the locus $\Re\left(z\right)>\delta$, then this extended
	union coincides with the simple union. Otherwise, then let $\left\{ \mu_{0},...,\mu_{k}\right\} $
	be a maximal set of resonant indicial roots, so that $\mu_{i}=\mu_{0}+m_{i}$
	for $0=m_{0}<m_{1}<\cdots<m_{k}\in\mathbb{N}$. We then have $\tilde{M}_{\mu_{j}}=M_{\mu_{0}}+\cdots+M_{\mu_{j}}+j$.
\end{rem}

Define
\[
\boldsymbol{E}{}_{L}=\bigoplus_{\Re\left(\mu\right)\in\left(\delta,\overline{\delta}\right)}E_{\mu,\tilde{M}_{\mu}}.
\]
If $w\in\mathcal{A}_{\phg}^{\mathcal{E}_{\lf}}\left(X\right)$ has
an expansion with respect to $V$ given by
\begin{equation}
	w\sim\sum_{\left(\alpha,l\right)\in\mathcal{E}_{\lf}}w_{\alpha,l}x^{\alpha}\left(\log x\right)^{l},\label{eq:expansion-w}
\end{equation}
for every critical indicial root $\mu$ the term $\sum_{l}w_{\mu,l}x^{\mu}\left(\log x\right)^{l}$
can be interpreted as a smooth section of $E_{\mu,\tilde{M}_{\mu}}$.
Thus, the collection of coefficients $\left\{ w_{\mu,l}\right\} $
can be interpreted as a smooth section of $\boldsymbol{E}_{L}$.
\begin{defn}
	The \emph{trace map} for $L$ relative to $\delta,\overline{\delta}$
	(and the auxiliary choices of $x,\omega,V$) is the continuous linear
	map
	\[
	\boldsymbol{A}_{L}:x^{\delta}H_{b}^{\infty}\left(X\right)\to C^{\infty}\left(\partial X;\boldsymbol{E}_{L}\right)
	\]
	sending a function $u\in x^{\delta}H_{b}^{\infty}\left(X\right)$
	to the term
	\[
	\sum_{\Re\left(\mu\right)\in\left(\delta,\overline{\delta}\right)}\sum_{l\leq\tilde{M}_{\mu}}w_{\mu,l}x^{\mu}\left(\log x\right)^{l}
	\]
	in the expansion of $w=P_{1}u$.
\end{defn}

We can now introduce, albeit not yet in its fully precise form, the
notion of \emph{boundary condition }for $L$ relative to $\delta,\overline{\delta}$
(cf.\ Definition \ref{def:elliptic-boundary-condition-precise} for
a more precise version).
\begin{defn}
	A \emph{boundary condition} for $L$ relative to $\delta,\overline{\delta}$
	is a pseudodifferential operator $\boldsymbol{Q}\in\Psi^{\bullet}\left(\partial X;\boldsymbol{E}_{L},\boldsymbol{W}\right)$,
	where $\boldsymbol{W}\to\partial X$ is a vector bundle on $\partial X$.
	The corresponding boundary value problem (on $x^{\delta}H_{b}^{\infty}$
	functions) is
	\begin{equation}
		\begin{cases}
			Lu=v & u,v\in x^{\delta}H_{b}^{\infty}\left(X\right)\\
			\boldsymbol{Q}\boldsymbol{A}_{L}u=\boldsymbol{\varphi} & \boldsymbol{\varphi}\in C^{\infty}\left(\partial X;\boldsymbol{W}\right)
		\end{cases}.\label{eq:BVP-with-full-trace}
	\end{equation}
	We say that the boundary value problem is \emph{well-posed} if a solution
	exists for essentially every pair $\left(v,\boldsymbol{\varphi}\right)$,
	and the solution is essentially unique. In other words, the boundary
	value problem induced by $\boldsymbol{Q}$ is well-posed if the map
	\[
	\mathcal{L}=\left(\begin{matrix}L\\
		\boldsymbol{Q}\boldsymbol{A}_{L}
	\end{matrix}\right):x^{\delta}H_{b}^{\infty}\left(X\right)\to x^{\delta}H_{b}^{\infty}\left(X\right)\oplus C^{\infty}\left(\partial X;\boldsymbol{W}\right)
	\]
	is \emph{Fredholm}.
\end{defn}

\begin{rem}
	\label{rem:why-not-reduced-trace-map}When the critical strip contains
	resonant indicial roots, the information contained in $\boldsymbol{A}_{L}u$
	is slightly redundant. To explain this better, let us consider a simple
	situation where the critical strip contains two simple indicial roots
	$\mu_{0}=0$ and $\mu_{1}=k$ for some $k\in\mathbb{N}^{+}$. In this
	case $E_{0,0}$ is canonically trivial, and $E_{k,1}$ can be trivialized
	by means of the choice of $V$, with global frame $x^{k},x^{k}\log x$.
	We can then write the trace map as
	\[
	\boldsymbol{A}_{L}=A_{L}^{0}+A_{L}^{k,0}x^{k}+A_{L}^{k,1}x^{k}\log x.
	\]
	While the two components $A_{L}^{0}$ and $A_{L}^{k,0}$ are independent
	of each other, the component $A_{L}^{k,1}$ can be recovered from
	$A_{L}^{0}$. More precisely, by an elementary formal analysis of
	the equation $Lu=0$, one can construct a differential operator $P_{k}\in\Diff^{k}\left(\partial X\right)$,
	dependent on $L$ and the auxiliary choice of $V$, such that $A_{L}^{k,1}=P_{k}A_{L}^{0}$.
	Therefore, if $\boldsymbol{Q}$ is a boundary condition, expressed
	in terms of the frame $1,x^{k},x^{k}\log x$ as $\left(\begin{matrix}Q_{0} & Q_{k,0} & Q_{k,1}\end{matrix}\right)$,
	then the equation $\boldsymbol{Q}\boldsymbol{A}_{L}u=\varphi$ can
	be equivalently expressed as an equation only involving $A_{L}^{0}$
	and $A_{L}^{k,0}$, namely
	\[
	\left[\left(Q_{0}+Q_{k,1}P_{k}\right)A_{L}^{0}+Q_{k,0}A_{L}^{k,0}\right]u=\varphi.
	\]
	This would suggest that it is simpler to directly use the ``reduced
	trace map'' $\boldsymbol{A}_{L}^{\mathrm{red}}=A_{L}^{0}+A_{L}^{k,0}x^{k}$,
	rather than the full trace map $\boldsymbol{A}_{L}$. Unfortunately,
	$\boldsymbol{A}_{L}^{\mathrm{red}}$ does not satisfy some nice properties
	which $\boldsymbol{A}_{L}$ possesses (in particular, the natural
	Bessel map of $\boldsymbol{A}_{L}^{\mathrm{red}}$ is not ``twisted
	homogeneous'', cf.\ \S\ref{subsec:analysis-of-The-Bessel-trace}).
	This makes $\boldsymbol{A}_{L}^{\mathrm{red}}$ more difficult to
	handle than $\boldsymbol{A}_{L}$. In practice, many natural boundary
	value problems involving resonant indicial roots can be expressed
	elegantly using the full trace map $\boldsymbol{A}_{L}$. We remark
	that the resonance phenomenon appears frequently in geometrically
	important situations; an example is the Laplacian on conformally compact
	manifolds of odd dimension (cf.\ \S\ref{subsec:The-guiding-example-=0000A77}).
	More details, with examples, on how to formulate $0$-elliptic boundary
	value problems involving resonant indicial roots, will be discussed
	in a sequel to this paper. One important class
	of examples where this issue does \emph{not }arise is the class, mentioned
	in Remark \ref{rem:rescaled-x^mL}, of $0$-elliptic operators of
	the form $L=x^{m}\tilde{L}$ with $\tilde{L}\in\Diff^{m}\left(X\right)$
	elliptic in the standard sense. $L$ has simple indicial roots $0,...,m-1$;
	however, upon a formal analysis of the equation $Lu=0$, it turns
	out that there is no actual resonance: the expansion of a solution
	of $Lu=0$ does not actually contain log terms. In this case, the
	full trace map coincides with the reduced trace map; more precisely,
	the $k$-th component of $\boldsymbol{A}_{L}$ taking (a priori) values
	on sections of $E_{k,k}$, actually maps into the space of sections
	of the line bundle $E_{k,0}\subset E_{k,k}$.
\end{rem}

\subsection{\label{subsec:Elliptic-boundary-conditions}Elliptic boundary conditions}

In order to formulate the concept of ellipticity of a boundary condition,
we need to describe the model problem at the Bessel level. Let $p\in\partial X$,
and let $\eta\in T_{p}^{*}\partial X\backslash0$. As $\hat{N}_{\eta}\left(L\right)$
is Fredholm as an operator \[x^{\delta}H_{b}^{m}\left(N_{p}^{+}\partial X\right)\to x^{\delta}L_{b}^{2}\left(N_{p}^{+}\partial X\right),\]
one can talk about \emph{the Bessel trace map} for $\hat{N}_{\eta}\left(L\right)$
relative to $\delta,\overline{\delta}$ (and the choices of $V,x,\omega$).

To define the Bessel trace map, denote by $\hat{N}_{\eta}\left(P_{1}\right)$
the Bessel operator of the orthogonal projector $P_{1}$ onto the
kernel of $L$. Observe that $\hat{N}_{\eta}\left(P_{1}\right)$ is
well-defined, since by Theorem \ref{thm:(Mazzeo)-generalized-projectors-and-gen-inverse}
$P_{1}$ is in $\Psi_{0}^{-\infty,\mathcal{E}}\left(X\right)$ for
some triple of index sets $\left(\mathcal{E}_{\lf},\mathcal{E}_{\rf},\mathcal{E}_{\ff_{0}}\right)$
with $\left[\mathcal{E}_{\ff_{0}}\right]=0$. The operator $\hat{N}_{\eta}\left(P_{1}\right)$
is precisely the $x^{\delta}L_{b}^{2}$ orthogonal projector onto
the kernel of $\hat{N}_{\eta}\left(L\right)$. To see this, first
note that the auxiliary choices $V,x,\omega$ determine the inner
product on $x^{\delta}L_{b}^{2}\left(X\right)$ induced by the singular
density $x^{-2\delta-1}\omega$, and also the inner product on $x^{\delta}L_{b}^{2}\left(N_{p}^{+}\partial X\right)$
induced by the singular density $x^{-2\delta-1}dx$. Calling $\cdot^{\dagger}$
the $x^{\delta}L_{b}^{2}$ formal adjoint operation, both on $x^{\delta}L_{b}^{2}\left(X\right)$
and on $x^{\delta}L_{b}^{2}\left(N_{p}^{+}\partial X\right)$, it
is easy to see that $\hat{N}_{\eta}\left(\cdot\right)$ commutes with
$\cdot^{\dagger}$. Moreover, $\hat{N}_{\eta}\left(\cdot\right)$
is also compatible with compositions. Now, let $G$ be the $x^{\delta}L_{b}^{2}$
generalized inverse $G$ of $L$. By definition of generalized inverse,
$G$ is characterized (among the bounded linear maps $x^{\delta}H_{b}^{m}\left(X\right)\to x^{\delta}L_{b}^{2}\left(X\right)$)
by the equations
\begin{align*}
	LGL & =L\\
	GLG & =G\\
	\left(LG\right)^{\dagger} & =LG\\
	\left(GL\right)^{\dagger} & =GL,
\end{align*}
and $P_{1}$ is then characterized by the equation
\[
P_{1}=I-GL.
\]
By Theorem \ref{thm:(Mazzeo)-generalized-projectors-and-gen-inverse},
every operator involved in the previous equations is an element of the large
$0$-calculus\footnote{Strictly speaking, we only defined $\hat{N}_{\eta}\left(\cdot\right)$
	on $0$-pseudodifferential operators with differential order $-\infty$,
	while $I$ and $G$ have differential orders $0$ and $-m$, respectively.
	The operator $\hat{N}_{\eta}\left(\cdot\right)$ can indeed be defined
	for elements of the small $0$-calculus $\Psi_{0}^{\bullet}\left(X\right)$,
	but this is not very important in this paper. We omit the precise
	definition, and refer to \cite{Hintz0calculus} for a thorough description.
	In the argument above, what matters for us is that $\hat{N}_{\eta}\left(G\right)$
	is well-defined as a map $x^{\delta}L_{b}^{2}\left(N_{p}^{+}\partial X\right)\to x^{\delta}H_{b}^{m}\left(N_{p}^{+}\partial X\right)$
	and that $\hat{N}_{\eta}\left(I\right)$ is the identity on functions
	over $N_{p}^{+}\partial X$.}. We can then apply the operator $\hat{N}_{\eta}\left(\cdot\right)$,
and conclude that $\hat{N}_{\eta}\left(G\right)$ is the $x^{\delta}L_{b}^{2}$
generalized inverse of $\hat{N}_{\eta}\left(L\right)$; therefore,
from the equation $\hat{N}_{\eta}\left(P_{1}\right)=I-\hat{N}_{\eta}\left(G\right)\hat{N}_{\eta}\left(L\right)$,
we obtain that $\hat{N}_{\eta}\left(P_{1}\right)$ is indeed the $x^{\delta}L_{b}^{2}$
orthogonal projector onto the kernel of $\hat{N}_{\eta}\left(L\right)$.
\begin{defn}
	\label{def:construction-of-theBessel-trace-map}The \emph{Bessel trace
		family} for $L$ relative to $\delta,\overline{\delta}$ (and the
	auxiliary choices of $V,x,\omega$) is the family $\hat{\boldsymbol{a}}_{L}:\eta\mapsto\hat{\boldsymbol{a}}_{L,\eta}$,
	parametrized by $\eta\in T^{*}\partial X\backslash0$, of operators
	\[
	\hat{\boldsymbol{a}}_{L,\eta}:x^{\delta}L_{b}^{2}\left(N_{p}^{+}\partial X\right)\to\left(\pi^{*}\boldsymbol{E}_{L}\right)_{\eta}
	\]
	sending a function $u\in x^{\delta}L_{b}^{2}\left(N_{p}^{+}\partial X\right)$
	to the term
	\[
	\sum_{\Re\left(\mu\right)\in\left(\delta,\overline{\delta}\right)}\sum_{l\leq\tilde{M}_{\mu}}w_{\mu,l}x^{\mu}\left(\log x\right)^{l}
	\]
	in the polyhomogeneous expansion of $w=\hat{N}_{\eta}\left(P_{1}\right)u$.
\end{defn}

\begin{rem}
	The definition above seems to suggest that the trace map $\boldsymbol{A}_{L}$
	defined above belongs to the class of $0$-trace operators introduced
	in \S\ref{sec:Recap-on-the-0-calc}, and that $\eta\mapsto\hat{\boldsymbol{a}}_{\eta}$
	is its Bessel family. Although it is true that $\boldsymbol{A}_{L}$
	\emph{can} be thought\emph{ }as $0$-trace operator, in general the
	class $\Psi_{0\tr}^{-\infty,\bullet}\left(X;\partial X,\boldsymbol{E}_{L}\right)$
	is \emph{not} a good fit to describe $\boldsymbol{A}_{L}$. In particular,
	if we think of $\boldsymbol{A}_{L}$ as an element of $\Psi_{0\tr}^{-\infty,\bullet}\left(X;\partial X,\boldsymbol{E}_{L}\right)$,
	then $\eta\mapsto\hat{\boldsymbol{a}}_{L,\eta}$ is \emph{not }the
	Bessel family of $\boldsymbol{A}_{L}$ in the sense of \S\ref{sec:Recap-on-the-0-calc}.
	To be more precise, let $A\in\Psi_{0\tr}^{-\infty,\left(\mathcal{E}_{\of},\mathcal{E}_{\ff}\right)}\left(\partial X,X\right)$
	with $\left[\mathcal{E}_{\ff}\right]=\alpha\in\mathbb{C}$. One can
	generalize the definition of the normal family given in \S\ref{sec:Recap-on-the-0-calc},
	in order to allow it to capture the leading order homogeneous of degree
	$\alpha$ coefficient in the expansion of $\kappa_{A}$ at $\ff$.
	In this case, the Bessel family $\hat{N}_{\eta}\left(A\right)$ satisfies
	\emph{by definition} the homogeneity condition 
	\[
	\hat{N}_{t\eta}\left(A\right)=t^{\alpha}\hat{N}_{\eta}\left(A\right)\circ\lambda_{t^{-1}}^{*}.
	\]
	On the other hand, as we shall see in \S\ref{subsec:analysis-of-The-Bessel-trace},
	the Bessel family of $\boldsymbol{A}_{L}$ satisfies (in general)
	a \emph{twisted }homogeneity condition
	\[
	\hat{N}_{t\eta}\left(\boldsymbol{A}_{L}\right)=t^{\boldsymbol{\mathfrak{s}}_{L}}\hat{N}_{\eta}\left(A\right)\circ\lambda_{t^{-1}}^{*},
	\]
	where $\boldsymbol{\mathfrak{s}}_{L}:\boldsymbol{E}_{L}\to\boldsymbol{E}_{L}$
	is an \emph{endomorphism} of $\boldsymbol{E}_{L}$ with eigenvalues
	equal to the indicial roots of $L$ with real part in $\left(\delta,\overline{\delta}\right)$.
	Even if we choose index sets $\left(\mathcal{E}_{\rf},\mathcal{E}_{\ff}\right)$
	large enough to have $\boldsymbol{A}_{L}\in\Psi_{0\tr}^{-\infty,\left(\mathcal{E}_{\rf},\mathcal{E}_{\ff}\right)}\left(\partial X,X;\boldsymbol{E}_{L}\right)$,
	the Bessel family map on $\Psi_{0\tr}^{-\infty,\left(\mathcal{E}_{\rf},\mathcal{E}_{\ff}\right)}\left(\partial X,X;\boldsymbol{E}_{L}\right)$
	is not capable of capturing the twisted homogeneity of $\hat{N}_{\eta}\left(\boldsymbol{A}_{L}\right)$.
	In \S\ref{sec:The-symbolic-0-calculus} (see in particular \S\ref{subsec:Twisted--trace-and-poisson})
	we will introduce an alternative class of operators of trace type,
	which we shall call \emph{twisted symbolic }$0$-trace operators.
	By\emph{ }design, the elements of this class have Bessel maps satisfying
	twisted homogeneity conditions. In \S\ref{subsec:The-trace-map-for-L},
	we shall see that $\boldsymbol{A}_{L}$ does indeed belong to this
	class, and that the Bessel trace family $\eta\mapsto\hat{\boldsymbol{a}}_{L,\eta}$
	is indeed the Bessel family of $\boldsymbol{A}_{L}$ in this sense.
	To conclude, we remark that if the critical strip contains a single,
	simple indicial root, then $\hat{N}_{\eta}\left(\boldsymbol{A}_{L}\right)$
	is homogeneous in the standard sense; in this case, it is possible
	to make the Schwartz kernel strategy work, cf.\ \cite{MazzeoEdgeII, UsulaPhD}.
\end{rem}

Observe that, since $\delta$ is a surjective weight, the Bessel trace
map $\hat{\boldsymbol{a}}_{L,\eta}$ restricts to an isomorphism between
the kernel of $\hat{N}_{\eta}\left(P_{1}\right)$ and the range of
$\hat{\boldsymbol{a}}_{L,\eta}$. Indeed:
\begin{enumerate}
	\item by construction, we have $\hat{\boldsymbol{a}}_{L,\eta}=\hat{\boldsymbol{a}}_{L,\eta}\hat{N}_{\eta}\left(P_{1}\right)$,
	and therefore the range of $\hat{\boldsymbol{a}}_{L,\eta}$ coincides
	with the range of $\hat{\boldsymbol{a}}_{L,\eta}$ restricted to the
	kernel of $\hat{N}_{\eta}\left(L\right)$;
	\item if $\hat{N}_{\eta}\left(L\right)u=0$ and $\hat{\boldsymbol{a}}_{L,\eta}u=0$,
	then from the previous discussions we know that $u$ is actually $o\left(x^{\overline{\delta}}\right)$;
	but this then implies that $u\equiv0$, because this would contradict
	the fact that $\overline{\delta}$ is an injective weight.
\end{enumerate}
Since the map $\hat{N}_{\eta}\left(L\right)$ is surjective on $x^{\delta}L_{b}^{2}\left(N_{p}^{+}\partial X\right)$,
it follows that its $x^{\delta}L_{b}^{2}$ kernel, identified with
the range of $\hat{\boldsymbol{a}}_{L,\eta}$, determines as $\eta$
varies in $T^{*}\partial X\backslash0$ a smooth subbundle $\eta\mapsto\mathcal{C}_{\eta}$
of the pull-back bundle $\pi^{*}\boldsymbol{E}_{L}$, where $\pi:T^{*}\partial X\backslash0\to\partial X$
is the bundle projection.
\begin{defn}
	The \emph{Calderón bundle }of $L$ relative to $\delta,\overline{\delta}$
	is the subbundle $\mathcal{C}\to T^{*}\partial X\backslash0$ of $\pi^{*}\boldsymbol{E}_{L}$
	obtained as the range of the Bessel trace map $\hat{\boldsymbol{a}}_{L,\eta}$.
\end{defn}

Let us sketch the notion of \emph{ellipticity }of a boundary condition.
This is the notion formulated by Mazzeo and Vertman in \cite{MazzeoEdgeII}:
\begin{defn}
	A boundary condition $\boldsymbol{Q}\in\Psi^{\bullet}\left(\partial X;\boldsymbol{E}_{L},\boldsymbol{W}\right)$
	for $L$ relative to $\delta,\overline{\delta}$ is \emph{elliptic}
	if, for every $\eta\in T^{*}\partial X\backslash0$, the principal
	symbol $\sigma_{\eta}\left(\boldsymbol{Q}\right):\left(\pi^{*}\boldsymbol{E}_{L}\right)_{\eta}\to\left(\pi^{*}\boldsymbol{W}\right)_{\eta}$
	restricts to an isomorphism between the Calderón space $\mathcal{C}_{\eta}$
	and $\left(\pi^{*}\boldsymbol{W}\right)_{\eta}$.
\end{defn}

\begin{rem}
	Observe that, depending on the nature of the operator $L$, elliptic
	boundary conditions might not exist. Indeed, a necessary condition
	for an elliptic boundary condition to exist is that the Calderón bundle
	$\mathcal{C}$ is isomorphic to the pull-back of a vector bundle on
	$\partial X$ to $T^{*}\partial X$. In particular, $\mathcal{C}$
	has to extend to a trivial bundle on $T_{p}^{*}\partial X\backslash0$,
	for every $p\in\partial X$. The failure of this condition is governed
	by a topological obstruction, which can be thought of as the $0$-differential
	analogue of the classical Atiyah--Bott obstruction. Analogously to
	the classical case, Dirac operators on even-dimensional and spin conformally
	compact manifolds often provide examples where this obstruction does
	not vanish.
\end{rem}

\subsection{\label{subsec:analysis-of-The-Bessel-trace}Analysis of the Bessel
	trace map}

We will now compute in detail the Bessel trace family $\eta\mapsto\hat{\boldsymbol{a}}_{L,\eta}$
as a function of $\eta\in T^{*}\partial X\backslash0$. In order to
perform concrete computations, we need to choose coordinates $y$
for $\partial X$ centered at $p$. Write near $p$
\[
L=\sum_{j+\left|\alpha\right|\leq m}L_{j,\alpha}\left(x,y\right)\left(x\partial_{x}\right)\left(x\partial_{y}\right)^{\alpha};
\]
then the Bessel family of $L$ (computed with respect to the chosen
vector field $V$) is
\[
\hat{N}_{\eta}\left(L\right)=\sum_{j+\left|\alpha\right|\leq m}L_{j,\alpha}\left(0,0\right)\left(x\partial_{x}\right)\left(ix\eta\right)^{\alpha}.
\]
Assume that an elliptic boundary condition exists. Then the Calderón
bundle $\mathcal{C}$ is trivial along the unit sphere in $\mathbb{R}^{n}\backslash0$.
Since $\hat{\boldsymbol{a}}_{L,\eta}$ provides an isomorphism between
the kernel of $\hat{N}_{\eta}\left(L\right)$ and $\mathcal{C}_{\eta}$,
this means that for $\hat{\eta}=\eta/\left|\eta\right|\in S^{n-1}$
we can choose an orthonormal basis of solutions $\phi_{1}\left(x;\hat{\eta}\right),...,\phi_{k}\left(x;\hat{\eta}\right)$
of $\hat{N}_{\hat{\eta}}\left(L\right)\phi=0$ such that the traces
$\hat{\boldsymbol{a}}_{L,\hat{\eta}}\phi_{i}$ provide a smooth frame
of $\mathcal{C}$ restricted to the unit sphere. But then this allows
us to extend $\phi_{1}\left(x;\hat{\eta}\right),...,\phi_{k}\left(x;\hat{\eta}\right)$
to a basis of the kernel of $\hat{N}_{\eta}\left(L\right)$ for \emph{every
}$\eta\in\mathbb{R}^{n}\backslash0$. Indeed, we extend the $\phi_{i}\left(x;\hat{\eta}\right)$
off $\mathbb{R}_{1}^{1}\times S^{n-1}$ to $\mathbb{R}_{1}^{1}\times\left(\mathbb{R}^{n}\backslash0\right)$
as
\[
\phi_{i}\left(x;\eta\right):=\phi_{i}\left(\left|\eta\right|x;\hat{\eta}\right)\left|\eta\right|^{-\delta}.
\]
Since $\hat{N}_{\eta}\left(L\right)$ satisfies the homogeneity condition
$\lambda_{t}^{*}\circ\hat{N}_{\eta}\left(L\right)\circ\lambda_{t^{-1}}^{*}=\hat{N}_{t\eta}\left(L\right)$,
we have
\begin{align*}
	\hat{N}_{\eta}\left(L\right)\phi_{i}\left(\cdot;\eta\right) & =\left(\lambda_{\left|\eta\right|}^{*}\circ\hat{N}_{\hat{\eta}}\left(L\right)\circ\lambda_{\left|\eta\right|^{-1}}^{*}\right)\phi_{i}\left(\cdot;\eta\right)\\
	& =\lambda_{\left|\eta\right|}^{*}\circ\hat{N}_{\hat{\eta}}\left(L\right)\left(\phi_{i}\left(x\left|\eta\right|^{-1};\eta\right)\right)\\
	& =\lambda_{\left|\eta\right|}^{*}\circ\hat{N}_{\hat{\eta}}\left(L\right)\phi_{i}\left(x;\hat{\eta}\right)\left|\eta\right|^{-\delta}\\
	& =0,
\end{align*}
and therefore the $\phi_{i}\left(\cdot;\eta\right)$ span the kernel
of $\hat{N}_{\eta}\left(L\right)$ and are smooth in $\eta$. They
are also orthonormal: indeed, we have
\begin{align*}
	\left(\phi_{i}\left(\cdot;\eta\right)\phi_{j}\left(\cdot;\eta\right)\right)_{x^{\delta}L_{b}^{2}} & =\int_{0}^{\infty}\phi_{i}\left(\left|\eta\right|x;\hat{\eta}\right)\left|\eta\right|^{-\delta}\overline{\phi_{j}\left(\left|\eta\right|x;\hat{\eta}\right)}\left|\eta\right|^{-\delta}x^{-2\delta}\frac{dx}{x}\\
	& =\int_{0}^{\infty}\phi_{i}\left(\left|\eta\right|x;\hat{\eta}\right)\overline{\phi_{j}\left(\left|\eta\right|x;\hat{\eta}\right)}\left(\left|\eta\right|x\right)^{-2\delta}\frac{dx}{x}\\
	& =\int_{0}^{\infty}\phi_{i}\left(x;\hat{\eta}\right)\overline{\phi_{j}\left(x;\hat{\eta}\right)}x^{-2\delta}\frac{dx}{x}\\
	& =\left(\phi_{i}\left(\cdot;\hat{\eta}\right),\phi_{j}\left(\cdot;\hat{\eta}\right)\right)_{x^{\delta}L_{b}^{2}}.
\end{align*}
The traces of the solutions $\phi_{i}\left(\cdot;\eta\right)$ provide
a smooth global trivialization of $\mathcal{C}$ on $T_{p}^{*}\partial X\backslash0$.

This allows us to describe explicitly the Schwartz kernel of $\hat{N}_{\eta}\left(P_{1}\right)$.
Indeed, $\hat{N}_{\eta}\left(P_{1}\right)$ is the orthogonal projector
onto the kernel of $\hat{N}_{\eta}\left(L\right)$, which means that
\begin{align*}
	\hat{N}_{\eta}\left(P_{1}\right) & =\sum_{i=1}^{k}\phi_{i}\left(x;\eta\right)\overline{\phi_{i}\left(\tilde{x};\eta\right)}\tilde{x}^{-2\delta}\frac{d\tilde{x}}{\tilde{x}}\\
	& =\sum_{i=1}^{k}\phi_{i}\left(x\left|\eta\right|;\frac{\eta}{\left|\eta\right|}\right)\overline{\phi_{i}\left(\tilde{x}\left|\eta\right|;\frac{\eta}{\left|\eta\right|}\right)}\left(\tilde{x}\left|\eta\right|\right)^{-2\delta-1}\left|\eta\right|d\tilde{x}\\
	& =\hat{k}\left(x\left|\eta\right|,\tilde{x}\left|\eta\right|;\frac{\eta}{\left|\eta\right|}\right)\left|\eta\right|d\tilde{x}.
\end{align*}
Since the functions $\phi_{i}\left(x;\hat{\eta}\right)$ are elements
of $\mathcal{A}_{\phg}^{\left(\mathcal{E}_{\lf},\infty\right)}\left(\overline{\mathbb{R}_{1}^{1}}\times S^{n-1}\right)$,
and since (cf.\ Theorem \ref{thm:(Mazzeo)-generalized-projectors-and-gen-inverse})
$\mathcal{E}_{\rf}=\overline{\mathcal{E}_{\lf}}-2\delta-1$, the function
\[
\hat{k}\left(x,\tilde{x};\hat{\eta}\right):=\sum_{i=1}^{k}\phi_{i}\left(x;\hat{\eta}\right)\overline{\phi_{i}\left(\tilde{x};\hat{\eta}\right)}\tilde{x}^{-2\delta-1}
\]
is an element of $\mathcal{A}_{\phg}^{\left(\mathcal{E}_{\lf},\mathcal{E}_{\rf},\infty\right)}\left(\overline{\mathbb{R}_{2}^{2}}\times S^{n-1}\right)$
where $\overline{\mathbb{R}_{2}^{2}}$ is the radial compactification
of the quadrant $\mathbb{R}_{2}^{2}=\left\{ \left(x,\tilde{x}\right)\in\mathbb{R}^{2}:x\geq0,\tilde{x}\geq0\right\} $.

The computations above allow us to describe $\hat{N}_{\eta}\left(P_{1}\right)$
for $\eta\in T_{p}^{*}\partial X\backslash0$. It is nice to interpret
$\hat{N}\left(P_{1}\right)$ invariantly as a section of a pull-back
Fréchet bundle over $T^{*}\partial X\backslash0$.
\begin{defn}
	Given $\mathcal{F}=\left(\mathcal{F}_{\lf},\mathcal{F}_{\rf}\right)$
	a pair of index sets, denote by $\Psi^{-\infty,\mathcal{F}}\left(\mathbb{R}_{1}^{1}\right)$
	the space of very residual operators on functions over the radial
	compactification $\overline{\mathbb{R}_{1}^{1}}=\left[0,+\infty\right]$
	of the form $F=f\left(x,\tilde{x}\right)d\tilde{x}$, where $f\left(x,\tilde{x}\right)\in\mathcal{A}_{\phg}^{\left(\mathcal{E}_{\lf},\mathcal{E}_{\rf},\infty\right)}\left(\overline{\mathbb{R}_{2}^{2}}\right)$.
\end{defn}

The group $\mathbb{R}^{+}$ acts linearly on $\Psi^{-\infty,\mathcal{F}}\left(\mathbb{R}_{1}^{1}\right)$
via the map
\begin{align*}
	\mathbb{R}^{+} & \to\GL\left(\Psi^{-\infty,\mathcal{F}}\left(\mathbb{R}_{1}^{1}\right)\right)\\
	t & \mapsto\left(F\mapsto\lambda_{t}^{*}\circ F\circ\lambda_{t^{-1}}^{*}\right),
\end{align*}
and therefore we can define a Fréchet bundle $\Psi^{-\infty,\mathcal{F}}\left(N^{+}\partial X\right)\to\partial X$,
with typical fiber $\Psi^{-\infty,\mathcal{F}}\left(\mathbb{R}_{1}^{1}\right)$,
associated to the principal bundle $N^{+}\partial X\backslash0$ via
the representation above. The computations above then prove that
\begin{prop}
	\label{prop:The-Bessel-family-of-P1}The Bessel family $\eta\mapsto\hat{N}_{\eta}\left(P_{1}\right)$
	(induced by the choices of $V,x,\omega$) is a section of the pull-back
	bundle $\pi^{*}\Psi^{-\infty,\left(\mathcal{E}_{\lf},\mathcal{E}_{\rf}\right)}\left(N^{+}\partial X\right)$
	over $T^{*}\partial X\backslash0$, with the following homogeneity
	property: for every $t\in\mathbb{R}^{+}$, we have
	\[
	\hat{N}_{t\eta}\left(P_{1}\right)=\lambda_{t}^{*}\circ\hat{N}_{\eta}\left(P_{1}\right)\circ\lambda_{t^{-1}}^{*}.
	\]
\end{prop}

\begin{rem}
	We want to stress the formal similarity between the usual principal
	symbol (a fibrewise homogeneous function on the total space of the
	cotangent bundle) and the Bessel family (a fibrewise homogeneous section
	of an associated bundle on the total space of the cotangent bundle).
	Indeed, as part of our symbolic approach to the $0$-calculus, we
	shall see that the Bessel family $\hat{N}\left(P_{1}\right)$ \emph{is
		a principal symbol} in a very concrete sense.
\end{rem}

Now that we have an explicit description for the Schwartz kernel of
$\hat{N}_{\eta}\left(P_{1}\right)$, the Schwartz kernel of the Bessel
trace map $\hat{\boldsymbol{a}}_{L,\eta}$ is easily obtained. For
every critical indicial root $\mu$ and every point $p\in\partial X$,
denote by
\[
\tr_{\mu,\tilde{M}_{\mu}}^{V_{p}}:\mathcal{A}_{\phg}^{\left(\mathcal{E}_{\lf},\infty\right)}\left(\overline{N_{p}^{+}\partial X}\right)\to\left(E_{\mu,\tilde{M}_{\mu}}\right)_{p}
\]
the map sending a function $w\in\mathcal{A}_{\phg}^{\left(\mathcal{E}_{\lf},\infty\right)}\left(\overline{N_{p}^{+}\partial X}\right)$
to the term $\sum_{l\leq\tilde{M}_{\mu}}w_{\mu,l}x^{\mu}\left(\log x\right)^{l}$
in its expansion at $x=0$, computed with respect to the trivialization
$N_{p}^{+}\partial X\equiv\mathbb{R}_{1}^{1}$ induced by $V_{p}$.
By definition,
\[
\hat{\boldsymbol{a}}_{L,\eta}u=\bigoplus_{\text{\ensuremath{\Re\left(\mu\right)\in\left(\delta,\overline{\delta}\right)}}}\tr_{\mu,\tilde{M}_{\mu}}^{V_{p}}\circ\hat{N}_{\eta}\left(P_{1}\right)u.
\]
Arrange the critical indicial roots as $\left(\mu^{1},...,\mu^{I}\right)^{T}$, and write $\hat{\boldsymbol{a}}_{L,\eta}=\left(\hat{a}_{L,\eta}^{1},...,\hat{a}_{L,\eta}^{I}\right)^{T}$,
where $\hat{a}_{L,\eta}^{i}=\tr_{\mu^{i},\tilde{M}_{\mu^{i}}}\circ\hat{N}_{\eta}\left(P_{1}\right)$.
Then, using the coordinates chosen above, we have
\[
\hat{a}_{L,\eta}^{i}=\tr_{\mu^{i},\tilde{M}_{\mu^{i}}}^{\partial_{x}}\left(\hat{k}\left(x\left|\eta\right|,\tilde{x}\left|\eta\right|;\frac{\eta}{\left|\eta\right|}\right)\right)\left|\eta\right|d\tilde{x}.
\]
Now we come to a crucial point, that is, the behavior of the operator
$\tr_{\mu^{i},\tilde{M}_{\mu^{i}}}^{\partial_{x}}$ with respect to
dilations. Indeed, the function $x\mapsto\hat{k}\left(x\left|\eta\right|,\tilde{x}\left|\eta\right|;\frac{\eta}{\left|\eta\right|}\right)$
is the pull-back $\left(\lambda_{\left|\eta\right|}^{*}\hat{k}\right)\left(x,\tilde{x}\left|\eta\right|;\frac{\eta}{\left|\eta\right|}\right)$.
\begin{lem}
	\label{lem:trace-and-dilations}Let $t\in\mathbb{R}^{+}$. Then
	\[
	\tr_{\mu,\tilde{M}_{\mu}}^{\partial_{x}}\circ\lambda_{t}^{*}=t^{\mu+\mathfrak{N}}\circ\tr_{\mu,\tilde{M}_{\mu}}^{\partial_{x}},
	\]
	where $\mathfrak{N}=\mathfrak{N}_{\tilde{M}_{\mu}}$ is the endomorphism
	of $E_{\mu,\tilde{M}_{\mu}}$ induced by $x\partial_{x}-\mu$.
\end{lem}

\begin{proof}
	This proof is essentially the computation of \S\ref{subsubsec:Log-homogeneous-functions}.
	Indeed, if the log-homogeneous term of degree $\leq\left(\mu,\tilde{M}_{\mu}\right)$
	of a function $u\in\mathcal{A}_{\phg}^{\mathcal{E}_{\lf}}\left(\mathbb{R}_{1}^{1}\right)$
	is
	\[
	\sum_{l\leq\tilde{M}_{\mu}}u_{\mu,l}x^{\mu}\left(\log x\right)^{l},
	\]
	then the log-homogeneous term of degree $\leq\left(\mu,\tilde{M}_{\mu}\right)$
	in the expansion of $\tilde{u}=\lambda_{t}^{*}u$ is
	\[
	\sum_{l\leq\tilde{M}_{\mu}}\tilde{u}_{\mu,l}x^{\mu}\left(\log x\right)^{l}=\sum_{l\leq\tilde{M}_{\mu}}u_{\mu,l}\left(tx\right)^{\mu}\left(\log tx\right)^{l},
	\]
	so the vectors $\left(u_{\mu,0},...,u_{\mu,\tilde{M}_{\mu}}\right)$
	and $\left(\tilde{u}_{\mu,0},...,\tilde{u}_{\mu,\tilde{M}_{\mu}}\right)$
	are related to each other by the formula
	\[
	\left(\begin{matrix}\tilde{u}_{\mu,0}\\
		\vdots\\
		\tilde{u}_{\mu,\tilde{M}_{\mu}}
	\end{matrix}\right)=t^{\mu+\mathfrak{N}}\left(\begin{matrix}u_{\mu,0}\\
		\vdots\\
		u_{\mu,\tilde{M}_{\mu}}
	\end{matrix}\right)
	\]
	where $\mathfrak{N}$ is identified here with its matrix
	\[
	\mathfrak{N}=\left(\begin{matrix}0 & 1\\
		& 0 & \ddots\\
		&  & \ddots & \tilde{M}_{\mu}\\
		&  &  & 0
	\end{matrix}\right)
	\]
	with respect to the frame $x^{\mu},...,x^{\mu}\left(\log x\right)^{\tilde{M}_{\mu}}$
	of $E_{\mu,\tilde{M}_{\mu}}$.
\end{proof}
Now, to simplify the notation, call $E_{i}=E_{\mu^{i},\tilde{M}_{\mu^{i}}}$
so that $\boldsymbol{E}_{L}=\bigoplus_{\text{\ensuremath{i}}}E_{i}$.
Write $\mathfrak{N}_{i}=\mathfrak{N}_{\tilde{M}_{\mu^{i}}}$ for the
nilpotent endomorphism of $E_{i}$ induced by $x\partial_{x}-\mu^{i}$.
Then the endomorphism $\boldsymbol{\mathfrak{s}}_{L}$ of $\boldsymbol{E}_{L}$
induced by the action of $x\partial_{x}$ takes the block form
\[
\boldsymbol{\mathfrak{s}}_{L}=\left(\begin{matrix}\mu^{1}+\mathfrak{N}_{1}\\
	& \ddots\\
	&  & \mu^{I}+\mathfrak{N}_{I}
\end{matrix}\right).
\]
We then have
\[
\hat{\boldsymbol{a}}_{L,\eta}=\left(\begin{matrix}\left(\tr_{\mu^{1},\tilde{M}_{\mu^{1}}}\circ\lambda_{\left|\eta\right|}^{*}\right)\left(\hat{k}\left(\cdot,\tilde{x}\left|\eta\right|;\frac{\eta}{\left|\eta\right|}\right)\right)\\
	\vdots\\
	\left(\tr_{\mu^{I},\tilde{M}_{\mu^{I}}}\circ\lambda_{\left|\eta\right|}^{*}\right)\left(\hat{k}\left(\cdot,\tilde{x}\left|\eta\right|;\frac{\eta}{\left|\eta\right|}\right)\right)
\end{matrix}\right)\left|\eta\right|d\tilde{x}
\]
which can be rewritten, using the previous lemma, as
\[
\hat{\boldsymbol{a}}_{L,\eta}=\left|\eta\right|^{\boldsymbol{\mathfrak{s}}_{L}}\left(\begin{matrix}\left(\tr_{\mu^{1},\tilde{M}_{\mu^{1}}}^{\partial_{x}}\hat{k}\right)\left(\tilde{x}\left|\eta\right|;\frac{\eta}{\left|\eta\right|}\right)\\
	\vdots\\
	\left(\tr_{\mu^{I},\tilde{M}_{\mu^{I}}}^{\partial_{x}}\hat{k}\right)\left(\tilde{x}\left|\eta\right|;\frac{\eta}{\left|\eta\right|}\right)
\end{matrix}\right)\left|\eta\right|d\tilde{x}.
\]
Since \[\hat{k}\left(x,\tilde{x};\hat{\eta}\right)\in\mathcal{A}_{\phg}^{\left(\mathcal{E}_{\lf},\mathcal{E}_{\rf},\infty\right)}\left(\overline{\mathbb{R}_{2}^{2}}\times S^{n-1}\right),\]
we have \[\left(\tr_{\mu^{i},\tilde{M}_{\mu^{i}}}^{\partial_{x}}\hat{k}\right)\left(\tau;\hat{\eta}\right)\in\mathcal{A}_{\phg}^{\left(\mathcal{E}_{\rf},\infty\right)}\left(\overline{\mathbb{R}_{1}^{1}}\times S^{n-1};\left(E_{\mu,\tilde{M}_{\mu}}\right)_{p}\right).\]
As in the case of the Bessel family $\hat{N}\left(P_{1}\right)$,
the Bessel trace family $\hat{\boldsymbol{a}}_{L,\eta}$ can be given
a nice invariant meaning (modulo the choice of $V$).
\begin{defn}
	Let $\mathcal{F}$ be an index set. Denote by $\Psi_{\tr}^{-\infty,\mathcal{F}}\left(\mathbb{R}_{1}^{1}\right)$
	the space of operators from functions on $\overline{\mathbb{R}_{1}^{1}}$
	to $\mathbb{C}$, given by Schwartz kernels of the form $a\left(\tilde{x}\right)d\tilde{x}$
	for some $a\in\mathcal{A}_{\phg}^{\left(\mathcal{F},\infty\right)}\left(\overline{\mathbb{R}_{1}^{1}}\right)$.
\end{defn}

Again, $\mathbb{R}^{+}$ acts linearly on $\Psi_{\tr}^{-\infty,\mathcal{F}}\left(\mathbb{R}_{1}^{1}\right)$
via the map
\begin{align*}
	\mathbb{R}^{+} & \to\GL\left(\Psi_{\tr}^{-\infty,\mathcal{F}}\left(\mathbb{R}_{1}^{1}\right)\right)\\
	t & \mapsto\left(A\mapsto A\circ\lambda_{t^{-1}}^{*}\right)
\end{align*}
and therefore we can define an associated Fréchet bundle $\Psi_{\tr}^{-\infty,\mathcal{F}}\left(N^{+}\partial X\right)\to\partial X$,
with typical fiber $\Psi_{\tr}^{-\infty,\mathcal{F}}\left(\mathbb{R}_{1}^{1}\right)$.
The computations above show that the Bessel trace family $\eta\mapsto\hat{\boldsymbol{a}}_{L,\eta}$
is a smooth section of the pull-back bundle $\pi^{*}\Psi_{\tr}^{-\infty,\mathcal{F}}\left(N^{+}\partial X;\boldsymbol{E}_{L}\right)$
over $T^{*}\partial X\backslash0$. However, this section is in general
\emph{not} homogeneous!\footnote{It is homogeneous if there is only one critical, simple indicial root.}
Rather, it is \emph{twisted} homogeneous in the following sense:
\begin{prop}
	The Bessel trace family $\eta\mapsto\hat{\boldsymbol{a}}_{L,\eta}$
	(induced by the choices of $V,x,\omega$) is a section of the pull-back
	bundle $\pi^{*}\Psi_{\tr}^{-\infty,\mathcal{E}_{\rf}}\left(N^{+}\partial X;\boldsymbol{E}_{L}\right)$
	over $T^{*}\partial X\backslash0$, and satisfies the following\emph{
		twisted} homogeneity property: for every $t\in\mathbb{R}^{+}$, we
	have
	\[
	\hat{\boldsymbol{a}}_{L,t\eta}=t^{\boldsymbol{\mathfrak{s}}_{L}}\circ\hat{\boldsymbol{a}}_{L,\eta}\circ\lambda_{t^{-1}}^{*}.
	\]
\end{prop}

\begin{proof}
	It remains to prove twisted homogeneity. Using the notation discussed
	above, we have
	\begin{align*}
		\hat{\boldsymbol{a}}_{L,\eta}\circ\lambda_{t^{-1}}^{*} & =\left|\eta\right|^{\boldsymbol{\mathfrak{s}}_{L}}\left(\begin{matrix}\left(\tr_{\mu^{1},\tilde{M}_{\mu^{1}}}^{\partial_{x}}\hat{k}\right)\left(t\tilde{x}\left|\eta\right|;\frac{\eta}{\left|\eta\right|}\right)\\
			\vdots\\
			\left(\tr_{\mu^{I},\tilde{M}_{\mu^{I}}}^{\partial_{x}}\hat{k}\right)\left(t\tilde{x}\left|\eta\right|;\frac{\eta}{\left|\eta\right|}\right)
		\end{matrix}\right)\left|\eta\right|td\tilde{x}\\
		& =t^{-\boldsymbol{\mathfrak{s}}_{L}}\left|t\eta\right|^{\boldsymbol{\mathfrak{s}}_{L}}\left(\begin{matrix}\left(\tr_{\mu^{1},\tilde{M}_{\mu^{1}}}^{\partial_{x}}\hat{k}\right)\left(\tilde{x}\left|t\eta\right|;\frac{\eta}{\left|\eta\right|}\right)\\
			\vdots\\
			\left(\tr_{\mu^{I},\tilde{M}_{\mu^{I}}}^{\partial_{x}}\hat{k}\right)\left(\tilde{x}\left|t\eta\right|;\frac{\eta}{\left|\eta\right|}\right)
		\end{matrix}\right)\left|t\eta\right|d\tilde{x}\\
		& =t^{-\boldsymbol{\mathfrak{s}}_{L}}\hat{\boldsymbol{a}}_{L,t\eta}
	\end{align*}
	thus proving the claim.
\end{proof}

\subsection{\label{subsec:Solving-the-model}Solution of the model problem}

Let's recap the situation. Given an elliptic boundary condition $\boldsymbol{Q}$
for $L$ relative to $\delta,\overline{\delta}$, we can formulate
the ``Bessel model problem'' at any covector $\eta\in T_{p}^{*}\partial X\backslash0$.
\begin{problem}
	Given a function $v\in x^{\delta}H_{b}^{\infty}\left(N_{p}^{+}\partial X\right)$
	and a vector $\varphi\in\left(\pi^{*}\boldsymbol{W}\right)_{\eta}=\boldsymbol{W}_{p}$,
	find a function $u\in x^{\delta}H_{b}^{\infty}\left(N_{p}^{+}\partial X\right)$
	such that
	\[
	\begin{cases}
		\hat{N}_{\eta}\left(L\right)u & =v\\
		\sigma_{\eta}\left(\boldsymbol{Q}\right)\hat{\boldsymbol{a}}_{L,\eta}u & =\varphi
	\end{cases}.
	\]
\end{problem}

Naturally, we would like to solve the model problem, i.e.\ invert
the map \[\hat{N}_{\eta}\left(L\right)\oplus\sigma_{\eta}\left(\boldsymbol{Q}\right)\hat{\boldsymbol{a}}_{L,\eta},\]
in such a way that the inverse depends smoothly on $p$ and $\eta$
and possesses the correct homogeneity properties in $\eta$. We have
discussed above that the Bessel family $\eta\mapsto\hat{N}_{\eta}\left(L\right)$
is homogeneous of degree $0$ in $\eta$, i.e.
\[
\hat{N}_{t\eta}\left(L\right)=\lambda_{t}^{*}\circ\hat{N}_{\eta}\left(L\right)\circ\lambda_{t^{-1}}^{*}
\]
for every $t>0$. Moreover, we have discovered that the Bessel trace
family $\eta\mapsto\hat{\boldsymbol{a}}_{L,\eta}$ is \emph{twisted
	homogeneous of degree $\boldsymbol{\mathfrak{s}}_{L}$}, in the sense
that
\[
\hat{\boldsymbol{a}}_{L,t\eta}=t^{\boldsymbol{\mathfrak{s}}_{L}}\circ\hat{\boldsymbol{a}}_{L,\eta}\circ\lambda_{t^{-1}}^{*}
\]
where $\boldsymbol{\mathfrak{s}}_{L}$ is the canonical endomorphism
of $\boldsymbol{E}_{L}$ induced by the action of $x\partial_{x}$
on the fibers $N^{+}\partial X$.

In view of the twisted homogeneity property of the Bessel trace family
$\hat{\boldsymbol{a}}_{L}$, it seems natural to ask a similar homogeneity
property for the principal symbol $\sigma\left(\boldsymbol{Q}\right)$.
Abstracting, we formulate the following
\begin{defn}
	Let $\boldsymbol{E},\boldsymbol{F}$ be smooth vector bundles over
	$\partial X$, equipped with smooth endomorphisms $\boldsymbol{\mathfrak{a}},\boldsymbol{\mathfrak{b}}$
	with constant eigenvalues. We say that a smooth section
	\[
	S:T^{*}\partial X\backslash0\to\pi^{*}\hom\left(\boldsymbol{E},\boldsymbol{F}\right)
	\]
	is \emph{fibrewise} \emph{twisted homogeneous of degree $\left(\boldsymbol{\mathfrak{a}},\boldsymbol{\mathfrak{b}}\right)$}
	if, for every $t\in\mathbb{R}^{+}$ and every $\eta\in T^{*}\partial X\backslash0$,
	we have
	\[
	S_{t\eta}=t^{-\boldsymbol{\mathfrak{b}}}\circ S_{\eta}\circ t^{\boldsymbol{\mathfrak{a}}}.
	\]
\end{defn}

\begin{rem}
	As already anticipated, in \S\ref{sec:The-symbolic-0-calculus} we
	will develop an alternative version of the $0$-calculus which contains
	naturally the trace map $\boldsymbol{A}_{L}$ and similar objects.
	As part of the calculus, we will develop a calculus of pseudodifferential
	operators on bundles over $\partial X$ \emph{whose principal symbols
		are fibrewise twisted homogeneous}. In fact, the boundary calculus
	presented in this paper is a particularly simple subcalculus of the
	calculus of \emph{pseudodifferential operators with variable orders}
	introduced by Krainer and Mendoza in \cite{KrainerMendozaVariableOrders}.
	The aim of Krainer and Mendoza was precisely to formulate elliptic
	boundary problems for elliptic \emph{edge} operators. Our analysis
	is greatly simplified by the fact that we only allow twisting endomorphisms
	with \emph{constant }eigenvalues. The result of this simplification
	is that our boundary calculus is essentially a mild generalization
	of the usual Douglis--Nirenberg calculus. This is closely related
	to the fact that we only consider $0$-elliptic operators with constant
	indicial roots.
\end{rem}

Since the precise formulation of the twisted boundary calculus is
somewhat delicate, we postpone it to later sections, and we \emph{assume}
that our elliptic boundary condition $\boldsymbol{Q}$ takes values
in sections of a bundle $\boldsymbol{W}$ equipped with an endomorphism
$\boldsymbol{\mathfrak{t}}:\boldsymbol{W}\to\boldsymbol{W}$ with
constant eigenvalues, and that its principal symbol $\eta\mapsto\sigma_{\eta}\left(\boldsymbol{Q}\right)$
is twisted homogeneous of degree $\left(-\boldsymbol{\mathfrak{s}}_{L},\boldsymbol{\mathfrak{t}}\right)$.
We can now easily construct the inverse of the operator
\[
\hat{N}_{\eta}\left(\mathcal{L}\right)=\left(\begin{matrix}\hat{N}_{\eta}\left(L\right)\\
	\sigma_{\eta}\left(\boldsymbol{Q}\right)\hat{\boldsymbol{a}}_{L,\eta}
\end{matrix}\right)
\]
in two steps.
\begin{enumerate}
	\item Choose a smooth family 
	\[
	\hat{\boldsymbol{k}}_{\eta}:T^{*}\partial X\backslash0\to\pi^{*}\hom\left(\boldsymbol{W},\boldsymbol{E}_{L}\right),
	\]
	twisted homogeneous of degree $\left(\boldsymbol{\mathfrak{t}},-\boldsymbol{\mathfrak{s}}_{L}\right)$,
	so that for every $\eta\in T^{*}\partial X\backslash0$
	the map $\sigma_{\eta}\left(\boldsymbol{Q}\right)\hat{\boldsymbol{k}}_{\eta}$
	coincides with $1_{\left(\pi^{*}\boldsymbol{W}\right)_{\eta}}$, and
	the map $\hat{\boldsymbol{k}}_{\eta}\sigma_{\eta}\left(\boldsymbol{Q}\right)$
	is a projector of $\left(\pi^{*}\boldsymbol{E}_{L}\right)_{\eta}$
	onto $\mathcal{C}_{\eta}$ (i.e.\ $\hat{\boldsymbol{k}}_{\eta}\sigma_{\eta}\left(\boldsymbol{Q}\right)$
	is idempotent and has range equal to $\mathcal{C}_{\eta}$). Note
	that this family exists thanks to the fact that $\sigma\left(\boldsymbol{Q}\right)_{\eta}:\left(\pi^{*}\boldsymbol{E}_{L}\right)_{\eta}\to\left(\pi^{*}\boldsymbol{W}\right)_{\eta}$
	restricts to an isomorphism $\mathcal{C}_{\eta}\to\left(\pi^{*}\boldsymbol{W}\right)_{\eta}$.
	\item Construct the \emph{Bessel Poisson} map $\hat{\boldsymbol{b}}_{\eta}:\left(\pi^{*}\boldsymbol{E}_{L}\right)_{\eta}\to x^{\delta}H_{b}^{\infty}\left(N_{p}^{+}\partial X\right)$
	associated to the projector $\hat{\boldsymbol{k}}_{\eta}\sigma_{\eta}\left(\boldsymbol{Q}\right)$,
	defined as follows: for every $\psi\in\left(\pi^{*}\boldsymbol{E}_{L}\right)_{\eta}$,
	$\hat{\boldsymbol{b}}_{\eta}\psi$ is the unique $x^{\delta}L_{b}^{2}$
	solution of $\hat{N}_{\eta}\left(L\right)w=0$ with trace $\hat{\boldsymbol{a}}_{L,\eta}w$
	equal to $\hat{\boldsymbol{k}}_{\eta}\sigma_{\eta}\left(\boldsymbol{Q}\right)\psi$.
	The map $\hat{\boldsymbol{b}}_{\eta}$ is well-defined exactly because
	$\hat{\boldsymbol{k}}_{\eta}\sigma_{\eta}\left(\boldsymbol{Q}\right)$
	is a projector onto the Calderón space $\mathcal{C}_{\eta}$, and
	$\hat{\boldsymbol{a}}_{L,\eta}$ restricts to an isomorphism between
	the kernel of $\hat{N}_{\eta}\left(L\right)$ and $\mathcal{C}_{\eta}$.
\end{enumerate}
Then an easy computation shows that the inverse of $\hat{N}_{\eta}\left(\mathcal{L}\right)$
is precisely
\[
\left(\begin{matrix}\hat{N}_{\eta}\left(G\right) & \hat{\boldsymbol{b}}_{\eta}\hat{\boldsymbol{k}}_{\eta}\end{matrix}\right),
\]
where $\hat{N}_{\eta}\left(G\right)$ is the generalized inverse of
$\hat{N}_{\eta}\left(L\right)$ and coincides with the Bessel operator
of the generalized inverse $G$ of $L$. In \S\ref{sec:Parametrix-construction},
the family $\hat{\boldsymbol{k}}_{\eta}$ will be realized as the
principal symbol of an operator in our twisted boundary calculus.
Similarly, our symbolic $0$-calculus contains a class of \emph{twisted
	symbolic $0$-Poisson operators}, which by design come equipped with
a Bessel family map; $\hat{\boldsymbol{b}}_{\eta}$ will be realized
as the Bessel family of a twisted symbolic $0$-Poisson operator.

To conclude the section, we clarify the behavior of the inverse $\left(\begin{matrix}\hat{N}_{\eta}\left(G\right) & \hat{\boldsymbol{b}}_{\eta}\hat{\boldsymbol{k}}_{\eta}\end{matrix}\right)$
as $\eta$ varies. The Bessel generalized inverse $\hat{N}\left(G\right)$
was constructed in \cite{MazzeoEdgeI}, and its structure was clarified
further in \cite{Lauter, Hintz0calculus}. We do not need to describe
it in detail (it is again a fibrewise homogeneous section of a pull-back
Fréchet bundle over $T^{*}\partial X\backslash0$), it is sufficient
to note that just like $\hat{N}\left(P_{1}\right)$ and $\hat{N}\left(L\right)$,
$\hat{N}\left(G\right)$ is homogeneous in $\eta$ of degree $0$:
for every $t\in\mathbb{R}^{+}$, we have
\[
\hat{N}_{t\eta}\left(G\right)=\lambda_{t}^{*}\circ\hat{N}_{\eta}\left(G\right)\circ\lambda_{t^{-1}}^{*}.
\]
Let's now describe the structure of $\hat{\boldsymbol{b}}_{\eta}$
in detail. Working again in coordinates centered at a point $p\in\partial X$
as in \ref{subsec:analysis-of-The-Bessel-trace}, let $\phi_{i}\left(x;\eta\right)$
be an orthonormal basis for the $x^{\delta}L_{b}^{2}$ kernel of $\hat{N}_{\eta}\left(L\right)$,
so that
\[
\phi_{i}\left(x;\eta\right)=\phi_{i}\left(\left|\eta\right|x;\frac{\eta}{\left|\eta\right|}\right)\left|\eta\right|^{-\delta}.
\]
Denote by $\boldsymbol{\gamma}_{i}\left(\eta\right)\in\left(\pi^{*}\boldsymbol{E}_{L}\right)_{\eta}=\left(\boldsymbol{E}_{L}\right)_{p}$
the trace of $\phi_{i}\left(x;\eta\right)$, i.e.\ $\boldsymbol{\gamma}_{i}\left(\eta\right)=\hat{\boldsymbol{a}}_{L,\eta}\phi_{i}\left(\cdot;\eta\right)$.
Observe that, by Lemma \ref{lem:trace-and-dilations}, we have
\begin{align*}
	\boldsymbol{\gamma}_{i}\left(t\eta\right) & =\hat{\boldsymbol{a}}_{L,t\eta}\phi_{i}\left(\cdot;t\eta\right)\\
	& =\bigoplus_{\text{\ensuremath{\Re\left(\mu\right)\in\left(\delta,\overline{\delta}\right)}}}\tr_{\mu,\tilde{M}_{\mu}}^{\partial_{x}}\phi_{i}\left(\left|\eta\right|tx;\frac{\eta}{\left|\eta\right|}\right)\left|t\eta\right|^{-\delta}\\
	& =\bigoplus_{\text{\ensuremath{\Re\left(\mu\right)\in\left(\delta,\overline{\delta}\right)}}}\tr_{\mu,\tilde{M}_{\mu}}^{\partial_{x}}\phi_{i}\left(tx;\eta\right)t^{-\delta}\\
	& =t^{\boldsymbol{\mathfrak{s}}_{L}-\delta}\boldsymbol{\gamma}_{i}\left(\eta\right).
\end{align*}
Now, let $\boldsymbol{\chi}^{i}\left(\eta\right)$ be the element
of $\left(\pi^{*}\boldsymbol{E}_{L}\right)_{\eta}^{*}$ sending a
vector $v\in\left(\pi^{*}\boldsymbol{E}_{L}\right)_{\eta}$ to the
coefficient of $\boldsymbol{\gamma}_{i}\left(\eta\right)$ in the
decomposition of $\hat{\boldsymbol{k}}_{\eta}\sigma_{\eta}\left(\boldsymbol{Q}\right)v$
in terms of the basis $\boldsymbol{\gamma}_{i}\left(\eta\right)$.
Then we have
\[
\hat{\boldsymbol{k}}_{\eta}\sigma_{\eta}\left(\boldsymbol{Q}\right)=\sum_{i}\boldsymbol{\gamma}_{i}\left(\eta\right)\otimes\boldsymbol{\chi}^{i}\left(\eta\right).
\]
Since $\sigma_{\eta}\left(\boldsymbol{Q}\right)$ is twisted homogeneous
of degree $\left(-\boldsymbol{\mathfrak{s}}_{L},\boldsymbol{\mathfrak{t}}\right)$,
and $\hat{\boldsymbol{k}}_{\eta}$ is twisted homogeneous of degree
$\left(\boldsymbol{\mathfrak{t}},-\boldsymbol{\mathfrak{s}}_{L}\right)$,
the composition $\hat{\boldsymbol{k}}_{\eta}\sigma_{\eta}\left(\boldsymbol{Q}\right)$
is twisted homogeneous of degree $\left(-\boldsymbol{\mathfrak{s}}_{L},-\boldsymbol{\mathfrak{s}}_{L}\right)$,
and this implies that
\[
\boldsymbol{\chi}^{i}\left(t\eta\right)=\boldsymbol{\chi}^{i}\left(\eta\right)t^{-\boldsymbol{\mathfrak{s}}_{L}+\delta}.
\]
The Bessel Poisson map $\hat{\boldsymbol{b}}_{\eta}$ is then
\begin{align*}
	\hat{\boldsymbol{b}}_{\eta} & =\sum_{i}\phi_{i}\left(x;\eta\right)\boldsymbol{\chi}^{i}\left(\eta\right)\\
	& =\sum_{i}\phi_{i}\left(x\left|\eta\right|;\frac{\eta}{\left|\eta\right|}\right)\boldsymbol{\chi}^{i}\left(\frac{\eta}{\left|\eta\right|}\right)\left|\eta\right|^{-\boldsymbol{\mathfrak{s}}_{L}}.
\end{align*}
The fact that the maps $\phi_{i}\left(\tau;\hat{\eta}\right)$ are
in $\mathcal{A}_{\phg}^{\left(\mathcal{E}_{\lf},\infty\right)}\left(\overline{\mathbb{R}_{1}^{1}}\times S^{n-1}\right)$
implies that
\[
\sum_{i}\phi_{i}\left(\tau;\hat{\eta}\right)\boldsymbol{\chi}^{i}\left(\hat{\eta}\right)\in\mathcal{A}_{\phg}^{\left(\mathcal{E}_{\lf},\infty\right)}\left(\overline{\mathbb{R}_{1}^{1}}\times S^{n-1};\boldsymbol{E}_{p}^{*}\right).
\]
To interpret $\hat{\boldsymbol{b}}_{\eta}$ invariantly, we proceed
analogously to the case of the Bessel trace map $\hat{\boldsymbol{a}}_{L,\eta}$.
\begin{defn}
	Let $\mathcal{F}$ be an index set. We denote by $\Psi_{\po}^{-\infty,\mathcal{F}}\left(\mathbb{R}_{1}^{1}\right)$
	the space of functions $\mathcal{A}_{\phg}^{\left(\mathcal{F},\infty\right)}\left(\overline{\mathbb{R}_{1}^{1}}\right)$,
	interpreted as Schwartz kernels of operators $B:\mathbb{C}\to C^{-\infty}\left(\overline{\mathbb{R}_{1}^{1}}\right)$.
\end{defn}

Again, the representation
\begin{align*}
	\mathbb{R}^{+} & \to\GL\left(\Psi_{\po}^{-\infty,\mathcal{F}}\left(\mathbb{R}_{1}^{1}\right)\right)\\
	t & \mapsto\left(B\mapsto\lambda_{t}^{*}\circ B\right)
\end{align*}
allows us to define the Fréchet bundle $\Psi_{\po}^{-\infty,\mathcal{F}}\left(N^{+}\partial X\right)\to\partial X$
with typical fiber $\Psi_{\po}^{-\infty,\mathcal{F}}\left(\mathbb{R}_{1}^{1}\right)$.
The computations above then show that
\begin{prop}
	The Bessel Poisson family $\eta\mapsto\hat{\boldsymbol{b}}_{\eta}$
	constructed above is a section of the pull-back Fréchet bundle $\pi^{*}\Psi_{\po}^{-\infty,\mathcal{E}_{\lf}}\left(N^{+}\partial X;\boldsymbol{E}_{L}^{*}\right)$
	over $T^{*}\partial X\backslash0$, with the following \emph{twisted}
	homogeneity property: for every $t\in\mathbb{R}^{+}$ and every $\eta\in T^{*}\partial X\backslash0$,
	we have
	\[
	\hat{\boldsymbol{b}}_{t\eta}=\lambda_{t}^{*}\circ\hat{\boldsymbol{b}}_{\eta}\circ t^{-\boldsymbol{\mathfrak{s}}_{L}}.
	\]
\end{prop}

\subsection{\label{subsec:The-guiding-example-=0000A73}The guiding example (continued)}

We continue from \S\ref{subsec:The-guiding-example-=0000A72} our
analysis of the scalar Laplacian on the asymptotically hyperbolic
manifold $\left(X^{n+1},g\right)$.

\subsubsection{The Dirichlet Bessel trace map}

Choose a weight $\delta<0$ and a weight $\overline{\delta}\in\left(0,n\right)$.
By the discussion in \S\ref{subsec:The-guiding-example-=0000A72},
the weight $\delta$ is surjective but not injective, while the weight
$\overline{\delta}$ is injective. Since the only indicial root in
the interval $\left(\delta,\overline{\delta}\right)$ is the simple
indicial root $0$, the trace map $A_{\Delta}$ relative to $\delta,\overline{\delta}$
takes values on sections of the line bundle $E_{0,0}\to\partial X$,
which is canonically isomorphic to the trivial line bundle. By definition,
$A_{\Delta}$ maps a section $u\in x^{\delta}H_{b}^{\infty}\left(X\right)$
to the restriction $w_{|\partial X}$, where $w=P_{1}u$ is the orthogonal
projection of $u$ to the kernel of $\Delta$.

We now construct the Bessel trace map $\hat{N}_{\eta}\left(A_{\Delta}\right)$.
By definition, $\hat{N}_{\eta}\left(A_{\Delta}\right)$ is the map
from $x^{\delta}L_{b}^{2}\left(\mathbb{R}_{1}^{1}\right)$ to $\mathbb{C}$
sending a function $u\in x^{\delta}L_{b}^{2}\left(\mathbb{R}_{1}^{1}\right)$
to the restriction of $\hat{N}_{\eta}\left(P_{1}\right)u$ to $x=0$.
We have seen that the $x^{\delta}L_{b}^{2}$ kernel of $\hat{N}_{\eta}\left(\Delta\right)$
is $1$-dimensional, spanned by $x^{\frac{n}{2}}K_{\frac{n}{2}}\left(\left|\eta\right|x\right)$.
Denote by $C_{n}$ the positive constant such that $C_{n}x^{\frac{n}{2}}K_{\frac{n}{2}}\left(x\right)$
is unit with respect to the $x^{\delta}L_{b}^{2}$ inner product.
Then
\[
\phi\left(x;\eta\right):=C_{n}\left|\eta\right|^{-\delta}\left(\left|\eta\right|x\right)^{\frac{n}{2}}K_{\frac{n}{2}}\left(\left|\eta\right|x\right)
\]
is unit, and spans the kernel of $\hat{N}_{\eta}\left(\Delta\right)$.
Therefore, since $\hat{N}_{\eta}\left(P_{1}\right)$ is the orthogonal
projector onto the kernel of $\hat{N}_{\eta}\left(\Delta\right)$,
we have
\[
\hat{N}_{\eta}\left(P_{1}\right)=\phi\left(x;\eta\right)\phi\left(\tilde{x};\eta\right)d\tilde{x}^{-2\delta}\frac{d\tilde{x}}{\tilde{x}}.
\]
We obtain $\hat{N}_{\eta}\left(A_{\Delta}\right)$ by evaluating $\hat{N}_{\eta}\left(P_{1}\right)$
at $x=0$: calling $\gamma_{0,0}$ the evaluation of $C_{n}x^{\frac{n}{2}}K_{\frac{n}{2}}\left(x\right)$
at $x=0$, we have
\[
\hat{N}_{\eta}\left(A_{\Delta}\right)=\gamma_{0,0}\left|\eta\right|^{-\delta}\phi\left(\tilde{x};\eta\right)\tilde{x}^{-2\delta}\frac{d\tilde{x}}{\tilde{x}}.
\]
Let us check the homogeneity of $\hat{N}_{\eta}\left(A_{\Delta}\right)$
directly. Let $t\in\mathbb{R}^{+}$; by construction, we have
\[
\phi\left(x;t\eta\right)=t^{-\delta}\phi\left(tx;\eta\right),
\]
and therefore as expected we have
\begin{align*}
	\hat{N}_{t\eta}\left(A_{\Delta}\right) & =\gamma_{0,0}\left|t\eta\right|^{-\delta}\phi\left(\tilde{x};t\eta\right)\tilde{x}^{-2\delta}\frac{d\tilde{x}}{\tilde{x}}\\
	& =\gamma_{0,0}\left|\eta\right|^{-\delta}\phi\left(t\tilde{x};\eta\right)\left(t\tilde{x}\right)^{-2\delta}\frac{d\tilde{x}}{\tilde{x}}\\
	& =\hat{N}_{\eta}\left(A_{\Delta}\right)\circ\lambda_{t^{-1}}^{*}
\end{align*}
where $\lambda_{t^{-1}}:\mathbb{R}_{1}^{1}\to\mathbb{R}_{1}^{1}$
is the ``dilation by $t^{-1}$'' operator.

\subsubsection{Solving the model Dirichlet problem}

The Dirichlet boundary condition for $\Delta$ relative to $\delta,\overline{\delta}$
is simply the identity operator. This is clearly an elliptic boundary
condition, because the range of $\hat{N}_{\eta}\left(A_{\Delta}\right)$
is $1$-dimensional.

We can explicitly construct the inverse of the operator
\[
\hat{N}_{\eta}\left(\mathcal{L}\right)=\left(\begin{matrix}\hat{N}_{\eta}\left(\Delta\right)\\
	\hat{N}_{\eta}\left(A_{\Delta}\right)
\end{matrix}\right).
\]
Denote by $\hat{N}_{\eta}\left(G\right)$ the $x^{\delta}L_{b}^{2}$
generalized inverse of $\hat{N}_{\eta}\left(\Delta\right)$. Since
$\delta$ is a surjective weight, and by definition of $\hat{N}_{\eta}\left(P_{1}\right)$
and $\hat{N}_{\eta}\left(A_{\Delta}\right)$, we have
\begin{align*}
	\hat{N}_{\eta}\left(\Delta\right)\hat{N}_{\eta}\left(G\right) & =\id_{\mathbb{R}_{1}^{1}}\\
	\hat{N}_{\eta}\left(G\right)\hat{N}_{\eta}\left(\Delta\right) & =\id_{\mathbb{R}_{1}^{1}}-\hat{N}_{\eta}\left(P_{1}\right)\\
	\hat{N}_{\eta}\left(A_{\Delta}\right)\hat{N}_{\eta}\left(G\right) & =0.
\end{align*}
Define now $\hat{b}_{\eta}^{0}$ as the linear map $\mathbb{C}\to x^{\delta}L_{b}^{2}\left(\mathbb{R}_{1}^{1}\right)$
sending $\phi\left(0;\eta\right)$ to $\phi\left(x;\eta\right)$.
Then the inverse of $\hat{N}_{\eta}\left(\mathcal{L}\right)$ is
\[
\hat{N}_{\eta}\left(\mathcal{G}\right):=\left(\begin{matrix}\hat{N}_{\eta}\left(G\right) & \hat{b}_{\eta}^{0}\end{matrix}\right).
\]
To see this, the only non-trivial thing to observe is that \[\hat{b}_{\eta}^{0}\hat{N}_{\eta}\left(A_{\Delta}\right)=\phi\left(x;\eta\right)\phi\left(\tilde{x};\eta\right)\tilde{x}^{-2\delta}\frac{d\tilde{x}}{\tilde{x}}=\hat{N}_{\eta}\left(P_{1}\right).\]
We conclude by observing the explicit homogeneity of $\hat{b}_{\eta}^{0}$:
$\hat{b}_{t\eta}^{0}$ maps $\phi\left(0;t\eta\right)=t^{-\delta}\phi\left(0;\eta\right)$
to $\phi\left(x;t\eta\right)=t^{-\delta}\phi\left(tx;\eta\right)$.
Therefore, $\hat{b}_{t\eta}^{0}=\lambda_{t}^{*}\hat{b}_{\eta}^{0}$.

The analysis of this example continues in \S\ref{subsec:The-guiding-example-=0000A76}.

\section{\label{sec:The-symbolic-0-calculus}The symbolic \texorpdfstring{$0$}{0}-calculus}

This section and the next form the core of the paper. We will formulate
the calculus which we shall use in \S\ref{sec:Parametrix-construction}
to construct the parametrix for our elliptic boundary value problem.
We will also describe in detail the relation between our calculus
and the $0$-calculus recalled in \S\ref{sec:Recap-on-the-0-calc}.

Our calculus consists of several components:
\begin{enumerate}
	\item classes $\hat{\Psi}_{0\tr}^{-\infty,\bullet}\left(X,\partial X\right)$
	and $\hat{\Psi}_{0\po}^{-\infty,\bullet}\left(\partial X,X\right)$
	of \emph{symbolic $0$-trace and $0$-Poisson }operators;
	\item classes $\hat{\Psi}_{0}^{-\infty,\bullet}\left(X\right)$ and $\hat{\Psi}_{0b}^{-\infty,\bullet}\left(X\right)$
	of \emph{$0$-interior and $0b$-interior }operators;
	\item variants $\hat{\Psi}_{0\tr}^{-\infty,\bullet,\boldsymbol{\mathfrak{s}}}\left(X;\partial X,\boldsymbol{E}\right)$
	and $\hat{\Psi}_{0\po}^{-\infty,\bullet,\boldsymbol{\mathfrak{s}}}\left(\partial X,\boldsymbol{E};X\right)$
	of the classes of symbolic $0$-trace and $0$-Poisson operators,
	which are ``twisted'' by a vector bundle $\boldsymbol{E}\to\partial X$
	equipped with an endomorphism $\boldsymbol{\mathfrak{s}}$ with constant
	eigenvalues;
	\item a class of pseudodifferential operators on $\partial X$ $\Psi_{\phg}^{\bullet,\left(\boldsymbol{\mathfrak{s}},\boldsymbol{\mathfrak{t}}\right)}\left(\partial X;\boldsymbol{E},\boldsymbol{F}\right)$,
	called the \emph{twisted boundary calculus}, acting between sections
	of vector bundles equipped with endomorphisms with constant eigenvalues.
\end{enumerate}
The classes of operators recalled in \S\ref{sec:Recap-on-the-0-calc}
are all characterized as Schwartz kernels which lift to polyhomogeneous
right densities on appropriate blow-ups of the double space. Instead,
the classes in our calculus will be characterized (locally near $\partial\Delta$)
as oscillatory integrals of \emph{symbols} which lift to polyhomogeneous
functions on appropriate blown-up model spaces. Since our approach
is local, the theory involves several familiar steps (e.g.\ left
and right reductions, asymptotic expansions of symbols, diffeomorphism
invariance etc.) from the standard theory of pseudodifferential operators
on closed manifolds. Most of these arguments are in fact essentially
identical to the traditional arguments for pseudodifferential theory
on closed manifolds, so we decided to omit the ones which do not present
significant differences with the standard theory.

\subsection{\label{subsec:Definitions-and-relation-with-0-calculus}Definitions and relation with the \texorpdfstring{$0$}{0}-calculus}

\subsubsection{\label{subsubsec:Symbolic-0-Poisson-operators}Symbolic \texorpdfstring{$0$}{0}-Poisson operators}

We start with symbolic $0$-Poisson operators. First, we introduce
the relevant model space. Define $\hat{P}^{2}=\overline{\mathbb{R}_{1}^{1}}\times\overline{\mathbb{R}^{n}}$
with coordinates $\left(x,\eta\right)$. This space should be thought
of as the ``frequency counterpart'' of the model space $P^{2}$
used in \S\ref{subsubsec:local-physical-0-trace-and-0-Poisson} to
define the class $\Psi_{0\po,\mathcal{S}}^{-\infty,\bullet}\left(\mathbb{R}^{n},\mathbb{R}_{1}^{n+1}\right)$.
Like $P^{2}$, $\hat{P}^{2}$ has three boundary hyperfaces
\[
\of=\left\{ x=0\right\} ,\iif_{\eta}=\left\{ \left|\eta\right|=\infty\right\} ,\iif_{x}=\left\{ x=\infty\right\} .
\]
As in \S\ref{subsubsec:local-physical-0-trace-and-0-Poisson}, $\of$
stands for ``original face'', while $\iif$ stands for ``infinity
face''. We now introduce the blow-up
\[
\hat{P}_{0}^{2}=\left[\hat{P}^{2};\left\{ x=0,\left|\eta\right|=\infty\right\} \right].
\]
This space has four boundary hyperfaces $\of,\ff,\iif_{\eta},\iif_{x}$,
where again $\ff$ stands for ``front face''. The other faces are
the lifts of the corresponding faces of $\hat{P}^{2}$, while $\ff$
is the new face obtained from the blow-up. The space $\hat{P}_{0}^{2}$
is represented in Figure \ref{fig:Phat20}\begin{figure}
	\centering
	\caption{$\hat{P}^2_0$}
	\label{fig:Phat20}
	
	\tikzmath{\L = 4;\R = \L /3;}
	\begin{tikzpicture}
		
		%%% axes
		%\coordinate (O) at (0,0) ;
		%\coordinate (X) at (1,0) ;
		%\coordinate (Y) at (0,1) ;
		%\draw[->] (O) -- (X) node {$x$};
		%\draw[->] (O) -- (Y) node {$y$};
		
		%%% points
		\coordinate (A) at (-\L, 0);
		\coordinate (B) at (\L, 0);
		\coordinate (C) at (-\L, \L);
		\coordinate (D) at (\L, \L);
		\coordinate (A1) at ({-\L+\R}, 0);
		\coordinate (A2) at (-\L, \R);
		\coordinate (B1) at ({\L-\R}, 0);
		\coordinate (B2) at (\L, \R);
		
		%%% segments
		\draw (A1) -- (B1);
		\draw (A2) -- (C);
		\draw (B2) -- (D);
		\draw (C) -- (D);
		
		%%% arcs
		\draw (A1) arc(0:90:\R);
		\draw (B1) arc(180:90:\R);
		
		%%% coordinates
		%\draw[->,red] (-{\L/6},{\L/6}) -- node[below, midway] {$\eta$} ({\L/6},{\L/6});
		%\draw[->,red] (0,{2*\L/6}) -- node[right, midway] {$x$} (0,{4*\L/6});
		%\draw[->,blue] ({\L/2}, 0) arc(180:110:{\L/2}) node[right, midway] {$t$} ;
		
		%%% annotations
		\node at ({-\L+1.8*\L/6},{1.8*\L/6}) {${\text{\Large ff}}$};
		\node at ({\L-1.8*\L/6},{1.8*\L/6}) {${\text{\Large ff}}$};
		\node at (0,{0.4*\L/6}) {${\text{\Large of}}$};
		\node at (0,{\L - 0.6*\L/6}) {${\text{\Large if}_x}$};
		\node at ({-\L+0.6*\L/6},{\L*2/3}) {${\text{\Large if}_\eta}$};
		\node at ({\L-0.6*\L/6},{\L*2/3}) {${\text{\Large if}_\eta}$};

	\end{tikzpicture}
\end{figure}.
\begin{defn}
	($0$-Poisson symbols) Let $\mathcal{E}=\left(\mathcal{E}_{\of},\mathcal{E}_{\ff}\right)$
	be a pair of index sets. We define
	\begin{align*}
		S_{0\po,\mathcal{S}}^{-\infty,\mathcal{E}}\left(\mathbb{R}^{k};\mathbb{R}_{1}^{n+1}\right) & =\mathcal{S}\left(\mathbb{R}^{k}\right)\hat{\otimes}\mathcal{A}_{\phg}^{\left(\mathcal{E}_{\of},\mathcal{E}_{\ff},\infty,\infty\right)}\left(\hat{P}_{0}^{2}\right)\\
		S_{\po,\mathcal{S}}^{-\infty,\mathcal{E}_{\of}}\left(\mathbb{R}^{k};\mathbb{R}_{1}^{n+1}\right) & =\mathcal{S}\left(\mathbb{R}^{k}\right)\hat{\otimes}\mathcal{A}_{\phg}^{\left(\mathcal{E}_{\of},\infty,\infty\right)}\left(\hat{P}^{2}\right).
	\end{align*}
	We call the elements of $S_{0\po,\mathcal{S}}^{-\infty,\mathcal{E}}\left(\mathbb{R}^{k};\mathbb{R}_{1}^{n+1}\right)$
	(resp.\ $S_{\po,\mathcal{S}}^{-\infty,\mathcal{E}_{\of}}\left(\mathbb{R}^{k};\mathbb{R}_{1}^{n+1}\right)$)
	\emph{$0$-Poisson} (resp.\ \emph{residual Poisson})\emph{ }symbols.
\end{defn}

We now define the class of \emph{local symbolic $0$-Poisson operators}.
These operators are obtained as \emph{left quantizations }of $0$-Poisson
symbols. More precisely, denote by $\Op_{L}^{\po}$ the ``left Poisson
quantization'' map
\begin{align*}
	\Op_{L}^{\po}:S_{0\po,\mathcal{S}}^{-\infty,\mathcal{E}}\left(\mathbb{R}^{n};\mathbb{R}_{1}^{n+1}\right)\times\mathcal{S}\left(\mathbb{R}^{n}\right) & \to\mathcal{A}_{\phg}^{\left(\mathcal{E}_{\of},\infty,\infty\right)}\left(\overline{\mathbb{R}_{1}^{1}}\times\overline{\mathbb{R}^{n}}\right)\\
	\left(b\left(y;x,\eta\right),u\left(y\right)\right) & \mapsto\frac{1}{\left(2\pi\right)^{n}}\int e^{iy\eta}b\left(y;x,\eta\right)\hat{u}\left(\eta\right)d\eta.
\end{align*}
Let's check that this map is well-defined. Since $u\left(y\right)\in\mathcal{S}\left(\mathbb{R}^{n}\right)$,
its Fourier transform $\hat{u}\left(\eta\right)$ is in $\mathcal{S}\left(\mathbb{R}^{n}\right)$.
Seeing $\hat{u}\left(\eta\right)$ as a function on $\overline{\mathbb{R}^{n}}\times\hat{P}^{2}$
(i.e.\ pulling $\hat{u}\left(\eta\right)$ back via the canonical
projection $\overline{\mathbb{R}^{n}}\times\hat{P}^{2}\to\overline{\mathbb{R}^{n}}$
given by $\left(y;x,\eta\right)\mapsto\eta$), $\hat{u}\left(\eta\right)$
is in \[C^{\infty}\left(\overline{\mathbb{R}^{n}}\right)\hat{\otimes}\mathcal{A}_{\phg}^{\left(0,\infty,0\right)}\left(\hat{P}^{2}\right).\]
By the pull-back Theorem, its lift to $\overline{\mathbb{R}^{n}}\times\hat{P}_{0}^{2}$
is then in \[C^{\infty}\left(\overline{\mathbb{R}^{n}}\right)\hat{\otimes}\mathcal{A}_{\phg}^{\left(0,\infty,\infty,0\right)}\left(\hat{P}_{0}^{2}\right).\]
Therefore, since \[b\in S_{0\po,\mathcal{S}}^{-\infty,\mathcal{E}}\left(\mathbb{R}^{n};\mathbb{R}_{1}^{n+1}\right)=\mathcal{S}\left(\mathbb{R}^{n}\right)\hat{\otimes}\mathcal{A}_{\phg}^{\left(\mathcal{E}_{\of},\mathcal{E}_{\ff},\infty,\infty\right)}\left(\hat{P}_{0}^{2}\right)\]
the product $b\left(y;x,\eta\right)\hat{u}\left(\eta\right)$ is in
\[\mathcal{S}\left(\mathbb{R}^{n}\right)\hat{\otimes}\mathcal{A}_{\phg}^{\left(\mathcal{E}_{\of},\infty,\infty,\infty\right)}\left(\hat{P}_{0}^{2}\right),\]
which coincides with \[\mathcal{S}\left(\mathbb{R}^{n}\right)\hat{\otimes}\mathcal{A}_{\phg}^{\left(\mathcal{E}_{\of},\infty\right)}\left(\overline{\mathbb{R}_{1}^{1}}\right)\hat{\otimes}\mathcal{S}\left(\mathbb{R}^{n}\right).\]
Therefore, applying the inverse Fourier transform $\frac{1}{\left(2\pi\right)^{n}}\int e^{iy\eta}\cdot d\eta$,
we obtain an element of $\mathcal{A}_{\phg}^{\left(\mathcal{E}_{\of},\infty\right)}\left(\overline{\mathbb{R}_{1}^{1}}\right)\hat{\otimes}\mathcal{S}\left(\mathbb{R}^{n}\right)=\mathcal{A}_{\phg}^{\left(\mathcal{E}_{\of},\infty,\infty\right)}\left(\overline{\mathbb{R}_{1}^{1}}\times\overline{\mathbb{R}^{n}}\right)$.

As in \S\ref{subsubsec:local-physical-0-trace-and-0-Poisson}, we
denote by $\Op\left(\mathbb{R}^{n},\mathbb{R}_{1}^{n+1}\right)$ the
class of operators $\mathcal{S}\left(\mathbb{R}^{n}\right)\to C^{-\infty}\left(\overline{\mathbb{R}_{1}^{1}}\times\overline{\mathbb{R}^{n}}\right)$.
Then the left Poisson quantization map defined above determines a
continuous linear map
\begin{align*}
	\Op_{L}^{\po}:S_{0\po,\mathcal{S}}^{-\infty,\mathcal{E}}\left(\mathbb{R}^{n};\mathbb{R}_{1}^{n+1}\right) & \to\Op\left(\mathbb{R}^{n},\mathbb{R}_{1}^{n+1}\right)
\end{align*}
given by
\[
b\left(y;x,\eta\right)\mapsto\left[\frac{1}{\left(2\pi\right)^{n}}\int e^{i\left(y-\tilde{y}\right)\eta}b\left(y;x,\eta\right)d\eta\right]d\tilde{y}.
\]

\begin{defn}
	(Local symbolic $0$-Poisson operators) We denote by $\hat{\Psi}_{0\po,\mathcal{S}}^{-\infty,\mathcal{E}}\left(\mathbb{R}^{n},\mathbb{R}_{1}^{n+1}\right)$
	(note the hat!) the subclass of $\Op\left(\mathbb{R}^{n},\mathbb{R}_{1}^{n+1}\right)$
	consisting of operators of the form $\Op_{L}^{\po}\left(b\right)$,
	with $b\in S_{0\po,\mathcal{S}}^{-\infty,\mathcal{E}}\left(\mathbb{R}^{n};\mathbb{R}_{1}^{n+1}\right)$.
	We call its elements local \emph{symbolic $0$-Poisson operators}.
\end{defn}

We now wish to compare the classes $\Psi_{0\po,\mathcal{S}}^{-\infty,\bullet}\left(\mathbb{R}^{n},\mathbb{R}_{1}^{n+1}\right)$
and $\hat{\Psi}_{0\po,\mathcal{S}}^{-\infty,\bullet}\left(\mathbb{R}^{n},\mathbb{R}_{1}^{n+1}\right)$.
Observe that, if $B=\Op_{L}^{\po}\left(b\right)$ for some $b\in S_{0\po,\mathcal{S}}^{-\infty,\mathcal{E}}\left(\mathbb{R}^{n};\mathbb{R}_{1}^{n+1}\right)$,
then the Schwartz kernel of $B$ is written as $B\left(y;x,y-\tilde{y}\right)d\tilde{y}$,
where $B\left(y;x,Y\right)$ is the inverse Fourier transform of $b\left(y;x,\eta\right)$
in the $\eta$ variable. Therefore, we need to understand how the
Fourier transform in the $\eta$ variables behaves on functions on
the model space $\hat{P}_{0}^{2}$. Luckily for us, this problem was
considered and solved before by Hintz in \cite{HintzDilations}:
\begin{prop}
	\label{prop:Hintz-lemma}(Hintz, Propositions 2.28 and 2.29 of \cite{HintzDilations})
	Consider the space $P^{2}=\overline{\mathbb{R}_{1}^{1}}\times\overline{\mathbb{R}^{n}}$
	with coordinates $\left(x,Y\right)$, and the space $\hat{P}=\overline{\mathbb{R}_{1}^{1}}\times\overline{\mathbb{R}^{n}}$
	with coordinates $\left(x,\eta\right)$. Then:
	\begin{enumerate}
		\item The Fourier transform in the $Y$ variables extends from $\dot{C}^{\infty}\left(P^{2}\right)$
		to a continuous linear map
		\[
		\mathcal{A}_{\phg}^{\left(\mathcal{E}_{\of},\mathcal{E}_{\ff}-n,\infty,\infty\right)}\left(P_{0}^{2}\right)\to\mathcal{A}_{\phg}^{\left(\mathcal{E}_{\of}\eunion\mathcal{E}_{\ff},\mathcal{E}_{\ff},\infty,\infty\right)}\left(\hat{P}_{0}^{2}\right).
		\]
		\item The inverse Fourier transform in the $\eta$ variables extends from
		$\dot{C}^{\infty}\left(\hat{P}^{2}\right)$ to a continuous linear
		map
		\[
		\mathcal{A}_{\phg}^{\left(\mathcal{E}_{\of},\mathcal{E}_{\ff},\infty,\infty\right)}\left(\hat{P}_{0}^{2}\right)\to\mathcal{A}_{\phg}^{\left(\mathcal{E}_{\of},\left(\mathcal{E}_{\ff}-n\right)\eunion\mathcal{E}_{\of},\infty,\infty\right)}\left(P_{0}^{2}\right).
		\]
	\end{enumerate}
\end{prop}

From this proposition, we immediately obtain the following
\begin{cor}
	\label{cor:relation-symbolic-0-poisson-physical-0-poisson}Let $\mathcal{E}=\left(\mathcal{E}_{\of},\mathcal{E}_{\ff}\right)$
	be a pair of index sets. Then:
	\begin{enumerate}
		\item the space $\hat{\Psi}_{0\po,\mathcal{S}}^{-\infty,\left(\mathcal{E}_{\of},\infty\right)}\left(\mathbb{R}^{n},\mathbb{R}_{1}^{n+1}\right)$
		coincides with $\Psi_{\po,\mathcal{S}}^{-\infty,\mathcal{E}_{\of}}\left(\mathbb{R}^{n},\mathbb{R}_{1}^{n+1}\right)$;
		\item the space $\hat{\Psi}_{0\po,\mathcal{S}}^{-\infty,\mathcal{E}}\left(\mathbb{R}^{n},\mathbb{R}_{1}^{n+1}\right)$
		is contained in $\Psi_{0\po,\mathcal{S}}^{-\infty,\left(\mathcal{E}_{\of},\mathcal{E}_{\ff}\eunion\left(\mathcal{E}_{\of}+n\right)\right)}\left(\mathbb{R}^{n},\mathbb{R}_{1}^{n+1}\right)$;
		\item the space $\Psi_{0\po,\mathcal{S}}^{-\infty,\mathcal{E}}\left(\mathbb{R}^{n},\mathbb{R}_{1}^{n+1}\right)$
		is contained in $\hat{\Psi}_{0\po,\mathcal{S}}^{-\infty,\left(\mathcal{E}_{\of}\eunion\mathcal{E}_{\ff},\mathcal{E}_{\ff}\right)}\left(\mathbb{R}^{n},\mathbb{R}_{1}^{n+1}\right)$.
	\end{enumerate}
\end{cor}

\begin{proof}
	Recall that $\Psi_{0\po,\mathcal{S}}^{-\infty,\mathcal{F}}\left(\mathbb{R}^{n},\mathbb{R}_{1}^{n+1}\right)$
	is the space of operators with kernel $K\left(y;x,y-\tilde{y}\right)d\tilde{y}$
	with \[K\left(y;x,Y\right)\in\mathcal{S}\left(\mathbb{R}^{n}\right)\hat{\otimes}\mathcal{A}_{\phg}^{\left(\mathcal{F}_{\of},\mathcal{F}_{\ff}-n,\infty,\infty\right)}\left(P_{0}^{2}\right),\]
	while in the case of $\Psi_{\po,\mathcal{S}}^{-\infty,\mathcal{F}_{\of}}\left(\mathbb{R}^{n},\mathbb{R}_{1}^{n+1}\right)$
	the function $K\left(y;x,Y\right)$ is in \[\mathcal{S}\left(\mathbb{R}^{n}\right)\hat{\otimes}\mathcal{A}_{\phg}^{\left(\mathcal{F}_{\of},\infty,\infty\right)}\left(P^{2}\right).\]
	The first point then follows immediately from the equality \[S_{\po,\mathcal{S}}^{-\infty,\mathcal{E}_{\of}}\left(\mathbb{R}^{k};\mathbb{R}_{1}^{n+1}\right)=S_{0\po,\mathcal{S}}^{-\infty,\left(\mathcal{E}_{\of},\infty\right)}\left(\mathbb{R}^{k};\mathbb{R}_{1}^{n+1}\right)\]
	and the fact that the Fourier transform in the $Y$ variables induces
	an isomorphism \[\mathcal{A}_{\phg}^{\left(\mathcal{F}_{\of},\infty,\infty\right)}\left(P^{2}\right)\to\mathcal{A}_{\phg}^{\left(\mathcal{F}_{\of},\infty,\infty\right)}\left(\hat{P}^{2}\right),\]
	while the second and third points follow from Proposition \ref{prop:Hintz-lemma}.
\end{proof}
We now pass to symbolic $0$-Poisson operators on manifolds. We fix
an auxiliary smooth vector field $V$ on $X$ transverse to $\partial X$
and inward-pointing.
\begin{defn}
	(Symbolic $0$-Poisson operators) Let $\mathcal{E}=\left(\mathcal{E}_{\of},\mathcal{E}_{\ff}\right)$
	be a pair of index sets. We denote by $\hat{\Psi}_{0\po}^{-\infty,\mathcal{E}}\left(\partial X,X\right)$
	the class of operators $B:C^{\infty}\left(\partial X\right)\to C^{-\infty}\left(X\right)$
	with the following properties:
	\begin{enumerate}
		\item in the complement of any neighborhood of $\partial\Delta$, $K_{B}$
		coincides with a residual Poisson operator in $\Psi_{\po}^{-\infty,\mathcal{E}_{\of}}\left(\partial X,X\right)=\mathcal{A}_{\phg}^{\mathcal{E}_{\of}}\left(X\times\partial X;\pi_{R}^{*}\mathcal{D}_{\partial X}^{1}\right)$;
		\item for every $p\in\partial X$ and coordinates $\left(x,y\right)$ for
		$X$ centered at $p$ and compatible with $V$ (i.e.\ $Vx\equiv1$
		near $p$), $B$ coincides near $\left(p,p\right)$ with an element
		of $\hat{\Psi}_{0\po,\mathcal{S}}^{-\infty,\mathcal{E}}\left(\mathbb{R}^{n},\mathbb{R}_{1}^{n+1}\right)$;
		more precisely, there exists a $0$-Poisson symbol $b\left(y;x,\eta\right)$
		in $S_{0\po,\mathcal{S}}^{-\infty,\mathcal{E}}\left(\mathbb{R}^{n};\mathbb{R}_{1}^{n+1}\right)$
		such that, in a neighborhood of the origin, we have
		\[
		K_{B}\equiv K\left(y;x,y-\tilde{y}\right)d\tilde{y}=\frac{1}{\left(2\pi\right)^{n}}\int e^{i\left(y-\tilde{y}\right)\eta}b\left(y;x,\eta\right)d\eta d\tilde{y}.
		\]
	\end{enumerate}
\end{defn}

Of course, one needs to check that the definition above is well-posed.
More specifically, one needs to check that the condition given in
Point 2 does not depend on the choice of coordinates, and that the
two conditions given in Points 1 and 2 are compatible with each other.
Compatibility of Points 1 and 2 is guaranteed by Corollary \ref{cor:relation-symbolic-0-poisson-physical-0-poisson}:
indeed, if $B\in\hat{\Psi}_{0\po,\mathcal{S}}^{-\infty,\mathcal{E}}\left(\mathbb{R}^{n},\mathbb{R}_{1}^{n+1}\right)$,
then $K_{B}$ is polyhomogeneous with index set $\mathcal{E}_{\of}$
at the original face, in the complement of any neighborhood of $\partial\Delta$.
Coordinate invariance will be discussed in detail in \S\ref{subsec:Technical-details-on-proofs}
using standard pseudodifferential techniques.

To summarize the relation between the classes $\hat{\Psi}_{0\po}^{-\infty,\bullet}\left(\partial X,X\right)$
and $\Psi_{0\po}^{-\infty,\bullet}\left(\partial X,X\right)$, from
the discussion above and the local characterization of $\Psi_{0\po}^{-\infty,\bullet}\left(\partial X,X\right)$
provided in \S\ref{subsubsec:local-physical-0-trace-and-0-Poisson},
we obtain the following
\begin{cor}
	\label{cor:global-symbolic-0-poisson-are-physical-0-poisson-and-viceversa}Let
	$\mathcal{E}=\left(\mathcal{E}_{\of},\mathcal{E}_{\ff}\right)$ be
	a pair of index sets. Then:
	\begin{enumerate}
		\item the space $\hat{\Psi}_{0\po}^{-\infty,\left(\mathcal{E}_{\of},\infty\right)}\left(\partial X,X\right)$
		coincides with the space of residual Poisson operators $\Psi_{\po}^{-\infty,\mathcal{E}_{\of}}\left(\partial X,X\right)$;
		\item the space $\hat{\Psi}_{0\po}^{-\infty,\mathcal{E}}\left(\partial X,X\right)$
		is contained in $\Psi_{0\po}^{-\infty,\left(\mathcal{E}_{\of},\mathcal{E}_{\ff}\eunion\left(\mathcal{E}_{\of}+n\right)\right)}\left(\partial X,X\right)$;
		\item the space $\Psi_{0\po}^{-\infty,\mathcal{E}}\left(\partial X,X\right)$
		is contained in $\hat{\Psi}_{0\po}^{-\infty,\left(\mathcal{E}_{\of}\eunion\mathcal{E}_{\ff},\mathcal{E}_{\ff}\right)}\left(\partial X,X\right)$.
	\end{enumerate}
\end{cor}

\subsubsection{\label{subsubsec:Symbolic-0-trace-operators}Symbolic $0$-trace
	operators}

We consider now the case of symbolic $0$-trace operators. These operators
are essentially the formal adjoints of symbolic $0$-Poisson operators,
so we shall be brief.

Define the model space $\hat{T}^{2}=\overline{\mathbb{R}_{1}^{1}}\times\overline{\mathbb{R}^{n}}$,
this time with coordinates $\left(\tilde{x},\eta\right)$. We call
its three boundary hyperfaces
\[
\of=\left\{ \tilde{x}=0\right\} ,\iif_{\eta}=\left\{ \left|\eta\right|=\infty\right\} ,\iif_{\tilde{x}}=\left\{ \tilde{x}=\infty\right\} .
\]
We then define
\[
\hat{T}_{0}^{2}=\left[\hat{T}^{2};\left\{ \tilde{x}=0,\left|\eta\right|=\infty\right\} \right],
\]
and we call its four boundary hyperfaces $\of,\ff,\iif_{\eta},\iif_{\tilde{x}}$.
\begin{defn}
	($0$-trace symbols) Let $\mathcal{E}=\left(\mathcal{E}_{\of},\mathcal{E}_{\ff}\right)$
	be a pair of index sets. We define
	\begin{align*}
		S_{0\tr,\mathcal{S}}^{-\infty,\mathcal{E}}\left(\mathbb{R}^{k};\mathbb{R}_{1}^{n+1}\right) & =\mathcal{S}\left(\mathbb{R}^{k}\right)\hat{\otimes}\mathcal{A}_{\phg}^{\left(\mathcal{E}_{\of},\mathcal{E}_{\ff}-1,\infty,\infty\right)}\left(\hat{T}_{0}^{2}\right)\\
		S_{\tr,\mathcal{S}}^{-\infty,\mathcal{E}_{\of}}\left(\mathbb{R}^{k};\mathbb{R}_{1}^{n+1}\right) & =\mathcal{S}\left(\mathbb{R}^{k}\right)\hat{\otimes}\mathcal{A}_{\phg}^{\left(\mathcal{E}_{\of},\infty,\infty\right)}\left(\hat{T}^{2}\right).
	\end{align*}
	We call the elements of $S_{0\tr,\mathcal{S}}^{-\infty,\mathcal{E}}\left(\mathbb{R}^{k};\mathbb{R}_{1}^{n+1}\right)$
	(resp.\ $S_{\tr,\mathcal{S}}^{-\infty,\mathcal{E}_{\of}}\left(\mathbb{R}^{k};\mathbb{R}_{1}^{n+1}\right)$)
	\emph{$0$-trace} (resp.\ \emph{residual trace})\emph{ }symbols.
\end{defn}

Local symbolic $0$-trace operators are defined as left quantizations
of $0$-trace operators. We define the\emph{ left trace quantization}
map 
\begin{align*}
	\Op_{L}^{\tr}:S_{0\tr,\mathcal{S}}^{-\infty,\mathcal{E}}\left(\mathbb{R}^{n};\mathbb{R}_{1}^{n+1}\right)\times\dot{C}^{\infty}\left(\overline{\mathbb{R}_{1}^{1}}\times\overline{\mathbb{R}^{n}}\right) & \to\mathcal{S}\left(\mathbb{R}^{n}\right)
\end{align*}
as
\[
\left(a\left(y;\tilde{x},\eta\right),u\left(x,y\right)\right)\mapsto\frac{1}{\left(2\pi\right)^{n}}\int e^{iy\eta}a\left(y;\tilde{x},\eta\right)\hat{u}\left(\tilde{x},\eta\right)d\tilde{x}d\eta.
\]
Again let's check that the map is well-defined. If $u\left(x,y\right)\in\dot{C}^{\infty}\left(\overline{\mathbb{R}_{1}^{1}}\times\overline{\mathbb{R}^{n}}\right)$,
then its Fourier transform $\hat{u}\left(x,\eta\right)$ in the $y$
variables is again in $\dot{C}^{\infty}\left(\overline{\mathbb{R}_{1}^{1}}\times\overline{\mathbb{R}^{n}}\right)$
since \[\dot{C}^{\infty}\left(\overline{\mathbb{R}_{1}^{1}}\times\overline{\mathbb{R}^{n}}\right)=\dot{C}^{\infty}\left(\overline{\mathbb{R}_{1}^{1}}\right)\hat{\otimes}\mathcal{S}\left(\mathbb{R}^{n}\right).\]
Now, pulling back $\hat{u}\left(\tilde{x},\eta\right)$ to $\hat{T}_{0}^{2}$,
we get an element of $\dot{C}^{\infty}\left(\hat{T}_{0}^{2}\right)$.
Multiplying with $a\left(y;\tilde{x},\eta\right)$, we get an element
of $\dot{C}^{\infty}\left(\overline{\mathbb{R}^{n}}\times\hat{T}_{0}^{2}\right)=\mathcal{S}\left(\mathbb{R}^{n}\right)\hat{\otimes}\dot{C}^{\infty}\left(\overline{\mathbb{R}_{1}^{1}}\right)\hat{\otimes}\mathcal{S}\left(\mathbb{R}^{n}\right)$.
Integrating in the $\tilde{x}$ variable, we get a function of $\left(y,\eta\right)$
in $\mathcal{S}\left(\mathbb{R}^{n}\right)\hat{\otimes}\mathcal{S}\left(\mathbb{R}^{n}\right)$.
Finally, applying $\frac{1}{\left(2\pi\right)^{n}}\int e^{iy\eta}\cdot d\eta$,
we get a function in $\mathcal{S}\left(\mathbb{R}^{n}\right)$.

As in \S\ref{subsubsec:local-physical-0-trace-and-0-Poisson}, we
denote by $\Op\left(\mathbb{R}_{1}^{n+1},\mathbb{R}^{n}\right)$ the
class of continuous linear maps \[\dot{C}^{\infty}\left(\overline{\mathbb{R}_{1}^{1}}\times\overline{\mathbb{R}^{n}}\right)\to\mathcal{S}'\left(\mathbb{R}^{n}\right).\]
Then the left trace quantization map determines a continuous linear
map
\begin{align*}
	\Op_{L}^{\tr}:S_{0\tr,\mathcal{S}}^{-\infty,\mathcal{E}}\left(\mathbb{R}^{n};\mathbb{R}_{1}^{n+1}\right) & \to\Op\left(\mathbb{R}_{1}^{n+1},\mathbb{R}^{n}\right)\\
	a\left(y;\tilde{x},\eta\right) & \mapsto\frac{1}{\left(2\pi\right)^{n}}\int e^{i\left(y-\tilde{y}\right)\eta}a\left(y;\tilde{x},\eta\right)d\eta d\tilde{x}d\tilde{y}.
\end{align*}

\begin{defn}
	(Local symbolic $0$-trace operators) We denote by \[\hat{\Psi}_{0\tr,\mathcal{S}}^{-\infty,\mathcal{E}}\left(\mathbb{R}_{1}^{n+1},\mathbb{R}^{n}\right)\]
	the subclass of $\Op\left(\mathbb{R}_{1}^{n+1},\mathbb{R}^{n}\right)$
	consisting of operators of the form $\Op_{L}^{\tr}\left(a\right)$,
	with $a$ in the space $S_{0\tr,\mathcal{S}}^{-\infty,\mathcal{E}}\left(\mathbb{R}^{n};\mathbb{R}_{1}^{n+1}\right)$.
	We call its elements local \emph{symbolic $0$-trace operators}.
\end{defn}

Using Proposition \ref{prop:Hintz-lemma} exactly as in the Poisson
case, we can describe the relationship between the classes $\hat{\Psi}_{0\tr,\mathcal{S}}^{-\infty,\left(\mathcal{E}_{\of},\mathcal{E}_{\ff}\right)}\left(\mathbb{R}_{1}^{n+1},\mathbb{R}^{n}\right)$
and $\Psi_{0\tr,\mathcal{S}}^{-\infty,\bullet}\left(\mathbb{R}_{1}^{n+1},\mathbb{R}^{n}\right)$.
\begin{cor}
	\label{cor:symbolic-0-trace-are-physical-0-trace-and-viceversa}Let
	$\mathcal{E}=\left(\mathcal{E}_{\of},\mathcal{E}_{\ff}\right)$ be
	a pair of index sets. Then:
	\begin{enumerate}
		\item the space $\hat{\Psi}_{0\tr,\mathcal{S}}^{-\infty,\left(\mathcal{E}_{\of},\infty\right)}\left(\mathbb{R}_{1}^{n+1},\mathbb{R}^{n}\right)$
		coincides with $\Psi_{\tr,\mathcal{S}}^{-\infty,\mathcal{E}_{\of}}\left(\mathbb{R}_{1}^{n+1},\mathbb{R}^{n}\right)$;
		\item the space $\hat{\Psi}_{0\tr,\mathcal{S}}^{-\infty,\left(\mathcal{E}_{\of},\mathcal{E}_{\ff}\right)}\left(\mathbb{R}_{1}^{n+1},\mathbb{R}^{n}\right)$
		is contained in $\Psi_{0\tr,\mathcal{S}}^{-\infty,\left(\mathcal{E}_{\of},\tilde{\mathcal{E}}_{\ff}\right)}\left(\mathbb{R}_{1}^{n+1},\mathbb{R}^{n}\right)$,
		where $\tilde{\mathcal{E}}_{\ff}=\mathcal{E}_{\ff}\eunion\left(\mathcal{E}_{\of}+n+1\right)$;
		\item the space $\Psi_{0\tr,\mathcal{S}}^{-\infty,\left(\mathcal{E}_{\of},\mathcal{E}_{\ff}\right)}\left(\mathbb{R}_{1}^{n+1},\mathbb{R}^{n}\right)$
		is contained in $\hat{\Psi}_{0\tr,\mathcal{S}}^{-\infty,\left(\tilde{\mathcal{E}}_{\of},\mathcal{E}_{\ff}\right)}\left(\mathbb{R}_{1}^{n+1},\mathbb{R}^{n}\right)$,
		where $\tilde{\mathcal{E}}_{\of}=\mathcal{E}_{\of}\eunion\left(\mathcal{E}_{\ff}-1\right)$.
	\end{enumerate}
\end{cor}

Let's pass to symbolic $0$-trace operators on manifolds. Again, fix
an auxiliary vector field $V$ on $X$ transverse to $\partial X$
and inward-pointing.
\begin{defn}
	(Symbolic $0$-trace operators) Let $\mathcal{E}=\left(\mathcal{E}_{\of},\mathcal{E}_{\ff}\right)$
	be a pair of index sets. We denote by $\hat{\Psi}_{0\tr}^{-\infty,\mathcal{E}}\left(X,\partial X\right)$
	the class of operators $A:\dot{C}^{\infty}\left(X\right)\to C^{-\infty}\left(\partial X\right)$
	with the following properties:
	\begin{enumerate}
		\item in the complement of any neighborhood of $\partial\Delta$, $K_{A}$
		coincides with a residual trace operator in $\Psi_{\tr}^{-\infty,\mathcal{E}_{\of}}\left(X,\partial X\right)=\mathcal{A}_{\phg}^{\mathcal{E}_{\of}}\left(\partial X\times X;\pi_{R}^{*}\mathcal{D}_{X}^{1}\right)$;
		\item for every $p\in\partial X$ and coordinates $\left(x,y\right)$ for
		$X$ centered at $p$ and compatible with $V$ (i.e.\ $Vx\equiv1$
		near $p$), $A$ coincides near $\left(p,p\right)$ with an element
		of $\hat{\Psi}_{0\tr,\mathcal{S}}^{-\infty,\mathcal{E}}\left(\mathbb{R}_{1}^{n+1},\mathbb{R}^{n}\right)$;
		more precisely, there exists a $0$-trace symbol $a\left(y;\tilde{x},\eta\right)$
		in $S_{0\tr,\mathcal{S}}^{-\infty,\mathcal{E}}\left(\mathbb{R}^{n};\mathbb{R}_{1}^{n+1}\right)$
		such that, in a neighborhood of the origin, we have
		\[
		K_{A}\equiv K\left(y;\tilde{x},y-\tilde{y}\right)d\tilde{x}d\tilde{y}=\frac{1}{\left(2\pi\right)^{n}}\int e^{i\left(y-\tilde{y}\right)\eta}a\left(y;\tilde{x},\eta\right)d\eta d\tilde{x}d\tilde{y}.
		\]
	\end{enumerate}
\end{defn}

As in the $0$-Poisson case, the discussion above and the results
of \S\ref{subsubsec:local-physical-0-trace-and-0-Poisson} imply
the following
\begin{cor}
	\label{cor:global-symbolic-0-trace-are-physical-0-trace-and-viceversa}Let
	$\mathcal{E}=\left(\mathcal{E}_{\of},\mathcal{E}_{\ff}\right)$ be
	a pair of index sets. Then:
	\begin{enumerate}
		\item the space $\hat{\Psi}_{0\tr}^{-\infty,\left(\mathcal{E}_{\of},\infty\right)}\left(X,\partial X\right)$
		coincides with $\Psi_{\tr}^{-\infty,\mathcal{E}_{\of}}\left(X,\partial X\right)$;
		\item the space $\hat{\Psi}_{0\tr}^{-\infty,\left(\mathcal{E}_{\of},\mathcal{E}_{\ff}\right)}\left(X,\partial X\right)$
		is contained in $\Psi_{0\tr}^{-\infty,\left(\mathcal{E}_{\of},\tilde{\mathcal{E}}_{\ff}\right)}\left(X,\partial X\right)$,
		where $\tilde{\mathcal{E}}_{\ff}=\mathcal{E}_{\ff}\eunion\left(\mathcal{E}_{\of}+n+1\right)$;
		\item the space $\Psi_{0\tr}^{-\infty,\left(\mathcal{E}_{\of},\mathcal{E}_{\ff}\right)}\left(X,\partial X\right)$
		is contained in $\hat{\Psi}_{0\tr}^{-\infty,\left(\tilde{\mathcal{E}}_{\of},\mathcal{E}_{\ff}\right)}\left(X,\partial X\right)$,
		where $\tilde{\mathcal{E}}_{\of}=\mathcal{E}_{\of}\eunion\left(\mathcal{E}_{\ff}-1\right)$.
	\end{enumerate}
\end{cor}

\subsubsection{\label{subsubsec:0-interior-and-0b-interior}$0$-interior and $0b$-interior
	operators}

We finally get to interior operators. As in the two previous cases,
we start by introducing a model space $\hat{M}^{2}=\overline{\mathbb{R}_{1}^{1}}\times\overline{\mathbb{R}_{1}^{1}}\times\overline{\mathbb{R}^{n}}$,
with coordinates $\left(x,\tilde{x},\eta\right)$. This space is the
frequency counterpart of the model space $M^{2}$ used in \S\ref{subsubsec:local-physical-0-calculus-extended-0-calculus}
to define the local classes $\Psi_{0,\mathcal{S}}^{-\infty,\bullet}\left(\mathbb{R}_{1}^{n+1}\right)$
and $\Psi_{0b,\mathcal{S}}^{-\infty,\bullet}\left(\mathbb{R}_{1}^{n+1}\right)$.
Analogously to $M^{2}$, we denote by $\lf$, $\rf$, $\iif_{\eta}$,
$\iif_{x}$, $\iif_{\tilde{x}}$ the faces of $\hat{M}^{2}$:
\[
\lf=\left\{ x=0\right\} ,\rf=\left\{ \tilde{x}=0\right\} 
\]
\[
\iif_{\eta}=\left\{ \left|\eta\right|=\infty\right\} ,\iif_{x}=\left\{ x=\infty\right\} ,\iif_{\tilde{x}}=\left\{ \tilde{x}=\infty\right\} .
\]
Consider now the two blow-ups
\begin{align*}
	\hat{M}_{0}^{2} & =\left[\hat{M}^{2};\left\{ x=\tilde{x}=0\right\} \cap\left\{ \left|\eta\right|=\infty\right\} \right]\\
	\hat{M}_{b}^{2} & =\left[\hat{M}^{2};\left\{ x=\tilde{x}=0\right\} \right].
\end{align*}
The space $\hat{M}_{b}^{2}$ is diffeomorphic (via the change of variables
$\eta\mapsto Y$) to the model space $M_{b}^{2}$ used in \S\ref{subsubsec:local-physical-0-calculus-extended-0-calculus}
(see Figure \ref{fig:M2&M2b}), and therefore we call its boundary
hyperfaces $\lf$, $\rf$, $\ff_{b}$, $\iif_{\eta}$, $\iif_{x}$,
$\iif_{\tilde{x}}$ in the same way. The space $\hat{M}_{0}^{2}$
is \emph{not} diffeomorphic to the space $M_{0}^{2}$ used in \S\ref{subsubsec:local-physical-0-calculus-extended-0-calculus};
nevertheless, we call again \emph{$0$-front face} $\ff_{0}$ the
new face obtained from the blow-up. The faces of $\hat{M}_{0}^{2}$
are ordered as $\lf$, $\rf$, $\ff_{0}$, $\iif_{\eta}$, $\iif_{x}$,
$\iif_{\tilde{x}}$. Finally, we consider the iterated blow-up
\begin{align*}
	\hat{M}_{0b}^{2} & =\left[\hat{M}_{0}^{2};\lf\cap\rf\right]\\
	& =\left[\hat{M}_{b}^{2};\ff_{b}\cap\left\{ \left|\eta\right|=\infty\right\} \right]
\end{align*}
and we call $\lf$, $\rf$, $\ff_{b}$, $\ff_{0}$, $\iif_{\eta}$,
$\iif_{x}$, $\iif_{\tilde{x}}$ its faces, all obtained as lifts
of the corresponding faces of $\hat{M}_{0}^{2}$ and $\hat{M}_{b}^{2}$.
The spaces $\hat{M}_{0}^{2},\hat{M}_{0b}^{2}$ are represented in
Figure \ref{fig:Mhat20&Mhat20b}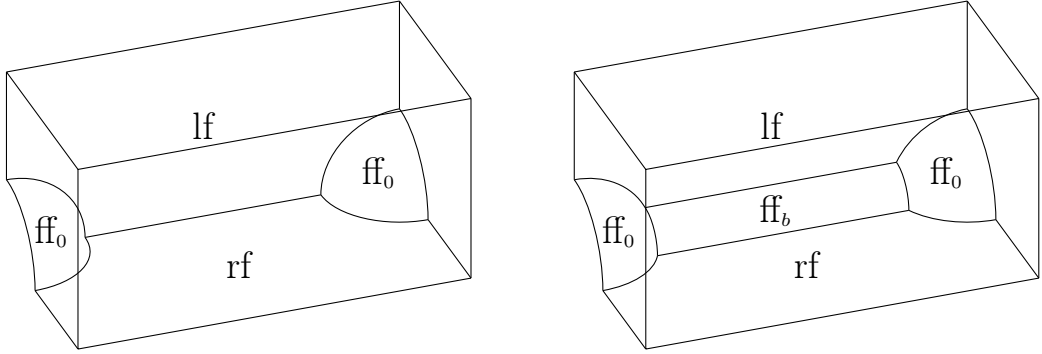
\begin{figure}
	\centering
	\caption{$\hat{M}^2_0$ and $\hat{M}^2_{0b}$}
	\label{fig:Mhat20&Mhat20b}
	
	\begin{minipage}{.49\textwidth}
		\centering
		\tikzmath{\L = 5;\R = \L /2.5;}
		\tdplotsetmaincoords{60}{160}
		\begin{tikzpicture}[tdplot_main_coords, scale = \myscale]
			
			%%% axes
			%\coordinate (O) at (0,0,0) ;
			%\coordinate (X) at (1,0,0) ;
			%\coordinate (Y) at (0,1,0) ;
			%\coordinate (Z) at (0,0,1) ;
			%\draw[->] (O) -- (X) node {$x$};
			%\draw[->] (O) -- (Y) node {$y$};
			%\draw[->] (O) -- (Z) node {$z$};
			
			%%% points
			\coordinate (A) at (-\L,0,0) ;
			\coordinate (A1) at (-\L+\R,0,0) ;
			\coordinate (A2) at (-\L,\R,0) ;
			\coordinate (A3) at (-\L,0,\R) ;
			
			\coordinate (B) at (\L,0,0) ;
			\coordinate (B1) at (\L-\R,0,0) ;
			\coordinate (B2) at (\L,\R,0) ;
			\coordinate (B3) at (\L,0,\R) ;
			
			\coordinate (C) at (-\L,\L,0) ;
			\coordinate (D) at (\L,\L,0) ;
			\coordinate (E) at (-\L,0,\L) ;
			\coordinate (F) at (\L,0,\L) ;
			\coordinate (G) at (-\L,\L,\L);
			\coordinate (H) at (\L,\L,\L);
			
			%%% segments
			\draw 
			(B3) -- (F) 
			(B2) -- (D) 
			(B1) -- (A1) 
			(A2) -- (C) 
			(A3) -- (E)
			;
			
			%%% dashed segments and arcs
			\draw
			(F) -- (E)
			(D) -- (C)
			(D) -- (H)
			(H) -- (F)
			(C) -- (G)
			(G) -- (E)
			(G) -- (H)
			;
			%\begin{scope}[canvas is zy plane at x=0]
			%	\draw[dashed] (E) arc(0:90:{3*\R});
			%	\draw[dashed] (F) arc(0:90:{3*\R});
			%\end{scope}
			
			%%% arcs
			\begin{scope}[canvas is zy plane at x=0]
				\draw (B3) arc(0:90:\R);
				\draw (A3) arc(0:90:\R);
			\end{scope}
			\begin{scope}[canvas is xy plane at z=0]
				\draw (B2) arc(90:180:\R);
				\draw (A2) arc(90:0:\R);
			\end{scope}
			\begin{scope}[canvas is xz plane at y=0]
				\draw (B3) arc(90:180:\R);
				\draw (A1) arc(0:90:\R);
			\end{scope}

			%%% text
			\node at (0,0,{\L/2}) {$\text{\Large lf}$};
			\node at (0,{\L/2},0) {$\text{\Large rf}$};
			\node at ({\L-\L/6},{\L/6},{\L/6}) {$\text{\Large ff}_0$};
			\node at ({\L/6-\L},{\L/6},{\L/6}) {$\text{\Large ff}_0$};
			
			%%% coordinates
			%\draw[->,black] ({\L/6},{\L/6},{\L/6}) -- node[below, midway] {$\eta$} (-{\L/6},{\L/6},{\L/6});
			%\draw[->,black] ({\L/6},0,{2*\L/6}) -- node[right, midway] {$\tilde x$} (1,0,4);
			%\draw[->,black] ({\L/6},{2*\L/6},0) -- node[right, midway] {$x$} (1,4,0);
			
			%\begin{scope}[shift={(-\L,0,0)}]
			%	\begin{scope}[rotate around x=-90, rotate around z=-45]
				%		\begin{scope}[canvas is xz plane at y=0]
					%			\draw[->,black] (\R,0) arc(0:70:\R) node[right] {$\tau$};
					%		\end{scope}
				%	\end{scope}
			%\end{scope}
			
			%\begin{scope}[shift={(-\L,0,0)}]
			%	\begin{scope}[rotate around z=45]
				%		\begin{scope}[canvas is xz plane at y=0]
					%			\draw[->,black] (\R,0) arc(0:70:\R) node[left] {$\tilde \tau$};
					%		\end{scope}
				%	\end{scope}
			%\end{scope}
			
		\end{tikzpicture}
		
	\end{minipage}
	\begin{minipage}{.49\textwidth}
		\centering
		\tikzmath{\L = 5;\R = \L /2.5; \u = 180/7;}
		\tdplotsetmaincoords{60}{160}
		\begin{tikzpicture}[tdplot_main_coords, scale = \myscale]
			
			%%% axes
			%\coordinate (O) at (0,0,0) ;
			%\coordinate (X) at (1,0,0) ;
			%\coordinate (Y) at (0,1,0) ;
			%\coordinate (Z) at (0,0,1) ;
			%\draw[->] (O) -- (X) node {$x$};
			%\draw[->] (O) -- (Y) node {$y$};
			%\draw[->] (O) -- (Z) node {$z$};
			
			%%% points
			\coordinate (A) at (-\L,0,0) ;
			\coordinate (A1) at (-\L+\R,0,0) ;
			\coordinate (A2) at (-\L,\R,0) ;
			\coordinate (A3) at (-\L,0,\R) ;
			\coordinate (A4) at ({-\L+\R*cos(\u)},{\R*sin(\u)},0) ;
			\coordinate (A5) at ({-\L+\R*cos(\u)},0,{\R*sin(\u)}) ;

			\coordinate (B) at (\L,0,0) ;
			\coordinate (B1) at (\L-\R,0,0) ;
			\coordinate (B2) at (\L,\R,0) ;
			\coordinate (B3) at (\L,0,\R) ;
			\coordinate (B4) at ({\L-\R*cos(\u)},{\R*sin(\u)},0) ;
			\coordinate (B5) at ({\L-\R*cos(\u)},0,{\R*sin(\u)}) ;

			\coordinate (C) at (-\L,\L,0) ;
			\coordinate (D) at (\L,\L,0) ;
			\coordinate (E) at (-\L,0,\L) ;
			\coordinate (F) at (\L,0,\L) ;
			\coordinate (G) at (-\L,\L,\L);
			\coordinate (H) at (\L,\L,\L);
			
			%%% segments
			\draw 
			(B3) -- (F) 
			(B2) -- (D)
			(A2) -- (C) 
			(A3) -- (E)
			(B4) -- (A4)
			(B5) -- (A5)
			;
			
			%%% dashed segments
			\draw
			(G) -- (H)
			(F) -- (H)
			(H) -- (D)
			(E) -- (G)
			(G) -- (C)
			(F) -- (E)
			(D) -- (C)
			;
			
			%%% arcs
			\begin{scope}[canvas is zy plane at x=0]
				\draw (B3) arc(0:90:\R);
				\draw (A3) arc(0:90:\R);
			\end{scope}
			\begin{scope}[canvas is xy plane at z=0]
				\draw (B2) arc(90:{180-\u}:\R);
				\draw (A2) arc(90:{0+\u}:\R);
			\end{scope}
			\begin{scope}[canvas is xz plane at y=0]
				\draw (B3) arc(90:{180-\u}:\R);
				\draw (A5) arc({0+\u}:90:\R);
			\end{scope}
			\begin{scope}[canvas is zy plane at x=0]
				\draw (B5) arc(0:90:{\R*sin(\u)});
				\draw (A5) arc(0:90:{\R*sin(\u)});
			\end{scope}
			
			%%% text
			\node at (0,0,{\L/2}) {$\text{\Large lf}$};
			\node at (0,{\L/2},0) {$\text{\Large rf}$};
			\node at ({\L-\L/6},{\L/6},{\L/6}) {$\text{\Large ff}_0$};
			\node at ({\L/6-\L},{\L/6},{\L/6}) {$\text{\Large ff}_0$};
			%\node at (\L,{\L/2},{\L-\L/6}) {$\text{\LARGE if}$};
			%\node at (-\L,{\L/2},{\L-\L/6}) {$\text{\LARGE if}$};
			\node at (0,{\L/15},{\L/15}) {$\text{\Large ff}_b$};
			
			%%% coordinates
			%\draw[->,red] ({\L/6},{\L/6},{\L/6}) -- node[below, midway] {$\eta$} (-{\L/6},{\L/6},{\L/6});
			%\draw[->,red] ({\L/6},0,{2*\L/6}) -- node[right, midway] {$\tilde x$} (1,0,4);
			%\draw[->,red] ({\L/6},{2*\L/6},0) -- node[right, midway] {$x$} (1,4,0);
			
			%\begin{scope}[shift={(-{1.2*\L/3},0,0)}, canvas is zy plane at x=0]
			%\draw[->,orange] ({\R*sin(\u)},0) arc(0:70:{\R*sin(\u)}) node[right, midway] {$s$};
			%\end{scope}
			
			%\begin{scope}[shift={(-\L,0,0)}]
			%\begin{scope}[rotate around x=-90, rotate around z=-45]
			%\begin{scope}[canvas is xz plane at y=0]
			%\draw[->,blue] (\R,0) arc(0:70:\R) node[right] {$t$};
			%\end{scope}
			%\end{scope}
			%\end{scope}
			
			%\begin{scope}[shift={(-\L,0,0)}]
			%\begin{scope}[rotate around z=45]
			%\begin{scope}[canvas is xz plane at y=0]
			%\draw[->,green] (\R,0) arc(0:70:\R) node[left] {$\tilde t$};
			%\end{scope}
			%\end{scope}
			%\end{scope}

		\end{tikzpicture}
		
	\end{minipage}
	
\end{figure}.
\begin{defn}
	($0$-interior and very residual symbols) Let $\mathcal{E}=\left(\mathcal{E}_{\lf},\mathcal{E}_{\rf},\mathcal{E}_{\ff_{0}}\right)$
	be a triple of index sets. We define
	\begin{align*}
		S_{0,\mathcal{S}}^{-\infty,\mathcal{E}}\left(\mathbb{R}^{k};\mathbb{R}_{2}^{n+2}\right) & =\mathcal{S}\left(\mathbb{R}^{k}\right)\hat{\otimes}\mathcal{A}_{\phg}^{\left(\mathcal{E}_{\lf},\mathcal{E}_{\rf},\mathcal{E}_{\ff_{0}}-1,\infty,\infty,\infty\right)}\left(\hat{M}_{0}^{2}\right)\\
		S_{\mathcal{S}}^{-\infty,\left(\mathcal{E}_{\lf},\mathcal{E}_{\rf}\right)}\left(\mathbb{R}^{k};\mathbb{R}_{2}^{n+2}\right) & =\mathcal{S}\left(\mathbb{R}^{k}\right)\hat{\otimes}\mathcal{A}_{\phg}^{\left(\mathcal{E}_{\lf},\mathcal{E}_{\rf},\infty,\infty,\infty\right)}\left(\hat{M}^{2}\right).
	\end{align*}
	We call the elements of $S_{0,\mathcal{S}}^{-\infty,\mathcal{E}}\left(\mathbb{R}^{k};\mathbb{R}_{2}^{n+2}\right)$ (resp. $S_{\mathcal{S}}^{-\infty,\left(\mathcal{E}_{\lf},\mathcal{E}_{\rf}\right)}\left(\mathbb{R}^{k};\mathbb{R}_{2}^{n+2}\right)$)
	\emph{$0$-interior }symbols (resp. \emph{very residual} symbols).
\end{defn}

\begin{defn}
	($0b$-interior and $b$-interior symbols) Let $\mathcal{E}=\left(\mathcal{E}_{\lf},\mathcal{E}_{\rf},\mathcal{E}_{\ff_{b}},\mathcal{E}_{\ff_{0}}\right)$
	be a tuple of index sets. We define
	\begin{align*}
		S_{0b,\mathcal{S}}^{-\infty,\mathcal{E}}\left(\mathbb{R}^{k};\mathbb{R}_{2}^{n+2}\right) & =\mathcal{S}\left(\mathbb{R}^{k}\right)\hat{\otimes}\mathcal{A}_{\phg}^{\left(\mathcal{E}_{\lf},\mathcal{E}_{\rf},\mathcal{E}_{\ff_{b}}-1,\mathcal{E}_{\ff_{0}}-1,\infty,\infty,\infty\right)}\left(\hat{M}_{0b}^{2}\right)\\
		S_{b,\mathcal{S}}^{-\infty,\left(\mathcal{E}_{\lf},\mathcal{E}_{\rf},\mathcal{E}_{\ff_{b}}\right)}\left(\mathbb{R}^{k};\mathbb{R}_{2}^{n+2}\right) & =\mathcal{S}\left(\mathbb{R}^{k}\right)\hat{\otimes}\mathcal{A}_{\phg}^{\left(\mathcal{E}_{\lf},\mathcal{E}_{\rf},\mathcal{E}_{\ff_{b}}-1,\infty,\infty,\infty\right)}\left(\hat{M}_{b}^{2}\right).
	\end{align*}
	We call the elements of $S_{0b,\mathcal{S}}^{-\infty,\mathcal{E}}\left(\mathbb{R}^{k};\mathbb{R}_{2}^{n+2}\right)$ (resp. $S_{b,\mathcal{S}}^{-\infty,\left(\mathcal{E}_{\lf},\mathcal{E}_{\rf}\right)}\left(\mathbb{R}^{k};\mathbb{R}_{2}^{n+2}\right)$)
	\emph{$0b$-interior }symbols (resp. \emph{$b$-interior }symbols).
\end{defn}

The local version of the classes of $0$-interior and $0b$-interior
operators, is again obtained by taking left quantizations of $0$-interior
and $0b$-interior symbols. The left interior quantization map
\[
\left(p\left(y;x,\tilde{x},\eta\right),u\left(x,y\right)\right)\mapsto\frac{1}{\left(2\pi\right)^{n}}\int e^{iy\eta}p\left(y;x,\tilde{x},\eta\right)\hat{u}\left(\tilde{x},\eta\right)d\eta d\tilde{x}
\]
determines continuous bilinear maps
\begin{align*}
	\Op_{L}^{\inte}:S_{0b,\mathcal{S}}^{-\infty,\mathcal{E}}\left(\mathbb{R}^{n};\mathbb{R}_{2}^{n+2}\right)\times\dot{C}^{\infty}\left(\overline{\mathbb{R}_{1}^{1}}\times\overline{\mathbb{R}^{n}}\right) & \to\mathcal{A}_{\phg}^{\left(\mathcal{E}_{\lf},\infty,\infty\right)}\left(\overline{\mathbb{R}_{1}^{1}}\times\overline{\mathbb{R}^{n}}\right)\\
	\Op_{L}^{\inte}:S_{0,\mathcal{S}}^{-\infty,\mathcal{E}}\left(\mathbb{R}^{n};\mathbb{R}_{2}^{n+2}\right)\times\dot{C}^{\infty}\left(\overline{\mathbb{R}_{1}^{1}}\times\overline{\mathbb{R}^{n}}\right) & \to\mathcal{A}_{\phg}^{\left(\mathcal{E}_{\lf},\infty,\infty\right)}\left(\overline{\mathbb{R}_{1}^{1}}\times\overline{\mathbb{R}^{n}}\right).
\end{align*}
Denote by $\Op\left(\mathbb{R}_{1}^{n+1}\right)$ the class of continuous
linear operators \[\dot{C}^{\infty}\left(\overline{\mathbb{R}_{1}^{1}}\times\overline{\mathbb{R}^{n}}\right)\to C^{-\infty}\left(\overline{\mathbb{R}_{1}^{1}}\times\overline{\mathbb{R}^{n}}\right).\]
\begin{defn}
	(Local $0$-interior and $0b$-interior operators) We define $\hat{\Psi}_{0,\mathcal{S}}^{-\infty,\mathcal{E}}\left(\mathbb{R}_{1}^{n+1}\right)$
	(resp.\ $\hat{\Psi}_{0b,\mathcal{S}}^{-\infty,\mathcal{E}}\left(\mathbb{R}_{1}^{n+1}\right)$)
	as the class of operators $\Op_{L}^{\inte}\left(p\right)$, with $p$ in $ S_{0,\mathcal{S}}^{-\infty,\mathcal{E}}\left(\mathbb{R}^{n};\mathbb{R}_{2}^{n+2}\right)$
	(resp.\ $p$ in $S_{0b,\mathcal{S}}^{-\infty,\mathcal{E}}\left(\mathbb{R}^{n};\mathbb{R}_{2}^{n+2}\right)$).
	We call its elements \emph{local $0$-interior }(resp.\ \emph{$0b$-interior})
	operators.
\end{defn}

We now discuss the relation between the classes $\hat{\Psi}_{0,\mathcal{S}}^{-\infty,\bullet}\left(\mathbb{R}_{1}^{n+1}\right)$,
$\hat{\Psi}_{0b,\mathcal{S}}^{-\infty,\bullet}\left(\mathbb{R}_{1}^{n+1}\right)$
and the classes $\Psi_{0,\mathcal{S}}^{-\infty,\bullet}\left(\mathbb{R}_{1}^{n+1}\right)$,
$\Psi_{0b,\mathcal{S}}^{-\infty,\bullet}\left(\mathbb{R}_{1}^{n+1}\right)$
described in \S\ref{subsubsec:local-physical-0-calculus-extended-0-calculus}.
Again, we need to understand how the Fourier transform in the $\eta$
variables behaves on polyhomogeneous functions on $\hat{M}_{0}^{2}$
and $\hat{M}_{0b}^{2}$. The key result is the ``$0b$-version''
of Proposition \ref{prop:Hintz-lemma}:
\begin{prop}
	\label{prop:symb-to-phys-and-phys-to-symb-0b}$ $
	\begin{enumerate}
		\item The Fourier transform in the $Y$ variables extends from $\dot{C}^{\infty}\left(M^{2}\right)$
		to a continuous linear map
		\[
		\mathcal{A}_{\phg}^{\left(\mathcal{E}_{\lf},\mathcal{E}_{\rf},\mathcal{E}_{\ff_{b}}-1,\mathcal{E}_{\ff_{0}}-n-1,\infty,\infty,\infty\right)}\left(M_{0b}^{2}\right)\to\mathcal{A}_{\phg}^{\left(\mathcal{E}_{\lf},\mathcal{E}_{\rf},\tilde{\mathcal{E}}_{\ff_{b}}-1,\mathcal{E}_{\ff_{0}}-1,\infty,\infty,\infty\right)}\left(\hat{M}_{0b}^{2}\right),
		\]
		where $\tilde{\mathcal{E}}_{\ff_{b}}=\mathcal{E}_{\ff_{b}}\eunion\mathcal{E}_{\ff_{0}}$.
		\item The inverse Fourier transform in the $\eta$ variables extends from
		$\dot{C}^{\infty}\left(\hat{M}^{2}\right)$ to a continuous linear
		map 
		\[
		\mathcal{A}_{\phg}^{\left(\mathcal{E}_{\lf},\mathcal{E}_{\rf},\mathcal{E}_{\ff_{b}}-1,\mathcal{E}_{\ff_{0}}-1,\infty,\infty,\infty\right)}\left(\hat{M}_{0b}^{2}\right)\to\mathcal{A}_{\phg}^{\left(\mathcal{E}_{\lf},\mathcal{E}_{\rf},\mathcal{E}_{\ff_{b}}-1,\tilde{\mathcal{E}}_{\ff_{0}}-n-1,\infty,\infty,\infty\right)}\left(M_{0b}^{2}\right)
		\]
		where $\tilde{\mathcal{E}}_{\ff_{0}}=\mathcal{E}_{\ff_{0}}\eunion\left(\mathcal{E}_{\ff_{b}}+n\right)$.
	\end{enumerate}
\end{prop}

\begin{proof}
	The result follows easily by Proposition \ref{prop:Hintz-lemma} if
	we use alternative compactifications $\hat{M}_{0b,\text{alt}}^{2}$
	and $M_{0b,\text{alt}}^{2}$ of the open manifolds with corners $\hat{M}_{0b}^{2}\backslash\left(\iif_{x}\cup\iif_{\tilde{x}}\right)$
	and $M_{0b}^{2}\backslash\left(\iif_{x}\cup\iif_{\tilde{x}}\right)$.
	These compactifications are obtained as follows. Instead of starting
	with the model spaces $\hat{M}^{2}$ and $M^{2}$, both diffeomorphic
	to $\overline{\mathbb{R}_{1}^{1}}\times\overline{\mathbb{R}_{1}^{1}}\times\overline{\mathbb{R}^{n}}$,
	we start with the model spaces $\hat{M}_{\text{alt}}^{2}$ and $M_{\text{alt}}^{2}$
	diffeomorphic to $\overline{\mathbb{R}_{2}^{2}}\times\overline{\mathbb{R}^{n}}$.
	We can now blow up the locus $x=\tilde{x}=0$ and obtain the spaces
	$\hat{M}_{b,\text{alt}}^{2}$ and $M_{b,\text{alt}}^{2}$, both diffeomorphic
	to $\left[\overline{\mathbb{R}_{2}^{2}};0\right]\times\overline{\mathbb{R}^{n}}$.
	Replacing the coordinates $x,\tilde{x}$ with polar coordinates $\rho,\theta$
	for $\overline{\mathbb{R}_{2}^{2}}$, i.e.\ setting
	\[
	\rho=\sqrt{x^{2}+\tilde{x}^{2}},\theta_{0}=\frac{x}{\rho},\theta_{1}=\frac{\tilde{x}}{\rho},
	\]
	we get a diffeomorphism $\left[\overline{\mathbb{R}_{2}^{2}};0\right]\simeq\left[0,1\right]\times\overline{\mathbb{R}_{1}^{1}}$
	(the front face is the quarter circle \[S_{2}^{1}=\left\{ \left(\theta_{0},\theta_{1}\right)\in S^{1}:\theta_{0}\geq0,\theta_{1}\geq0\right\} ,\]
	and we identify it with $\left[0,1\right]$ via the projection $\left(\theta_{0},\theta_{1}\right)\mapsto\theta_{0}$).
	The space $\hat{M}_{0b,\text{alt}}^{2}$ is obtained from $\hat{M}_{b,\text{alt}}^{2}=\left[0,1\right]\times\overline{\mathbb{R}_{1}^{1}}\times\overline{\mathbb{R}^{n}}$
	with coordinates $\theta_{0},\rho,\eta$ by blowing up the locus $\rho=0,\left|\eta\right|=\infty$.
	Therefore, $\hat{M}_{0b,\text{alt}}\simeq\left[0,1\right]\times\hat{P}_{0}^{2}$.
	Similarly, the space $M_{0b,\text{alt}}^{2}$ is obtained from $M_{b,\text{alt}}^{2}=\left[0,1\right]\times\overline{\mathbb{R}_{1}^{1}}\times\overline{\mathbb{R}^{n}}$
	with coordinates $\theta_{0},\rho,Y$ by blowing up the locus $\rho=0,Y=0$.
	This time, we get $M_{0b,\text{alt}}^{2}\simeq\left[0,1\right]\times P_{0}^{2}$.
	These spaces are represented in Figure \ref{fig:0b-variants}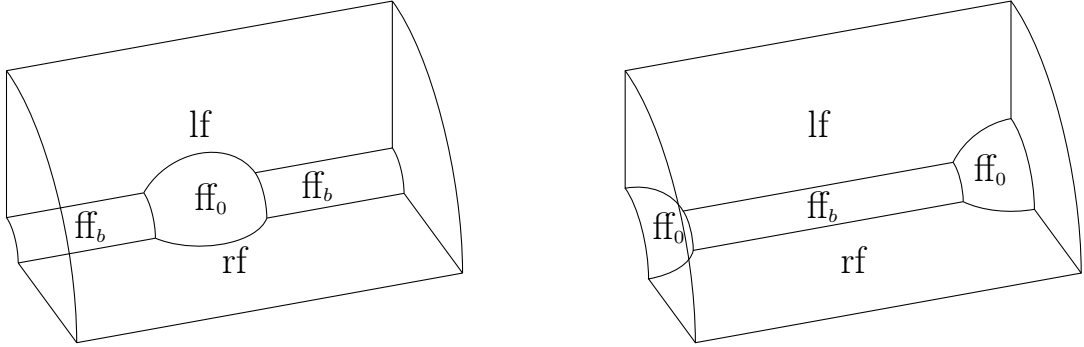
\begin{figure}
		
		\centering
		\caption{${M}^{2}_{{0b},\text{alt}}$ and $\hat{M}^{2}_{{0b},\text{alt}}$}
		\label{fig:0b-variants}
		
		\begin{minipage}{.49\textwidth}
			\centering
			\tikzmath{\L = 2.7;\R = \L /3; \u = 180/6;}
			\tdplotsetmaincoords{60}{160}
			\begin{tikzpicture}[tdplot_main_coords,  scale = 2*\myscale]
				
				%%% axes
				%\coordinate (O) at (0,0,0) ;
				%\coordinate (X) at (1,0,0) ;
				%\coordinate (Y) at (0,1,0) ;
				%\coordinate (Z) at (0,0,1) ;
				%\draw[->] (O) -- (X) node {$x$};
				%\draw[->] (O) -- (Y) node {$y$};
				%\draw[->] (O) -- (Z) node {$z$};
				
				%%% points
				\coordinate (A) at (-\L,0,0) ;
				\coordinate (A1) at (-\L+\R,0,0) ;
				\coordinate (A4) at (-\L,{\R*sin(\u)},0) ;
				\coordinate (A5) at (-\L,0,{\R*sin(\u)}) ;
				\coordinate (A10) at ({-\R*cos(\u)},{\R*sin(\u)},0) ;
				\coordinate (A11) at ({-\R*cos(\u)},0,{\R*sin(\u)}) ;
				
				\coordinate (B) at (\L,0,0) ;
				\coordinate (B3) at (\L,0,\R) ;
				\coordinate (B4) at (\L,{\R*sin(\u)},0) ;
				\coordinate (B5) at (\L,0,{\R*sin(\u)}) ;
				\coordinate (B10) at ({\R*cos(\u)},{\R*sin(\u)},0) ;
				\coordinate (B11) at ({\R*cos(\u)},0,{\R*sin(\u)}) ;
				
				\coordinate (C) at (-\L,\L,0) ;
				\coordinate (D) at (\L,\L,0) ;
				\coordinate (E) at (-\L,0,\L) ;
				\coordinate (F) at (\L,0,\L) ;
				\coordinate (G) at (-\L,\L,\L);
				\coordinate (H) at (\L,\L,\L);
				
				%%% segments
				\draw 
				(C) -- (D)
				(E) -- (F)
				(B5) -- (F) 
				(B4) -- (D)
				(A4) -- (C) 
				(A5) -- (E)
				(B4) -- (B10)
				(B5) -- (B11)
				(A4) -- (A10)
				(A5) -- (A11)
				;
				
				%%% arcs
				\begin{scope}[canvas is xy plane at z=0]
				\end{scope}
				\begin{scope}[canvas is xz plane at y=0]
				\end{scope}
				\begin{scope}[canvas is zy plane at x=0]
					\draw (B5) arc(0:90:{\R*sin(\u)});
					\draw (A5) arc(0:90:{\R*sin(\u)});
					\draw (E) arc(0:90:{\L});
					\draw (F) arc(0:90:{\L});
					\draw (B11) arc(0:90:{\R*sin(\u)});
					\draw (A11) arc(0:90:{\R*sin(\u)});
				\end{scope}
				\begin{scope}[canvas is xy plane at z=0]
					\draw ({\R*cos(\u)},{\R*sin(\u)}) arc(\u:{180-\u}:\R);
				\end{scope}
				\begin{scope}[canvas is xz plane at y=0]
					\draw ({\R*cos(\u)},{\R*sin(\u)}) arc(\u:{180-\u}:\R);
				\end{scope}

				%%% text
				\node at (0,0,{\L/2}) {$\text{\Large lf}$};
				\node at (0,{\L/2},0) {$\text{\Large rf}$};
				\node at (0,{\L/6},{\L/6}) {$\text{\Large ff}_0$};
				\node at ({\L/1.7},{\L/15},{\L/15}) {$\text{\Large ff}_b$};
				\node at ({-\L/1.7},{\L/15},{\L/15}) {$\text{\Large ff}_b$};
				
			\end{tikzpicture}
			
		\end{minipage}
		\begin{minipage}{.49\textwidth}
			\centering
			\tikzmath{
				\L = 2.7;
				\R = \L /3;
				\u = 180/7;
			}
			\tdplotsetmaincoords{60}{160}
			\begin{tikzpicture}[tdplot_main_coords, scale = 2*\myscale]
				
				%%% axes
				%\coordinate (O) at (0,0,0) ;
				%\coordinate (X) at (1,0,0) ;
				%\coordinate (Y) at (0,1,0) ;
				%\coordinate (Z) at (0,0,1) ;
				%\draw[->] (O) -- (X) node {$x$};
				%\draw[->] (O) -- (Y) node {$y$};
				%\draw[->] (O) -- (Z) node {$z$};
				
				%%% points
				\coordinate (A) at (-\L,0,0) ;
				\coordinate (A1) at (-\L+\R,0,0) ;
				\coordinate (A2) at (-\L,\R,0) ;
				\coordinate (A3) at (-\L,0,\R) ;
				\coordinate (A4) at ({-\L+\R*cos(\u)},{\R*sin(\u)},0) ;
				\coordinate (A5) at ({-\L+\R*cos(\u)},0,{\R*sin(\u)}) ;

				\coordinate (B) at (\L,0,0) ;
				\coordinate (B1) at (\L-\R,0,0) ;
				\coordinate (B2) at (\L,\R,0) ;
				\coordinate (B3) at (\L,0,\R) ;
				\coordinate (B4) at ({\L-\R*cos(\u)},{\R*sin(\u)},0) ;
				\coordinate (B5) at ({\L-\R*cos(\u)},0,{\R*sin(\u)}) ;

				\coordinate (C) at (-\L,\L,0) ;
				\coordinate (D) at (\L,\L,0) ;
				\coordinate (E) at (-\L,0,\L) ;
				\coordinate (F) at (\L,0,\L) ;
				\coordinate (G) at (-\L,\L,\L);
				\coordinate (H) at (\L,\L,\L);
				
				%%% right face
				\draw
				(C) -- (D) -- (B2) arc(90:{180-\u}:\R) -- (A4) arc({\u}:90:\R) -- cycle;
				
				%%% left face
				\begin{scope}[canvas is zx plane at y=0]
					\draw
					(F) -- (E) -- (A3) arc(0:{90-\u}:\R) -- (B5) arc({270+\u}:360:\R) --  cycle;
				\end{scope}
				
				%%% remaining arcs
				\begin{scope}[canvas is zy plane at x=0]
					\draw (B3) arc(0:90:\R);
					\draw (A3) arc(0:90:\R);
					\draw (B5) arc(0:90:{\R*sin(\u)});
					\draw (A5) arc(0:90:{\R*sin(\u)});
					\draw (F) arc(0:90:\L);
					\draw (E) arc(0:90:\L);
					
				\end{scope}
				
				%%% additional dashed segments
				%\draw[dashed]
				%	(-\L, {\R*cos(1*90/5)},{\R*sin(1*90/5)}) -- (-\L,{\L*cos(1*90/5)},{\L*sin(1*90/5)})
				%	(-\L, {\R*cos(2*90/5)},{\R*sin(2*90/5)}) -- (-\L,{\L*cos(2*90/5)},{\L*sin(2*90/5)})
				%	(-\L, {\R*cos(3*90/5)},{\R*sin(3*90/5)}) -- (-\L,{\L*cos(3*90/5)},{\L*sin(3*90/5)})
				%	(-\L, {\R*cos(4*90/5)},{\R*sin(4*90/5)}) -- (-\L,{\L*cos(4*90/5)},{\L*sin(4*90/5)});
				
				%\draw[dashed]
				%	(\L, {\R*cos(1*90/5)},{\R*sin(1*90/5)}) -- (\L,{\L*cos(1*90/5)},{\L*sin(1*90/5)})
				%	(\L, {\R*cos(2*90/5)},{\R*sin(2*90/5)}) -- (\L,{\L*cos(2*90/5)},{\L*sin(2*90/5)})
				%	(\L, {\R*cos(3*90/5)},{\R*sin(3*90/5)}) -- (\L,{\L*cos(3*90/5)},{\L*sin(3*90/5)})
				%	(\L, {\R*cos(4*90/5)},{\R*sin(4*90/5)}) -- (\L,{\L*cos(4*90/5)},{\L*sin(4*90/5)});

				%%% text
				\node at (0,0,{\L/2}) {$\text{\Large lf}$};
				\node at (0,{\L/2},0) {$\text{\Large rf}$};
				\node at ({\L-\L/6},{\L/6},{\L/6}) {$\text{\Large ff}_0$};
				\node at ({\L/6-\L},{\L/6},{\L/6}) {$\text{\Large ff}_0$};
				%\node at (\L,{\L/2},{\L-\L/6}) {$\text{\LARGE if}$};
				%\node at (-\L,{\L/2},{\L-\L/6}) {$\text{\LARGE if}$};
				\node at (0,{\L/15},{\L/15}) {$\text{\Large ff}_b$};
				
				%%% coordinates
				%\draw[->,red] ({\L/6},{\L/6},{\L/6}) -- node[below, midway] {$\eta$} (-{\L/6},{\L/6},{\L/6});
				%\draw[->,red] ({\L/6},0,{2*\L/6}) -- node[right, midway] {$\tilde x$} (1,0,4);
				%\draw[->,red] ({\L/6},{2*\L/6},0) -- node[right, midway] {$x$} (1,4,0);
				
				%\begin{scope}[shift={(-{1.2*\L/3},0,0)}, canvas is zy plane at x=0]
				%\draw[->,orange] ({\R*sin(\u)},0) arc(0:70:{\R*sin(\u)}) node[right, midway] {$s$};
				%\end{scope}
				
				%\begin{scope}[shift={(-\L,0,0)}]
				%\begin{scope}[rotate around x=-90, rotate around z=-45]
				%\begin{scope}[canvas is xz plane at y=0]
				%\draw[->,blue] (\R,0) arc(0:70:\R) node[right] {$t$};
				%\end{scope}
				%\end{scope}
				%\end{scope}
				
				%\begin{scope}[shift={(-\L,0,0)}]
				%\begin{scope}[rotate around z=45]
				%\begin{scope}[canvas is xz plane at y=0]
				%\draw[->,green] (\R,0) arc(0:70:\R) node[left] {$\tilde t$};
				%\end{scope}
				%\end{scope}
				%\end{scope}

			\end{tikzpicture}
		\end{minipage}
		
	\end{figure}. The faces of $M_{0b,\text{alt}}^{2}$ are $\lf,\rf,\ff_{b},\ff_{0},\iif_{Y}$
	with obvious meanings, while the faces $\iif_{x},\iif_{\tilde{x}}$
	of $M_{0b}^{2}$ are replaced by the single face $\iif_{\rho}$ corresponding
	to the locus $\rho=\infty$. Similarly, we call the faces of $\hat{M}_{0b,\text{alt}}$
	$\lf,\rf,\ff_{b},\ff_{0},\iif_{\eta},\iif_{\rho}$. Now, the factor
	$\left[0,1\right]$ does not affect the Fourier transform in the $Y$
	variables and its inverse in the $\eta$ variables, so from Proposition
	\ref{prop:Hintz-lemma} we immediately get the inclusions
	\begin{align*}
		\mathcal{A}_{\phg}^{\left(\mathcal{E}_{\lf},\mathcal{E}_{\rf},\mathcal{E}_{\ff_{b}}-1,\mathcal{E}_{\ff_{0}}-n-1,\infty,\infty\right)}\left(M_{0b,\text{alt}}^{2}\right) & \to\mathcal{A}_{\phg}^{\left(\mathcal{E}_{\lf},\mathcal{E}_{\rf},\tilde{\mathcal{E}}_{\ff_{b}}-1,\mathcal{E}_{\ff_{0}}-1,\infty,\infty,\infty\right)}\left(\hat{M}_{0b,\text{alt}}^{2}\right)\\
		\mathcal{A}_{\phg}^{\left(\mathcal{E}_{\lf},\mathcal{E}_{\rf},\mathcal{E}_{\ff_{b}}-1,\mathcal{E}_{\ff_{0}}-1,\infty,\infty\right)}\left(\hat{M}_{0b,\text{alt}}^{2}\right) & \to\mathcal{A}_{\phg}^{\left(\mathcal{E}_{\lf},\mathcal{E}_{\rf},\mathcal{E}_{\ff_{b}}-1,\tilde{\mathcal{E}}_{\ff_{0}}-n-1,\infty,\infty,\infty\right)}\left(M_{0b,\text{alt}}^{2}\right)
	\end{align*}
	where $\tilde{\mathcal{E}}_{\ff_{b}}=\mathcal{E}_{\ff_{b}}\eunion\mathcal{E}_{\ff_{0}}$
	and $\tilde{\mathcal{E}}_{\ff_{0}}=\mathcal{E}_{\ff_{0}}\eunion\left(\mathcal{E}_{\ff_{b}}+n\right)$.
	Finally, the proof is concluded by observing that, for any quadruple
	$\left(\mathcal{F}_{\lf},\mathcal{F}_{\rf},\mathcal{F}_{\ff_{b}},\mathcal{F}_{\ff_{0}}\right)$,
	we have canonical identifications
	\begin{align*}
		\mathcal{A}_{\phg}^{\left(\mathcal{F}_{\lf},\mathcal{F}_{\rf},\mathcal{F}_{\ff_{b}},\mathcal{F}_{\ff_{0}},\infty,\infty\right)}\left(M_{0b,\text{alt}}^{2}\right) & =\mathcal{A}_{\phg}^{\left(\mathcal{F}_{\lf},\mathcal{F}_{\rf},\mathcal{F}_{\ff_{b}},\mathcal{F}_{\ff_{0}},\infty,\infty,\infty\right)}\left(M_{0b}^{2}\right)\\
		\mathcal{A}_{\phg}^{\left(\mathcal{F}_{\lf},\mathcal{F}_{\rf},\mathcal{F}_{\ff_{b}},\mathcal{F}_{\ff_{0}},\infty,\infty\right)}\left(\hat{M}_{0b,\text{alt}}^{2}\right) & =\mathcal{A}_{\phg}^{\left(\mathcal{F}_{\lf},\mathcal{F}_{\rf},\mathcal{F}_{\ff_{b}},\mathcal{F}_{\ff_{0}},\infty,\infty,\infty\right)}\left(\hat{M}_{0b}^{2}\right).
	\end{align*}
\end{proof}
We can now use the previous result to describe the relation between
$\hat{\Psi}_{0b,\mathcal{S}}^{-\infty,\bullet}\left(\mathbb{R}_{1}^{n+1}\right)$
and $\Psi_{0b,\mathcal{S}}^{-\infty,\bullet}\left(\mathbb{R}_{1}^{n+1}\right)$:
\begin{cor}
	\label{cor:0b-interior-are-extended-0-calc-and-viceversa}Let $\mathcal{E}=\left(\mathcal{E}_{\lf},\mathcal{E}_{\rf},\mathcal{E}_{\ff_{b}},\mathcal{E}_{\ff_{0}}\right)$
	be a quadruple of index sets. Then:
	\begin{enumerate}
		\item the space $\hat{\Psi}_{0b,\mathcal{S}}^{-\infty,\left(\mathcal{E}_{\lf},\mathcal{E}_{\rf},\mathcal{E}_{\ff_{b}},\infty\right)}\left(\mathbb{R}_{1}^{n+1}\right)$
		coincides with $\Psi_{b,\mathcal{S}}^{-\infty,\left(\mathcal{E}_{\lf},\mathcal{E}_{\rf},\mathcal{E}_{\ff_{b}}\right)}\left(\mathbb{R}_{1}^{n+1}\right)$;
		\item the space $\Psi_{0b,\mathcal{S}}^{-\infty,\mathcal{E}}\left(\mathbb{R}_{1}^{n+1}\right)$
		is included in $\hat{\Psi}_{0b,\mathcal{S}}^{-\infty,\left(\mathcal{E}_{\lf},\mathcal{E}_{\rf},\tilde{\mathcal{E}}_{\ff_{b}},\mathcal{E}_{\ff_{0}}\right)}\left(\mathbb{R}_{1}^{n+1}\right)$,
		where $\tilde{\mathcal{E}}_{\ff_{b}}=\mathcal{E}_{\ff_{b}}\eunion\mathcal{E}_{\ff_{0}}$;
		\item the space $\hat{\Psi}_{0b,\mathcal{S}}^{-\infty,\mathcal{E}}\left(\mathbb{R}_{1}^{n+1}\right)$
		is included in $\Psi_{0b,\mathcal{S}}^{-\infty,\left(\mathcal{E}_{\lf},\mathcal{E}_{\rf},\mathcal{E}_{\ff_{b}},\tilde{\mathcal{E}}_{\ff_{0}}\right)}\left(\mathbb{R}_{1}^{n+1}\right)$,
		where $\tilde{\mathcal{E}}_{\ff_{0}}=\mathcal{E}_{\ff_{0}}\eunion\left(\mathcal{E}_{\ff_{b}}+n\right)$.
	\end{enumerate}
\end{cor}

\begin{proof}
	Points 2 and 3 follow immediately from Proposition \ref{prop:symb-to-phys-and-phys-to-symb-0b}.
	Point 1 follows by observing that the space $S_{0b,\mathcal{S}}^{-\infty,\left(\mathcal{E}_{\lf},\mathcal{E}_{\rf},\mathcal{E}_{\ff_{b}},\infty\right)}\left(\mathbb{R}^{k};\mathbb{R}_{2}^{n+2}\right)$
	coincides with $S_{b,\mathcal{S}}^{-\infty,\left(\mathcal{E}_{\lf},\mathcal{E}_{\rf},\mathcal{E}_{\ff_{b}}\right)}\left(\mathbb{R}^{k};\mathbb{R}_{2}^{n+2}\right)$,
	and that the Fourier transform in the $Y$ variables determines an
	isomorphism between $\mathcal{A}_{\phg}^{\left(\mathcal{E}_{\lf},\mathcal{E}_{\rf},\mathcal{E}_{\ff_{b}},\infty,\infty,\infty\right)}\left(M_{b}^{2}\right)$
	and $\mathcal{A}_{\phg}^{\left(\mathcal{E}_{\lf},\mathcal{E}_{\rf},\mathcal{E}_{\ff_{b}},\infty,\infty,\infty\right)}\left(\hat{M}_{b}^{2}\right)$.
\end{proof}
Let's now pass to $0$-interior operators.
\begin{lem}
	\label{lem:0-interior-symbols-are-0b-interior-symbols}Let $\mathcal{E}=\left(\mathcal{E}_{\lf},\mathcal{E}_{\rf},\mathcal{E}_{\ff_{0}}\right)$
	be a triple of index sets. Then we have an inclusion
	\[
	\hat{\Psi}_{0,\mathcal{S}}^{-\infty,\mathcal{E}}\left(\mathbb{R}_{1}^{n+1}\right)\subseteq\hat{\Psi}_{0b,\mathcal{S}}^{-\infty,\left(\mathcal{E}_{\lf},\mathcal{E}_{\rf},\mathcal{E}_{\lf}+\mathcal{E}_{\rf}+1,\mathcal{E}_{\ff_{0}}\right)}\left(\mathbb{R}_{1}^{n+1}\right).
	\]
\end{lem}

\begin{proof}
	It suffices to observe that $S_{0,\mathcal{S}}^{-\infty,\mathcal{E}}\left(\mathbb{R}^{k};\mathbb{R}_{2}^{n+2}\right)$
	is included in \[S_{0b,\mathcal{S}}^{-\infty,\left(\mathcal{E}_{\lf},\mathcal{E}_{\rf},\mathcal{E}_{\lf}+\mathcal{E}_{\rf}+1,\mathcal{E}_{\ff_{0}}\right)}\left(\mathbb{R}^{k};\mathbb{R}_{2}^{n+2}\right).\]
	This is a straightforward application of the pull-back Theorem.
\end{proof}
This lemma, together with the previous corollary, implies that $0$-interior
operators are in the extended $0$-calculus. However, we can prove
a stronger statement: $0$-interior operators are in fact in the $0$-calculus.
The key is the following improvement of Proposition \ref{prop:symb-to-phys-and-phys-to-symb-0b}:
\begin{prop}
	\label{prop:improved-hintz-lemma-for-0-interior}Let \[p\left(x,\tilde{x},\eta\right)\in\mathcal{A}_{\phg}^{\left(\mathcal{E}_{\lf},\mathcal{E}_{\rf},\mathcal{E}_{\ff_{0}}-1,\infty,\infty,\infty\right)}\left(\hat{M}_{0}^{2}\right).\]
	Then its inverse Fourier transform is in
	\[\mathcal{A}_{\phg}^{\left(\mathcal{E}_{\lf},\mathcal{E}_{\rf},\tilde{\mathcal{E}}_{\ff_{0}}-n-1,\infty,\infty,\infty\right)}\left(M_{0}^{2}\right),\]
	where $\tilde{\mathcal{E}}_{\ff_{0}}=\mathcal{E}_{\ff_{0}}\eunion\left(\mathcal{E}_{\lf}+\mathcal{E}_{\rf}+n+1\right)$.
\end{prop}

\begin{proof}
	We first prove a preliminary fact. Let $a,b,c\in\mathbb{R}$. Then
	we have
	\[
	\mathcal{A}^{\left(a,b,a+b,c,\infty,\infty,\infty\right)}\left(M_{0b}^{2}\right)\subseteq\mathcal{A}^{\left(a,b,c,\infty,\infty,\infty\right)}\left(M_{0}^{2}\right).
	\]
	To see this, let $r_{\lf},r_{\rf}$ be boundary defining functions
	for the left and right faces of\emph{ $M_{0}^{2}$}. By the pull-back
	Theorem, $r_{\lf}$ lifts via the blow-down map $M_{0b}^{2}\to M_{0}^{2}$
	to a product $\tilde{r}_{\lf}\tilde{r}_{\ff_{b}}$, where $\tilde{r}_{\lf}$
	and $\tilde{r}_{\ff_{b}}$ are boundary defining functions for the
	faces $\lf,\ff_{b}$ of $M_{0b}^{2}$. Similarly, $r_{\rf}$ lifts
	to a product $\tilde{r}_{\rf}\tilde{r}_{\ff_{b}}$, where $\tilde{r}_{\rf}$
	is a boundary defining function for the right face of $M_{0b}^{2}$.
	Call now $r_{\ff_{0}}$ a boundary defining function for the face
	$\ff_{0}$ of $M_{0}^{2}$. Then $r_{\ff_{0}}$ lifts to a boundary
	defining function for the face $\ff_{0}$ of $M_{0b}^{2}$, which
	we denote by $r_{\ff_{0}}$. From this discussion, it follows that
	if $f\in\mathcal{A}^{\left(a,b,a+b,c,\infty,\infty,\infty\right)}\left(M_{0b}^{2}\right)$
	then 
	\[
	r_{\lf}^{-a}r_{\rf}^{-b}r_{\ff_{0}}^{-c}f\in\mathcal{A}^{\left(0,0,0,0,\infty,\infty,\infty\right)}\left(M_{0b}^{2}\right).
	\]
	This implies that $r_{\lf}^{-a}r_{\rf}^{-b}r_{\ff_{0}}^{-c}f\in L^{\infty}\left(M_{0}^{2}\right)$.
	Now, let $V_{1},...,V_{k}$ be $b$-vector fields on $M_{0}^{2}$.
	Since $M_{0b}^{2}$ is obtained from $M_{0}^{2}$ by blowing up a
	corner, the $V_{1},...,V_{k}$ lift to $b$-vector fields on $M_{0b}^{2}$.
	It follows that
	\begin{align*}
		\left(V_{1}\cdots V_{k}\right)\left(r_{\lf}^{-a}r_{\rf}^{-b}r_{\ff_{0}}^{-c}f\right) & \in\mathcal{A}^{\left(0,0,0,0,\infty,\infty,\infty\right)}\left(M_{0b}^{2}\right)\\
		& \subseteq L^{\infty}\left(M_{0}^{2}\right).
	\end{align*}
	This discussion shows that $r_{\lf}^{-a}r_{\rf}^{-b}r_{\ff_{0}}^{-c}f\in\mathcal{A}^{\left(0,0,0,\infty,\infty,\infty\right)}\left(M_{0}^{2}\right)$,
	which is equivalent to
	\[
	f\in\mathcal{A}^{\left(a,b,c,\infty,\infty,\infty\right)}\left(M_{0}^{2}\right)
	\]
	as claimed.
	
	We can then prove the lemma. Let $p\left(x,\tilde{x},\eta\right)\in\mathcal{A}_{\phg}^{\left(\mathcal{E}_{\lf},\mathcal{E}_{\rf},\mathcal{E}_{\ff_{0}}-1,\infty,\infty,\infty\right)}\left(\hat{M}_{0}^{2}\right)$.
	By the pull-back Theorem we have
	\[
	p\left(x,\tilde{x},\eta\right)\in\mathcal{A}_{\phg}^{\left(\mathcal{E}_{\lf},\mathcal{E}_{\rf},\mathcal{E}_{\lf}+\mathcal{E}_{\rf},\mathcal{E}_{\ff_{0}}-1,\infty,\infty,\infty\right)}\left(\hat{M}_{0b}^{2}\right),
	\]
	and by Proposition \ref{prop:symb-to-phys-and-phys-to-symb-0b} we
	have
	\begin{align*}
		P\left(x,\tilde{x},Y\right) & \in\mathcal{A}_{\phg}^{\left(\mathcal{E}_{\lf},\mathcal{E}_{\rf},\mathcal{E}_{\lf}+\mathcal{E}_{\rf},\tilde{\mathcal{E}}_{\ff_{0}}-n-1,\infty,\infty,\infty\right)}\left(M_{0b}^{2}\right),
	\end{align*}
	where $\tilde{\mathcal{E}}_{\ff_{0}}=\mathcal{E}_{\ff_{0}}\eunion\left(\mathcal{E}_{\lf}+\mathcal{E}_{\rf}+n+1\right)$.
	Thus, if $\Re\left(\mathcal{E}_{\lf}\right)>a$, $\Re\left(\mathcal{E}_{\rf}\right)>b$
	and $\Re\left(\tilde{\mathcal{E}}_{\ff_{0}}-n-1\right)>c$, we have
	\[
	P\left(x,\tilde{x},Y\right)\in\mathcal{A}^{\left(a,b,c,\infty,\infty,\infty\right)}\left(M_{0}^{2}\right).
	\]
	Now define
	\[
	Q_{\lambda}=\prod_{\begin{smallmatrix}\left(\alpha,l\right)\in\mathcal{E}_{\lf}\\
			\Re\left(\alpha\right)\leq\lambda
	\end{smallmatrix}}\left(x\partial_{x}-\alpha\right)
	\]
	and
	\[
	\tilde{Q}_{\lambda}=\prod_{\begin{smallmatrix}\left(\alpha,l\right)\in\mathcal{E}_{\rf}\\
			\Re\left(\alpha\right)\leq\lambda
	\end{smallmatrix}}\left(\tilde{x}\partial_{\tilde{x}}-\alpha\right).
	\]
	We then have
	\begin{align*}
		Q_{\lambda}p & \in\mathcal{A}_{\phg}^{\left(\mathcal{E}_{\lf}\left(\lambda\right),\mathcal{E}_{\rf},\mathcal{E}_{\lf}\left(\lambda\right)+\mathcal{E}_{\rf},\mathcal{E}_{\ff_{0}}-1,\infty,\infty,\infty\right)}\left(\hat{M}_{0b}^{2}\right)\\
		\tilde{Q}_{\lambda}p & \in\mathcal{A}_{\phg}^{\left(\mathcal{E}_{\lf},\mathcal{E}_{\rf}\left(\lambda\right),\mathcal{E}_{\lf}+\mathcal{E}_{\rf}\left(\lambda\right),\mathcal{E}_{\ff_{0}}-1,\infty,\infty,\infty\right)}\left(\hat{M}_{0b}^{2}\right)
	\end{align*}
	where $\mathcal{E}_{\lf}\left(\lambda\right)$ is the set of pairs
	$\left(\alpha,l\right)\in\mathcal{E}_{\lf}$ with $\Re\left(\alpha\right)>\lambda$,
	and $\mathcal{E}_{\rf}\left(\lambda\right)$ is defined similarly.
	It follows that
	\begin{align*}
		Q_{\lambda}P & \in\mathcal{A}_{\phg}^{\left(\mathcal{E}_{\lf}\left(\lambda\right),\mathcal{E}_{\rf},\mathcal{E}_{\lf}\left(\lambda\right)+\mathcal{E}_{\rf},\tilde{\mathcal{E}}_{\ff_{0}}-n-1,\infty,\infty,\infty\right)}\left(M_{0b}^{2}\right)\\
		\tilde{Q}_{\lambda}P & \in\mathcal{A}_{\phg}^{\left(\mathcal{E}_{\lf},\mathcal{E}_{\rf}\left(\lambda\right),\mathcal{E}_{\lf}+\mathcal{E}_{\rf}\left(\lambda\right),\tilde{\mathcal{E}}_{\ff_{0}}-n-1,\infty,\infty,\infty\right)}\left(M_{0b}^{2}\right)
	\end{align*}
	which implies
	\begin{align*}
		Q_{\lambda}P & \in\mathcal{A}^{\left(\lambda,b,c,\infty,\infty,\infty\right)}\left(M_{0}^{2}\right)\\
		\tilde{Q}_{\lambda}P & \in\mathcal{A}^{\left(a,\lambda,c,\infty,\infty,\infty\right)}\left(M_{0}^{2}\right).
	\end{align*}
	Finally, let $V$ be a $b$-vector field on $M_{0}^{2}$, transverse
	to $\ff_{0}$, inward-pointing, tangent to the other faces, and satisfying
	$Vr_{\ff_{0}}=1$ near $\ff_{0}$. Let $\mathcal{F}=\tilde{\mathcal{E}}_{\ff_{0}}-n-1$,
	and define
	\[
	R_{\lambda}=\prod_{\begin{smallmatrix}\left(\alpha,l\right)\in\mathcal{F}\\
			\Re\left(\alpha\right)\leq\lambda
	\end{smallmatrix}}\left(r_{\ff_{0}}V-\alpha\right).
	\]
	Then we have
	\[
	R_{\lambda}P\in\mathcal{A}_{\phg}^{\left(\mathcal{E}_{\lf},\mathcal{E}_{\rf},\mathcal{E}_{\lf}+\mathcal{E}_{\rf},\mathcal{F}\left(\lambda\right),\infty,\infty,\infty\right)}\left(M_{0b}^{2}\right),
	\]
	which implies
	\[
	R_{\lambda}P\in\mathcal{A}^{\left(a,b,\lambda,\infty,\infty,\infty\right)}\left(M_{0}^{2}\right).
	\]
	This concludes the proof.
\end{proof}
As a consequence, we obtain the following
\begin{cor}
	\label{cor:local-0-interior-is-in-0-calculus}Let $\mathcal{E}=\left(\mathcal{E}_{\lf},\mathcal{E}_{\rf},\mathcal{E}_{\ff_{0}}\right)$
	be a triple of index sets. Then:
	\begin{enumerate}
		\item the space $\hat{\Psi}_{0,\mathcal{S}}^{-\infty,\left(\mathcal{E}_{\lf},\mathcal{E}_{\rf},\infty\right)}\left(\mathbb{R}_{1}^{n+1}\right)$
		coincides with $\Psi_{\mathcal{S}}^{-\infty,\left(\mathcal{E}_{\lf},\mathcal{E}_{\rf}\right)}\left(\mathbb{R}_{1}^{n+1}\right)$;
		\item the space $\hat{\Psi}_{0,\mathcal{S}}^{-\infty,\mathcal{E}}\left(\mathbb{R}_{1}^{n+1}\right)$
		is included in $\Psi_{0,\mathcal{S}}^{-\infty,\left(\mathcal{E}_{\lf},\mathcal{E}_{\rf},\tilde{\mathcal{E}}_{\ff_{0}}\right)}\left(\mathbb{R}_{1}^{n+1}\right)$,
		where \[\tilde{\mathcal{E}}_{\ff_{0}}=\mathcal{E}_{\ff_{0}}\eunion\left(\mathcal{E}_{\lf}+\mathcal{E}_{\rf}+n+1\right).\]
	\end{enumerate}
\end{cor}

\begin{proof}
	The first point is a consequence of the fact that the symbol space
	\[S_{0,\mathcal{S}}^{-\infty,\left(\mathcal{E}_{\lf},\mathcal{E}_{\rf},\infty\right)}\left(\mathbb{R}^{k};\mathbb{R}_{2}^{n+2}\right)\]
	coincides with \[S_{\mathcal{S}}^{-\infty,\left(\mathcal{E}_{\lf},\mathcal{E}_{\rf}\right)}\left(\mathbb{R}^{k};\mathbb{R}_{2}^{n+2}\right),\]
	and the fact that the Fourier transform in the $Y$ variables induces
	an isomorphism \[\mathcal{A}_{\phg}^{\left(\mathcal{E}_{\lf},\mathcal{E}_{\rf},\infty,\infty,\infty\right)}\left(M^{2}\right)\to\mathcal{A}_{\phg}^{\left(\mathcal{E}_{\lf},\mathcal{E}_{\rf},\infty,\infty,\infty\right)}\left(\hat{M}^{2}\right).\]
	The second point follows immediately from Proposition \ref{prop:improved-hintz-lemma-for-0-interior}. 
\end{proof}
We are now ready to define the classes $\hat{\Psi}_{0}^{-\infty,\bullet}\left(X\right)$
and $\hat{\Psi}_{0b}^{-\infty,\bullet}\left(X\right)$.
\begin{defn}
	(Symbolic $0$-interior operators) Let $\mathcal{E}=\left(\mathcal{E}_{\lf},\mathcal{E}_{\rf},\mathcal{E}_{\ff_{0}}\right)$
	be a triple of index sets. We denote by $\hat{\Psi}_{0}^{-\infty,\mathcal{E}}\left(X\right)$
	the class of operators $P:\dot{C}^{\infty}\left(X\right)\to C^{-\infty}\left(X\right)$
	with the following properties:
	\begin{enumerate}
		\item in the complement of any neighborhood of $\partial\Delta$, $K_{P}$
		coincides with a very residual operator in $\Psi^{-\infty,\left(\mathcal{E}_{\lf},\mathcal{E}_{\rf}\right)}\left(X\right)=\mathcal{A}_{\phg}^{\left(\mathcal{E}_{\lf},\mathcal{E}_{\rf}\right)}\left(X^{2};\pi_{R}^{*}\mathcal{D}_{X}^{1}\right)$;
		\item for every $q\in\partial X$ and coordinates $\left(x,y\right)$ for
		$X$ centered at $q$ and compatible with $V$ (i.e.\ $Vx\equiv1$
		near $q$), $P$ coincides near $\left(q,q\right)$ with an element
		of $\hat{\Psi}_{0,\mathcal{S}}^{-\infty,\mathcal{E}}\left(\mathbb{R}_{1}^{n+1}\right)$;
		more precisely, there exists a $0$-interior symbol $p\left(y;x,\tilde{x},\eta\right)$
		in $S_{0,\mathcal{S}}^{-\infty,\mathcal{E}}\left(\mathbb{R}^{n};\mathbb{R}_{2}^{n+2}\right)$
		such that, in a neighborhood of the origin, we have
		\[
		K_{P}\equiv K\left(y;x,\tilde{x},y-\tilde{y}\right)d\tilde{x}d\tilde{y}=\frac{1}{\left(2\pi\right)^{n}}\int e^{i\left(y-\tilde{y}\right)\eta}p\left(y;x,\tilde{x},\eta\right)d\eta d\tilde{x}d\tilde{y}.
		\]
	\end{enumerate}
\end{defn}

\begin{defn}
	(Symbolic $0b$-interior operators) Let $\mathcal{E}=\left(\mathcal{E}_{\lf},\mathcal{E}_{\rf},\mathcal{E}_{\ff_{b}},\mathcal{E}_{\ff_{0}}\right)$
	be a quadruple of index sets. We denote by $\hat{\Psi}_{0b}^{-\infty,\mathcal{E}}\left(X\right)$
	the class of operators $P:\dot{C}^{\infty}\left(X\right)\to C^{-\infty}\left(X\right)$
	with the following properties:
	\begin{enumerate}
		\item in the complement of any neighborhood of $\partial\Delta$, $K_{P}$
		coincides with an element of the residual large $b$-calculus $\Psi_{b}^{-\infty,\left(\mathcal{E}_{\lf},\mathcal{E}_{\rf},\mathcal{E}_{\ff_{b}}\right)}\left(X\right)=\mathcal{A}_{\phg}^{\left(\mathcal{E}_{\lf},\mathcal{E}_{\rf},\mathcal{E}_{\ff_{b}}\right)}\left(X_{b}^{2};r_{\ff_{b}}^{-1}\beta_{b,R}^{*}\mathcal{D}_{X}^{1}\right)$;
		\item for every $q\in\partial X$ and coordinates $\left(x,y\right)$ for
		$X$ centered at $q$ and compatible with $V$ (i.e.\ $Vx\equiv1$
		near $q$), $P$ coincides near $\left(q,q\right)$ with an element
		of $\hat{\Psi}_{0b,\mathcal{S}}^{-\infty,\mathcal{E}}\left(\mathbb{R}_{1}^{n+1}\right)$;
		more precisely, there exists a $0b$-interior symbol $p\left(y;x,\tilde{x},\eta\right)$
		in $S_{0b,\mathcal{S}}^{-\infty,\mathcal{E}}\left(\mathbb{R}^{n};\mathbb{R}_{2}^{n+2}\right)$
		such that, in a neighborhood of the origin, we have
		\[
		K_{P}\equiv K\left(y;x,\tilde{x},y-\tilde{y}\right)d\tilde{x}d\tilde{y}=\frac{1}{\left(2\pi\right)^{n}}\int e^{i\left(y-\tilde{y}\right)\eta}p\left(y;x,\tilde{x},\eta\right)d\eta d\tilde{x}d\tilde{y}.
		\]
	\end{enumerate}
\end{defn}

Summarizing, the discussions above and the local characterizations
of the operator spaces $\Psi_{0}^{-\infty,\bullet}\left(X\right)$, $\Psi_{0b}^{-\infty,\bullet}\left(X\right)$,
$\Psi^{-\infty,\bullet}\left(X\right)$, $\Psi_{b}^{-\infty,\bullet}\left(X\right)$
given in \S\ref{subsubsec:local-physical-0-calculus-extended-0-calculus},
imply the following
\begin{cor}
	\label{cor:global-relation-symbolic0b-physical0b}Let $\mathcal{E}_{\lf},\mathcal{E}_{\rf},\mathcal{E}_{\ff_{b}},\mathcal{E}_{\ff_{0}}$
	be index sets. Then:
	\begin{enumerate}
		\item the space $\hat{\Psi}_{0}^{-\infty,\left(\mathcal{E}_{\lf},\mathcal{E}_{\rf},\infty\right)}\left(X\right)$
		coincides with $\Psi^{-\infty,\left(\mathcal{E}_{\lf},\mathcal{E}_{\rf}\right)}\left(X\right)$;
		\item the space $\hat{\Psi}_{0b}^{-\infty,\left(\mathcal{E}_{\lf},\mathcal{E}_{\rf},\mathcal{E}_{\ff_{b}},\infty\right)}\left(X\right)$
		coincides with the instance of the large residual $b$-calculus $\Psi_{b}^{-\infty,\left(\mathcal{E}_{\lf},\mathcal{E}_{\rf},\mathcal{E}_{\ff_{b}}\right)}\left(X\right)$;
		\item the space $\Psi_{0b}^{-\infty,\left(\mathcal{E}_{\lf},\mathcal{E}_{\rf},\mathcal{E}_{\ff_{b}},\mathcal{E}_{\ff_{0}}\right)}\left(X\right)$
		is included in $\hat{\Psi}_{0b}^{-\infty,\left(\mathcal{E}_{\lf},\mathcal{E}_{\rf},\tilde{\mathcal{E}}_{\ff_{b}},\mathcal{E}_{\ff_{0}}\right)}\left(X\right)$,
		where \[\tilde{\mathcal{E}}_{\ff_{b}}=\mathcal{E}_{\ff_{b}}\eunion\mathcal{E}_{\ff_{0}};\]
		\item the space $\hat{\Psi}_{0b}^{-\infty,\left(\mathcal{E}_{\lf},\mathcal{E}_{\rf},\mathcal{E}_{\ff_{b}},\mathcal{E}_{\ff_{0}}\right)}\left(X\right)$
		is included in $\Psi_{0b}^{-\infty,\left(\mathcal{E}_{\lf},\mathcal{E}_{\rf},\mathcal{E}_{\ff_{b}},\tilde{\mathcal{E}}_{\ff_{0}}\right)}\left(X\right)$,
		where \[\tilde{\mathcal{E}}_{\ff_{0}}=\mathcal{E}_{\ff_{0}}\eunion\left(\mathcal{E}_{\ff_{b}}+n\right);\]
		\item the space $\hat{\Psi}_{0}^{-\infty,\left(\mathcal{E}_{\lf},\mathcal{E}_{\rf},\mathcal{E}_{\ff_{0}}\right)}\left(X\right)$
		is included in $\Psi_{0}^{-\infty,\left(\mathcal{E}_{\lf},\mathcal{E}_{\rf},\tilde{\mathcal{E}}_{\ff_{0}}\right)}\left(X\right)$,
		where \[\tilde{\mathcal{E}}_{\ff_{0}}=\mathcal{E}_{\ff_{0}}\eunion\left(\mathcal{E}_{\lf}+\mathcal{E}_{\rf}+n+1\right).\]
	\end{enumerate}
\end{cor}

\begin{rem}
	\label{rem:0-calculus-not-in-0-interior}We remark that, despite the
	notation, an element of the large residual $0$-calculus $\Psi_{0}^{-\infty,\bullet}\left(X\right)$
	is in general \emph{not} a $0$-interior operator. More precisely,
	we have two strict inclusions
	\begin{align*}
		\Psi^{-\infty,\left(\mathcal{E}_{\lf},\mathcal{E}_{\rf}\right)}\left(X\right) & \subset\hat{\Psi}_{0}^{-\infty,\left(\mathcal{E}_{\lf},\mathcal{E}_{\rf},\mathcal{E}_{\ff_{0}}\right)}\left(X\right)\subset\Psi_{0}^{-\infty,\left(\mathcal{E}_{\lf},\mathcal{E}_{\rf},\tilde{\mathcal{E}}_{\ff_{0}}\right)}\left(X\right)
	\end{align*}
	where $\tilde{\mathcal{E}}_{\ff_{0}}=\mathcal{E}_{\ff_{0}}\eunion\left(\mathcal{E}_{\lf}+\mathcal{E}_{\rf}+n+1\right)$.
	The fact that the second inclusion is strict follows from the analysis
	of the Bessel family, carried out in the next sub-section (cf.\ Remark
	\ref{rem:why-0-calculus-not-contained-in-symbolic-0-operators} for
	concrete examples). Moreover, operators in $\hat{\Psi}_{0}^{-\infty,\left(\mathcal{E}_{\lf},\mathcal{E}_{\rf},\mathcal{E}_{\ff_{0}}\right)}\left(X\right)$
	have better mapping properties on polyhomogeneous functions than (generic)
	elements of $\Psi_{0}^{-\infty,\left(\mathcal{E}_{\lf},\mathcal{E}_{\rf},\tilde{\mathcal{E}}_{\ff_{0}}\right)}\left(X\right)$
	(cf.\ Remark \ref{rem:0-interior-better-mapping-than-0-calculus}).
	The intermediate class $\hat{\Psi}_{0}^{-\infty,\bullet}\left(X\right)$
	will turn out to be very useful in our theory of boundary value problems.
\end{rem}

\subsection{\label{subsec:Bessel-families}Bessel families}

In \S\ref{subsubsec:The-normal-family} and \S\ref{subsubsec:The-Bessel-family-physical},
we discussed the normal and Bessel families for operators in the classes
$\Psi_{0\tr}^{-\infty,\bullet}\left(X,\partial X\right)$, $\Psi_{0\po}^{-\infty,\bullet}\left(X,\partial X\right)$,
$\Psi_{0}^{-\infty,\bullet}\left(X\right)$, $\Psi_{0b}^{-\infty,\bullet}\left(X\right)$.
For example, if $Q\in\Psi_{0b}^{-\infty,\mathcal{E}}\left(X\right)$
and $\left[\mathcal{E}_{\ff_{0}}\right]=0$, then the normal family
$N\left(Q\right)$ encodes invariantly the leading order term in the
expansion of $\kappa_{Q}=\beta_{0b}^{*}K_{Q}$ at the $0$-front face
of $X_{0b}^{2}$. We explained how $N\left(Q\right)$ can be interpreted
as a smooth family $p\mapsto N_{p}\left(Q\right)$, parametrized by
$p\in\partial X$, of dilation and translation invariant operators
acting on functions over the model space $X_{p}=T_{p}^{+}X$. We then
explained how, using an auxiliary vector field $V$ on $X$ transverse
to $\partial X$ and inward-pointing, and the invariant Fourier transform
on the fibers of $T\partial X$, we can \emph{transform} $N\left(Q\right)$
into an equivalent object, a smooth family $\eta\mapsto\hat{N}_{\eta}\left(Q\right)$
of operators on functions over the fibers of $N^{+}\partial X$. We
argued similarly for operators in $\Psi_{0\tr}^{-\infty,\bullet}\left(X,\partial X\right)$,
$\Psi_{0\po}^{-\infty,\bullet}\left(X,\partial X\right)$, $\Psi_{0}^{-\infty,\bullet}\left(X\right)$.

It is in general quite difficult to characterize the range of the
Bessel family map on the classes $\Psi_{0\tr}^{-\infty,\bullet}\left(X,\partial X\right)$,
$\Psi_{0\po}^{-\infty,\bullet}\left(X,\partial X\right)$, $\Psi_{0}^{-\infty,\bullet}\left(X\right)$,
$\Psi_{0b}^{-\infty,\bullet}\left(X\right)$ (\cite{Hintz0calculus}
discusses this issue for the classes $\Psi_{0}^{-\infty,\bullet}\left(X\right)$,
$\Psi_{0b}^{-\infty,\bullet}\left(X\right)$). In this sub-section
we will define the Bessel family map for the classes $\hat{\Psi}_{0\tr}^{-\infty,\bullet}\left(X,\partial X\right)$,
$\hat{\Psi}_{0\po}^{-\infty,\bullet}\left(X,\partial X\right)$, $\hat{\Psi}_{0}^{-\infty,\bullet}\left(X\right)$,
$\hat{\Psi}_{0b}^{-\infty,\bullet}\left(X\right)$, and we will show
how this map is in a very concrete sense a ``principal symbol map''.
As such, it fits into a short exact sequence formally similar to the
usual principal symbol short exact sequence; in particular, the Bessel
family map is \emph{by design} surjective onto an appropriate space
of \emph{fibrewise homogeneous sections} of an appropriate Fréchet
bundle over $T^{*}\partial X\backslash0$.

\subsubsection{\label{subsubsec:Principal-symbols-of-psidos}Principal symbols of
	polyhomogeneous pseudodifferential operators}

We start with a review of the definition of the principal symbol map
for classical pseudodifferential operators on the closed manifold
$\partial X$. In accordance with the approach adopted in this paper,
we can define the usual class of pseudodifferential operators $\Psi^{\bullet}\left(\partial X\right)$
as follows. Given $m\in\mathbb{R}$, an operator $Q:C^{\infty}\left(\partial X\right)\to C^{-\infty}\left(\partial X\right)$
is in $\Psi^{m}\left(\partial X\right)$ if and only if:
\begin{enumerate}
	\item the Schwartz kernel $K_{Q}$ coincides, in the complement of any neighborhood
	of the diagonal $\Delta$, with an element of $\Psi^{-\infty}\left(\partial X\right)=C^{\infty}\left(\left(\partial X\right)^{2};\pi_{R}^{*}\mathcal{D}_{\partial X}^{1}\right)$;
	\item for every $p\in\partial X$ and coordinates $y$ centered at $p$,
	there exists a symbol $q\left(y;\eta\right)\in\mathcal{S}\left(\mathbb{R}^{n}\right)\hat{\otimes}\mathcal{A}^{-m}\left(\overline{\mathbb{R}^{n}}\right)=:S_{\mathcal{S}}^{m}\left(\mathbb{R}^{n};\mathbb{R}^{n}\right)$
	such that, in the induced coordinates $\left(y,\tilde{y}\right)$
	for $Y^{2}$ centered at $\left(p,p\right)$, we have
	\[
	K_{Y}\equiv K\left(y;y-\tilde{y}\right)d\tilde{y}=\frac{1}{\left(2\pi\right)^{n}}\int e^{i\left(y-\tilde{y}\right)\eta}q\left(y;\eta\right)d\eta d\tilde{y}
	\]
	near the origin.
\end{enumerate}
The class $\Psi_{\phg}^{m}\left(\partial X\right)$ (where we allow
$m\in\mathbb{C}$) of \emph{classical }pseudodifferential operators
of order $m$ is obtained as above, except that we replace $\mathcal{A}^{-m}\left(\overline{\mathbb{R}^{n}}\right)$
with the space \[\left\langle \eta\right\rangle ^{m}C^{\infty}\left(\overline{\mathbb{R}^{n}}\right)=\mathcal{A}_{\phg}^{-m+\mathbb{N}}\left(\overline{\mathbb{R}^{n}}\right),\]
where $\left\langle \eta\right\rangle =\sqrt{1+\left|\eta\right|^{2}}$
(note that $\left\langle \eta\right\rangle ^{-1}$ is a boundary defining
function for $\overline{\mathbb{R}^{n}}$). More generally, given
an index set $\mathcal{E}$, one can define $\Psi_{\phg}^{-\mathcal{E}}\left(\partial X\right)$
by replacing $\mathcal{A}^{-m}\left(\overline{\mathbb{R}^{n}}\right)$
with $\mathcal{A}_{\phg}^{\mathcal{E}}\left(\overline{\mathbb{R}^{n}}\right)$.
We denote by
\[
S_{\phg,\mathcal{S}}^{-\mathcal{E}}\left(\mathbb{R}^{n};\mathbb{R}^{n}\right)=\mathcal{S}\left(\mathbb{R}^{n}\right)\hat{\otimes}\mathcal{A}_{\phg}^{\mathcal{E}}\left(\overline{\mathbb{R}^{n}}\right)
\]
the corresponding class of symbols.

The \emph{principal symbol }$\sigma\left(Q\right)$ of an operator
$Q\in\Psi_{\phg}^{m}\left(\partial X\right)$ (or, more generally, an operator
$Q\in\Psi_{\phg}^{-\mathcal{E}}\left(\partial X\right)$ with $\left[\mathcal{E}\right]=-m$)
is a smooth, fibrewise homogeneous function $\eta\mapsto\sigma\left(Q\right)$
defined on $T^{*}\partial X\backslash0$, constructed as follows.
Given $p\in\partial X$, choose coordinates $y$ for $\partial X$
centered at $p$, and let $q\left(y;\eta\right)\in S_{\phg,\mathcal{S}}^{-\mathcal{E}}\left(\mathbb{R}^{n};\mathbb{R}^{n}\right)$
be a classical symbol such that $Q\equiv\Op\left(q\right)$ near the
origin in the given coordinates. Then, in these coordinates, for every
$\eta\in T_{p}^{*}\partial X\backslash0\equiv\mathbb{R}^{n}\backslash0$
we have
\begin{align*}
	\sigma_{\eta}\left(Q\right) & :=\lim_{t\to0^{+}}t^{m}q\left(0;t^{-1}\eta\right).
\end{align*}
This limit picks up the leading order term in the expansion of $q\left(0;\eta\right)$
as $\left|\eta\right|\to+\infty$. More precisely, since $q\left(y;\eta\right)\in S_{\phg,\mathcal{S}}^{-\mathcal{E}}\left(\mathbb{R}^{n};\mathbb{R}^{n}\right)$,
by Taylor's Theorem we can write, away from the locus $\left|\eta\right|=0$,
\[
q\left(0;\eta\right)=q_{0}\left(\frac{\eta}{\left|\eta\right|},\left|\eta\right|^{-1}\right)\left|\eta\right|^{m}+r\left(\frac{\eta}{\left|\eta\right|},\left|\eta\right|^{-1}\right)\left|\eta\right|^{m-1}
\]
where $q_{0}\left(\hat{\eta},\rho\right),r\left(\hat{\eta},\rho\right)\in C^{\infty}\left(S^{n-1}\times\mathbb{R}_{1}^{1}\right)$.
We then get
\begin{align*}
	\lim_{t\to0^{+}}t^{m}q\left(0;t^{-1}\eta\right) & =\lim_{t\to0^{+}}t^{m}\left\langle t^{-1}\eta\right\rangle ^{m}q_{0}'\left(\frac{\eta}{\left|\eta\right|},t\left|\eta\right|^{-1}\right)\\
	& +t^{m}\left\langle t^{-1}\eta\right\rangle ^{m-1}r\left(\hat{\eta},t\left|\eta\right|^{-1}\right)\\
	& =\left|\eta\right|^{m}q_{0}'\left(\frac{\eta}{\left|\eta\right|},0\right)\\
	& =\left|\eta\right|^{m}\tilde{q}_{0}\left(\frac{\eta}{\left|\eta\right|}\right),
\end{align*}
where $\tilde{q}_{0}\left(\hat{\eta}\right)=q_{0}'\left(\hat{\eta},0\right)\in C^{\infty}\left(S^{n-1}\right)$.
The function $\tilde{q}_{0}\left(\eta/\left|\eta\right|\right)\left|\eta\right|^{m}=\sigma_{\eta}\left(Q\right)$
does not depend on the particular symbol but only on the operator
$Q$; it is a well-defined homogeneous of degree $m$ function on
$T_{p}^{*}\partial X\backslash0$. The principal symbol map fits into
a short exact sequence
\[
\xymatrix{0\ar[r] & \Psi_{\phg}^{-\left(\mathcal{E}\backslash\left\{ m\right\} \right)}\left(\partial X\right)\ar[r] & \Psi_{\phg}^{-\mathcal{E}}\left(\partial X\right)\ar[r]^{\sigma} & S^{\left[m\right]}\left(T^{*}\partial X\backslash0\right)\ar[r] & 0}
\]
where $S^{\left[m\right]}\left(T^{*}\partial X\backslash0\right)$
denotes the space of smooth, fibrewise homogeneous of degree $m$,
functions on $T^{*}Y\backslash0$. Moreover, the principal symbol
map $\sigma$ is compatible with compositions and formal adjoints.

\subsubsection{\label{subsubsec:The-Bessel-family-of-symbolic-0-poisson}The Bessel
	family of a symbolic $0$-Poisson operator}

The Bessel family map for the class $\hat{\Psi}_{0\po}^{-\infty,\bullet}\left(\partial X,X\right)$
is conceptually very similar to the principal symbol map described
above. Fix a pair of index sets $\mathcal{E}=\left(\mathcal{E}_{\of},\mathcal{E}_{\ff}\right)$
with $\left[\mathcal{E}_{\ff}\right]=m\in\mathbb{C}$. First of all,
we recall from \S\ref{subsec:Solving-the-model} the definition of
the Fréchet bundle over $T^{*}\partial X\backslash0$ where the Bessel
family map takes values. We denote by $\Psi_{\po}^{-\infty,\mathcal{E}_{\of}}\left(\mathbb{R}_{1}^{1}\right)$
the space of linear maps $\mathbb{C}\to\mathcal{A}_{\phg}^{\left(\mathcal{E}_{\of},\infty\right)}\left(\overline{\mathbb{R}_{1}^{1}}\right)$
(each such map can be identified with the element of $\mathcal{A}_{\phg}^{\left(\mathcal{E}_{\of},\infty\right)}\left(\overline{\mathbb{R}_{1}^{1}}\right)$
obtained by evaluating the map at $1$). The space is acted upon by
$\mathbb{R}^{+}$ via the representation
\begin{align*}
	\mathbb{R}^{+} & \to\GL\left(\Psi_{\po}^{-\infty,\mathcal{E}_{\of}}\left(\mathbb{R}_{1}^{1}\right)\right)\\
	t & \mapsto\left(B\mapsto\lambda_{t}^{*}\circ B\right),
\end{align*}
and therefore we can define a smooth Fréchet bundle with typical fiber
$\Psi_{\po}^{-\infty,\mathcal{E}_{\of}}\left(\mathbb{R}_{1}^{1}\right)$
as the bundle associated to the principal $\mathbb{R}^{+}$ bundle
$N^{+}\partial X\backslash0$ via the representation above. We denote
this bundle by $\Psi_{\po}^{-\infty,\mathcal{E}_{\of}}\left(N^{+}\partial X\right)\to\partial X$.
Its fiber at a point $p\in\partial X$ is the space of operators $\mathbb{C}\to\mathcal{A}_{\phg}^{\left(\mathcal{E}_{\of},\infty\right)}\left(\overline{N_{p}^{+}\partial X}\right)$.
The choice of a vector field $V$ on $X$ transverse to $\partial X$
and inward pointing, determines a global trivialization of $N^{+}\partial X$
and hence also a global trivialization of $\Psi_{\po}^{-\infty,\mathcal{E}_{\of}}\left(N^{+}\partial X\right)$.

Now, let $B\in\hat{\Psi}_{0\po}^{-\infty,\mathcal{E}}\left(\partial X,X\right)$.
The Bessel family of $B$ is a \emph{section} $\hat{N}\left(B\right)$
of the pull-back bundle $\pi^{*}\Psi_{\po}^{-\infty,\mathcal{E}_{\of}}\left(N^{+}\partial X\right)\to T^{*}\partial X\backslash0$,
with the following homogeneity property: for every $\eta\in T^{*}\partial X\backslash0$
and $t\in\mathbb{R}^{+}$, we have
\[
\hat{N}_{t\eta}\left(B\right)=\lambda_{t}^{*}\circ\hat{N}_{\eta}\left(B\right)\circ t^{-m}.
\]
Just as the principal symbol of an operator $Q\in\Psi_{\phg}^{m}\left(Y\right)$
encodes the leading order term in the expansion at infinity of a symbol
representing $Q$, the Bessel family of $B$ encodes the leading order
term in the expansion \emph{at the front face} of a $0$-Poisson symbol
representing $B$.

To be more precise, choose coordinates $\left(x,y\right)$ for $X$
centered at $p$ and compatible with $V$, in the sense that $Vx\equiv1$
near $p$. By definition of $\hat{\Psi}_{0\po}^{-\infty,\mathcal{E}}\left(\partial X,X\right)$,
there is a $0$-Poisson symbol $b\left(y;x,\eta\right)$ in $S_{0\po,\mathcal{S}}^{-\infty,\mathcal{E}}\left(\mathbb{R}^{n};\mathbb{R}_{1}^{n+1}\right)$
such that, in the induced coordinates $\left(x,y,\tilde{y}\right)$
on $X\times\partial X$, we have near the origin
\[
K_{B}\equiv\left[\frac{1}{\left(2\pi\right)^{n}}\int e^{i\left(y-\tilde{y}\right)\eta}b\left(y;x,\eta\right)d\eta\right]d\tilde{y}.
\]
Given $x\in\mathbb{R}^{+}$ and $\eta\in\mathbb{R}^{n}\backslash0$,
we define
\[
\hat{N}_{\eta}\left(B\right):=\lim_{t\to0^{+}}b\left(0;tx,t^{-1}\eta\right)t^{-m}.
\]
We claim that this limit is well-defined and has the form
\[
\hat{N}_{\eta}\left(B\right)=\tilde{b}_{0}\left(x\left|\eta\right|,\frac{\eta}{\left|\eta\right|}\right)\left|\eta\right|^{-m},
\]
where $\tilde{b}_{0}\left(\tau,\hat{\eta}\right)\in\mathcal{A}^{\left(\mathcal{E}_{\of},\infty\right)}\left(\overline{\mathbb{R}_{1}^{1}}\times S^{n-1}\right)$
does not depend on the particular choice of the $0$-Poisson symbol
$b$, but only on the operator $B$ and on the particular choice of
the coordinates.

To see this, observe that since \[b\left(0;x,\eta\right)\in\mathcal{A}_{\phg}^{\left(\mathcal{E}_{\of},\mathcal{E}_{\ff},\infty,\infty\right)}\left(\hat{P}_{0}^{2}\right),\]
we can write away from the locus $\left|\eta\right|=0$
\[
b\left(0;x,\eta\right)=b_{0}\left(x\left|\eta\right|,\hat{\eta},\left|\eta\right|^{-1}\right)\left|\eta\right|^{-m}+r\left(x\left|\eta\right|,\hat{\eta},\left|\eta\right|^{-1}\right)\left|\eta\right|^{-m-1},
\]
where $b_{0}\left(\tau,\hat{\eta},\rho\right)$ and $r\left(\tau,\hat{\eta},\rho\right)$
are elements of \[\mathcal{A}_{\phg}^{\left(\mathcal{E}_{\of},\infty\right)}\left(\overline{\mathbb{R}_{1}^{1}}\times S^{n-1}\right)\hat{\otimes}\mathcal{A}_{\phg}^{\mathcal{E}_{\ff}-m}\left(\mathbb{R}_{1}^{1}\right).\]
Since $\left[\mathcal{E}_{\ff}\right]=m$, the evaluations $\tilde{b}_{0}\left(\tau,\hat{\eta}\right)=b_{0}\left(\tau,\hat{\eta},0\right)$
and $r\left(\tau,\hat{\eta},0\right)$ are elements of \[\mathcal{A}_{\phg}^{\left(\mathcal{E}_{\of},\infty\right)}\left(\overline{\mathbb{R}_{1}^{1}}\times S^{n-1}\right),\]
and we have
\begin{align*}
	\lim_{t\to0^{+}}b\left(0;tx,t^{-1}\eta\right)t^{-m} & =\tilde{b}_{0}\left(x\left|\eta\right|,\frac{\eta}{\left|\eta\right|}\right)\left|\eta\right|^{-m}.
\end{align*}
as claimed. Observe that this limit does not depend on the particular
choice of the symbol $b$, but only on the operator $B$ and the choice
of coordinates: indeed, if $b_{1}$ and $b_{2}$ are two $0$-Poisson
symbols such that $B\equiv\Op_{L}^{\po}\left(b_{1}\right)\equiv\Op_{L}^{\po}\left(b_{2}\right)$
near the origin in the given coordinates, then this means that the
inverse Fourier transform $B_{1}\left(0;x,Y\right)-B_{2}\left(0;x,Y\right)$
must \emph{vanish} in a neighborhood of $x=0,\left|Y\right|=0$. Therefore,
$B_{1}\left(0;x,Y\right)-B_{2}\left(0;x,Y\right)$ must be in $\mathcal{A}_{\phg}^{\left(\mathcal{E}_{\of},\infty,\infty\right)}\left(P^{2}\right)$,
which implies that $b_{1}\left(0;x,\eta\right)-b_{2}\left(0;x,\eta\right)$
must be in $\mathcal{A}_{\phg}^{\left(\mathcal{E}_{\of},\infty,\infty\right)}\left(\hat{P}_{0}^{2}\right)$.
Thus, the lifts of $b_{1}\left(0;x,\eta\right)$ and $b_{2}\left(0;x,\eta\right)$
to $\hat{P}_{0}^{2}$ must agree to infinite order at $\ff$.

We omit the annoying (but elementary) check
that the correspondence $\hat{N}\left(B\right):\eta\mapsto\tilde{b}_{0}\left(x\left|\eta\right|,\frac{\eta}{\left|\eta\right|}\right)\left|\eta\right|^{-m}$
defined above does indeed define invariantly a smooth section of the
pull-back Fréchet bundle $\pi^{*}\Psi_{\po}^{-\infty,\mathcal{E}_{\of}}\left(N^{+}\partial X\right)\to T^{*}\partial X\backslash0$.
Let us just check that $\hat{N}\left(B\right)$ is fibrewise homogeneous
of degree $-m$. Using the coordinates above, given $\eta\in T_{p}^{*}\partial X\backslash0\equiv\mathbb{R}^{n}\backslash0$
and $t\in\mathbb{R}^{+}$, we have
\begin{align*}
	\hat{N}_{t\eta}\left(B\right) & =\tilde{b}_{0}\left(tx\left|\eta\right|,\frac{\eta}{\left|\eta\right|}\right)\left|\eta\right|^{-m}t^{-m}\\
	& =\left(\lambda_{t}^{*}\tilde{b}_{0}\right)\left(x\left|\eta\right|,\frac{\eta}{\left|\eta\right|}\right)\left|\eta\right|^{-m}t^{-m}\\
	& =\lambda_{t}^{*}\circ\hat{N}_{\eta}\left(B\right)t^{-m}.
\end{align*}
By construction, it is clear that the Bessel family map $\hat{N}$
maps $\hat{\Psi}_{0\po}^{-\infty,\mathcal{E}}\left(\partial X,X\right)$
surjectively onto the space of smooth sections of $\pi^{*}\Psi_{\po}^{-\infty,\mathcal{E}_{\of}}\left(N^{+}\partial X\right)$
over $T^{*}\partial X\backslash0$ which are fibrewise homogeneous
of degree $-m$ in the sense above. Furthermore, the kernel of $\hat{N}$
consists of the subclass $\hat{\Psi}_{0\po}^{-\infty,\left(\mathcal{E}_{\of},\mathcal{E}_{\ff}\backslash\left\{ m\right\} \right)}\left(\partial X,X\right)$.

\subsubsection{\label{subsubsec:The-Bessel-family-of-symbolic-0-trace}The Bessel
	family of a symbolic $0$-trace operator}

The construction of the Bessel family map for symbolic $0$-trace
operators is carried out in a very similar way. As above, the Bessel
map takes values in a space of fibrewise homogeneous sections of a
natural Fréchet bundle over $T^{*}\partial X\backslash0$. Let's recall
its definition from \S\ref{subsec:analysis-of-The-Bessel-trace}.
We denote by $\Psi_{\tr}^{-\infty,\mathcal{E}_{\of}}\left(\mathbb{R}_{1}^{1}\right)$
the space of operators $\dot{C}^{\infty}\left(\overline{\mathbb{R}_{1}^{1}}\right)\to\mathbb{C}$
of the form $f\left(\tilde{x}\right)d\tilde{x}$ for some $f\left(\tilde{x}\right)\in\mathcal{A}_{\phg}^{\left(\mathcal{E}_{\of},\infty\right)}\left(\overline{\mathbb{R}_{1}^{1}}\right)$.
This space is acted upon by $\mathbb{R}^{+}$ via the representation
\begin{align*}
	\mathbb{R}^{+} & \to\Psi_{\tr}^{-\infty,\mathcal{E}_{\of}}\left(\mathbb{R}_{1}^{1}\right)\\
	t & \mapsto\left(A\mapsto A\circ\lambda_{t^{-1}}^{*}\right)
\end{align*}
and this allows us to define a bundle \[\Psi_{\tr}^{-\infty,\mathcal{E}_{\of}}\left(N^{+}\partial X\right)\to\partial X,\]
with typical fiber $\mathcal{A}_{\phg}^{\left(\mathcal{E}_{\of},\infty\right)}\left(\overline{\mathbb{R}_{1}^{1}}\right)$,
associated to $N^{+}\partial X\backslash0$ via the representation
above. Given an operator $A\in\hat{\Psi}_{0\tr}^{-\infty,\mathcal{E}}\left(X,\partial X\right)$
with $\mathcal{E}=\left(\mathcal{E}_{\of},\mathcal{E}_{\ff}\right)$
and $\left[\mathcal{E}_{\ff}\right]=m\in\mathbb{C}$, the \emph{Bessel
	family} of $A$ is a smooth, fibrewise homogeneous of degree $-m$,
section of the Fréchet bundle \[\pi^{*}\Psi_{\tr}^{-\infty,\mathcal{E}_{\of}}\left(N^{+}\partial X\right)\to T^{*}\partial X\backslash0,\]
and is defined as follows.

Choose $p\in\partial X$, and let $\left(x,y\right)$ be coordinates
for $X$ centered at $p$ and compatible with $V$. Then we can find
a $0$-trace symbol $a\left(y;\tilde{x},\eta\right)\in S_{0\tr,\mathcal{S}}^{-\infty,\mathcal{E}}\left(\mathbb{R}^{n};\mathbb{R}_{1}^{n+1}\right)$
such that in the induced coordinates $\left(y,\tilde{x},\tilde{y}\right)$
for $\partial X\times X$ centered at $\left(p,p\right)$ we have
\[
K_{A}\equiv\left[\frac{1}{\left(2\pi\right)^{n}}\int e^{i\left(y-\tilde{y}\right)\eta}a\left(y;\tilde{x},\eta\right)d\eta\right]d\tilde{x}d\tilde{y}.
\]
Given $\eta\in T_{p}^{*}\partial X\backslash0\equiv\mathbb{R}^{n}\backslash0$,
we define
\[
\hat{N}_{\eta}\left(A\right)=\left[\lim_{t\to0^{+}}t^{-m}a\left(0;t\tilde{x},t^{-1}\eta\right)t\right]d\tilde{x}.
\]
We claim that this limit determines a density of the form
\[
\hat{N}_{\eta}\left(A\right)=\left|\eta\right|^{-m}\tilde{a}_{0}\left(\tilde{x}\left|\eta\right|,\frac{\eta}{\left|\eta\right|}\right)\left|\eta\right|d\tilde{x},
\]
where $\tilde{a}_{0}\left(\tau,\hat{\eta}\right)\in\mathcal{A}_{\phg}^{\left(\mathcal{E}_{\of},\infty\right)}\left(\overline{\mathbb{R}_{1}^{1}}\times S^{n-1}\right)$.
To show this, we argue as in the $0$-Poisson case: since $a\left(0;\tilde{x},\eta\right)\in\mathcal{A}_{\phg}^{\left(\mathcal{E}_{\of},\mathcal{E}_{\ff}-1,\infty,\infty\right)}\left(\hat{T}_{0}^{2}\right)$
and $\left[\mathcal{E}_{\ff}\right]=m$, we can write $a\left(0;\tilde{x},\eta\right)$
away from $\left|\eta\right|=0$ as
\[
a\left(0;\tilde{x},\eta\right)=\left|\eta\right|^{-m+1}a_{0}\left(\tilde{x}\left|\eta\right|,\frac{\eta}{\left|\eta\right|},\left|\eta\right|^{-1}\right)+\left|\eta\right|^{-m}r\left(\tilde{x}\left|\eta\right|,\frac{\eta}{\left|\eta\right|},\left|\eta\right|^{-1}\right),
\]
where $a_{0}\left(\tau,\hat{\eta},\rho\right)$ and $r\left(\tau,\hat{\eta},\rho\right)$
are in $\mathcal{A}_{\phg}^{\left(\mathcal{E}_{\of},\infty\right)}\left(\overline{\mathbb{R}_{1}^{1}}\times S^{n-1}\right)\hat{\otimes}\mathcal{A}_{\phg}^{\mathcal{E}_{\ff}-m}\left(\mathbb{R}_{1}^{1}\right)$.
Thus, we have
\[
\lim_{t\to0^{+}}t^{-m}a\left(0;t\tilde{x},t^{-1}\eta\right)td\tilde{x}=\left|\eta\right|^{-m}\tilde{a}_{0}\left(\tilde{x}\left|\eta\right|,\frac{\eta}{\left|\eta\right|}\right)\left|\eta\right|d\tilde{x},
\]
where $\tilde{a}_{0}\left(\tau,\hat{\eta}\right)=a_{0}\left(\tau,\hat{\eta},0\right)\in\mathcal{A}_{\phg}^{\left(\mathcal{E}_{\of},\infty\right)}\left(\overline{\mathbb{R}_{1}^{1}}\times S^{n-1}\right)$.
One proves exactly as in the $0$-Poisson case that if two $0$-trace
symbols $a_{1},a_{2}$ are such that $A\equiv\Op_{L}^{\tr}\left(a_{1}\right)\equiv\Op_{L}^{\tr}\left(a_{2}\right)$
in the chosen coordinates near the origin, then $a_{1}$ and $a_{2}$
must agree to infinite order at $\ff$, so $\hat{N}_{\eta}\left(A\right)$
depends only on $A$ and the coordinate choices.

As mentioned above, the correspondence $\eta\mapsto\hat{N}_{\eta}\left(A\right)$
determines invariantly a smooth section of the Fréchet bundle $\pi^{*}\Psi_{\tr}^{-\infty,\mathcal{E}_{\of}}\left(N^{+}\partial X\right)\to T^{*}\partial X\backslash0$,
fibrewise homogeneous of degree $-m$. Let's check this at the fiber
$T_{p}^{*}\partial X\backslash0$ in the coordinates above. We have
\begin{align*}
	\hat{N}_{t\eta}\left(A\right) & =t^{-m+1}\left|\eta\right|^{-m}\tilde{a}_{0}\left(t\tilde{x}\left|\eta\right|,\frac{\eta}{\left|\eta\right|}\right)\left|\eta\right|d\tilde{x}.
\end{align*}
On the other hand, we have
\begin{align*}
	t^{-m}\hat{N}_{\eta}\left(A\right)\left(\lambda_{t^{-1}}^{*}u\right) & =t^{-m}\int\left|\eta\right|^{-m}\tilde{a}_{0}\left(\tilde{x}\left|\eta\right|,\frac{\eta}{\left|\eta\right|}\right)u\left(t^{-1}\tilde{x}\right)\left|\eta\right|d\tilde{x}\\
	& =t^{-m}\int\left|\eta\right|^{-m}\tilde{a}_{0}\left(t\tilde{X}\left|\eta\right|,\frac{\eta}{\left|\eta\right|}\right)u\left(\tilde{X}\right)\left|\eta\right|td\tilde{X}
\end{align*}
so indeed we get
\[
\hat{N}_{t\eta}\left(A\right)=t^{-m}\hat{N}_{\eta}\left(A\right)\circ\lambda_{t^{-1}}^{*}
\]
as claimed.

\subsubsection{\label{subsubsec:The-Bessel-family-of-interior}The Bessel family
	of a $0$-interior or $0b$-interior operator}

We finally discuss the Bessel family map for the classes of $0$-interior
and $0b$-interior operators. As above, we first define the Fréchet
bundles where the Bessel family lives. For the $0$-interior case,
we already encountered the relevant Fréchet bundle in \S\ref{subsec:analysis-of-The-Bessel-trace}
where we discussed the Bessel family $\hat{N}\left(P_{1}\right)$
of the orthogonal projector $P_{1}$ onto the kernel of a $0$-elliptic
operator $L$.

Fix index sets $\mathcal{E}_{\lf},\mathcal{E}_{\rf},\mathcal{E}_{\ff_{b}},\mathcal{E}_{\ff_{0}}$.
We define $\Psi^{-\infty,\left(\mathcal{E}_{\lf},\mathcal{E}_{\rf}\right)}\left(\mathbb{R}_{1}^{1}\right)$
as the space of operators \[\dot{C}^{\infty}\left(\overline{\mathbb{R}_{1}^{1}}\right)\to C^{-\infty}\left(\overline{\mathbb{R}_{1}^{1}}\right)\]
given by right densities $f\left(x,\tilde{x}\right)d\tilde{x}$, where
$f\in\mathcal{A}_{\phg}^{\left(\mathcal{E}_{\lf},\mathcal{E}_{\rf},\infty\right)}\left(\overline{\mathbb{R}_{2}^{2}}\right)$.
Again, this space is acted upon by $\mathbb{R}^{+}$ via the representation associating
$t\in\mathbb{R}^{+}$ to the conjugation $\lambda_{t}^{*}\circ\cdot\circ\lambda_{t^{-1}}^{*}$,
and this allows us to define the associated Fréchet bundle $\Psi^{-\infty,\left(\mathcal{E}_{\lf},\mathcal{E}_{\rf}\right)}\left(N^{+}\partial X\right)\to\partial X$.

Similarly, define $\Psi_{b}^{-\infty,\left(\mathcal{E}_{\lf},\mathcal{E}_{\rf},\mathcal{E}_{\ff_{b}}\right)}\left(\mathbb{R}_{1}^{1}\right)$
as the space of operators $\dot{C}^{\infty}\left(\overline{\mathbb{R}_{1}^{1}}\right)\to C^{-\infty}\left(\overline{\mathbb{R}_{1}^{1}}\right)$
given by \emph{distributional} right densities $f\left(x,\tilde{x}\right)d\tilde{x}$,
where $f$ \emph{lifts} from $\overline{\mathbb{R}_{2}^{2}}$ to an
element of \[\mathcal{A}_{\phg}^{\left(\mathcal{E}_{\lf},\mathcal{E}_{\rf},\mathcal{E}_{\ff_{b}}-1,\infty\right)}\left(\left[\overline{\mathbb{R}_{2}^{2}};0\right]\right).\]
Here the index set $\mathcal{E}_{\ff_{b}}-1$ refers to the front
face of the blow-up $\left[\overline{\mathbb{R}_{2}^{2}};0\right]$.
Again, the conjugation action $t\mapsto\lambda_{t}^{*}\circ\cdot\circ\lambda_{t^{-1}}^{*}$
determines a representation of $\mathbb{R}^{+}$ in $\Psi_{b}^{-\infty,\left(\mathcal{E}_{\lf},\mathcal{E}_{\rf},\mathcal{E}_{\ff_{b}}\right)}\left(\mathbb{R}_{1}^{1}\right)$,
and therefore we can define the bundle $\Psi_{b}^{-\infty,\left(\mathcal{E}_{\lf},\mathcal{E}_{\rf},\mathcal{E}_{\ff_{b}}\right)}\left(N^{+}\partial X\right)\to\partial X$
as an associated bundle to $N^{+}\partial X\backslash0$.

If $P\in\hat{\Psi}_{0}^{-\infty,\left(\mathcal{E}_{\lf},\mathcal{E}_{\rf},\mathcal{E}_{\ff_{0}}\right)}\left(X\right)$
and $\left[\mathcal{E}_{\ff_{0}}\right]=m$, then its Bessel family
$\hat{N}\left(P\right)$ is a smooth section of the pull-back bundle
$\pi^{*}\Psi^{-\infty,\left(\mathcal{E}_{\lf},\mathcal{E}_{\rf}\right)}\left(N^{+}\partial X\right)\to T^{*}\partial X\backslash0$,
fibrewise homogeneous of degree $-m$, in the sense that
\[
\hat{N}_{t\eta}\left(P\right)=t^{-m}\lambda_{t}^{*}\circ\hat{N}_{\eta}\left(P\right)\circ\lambda_{t^{-1}}^{*}.
\]
Similarly, if $P\in\hat{\Psi}_{0b}^{-\infty,\left(\mathcal{E}_{\lf},\mathcal{E}_{\rf},\mathcal{E}_{\ff_{b}},\mathcal{E}_{\ff_{0}}\right)}\left(X\right)$
and $\left[\mathcal{E}_{\ff_{0}}\right]=m$, then its Bessel family
$\hat{N}\left(P\right)$ is a smooth section of the pull-back bundle
$\pi^{*}\Psi_{b}^{-\infty,\left(\mathcal{E}_{\lf},\mathcal{E}_{\rf},\mathcal{E}_{\ff_{b}}\right)}\left(N^{+}\partial X\right)\to T^{*}\partial X\backslash0$,
fibrewise homogeneous of degree $-m$.

The definition is similar to the previous two cases. Let $P$ be either
in $\hat{\Psi}_{0}^{-\infty,\left(\mathcal{E}_{\lf},\mathcal{E}_{\rf},\mathcal{E}_{\ff_{0}}\right)}\left(X\right)$
or in $\hat{\Psi}_{0b}^{-\infty,\left(\mathcal{E}_{\lf},\mathcal{E}_{\rf},\mathcal{E}_{\ff_{b}},\mathcal{E}_{\ff_{0}}\right)}\left(X\right)$.
Let $q\in\partial X$, and let $\left(x,y\right)$ be coordinates
for $X$ centered at $q$. Then there exists a $0$-interior or $0b$-interior
symbol $p\left(y;x,\tilde{x},\eta\right)$ such that, near the origin
in the induced coordinates $\left(x,y,\tilde{x},\tilde{y}\right)$
for $X^{2}$ centered at $\left(q,q\right)$, we have
\[
P\equiv\frac{1}{\left(2\pi\right)^{n}}\int e^{i\left(y-\tilde{y}\right)\eta}p\left(y;x,\tilde{x},\eta\right)d\eta d\tilde{x}d\tilde{y}.
\]
Given $\eta\in T_{q}^{*}\partial X\backslash0\equiv\mathbb{R}^{n}\backslash0$,
we define
\[
\hat{N}_{\eta}\left(P\right)=\lim_{t\to0^{+}}t^{-m}p\left(0;tx,t\tilde{x},t^{-1}\eta\right)td\tilde{x}.
\]
This limit takes the form
\[
\hat{N}_{\eta}\left(P\right)=\left|\eta\right|^{-m}\tilde{p}_{0}\left(x\left|\eta\right|,\tilde{x}\left|\eta\right|,\frac{\eta}{\left|\eta\right|}\right)\left|\eta\right|d\tilde{x};
\]
in the $0$-interior case, we have \[\tilde{p}_{0}\left(\tau,\tilde{\tau},\hat{\eta}\right)\in\mathcal{A}_{\phg}^{\left(\mathcal{E}_{\lf},\mathcal{E}_{\rf},\infty\right)}\left(\overline{\mathbb{R}_{2}^{2}}\times S^{n-1}\right),\]
while in the $0b$-interior case the function $\tilde{p}_{0}\left(\tau,\tilde{\tau},\hat{\eta}\right)$
\emph{lifts }to a polyhomogeneous function in $\mathcal{A}_{\phg}^{\left(\mathcal{E}_{\lf},\mathcal{E}_{\rf},\mathcal{E}_{\ff_{b}}-1,\infty\right)}\left(\left[\overline{\mathbb{R}_{2}^{2}};0\right]\times S^{n-1}\right)$.
To find $\tilde{p}_{0}$, in both cases we write
\begin{align*}
	p\left(0;x,\tilde{x},\eta\right) & =p_{0}\left(x\left|\eta\right|,\tilde{x}\left|\eta\right|,\frac{\eta}{\left|\eta\right|},\left|\eta\right|^{-1}\right)\left|\eta\right|^{-m+1}\\
	& +r\left(x\left|\eta\right|,\tilde{x}\left|\eta\right|,\frac{\eta}{\left|\eta\right|},\left|\eta\right|^{-1}\right)\left|\eta\right|^{-m},
\end{align*}
where in the $0$-interior case $p_{0}\left(\tau,\tilde{\tau},\hat{\eta},\rho\right)$
is in $\mathcal{A}_{\phg}^{\left(\mathcal{E}_{\lf},\mathcal{E}_{\rf},\infty\right)}\left(\overline{\mathbb{R}_{2}^{2}}\times S^{n-1}\right)\hat{\otimes}\mathcal{A}_{\phg}^{\mathcal{E}_{\ff_{0}}-m}\left(\mathbb{R}_{1}^{1}\right)$,
while in the $0b$-interior case $p_{0}\left(\tau,\tilde{\tau},\hat{\eta},\rho\right)$
lifts to an element of \[\mathcal{A}_{\phg}^{\left(\mathcal{E}_{\lf},\mathcal{E}_{\rf},\mathcal{E}_{\ff_{b}}-1,\infty\right)}\left(\left[\overline{\mathbb{R}_{2}^{2}};0\right]\times S^{n-1}\right)\hat{\otimes}\mathcal{A}_{\phg}^{\mathcal{E}_{\ff_{0}}-m}\left(\mathbb{R}_{1}^{1}\right).\]
$\tilde{p}_{0}\left(\tau,\tilde{\tau},\hat{\eta}\right)$ is simply
the evaluation of $p_{0}\left(\tau,\tilde{\tau},\hat{\eta},\rho\right)$
to $\rho=0$. One can again check that if $p_{1},p_{2}$ are two symbols
such that $P\equiv\Op_{L}^{\inte}\left(p_{1}\right)\equiv\Op_{L}^{\inte}\left(p_{2}\right)$
near the origin, then $p_{1}$ and $p_{2}$ must agree to infinite
order at $\ff_{0}$, so the coefficient $\tilde{p}_{0}\left(\tau,\tilde{\tau},\hat{\eta}\right)$
is in fact only dependent on $P$ and on the choice of coordinates.
Finally, let us check the homogeneity property: we have
\begin{align*}
	\lambda_{t}^{*}\circ\hat{N}_{\eta}\left(P\right)\circ\lambda_{t^{-1}}^{*} & =\left|\eta\right|^{-m}\tilde{p}_{0}\left(tx\left|\eta\right|,t\tilde{x}\left|\eta\right|,\frac{\eta}{\left|\eta\right|}\right)\left|\eta\right|td\tilde{x}\\
	\hat{N}_{t\eta}\left(P\right) & =t^{-m}\left|\eta\right|^{-m}\tilde{p}_{0}\left(tx\left|\eta\right|,t\tilde{x}\left|\eta\right|,\frac{\eta}{\left|\eta\right|}\right)\left|\eta\right|td\tilde{x},
\end{align*}
from which it follows that
\[
\hat{N}_{t\eta}\left(P\right)=t^{-m}\lambda_{t}^{*}\circ\hat{N}_{\eta}\left(P\right)\circ\lambda_{t^{-1}}^{*}
\]
as claimed.
\begin{rem}
	\label{rem:why-0-calculus-not-contained-in-symbolic-0-operators}In
	Remark \ref{rem:0-calculus-not-in-0-interior}, we claimed that given
	index sets $\mathcal{E}_{\lf},\mathcal{E}_{\rf},\mathcal{E}_{\ff_{0}}$,
	we have a \emph{strict} inclusion
	\[
	\hat{\Psi}_{0}^{-\infty,\left(\mathcal{E}_{\lf},\mathcal{E}_{\rf},\mathcal{E}_{\ff_{0}}\right)}\left(X\right)\subset\Psi_{0}^{-\infty,\left(\mathcal{E}_{\lf},\mathcal{E}_{\rf},\tilde{\mathcal{E}}_{\ff_{0}}\right)}\left(X\right)
	\]
	where $\tilde{\mathcal{E}}_{\ff_{0}}=\mathcal{E}_{\ff_{0}}\eunion\left(\mathcal{E}_{\lf}+\mathcal{E}_{\rf}+n+1\right)$.
	We can use the Bessel family to explain why. Indeed, let $P\in\hat{\Psi}_{0}^{-\infty,\left(\mathcal{E}_{\lf},\mathcal{E}_{\rf},\mathcal{E}_{\ff_{0}}\right)}\left(X\right)$
	and suppose for simplicity that $\left[\mathcal{E}_{\ff_{0}}\right]=0$
	and that $\mathcal{E}_{\lf}+\mathcal{E}_{\rf}>-1$, so that $\left[\tilde{\mathcal{E}}_{\ff_{0}}\right]=0$
	too. By definition, the Bessel family $\hat{N}\left(P\right)$ is
	a (smooth, fibrewise homogeneous of degree $0$) section of $\pi^{*}\Psi^{-\infty,\left(\mathcal{E}_{\lf},\mathcal{E}_{\rf}\right)}\left(N^{+}\partial X\right)\to T^{*}\partial X\backslash0$.
	Therefore, for every $\eta\in T_{p}^{*}\partial X\backslash0$, $\hat{N}_{\eta}\left(P\right)$
	is a \emph{very residual} operator on $N_{p}^{+}\partial X$. On the
	other hand, for a generic element $Q\in\Psi_{0}^{-\infty,\left(\mathcal{E}_{\lf},\mathcal{E}_{\rf},\tilde{\mathcal{E}}_{\ff_{0}}\right)}\left(X\right)$,
	the Bessel family $\hat{N}\left(Q\right)$ is a (smooth, fibrewise
	homogeneous of degree $0$) section of $\pi^{*}\Psi_{b}^{-\infty,\left(\mathcal{E}_{\lf},\mathcal{E}_{\rf},\tilde{\mathcal{E}}_{\ff_{b}}\right)}\left(N^{+}\partial X\right)\to T^{*}\partial X\backslash0$.
	Therefore, for every $\eta\in T_{p}^{*}\partial X\backslash0$, $\hat{N}_{\eta}\left(Q\right)$
	is in the \emph{large residual $b$-calculus} on $N_{p}^{+}\partial X$!
	This is generically the best we can say. A concrete example of this
	phenomenon comes from the parametrix construction for $0$-elliptic
	operators carried out in \cite{MazzeoEdgeI}. Let $L\in\Diff_{0}^{m}\left(X\right)$
	be $0$-elliptic with constant indicial roots, and assume that $\delta\in\mathbb{R}$
	is an invertible weight, so that one can construct a parametrix $Q\in\Psi_{0}^{-m}\left(X\right)+\Psi_{0}^{-\infty,\mathcal{H}}\left(X\right)$,
	with $\mathcal{H}=\left(\mathcal{H}_{\lf},\mathcal{H}_{\rf},\mathcal{H}_{\ff_{0}}\right)$
	satisfying $\Re\left(\mathcal{H}_{\lf}\right)>\delta$, $\Re\left(\mathcal{H}_{\rf}\right)>-\delta-1$,
	and $\left[\mathcal{H}_{\ff_{0}}\right]=0$, such that $LQ$ and $QL$
	differ from the identity by very residual operators. The crucial step
	in the construction of $Q$ is the inversion of the Bessel family
	$\hat{N}\left(L\right)$: it is explained in \cite{MazzeoEdgeI} (see
	also \cite{Hintz0calculus}) that, decomposing $Q=Q_{0}+Q_{1}$ with
	$Q_{0}\in\Psi_{0}^{-m}\left(X\right)$ and $Q_{1}\in\Psi_{0}^{-\infty,\mathcal{H}}\left(X\right)$,
	then for any $p\in\partial X$ and $\eta\in T_{p}^{*}\partial X\backslash0$,
	we have
	\begin{align*}
		\hat{N}_{\eta}\left(Q_{0}\right) & \in\Psi_{b,\text{sc}}^{-m}\left(N_{p}^{+}\partial X\right)\\
		\hat{N}_{\eta}\left(Q_{1}\right) & \in\Psi_{b}^{-\infty,\mathcal{H}}\left(N_{p}^{+}\partial X\right).
	\end{align*}
	The space $\Psi_{b,\text{sc}}^{-m}\left(N_{p}^{+}\partial X\right)$
	where $\hat{N}_{\eta}\left(Q_{0}\right)$ lives is the \emph{$b$-scattering
		calculus} on $N_{p}^{+}\partial X$; it is not important for our purposes,
	but a precise definition is contained in \cite{Hintz0calculus}. The
	main point for this is that $\hat{N}_{\eta}\left(Q_{1}\right)$ is
	\emph{not} a very residual operator on $N_{p}^{+}\partial X$. Therefore
	$Q_{1}$ is an element of the large residual $0$-calculus, but it
	is not a $0$-interior operator.
\end{rem}

\subsubsection{Compatibility with the physical Bessel family}

Let $A\in\hat{\Psi}_{0\tr}^{-\infty,\left(\mathcal{E}_{\of},\mathcal{E}_{\ff}\right)}\left(X,\partial X\right)$.
Then, by Corollary \ref{cor:global-symbolic-0-trace-are-physical-0-trace-and-viceversa},
we can think of $A$ as an element of $A\in\Psi_{0\tr}^{-\infty,\left(\mathcal{E}_{\of},\tilde{\mathcal{E}}_{\ff}\right)}\left(X,\partial X\right)$
with $\tilde{\mathcal{E}}_{\ff}=\mathcal{E}_{\ff}\eunion\left(\mathcal{E}_{\of}+n+1\right)$.
If $\left[\mathcal{E}_{\ff}\right]=0$ and $\Re\left(\mathcal{E}_{\of}\right)>-n-1$,
then $\left[\tilde{\mathcal{E}}_{\ff}\right]=0$ as well and we can
define the Bessel family $\hat{N}\left(A\right)$ either following
the ``physical'' approach described in \S\ref{subsubsec:The-Bessel-family-physical},
or the ``symbolic'' approach described in \S\ref{subsubsec:The-Bessel-family-of-symbolic-0-trace}.
These two definitions of course coincide. More in general (cf.\ \S2.6.2
of \cite{UsulaPhD}), the ``physical'' definition of the Bessel
family can be extended to operators in $\Psi_{0\tr}^{-\infty,\left(\mathcal{E}_{\of},\tilde{\mathcal{E}}_{\ff}\right)}\left(X,\partial X\right)$
with $\left[\tilde{\mathcal{E}}_{\ff}\right]=m\in\mathbb{R}$; in
this case, the physical notion coincides with the symbolic one, provided
that $\Re\left(\mathcal{E}_{\of}\right)>m-n-1$. Similar statements
hold for the other classes. Specifically, the compatibility between
physical and symbolic Bessel family holds in the following cases:
\begin{enumerate}
	\item $B\in\hat{\Psi}_{0\po}^{-\infty,\left(\mathcal{E}_{\of},\mathcal{E}_{\ff}\right)}\left(\partial X,X\right)$
	with $\left[\mathcal{E}_{\ff}\right]=m$ and $\Re\left(\mathcal{E}_{\of}\right)>m-n$
	(cf.\ Corollary \ref{cor:global-symbolic-0-poisson-are-physical-0-poisson-and-viceversa});
	\item $P\in\hat{\Psi}_{0}^{-\infty,\left(\mathcal{E}_{\lf},\mathcal{E}_{\rf},\mathcal{E}_{\ff_{0}}\right)}\left(X\right)$
	with $\left[\mathcal{E}_{\ff_{0}}\right]=m$ and $\Re\left(\mathcal{E}_{\lf}+\mathcal{E}_{\rf}\right)>m-n-1$
	(cf.\ Corollary \ref{cor:global-relation-symbolic0b-physical0b});
	\item $P\in\hat{\Psi}_{0b}^{-\infty,\left(\mathcal{E}_{\lf},\mathcal{E}_{\rf},\mathcal{E}_{\ff_{0}},\mathcal{E}_{\ff_{b}}\right)}\left(X\right)$
	with $\left[\mathcal{E}_{\ff_{0}}\right]=m$ and $\Re\left(\mathcal{E}_{\ff_{b}}\right)>m-n$
	(cf.\ Corollary \ref{cor:global-relation-symbolic0b-physical0b}).
\end{enumerate}
For simplicity, we discuss the proof in detail only in the $0$-trace
case with $m=0$; all the other proofs are essentially identical.

Let $A\in\hat{\Psi}_{0\tr}^{-\infty,\left(\mathcal{E}_{\of},\mathcal{E}_{\ff}\right)}\left(X,\partial X\right)$
with $\left[\mathcal{E}_{\ff}\right]=0$ and $\Re\left(\mathcal{E}_{\of}\right)>-n-1$.
Let us briefly recall from \S\ref{subsubsec:The-Bessel-family-physical}
the definition of the ``physical'' Bessel map at a point $p\in\partial X$.
Choose coordinates $\left(x,y\right)$ centered at $p$ such that
$Vx\equiv1$. Since $A\in\Psi_{0\tr}^{-\infty,\left(\mathcal{E}_{\of},\tilde{\mathcal{E}}_{\ff}\right)}\left(X,\partial X\right)$
by Corollary \ref{cor:global-symbolic-0-trace-are-physical-0-trace-and-viceversa},
there exists a distribution $K\left(y;\tilde{x},Y\right)$ on $\mathbb{R}^{n}\times\mathbb{R}_{1}^{n+1}$
of the form
\[
K\left(y;\tilde{x},Y\right)=K''\left(y;\sqrt{\tilde{x}^{2}+\left|Y\right|^{2}},\frac{Y}{\tilde{x}}\right)
\]
for some $K''\left(y;r,u\right)$ polyhomogeneous on the model space
$\mathbb{R}^{n}\times\mathbb{R}_{1}^{1}\times\overline{\mathbb{R}^{n}}$
with index sets $\tilde{\mathcal{E}}_{\ff}$ at $r=0$ and $\mathcal{E}_{\of}+n+1$
at $\left|u\right|=\infty$, such that near $\tilde{x}=0$, $y=\tilde{y}=0$
we have
\begin{align*}
	K_{A} & \equiv K\left(y;\tilde{x},y-\tilde{y}\right)d\tilde{x}d\tilde{y}\\
	& \equiv K''\left(y;\sqrt{\tilde{x}^{2}+\left|y-\tilde{y}\right|^{2}},\frac{y-\tilde{y}}{\tilde{x}}\right)\tilde{x}^{-n-1}d\tilde{x}d\tilde{y}.
\end{align*}
The normal operator is the operator from functions on $X_{p}\equiv\mathbb{R}_{1}^{n+1}$
to functions on $\partial X_{p}=\mathbb{R}^{n}$, expressed in the
induced half-space coordinates on $X_{p}$ as
\[
K''_{0}\left(\frac{y-\tilde{y}}{\tilde{x}}\right)\tilde{x}^{-n-1}d\tilde{x}d\tilde{y},
\]
with $K_{0}''\in\mathcal{A}_{\phg}^{\mathcal{E}_{\of}+n+1}\left(\overline{\mathbb{R}^{n}}\right)$
defined as $K''_{0}\left(u\right)=K''\left(0;0,u\right)$. It follows
that an equivalent way to define $N_{p}\left(A\right)$ is via the
limit
\begin{align*}
	N_{p}\left(A\right) & =\lim_{t\to0^{+}}\left[K''\left(ty;t\sqrt{\tilde{x}^{2}+\left|y-\tilde{y}\right|^{2}},\frac{y-\tilde{y}}{\tilde{x}}\right)\tilde{x}^{-n-1}\right]d\tilde{x}d\tilde{y}\\
	& =\lim_{t\to0^{+}}\left[K\left(ty;t\tilde{x},t\left(y-\tilde{y}\right)\right)t^{n+1}\right]d\tilde{x}d\tilde{y}.
\end{align*}
The Bessel family $\hat{N}_{\eta}\left(A\right)$ at $T_{p}^{*}\partial X\backslash0\equiv\mathbb{R}^{n}\backslash0$
is then defined as the map from functions on $\mathbb{R}_{1}^{1}$
to $\mathbb{C}$, parametrized by $\eta\in\mathbb{R}^{n}\backslash0$,
obtained as
\[
\hat{N}_{\eta}\left(A\right)=\hat{K}''_{0}\left(\tilde{x}\eta\right)\frac{d\tilde{x}}{\tilde{x}}
\]
where $\hat{K}''_{0}\left(\eta\right)$ is the Fourier transform of
$K''_{0}$. In formulas,
\begin{align*}
	\hat{N}_{\eta}\left(A\right) & =\left[\int e^{-i\tilde{x}\eta u}K''_{0}\left(u\right)du\right]\frac{d\tilde{x}}{\tilde{x}}\\
	& =\left[\int e^{-i\tilde{x}\eta u}\left[\lim_{t\to0^{+}}K''\left(ty;t\tilde{x}\sqrt{1+\left|u\right|^{2}},u\right)\right]du\right]\frac{d\tilde{x}}{\tilde{x}}\\
	& =\left[\int e^{-i\tilde{x}\eta u}\left[\lim_{t\to0^{+}}K\left(ty;t\tilde{x},t\tilde{x}u\right)t^{n+1}\tilde{x}^{n+1}\right]du\right]\frac{d\tilde{x}}{\tilde{x}}\\
	& =\left[\int e^{-i\eta u}\left[\lim_{t\to0^{+}}K\left(ty;t\tilde{x},tu\right)t^{n+1}\right]du\right]d\tilde{x}.
\end{align*}
Now, since $A\in\hat{\Psi}_{0\tr}^{-\infty,\left(\mathcal{E}_{\of},\mathcal{E}_{\ff}\right)}\left(X,\partial X\right)$,
we can assume that $K\left(y;\tilde{x},Y\right)$ is obtained as the
inverse Fourier transform of a left-reduced symbol of $0$-trace type,
\[
K\left(y;\tilde{x},Y\right)=\frac{1}{\left(2\pi\right)^{n}}\int e^{iY\xi}a\left(y;\tilde{x},\xi\right)d\xi
\]
for some $a\in S_{0\tr,\mathcal{S}}^{-\infty,\mathcal{E}}\left(\mathbb{R}^{n};\mathbb{R}_{1}^{n+1}\right)$.
Rescaling and changing variables, we obtain
\[
K\left(ty;t\tilde{x},tu\right)t^{n+1}=\frac{1}{\left(2\pi\right)^{n}}\int e^{iu\xi}a\left(ty;t\tilde{x},t^{-1}\xi\right)td\xi.
\]
The integrand above in the formula for $\hat{N}_{\eta}\left(A\right)$
is therefore the limit
\[
\lim_{t\to0^{+}}\frac{1}{\left(2\pi\right)^{n}}\int e^{iu\xi}a\left(ty;t\tilde{x},t^{-1}\xi\right)td\xi.
\]
Using a standard integration by parts argument for oscillatory integrals
(cf.\ Chapter 4 of \cite{HintzMicrolocal}), we can take the limit
inside the integral and obtain
\begin{align*}
	\lim_{t\to0^{+}}K\left(ty;t\tilde{x},tu\right)t^{n+1} & =\frac{1}{\left(2\pi\right)^{n}}\int e^{iu\xi}\left[\lim_{t\to0^{+}}a\left(ty;t\tilde{x},t^{-1}\xi\right)t\right]d\xi.
\end{align*}
Substituting in the formula above for $\hat{N}_{\eta}\left(A\right)$
and applying the Fourier inversion theorem, we obtain
\[
\hat{N}_{\eta}\left(A\right)=\left[\lim_{t\to0^{+}}a\left(ty;t\tilde{x},t^{-1}\eta\right)t\right]d\tilde{x},
\]
which coincides with the ``symbolic'' Bessel family of $A$ as defined
in \S\ref{subsubsec:The-Bessel-family-of-symbolic-0-trace}.

\subsection{\label{subsec:Twisted--trace-and-poisson}Twisted symbolic \texorpdfstring{$0$}{0}-trace
	and \texorpdfstring{$0$}{0}-Poisson operators}

In the previous sub-sections, we introduced the classes $\hat{\Psi}_{0\tr}^{-\infty,\bullet}\left(X,\partial X\right)$
and $\hat{\Psi}_{0\po}^{-\infty,\bullet}\left(\partial X,X\right)$,
and we constructed the Bessel family maps, associating to elements
of $\hat{\Psi}_{0\tr}^{-\infty,\bullet}\left(X,\partial X\right)$
(resp.\ $\hat{\Psi}_{0\po}^{-\infty,\bullet}\left(\partial X,X\right)$)
fibrewise homogeneous sections of certain Fréchet bundles over $T^{*}\partial X\backslash0$. 

Now, in \S\ref{subsec:analysis-of-The-Bessel-trace}, we constructed
the \emph{Bessel trace family }$\hat{N}\left(\boldsymbol{A}_{L}\right)$
for a $0$-elliptic operator $L\in\Diff_{0}^{m}\left(X\right)$ with
constant indicial roots, relative to a surjective, not injective weight
$\delta$ for $L$, and an injective weight $\overline{\delta}$.
More precisely, $\hat{N}\left(\boldsymbol{A}_{L}\right)$ depends
on auxiliary choices of a vector field $V$ on $X$ transverse to
$\partial X$ and inward-pointing, a boundary defining function $x$
compatible with $V$, and a smooth positive density $\omega$ on $X$.
Given these choices, for $p\in\partial X$ and $\eta\in T_{p}^{*}\partial X\backslash0$,
$\hat{N}_{\eta}\left(\boldsymbol{A}_{L}\right)$ maps functions $u\in x^{\delta}L_{b}^{2}\left(N_{p}^{+}\partial X\right)$
to the terms in the expansion of $\hat{N}_{\eta}\left(P_{1}\right)u$
(the orthogonal projection of $u$ onto the kernel of $\hat{N}_{\eta}\left(L\right)$,
with respect to the inner product induced by the choices above) corresponding
to the critical indicial roots of $L$. As such, $\hat{N}_{\eta}\left(\boldsymbol{A}_{L}\right)$
takes values in the fibre $\boldsymbol{E}_{L|p}$, where 
\[
\boldsymbol{E}_{L}=\bigoplus_{\text{\ensuremath{\Re\left(\mu\right)\in\left(\delta,\overline{\delta}\right)}}}E_{\mu,\tilde{M}_{\mu}}
\]
and $E_{\mu,\tilde{M}_{\mu}}\to\partial X$ is the bundle whose sections
are the fibrewise log-homogeneous functions on $N^{+}\partial X$
of degree $\leq\left(\mu,\tilde{M}_{\mu}\right)$ (cf.\ \S\ref{subsubsec:Log-homogeneous-functions}).
We observed that $\boldsymbol{E}_{L}$ comes equipped with a natural
endomorphism $\boldsymbol{\mathfrak{s}}_{L}:\boldsymbol{E}_{L}\to\boldsymbol{E}_{L}$,
induced by the action of the canonical dilation invariant vector field
on $N^{+}\partial X$ tangent to the fibers. This endomorphism has
constant eigenvalues, corresponding to the critical indicial roots
of $L$. We then observed that $\hat{N}\left(\boldsymbol{A}_{L}\right)$
is a smooth section of \[\pi^{*}\Psi_{\tr}^{-\infty,\mathcal{E}_{\rf}}\left(N^{+}\partial X;\boldsymbol{E}_{L}\right)\to T^{*}\partial X\backslash0,\]
\emph{fibrewise} \emph{twisted homogeneous of} \emph{degree} $\boldsymbol{\mathfrak{s}}_{L}$,
in the sense that for every $\eta\in T^{*}\partial X\backslash0$
and $t\in\mathbb{R}^{+}$, we have
\[
\hat{N}_{t\eta}\left(\boldsymbol{A}_{L}\right)=t^{\boldsymbol{\mathfrak{s}}_{L}}\hat{N}_{\eta}\left(\boldsymbol{A}_{L}\right)\circ\lambda_{t^{-1}}^{*}.
\]
Furthermore, in \S\ref{subsec:Solving-the-model}, in order to solve
our model boundary value problem at the Bessel level, we had to construct
a \emph{Bessel Poisson family} $\hat{N}\left(\boldsymbol{B}\right)$
associated to an elliptic boundary condition. This Bessel Poisson
family $\hat{N}\left(\boldsymbol{B}\right)$ is a smooth section of
\[\pi^{*}\Psi_{\po}^{-\infty,\mathcal{E}_{\lf}}\left(N^{+}\partial X;\boldsymbol{E}_{L}^{*}\right)\to T^{*}\partial X\backslash0.\]
Given $p\in\partial X$ and $\eta\in T_{p}^{*}\partial X\backslash0$,
$\hat{N}_{\eta}\left(\boldsymbol{B}\right)$ maps a vector $\varphi\in\boldsymbol{E}_{L|p}$
to the unique $x^{\delta}L_{b}^{2}$ solution $u$ of $\hat{N}_{\eta}\left(L\right)u=0$
for which its trace $\hat{N}_{\eta}\left(\boldsymbol{A}_{L}\right)u$
equals a certain projection of $\varphi$ onto the range of $\hat{N}_{\eta}\left(\boldsymbol{A}_{L}\right)$,
where this projection is determined by the boundary condition. We
observed that this time $\hat{N}\left(\boldsymbol{B}\right)$ is \emph{fibrewise
	twisted homogeneous }of degree $-\boldsymbol{\mathfrak{s}}_{L}$ ,
meaning that for every $\eta\in T^{*}\partial X\backslash0$ and every
$t\in\mathbb{R}^{+}$, we have
\[
\hat{N}_{t\eta}\left(\boldsymbol{B}\right)=\lambda_{t}^{*}\circ\hat{N}_{\eta}\left(\boldsymbol{B}\right)t^{-\boldsymbol{\mathfrak{s}}_{L}}.
\]

From this discussion, we see the need to generalize the classes of $0$-trace and $0$-Poisson operators to classes whose Bessel families are fibrewise \emph{twisted} homogeneous
sections of the relevant Fréchet bundles over $T^{*}\partial X\backslash0$.
In this sub-section, we construct these generalizations, which we
call \emph{twisted symbolic $0$-trace }and \emph{$0$-Poisson operators}.

\subsubsection{\label{subsec:Definitions-twisted-trace-poisson}Definitions}

Let us establish some notations and some easy properties. Let $\boldsymbol{E}\to\partial X$
be a complex vector bundle, and let $\boldsymbol{\mathfrak{s}}:\boldsymbol{E}\to\boldsymbol{E}$
be a smooth bundle endomorphism with\emph{ }constant\emph{ }eigenvalues.
Enumerate the eigenvalues of $\boldsymbol{\mathfrak{s}}$ as $\boldsymbol{\mu}=\left(\mu_{1},...,\mu_{I}\right)$.
The fact that the eigenvalues are constant implies that the multiplicity
$M_{i}+1$ of $\mu_{i}$ is constant as well, for every $i$. Recall
that the \emph{generalized eigenbundle }of $\boldsymbol{\mathfrak{s}}$
with respect to $\mu_{i}$ is the kernel of $\left(\boldsymbol{\mathfrak{s}}-\mu_{i}\right)^{M_{i}}$.
\begin{lem}
	The generalized eigenbundle $E_{i}$ of $\boldsymbol{\mathfrak{s}}$
	with respect to $\mu_{i}$ is a smooth subbundle of $\boldsymbol{E}$
	of rank $M_{i}$. The bundle $\boldsymbol{E}$ decomposes as the direct
	sum $\boldsymbol{E}=\bigoplus_{i}E_{i}$, and $\boldsymbol{\mathfrak{s}}$
	decomposes accordingly as $\boldsymbol{\mathfrak{s}}=\bigoplus_{i}\mathfrak{s}_{i}$
	where $\mathfrak{s}_{i}:E_{i}\to E_{i}$ has a unique eigenvalue $\mu_{i}$.
\end{lem}

\begin{proof}
	The fact that, for every $p\in\partial X$, the fiber $E_{i|p}$ has
	dimension $M_{i}$ is a standard linear algebra fact. Now, the endomorphism
	$\left(\boldsymbol{\mathfrak{s}}-\mu_{i}\right)^{M_{i}}$ acts fibrewise
	isomorphically on $\bigoplus_{j\not=i}E_{j}$. It follows that $\left(\boldsymbol{\mathfrak{s}}-\mu_{i}\right)^{M_{i}}$,
	as an endomorphism of $\boldsymbol{E}$, has constant rank. This guarantees
	that $E_{i}$ is a smooth subbundle of $\boldsymbol{E}$.
\end{proof}
\begin{rem}
	The example we need to keep in mind is the bundle $\boldsymbol{E}_{L}\to\partial X$
	associated to a $0$-elliptic operator $L$ with constant indicial
	roots and a surjective, not injective weight $\delta$. $\boldsymbol{E}_{L}$
	comes equipped with an endomorphism $\boldsymbol{\mathfrak{s}}_{L}$
	whose eigenvalues are the indicial roots of $L$. The decomposition
	$\boldsymbol{E}_{L}$ as a direct sum $\bigoplus_{\text{\ensuremath{\Re\left(\mu\right)\in\left(\delta,\overline{\delta}\right)}}}E_{\mu,\tilde{M}_{\mu}}$,
	and the corresponding decomposition 
	\[
	\boldsymbol{\mathfrak{s}}_{L}=\bigoplus_{\Re\left(\mu\right)\in\left(\delta,\overline{\delta}\right)}\left(\mu+\mathfrak{N}_{\tilde{M}_{\mu}}\right),
	\]
	correspond precisely to the generalized eigenbundle decomposition
	of $\boldsymbol{E}_{L}$ with respect to $\boldsymbol{\mathfrak{s}}_{L}$.
\end{rem}

We now discuss the local theory. Fix $M\in\mathbb{N}$. We will now
define subclasses of \[\hat{\Psi}_{0\po,\mathcal{S}}^{-\infty,\bullet}\left(\mathbb{R}^{n},\mathbb{C}^{M};\mathbb{R}_{1}^{n+1}\right)\]
and \[\hat{\Psi}_{0\tr,\mathcal{S}}^{-\infty,\bullet}\left(\mathbb{R}_{1}^{n+1};\mathbb{R}^{n},\mathbb{C}^{M}\right),\]
which we call the classes of \emph{twisted $0$-Poisson }and $0$\emph{-trace}
operators. In what follows, we distinguish the space $\mathbb{C}^{M}$,
whose elements are interpreted as column vectors, from its dual $\left(\mathbb{C}^{M}\right)^{*}$
whose elements are interpreted as row vectors.

Given an index family $\mathcal{E}=\left(\mathcal{E}_{\of},\mathcal{E}_{\ff}\right)$,
we define
\begin{align*}
	S_{0\po,\mathcal{S}}^{-\infty,\mathcal{E}}\left(\mathbb{R}^{n};\mathbb{R}_{1}^{n+1};\left(\mathbb{C}^{M}\right)^{*}\right) & =S_{0\po,\mathcal{S}}^{-\infty,\mathcal{E}}\left(\mathbb{R}^{n};\mathbb{R}_{1}^{n+1}\right)\otimes\left(\mathbb{C}^{M}\right)^{*}\\
	S_{0\tr,\mathcal{S}}^{-\infty,\mathcal{E}}\left(\mathbb{R}^{n};\mathbb{R}_{1}^{n+1};\mathbb{C}^{M}\right) & =S_{0\tr,\mathcal{S}}^{-\infty,\mathcal{E}}\left(\mathbb{R}^{n};\mathbb{R}_{1}^{n+1}\right)\otimes\mathbb{C}^{M}.
\end{align*}
In other words, an element $b\in S_{0\po,\mathcal{S}}^{-\infty,\mathcal{E}}\left(\mathbb{R}^{n};\mathbb{R}_{1}^{n+1};\left(\mathbb{C}^{M}\right)^{*}\right)$
is simply a row vector
\[
b=\left(\begin{matrix}b_{1} & \cdots & b_{M}\end{matrix}\right)
\]
with entries in $S_{0\po,\mathcal{S}}^{-\infty,\mathcal{E}}\left(\mathbb{R}^{n};\mathbb{R}_{1}^{n+1}\right)$
for every $i$, while an element $a\in S_{0\tr,\mathcal{S}}^{-\infty,\mathcal{E}}\left(\mathbb{R}^{n};\mathbb{R}_{1}^{n+1};\mathbb{C}^{M}\right)$
is simply a column vector
\[
a=\left(\begin{matrix}a^{1}\\
	\vdots\\
	a^{M}
\end{matrix}\right)
\]
with entries in $S_{0\tr,\mathcal{S}}^{-\infty,\mathcal{E}}\left(\mathbb{R}^{n};\mathbb{R}_{1}^{n+1}\right)$.
Given $b\in S_{0\po,\mathcal{S}}^{-\infty,\mathcal{E}}\left(\mathbb{R}^{n};\mathbb{R}_{1}^{n+1};\left(\mathbb{C}^{M}\right)^{*}\right)$,
we interpret its left quantization $B=\Op_{L}^{\po}\left(b\right)$
as an operator taking $\mathbb{C}^{M}$-valued functions on $\mathbb{R}^{n}$
to functions on $\mathbb{R}_{1}^{n+1}$; as such, we write $B\in\hat{\Psi}_{0\po,\mathcal{S}}^{-\infty,\mathcal{E}}\left(\mathbb{R}^{n},\mathbb{C}^{M};\mathbb{R}_{1}^{n+1}\right)$.
Similarly, given $a\in S_{0\tr,\mathcal{S}}^{-\infty,\mathcal{E}}\left(\mathbb{R}^{n};\mathbb{R}_{1}^{n+1};\mathbb{C}^{M}\right)$,
we interpret its left quantization $A=\Op_{L}^{\tr}\left(a\right)$
as an operator taking functions on $\mathbb{R}_{1}^{n+1}$ to $\mathbb{C}^{M}$-valued
functions on $\mathbb{R}^{n}$; as such, we write $A\in\hat{\Psi}_{0\tr,\mathcal{S}}^{-\infty,\mathcal{E}}\left(\mathbb{R}_{1}^{n+1};\mathbb{R}^{n},\mathbb{C}^{M}\right)$.
In other words, on the operator side, semicolons separate domain from
codomain, and the meaning of the comma is ``with values in''.
\begin{defn}
	(Twisted $0$-trace and $0$-Poisson symbols) Let $\mathfrak{s}:\overline{\mathbb{R}^{n}}\to\mathfrak{gl}\left(M,\mathbb{C}\right)$,
	$\mathfrak{s}=\mathfrak{s}\left(y\right)$, be a smooth family of
	matrices \emph{with a constant, single eigenvalue}. We define
	\begin{align*}
		S_{0\po,\mathcal{S}}^{-\infty,\mathcal{E},\mathfrak{s}}\left(\mathbb{R}^{n};\mathbb{R}_{1}^{n+1};\left(\mathbb{C}^{M}\right)^{*}\right) & =S_{0\po,\mathcal{S}}^{-\infty,\mathcal{E}}\left(\mathbb{R}^{n};\mathbb{R}_{1}^{n+1};\left(\mathbb{C}^{M}\right)^{*}\right)\left\langle \eta\right\rangle ^{\mathfrak{s}\left(y\right)}\\
		S_{0\tr,\mathcal{S}}^{-\infty,\mathcal{E},\mathfrak{s}}\left(\mathbb{R}^{n};\mathbb{R}_{1}^{n+1};\mathbb{C}^{M}\right) & =\left\langle \eta\right\rangle ^{-\mathfrak{s}\left(y\right)}S_{0\tr,\mathcal{S}}^{-\infty,\mathcal{E}}\left(\mathbb{R}^{n};\mathbb{R}_{1}^{n+1};\mathbb{C}^{M}\right).
	\end{align*}
	We call the elements of $S_{0\po,\mathcal{S}}^{-\infty,\mathcal{E},\mathfrak{s}}\left(\mathbb{R}^{n};\mathbb{R}_{1}^{n+1};\left(\mathbb{C}^{M}\right)^{*}\right)$
	(resp.\ $S_{0\tr,\mathcal{S}}^{-\infty,\mathcal{E},\mathfrak{s}}\left(\mathbb{R}^{n};\mathbb{R}_{1}^{n+1};\mathbb{C}^{M}\right)$)
	\emph{twisted $0$-Poisson symbols }(resp.\ \emph{twisted $0$-trace
		symbols}).
\end{defn}

\begin{rem}
	Note the different sign convention for the exponent of the twisting
	factor in the $0$-trace and $0$-Poisson cases. This is just a conventional
	choice, made in order to make the composition theorems look more natural.
\end{rem}

\begin{rem}
	If $\mu$ is the only eigenvalue of $\mathfrak{s}$, call $\mathfrak{s}'=\mathfrak{s}-\mu$.
	Then $\mathfrak{s}'$ is a smooth family of \emph{nilpotent} matrices,
	and
	\begin{align*}
		S_{0\po,\mathcal{S}}^{-\infty,\mathcal{E},\mathfrak{s}}\left(\mathbb{R}^{n};\mathbb{R}_{1}^{n+1};\left(\mathbb{C}^{M}\right)^{*}\right) & =S_{0\po,\mathcal{S}}^{-\infty,\left(\mathcal{E}_{\of},\mathcal{E}_{\ff}-\mu\right),\mathfrak{s}'}\left(\mathbb{R}^{n};\mathbb{R}_{1}^{n+1};\left(\mathbb{C}^{M}\right)^{*}\right)\\
		S_{0\tr,\mathcal{S}}^{-\infty,\mathcal{E},\mathfrak{s}}\left(\mathbb{R}^{n};\mathbb{R}_{1}^{n+1};\mathbb{C}^{M}\right) & =S_{0\tr,\mathcal{S}}^{-\infty,\left(\mathcal{E}_{\of},\mathcal{E}_{\ff}+\mu\right),\mathfrak{s}'}\left(\mathbb{R}^{n};\mathbb{R}_{1}^{n+1};\mathbb{C}^{M}\right).
	\end{align*}
	For this reason, it is sufficient to establish the local theory when
	$\mathfrak{s}$ is nilpotent.
\end{rem}

The next lemma shows that twisted $0$-trace and $0$-Poisson symbols
can be interpreted as (vector-valued) un-twisted $0$-trace and $0$-Poisson
symbols.
\begin{lem}
	\label{lem:twisting-poisson-trace-are-polyhomogeneous-1}Let $\mathfrak{s}:\overline{\mathbb{R}^{n}}\to\mathfrak{gl}\left(M,\mathbb{C}\right)$
	be a smooth family of matrices with a single constant eigenvalue $\mu\in\mathbb{C}$,
	and let $\mathcal{E}=\left(\mathcal{E}_{\of},\mathcal{E}_{\ff}\right)$
	be a pair of index sets. Then
	\begin{align*}
		S_{0\po,\mathcal{S}}^{-\infty,\mathcal{E},\mathfrak{s}}\left(\mathbb{R}^{n};\mathbb{R}_{1}^{n+1};\left(\mathbb{C}^{M}\right)^{*}\right) & \subseteq S_{0\po,\mathcal{S}}^{-\infty,\left(\mathcal{E}_{\of},\mathcal{F}\right)}\left(\mathbb{R}^{n};\mathbb{R}_{1}^{n+1};\left(\mathbb{C}^{M}\right)^{*}\right)\\
		S_{0\tr,\mathcal{S}}^{-\infty,\mathcal{E},\mathfrak{s}}\left(\mathbb{R}^{n};\mathbb{R}_{1}^{n+1};\mathbb{C}^{M}\right) & \subseteq S_{0\tr,\mathcal{S}}^{-\infty,\left(\mathcal{E}_{\of},\mathcal{G}\right)}\left(\mathbb{R}^{n};\mathbb{R}_{1}^{n+1};\mathbb{C}^{M}\right),
	\end{align*}
	where $\mathcal{F}$ (resp.\ $\mathcal{G}$) is the index set generated
	by the pairs $\left(\alpha-\mu,k+l\right)$ (resp.\ $\left(\alpha+\mu,k+l\right)$)
	such that $\left(\alpha,k\right)\in\mathcal{E}_{\ff}$ and $0\leq l<M$.
\end{lem}

\begin{proof}
	We prove only the $0$-Poisson case, since the $0$-trace case is
	essentially equal. Consider a twisted $0$-Poisson symbol of the form
	\[
	\tilde{b}\left(y;x,\eta\right)=b\left(y;x,\eta\right)\left\langle \eta\right\rangle ^{\mathfrak{s}\left(y\right)},
	\]
	with $b\in S_{0\po,\mathcal{S}}^{-\infty,\mathcal{E}}$, so that $\tilde{b}\in S_{0\po,\mathcal{S}}^{-\infty,\mathcal{E},\mathfrak{s}}$.
	Since $\mathfrak{s}\left(y\right)$ has a unique constant eigenvalue
	$\mu$, the difference $\mathfrak{s}'\left(y\right)=\mathfrak{s}\left(y\right)-\mu$
	is a smooth family of nilpotent matrices, and therefore we can write
	\[
	\left\langle \eta\right\rangle ^{\mathfrak{s}\left(y\right)}=\left\langle \eta\right\rangle ^{\mu}\sum_{j<M}\frac{1}{j!}\left(\log\left\langle \eta\right\rangle \right)^{j}\left(\mathfrak{s}'\left(y\right)\right)^{j}.
	\]
	Now, observe that $\left\langle \eta\right\rangle $ has index sets
	$0$ at the faces $\of,\iif_{x}$ of $\hat{P}^{2}$, while it has
	index set $-1$ at $\iif_{\eta}$. Consequently, by the pull-back
	Theorem, its lift to $\hat{P}_{0}^{2}$ has index sets $\left(0,-1,-1,0\right)$.
	It follows that $\left\langle \eta\right\rangle ^{\mu}\left(\log\left\langle \eta\right\rangle \right)^{j}$
	has index set generated by $\left(-\mu,j\right)$ at $\ff$ and $\iif_{\eta}$,
	and $0$ at $\of$ and $\iif_{x}$. Since the entries of $\mathfrak{s}'$
	are smooth in $\overline{\mathbb{R}^{n}}$, it follows that the entries
	of $\left\langle \eta\right\rangle ^{\mathfrak{s}\left(y\right)}$
	are in $C^{\infty}\left(\overline{\mathbb{R}^{n}}\right)\hat{\otimes}\mathcal{A}_{\phg}^{\left(0,\tilde{\mathcal{F}},\tilde{\mathcal{F}},0\right)}\left(\hat{P}_{0}^{2}\right)$
	where $\tilde{\mathcal{F}}$ is the index set generated by the pairs
	$\left(-\mu,j\right)$ for $j<M$. Since the entries of $b\left(y;x,\eta\right)$
	are in $S_{0\po,\mathcal{S}}^{-\infty,\mathcal{E}}\left(\mathbb{R}^{n};\mathbb{R}_{1}^{n+1}\right)=\mathcal{S}\left(\mathbb{R}^{n}\right)\hat{\otimes}\mathcal{A}_{\phg}^{\left(\mathcal{E}_{\of},\mathcal{E}_{\ff},\infty,\infty\right)}\left(\hat{P}_{0}^{2}\right)$,
	it follows that the entries of $\tilde{b}\left(y;x,\eta\right)$ are
	in $\mathcal{S}\left(\mathbb{R}^{n}\right)\hat{\otimes}\mathcal{A}_{\phg}^{\left(\mathcal{E}_{\of},\mathcal{E}_{\ff}+\tilde{\mathcal{F}},\infty,\infty\right)}\left(\hat{P}_{0}^{2}\right)$.
	The index set $\mathcal{F}$ defined in the statement of the lemma
	is exactly $\mathcal{E}_{\ff}+\tilde{\mathcal{F}}$, so the proof
	is concluded.
\end{proof}
We now define the local version of the twisted $0$-trace and $0$-Poisson
classes:
\begin{defn}
	(Local twisted symbolic $0$-trace operators) Let $\mathfrak{s}:\overline{\mathbb{R}^{n}}\to\mathfrak{gl}\left(M,\mathbb{C}\right)$
	be a smooth family of matrices with a single constant eigenvalue,
	and let $\mathcal{E}=\left(\mathcal{E}_{\of},\mathcal{E}_{\ff}\right)$
	be a pair of index sets. We denote by $\hat{\Psi}_{0\tr,\mathcal{S}}^{-\infty,\mathcal{E},\mathfrak{s}}\left(\mathbb{R}_{1}^{n+1};\mathbb{R}^{n},\mathbb{C}^{M}\right)$
	the class of operators $\dot{C}^{\infty}\left(\overline{\mathbb{R}_{1}^{1}}\times\overline{\mathbb{R}^{n}}\right)\to\mathcal{S}'\left(\mathbb{R}^{n};\mathbb{C}^{M}\right)$
	of the form $\Op_{L}^{\tr}\left(a\right)$, where $a\in S_{0\tr,\mathcal{S}}^{-\infty,\mathcal{E},\mathfrak{s}}\left(\mathbb{R}^{n};\mathbb{R}_{1}^{n+1};\mathbb{C}^{M}\right)$.
\end{defn}

\begin{defn}
	(Local twisted symbolic $0$-Poisson operators) Let $\mathfrak{s}:\overline{\mathbb{R}^{n}}\to\mathfrak{gl}\left(M,\mathbb{C}\right)$
	be a smooth family of matrices with a single constant eigenvalue,
	and let $\mathcal{E}=\left(\mathcal{E}_{\of},\mathcal{E}_{\ff}\right)$
	be a pair of index sets. We denote by $\hat{\Psi}_{0\po,\mathcal{S}}^{-\infty,\mathcal{E},\mathfrak{s}}\left(\mathbb{R}^{n},\mathbb{C}^{M};\mathbb{R}_{1}^{n+1}\right)$
	the class of operators $\mathcal{S}\left(\mathbb{R}^{n};\mathbb{C}^{M}\right)\to C^{-\infty}\left(\overline{\mathbb{R}_{1}^{1}}\times\overline{\mathbb{R}^{n}}\right)$
	of the form $\Op_{L}^{\po}\left(b\right)$, where $b\in S_{0\po,\mathcal{S}}^{-\infty,\mathcal{E},\mathfrak{s}}\left(\mathbb{R}^{n};\mathbb{R}_{1}^{n+1};\left(\mathbb{C}^{M}\right)^{*}\right)$.
\end{defn}

Lemma \ref{lem:twisting-poisson-trace-are-polyhomogeneous-1} guarantees
that if \[A \in \hat{\Psi}_{0\tr,\mathcal{S}}^{-\infty,\mathcal{E},\mathfrak{s}}\left(\mathbb{R}_{1}^{n+1};\mathbb{R}^{n},\mathbb{C}^{M}\right)\]
and \[B \in \hat{\Psi}_{0\po,\mathcal{S}}^{-\infty,\mathcal{E},\mathfrak{s}}\left(\mathbb{R}^{n},\mathbb{C}^{M};\mathbb{R}_{1}^{n+1}\right),\]
then the Schwartz kernels of $A$ and $B$ are polyhomogeneous with
index set $\mathcal{E}_{\of}$ away from a neighborhood of $\partial\Delta$.
We can thus pass immediately to the definition on manifolds. Fix a
smooth vector bundle $\boldsymbol{E}\to\partial X$, equipped with
a smooth endomorphism $\boldsymbol{\mathfrak{s}}:\boldsymbol{E}\to\boldsymbol{E}$
with constant eigenvalues. Call $\boldsymbol{\mu}=\left(\mu_{1},...,\mu_{I}\right)$
the vector of eigenvalues of $\boldsymbol{\mathfrak{s}}$, and decompose
$\boldsymbol{E}=\bigoplus_{i}E_{i}$ and $\boldsymbol{\mathfrak{s}}=\bigoplus_{i}\mathfrak{s}_{i}$
into generalized eigenbundles. Call $M_{i}$ the rank of $E_{i}$.
\begin{defn}
	(Twisted symbolic $0$-trace operators) Let $\mathcal{E}=\left(\mathcal{E}_{\of},\mathcal{E}_{\ff}\right)$
	be a pair of index sets. We define $\hat{\Psi}_{0\tr}^{-\infty,\left(\mathcal{E}_{\of},\left[\mathcal{E}_{\ff}\right]\right),\boldsymbol{\mathfrak{s}}}\left(X;\partial X,\boldsymbol{E}\right)$
	as the space of operators $\boldsymbol{A}:\dot{C}^{\infty}\left(X\right)\to C^{-\infty}\left(\partial X;\boldsymbol{E}\right)$
	of the form 
	\[
	\boldsymbol{A}=\left(\begin{matrix}A^{1}\\
		\vdots\\
		A^{I}
	\end{matrix}\right),
	\]
	where $A^{i}:\dot{C}^{\infty}\left(X\right)\to C^{-\infty}\left(\partial X;E_{i}\right)$
	satisfies the following properties:
	\begin{enumerate}
		\item in the complement of any neighborhood of $\partial\Delta$, $K_{A^{i}}$
		coincides with a residual trace operator in $\Psi_{0\tr}^{-\infty,\mathcal{E}_{\of}}\left(X;\partial X,E_{i}\right)$;
		\item for every $p\in\partial X$, given coordinates $\left(x,y\right)$
		centered at $p$ such that $Vx\equiv1$ near $p$, and a local frame
		for $E_{i}$ defined near $p$, denote by $\mathfrak{s}_{i}\left(y\right)$
		the local expression of $\mathfrak{s}_{i}$ in the given coordinates
		and frame; then there exists an index set $\tilde{\mathcal{E}}_{\ff}$
		with $\left[\tilde{\mathcal{E}}_{\ff}\right]=\left[\mathcal{E}_{\ff}\right]$,
		and twisted $0$-trace symbols \[\left\langle \eta\right\rangle ^{-\mathfrak{s}_{i}\left(y\right)}a^{i}\left(y;\tilde{x},\eta\right)\in S_{0\tr,\mathcal{S}}^{-\infty,\left(\mathcal{E}_{\of},\tilde{\mathcal{E}}_{\ff}\right),\mathfrak{s}_{i}}\left(\mathbb{R}^{n};\mathbb{R}_{1}^{n+1};\mathbb{C}^{M_{i}}\right)\]
		such that in the chosen coordinates and frame we have
		\[
		K_{A^{i}}=\frac{1}{\left(2\pi\right)^{n}}\int e^{i\left(y-\tilde{y}\right)\eta}\left\langle \eta\right\rangle ^{-\mathfrak{s}_{i}\left(y\right)}a^{i}\left(y;\tilde{x},\eta\right)d\eta d\tilde{x}d\tilde{y}
		\]
		in a neighborhood of the origin.
	\end{enumerate}
\end{defn}

\begin{defn}
	(Twisted symbolic $0$-Poisson operators) Let $\mathcal{E}=\left(\mathcal{E}_{\of},\mathcal{E}_{\ff}\right)$
	be a pair of index sets. We define $\hat{\Psi}_{0\po}^{-\infty,\left(\mathcal{E}_{\of},\left[\mathcal{E}_{\ff}\right]\right),\boldsymbol{\mathfrak{s}}}\left(\partial X,\boldsymbol{E};X\right)$
	as the space of operators $\boldsymbol{B}:C^{\infty}\left(\partial X;\boldsymbol{E}\right)\to C^{-\infty}\left(X\right)$
	of the form 
	\[
	\boldsymbol{B}=\left(\begin{matrix}B_{1} & \cdots & B_{I}\end{matrix}\right),
	\]
	where $B_{i}:C^{\infty}\left(\partial X;E_{i}\right)\to C^{-\infty}\left(X\right)$
	satisfies the following properties:
	\begin{enumerate}
		\item in the complement of any neighborhood of $\partial\Delta$, $K_{B_{i}}$
		coincides with a residual Poisson operator in $\Psi_{0\po}^{-\infty,\mathcal{E}_{\of}}\left(\partial X,E_{i};X\right)$;
		\item for every $p\in\partial X$, given coordinates $\left(x,y\right)$
		centered at $p$ such that $Vx\equiv1$ near $p$, and a local frame
		for $E_{i}$ defined near $p$, denote by $\mathfrak{s}_{i}\left(y\right)$
		the local expression of $\mathfrak{s}_{i}$ in the given coordinates
		and frame; then there exists an index set $\tilde{\mathcal{E}}_{\ff}$
		with $\left[\tilde{\mathcal{E}}_{\ff}\right]=\left[\mathcal{E}_{\ff}\right]$,
		and twisted $0$-Poisson symbols $b_{i}\left(y;x,\eta\right)\left\langle \eta\right\rangle ^{\mathfrak{s}_{i}\left(y\right)}\in S_{0\po,\mathcal{S}}^{-\infty,\left(\mathcal{E}_{\of},\tilde{\mathcal{E}}_{\ff}\right),\mathfrak{s}_{i}}\left(\mathbb{R}^{n};\mathbb{R}_{1}^{n+1};\left(\mathbb{C}^{M_{i}}\right)^{*}\right)$
		such that in the chosen coordinates and frame we have
		\[
		K_{B_{i}}=\frac{1}{\left(2\pi\right)^{n}}\int e^{i\left(y-\tilde{y}\right)\eta}b\left(y;x,\eta\right)\left\langle \eta\right\rangle ^{\mathfrak{s}_{i}\left(y\right)}d\eta d\tilde{y}
		\]
		in a neighborhood of the origin.
	\end{enumerate}
\end{defn}

\begin{rem}
	\label{rem:why-only-leading-sets-1}We remark that in the definitions
	of the spaces $\hat{\Psi}_{0\tr}^{-\infty,\left(\mathcal{E}_{\of},\left[\mathcal{E}_{\ff}\right]\right),\boldsymbol{\mathfrak{s}}}\left(X;\partial X,\boldsymbol{E}\right)$
	and $\hat{\Psi}_{0\po}^{-\infty,\left(\mathcal{E}_{\of},\left[\mathcal{E}_{\ff}\right]\right),\boldsymbol{\mathfrak{s}}}\left(\partial X,\boldsymbol{E};X\right)$
	we are not keeping track of the precise index set at the front face,
	but only of its leading set. The reason for this choice is that, due
	to the presence of the twisting factor $\left\langle \eta\right\rangle ^{\mathfrak{s}\left(y\right)}$
	in the symbol, the full index set $\mathcal{E}_{\ff}$ is not invariant
	under changes of coordinates; however, the leading set $\left[\mathcal{E}_{\ff}\right]$
	is, and this is enough for our applications. It is not clear whether
	this phenomenon is just an artifact of the method used to prove coordinate
	invariance. We will discuss this more in detail in \S\ref{subsec:Technical-details-on-proofs}.
	
	By Lemma \ref{lem:twisting-poisson-trace-are-polyhomogeneous-1},
	we immediately obtain the following
\end{rem}

\begin{cor}
	\label{cor:twisted-0-trace-0-poisson-are-untwisted}Let $\boldsymbol{E}=\bigoplus_{i}E_{i}$,
	$\boldsymbol{\mathfrak{s}}=\bigoplus_{i}\mathfrak{s}_{i}$, $M_{1},...,M_{I}$
	and $\boldsymbol{\mu}=\left(\mu_{1},...,\mu_{I}\right)$ be as above.
	\begin{enumerate}
		\item Let $\boldsymbol{A}=\left(A^{1},...,A^{I}\right)^{T}\in\hat{\Psi}_{0\tr}^{-\infty,\left(\mathcal{E}_{\of},\left[\mathcal{E}_{\ff}\right]\right),\boldsymbol{\mathfrak{s}}}\left(X;\partial X,\boldsymbol{E}\right)$.
		Then $A^{i}\in\hat{\Psi}_{0\tr}^{-\infty,\left(\mathcal{E}_{\of},\tilde{\mathcal{E}}_{\ff}\right)}\left(X;\partial X,E_{i}\right)$,
		where $\tilde{\mathcal{E}}_{\ff}$ is an index set such that $\left[\tilde{\mathcal{E}}_{\ff}\right]$
		consists of the pairs $\left(\alpha+\mu_{i},k+l\right)$ with $\left(\alpha,k\right)\in\left[\mathcal{E}_{\ff}\right]$
		and $0\leq l<M_{i}$.
		\item Let $\boldsymbol{B}=\left(B_{1},...,B_{I}\right)\in\hat{\Psi}_{0\po}^{-\infty,\left(\mathcal{E}_{\of},\left[\mathcal{E}_{\ff}\right]\right),\boldsymbol{\mathfrak{s}}}\left(\partial X,\boldsymbol{E};X\right)$.
		Then $B_{i}\in\hat{\Psi}_{0\po}^{-\infty,\left(\mathcal{E}_{\of},\tilde{\mathcal{E}}_{\ff}\right)}\left(\partial X,E_{i};X\right)$,
		where $\tilde{\mathcal{E}}_{\ff}$ is an index set such that $\left[\tilde{\mathcal{E}}_{\ff}\right]$
		consists of the pairs $\left(\alpha-\mu_{i},k+l\right)$ with $\left(\alpha,k\right)\in\left[\mathcal{E}_{\ff}\right]$
		and $0\leq l<M_{i}$.
	\end{enumerate}
\end{cor}

This lemma ensures that twisted symbolic $0$-Poisson and $0$-trace
operators can be seen as vectors of un-twisted symbolic $0$-Poisson
and $0$-trace operators, at the expense of losing some information
at the front face. This information is however crucial in order to
correctly compute the Bessel family for these operators.

\subsubsection{\label{subsubsec:The-Bessel-family-of-twisted-poisson}The Bessel
	family of a twisted symbolic $0$-Poisson operator}

Let \[\boldsymbol{B}\in\hat{\Psi}_{0\po}^{-\infty,\left(\mathcal{E}_{\of},\left[\mathcal{E}_{\ff}\right]\right),\boldsymbol{\mathfrak{s}}}\left(\partial X,\boldsymbol{E};X\right)\]
as above. Assume that $\left[\mathcal{E}_{\ff}\right]=m$ for some
$m\in\mathbb{C}$. We will now construct the Bessel family of $\boldsymbol{B}$.
As in the un-twisted case, the Bessel family $\hat{N}\left(\boldsymbol{B}\right)$
is a smooth section of the Fréchet bundle $\Psi_{0\po}^{-\infty,\mathcal{E}_{\of}}\left(N^{+}\partial X;\boldsymbol{E}^{*}\right)\to T^{*}\partial X$.
However, unlike the un-twisted case, $\hat{N}\left(\boldsymbol{B}\right)$
is \emph{not }fibrewise homogeneous: rather, it is fibrewise \emph{twisted}
homogeneous of degree $-m+\boldsymbol{\mathfrak{s}}$ in the following
sense: for every $\eta\in T^{*}\partial X\backslash0$ and for every
$t\in\mathbb{R}^{+}$, we have
\[
\hat{N}_{t\eta}\left(\boldsymbol{B}\right)=\lambda_{t}^{*}\circ\hat{N}_{\eta}\left(\boldsymbol{B}\right)t^{-m+\boldsymbol{\mathfrak{s}}}.
\]
The construction proceeds almost exactly as in \S\ref{subsubsec:Principal-symbols-of-psidos}.
Choose a point $p\in\partial X$, and let $\left(x,y\right)$ be coordinates
for $X$ centered at $p$ and compatible with $V$, so that $Vx\equiv1$.
Choose also frames of $E_{i}$ near $p$, for every $i$. Then we
can find an index set $\tilde{\mathcal{E}}_{\ff}$ with $\left[\tilde{\mathcal{E}}_{\ff}\right]=\left[\mathcal{E}_{\ff}\right]$,
and a $0$-Poisson symbol $b_{i}\left(y;x,\eta\right)\in S_{0\po,\mathcal{S}}^{-\infty,\left(\mathcal{E}_{\of},\tilde{\mathcal{E}}_{\ff}\right)}\left(\mathbb{R}^{n};\mathbb{R}_{1}^{n+1},\left(\mathbb{C}^{M_{i}}\right)^{*}\right)$
such that, in the chosen coordinates and frames, we have
\[
B_{i}\equiv\frac{1}{\left(2\pi\right)^{n}}\int e^{i\left(y-\tilde{y}\right)\eta}b_{i}\left(y;x,\eta\right)\left\langle \eta\right\rangle ^{\mathfrak{s}_{i}\left(y\right)}d\eta d\tilde{y}
\]
near the origin. Given $\eta\in T_{p}^{*}\partial X\backslash0\equiv\mathbb{R}^{n}\backslash0$,
we define
\[
\hat{N}_{\eta}\left(\boldsymbol{B}\right)=\lim_{t\to0^{+}}b_{i}\left(0;tx,t^{-1}\eta\right)\left\langle t^{-1}\eta\right\rangle ^{\mathfrak{s}_{i}\left(0\right)}t^{-m+\mathfrak{s}_{i}\left(0\right)}.
\]
To check that this limit is indeed well-defined, similarly to \S\ref{subsubsec:Principal-symbols-of-psidos}
write in the locus $\left|\eta\right|\not=0$
\begin{align*}
	b_{i}\left(0;x,\eta\right) & =b_{i,0}\left(x\left|\eta\right|,\frac{\eta}{\left|\eta\right|},\left|\eta\right|^{-1}\right)\left|\eta\right|^{-m}\\
	& +r_{i}\left(x\left|\eta\right|,\frac{\eta}{\left|\eta\right|},\left|\eta\right|^{-1}\right)\left|\eta\right|^{-m-1},
\end{align*}
where $b_{i,0}\left(\tau,\hat{\eta},\rho\right)$ and $r_{i}\left(\tau,\hat{\eta},\rho\right)$
are in $\mathcal{A}_{\phg}^{\left(\mathcal{E}_{\of},\infty,\tilde{\mathcal{E}}_{\ff}-m\right)}\left(\overline{\mathbb{R}_{1}^{1}}\times S^{n-1}\times\mathbb{R}_{1}^{1};\left(\mathbb{C}^{M_{i}}\right)^{*}\right)$.
As in the un-twisted case, from $\left[\tilde{\mathcal{E}}_{\ff}\right]=m$
we have
\begin{align*}
	\lim_{t\to0^{+}}b_{i,0}\left(x\left|\eta\right|,\frac{\eta}{\left|\eta\right|},t\left|\eta\right|^{-1}\right)\left|\eta\right|^{-m} & =\tilde{b}_{i,0}\left(x\left|\eta\right|,\frac{\eta}{\left|\eta\right|}\right)\left|\eta\right|^{-m}\\
	\lim_{t\to0^{+}}r_{i}\left(x\left|\eta\right|,\frac{\eta}{\left|\eta\right|},t\left|\eta\right|^{-1}\right)\left|\eta\right|^{-m-1}t & =0
\end{align*}
where $\tilde{b}_{i,0}\left(\tau,\hat{\eta}\right)\in\mathcal{A}_{\phg}^{\left(\mathcal{E}_{\of},\infty\right)}\left(\overline{\mathbb{R}_{1}^{1}}\times S^{n-1};\left(\mathbb{C}^{M_{i}}\right)^{*}\right)$,
Finally, we have
\begin{align*}
	\lim_{t\to0^{+}}\left\langle t^{-1}\eta\right\rangle ^{\mathfrak{s}_{i}\left(0\right)}t^{\mathfrak{s}_{i}\left(0\right)} & =\lim_{t\to0^{+}}\left(\left\langle t^{-1}\eta\right\rangle t\right)^{\mathfrak{s}_{i}\left(0\right)}\\
	& =\left|\eta\right|^{\mathfrak{s}_{i}\left(0\right)}.
\end{align*}
Calling 
\[
\tilde{\boldsymbol{b}}_{0}=\left(\begin{matrix}\tilde{b}_{1,0} & \cdots & \tilde{b}_{I,0}\end{matrix}\right)
\]
we then have
\[
\hat{N}_{\eta}\left(\boldsymbol{B}\right)=\tilde{\boldsymbol{b}}_{0}\left(x\left|\eta\right|,\frac{\eta}{\left|\eta\right|}\right)\left|\eta\right|^{-m+\boldsymbol{\mathfrak{s}}\left(0\right)}.
\]
From this formula, it follows that $\hat{N}\left(\boldsymbol{B}\right)$
is twisted homogeneous in the sense above.

\subsubsection{\label{subsubsec:The-Bessel-family-of-twisted-0-trace}The Bessel
	family of a twisted symbolic $0$-trace operator}

The $0$-trace case is essentially equal to the $0$-Poisson case.
Let $\boldsymbol{A}\in\hat{\Psi}_{0\tr}^{-\infty,\left(\mathcal{E}_{\of},\left[\mathcal{E}_{\ff}\right]\right),\boldsymbol{\mathfrak{s}}}\left(X;\partial X,\boldsymbol{E}\right)$
as above. Assume that $\left[\mathcal{E}_{\ff}\right]=m$ for some
$m\in\mathbb{C}$. The Bessel family of $\boldsymbol{A}$ is a smooth
section of the Fréchet bundle $\Psi_{0\tr}^{-\infty,\mathcal{E}_{\of}}\left(N^{+}\partial X;\boldsymbol{E}\right)\to T^{*}\partial X$,
fibrewise \emph{twisted }homogeneous of degree $-m-\boldsymbol{\mathfrak{s}}$
in the following sense: for every $\eta\in T^{*}\partial X\backslash0$
and for every $t\in\mathbb{R}^{+}$, we have
\[
\hat{N}_{t\eta}\left(\boldsymbol{A}\right)=t^{-m-\boldsymbol{\mathfrak{s}}}\circ\hat{N}_{t\eta}\left(\boldsymbol{A}\right)\circ\lambda_{t^{-1}}^{*}.
\]

Choose a point $p\in\partial X$, and let $\left(x,y\right)$ be coordinates
for $X$ centered at $p$ and compatible with $V$, so that $Vx\equiv1$.
Choose also frames of $E_{i}$ near $p$, for every $i$. Then we
can find an index set $\tilde{\mathcal{E}}_{\ff}$ with $\left[\tilde{\mathcal{E}}_{\ff}\right]=\left[\mathcal{E}_{\ff}\right]$,
and a $0$-trace symbol $a^{i}\left(y;\tilde{x},\eta\right)\in S_{0\tr,\mathcal{S}}^{-\infty,\left(\mathcal{E}_{\of},\tilde{\mathcal{E}}_{\ff}\right)}\left(\mathbb{R}^{n};\mathbb{R}_{1}^{n+1},\mathbb{C}^{M_{i}}\right)$
such that, in the chosen coordinates and frames, we have
\[
A^{i}\equiv\frac{1}{\left(2\pi\right)^{n}}\int e^{i\left(y-\tilde{y}\right)\eta}\left\langle \eta\right\rangle ^{-\mathfrak{s}_{i}\left(y\right)}a^{i}\left(y;\tilde{x},\eta\right)d\eta d\tilde{x}d\tilde{y}
\]
near the origin. Given $\eta\in T_{p}^{*}\partial X\backslash0\equiv\mathbb{R}^{n}\backslash0$,
we define
\[
\hat{N}_{\eta}\left(\boldsymbol{A}\right)=\lim_{t\to0^{+}}t^{-m-\mathfrak{s}_{i}\left(0\right)}\left\langle t^{-1}\eta\right\rangle ^{-\mathfrak{s}_{i}\left(0\right)}a^{i}\left(0;t\tilde{x},t^{-1}\eta\right)td\tilde{x}.
\]
To check that this limit is well-defined, write in the locus $\left|\eta\right|\not=0$
\begin{align*}
	a^{i}\left(0;\tilde{x},\eta\right) & =\left|\eta\right|^{-m+1}a_{0}^{i}\left(\tilde{x}\left|\eta\right|,\frac{\eta}{\left|\eta\right|},\left|\eta\right|^{-1}\right)\\
	& +\left|\eta\right|^{-m}r^{i}\left(\tilde{x}\left|\eta\right|,\frac{\eta}{\left|\eta\right|},\left|\eta\right|^{-1}\right),
\end{align*}
where $a_{0}^{i}\left(\tau,\hat{\eta},\rho\right)$ and $r_{i}\left(\tau,\hat{\eta},\rho\right)$
are in $\mathcal{A}_{\phg}^{\left(\mathcal{E}_{\of},\infty,\tilde{\mathcal{E}}_{\ff}-m\right)}\left(\overline{\mathbb{R}_{1}^{1}}\times S^{n-1}\times\mathbb{R}_{1}^{1};\mathbb{C}^{M_{i}}\right)$.
Taking the limit as $t\to0^{+}$, we get
\[
\hat{N}_{\eta}\left(A^{i}\right)=\left|\eta\right|^{-m-\mathfrak{s}_{i}\left(0\right)}\tilde{a}_{0}^{i}\left(\tilde{x}\left|\eta\right|,\frac{\eta}{\left|\eta\right|}\right)\left|\eta\right|d\tilde{x}
\]
where $\tilde{a}_{0}^{i}\left(\tau,\hat{\eta}\right)\in\mathcal{A}_{\phg}^{\left(\mathcal{E}_{\of},\infty\right)}\left(\overline{\mathbb{R}_{1}^{1}}\times S^{n-1};\mathbb{C}^{M_{i}}\right)$.
Calling
\[
\tilde{\boldsymbol{a}}_{0}=\left(\begin{matrix}\tilde{a}_{0}^{1}\\
	\vdots\\
	\tilde{a}_{0}^{I}
\end{matrix}\right),
\]
we have
\[
\hat{N}_{\eta}\left(\boldsymbol{A}\right)=\left|\eta\right|^{-m-\boldsymbol{\mathfrak{s}}\left(0\right)}\tilde{\boldsymbol{a}}_{0}\left(\tilde{x}\left|\eta\right|,\frac{\eta}{\left|\eta\right|}\right)\left|\eta\right|d\tilde{x}.
\]
From this formula, it follows that $\hat{N}\left(\boldsymbol{A}\right)$
is twisted homogeneous in the sense above.

\subsection{\label{subsec:The-twisted-boundary-calculus}The twisted boundary
	calculus}

To conclude the section, we will introduce the calculus of pseudodifferential
operators on the boundary which we will use to formulate boundary
value problems for $0$-elliptic operators. This calculus fits perfectly
the classes $\hat{\Psi}_{0\tr}^{-\infty,\bullet,\boldsymbol{\mathfrak{s}}}\left(X;\partial X,\boldsymbol{E}\right)$
and $\hat{\Psi}_{0\po}^{-\infty,\bullet,\boldsymbol{\mathfrak{s}}}\left(\partial X,\boldsymbol{E};X\right)$
described above. We remark that this calculus is not really new. Indeed,
it is a particularly simple subcalculus of the calculus of ``elliptic
systems with variable orders'' introduced by Krainer and Mendoza
in \cite{KrainerMendozaVariableOrders}. In fact, this calculus was
developed exactly to formulate elliptic boundary conditions for \emph{edge
}elliptic operators, with possibly variable indicial roots. Our simplifications
stem from the fact that we restrict ourselves to the constant indicial
root case, and this ultimately allows us to work with polyhomogeneous
symbols.

\subsubsection{Definitions}

Operators in the twisted boundary calculus act between sections of
vector bundles equipped with smooth endomorphisms with constant eigenvalues.
Before getting into that, let's establish the local theory. For reference,
given an index set $\mathcal{E}$, denote by $S_{\phg,\mathcal{S}}^{-\mathcal{E}}\left(\mathbb{R}^{n};\mathbb{R}^{n}\right)$
the class of polyhomogeneous symbols with Schwartz coefficients, i.e.\ $S_{\phg,\mathcal{S}}^{-\mathcal{E}}\left(\mathbb{R}^{n};\mathbb{R}^{n}\right)=\mathcal{S}\left(\mathbb{R}^{n}\right)\hat{\otimes}\mathcal{A}_{\phg}^{\mathcal{E}}\left(\overline{\mathbb{R}^{n}}\right)$.
Given $m\in\mathbb{R}$, denote by $S_{\mathcal{S}}^{-m}\left(\mathbb{R}^{n};\mathbb{R}^{n}\right)$
the class of standard symbols with Schwartz coefficients $\mathcal{S}\left(\mathbb{R}^{n}\right)\hat{\otimes}\mathcal{A}^{m}\left(\overline{\mathbb{R}^{n}}\right)$.
Now, fix $\mathfrak{a}:\overline{\mathbb{R}^{n}}\to\mathfrak{gl}\left(\mathbb{C}^{N}\right)$
and $\mathfrak{b}:\overline{\mathbb{R}^{n}}\to\mathfrak{gl}\left(\mathbb{C}^{M}\right)$
smooth families of matrices with a constant, single eigenvalue. Call
$\mu$ (resp.\ $\lambda$) the eigenvalue of $\mathfrak{a}$ (resp.\ $\mathfrak{b}$).
\begin{defn}
	(Twisted boundary symbols) We define $S_{\phg,\mathcal{S}}^{-\mathcal{E},\left(\mathfrak{a},\mathfrak{b}\right)}\left(\mathbb{R}^{n};\mathbb{R}^{n};\hom\left(\mathbb{C}^{N},\mathbb{C}^{M}\right)\right)$
	as the space
	\[
	\left\langle \eta\right\rangle ^{-\mathfrak{b}\left(y\right)}S_{\phg,\mathcal{S}}^{-\mathcal{E}}\left(\mathbb{R}^{n};\mathbb{R}^{n};\hom\left(\mathbb{C}^{N},\mathbb{C}^{M}\right)\right)\left\langle \eta\right\rangle ^{\mathfrak{a}\left(y\right)}.
	\]
	We call the symbols in these classes \emph{twisted boundary} \emph{symbols}.
\end{defn}

The following lemma is proved exactly as Lemma \ref{lem:twisting-poisson-trace-are-polyhomogeneous-1}:
\begin{lem}
	\label{lem:twisted-boundary-symbols-are-phg-1}Let $q\left(y;\eta\right)\in S_{\phg,\mathcal{S}}^{-\mathcal{E},\left(\mathfrak{a},\mathfrak{b}\right)}\left(\mathbb{R}^{n};\mathbb{R}^{n};\hom\left(\mathbb{C}^{N},\mathbb{C}^{M}\right)\right)$,
	and think of $q\left(y;\eta\right)$ as an $M\times N$ matrix of
	scalar symbols. Then its entries are elements of $S_{\phg,\mathcal{S}}^{-\tilde{\mathcal{E}}}\left(\mathbb{R}^{n};\mathbb{R}^{n}\right)$,
	where $\tilde{\mathcal{E}}$ is the index set consisting of pairs
	of the form $\left(\alpha-\mu+\lambda,l+l'+l''\right)$, where $\left(\alpha,l\right)\in\mathcal{E}$
	and $0\leq l'<N$ and $0\leq l''<M$. Consequently, these entries
	are in $S_{\mathcal{S}}^{-m}\left(\mathbb{R}^{n};\mathbb{R}^{n}\right)$
	for every $m<\Re\left(\mathcal{E}-\mu+\lambda\right)$.
\end{lem}

Observe that the lemma above implies that the residual class \[S_{\phg,\mathcal{S}}^{-\infty,\left(\mathfrak{a},\mathfrak{b}\right)}\left(\mathbb{R}^{n};\mathbb{R}^{n};\hom\left(\mathbb{C}^{N},\mathbb{C}^{M}\right)\right)\]
does not depend on $\mathfrak{a}$ and $\mathfrak{b}$, and coincides
with the class $S_{\mathcal{S}}^{-\infty}\left(\mathbb{R}^{n};\mathbb{R}^{n};\hom\left(\mathbb{C}^{N},\mathbb{C}^{M}\right)\right)$
of $\hom\left(\mathbb{C}^{N},\mathbb{C}^{M}\right)$-valued smoothing
symbols with Schwartz coefficients.

Since the entries of $S_{\phg,\mathcal{S}}^{-\mathcal{E},\left(\mathfrak{a},\mathfrak{b}\right)}\left(\mathbb{R}^{n};\mathbb{R}^{n};\hom\left(\mathbb{C}^{N},\mathbb{C}^{M}\right)\right)$
are standard symbols, we can apply the left quantization map:
\begin{defn}
	(Local twisted boundary symbols) We define $\Psi_{\phg,\mathcal{S}}^{-\mathcal{E},\left(\mathfrak{a},\mathfrak{b}\right)}\left(\mathbb{R}^{n};\mathbb{C}^{N},\mathbb{C}^{M}\right)$
	as the class of operators $\Op_{L}\left(q\right)$ with $q\in S_{\phg,\mathcal{S}}^{-\tilde{\mathcal{E}},\left(\mathfrak{a},\mathfrak{b}\right)}\left(\mathbb{R}^{n};\mathbb{R}^{n};\hom\left(\mathbb{C}^{N},\mathbb{C}^{M}\right)\right)$,
	for some index set $\tilde{\mathcal{E}}$ with $\left[\tilde{\mathcal{E}}\right]=\left[\mathcal{E}\right]$.
\end{defn}

From Lemma \ref{lem:twisted-boundary-symbols-are-phg-1}, we get immediately
the following
\begin{cor}
	\label{cor:local-twisted-psidos-have-phg-entries}Let $Q\in\Psi_{\phg,\mathcal{S}}^{-\mathcal{E},\left(\mathfrak{a},\mathfrak{b}\right)}\left(\mathbb{R}^{n};\mathbb{C}^{N},\mathbb{C}^{M}\right)$.
	Then $Q\in\Psi_{\phg,\mathcal{S}}^{-\mathcal{F}}\left(\mathbb{R}^{n};\mathbb{C}^{N},\mathbb{C}^{M}\right)$,
	where $\mathcal{F}$ is an index set whose leading set consists of
	pairs of the form $\left(\alpha-\mu+\lambda,l+l'+l''\right)$, where
	$\left(\alpha,l\right)\in\left[\mathcal{E}\right]$, $0\leq l'<N$
	and $0\leq l''<M$.
\end{cor}

In particular, the Schwartz kernel of an element of $\Psi_{\phg,\mathcal{S}}^{-\mathcal{E},\left(\mathfrak{a},\mathfrak{b}\right)}\left(\mathbb{R}^{n};\mathbb{C}^{N},\mathbb{C}^{M}\right)$
is smooth in the complement of any neighborhood of the diagonal. This
allows us to define the corresponding calculus on the manifold $\partial X$.
Fix two smooth vector bundles $\boldsymbol{E},\boldsymbol{F}\to\partial X$,
equipped with endomorphisms $\boldsymbol{\mathfrak{s}},\boldsymbol{\mathfrak{t}}$
with constant eigenvalues. Call $\boldsymbol{\mu}=\left(\mu_{1},...,\mu_{I}\right)$
and $\boldsymbol{\lambda}=\left(\lambda_{1},...,\lambda_{J}\right)$
the eigenvalues of $\boldsymbol{\mathfrak{s}},\boldsymbol{\mathfrak{t}}$,
and decompose $\boldsymbol{E}=\bigoplus_{i}E_{i}$, $\boldsymbol{F}=\bigoplus_{j}F_{j}$
and $\boldsymbol{\mathfrak{s}}=\bigoplus_{i}\mathfrak{s}_{i}$, $\boldsymbol{\mathfrak{t}}=\bigoplus_{j}\mathfrak{t}_{j}$
in terms of the generalized eigenbundle decompositions of $\boldsymbol{\mathfrak{s}},\boldsymbol{\mathfrak{t}}$.
Call $N_{i},M_{j}$ the ranks of $E_{i},F_{j}$.
\begin{defn}
	(Twisted boundary operators) Let $\mathcal{E}$ be an index set. We
	denote by \[\Psi_{\phg}^{-\left[\mathcal{E}\right],\left(\boldsymbol{\mathfrak{s}},\boldsymbol{\mathfrak{t}}\right)}\left(\partial X;\boldsymbol{E},\boldsymbol{F}\right)\]
	the class of operators $\boldsymbol{Q}:C^{\infty}\left(\partial X;\boldsymbol{E}\right)\to C^{-\infty}\left(\partial X;\boldsymbol{F}\right)$
	of the form
	\[
	\boldsymbol{Q}=\left(\begin{matrix}Q_{1}^{1} & \cdots & Q_{I}^{1}\\
		\vdots & \ddots & \vdots\\
		Q_{1}^{J} & \cdots & Q_{I}^{J}
	\end{matrix}\right),
	\]
	where $Q_{i}^{j}:C^{\infty}\left(\partial X;E_{i}\right)\to C^{-\infty}\left(\partial X;F_{j}\right)$
	satisfies the following properties:
	\begin{enumerate}
		\item in the complement of any neighborhood of $\partial\Delta$, $K_{Q_{i}^{j}}$
		is smooth;
		\item for every $p\in\partial X$, in coordinates $y$ for $\partial X$
		centered at $p$, and with respect to local frames for $E_{i}$ and
		$F_{j}$ defined near $p$, there exists an index set $\tilde{\mathcal{E}}$
		with $\left[\tilde{\mathcal{E}}\right]=\left[\mathcal{E}\right]$
		and a symbol $\left\langle \eta\right\rangle ^{-\mathfrak{t}_{j}\left(y\right)}q\left(y;\eta\right)\left\langle \eta\right\rangle ^{\mathfrak{s}_{i}\left(y\right)}$
		in $S_{\phg,\mathcal{S}}^{-\mathcal{E},\left(\mathfrak{s}_{i},\mathfrak{t}_{j}\right)}\left(\mathbb{R}^{n};\mathbb{R}^{n};\hom\left(\mathbb{C}^{N_{i}},\mathbb{C}^{M_{j}}\right)\right)$
		such that, near the origin, we have
		\[
		K_{Q_{i}^{j}}=\frac{1}{\left(2\pi\right)^{n}}\int e^{i\left(y-\tilde{y}\right)\eta}\left\langle \eta\right\rangle ^{-\mathfrak{t}_{j}\left(y\right)}q\left(y;\eta\right)\left\langle \eta\right\rangle ^{\mathfrak{s}_{i}\left(y\right)}d\eta d\tilde{y}.
		\]
	\end{enumerate}
\end{defn}

\begin{rem}
	The reader should think of the calculus $\Psi_{\phg}^{-\left[\mathcal{E}\right],\left(\boldsymbol{\mathfrak{s}},\boldsymbol{\mathfrak{t}}\right)}\left(\partial X;\boldsymbol{E},\boldsymbol{F}\right)$
	as a mild generalization of the usual class of classical pseudodifferential
	operators of Douglis--Nirenberg type. Indeed, if $\boldsymbol{Q}\in\Psi_{\phg}^{-\left[\mathcal{E}\right],\left(\boldsymbol{\mathfrak{s}},\boldsymbol{\mathfrak{t}}\right)}\left(\partial X;\boldsymbol{E},\boldsymbol{F}\right)$
	and $\left[\mathcal{E}\right]=m$, then by Corollary \ref{cor:local-twisted-psidos-have-phg-entries}
	the entry $Q_{i}^{j}\in\Psi_{\phg}^{-\left[\mathcal{E}\right],\left(\mathfrak{s}_{i},\mathfrak{t}_{j}\right)}\left(\partial X;E_{i},F_{j}\right)$
	is an element of $\Psi_{\phg}^{-\mathcal{F}}\left(\partial X;E_{i},F_{j}\right)$,
	where $\mathcal{F}$ is an index set with $\left[\mathcal{F}\right]$
	equal to the family of pairs $\left(m-\mu_{i}+\lambda_{j},k+l\right)$
	with $k<N_{i}$ and $l<M_{j}$. In particular, if $E_{i}$ and $F_{j}$
	are line bundles, then $Q_{i}^{j}$ is an element of $\Psi_{\phg}^{-m+\mu_{i}-\lambda_{j}}\left(\partial X\right)$.
\end{rem}

The previous corollary then implies the following
\begin{cor}
	\label{cor:twisted-psidos-have-phg-entries}Let $\boldsymbol{Q}\in\Psi_{\phg}^{-\left[\mathcal{E}\right],\left(\boldsymbol{\mathfrak{s}},\boldsymbol{\mathfrak{t}}\right)}\left(\partial X;\boldsymbol{E},\boldsymbol{F}\right)$.
	Then $Q_{i}^{j}\in\Psi_{\phg}^{-\mathcal{F}}\left(\partial X;E_{i},F_{j}\right)$,
	where $\mathcal{F}$ is an index set whose leading set consists of
	pairs of the form $\left(\alpha-\mu_{i}+\lambda_{j},l+l'+l''\right)$,
	where $\left(\alpha,l\right)\in\left[\mathcal{E}\right]$, $0\leq l'<N_{i}$
	and $0\leq l''<M_{j}$.
\end{cor}

\subsubsection{Principal symbols and ellipticity}

Let's finally discuss the principal symbol map for these operators.
The idea is exactly the same as above. If $\boldsymbol{Q}\in\Psi_{\phg}^{-\left[\mathcal{E}\right],\left(\boldsymbol{\mathfrak{s}},\boldsymbol{\mathfrak{t}}\right)}\left(\partial X;\boldsymbol{E},\boldsymbol{F}\right)$
and $\left[\mathcal{E}\right]=m$, then $\sigma\left(\boldsymbol{Q}\right)$
is a smooth section of $\pi^{*}\left(\boldsymbol{E}^{*}\otimes\boldsymbol{F}\right)\to T^{*}\partial X\backslash0$,
fibrewise \emph{twisted }homogeneous of degree $-m+\left(\boldsymbol{\mathfrak{s}},\boldsymbol{\mathfrak{t}}\right)$
in the following sense: for every $\eta\in T^{*}\partial X\backslash0$
and $t\in\mathbb{R}^{+}$, we have
\[
\sigma_{t\eta}\left(\boldsymbol{Q}\right)=t^{-m}t^{-\boldsymbol{\mathfrak{t}}}\sigma_{\eta}\left(\boldsymbol{Q}\right)t^{\boldsymbol{\mathfrak{s}}}.
\]
To construct $\sigma\left(\boldsymbol{Q}\right)$ in the fiber $T_{p}^{*}\partial X\backslash0$,
choose coordinates $y$ for $\partial X$ centered at $p$ and local
frames for the bundles $E_{i},F_{j}$. Then there exists an index
set $\tilde{\mathcal{E}}$ with $\left[\tilde{\mathcal{E}}\right]=\left[\mathcal{E}\right]$,
and symbols $q_{i}^{j}\left(y;\eta\right)\in S_{\phg}^{-\tilde{\mathcal{E}}}\left(\mathbb{R}^{n};\mathbb{R}^{n};\hom\left(\mathbb{C}^{N_{i}},\mathbb{C}^{M_{j}}\right)\right)$
such that near the point $\left(p,p\right)$ in the given coordinates
and frames, we have
\[
Q_{i}^{j}=\frac{1}{\left(2\pi\right)^{n}}\int e^{i\left(y-\tilde{y}\right)\eta}\left\langle \eta\right\rangle ^{-\mathfrak{t}_{j}\left(y\right)}q_{i}^{j}\left(y;\eta\right)\left\langle \eta\right\rangle ^{\mathfrak{s}_{i}\left(y\right)}d\eta d\tilde{y}.
\]
Let $\boldsymbol{q}\left(y;\eta\right)$ be the block-matrix of symbols
$\boldsymbol{q}\left(y;\eta\right)=\left(q_{i}^{j}\left(y;\eta\right)\right)$.
Then, for $\eta\in T_{p}^{*}\partial X\backslash0\equiv\mathbb{R}^{n}\backslash0$,
we define
\[
\sigma_{\eta}\left(\boldsymbol{Q}\right)=\lim_{t\to0^{+}}t^{-\mathfrak{t}_{j}\left(0\right)}\left\langle t^{-1}\eta\right\rangle ^{-\mathfrak{t}_{j}\left(0\right)}t^{-m}\boldsymbol{q}\left(0;t^{-1}\eta\right)\left\langle t^{-1}\eta\right\rangle ^{\mathfrak{s}_{i}\left(0\right)}t^{\mathfrak{s}_{i}\left(0\right)}.
\]
To check that this limit exists, observe that since \[q_{i}^{j}\left(y;\eta\right) \in S_{\phg}^{-\tilde{\mathcal{E}}}\left(\mathbb{R}^{n};\mathbb{R}^{n};\hom\left(\mathbb{C}^{N_{i}},\mathbb{C}^{M_{j}}\right)\right)\]
and $\left[\tilde{\mathcal{E}}\right]=m$, we can write away from
$\left|\eta\right|=0$
\begin{align*}
	q_{i}^{j}\left(0;\eta\right) & =q_{i,0}^{j}\left(\frac{\eta}{\left|\eta\right|},\left|\eta\right|^{-1}\right)\left|\eta\right|^{-m}\\
	& +r_{i}^{j}\left(\frac{\eta}{\left|\eta\right|},\left|\eta\right|^{-1}\right)\left|\eta\right|^{-m-1}
\end{align*}
for some \[q_{i,0}^{j},r_{i}^{j}\in\mathcal{A}_{\phg}^{\tilde{\mathcal{E}}-m}\left(S^{n-1}\times\mathbb{R}_{1}^{1};\hom\left(\mathbb{C}^{N_{i}},\mathbb{C}^{M_{j}}\right)\right).\]
Calling $\tilde{q}_{i,0}^{j}\left(\hat{\eta}\right)=q_{i,0}^{j}\left(\hat{\eta},0\right)$,
and $\tilde{\boldsymbol{q}}_{0}=\left(\tilde{q}_{i,0}^{j}\right)$,
we have
\begin{align*}
	\sigma_{\eta}\left(\boldsymbol{Q}\right) & =\left|\eta\right|^{-m}\left|\eta\right|^{-\boldsymbol{\mathfrak{t}}\left(0\right)}\tilde{\boldsymbol{q}}_{0}\left(\frac{\eta}{\left|\eta\right|}\right)\left|\eta\right|^{\boldsymbol{\mathfrak{s}}\left(0\right)}.
\end{align*}
This shows that indeed $\sigma\left(\boldsymbol{Q}\right)$ is fibrewise
twisted homogeneous in the sense described above.
\begin{defn}
	An operator $\boldsymbol{Q}\in\Psi_{\phg}^{-\left[0\right],\left(\boldsymbol{\mathfrak{s}},\boldsymbol{\mathfrak{t}}\right)}\left(\partial X;\boldsymbol{E},\boldsymbol{F}\right)$
	is called \emph{elliptic} if the principal symbol $\sigma_{\eta}\left(\boldsymbol{Q}\right)$
	is invertible for every $\eta\in T^{*}\partial X\backslash0$.
\end{defn}

The following theorem can be proved exactly as in the standard case:
\begin{thm}
	Let $\boldsymbol{Q}\in\Psi_{\phg}^{-\left[0\right],\left(\boldsymbol{\mathfrak{s}},\boldsymbol{\mathfrak{t}}\right)}\left(\partial X;\boldsymbol{E},\boldsymbol{F}\right)$
	be elliptic. Then there exist operators $\boldsymbol{K}_{1},\boldsymbol{K}_{2}\in\Psi_{\phg}^{-\left[0\right],\left(\boldsymbol{\mathfrak{t}},\boldsymbol{\mathfrak{s}}\right)}\left(\partial X;\boldsymbol{F},\boldsymbol{E}\right)$
	such that
	\begin{align*}
		\boldsymbol{Q}\boldsymbol{K}_{1} & \equiv I\mod\Psi^{-\infty}\left(\partial X;\boldsymbol{F}\right)\\
		\boldsymbol{K}_{2}\boldsymbol{Q} & \equiv I\mod\Psi^{-\infty}\left(\partial X;\boldsymbol{E}\right).
	\end{align*}
\end{thm}

A much stronger theorem of this type is proved in \cite{KrainerMendozaVariableOrders}
for elliptic operators in their more general calculus. Anyway, the
proof is essentially equal to the usual proof in standard elliptic
theory and uses the composition Theorem \ref{thm:global-twisted-compositions-involving-boundary},
which we still have to prove. If $\boldsymbol{Q}\in\Psi_{\phg}^{-\left[0\right],\left(\boldsymbol{\mathfrak{s}},\boldsymbol{\mathfrak{t}}\right)}\left(\partial X;\boldsymbol{E},\boldsymbol{F}\right)$
is elliptic, then the inverse $\sigma\left(\boldsymbol{Q}\right)^{-1}$
of the principal symbol is a smooth section of $\pi^{*}\left(\boldsymbol{F}^{*}\otimes\boldsymbol{E}\right)\to T^{*}\partial X\backslash0$,
fibrewise twisted homogeneous of degree $\left(\boldsymbol{\mathfrak{t}},\boldsymbol{\mathfrak{s}}\right)$.
Given $\boldsymbol{K}_{0}\in\Psi_{\phg}^{-\left[0\right],\left(\boldsymbol{\mathfrak{t}},\boldsymbol{\mathfrak{s}}\right)}\left(\partial X;\boldsymbol{F},\boldsymbol{E}\right)$
with $\sigma\left(\boldsymbol{K}_{0}\right)=\sigma\left(\boldsymbol{Q}\right)^{-1}$,
then by Theorem \ref{thm:global-twisted-compositions-involving-boundary}
we have $\boldsymbol{K}_{0}\boldsymbol{Q}\in\Psi_{\phg}^{-\left[0\right],\left(\boldsymbol{\mathfrak{s}},\boldsymbol{\mathfrak{s}}\right)}\left(\partial X;\boldsymbol{E}\right)$
and $\boldsymbol{Q}\boldsymbol{K}_{0}\in\Psi_{\phg}^{-\left[0\right],\left(\boldsymbol{\mathfrak{t}},\boldsymbol{\mathfrak{t}}\right)}\left(\partial X;\boldsymbol{F}\right)$,
and moreover $\sigma\left(\boldsymbol{K}_{0}\boldsymbol{Q}\right)\equiv1$.
This means that the remainders $\boldsymbol{R}_{1}=I-\boldsymbol{K}_{0}\boldsymbol{Q}$
and $\boldsymbol{R}_{2}=I-\boldsymbol{Q}\boldsymbol{K}_{0}$ are respectively
in $\Psi_{\phg}^{-\left[\mathcal{H}\right],\left(\boldsymbol{\mathfrak{s}},\boldsymbol{\mathfrak{s}}\right)}\left(\partial X;\boldsymbol{E}\right)$
and $\Psi_{\phg}^{-\left[\mathcal{H}\right],\left(\boldsymbol{\mathfrak{t}},\boldsymbol{\mathfrak{t}}\right)}\left(\partial X;\boldsymbol{F}\right)$,
for a certain index set $\mathcal{H}$ with $\Re\left(\mathcal{H}\right)>0$.
We can now promote $\boldsymbol{K}_{0}$ to two left and right parametrices
$\boldsymbol{K}_{2},\boldsymbol{K}_{1}$ as usual, by composing $\boldsymbol{K}_{0}$
with asymptotic sums for $\left(I-\boldsymbol{R}_{2}\right)^{-1}$
and $\left(I-\boldsymbol{R}_{1}\right)^{-1}$.

\section{\label{sec:Adjoints-compositions-mapping}Adjoints, compositions,
	mapping properties}

In this section, we discuss various properties of the calculus developed
in \S\ref{sec:The-symbolic-0-calculus}, the most important of which
are the \emph{composition theorems} between the various parts of the
calculus.

\subsection{\label{subsec:The-pull-back-push-forward-technique}The pull-back
	/ push-forward technique}

Before starting, let us briefly review the technique used to prove
the composition theorems for the classes of operators described in
\S\ref{sec:Recap-on-the-0-calc}. This well-established technique
is due to Melrose, and is fundamental (albeit in a different form)
for our approach as well. To exemplify the argument, let's consider
two operators $A\in\Psi_{0\tr}^{-\infty,\mathcal{E}}\left(X,\partial X\right)$
and $P\in\Psi_{0}^{-\infty,\mathcal{F}}\left(X\right)$. We want to
understand whether the composition $AP$ is well-defined, and if it
is, on which space does it live.

In order to understand whether $AP$ is well-defined, we need to understand
mapping properties of $A$ and $P$ on polyhomogeneous functions.
Firstly, one proves (cf.\ \cite{MazzeoEdgeI}) that $P$ induces
a continuous linear map $P:\dot{C}^{\infty}\left(X\right)\to\mathcal{A}_{\phg}^{\mathcal{F}_{\lf}}\left(X\right)$.
The way to prove this is to observe that, if $u\in\dot{C}^{\infty}\left(X\right)$,
then
\begin{align*}
	Pu & =\left(\pi_{L}\right)_{*}\left(K_{P}\cdot\pi_{R}^{*}u\right)\\
	& =\left(\beta_{0,L}\right)_{*}\left(\kappa_{P}\cdot\beta_{0,R}^{*}u\right)
\end{align*}
where $\kappa_{P}=\beta_{0}^{*}K_{P}$. We can then apply Melrose's
pull-back Theorem to recognize $\beta_{0,R}^{*}u$ as a polyhomogeneous
function on $X_{0}^{2}$, and then Melrose's
push-forward Theorem to prove that the push-forward $\left(\beta_{0,L}\right)_{*}\left(\kappa_{P}\cdot\beta_{0,R}^{*}u\right)$
is well-defined and polyhomogeneous. Then, one proves (cf.\ Lemma
2.43 of \cite{UsulaPhD}) that if $\Re\left(\mathcal{E}_{\of}+\mathcal{F}_{\lf}\right)>-1$
and $\Re\left(\mathcal{E}_{\ff}+\mathcal{F}_{\lf}\right)>0$, then
$A$ extends from $\dot{C}^{\infty}\left(X\right)$ to a continuous
linear map $\mathcal{A}_{\phg}^{\mathcal{F}_{\lf}}\left(X\right)\to C^{\infty}\left(\partial X\right)$.
The technique is the same. If $u\in\dot{C}^{\infty}\left(X\right)$,
then $Au=\left(\beta_{\tr,L}\right)_{*}\left(\kappa_{A}\cdot\beta_{\tr,R}^{*}u\right)$
and this formula implies that $A$ is well-defined as a map $\dot{C}^{\infty}\left(X\right)\to C^{\infty}\left(\partial X\right)$;
the same argument also implies that if $u\in\mathcal{A}_{\phg}^{\mathcal{F}_{\lf}}\left(X\right)$
and the inequalities above are satisfied, then one can push-forward
(i.e.\ integrate along the fibers) $\kappa_{A}\cdot\beta_{\tr,R}^{*}u$
via $\beta_{\tr,L}$. The inequalities above are indeed integrability
conditions for this push-forward to be well-defined. The final result
is that \emph{if the integrability conditions $\Re\left(\mathcal{E}_{\of}+\mathcal{F}_{\lf}\right)>-1$}
\emph{and} $\Re\left(\mathcal{E}_{\ff}+\mathcal{F}_{\lf}\right)>0$\emph{
	are satisfied, then $AP$ is well-defined}.

Now that we know that $AP$ is well-defined, we need to understand
where $K_{AP}$ lives. In order to do this, one needs to work on the
triple space $\partial X\times X\times X$. Calling 
\begin{align*}
	\pi_{LC} & :\partial X\times X\times X\to\partial X\times X\\
	\pi_{CR} & :\partial X\times X\times X\to X\times X\\
	\pi_{LR} & :\partial X\times X\times X\to\partial X\times X
\end{align*}
the various projections, where $L,C,R$ stand for ``left'', ``center''
and ``right'', one has
\[
K_{AP}=\left(\pi_{LR}\right)_{*}\left(\pi_{LC}^{*}K_{A}\cdot\pi_{CR}^{*}K_{P}\right)
\]
where the push-forward is well-defined thanks to the integrability
conditions above. To desingularize $K_{AP}$, we have to lift $K_{AP}$
and the maps $\pi_{LC},\pi_{CR},\pi_{LR}$ to appropriate blow-ups
of the spaces involved. Calling
\begin{align*}
	D_{LC} & =\partial\Delta\times X\\
	D_{CR} & =\partial X\times\partial\Delta\\
	D_{LR} & =\left\{ \left(p,q,p\right):p\in\partial X,q\in X\right\} \\
	S & =D_{LC}\cap D_{CR}\cap D_{LR},
\end{align*}
one forms the blown-up triple space
\[
T=\left[\partial X\times X^{2};S;D_{LC}\cup D_{CR}\cup D_{LR}\right]
\]
and then realizes that, by switching the order of the blow-ups, one
can lift the partial projections $\pi_{LC},\pi_{CR},\pi_{LR}$ to
\emph{$b$-fibrations}
\begin{align*}
	\gamma_{LC} & :T\to\left(\partial X\times X\right)_{0}\\
	\gamma_{CR} & :T\to X_{0}^{2}\\
	\gamma_{LR} & :T\to\left(\partial X\times X\right)_{0}.
\end{align*}
Calling $\kappa_{AP}=\beta_{\tr}^{*}K_{AP}$ and lifting the formula
above for $K_{AP}$, one obtains the identity
\[
\kappa_{AP}=\left(\gamma_{LR}\right)_{*}\left(\gamma_{LC}^{*}\kappa_{A}\cdot\gamma_{CR}^{*}\kappa_{P}\right).
\]
Understanding $\kappa_{AP}$ is now a matter of applying the pull-back
and push-forward Theorems as above.

The approach we just outlined is very elegant and extremely powerful.
It is for example one of the main ingredients in the proof of Theorem
\ref{thm:(Mazzeo=002013Melrose)Good_parametrices_fully_0-elliptic}.
However, unfortunately, it does not seem to bring the desired results
for our applications to boundary value problems. More precisely, in
\cite{MazzeoEdgeII} this approach is used to discuss compositions
\begin{align*}
	\Psi_{0\tr}^{-\infty,\bullet}\left(X,\partial X\right)\circ\Psi_{0}^{-\infty,\bullet}\left(X\right)\\
	\Psi_{0}^{-\infty,\bullet}\left(X\right)\circ\Psi_{0\po}^{-\infty,\bullet}\left(\partial X,X\right)\\
	\Psi_{0\po}^{-\infty,\bullet}\left(\partial X,X\right)\circ\Psi_{0\tr}^{-\infty,\bullet}\left(X,\partial X\right) & .
\end{align*}
The discussion of these composition rules in \cite{MazzeoEdgeII}
however presents some technical issues, which are corrected in \S2.9
of \cite{UsulaPhD}. Unfortunately, the corrected results are not
sufficient to construct parametrices for $0$-elliptic boundary value
problems.

Another serious problem with the approach outlined above is that it
does not really work to accurately discuss the compositions
\begin{align*}
	\Psi^{\bullet}\left(\partial X\right)\circ\Psi_{0\tr}^{-\infty,\bullet}\left(X,\partial X\right)\\
	\Psi_{0\po}^{-\infty,\bullet}\left(\partial X,X\right)\circ\Psi^{\bullet}\left(\partial X\right)\\
	\Psi_{0\tr}^{-\infty,\bullet}\left(X,\partial X\right)\circ\Psi_{0\po}^{-\infty,\bullet}\left(\partial X,X\right).
\end{align*}
In the first two cases, the problem is that the Schwartz kernel of
an operator $Q\in\Psi^{\bullet}\left(\partial X\right)$ has a singularity
at $\partial\Delta$ which is \emph{not} polyhomogeneous in general.
Rather, it is a more general \emph{interior conormal singularity}.
Concerning the third case, if one tries to apply the technique above
to $AB$ for $A\in\Psi_{0\tr}^{-\infty,\bullet}\left(X,\partial X\right)$
and $B\in\Psi_{0\po}^{-\infty,\bullet}\left(\partial X,X\right)$,
one obtains that $AB$ is indeed a pseudodifferential operator on
$\partial X$, with a polyhomogeneous singularity at $\partial\Delta$;
however, the result tends to be much worse than one would expect,
because of the appearance of ``spurious'' extended unions by index
sets at the front face of $\left[\partial X\times\partial X;\partial\Delta\right]$.

We will not go into a detailed discussion of these delicate issues,
for which we refer to \S2 of \cite{UsulaPhD}. However, we claim
that \emph{all these issues disappear in the symbolic approach}. The
core of our proofs is again the pull-back / push-forward technique;
however, this technique is applied \emph{in frequency space} rather
than at the Schwartz kernel level. This is the reason for which we
insist on using symbols which are polyhomogeneous on appropriate blown-up
model spaces. We remark that one could use symbols which have \emph{boundary
	conormal singularities} at the various faces of the blown-up model
spaces. This would be necessary if one wants to adapt this theory
to $0$-elliptic boundary value problems with non-constant indicial
roots. However, some very delicate extra issues (for which we refer
to Krainer and Mendoza's works) arise in this case, so extra care
is needed.

\subsection{\label{subsec:Basic-mapping-properties}Basic mapping properties}

We will now establish mapping properties on polyhomogeneous functions.
The basic idea is to imitate the standard proof of the fact that if
$Q$ is a standard pseudodifferential operator on $\partial X$, then
$Q$ induces a continuous linear map $Q:C^{\infty}\left(\partial X\right)\to C^{\infty}\left(\partial X\right)$.
Let's sketch the argument. Let's say that $Q\in\Psi^{m}\left(\partial X\right)$.
Then, by definition of $\Psi^{m}\left(\partial X\right)$, given an
atlas $\left\{ \left(U_{i},\varphi_{i}\right)\right\} $ for $\partial X$
one can decompose $Q$ as a finite sum $Q=\sum_{i}Q_{i}+R$, where
$Q_{i}\in\Psi^{m}\left(\partial X\right)$ is compactly supported
on $U_{i}\times U_{i}$ and $R\in\Psi^{-\infty}\left(\partial X\right)$
vanishes in a neighborhood of $\partial\Delta$. To prove that $Q$
induces a continuous map $C^{\infty}\left(\partial X\right)\to C^{\infty}\left(\partial X\right)$,
it suffices to prove that each $Q_{i}$ and $R$ are continuous as
maps $C^{\infty}\left(\partial X\right)\to C^{\infty}\left(\partial X\right)$.
It is straightforward to check the claim for $R$. Concerning $Q_{i}$,
one observes that if $u\in C^{\infty}\left(\partial X\right)$ and
$\chi$ is a bump function compactly supported on $U_{i}$ and such
that $\chi\boxtimes\chi$ is equal to $1$ on the support of $Q_{i}$,
then $Q_{i}u=\chi Q_{i}\chi u$. Now, using the coordinates $y$ on
$U_{i}$ induced by the chart $\varphi_{i}$, we have $Q_{i}=\Op_{L}\left(q_{i}\left(y;\eta\right)\right)\in\Psi_{\mathcal{S}}^{m}\left(\mathbb{R}^{n}\right)$
for some $q_{i}\left(y;\eta\right)\in\mathcal{S}\left(\mathbb{R}^{n}\right)\hat{\otimes}\mathcal{A}^{-m}\left(\overline{\mathbb{R}^{n}}\right)$,
and therefore it is sufficient to prove that if $Q\in\Psi_{\mathcal{S}}^{m}\left(\mathbb{R}^{n}\right)$
then $Q$ is continuous as a map $\mathcal{S}\left(\mathbb{R}^{n}\right)\to\mathcal{S}\left(\mathbb{R}^{n}\right)$.
This is a consequence of the definition of $\Op_{L}$. We have
\[
\Op_{L}\left(q\left(y;\eta\right)\right)u=\frac{1}{\left(2\pi\right)^{n}}\int e^{iy\eta}q\left(y;\eta\right)\hat{u}\left(\eta\right)d\eta,
\]
and:
\begin{enumerate}
	\item the Fourier transform $u\left(\tilde{y}\right)\mapsto\hat{u}\left(\eta\right)$
	is continuous $\mathcal{S}\left(\mathbb{R}^{n}\right)\to\mathcal{S}\left(\mathbb{R}^{n}\right)$;
	\item the map $\hat{u}\left(\eta\right)\mapsto q\left(y;\eta\right)\hat{u}\left(\eta\right)$
	is continuous $\mathcal{S}\left(\mathbb{R}^{n}\right)\to S_{\mathcal{S}}^{-\infty}\left(\mathbb{R}^{n};\mathbb{R}^{n}\right)$;
	\item the inverse Fourier transform in the $\eta$ variables $\tilde{q}\left(y;\eta\right)\mapsto\frac{1}{\left(2\pi\right)^{n}}\int e^{iy\eta}\tilde{q}\left(y;\eta\right)d\eta$
	is continuous $S_{\mathcal{S}}^{-\infty}\left(\mathbb{R}^{n};\mathbb{R}^{n}\right)=\mathcal{S}\left(\mathbb{R}^{n}\right)\hat{\otimes}\mathcal{S}\left(\mathbb{R}^{n}\right)\to\mathcal{S}\left(\mathbb{R}^{n}\right)$.
\end{enumerate}
This concludes the proof.

All our proofs in this sub-section are variations on this basic argument.
Every step is essentially unchanged, except for the second step. We
first discuss the un-twisted case.
\begin{thm}
	\label{thm:mapping-properties-on-phg-untwisted}Let $\mathcal{F}$
	be an index set.
	\begin{enumerate}
		\item Let $A\in\hat{\Psi}_{0\tr}^{-\infty,\mathcal{E}}\left(X,\partial X\right)$.
		If $\Re\left(\mathcal{E}_{\of}+\mathcal{F}\right)>0$, then $A$ extends
		from $\dot{C}^{\infty}\left(X\right)$ to a continuous linear map
		\[
		A:\mathcal{A}_{\phg}^{\mathcal{F}}\left(X\right)\to C^{\infty}\left(\partial X\right);
		\]
		\item Let $B\in\hat{\Psi}_{0\po}^{-\infty,\mathcal{E}}\left(\partial X,X\right)$.
		Then $B$ induces a continuous linear map
		\[
		B:C^{\infty}\left(\partial X\right)\to\mathcal{A}_{\phg}^{\mathcal{E}_{\of}}\left(X\right);
		\]
		\item Let $P\in\hat{\Psi}_{0}^{-\infty,\mathcal{E}}\left(X\right)$. If
		$\Re\left(\mathcal{E}_{\rf}+\mathcal{F}\right)>-1$, then $P$ extends
		from $\dot{C}^{\infty}\left(X\right)$ to a continuous linear map
		\[
		P:\mathcal{A}_{\phg}^{\mathcal{F}}\left(X\right)\to\mathcal{A}_{\phg}^{\mathcal{E}_{\lf}}\left(X\right);
		\]
		\item Let $P\in\hat{\Psi}_{0b}^{-\infty,\mathcal{E}}\left(X\right)$. If
		$\Re\left(\mathcal{E}_{\rf}+\mathcal{F}\right)>-1$, then $P$ extends
		from $\dot{C}^{\infty}\left(X\right)$ to a continuous linear map
		\[
		P:\mathcal{A}_{\phg}^{\mathcal{F}}\left(X\right)\to\mathcal{A}_{\phg}^{\mathcal{E}_{\lf}\eunion\left(\mathcal{F}+\mathcal{E}_{\ff_{b}}\right)}\left(X\right).
		\]
	\end{enumerate}
\end{thm}

\begin{proof}
	For simplicity, we only discuss the $0b$-interior case since the
	other proofs are very similar but simpler. Recall that the classes
	above have been defined using a fixed auxiliary vector field $V$
	on $X$, transverse to the boundary and inward-pointing. We can use
	$V$ to define a collar neighborhood $[0,\varepsilon)\times\partial X$
	of $\partial X$ in $X$, and therefore we can identify a neighborhood
	of the corner $\partial X\times\partial X$ in $X^{2}$ with $\left([0,\varepsilon)\times\left(\partial X\right)\right)^{2}$.
	Now, let $P\in\hat{\Psi}_{0b}^{-\infty,\mathcal{E}}\left(X\right)$.
	Then, by definition of $\hat{\Psi}_{0b}^{-\infty,\mathcal{E}}\left(X\right)$,
	given an atlas $\left\{ \left(U_{i},\varphi_{i}\right)\right\} $
	for $\partial X$ we can decompose $P$ as a finite sum $\sum_{i}Q_{i}+R$
	where $Q_{i}$ is compactly supported on $\left([0,\varepsilon)\times U_{i}\right)^{2}$
	and $R$ is an element of $\Psi_{b}^{-\infty,\left(\mathcal{E}_{\lf},\mathcal{E}_{\rf},\mathcal{E}_{\ff_{b}}\right)}\left(X\right)$
	which vanishes in a neighborhood of the lift of $\partial\Delta$
	to $X_{b}^{2}$.
	
	The fact that the residual part $R$ satisfies the mapping property
	claimed above is proved in \cite{MazzeoEdgeI} using the pull-back
	/ push-forward technique. Now, since $Q_{i}$ is compactly supported
	on $\left([0,\varepsilon)\times U_{i}\right)^{2}$, we can find compactly
	supported bump functions $\chi_{1}$ on $[0,\varepsilon)$ and $\chi_{2}$
	on $U_{i}$ such that $Q_{i}=\chi_{1}\chi_{2}Q_{i}\chi_{1}\chi_{2}$.
	Finally, calling $y$ the coordinates induced by the chart $\varphi_{i}$,
	we have $Q_{i}=\Op_{L}^{\inte}\left(q_{i}\left(y;x,\tilde{x},\eta\right)\right)$
	for some $0b$-interior symbol $q_{i}\left(y;x,\tilde{x},\eta\right)$.
	Therefore, it is sufficient to prove that if $Q=\Op_{L}^{\inte}\left(q\left(y;x,\tilde{x},\eta\right)\right)$
	for some $q\in S_{0b,\mathcal{S}}^{-\infty,\mathcal{E}}\left(\mathbb{R}^{n};\mathbb{R}_{2}^{n+2}\right)$,
	and $\Re\left(\mathcal{E}_{\rf}+\mathcal{F}\right)>-1$, then $Q$
	induces a continuous linear map\emph{
		\[
		Q:\mathcal{A}_{\phg}^{\left(\mathcal{F},\infty,\infty\right)}\left(\overline{\mathbb{R}_{1}^{1}}\times\overline{\mathbb{R}^{n}}\right)\to\mathcal{A}_{\phg}^{\left(\mathcal{E}_{\lf}\eunion\left(\mathcal{F}+\mathcal{E}_{\ff_{b}}\right),\infty,\infty\right)}\left(\overline{\mathbb{R}_{1}^{1}}\times\overline{\mathbb{R}^{n}}\right).
		\]
	}By definition of the left interior quantization map $\Op_{L}^{\inte}$,
	$Q=\Op_{L}^{\inte}\left(q\right)$ acts on a function $u$ as
	\[
	u\left(x,y\right)\mapsto\frac{1}{\left(2\pi\right)^{n}}\int e^{iy\eta}q\left(y;x,\tilde{x},\eta\right)\hat{u}\left(\tilde{x},\eta\right)d\tilde{x}d\eta.
	\]
	From this formula and the decomposition
	\begin{align*}
		S_{0b,\mathcal{S}}^{-\infty,\mathcal{E}}\left(\mathbb{R}^{n};\mathbb{R}_{2}^{n+2}\right) & =\mathcal{S}\left(\mathbb{R}^{n}\right)\hat{\otimes}\mathcal{A}_{\phg}^{\left(\mathcal{E}_{\lf},\mathcal{E}_{\rf},\mathcal{E}_{\ff_{b}}-1,\mathcal{E}_{\ff_{0}}-1,\infty,\infty,\infty\right)}\left(\hat{M}_{0b}^{2}\right),
	\end{align*}
	it is clear that it suffices to prove that given $q\left(x,\tilde{x},\eta\right)\in\mathcal{A}_{\phg}^{\left(\mathcal{E}_{\lf},\mathcal{E}_{\rf},\mathcal{E}_{\ff_{b}}-1,\mathcal{E}_{\ff_{0}}-1,\infty,\infty,\infty\right)}\left(\hat{M}_{0b}^{2}\right)$,
	the ``contraction mapping''
	\begin{align*}
		\mathcal{A}_{\phg}^{\left(\mathcal{F},\infty,\infty\right)}\left(\overline{\mathbb{R}_{1}^{1}}\times\overline{\mathbb{R}^{n}}\right) & \to\mathcal{A}_{\phg}^{\left(\mathcal{E}_{\lf}\eunion\left(\mathcal{F}+\mathcal{E}_{\ff_{b}}\right),\infty,\infty\right)}\left(\overline{\mathbb{R}_{1}^{1}}\times\overline{\mathbb{R}^{n}}\right)\\
		v\left(x,\eta\right) & \mapsto\int q\left(x,\tilde{x},\eta\right)v\left(\tilde{x},\eta\right)d\tilde{x}
	\end{align*}
	is well-defined and continuous.
	
	This is where we apply the pull-back / push-forward technique. Recall
	that we denoted by $\beta_{0b}:\hat{M}_{0b}^{2}\to\hat{M}^{2}$ the
	blow-down map; denote by $\beta_{0b,L}$ and $\beta_{0b,R}$ the lifts
	of the canonical projections $\pi_{L},\pi_{R}:\hat{M}^{2}\to\text{\ensuremath{\overline{\mathbb{R}_{1}^{1}}\times}}\overline{\mathbb{R}^{n}}$
	given by $\pi_{L}:\left(x,\tilde{x},\eta\right)\mapsto\left(x,\eta\right)$
	and $\pi_{R}:\left(x,\tilde{x},\eta\right)\mapsto\left(\tilde{x},\eta\right)$.
	Note that the range of $\pi_{L}$ should be interpreted as the model
	space $\hat{P}^{2}$, while the range of $\pi_{R}$ should be interpreted
	as the model space $\hat{T}^{2}$. The integral above can be rewritten
	as
	\[
	\left(\beta_{0b,L}\right)_{*}\left(\beta_{0b}^{*}q\cdot\beta_{0b,R}^{*}v\right).
	\]
	However this is not the correct way to interpret this integral, since
	the maps $\beta_{0b,L}$ and $\beta_{0b,R}$ are not $b$-fibrations
	(they both push $\ff_{0}$ to a corner). To fix this, we blow $\hat{M}_{0b}^{2}$
	up at the corners $\iif_{\eta}\cap\lf$ and $\iif_{\eta}\cap\rf$,
	obtaining the space
	\[
	Z_{b}=\left[\hat{M}_{0b}^{2};\iif_{\eta}\cap\left(\lf\cup\rf\right)\right].
	\]
	We denote by $\ef_{x}$ and $\ef_{\tilde{x}}$ the two ``extra faces''
	obtained from the blow-up. This space is represented in Figure \ref{fig:Zb}\begin{figure}
		\centering
		\caption{$Z_b$}
		\label{fig:Zb}
		
		%%%%%%%%%%%%%%%%%%%%%%%%%%%%%%%
		%%%%%%%%%% double space
		%%%%%%%%%%%%%%%%%%%%%%%%%%%%%%%
		
		\tikzmath{\L = 5;\R = \L /2.5;\u=180/7;}
		\tdplotsetmaincoords{60}{160}
		\begin{tikzpicture}[tdplot_main_coords, scale = 0.9]
			\centering
			%%% axes
			%\coordinate (O) at (0,0,0) ;
			%\coordinate (X) at (1,0,0) ;
			%\coordinate (Y) at (0,1,0) ;
			%\coordinate (Z) at (0,0,1) ;
			%\draw[->] (O) -- (X) node {$x$};
			%\draw[->] (O) -- (Y) node {$y$};
			%\draw[->] (O) -- (Z) node {$z$};
			
			%%% points
			\coordinate (A) at (-\L,0,0) ;
			\coordinate (A1) at (-\L+\R,0,0) ;
			\coordinate (A2) at (-\L,\R,0) ;
			\coordinate (A3) at (-\L,0,\R) ;
			\coordinate (A4) at ({-\L+\R*cos(\u)},{\R*sin(\u)},0) ;
			\coordinate (A5) at ({-\L+\R*cos(\u)},0,{\R*sin(\u)}) ;
			\coordinate (A6) at (-\L,{\R*cos(\u)},{\R*sin(\u)}) ;
			\coordinate (A7) at (-\L,{\R*cos(90-\u)},{\R*sin(90-\u)}) ;
			\coordinate (A8) at ({-\L+\R*cos(90-\u)},{\R*sin(90-\u)},0) ;
			\coordinate (A9) at ({-\L+\R*sin(\u)},0,{\R*cos(\u)}) ;

			\coordinate (B) at (\L,0,0) ;
			\coordinate (B1) at (\L-\R,0,0) ;
			\coordinate (B2) at (\L,\R,0) ;
			\coordinate (B3) at (\L,0,\R) ;
			\coordinate (B4) at ({\L-\R*cos(\u)},{\R*sin(\u)},0) ;
			\coordinate (B5) at ({\L-\R*cos(\u)},0,{\R*sin(\u)}) ;
			\coordinate (B6) at (\L,{\R*cos(\u)},{\R*sin(\u)}) ;
			\coordinate (B7) at (\L,{\R*cos(90-\u)},{\R*sin(90-\u)}) ;
			\coordinate (B8) at ({\L-\R*cos(90-\u)},{\R*sin(90-\u)},0) ;
			\coordinate (B9) at ({\L-\R*sin(\u)},0,{\R*cos(\u)}) ;

			\coordinate (C) at (-\L,\L,0) ;
			\coordinate (C1) at ({-\L+\R*cos(90-\u)},\L,0) ;
			\coordinate (C2) at (-\L,\L,{\R*sin(\u)}) ;
			
			\coordinate (D) at (\L,\L,0) ;
			\coordinate (D1) at ({\L-\R*cos(90-\u)},\L,0) ;
			\coordinate (D2) at (\L,\L,{\R*sin(\u)}) ;
			
			\coordinate (E) at (-\L,0,\L) ;
			\coordinate (E1) at (-\L,{\R*cos(90-\u)},\L) ;
			\coordinate (E2) at ({-\L+\R*sin(\u)},0,\L) ;

			\coordinate (F) at (\L,0,\L) ;
			\coordinate (F1) at (\L,{\R*cos(90-\u)},\L) ;
			\coordinate (F2) at ({\L-\R*sin(\u)},0,\L) ;
			
			\coordinate (G) at (-\L,\L,\L);
			\coordinate (H) at (\L,\L,\L);
			
			%%% segments
			\draw
			(B4) -- (A4)
			(B5) -- (A5)
			(A6) -- (C2)
			(A7) -- (E1)
			(A8) -- (C1)
			(A9) -- (E2)
			(B7) -- (F1)
			(B9) -- (F2)
			(B6) -- (D2)
			(B8) -- (D1)
			
			;
			
			%%% dashed segments
			\draw[]
			(F1) -- (H)
			(H) -- (D2)
			(E1) -- (G)
			(G) -- (C2)
			(F2) -- (E2)
			(D1) -- (C1)
			(G) -- (H)
			;
			
			%%% arcs
			\begin{scope}[canvas is zy plane at x=0]
				\draw (B7) arc(\u:{90-\u}:\R);
				\draw (A7) arc(\u:{90-\u}:\R);
				\draw (B4) arc({90-\u}:{\u}:{\R*cos(\u)});
				\draw (A4) arc({90-\u}:{\u}:{\R*cos(\u)});
			\end{scope}
			\begin{scope}[canvas is xy plane at z=0]
				\draw (B8) arc({90+\u}:{180-\u}:\R);
				\draw (A8) arc({90-\u}:{\u}:\R);
				\draw (A9) arc(0:90:{\R*sin(\u)});
				\draw[] (E2) arc(0:90:{\R*sin(\u)});
				\draw (B7) arc(90:180:{\R*sin(\u)});
				\draw[] (F1) arc(90:180:{\R*sin(\u)});
			\end{scope}
			\begin{scope}[canvas is xz plane at y=0]
				\draw (B9) arc({90+\u}:{180-\u}:\R);
				\draw (A9) arc({90-\u}:{\u}:\R);
				\draw (A8) arc(0:90:{\R*sin(\u)});
				\draw[] (C1) arc(0:90:{\R*sin(\u)});
				\draw (B6) arc(90:180:{\R*sin(\u)});
				\draw[] (D2) arc(90:180:{\R*sin(\u)});
			\end{scope}
			
			%%% text
			
			% RIGHT FACE
			\node at (0,{2/3*\L},0) {$\text{rf}$};
			
			% LEFT FACE
			\node at (0,0,{2/3*\L}) {$\text{lf}$};
			
			% FF0 FACE
			\node at ({\L-\L/6},{\L/6},{\L/6}) {$\text{ff}_0$};
			\node at ({-\L+\L/6},{\L/6},{\L/6}) {$\text{ff}_0$};
			
			% EFx FACE
			\node at ({-\L+\L/12},{\L/12},{2/3*\L}) {$\text{ef}_x$};
			\node at ({\L-\L/8},{\L/8},{2/3*\L}) {$\text{ef}_x$};
			
			% EF\tildex FACE
			\node at ({-\L+\L/12},{2/3*\L},{\L/12}) {$\text{ef}_{\tilde{x}}$};
			\node at ({\L-\L/8},{2/3*\L},{\L/8}) {$\text{ef}_{\tilde{x}}$};
			
			% \eta INFINITY FACE
			%\node at ({-\L},{5/6 * \L},{5/6 * \L}) {$\text{if}_\eta$};
			%\node at ({\L},{5/6 * \L},{5/6 * \L}) {$\text{if}_\eta$};
			
			% b FRONT FACE
			\node at (0,{0.1*\L},{0.1*\L}) {$\text{ff}_b$};
			%%%%%%%%%%%%%%%%%%%%%%%%%%%%%%%
			%%%%%%%%%% Poisson space
			%%%%%%%%%%%%%%%%%%%%%%%%%%%%%%%
			
			\begin{scope}[shift = {(0,0,{2*\L})}]
				%%% axes
				%\coordinate (O) at (0,0,0) ;
				%\coordinate (X) at (1,0,0) ;
				%\coordinate (Y) at (0,1,0) ;
				%\coordinate (Z) at (0,0,1) ;
				%\draw[->] (O) -- (X) node {$x$};
				%\draw[->] (O) -- (Y) node {$y$};
				%\draw[->] (O) -- (Z) node {$z$};
				
				%%% points
				\coordinate (A) at (-\L,0,0) ;
				\coordinate (A1) at (-\L+\R/2,0,0) ;
				\coordinate (A2) at (-\L,\R/2,0) ;
				
				\coordinate (B) at (\L,0,0) ;
				\coordinate (B1) at (\L-\R/2,0,0) ;
				\coordinate (B2) at (\L,\R/2,0) ;
				
				\coordinate (C) at (-\L,\L,0) ;
				\coordinate (D) at (\L,\L,0) ;
				
				%%% segments
				\draw 
				(B2) -- (D) 
				(B1) -- (A1) 
				(A2) -- (C) 
				;
				
				%%% dashed segments
				\draw[]
				(D) -- (C)
				;
				
				%%% arcs
				\begin{scope}[canvas is xy plane at z=0]
					\draw (B2) arc(90:180:\R/2);
					\draw (A2) arc(90:0:\R/2);
				\end{scope}
				
				%%% text
				
				% ORIGINAL FACE
				\node at (0,{\L/9},0) {$\text{of}$};
				
				% FRONT FACE
				\node at ({\L-\L/4},{\L/4},0) {$\text{ff}$};
				\node at ({-\L+\L/4},{\L/4},0) {$\text{ff}$};
				
				% \eta INFINITY FACE
				%\node at ({\L-\L/12},{\L*2/3},0) {$\text{if}_\eta$};
				%\node at ({-\L+\L/12},{\L*2/3},0) {$\text{if}_\eta$};
				
			\end{scope}
			
			%%%%%%%%%%%%%%%%%%%%%%%%%%%%%%%
			%%%%%%%%%% Trace space
			%%%%%%%%%%%%%%%%%%%%%%%%%%%%%%%
			
			\begin{scope}[shift = {(0,{2.5*\L},0)}]
				%%% axes
				%\coordinate (O) at (0,0,0) ;
				%\coordinate (X) at (1,0,0) ;
				%\coordinate (Y) at (0,1,0) ;
				%\coordinate (Z) at (0,0,1) ;
				%\draw[->] (O) -- (X) node {$x$};
				%\draw[->] (O) -- (Y) node {$y$};
				%\draw[->] (O) -- (Z) node {$z$};
				
				%%% points
				\coordinate (A) at (-\L,0,0) ;
				\coordinate (A1) at (-\L+\R/2,0,0) ;
				\coordinate (A3) at (-\L,0,\R/2) ;
				
				\coordinate (B) at (\L,0,0) ;
				\coordinate (B1) at (\L-\R/2,0,0) ;
				\coordinate (B3) at (\L,0,\R/2) ;
				
				\coordinate (E) at (-\L,0,\L) ;
				\coordinate (F) at (\L,0,\L) ;
				
				%%% segments
				\draw 
				(B3) -- (F) 
				(B1) -- (A1) 
				(A3) -- (E)
				;
				
				%%% dashed segments
				\draw[]
				(F) -- (E)
				;
				
				%%% arcs
				\begin{scope}[canvas is xz plane at y=0]
					\draw (B3) arc(90:180:\R/2);
					\draw (A1) arc(0:90:\R/2);
				\end{scope}
				
				%%% text
				
				% ORIGINAL FACE
				\node at (0,0,{\L/9}) {$\text{of}$};
				
				% FRONT FACE
				\node at ({\L-\L/4},0,{\L/4}) {$\text{ff}$};
				\node at ({-\L+\L/4},0,{\L/4}) {$\text{ff}$};
				
				% \eta INFINITY FACE
				%\node at ({\L-\L/12},0,{\L*2/3}) {$\text{if}_\eta$};
				%\node at ({-\L+\L/12},0,{\L*2/3}) {$\text{if}_\eta$};
				
			\end{scope}
			
			%%%%%%%%%%%%%%%%%%%%%%%%%%%%%%%
			%%%%%%%%%% blow-down maps
			%%%%%%%%%%%%%%%%%%%%%%%%%%%%%%%
			
			\draw[-latex] (0,0,{\L+0.1*\L}) -- (0,0,{3/2*\L}) node[midway, right] {$\gamma_L$};
			\draw[-latex] (0,{\L+0.1*\L},0) -- (0,{3/2*\L},0) node[midway, right] {$\gamma_R$};
			
		\end{tikzpicture}
		
	\end{figure}. Now, the $b$-maps $\beta_{0b,L}:\hat{M}_{0b}^{2}\to\hat{P}^{2}$
	and $\beta_{0b,R}:\hat{M}_{0b}^{2}\to\hat{T}^{2}$ lift to two $b$-fibrations
	$\gamma_{L}:Z_{b}\to\hat{P}_{0}^{2}$ and $\gamma_{R}:Z_{b}\to\hat{T}_{0}^{2}$,
	and we can write the integral above as
	\[
	\left(\gamma_{L}\right)_{*}\left(\gamma^{*}q\cdot\gamma_{R}^{*}v\cdot\gamma_{R}^{*}d\tilde{x}\right).
	\]
	Choose boundary defining functions $\rho_{\of},\rho_{\ff},\rho_{\iif_{\eta}},\rho_{\iif_{x}}$
	for $\hat{P}_{0}^{2}$, and boundary defining functions $\tilde{\rho}_{\of},\tilde{\rho}_{\ff},\tilde{\rho}_{\iif_{\eta}},\tilde{\rho}_{\iif_{\tilde{x}}}$
	for $\hat{T}_{0}^{2}$. Then we have
	\[
	\begin{array}{ll}
		\gamma_{L}^{*}\rho_{\of}=r_{\ff_{b}}r_{\lf} & \gamma_{R}^{*}\tilde{\rho}_{\of}=r_{\ff_{b}}r_{\rf}\\
		\gamma_{L}^{*}\rho_{\ff}=r_{\ff_{0}}r_{\ef_{x}} & \gamma_{R}^{*}\tilde{\rho}_{\ff}=r_{\ff_{0}}r_{\ef_{\tilde{x}}}\\
		\gamma_{L}^{*}\rho_{\iif_{\eta}}=r_{\iif_{\eta}}r_{\ef_{\tilde{x}}} & \gamma_{R}^{*}\tilde{\rho}_{\iif_{\eta}}=r_{\iif_{\eta}}r_{\ef_{x}}\\
		\gamma_{L}^{*}\rho_{\iif_{x}}=r_{\iif_{x}} & \gamma_{R}^{*}\tilde{\rho}_{\iif_{\tilde{x}}}=r_{\iif_{\tilde{x}}}
	\end{array}
	\]
	where $r_{\lf},r_{\rf},r_{\ff_{b}},r_{\ff_{0}},r_{\ef_{x}},r_{\ef_{\tilde{x}}},r_{\iif_{\eta}},r_{\iif_{x}},r_{\iif_{\tilde{x}}}$
	are boundary defining functions for $Z_{b}$. Moreover, $\gamma_{L}$
	sends $\rf$ to the interior of $\hat{P}_{0}^{2}$, and $\gamma_{R}$
	sends $\lf$ to the interior of $\hat{T}_{0}^{2}$. These relations
	are clear from Figure \ref{fig:Zb}. Now, we have $v\in\mathcal{A}_{\phg}^{\left(\mathcal{E}_{\of},\infty,\infty\right)}\left(\overline{\mathbb{R}_{1}^{1}}\times\overline{\mathbb{R}^{n}}\right)$,
	so by the pull-back Theorem $\gamma_{R}^{*}v$ has index sets $0$
	at $\lf$, $\mathcal{F}$ at $\rf$ and $\ff_{b}$, and $\infty$
	at all the other faces. From the index sets of $\beta_{0b}^{*}q$
	at $\hat{M}_{0b}^{2}$, we obtain that $\gamma^{*}q\cdot\gamma_{R}^{*}v$
	has index sets $\mathcal{E}_{\lf}$ at $\lf$, $\mathcal{E}_{\rf}+\mathcal{F}$
	at $\rf$, $\mathcal{E}_{\ff_{b}}+\mathcal{F}-1$ at $\ff_{b}$, and
	$\infty$ at all the other faces. Now, using the properties of $\gamma_{L}$
	and $\gamma_{R}$ above, with some elementary computations we see
	that $\gamma^{*}q\cdot\gamma_{R}^{*}v\cdot\gamma_{R}^{*}d\tilde{x}$
	is a polyhomogeneous section of $^{b}\mathcal{D}_{Z}^{1}\otimes\gamma_{L}^{*}{^{b}\mathcal{D}_{\hat{P}_{0}^{2}}^{-1}}$
	with index sets $\mathcal{E}_{\lf}$ at $\lf$, $\mathcal{E}_{\rf}+\mathcal{F}+1$
	at $\rf$, $\mathcal{E}_{\ff_{b}}+\mathcal{F}$ at $\ff_{b}$, and
	$\infty$ at all the other faces. Therefore, $\left(\gamma_{L}\right)_{*}\left(\gamma^{*}q\cdot\gamma_{R}^{*}v\cdot\gamma_{R}^{*}d\tilde{x}\right)$
	is polyhomogeneous on $\hat{P}_{0}^{2}$ with index sets $\mathcal{E}_{\lf}\eunion\left(\mathcal{F}+\mathcal{E}_{\ff_{b}}\right)$
	at $\of$ and $\infty$ at all the other faces. In particular, the
	infinite order of vanishing at $\ff$ implies that $\left(\gamma_{L}\right)_{*}\left(\gamma^{*}q\cdot\gamma_{R}^{*}v\cdot\gamma_{R}^{*}d\tilde{x}\right)$
	is polyhomogeneous on the blow-down $\hat{P}^{2}$. This concludes
	the proof.
\end{proof}
\begin{rem}
	\label{rem:0-interior-better-mapping-than-0-calculus}Point 3 of the
	theorem above shows once again that $0$-interior operators are generically
	better than operators in the $0$-calculus. More precisely, recall
	from Corollary \ref{cor:global-relation-symbolic0b-physical0b} that
	we have an inclusion
	\[
	\hat{\Psi}_{0}^{-\infty,\left(\mathcal{E}_{\lf},\mathcal{E}_{\rf},\mathcal{E}_{\ff_{0}}\right)}\left(X\right)\subset\Psi_{0}^{-\infty,\left(\mathcal{E}_{\lf},\mathcal{E}_{\rf},\tilde{\mathcal{E}}_{\ff_{0}}\right)}\left(X\right)
	\]
	where $\tilde{\mathcal{E}}_{\ff_{0}}=\mathcal{E}_{\ff_{0}}\eunion\left(\mathcal{E}_{\lf}+\mathcal{E}_{\rf}+n+1\right)$.
	Now, let $\mathcal{F}$ be an index set such that $\Re\left(\mathcal{E}_{\rf}+\mathcal{F}\right)>-1$.
	By the mapping properties for the $0$-calculus proved in \cite{MazzeoEdgeI},
	we know that if $P\in\Psi_{0}^{-\infty,\left(\mathcal{E}_{\lf},\mathcal{E}_{\rf},\tilde{\mathcal{E}}_{\ff_{0}}\right)}\left(X\right)$
	then $P$ maps $\mathcal{A}_{\phg}^{\mathcal{F}}\left(X\right)$ to
	$\mathcal{A}_{\phg}^{\mathcal{E}_{\lf}\eunion\left(\mathcal{F}+\tilde{\mathcal{E}}_{\ff_{0}}\right)}\left(X\right)$.
	On the other hand, by Point 3 above, if $P\in\hat{\Psi}_{0}^{-\infty,\left(\mathcal{E}_{\lf},\mathcal{E}_{\rf},\mathcal{E}_{\ff_{0}}\right)}\left(X\right)$
	then $P$ maps $\mathcal{A}_{\phg}^{\mathcal{F}}\left(X\right)$ to
	$\mathcal{A}_{\phg}^{\mathcal{E}_{\lf}}\left(X\right)$, \emph{independently
		of $\mathcal{F}$ and $\mathcal{E}_{\ff_{0}}$} (as long as the integrability
	condition is satisfied). The crucial step in the proof consists in
	showing that, if $q\left(x,\tilde{x},\eta\right)\in\mathcal{A}_{\phg}^{\left(\mathcal{E}_{\lf},\mathcal{E}_{\rf},\mathcal{E}_{\ff_{0}}-1,\infty,\infty,\infty\right)}\left(\hat{M}_{0}^{2}\right)$,
	then we get a continuous contraction mapping
	\begin{align*}
		\mathcal{A}_{\phg}^{\left(\mathcal{F},\infty,\infty\right)}\left(\overline{\mathbb{R}_{1}^{1}}\times\overline{\mathbb{R}^{n}}\right) & \to\mathcal{A}_{\phg}^{\left(\mathcal{E}_{\lf},\infty,\infty\right)}\left(\overline{\mathbb{R}_{1}^{1}}\times\overline{\mathbb{R}^{n}}\right)\\
		v\left(x,\eta\right) & \mapsto\int q\left(x,\tilde{x},\eta\right)v\left(\tilde{x},\eta\right)d\tilde{x}.
	\end{align*}
	The proof is very similar to the one above, except that now we are
	\emph{not }blowing the corner $\lf\cap\rf$ up; effectively, we are
	replacing the space $Z_{b}$ with the simpler space $Z$ depicted
	in Figure \ref{fig:Z}. Since $Z$ has no $\ff_{b}$ face, the formula
	$\left(\gamma_{L}\right)_{*}\left(\gamma^{*}q\cdot\gamma_{R}^{*}v\cdot\gamma_{R}^{*}d\tilde{x}\right)$
	does not propagate the index set $\mathcal{F}$ to the final index
	set.
\end{rem}

The analogous statements for twisted symbolic $0$-trace and $0$-Poisson
operators, and for the twisted boundary calculus, follow easily. Fix
$\boldsymbol{E},\boldsymbol{F}\to\partial X$ vector bundles, and
$\boldsymbol{\mathfrak{s}},\boldsymbol{\mathfrak{t}}$ smooth endomorphisms
of $\boldsymbol{E},\boldsymbol{F}$ with constant indicial roots.
\begin{cor}
	\label{cor:mapping-properties-phg-twisted}Let $\mathcal{F}$ be an
	index set.
	\begin{enumerate}
		\item Let $\boldsymbol{A}\in\hat{\Psi}_{0\tr}^{-\infty,\left(\mathcal{E}_{\of},\left[\mathcal{E}_{\ff}\right]\right),\boldsymbol{\mathfrak{s}}}\left(X;\partial X,\boldsymbol{E}\right)$.
		If $\Re\left(\mathcal{E}_{\of}+\mathcal{F}\right)>0$, then $\boldsymbol{A}$
		extends from $\dot{C}^{\infty}\left(X\right)$ to a continuous linear
		map
		\[
		\boldsymbol{A}:\mathcal{A}_{\phg}^{\mathcal{F}}\left(X\right)\to C^{\infty}\left(\partial X;\boldsymbol{E}\right).
		\]
		\item Let $\boldsymbol{B}\in\hat{\Psi}_{0\po}^{-\infty,\left(\mathcal{E}_{\of},\left[\mathcal{E}_{\ff}\right]\right),\boldsymbol{\mathfrak{s}}}\left(\partial X,\boldsymbol{E};X\right)$.
		Then $\boldsymbol{B}$ induces a continuous linear map
		\[
		\boldsymbol{B}:C^{\infty}\left(\partial X;\boldsymbol{E}\right)\to\mathcal{A}_{\phg}^{\mathcal{E}_{\of}}\left(X\right).
		\]
		\item Let $\boldsymbol{Q}\in\Psi_{\phg}^{-\left[\mathcal{E}\right],\left(\boldsymbol{\mathfrak{s}},\boldsymbol{\mathfrak{t}}\right)}\left(\partial X;\boldsymbol{E},\boldsymbol{F}\right)$.
		Then $\boldsymbol{Q}$ induces a continuous linear map
		\[
		\boldsymbol{Q}:C^{\infty}\left(\partial X;\boldsymbol{E}\right)\to C^{\infty}\left(\partial X;\boldsymbol{F}\right).
		\]
	\end{enumerate}
\end{cor}

\begin{proof}
	Point 3 follows from Corollary \ref{cor:twisted-psidos-have-phg-entries}
	and the fact that pseudodifferential operators on $\partial X$ map
	smooth sections to smooth sections. Similarly, Points 1 and 2 follow
	from Corollary \ref{cor:twisted-0-trace-0-poisson-are-untwisted}
	and the previous theorem.
\end{proof}

\subsection{\label{subsec:Asymptotic-sums-and-differentiations-of-symbols}Asymptotic
	sums and differentiations of symbols}

Before discussing formal adjoints and composition theorems, we need
to pause for a moment and discuss asymptotic sums and differentiations
of $0$-trace, $0$-Poisson, $0$-interior and $0b$-interior symbols.

Let's first briefly recall the standard theory. If $q_{j}\left(y;\eta\right)\in S_{\mathcal{S}}^{m-j}\left(\mathbb{R}^{n};\mathbb{R}^{n}\right)$,
an \emph{asymptotic sum }$q\sim\sum_{j}q_{j}$ is a symbol $q\left(y;\eta\right)\in S_{\mathcal{S}}^{m}\left(\mathbb{R}^{n};\mathbb{R}^{n}\right)$
such that, for every $M$, we have
\[
q\left(y;\eta\right)-\sum_{j\leq M}q_{j}\left(y;\eta\right)\in S_{\mathcal{S}}^{m-M}\left(\mathbb{R}^{n};\mathbb{R}^{n}\right).
\]
The symbol space $S_{\mathcal{S}}^{m}\left(\mathbb{R}^{n};\mathbb{R}^{n}\right)$
is \emph{asymptotically complete}, meaning that given a sequence $q_{j}\left(y;\eta\right)\in S_{\mathcal{S}}^{m-j}\left(\mathbb{R}^{n};\mathbb{R}^{n}\right)$,
an asymptotic sum $q\sim\sum_{j}q_{j}$ exists. Such $q$ is unique
modulo $S_{\mathcal{S}}^{-\infty}\left(\mathbb{R}^{n};\mathbb{R}^{n}\right)$:
indeed, if $q,q'\sim\sum_{j}q_{j}$, and we call $q_{M}=\sum_{j\leq M}q_{j}$,
we have
\[
q-q'=\left(q-q_{M}\right)-\left(q'-q_{M}\right)\in S_{\mathcal{S}}^{m-M}\left(\mathbb{R}^{n};\mathbb{R}^{n}\right)
\]
for every $M$, which implies that $q-q'\in S_{\mathcal{S}}^{-\infty}\left(\mathbb{R}^{n};\mathbb{R}^{n}\right)$.
Concerning differentiation of symbols, we have continuous linear maps
\begin{align*}
	\partial_{y}^{\alpha} & :S_{\mathcal{S}}^{m}\left(\mathbb{R}^{n};\mathbb{R}^{n}\right)\to S_{\mathcal{S}}^{m}\left(\mathbb{R}^{n};\mathbb{R}^{n}\right)\\
	\partial_{\eta}^{\alpha} & :S_{\mathcal{S}}^{m}\left(\mathbb{R}^{n};\mathbb{R}^{n}\right)\to S_{\mathcal{S}}^{m-\left|\alpha\right|}\left(\mathbb{R}^{n};\mathbb{R}^{n}\right).
\end{align*}
Since $S_{\mathcal{S}}^{m}\left(\mathbb{R}^{n};\mathbb{R}^{n}\right)=\mathcal{S}\left(\mathbb{R}^{n}\right)\hat{\otimes}\mathcal{A}^{-m}\left(\overline{\mathbb{R}^{n}}\right)$,
the statement about $\partial_{y}^{\alpha}$ is obvious, while for
the statement about $\partial_{\eta}^{\alpha}$ it suffices to observe
that:
\begin{enumerate}
	\item $\left\langle \eta\right\rangle ^{\left|\alpha\right|}\partial_{\eta}^{\alpha}$
	is a $b$-differential operator on $\overline{\mathbb{R}^{n}}$, and
	therefore it acts continuously on $\mathcal{A}^{-m}\left(\overline{\mathbb{R}^{n}}\right)$;
	\item $\left\langle \eta\right\rangle ^{-1}$ vanishes simply at $\partial\overline{\mathbb{R}^{n}}$,
	and therefore the multiplication by $\left\langle \eta\right\rangle ^{-\left|\alpha\right|}$
	induces a continuous linear map $\mathcal{A}^{-m}\left(\overline{\mathbb{R}^{n}}\right)\to\mathcal{A}^{-m+\left|\alpha\right|}\left(\overline{\mathbb{R}^{n}}\right)$.
\end{enumerate}
Now consider the symbol spaces \[S_{0\tr,\mathcal{S}}^{-\infty,\mathcal{E}}\left(\mathbb{R}^{n};\mathbb{R}_{1}^{n+1}\right),\quad
S_{0\po,\mathcal{S}}^{-\infty,\mathcal{E}}\left(\mathbb{R}^{n};\mathbb{R}_{1}^{n+1}\right),\quad
S_{0,\mathcal{S}}^{-\infty,\mathcal{E}}\left(\mathbb{R}^{n};\mathbb{R}_{2}^{n+2}\right),
\quad S_{0b,\mathcal{S}}^{-\infty,\mathcal{E}}\left(\mathbb{R}^{n};\mathbb{R}_{2}^{n+2}\right).\]
These spaces are all asymptotically complete at each face. This is
just a feature of polyhomogeneous functions on manifolds with corners,
cf.\ \cite{MelroseCorners}. We are only concerned about asymptotic
completeness at the front face for $0$-trace and $0$-Poisson symbols,
and at the $0$-front face for $0$-interior and $0b$-interior symbols.
For example, in the $0$-Poisson case, consider a sequence $b_{j}\in S_{0\po,\mathcal{S}}^{-\infty,\left(\mathcal{E}_{\of},\mathcal{E}_{\ff}+j\right)}\left(\mathbb{R}^{n};\mathbb{R}_{1}^{n+1}\right)$.
Then asymptotic completeness of $S_{0\po,\mathcal{S}}^{-\infty,\left(\mathcal{E}_{\of},\mathcal{E}_{\ff}\right)}\left(\mathbb{R}^{n};\mathbb{R}_{1}^{n+1}\right)$
at $\ff$ asserts that there exists a $0$-Poisson symbol \[b\in S_{0\po,\mathcal{S}}^{-\infty,\left(\mathcal{E}_{\of},\mathcal{E}_{\ff}\right)}\left(\mathbb{R}^{n};\mathbb{R}_{1}^{n+1}\right)\]
such that $b\sim\sum_{j}b_{j}$ in the sense that, for every $M$
we have
\[
b-\sum_{j\leq M}b_{j}\in\mathcal{S}\left(\mathbb{R}^{n}\right)\hat{\otimes}\mathcal{B}^{\left(\mathcal{E}_{\of},\inf\left(\mathcal{E}_{\ff}\right)+M,\infty,\infty\right)}\left(\hat{P}_{0}^{2}\right).
\]
Here $\mathcal{B}^{\left(\mathcal{E}_{\of},m,\infty,\infty\right)}\left(\hat{P}_{0}^{2}\right)$
is the space of conormal functions on $\hat{P}_{0}^{2}$ which are
polyhomogeneous at $\of$ with index set $\mathcal{E}_{\of}$, vanish
to infinite order at $\iif_{\eta},\iif_{x}$, and are conormal of
order $m$ at $\ff$. A similar discussion holds for $0$-trace, $0$-interior
and $0b$-interior symbols.

Now let's discuss differentiations of symbols. The un-twisted case
is straightforward:
\begin{lem}
	\label{lem:differentiation-of-untwisted-symbols}$ $
	\begin{enumerate}
		\item The operator $\partial_{y}^{\alpha}$ acts continuously on the symbol
		spaces \[S_{0\tr,\mathcal{S}}^{-\infty,\mathcal{E}}\left(\mathbb{R}^{n};\mathbb{R}_{1}^{n+1}\right),\quad
		S_{0\po,\mathcal{S}}^{-\infty,\mathcal{E}}\left(\mathbb{R}^{n};\mathbb{R}_{1}^{n+1}\right),\quad
		S_{0,\mathcal{S}}^{-\infty,\mathcal{E}}\left(\mathbb{R}^{n};\mathbb{R}_{2}^{n+2}\right),\quad
		S_{0b,\mathcal{S}}^{-\infty,\mathcal{E}}\left(\mathbb{R}^{n};\mathbb{R}_{2}^{n+2}\right).\]
		\item The operator $\partial_{\eta}^{\alpha}$ induces continuous linear
		maps
		\begin{align*}
			S_{0\tr,\mathcal{S}}^{-\infty,\mathcal{E}}\left(\mathbb{R}^{n};\mathbb{R}_{1}^{n+1}\right) & \to S_{0\tr,\mathcal{S}}^{-\infty,\left(\mathcal{E}_{\of},\mathcal{E}_{\ff}+\left|\alpha\right|\right)}\left(\mathbb{R}^{n};\mathbb{R}_{1}^{n+1}\right)\\
			S_{0\po,\mathcal{S}}^{-\infty,\mathcal{E}}\left(\mathbb{R}^{n};\mathbb{R}_{1}^{n+1}\right) & \to S_{0\po,\mathcal{S}}^{-\infty,\left(\mathcal{E}_{\of},\mathcal{E}_{\ff}+\left|\alpha\right|\right)}\left(\mathbb{R}^{n};\mathbb{R}_{1}^{n+1}\right)\\
			S_{0,\mathcal{S}}^{-\infty,\mathcal{E}}\left(\mathbb{R}^{n};\mathbb{R}_{2}^{n+2}\right) & \to S_{0,\mathcal{S}}^{-\infty,\left(\mathcal{E}_{\lf},\mathcal{E}_{\rf},\mathcal{E}_{\ff_{0}}+\left|\alpha\right|\right)}\left(\mathbb{R}^{n};\mathbb{R}_{2}^{n+2}\right)\\
			S_{0b,\mathcal{S}}^{-\infty,\mathcal{E}}\left(\mathbb{R}^{n};\mathbb{R}_{2}^{n+2}\right) & \to S_{0b,\mathcal{S}}^{-\infty,\left(\mathcal{E}_{\lf},\mathcal{E}_{\rf},\mathcal{E}_{\ff_{b}},\mathcal{E}_{\ff_{0}}+\left|\alpha\right|\right)}\left(\mathbb{R}^{n};\mathbb{R}_{2}^{n+2}\right).
		\end{align*}
	\end{enumerate}
\end{lem}

\begin{proof}
	The proof is essentially the same for all the four cases, so let's
	discuss in detail the $0$-Poisson case. By definition, we have
	\[
	S_{0\po,\mathcal{S}}^{-\infty,\mathcal{E}}\left(\mathbb{R}^{n};\mathbb{R}_{1}^{n+1}\right)=\mathcal{S}\left(\mathbb{R}^{n}\right)\hat{\otimes}\mathcal{A}_{\phg}^{\left(\mathcal{E}_{\of},\mathcal{E}_{\ff},\infty,\infty\right)}\left(\hat{P}_{0}^{2}\right),
	\]
	so the statement about the operator $\partial_{y}^{\alpha}$ is obvious.
	Concerning the operator $\partial_{\eta}^{\alpha}$, it suffices to
	show that it induces a continuous linear map
	\[
	\partial_{\eta}^{\alpha}:\mathcal{A}_{\phg}^{\left(\mathcal{E}_{\of},\mathcal{E}_{\ff},\infty,\infty\right)}\left(\hat{P}_{0}^{2}\right)\to\mathcal{A}_{\phg}^{\left(\mathcal{E}_{\of},\mathcal{E}_{\ff}+\left|\alpha\right|,\infty,\infty\right)}\left(\hat{P}_{0}^{2}\right).
	\]
	Now, since $\left\langle \eta\right\rangle ^{\alpha}\partial_{\eta}^{\alpha}$
	is a $b$-differential operator on $\overline{\mathbb{R}^{n}}$, its
	lift to $\hat{P}^{2}=\overline{\mathbb{R}_{1}^{1}}\times\overline{\mathbb{R}^{n}}$
	is a $b$-differential operator again. But $\hat{P}_{0}^{2}$ is obtained
	from $\hat{P}^{2}$ by blowing up a corner; since $\left\langle \eta\right\rangle ^{\alpha}\partial_{\eta}^{\alpha}$
	is tangent to the corner, it lifts to a $b$-differential operator
	on $\hat{P}_{0}^{2}$ as well. This implies that $\left\langle \eta\right\rangle ^{\alpha}\partial_{\eta}^{\alpha}$
	acts continuously on $\mathcal{A}_{\phg}^{\left(\mathcal{E}_{\of},\mathcal{E}_{\ff},\infty,\infty\right)}\left(\hat{P}_{0}^{2}\right)$.
	Finally, since $\left\langle \eta\right\rangle ^{\alpha}$ has index
	sets $-\left|\alpha\right|$ at $\ff$ and $\iif_{\eta}$ and $0$
	at $\of$ and $\ff_{x}$, multiplication by $\left\langle \eta\right\rangle ^{-\alpha}$
	maps $\mathcal{A}_{\phg}^{\left(\mathcal{E}_{\of},\mathcal{E}_{\ff},\infty,\infty\right)}\left(\hat{P}_{0}^{2}\right)$
	continuously to $\mathcal{A}_{\phg}^{\left(\mathcal{E}_{\of},\mathcal{E}_{\ff}+\left|\alpha\right|,\infty,\infty\right)}\left(\hat{P}_{0}^{2}\right)$.
\end{proof}
The twisted\emph{ }case is trickier. Let us fix smooth families $\mathfrak{s}:\overline{\mathbb{R}^{n}}\to\mathfrak{gl}\left(N,\mathbb{C}\right)$
and $\mathfrak{t}:\overline{\mathbb{R}^{n}}\to\mathfrak{gl}\left(M,\mathbb{C}\right)$
of matrices with a constant, single eigenvalue.
\begin{lem}
	\label{lem:differentiation-of-twisted-symbols}Let \begin{equation*}
		\begin{aligned}
			b &\in S_{0\po,\mathcal{S}}^{-\infty,\mathcal{E},\mathfrak{s}}
			\left(\mathbb{R}^{n};\mathbb{R}_{1}^{n+1};
			\left(\mathbb{C}^{N}\right)^{*}\right),\\
			a &\in S_{0\tr,\mathcal{S}}^{-\infty,\mathcal{E},\mathfrak{s}}
			\left(\mathbb{R}^{n};\mathbb{R}_{1}^{n+1};
			\mathbb{C}^{N}\right),\\
			q &\in S_{\phg,\mathcal{S}}^{-\mathcal{E}_{\ff},
				\left(\mathfrak{s},\mathfrak{t}\right)}
			\left(\mathbb{R}^{n};
			\hom\left(\mathbb{C}^{N},\mathbb{C}^{M}\right)\right).
		\end{aligned}
	\end{equation*}
	Then:
	\begin{enumerate}
		\item we have
		\begin{align*}
			\partial_{\eta}^{\alpha}b & \in S_{0\po,\mathcal{S}}^{-\infty,\left(\mathcal{E}_{\of},\mathcal{E}_{\ff}+\left|\alpha\right|\right),\mathfrak{s}}\left(\mathbb{R}^{n};\mathbb{R}_{1}^{n+1};\left(\mathbb{C}^{N}\right)^{*}\right)\\
			\partial_{\eta}^{\alpha}a & \in S_{0\tr,\mathcal{S}}^{-\infty,\left(\mathcal{E}_{\of},\mathcal{E}_{\ff}+\left|\alpha\right|\right),\mathfrak{s}}\left(\mathbb{R}^{n};\mathbb{R}_{1}^{n+1};\mathbb{C}^{N}\right)\\
			\partial_{\eta}^{\alpha}q & \in S_{\phg,\mathcal{S}}^{-\left(\mathcal{E}_{\ff}+\left|\alpha\right|\right),\left(\mathfrak{s},\mathfrak{t}\right)}\left(\mathbb{R}^{n};\hom\left(\mathbb{C}^{N},\mathbb{C}^{M}\right)\right);
		\end{align*}
		\item we have
		\begin{align*}
			\partial_{y}^{\alpha}b & \in S_{0\po,\mathcal{S}}^{-\infty,\left(\mathcal{E}_{\of},\tilde{\mathcal{E}}_{\ff}\right),\mathfrak{s}}\left(\mathbb{R}^{n};\mathbb{R}_{1}^{n+1};\left(\mathbb{C}^{N}\right)^{*}\right)\\
			\partial_{y}^{\alpha}a & \in S_{0\tr,\mathcal{S}}^{-\infty,\left(\mathcal{E}_{\of},\tilde{\mathcal{E}}_{\ff}\right),\mathfrak{s}}\left(\mathbb{R}^{n};\mathbb{R}_{1}^{n+1};\mathbb{C}^{N}\right)\\
			\partial_{y}^{\alpha}q & \in S_{\phg,\mathcal{S}}^{-\tilde{\mathcal{E}}_{\ff},\left(\mathfrak{s},\mathfrak{t}\right)}\left(\mathbb{R}^{n};\hom\left(\mathbb{C}^{N},\mathbb{C}^{M}\right)\right),
		\end{align*}
		where $\tilde{\mathcal{E}}_{\ff}$ is an index set obtained from $\mathcal{E}_{\ff}$
		by increasing some of the logarithmic orders.
	\end{enumerate}
\end{lem}

\begin{proof}
	For simplicity we prove the statement only in the $0$-Poisson case,
	since the others are very similar. By definition of $S_{0\po,\mathcal{S}}^{-\infty,\mathcal{E},\mathfrak{s}}\left(\mathbb{R}^{n};\mathbb{R}_{1}^{n+1};\left(\mathbb{C}^{N}\right)^{*}\right)$,
	we can write
	\[
	b\left(y;x,\eta\right)=\tilde{b}\left(y;x,\eta\right)\left\langle \eta\right\rangle ^{\mathfrak{s}\left(y\right)}
	\]
	where $\tilde{b}\in S_{0\po,\mathcal{S}}^{-\infty,\mathcal{E}}\left(\mathbb{R}^{n};\mathbb{R}_{1}^{n+1};\left(\mathbb{C}^{N}\right)^{*}\right)$.
	By incorporating the eigenvalue of $\mathfrak{s}$ into the index
	set $\mathcal{E}_{\ff}$, we can assume without loss of generality
	that $\mathfrak{s}$ is nilpotent. Now, we have
	\begin{align*}
		\partial_{\eta_{j}}\left(\tilde{b}\left\langle \eta\right\rangle ^{\mathfrak{s}}\right) & =\left(\partial_{\eta_{j}}\tilde{b}\right)\left\langle \eta\right\rangle ^{\mathfrak{s}}+\tilde{b}\left(\partial_{\eta_{j}}\left\langle \eta\right\rangle ^{\mathfrak{s}}\right)\\
		& =\left(\partial_{\eta_{j}}\tilde{b}+\tilde{b}\left(\mathfrak{s}\left\langle \eta\right\rangle ^{-1}\frac{\eta^{j}}{\left\langle \eta\right\rangle }\right)\right)\left\langle \eta\right\rangle ^{\mathfrak{s}}.
	\end{align*}
	This computation is valid because $\mathfrak{s}=\mathfrak{s}\left(y\right)$
	does not depend on $\eta_{j}$. Now, the function $\eta_{j}/\left\langle \eta\right\rangle $
	is smooth on $\overline{\mathbb{R}^{n}}$, $\mathfrak{s}\left(y\right)$
	is a smooth factor, and $\left\langle \eta\right\rangle ^{-1}$ has
	index sets $1$ at $\ff$ and $\iif_{\eta}$ and $0$ at $\of$ and
	$\iif_{x}$; it then follows that \[\tilde{b}\left(\mathfrak{s}\left\langle \eta\right\rangle ^{-1}\frac{\eta^{j}}{\left\langle \eta\right\rangle }\right) \in S_{0\po,\mathcal{S}}^{-\infty,\left(\mathcal{E}_{\of},\mathcal{E}_{\ff}+1\right)}\left(\mathbb{R}^{n};\mathbb{R}_{1}^{n+1};\left(\mathbb{C}^{N}\right)^{*}\right).\]
	We already proved that $\partial_{\eta_{j}}$ increases the index
	set at the front face by $1$, so $\partial_{\eta_{j}}\tilde{b}$
	is in $S_{0\po,\mathcal{S}}^{-\infty,\left(\mathcal{E}_{\of},\mathcal{E}_{\ff}+1\right)}\left(\mathbb{R}^{n};\mathbb{R}_{1}^{n+1};\left(\mathbb{C}^{N}\right)^{*}\right)$
	as well. Therefore, iterating this argument, we have
	\begin{align*}
		\partial_{\eta_{j}}\left(\tilde{b}\left\langle \eta\right\rangle ^{\mathfrak{s}}\right) & \in S_{0\po,\mathcal{S}}^{-\infty,\left(\mathcal{E}_{\of},\mathcal{E}_{\ff}+\left|\alpha\right|\right),\mathfrak{s}}\left(\mathbb{R}^{n};\mathbb{R}_{1}^{n+1};\left(\mathbb{C}^{N}\right)^{*}\right).
	\end{align*}
	Now let's consider the action of $\partial_{y}^{\alpha}$. Here the
	hypothesis of nilpotency of $\mathfrak{s}$ helps. Indeed, write
	\begin{align*}
		\partial_{y_{j}}\left(\tilde{b}\left\langle \eta\right\rangle ^{\mathfrak{s}}\right) & =\left(\partial_{y_{j}}\tilde{b}\right)\left\langle \eta\right\rangle ^{\mathfrak{s}}+\tilde{b}\left(\sum_{k<N}\frac{1}{k!}\left(\log\left\langle \eta\right\rangle \right)^{k}\partial_{y_{j}}\left(\mathfrak{s}^{k}\right)\right)\\
		& =\left(\partial_{y_{j}}\tilde{b}+\tilde{b}\varphi\left(y,\eta\right)\right)\left\langle \eta\right\rangle ^{\mathfrak{s}},
	\end{align*}
	where
	\[
	\varphi\left(y,\eta\right)=\left[\sum_{k<N}\frac{1}{k!}\left(\log\left\langle \eta\right\rangle \right)^{k}\partial_{y_{j}}\left(\mathfrak{s}^{k}\right)\right]\left\langle \eta\right\rangle ^{-\mathfrak{s}}.
	\]
	This factor is smooth in $y$, but it is \emph{not} smooth in $\eta$.
	Rather, it is \emph{polyhomogeneous}, with index set generated by
	the pair $\left(0,2\left(N-1\right)\right)$. It follows that
	\begin{align*}
		\partial_{y_{j}}\left(\tilde{b}\left\langle \eta\right\rangle ^{\mathfrak{s}}\right) & \in S_{0\po,\mathcal{S}}^{-\infty,\left(\mathcal{E}_{\of},\tilde{\mathcal{E}}_{\ff}\right),\mathfrak{s}}\left(\mathbb{R}^{n};\mathbb{R}_{1}^{n+1};\left(\mathbb{C}^{N}\right)^{*}\right),
	\end{align*}
	where $\tilde{\mathcal{E}}_{\ff}$ is obtained from $\mathcal{E}_{\ff}$
	by increasing some of the logarithmic orders.
\end{proof}
\begin{rem}
	The phenomenon we encountered in the previous proof is essentially
	the reason for which, when we defined the classes \[\hat{\Psi}_{0\tr}^{-\infty,\left(\mathcal{E}_{\of},\left[\mathcal{E}_{\ff}\right]\right),\boldsymbol{\mathfrak{s}}}\left(X;\partial X,\boldsymbol{E}\right),\quad
	\hat{\Psi}_{0\po}^{-\infty,\left(\mathcal{E}_{\of},\left[\mathcal{E}_{\ff}\right]\right),\boldsymbol{\mathfrak{s}}}\left(\partial X,\boldsymbol{E};X\right),\quad
	\Psi_{\phg}^{-\left[\mathcal{E}\right],\left(\boldsymbol{\mathfrak{s}},\boldsymbol{\mathfrak{t}}\right)}\left(\partial X;\boldsymbol{E},\boldsymbol{F}\right),\]
	we did not keep track of the precise index set at the front face (or
	at $\left|\eta\right|=\infty$ in the boundary case), but only of
	its leading set. The reason is essentially that, when differentiating
	symbols in the classes
	\begin{align*}
		S_{0\tr,\mathcal{S}}^{-\infty,\mathcal{E},\mathfrak{s}}\left(\mathbb{R}^{n};\mathbb{R}_{1}^{n+1};\mathbb{C}^{N}\right)\\
		S_{0\po,\mathcal{S}}^{-\infty,\mathcal{E},\mathfrak{s}}\left(\mathbb{R}^{n};\mathbb{R}_{1}^{n+1};\left(\mathbb{C}^{N}\right)^{*}\right)\\
		S_{\phg,\mathcal{S}}^{-\mathcal{E},\left(\mathfrak{s},\mathfrak{t}\right)}\left(\mathbb{R}^{n};\hom\left(\mathbb{C}^{N},\mathbb{C}^{M}\right)\right) & ,
	\end{align*}
	the operator $\partial_{y}^{\alpha}$ causes a logarithmic loss of
	regularity (at the front face in the first two cases, and at $\left|\eta\right|=\infty$
	in the third case). When we will discuss formal adjoints and composition
	theorems for the classes $\hat{\Psi}_{0\tr}^{-\infty,\left(\mathcal{E}_{\of},\left[\mathcal{E}_{\ff}\right]\right),\boldsymbol{\mathfrak{s}}}\left(X;\partial X,\boldsymbol{E}\right)$,
	$\hat{\Psi}_{0\po}^{-\infty,\left(\mathcal{E}_{\of},\left[\mathcal{E}_{\ff}\right]\right),\boldsymbol{\mathfrak{s}}}\left(\partial X,\boldsymbol{E};X\right)$,
	and $\Psi_{\phg}^{-\left[\mathcal{E}\right],\left(\boldsymbol{\mathfrak{s}},\boldsymbol{\mathfrak{t}}\right)}\left(\partial X;\boldsymbol{E},\boldsymbol{F}\right)$,
	we will need to use asymptotic sums arguments; because of this logarithmic
	loss, we will not be able to control the whole index set, but only
	its leading set. This issue is however tamed by the fact that we are
	dealing with twisting endomorphisms with constant eigenvalues: indeed,
	thanks to this property, the logarithmic loss is still \emph{finite},
	and therefore we can stay in the polyhomogeneous class. If we want
	to allow twisting endomorphisms with non-constant eigenvalues (a necessary
	thing to do if one wants to consider $0$-elliptic problems with non-constant
	indicial roots), things get more complicated. The paper \cite{KrainerMendozaVariableOrders}
	discusses this issue extensively.
\end{rem}

\subsection{\label{subsec:Formal-adjoints}Formal adjoints}

We now discuss formal adjoints. For simplicity, let's fix a boundary
defining function $x$ and a smooth positive density $\omega$ on
$X$, and choose an arbitrary real weight $\delta$. This allows us
to define an inner product on $x^{\delta}L_{b}^{2}\left(X\right)$
by
\[
\left(u,v\right)_{x^{\delta}L_{b}^{2}}=\int_{X}u\overline{v}x^{-2\delta-1}\omega.
\]
Moreover, given the choice of our auxiliary vector field $V$, contracting
$\omega_{|\partial X}$ with $V_{|\partial X}$ we obtain a smooth
positive density on $\partial X$ which we call $\nu$. This allows
us to define an inner product on $L^{2}\left(\partial X\right)$.
We denote by $\dagger$ the formal adjoint operator with respect to
these choices of inner products.

The discussion of the properties of the operator $\dagger$ is again
reduced to a local discussion. For example, let us summarize how to
prove that if $Q\in\Psi^{m}\left(\partial X\right)$, then $Q^{\dagger}\in\Psi^{m}\left(\partial X\right)$
as well. Given a covering of $\partial X$ into coordinate charts
$\left(U_{i},\varphi_{i}\right)$, we can write $Q=\sum_{i}Q_{i}+R$
where $Q_{i}$ is compactly supported on $U_{i}\times U_{i}$ and
$R\in\Psi^{-\infty}\left(\partial X\right)$. The proof that $R^{\dagger}\in\Psi^{-\infty}\left(\partial X\right)$
is straightforward. Concerning the terms $Q_{i}$, call $y$ the coordinates
induced by the chart and write $Q_{i}=\Op_{L}\left(q_{i}\right)\in\Psi_{\mathcal{S}}^{m}\left(\mathbb{R}^{n}\right)$
for some $q_{i}\left(y;\eta\right)\in\mathcal{S}\left(\mathbb{R}^{n}\right)\hat{\otimes}\mathcal{A}^{-m}\left(\overline{\mathbb{R}^{n}}\right)$.
Then, up to conjugating $Q_{i}$ by a smooth compactly supported function,
it suffices to prove that $Q_{i}^{*}$, the formal adjoint with respect
to the coordinate density $dy$, is again in $\Psi_{\mathcal{S}}^{m}\left(\mathbb{R}^{n}\right)$.

A way to prove this is to use the \emph{right quantization map}. Formally,
this map is defined as
\begin{align*}
	\Op_{R}:S_{\mathcal{S}}^{m}\left(\mathbb{R}^{n};\mathbb{R}^{n}\right)\times\mathcal{S}\left(\mathbb{R}^{n}\right) & \to\mathcal{S}\left(\mathbb{R}^{n}\right)\\
	\left(q\left(\tilde{y};\eta\right),u\left(y\right)\right) & \mapsto\frac{1}{\left(2\pi\right)^{n}}\int e^{i\left(y-\tilde{y}\right)\eta}q\left(\tilde{y};\eta\right)u\left(\tilde{y}\right)d\tilde{y}d\eta.
\end{align*}
One way to check that this map is well-defined is to observe that
\[S_{\mathcal{S}}^{m}\left(\mathbb{R}^{n};\mathbb{R}^{n}\right)=\mathcal{S}\left(\mathbb{R}^{n}\right)\hat{\otimes}\mathcal{A}^{-m}\left(\overline{\mathbb{R}^{n}}\right),\]
and on decomposable symbols $q\left(\tilde{y};\eta\right)=a\left(\tilde{y}\right)q'\left(\eta\right)$
we have
\[
\Op_{R}\left(a\left(\tilde{y}\right)q'\left(\eta\right)\right)u=\frac{1}{\left(2\pi\right)^{n}}\int e^{iy\eta}q'\left(\eta\right)\widehat{a\cdot u}\left(\eta\right)d\eta;
\]
this map is clearly bilinear continuous in $a$ and $q'$, so by the
universal property of the completed projective tensor product $\hat{\otimes}$
$\Op_{R}$ extends to $S_{\mathcal{S}}^{m}\left(\mathbb{R}^{n};\mathbb{R}^{n}\right)$.
Now, it is not hard to prove that the range of $\Op_{R}$ on $S_{\mathcal{S}}^{m}\left(\mathbb{R}^{n};\mathbb{R}^{n}\right)$
coincides with $\Psi_{\mathcal{S}}^{m}\left(\mathbb{R}^{n}\right)$.
More precisely, for every ``right-reduced symbol'' $q\left(\tilde{y};\eta\right)\in S_{\mathcal{S}}^{m}\left(\mathbb{R}^{n};\mathbb{R}^{n}\right)$,
there exists a unique ``left-reduced symbol'' $q_{L}\left(y;\eta\right)\in S_{\mathcal{S}}^{m}\left(\mathbb{R}^{n};\mathbb{R}^{n}\right)$
such that $\Op_{R}\left(q\left(\tilde{y};\eta\right)\right)=\Op_{L}\left(q_{L}\left(y;\eta\right)\right)$.
Moreover, $q_{L}$ is determined modulo $S_{\mathcal{S}}^{-\infty}\left(\mathbb{R}^{n};\mathbb{R}^{n}\right)$
as an asymptotic sum
\[
q_{L}\left(y;\eta\right)\sim\sum_{\alpha}\frac{1}{\alpha!}\left(D_{\eta}^{\alpha}\partial_{\tilde{y}}^{\alpha}q\right)_{|\tilde{y}=y}.
\]
Now, if $Q=\Op_{L}\left(q\left(y;\eta\right)\right)$, its Schwartz
kernel is $Q\left(y;y-\tilde{y}\right)d\tilde{y}$, with $Q\left(y;Y\right)$
equal to the inverse Fourier transform of $q\left(y;\eta\right)$.
Similarly, if $Q=\Op_{R}\left(q\left(\tilde{y};\eta\right)\right)$,
its Schwartz kernel is $Q\left(\tilde{y};y-\tilde{y}\right)d\tilde{y}$
where $Q\left(\tilde{y};Y\right)$ is the inverse Fourier transform
of $q\left(\tilde{y};\eta\right)$. Finally, if $Q=\Op_{L}\left(q\right)=Q\left(y;y-\tilde{y}\right)d\tilde{y}$,
then it is easy to check that $Q^{*}=\overline{Q\left(\tilde{y};\tilde{y}-y\right)}d\tilde{y}$.
Therefore, $Q^{*}=\Op_{R}\left(q^{*}\right)$, where $q^{*}\left(\tilde{y};\eta\right)=\overline{q\left(\tilde{y};\eta\right)}$.
Since $q\in S_{\mathcal{S}}^{m}\left(\mathbb{R}^{n};\mathbb{R}^{n}\right)$,
it is clear that $q^{*}\in S_{\mathcal{S}}^{m}\left(\mathbb{R}^{n};\mathbb{R}^{n}\right)$
as well and therefore $Q^{*}\in\Psi_{\mathcal{S}}^{m}\left(\mathbb{R}^{n}\right)$.

The argument is very similar for our classes of operators. Let's start
with the un-twisted case:
\begin{prop}
	\label{prop:formal-adjoints}$ $
	\begin{enumerate}
		\item Let $A\in\hat{\Psi}_{0\tr}^{-\infty,\mathcal{E}}\left(X,\partial X\right)$.
		Then $A^{\dagger}\in\hat{\Psi}_{0\po}^{-\infty,\mathcal{F}}\left(\partial X,X\right)$,
		where
		\begin{align*}
			\mathcal{F}_{\of} & =\overline{\mathcal{E}_{\of}}+2\delta+1\\
			\mathcal{F}_{\ff} & =\overline{\mathcal{E}_{\ff}}+2\delta.
		\end{align*}
		\item Let $B\in\hat{\Psi}_{0\po}^{-\infty,\mathcal{E}}\left(\partial X,X\right)$.
		Then $B^{\dagger}\in\hat{\Psi}_{0\tr}^{-\infty,\mathcal{F}}\left(\partial X,X\right)$,
		where
		\begin{align*}
			\mathcal{F}{}_{\of} & :=\overline{\mathcal{E}_{\of}}-2\delta-1\\
			\mathcal{F}_{\ff} & :=\overline{\mathcal{E}_{\ff}}-2\delta.
		\end{align*}
		\item Let $P\in\hat{\Psi}_{0}^{-\infty,\mathcal{E}}\left(X\right)$. Then
		$P^{\dagger}\in\hat{\Psi}_{0}^{-\infty,\mathcal{F}}\left(X\right)$,
		where
		\begin{align*}
			\mathcal{F}_{\lf} & =\overline{\mathcal{E}_{\rf}}+2\delta+1\\
			\mathcal{F}_{\rf} & =\overline{\mathcal{E}_{\lf}}-2\delta-1\\
			\mathcal{F}_{\ff_{0}} & =\overline{\mathcal{E}_{\ff_{0}}}.
		\end{align*}
		\item Let $P\in\hat{\Psi}_{0b}^{-\infty,\mathcal{E}}\left(X\right)$. Then
		$P^{\dagger}\in\hat{\Psi}_{0b}^{-\infty,\mathcal{F}}\left(X\right)$,
		where
		\begin{align*}
			\mathcal{F}_{\lf} & =\overline{\mathcal{E}_{\rf}}+2\delta+1\\
			\mathcal{F}_{\rf} & =\overline{\mathcal{E}_{\lf}}-2\delta-1\\
			\mathcal{F}_{\ff_{0}} & =\overline{\mathcal{E}_{\ff_{0}}}\\
			\mathcal{F}_{\ff_{b}} & =\overline{\mathcal{E}_{\ff_{b}}}.
		\end{align*}
	\end{enumerate}
\end{prop}

\begin{proof}
	We exemplify the argument for the $0$-trace case, since all the other
	cases are proved similarly. Let $A\in\hat{\Psi}_{0\tr}^{-\infty,\mathcal{E}}\left(X,\partial X\right)$.
	Choose again an atlas $\left(U_{i},\varphi_{i}\right)$ for $\partial X$,
	and decompose as above $A=\sum_{i}A_{i}+R_{A}$ where $R_{A}\in\Psi_{\tr}^{-\infty,\mathcal{E}_{\of}}\left(X,\partial X\right)$
	and the $A_{i}$ are compactly supported in an open $U_{i}\times[0,\varepsilon)\times U_{i}$
	of $\partial X\times X$. It is straightforward to check that $R_{A}^{\dagger}\in\Psi_{\tr}^{-\infty,\mathcal{E}_{\of}}\left(X,\partial X\right)$,
	and in the chosen coordinates we have $A_{i}\in\hat{\Psi}_{0\tr,\mathcal{S}}^{-\infty,\mathcal{E}}\left(\mathbb{R}_{1}^{n+1},\mathbb{R}^{n}\right)$.
	It suffices to prove that the formal adjoint $A_{i}^{*}$ with respect
	to the density $x^{-2\delta-1}dxdy$ on $\mathbb{R}_{1}^{n+1}$ and
	the density $dy$ on $\mathbb{R}^{n}$ is in $\hat{\Psi}_{0\po,\mathcal{S}}^{-\infty,\mathcal{F}}\left(\mathbb{R}^{n},\mathbb{R}_{1}^{n+1}\right)$.
	As in the standard case, one construct the right trace and Poisson
	quantization maps, and one proves that the left and right quantization
	maps have the same range. We omit the proof because it is essentially
	equal to the standard one; the interested reader can easily reconstruct
	it by following the standard proof and using Lemma \ref{lem:differentiation-of-untwisted-symbols}.
	Now, if $A=\Op_{L}^{\tr}\left(a\right)\in\hat{\Psi}_{0\tr,\mathcal{S}}^{-\infty,\mathcal{E}}\left(\mathbb{R}_{1}^{n+1},\mathbb{R}^{n}\right)$
	for some $a=a\left(y;\tilde{x},\eta\right)\in S_{0\tr,\mathcal{S}}^{-\infty,\mathcal{E}}\left(\mathbb{R}^{n};\mathbb{R}_{1}^{n+1}\right)$,
	and we write the Schwartz kernel of $A$ as $A\left(y;\tilde{x},y-\tilde{y}\right)d\tilde{x}d\tilde{y}$
	so that $A\left(y;\tilde{x},Y\right)$ is the inverse Fourier transform
	of $a\left(y;\tilde{x},\eta\right)$, then the Schwartz kernel of
	$A^{\dagger}$ is $x^{2\delta+1}\overline{A\left(\tilde{y};x,\tilde{y}-y\right)}d\tilde{y}$.
	Therefore, $A^{\dagger}=\Op_{R}^{\po}\left(a^{\dagger}\right)$ ,
	where $a^{\dagger}\left(\tilde{y};x,\eta\right)=x^{2\delta+1}\overline{a\left(\tilde{y};x,\eta\right)}$.
	Since
	\begin{align*}
		a\left(y;\tilde{x},\eta\right) & \in S_{0\tr,\mathcal{S}}^{-\infty,\mathcal{E}}\left(\mathbb{R}^{n};\mathbb{R}_{1}^{n+1}\right)\\
		& =\mathcal{S}\left(\mathbb{R}^{n}\right)\hat{\otimes}\mathcal{A}_{\phg}^{\left(\mathcal{E}_{\of},\mathcal{E}_{\ff}-1,\infty,\infty\right)}\left(\hat{T}_{0}^{2}\right),
	\end{align*}
	it follows that $\overline{a\left(\tilde{y};x,\eta\right)}\in\mathcal{S}\left(\mathbb{R}^{n}\right)\hat{\otimes}\mathcal{A}_{\phg}^{\left(\overline{\mathcal{E}_{\of}},\overline{\mathcal{E}_{\ff}}-1,\infty,\infty\right)}\left(\hat{P}_{0}^{2}\right)$.
	Finally, $x$ is polyhomogeneous on $\hat{P}^{2}$ with index sets
	$\left(1,0,-1\right)$, so its lift to $\hat{P}_{0}^{2}$ has index
	sets $\left(1,1,0,-1\right)$. Therefore, 
	\begin{align*}
		x^{2\delta+1}\overline{a\left(\tilde{y};x,\eta\right)} & \in\mathcal{S}\left(\mathbb{R}^{n}\right)\hat{\otimes}\mathcal{A}_{\phg}^{\left(\overline{\mathcal{E}_{\of}}+2\delta+1,\overline{\mathcal{E}_{\ff}}+2\delta,\infty,\infty\right)}\left(\hat{P}_{0}^{2}\right)\\
		& =S_{0\po,\mathcal{S}}^{-\infty,\mathcal{F}}\left(\mathbb{R}^{n};\mathbb{R}_{1}^{n+1}\right).
	\end{align*}
	This concludes the proof.
\end{proof}
Let's now consider the twisted case.
\begin{prop}
	\label{prop:twisted-adjoint}Let $\boldsymbol{E},\boldsymbol{F}$
	be smooth vector bundles on $\partial X$, and let $\boldsymbol{\mathfrak{s}},\boldsymbol{\mathfrak{t}}$
	be smooth endomorphisms of $\boldsymbol{E},\boldsymbol{F}$ with constant
	eigenvalues. Choose Hermitian metrics on $\boldsymbol{E},\boldsymbol{F}$.
	Denote by $\boldsymbol{\mathfrak{s}}^{*},\boldsymbol{\mathfrak{t}}^{*}$
	the Hermitian adjoints of the operators above with respect to the
	metrics.
	\begin{enumerate}
		\item Let $\boldsymbol{A}\in\hat{\Psi}_{0\tr}^{-\infty,\left(\mathcal{E}_{\of},\left[\mathcal{E}_{\ff}\right]\right),\boldsymbol{\mathfrak{s}}}\left(X;\partial X,\boldsymbol{E}\right)$.
		Then $\boldsymbol{A}^{\dagger}\in\hat{\Psi}_{0\po}^{-\infty,\left(\mathcal{F}_{\of},\left[\mathcal{F}_{\ff}\right]\right),-\boldsymbol{\mathfrak{s}}^{*}}\left(\partial X,\boldsymbol{E};X\right)$,
		where
		\begin{align*}
			\mathcal{F}{}_{\of} & :=\overline{\mathcal{E}_{\of}}+2\delta+1\\
			\left[\mathcal{F}_{\ff}\right] & :=\left[\overline{\mathcal{E}_{\ff}}\right]+2\delta.
		\end{align*}
		\item Let $\boldsymbol{B}\in\hat{\Psi}_{0\po}^{-\infty,\left(\mathcal{E}_{\of},\left[\mathcal{E}_{\ff}\right]\right),\boldsymbol{\mathfrak{s}}}\left(\partial X,\boldsymbol{E};X\right)$.
		Then $\boldsymbol{B}^{\dagger}\in\hat{\Psi}_{0\tr}^{-\infty,\left(\mathcal{F}_{\of},\left[\mathcal{F}_{\ff}\right]\right),-\boldsymbol{\mathfrak{s}}^{*}}\left(X;\partial X,\boldsymbol{E}\right)$,
		where
		\begin{align*}
			\mathcal{F}{}_{\of} & :=\overline{\mathcal{E}_{\of}}-2\delta-1\\
			\left[\mathcal{F}_{\ff}\right] & :=\left[\overline{\mathcal{E}_{\ff}}\right]-2\delta.
		\end{align*}
		\item Let $\boldsymbol{Q}\in\Psi_{\phg}^{-\left[\mathcal{E}\right],\left(\boldsymbol{\mathfrak{s}},\boldsymbol{\mathfrak{t}}\right)}\left(\partial X;\boldsymbol{E},\boldsymbol{F}\right)$.
		Then $\boldsymbol{Q}^{*}\in\Psi_{\phg}^{-\left[\overline{\mathcal{E}}\right],\left(-\boldsymbol{\mathfrak{t}}^{*},-\boldsymbol{\mathfrak{s}}^{*}\right)}\left(\partial X;\boldsymbol{F},\boldsymbol{E}\right)$.
	\end{enumerate}
\end{prop}

\begin{proof}
	Again, let's only prove the $0$-trace case for simplicity. As in
	the proof of the un-twisted case, it suffices to prove the local version
	of the statement. Moreover, decomposing the operator in terms of the
	generalized eigenbundle decomposition of $\boldsymbol{\mathfrak{s}}$,
	we can assume that the twisting endomorphism has only one eigenvalue.
	Incorporating the eigenvalue in the index set at the front face, we
	can assume without loss of generality that the endomorphism is nilpotent.
	Finally, up to rescaling the operator, we can also assume that $\delta=-1/2$.
	Thus, we are left to prove the following statement: given \[A=\Op_{L}^{\tr}\left(\left\langle \eta\right\rangle ^{-\mathfrak{s}\left(y\right)}a\left(y;\tilde{x},\eta\right)\right)\in\hat{\Psi}_{0\tr,\mathcal{S}}^{-\infty,\left(\mathcal{E}_{\of},\mathcal{E}_{\ff}\right),\mathfrak{s}}\left(\mathbb{R}_{1}^{n+1};\mathbb{R}^{n},\mathbb{C}^{N}\right)\]
	where $a\left(y;\tilde{x},\eta\right)\in S_{0\tr,\mathcal{S}}^{-\infty,\left(\mathcal{E}_{\of},\mathcal{E}_{\ff}\right)}\left(\mathbb{R}^{n};\mathbb{R}_{1}^{n+1};\mathbb{C}^{N}\right)$,
	and $\mathfrak{s}:\overline{\mathbb{R}^{n}}\to\GL\left(N,\mathbb{C}\right)$
	a smooth family of nilpotent matrices, the formal adjoint $A^{*}$
	with respect to the density $dxdy$ on $\mathbb{R}_{1}^{n+1}$ and
	the density $dy$ on $\mathbb{R}^{n}$ is an element of $\hat{\Psi}_{0\po,\mathcal{S}}^{-\infty,\left(\overline{\mathcal{E}_{\of}},\mathcal{F}_{\ff}\right),-\mathfrak{s}^{*}}\left(\mathbb{R}^{n},\mathbb{C}^{N};\mathbb{R}_{1}^{n+1}\right)$
	where $\mathcal{F}_{\ff}$ is an index set such that $\left[\mathcal{F}_{\ff}\right]=\left[\overline{\mathcal{E}_{\ff}}\right]-1$.
	
	Call $\tilde{a}\left(y;\tilde{x},\eta\right)=\left\langle \eta\right\rangle ^{-\mathfrak{s}\left(y\right)}a\left(y;\tilde{x},\eta\right)$.
	We know from Lemma \ref{lem:twisting-poisson-trace-are-polyhomogeneous-1}
	that \[\tilde{a}\in S_{0\tr,\mathcal{S}}^{-\infty,\left(\mathcal{E}_{\of},\mathcal{E}_{\ff}'\right)}\left(\mathbb{R}^{n};\mathbb{R}_{1}^{n+1};\mathbb{C}^{N}\right),\]
	where $\mathcal{E}_{\ff}'$ is an index set obtained from $\mathcal{E}_{\ff}$
	by increasing some of the logarithmic orders. From the proof of the
	previous proposition, we know that $A^{*}=\Op_{R}^{\po}\left(\tilde{b}\left(\tilde{y};x,\eta\right)\right)$,
	where 
	\[
	\tilde{b}\left(\tilde{y};x,\eta\right)=\left(\tilde{a}\left(\tilde{y};x,\eta\right)\right)^{*}\in S_{0\po,\mathcal{S}}^{-\infty,\left(\overline{\mathcal{E}_{\of}},\overline{\mathcal{E}_{\ff}}'-1\right)}\left(\mathbb{R}^{n};\mathbb{R}_{1}^{n+1};\left(\mathbb{C}^{N}\right)^{*}\right)
	\]
	and the $*$ on the right denotes transpose conjugate. Now, write
	$\tilde{b}_{L}\left(y;x,\eta\right)$ for the left reduction of $\tilde{b}\left(\tilde{y};x,\eta\right)$.
	Analogously to the standard case of symbols in $\mathbb{R}^{n}$,
	$\tilde{b}_{L}\left(y;x,\eta\right)$ is determined modulo $S_{\po}^{-\infty,\mathcal{E}_{\of}}\left(\mathbb{R}^{n};\mathbb{R}_{1}^{n+1};\left(\mathbb{C}^{N}\right)^{*}\right)$
	as an asymptotic sum
	\[
	\tilde{b}_{L}\left(y;x,\eta\right)\sim\sum_{\alpha}\frac{1}{\alpha!}\left(D_{\eta}^{\alpha}\partial_{\tilde{y}}^{\alpha}\tilde{b}\right)_{|\tilde{y}=y}.
	\]
	Now, we have
	\begin{align*}
		\tilde{b}\left(\tilde{y};x,\eta\right) & =\left(\tilde{a}\left(\tilde{y};x,\eta\right)\right)^{*}\\
		& =a\left(\tilde{y};x,\eta\right)^{*}\left\langle \eta\right\rangle ^{-\mathfrak{s}^{*}\left(\tilde{y}\right)}\\
		& \in S_{0\po,\mathcal{S}}^{-\infty,\left(\overline{\mathcal{E}_{\of}},\overline{\mathcal{E}_{\ff}}-1\right),-\mathfrak{s}^{*}}\left(\mathbb{R}^{n};\mathbb{R}_{1}^{n+1};\left(\mathbb{C}^{N}\right)^{*}\right).
	\end{align*}
	Therefore, by Lemma \ref{lem:differentiation-of-twisted-symbols},
	we have for $\left|\alpha\right|>0$ 
	\[
	\frac{1}{\alpha!}\left(D_{\eta}^{\alpha}\partial_{\tilde{y}}^{\alpha}\tilde{b}\right)_{|\tilde{y}=y}\in S_{0\po,\mathcal{S}}^{-\infty,\left(\overline{\mathcal{E}_{\of}},\overline{\mathcal{E}_{\ff}}+\left|\alpha\right|-1\right),-\mathfrak{s}^{*}}\left(\mathbb{R}^{n};\mathbb{R}_{1}^{n+1};\left(\mathbb{C}^{N}\right)^{*}\right),
	\]
	where $\overline{\mathcal{E}_{\ff}}$ is an index set obtained from
	$\overline{\mathcal{E}_{\ff}}$ by increasing some of the logarithmic
	orders. Note that the union
	\[
	\mathcal{F}_{\ff}=\overline{\mathcal{E}_{\ff}}\bigcup_{\left|\alpha\right|>0}\left(\overline{\mathcal{E}_{\ff}}+\left|\alpha\right|-1\right)
	\]
	is an index set, because only finitely many log terms can be added
	at each level. Moreover, we have $\left[\mathcal{F}_{\ff}\right]=\left[\overline{\mathcal{E}_{\ff}}\right]-1$.
	By asymptotic completeness, we have
	\[
	\tilde{b}_{L}\in S_{0\po,\mathcal{S}}^{-\infty,\left(\overline{\mathcal{E}_{\of}},\mathcal{F}_{\ff}\right),-\mathfrak{s}^{*}}\left(\mathbb{R}^{n};\mathbb{R}_{1}^{n+1};\left(\mathbb{C}^{N}\right)^{*}\right)
	\]
	which implies that
	\[
	A^{*}\in\hat{\Psi}_{0\po,\mathcal{S}}^{-\infty,\left(\overline{\mathcal{E}_{\of}},\mathcal{F}_{\ff}\right),-\mathfrak{s}^{*}}\left(\mathbb{R}^{n},\mathbb{C}^{N};\mathbb{R}_{1}^{n+1}\right)
	\]
	as claimed.
\end{proof}

\subsection{\label{subsec:Composition-theorems}Composition theorems}

We now discuss compositions results. As in the previous sub-sections,
we imitate the standard proof of the composition theorem for pseudodifferential
operators on a closed manifold, i.e.\ the fact that
\[
\Psi^{m}\left(\partial X\right)\circ\Psi^{m'}\left(\partial X\right)\subseteq\Psi^{m+m'}\left(\partial X\right).
\]
Let's briefly review the proof. Let $A\in\Psi^{m}\left(\partial X\right)$
and $B\in\Psi^{m'}\left(\partial X\right)$. Choose an atlas $\left(U_{i},\varphi_{i}\right)$
such that, for every pair $i,j$ for which $U_{i}\cap U_{j}\not=\emptyset$,
the union $U_{i}\cup U_{j}$ is contained in the domain $\tilde{U}_{ij}$
of a chart $\tilde{\varphi}_{ij}$. Now, decompose $A=\sum_{i}A_{i}+R_{A}$
and $B=\sum_{i}B_{i}+R_{B}$, where $A_{i},B_{i}$ are compactly supported
in $U_{i}\times U_{i}$, and $R_{A},R_{B}\in\Psi^{-\infty}\left(\partial X\right)$.
Then
\[
AB=\sum_{ij}A_{i}B_{j}+\sum_{i}\left(A_{i}R_{B}+R_{A}B_{i}\right)+R_{A}R_{B}.
\]
It is straightforward to prove that $R_{A}R_{B}\in\Psi^{-\infty}\left(\partial X\right)$.
It is also easy to check that $A_{i}R_{B}\in\Psi^{-\infty}\left(\partial X\right)$:
since $A_{i}$ acts continuously as a map $A_{i}:C^{\infty}\left(\partial X\right)\to C^{\infty}\left(\partial X\right)$,
and since \[\Psi^{-\infty}\left(\partial X\right) = C^{\infty}\left(\partial X\right)\hat{\otimes}C^{\infty}\left(\partial X;\mathcal{D}_{\partial X}^{1}\right),\]
the action of $A_{i}$ on the left factor determines a continuous
linear map $A_{i}:\Psi^{-\infty}\left(\partial X\right)\to\Psi^{-\infty}\left(\partial X\right)$.
To prove that $R_{A}B_{i}\in\Psi^{-\infty}\left(\partial X\right)$
as well, we observe that $R_{A}B_{i}=\left(B_{i}^{*}R_{A}^{*}\right)^{*}$;
since $B_{i}^{*}\in\Psi^{m}\left(\partial X\right)$ and $R_{A}^{*}\in\Psi^{-\infty}\left(\partial X\right)$,
$B_{i}^{*}R_{A}^{*}\in\Psi^{-\infty}\left(\partial X\right)$ and
therefore $R_{A}B_{i}\in\Psi^{-\infty}\left(\partial X\right)$ as
well. Finally it remains to prove that $A_{i}B_{j}\in\Psi^{m+m'}\left(\partial X\right)$.
Now, the point is that if $U_{i}$ and $U_{j}$ are disjoint, then
$A_{i}B_{j}=0$, while if $U_{i}\cap U_{j}\not=\emptyset$, we can
think of $A_{i}$, $B_{j}$ in coordinates induced by the chart $\left(\tilde{U}_{ij},\tilde{\varphi}_{ij}\right)$
as operators $A_{i}\in\Psi_{\mathcal{S}}^{m}\left(\mathbb{R}^{n}\right)$,
$B_{j}\in\Psi_{\mathcal{S}}^{m'}\left(\mathbb{R}^{n}\right)$. Thus,
it suffices to prove that $\Psi_{\mathcal{S}}^{m}\left(\mathbb{R}^{n}\right)\circ\Psi_{\mathcal{S}}^{m'}\left(\mathbb{R}^{n}\right)\subseteq\Psi_{\mathcal{S}}^{m+m'}\left(\mathbb{R}^{n}\right)$.

To prove this statement, we use the \emph{full quantization map}
\begin{align*}
	\Op:S_{\mathcal{S}}^{m}\left(\mathbb{R}^{2n};\mathbb{R}^{n}\right) & \to\Op\left(\mathbb{R}^{n}\right)\\
	q\left(y,\tilde{y};\eta\right) & \mapsto\frac{1}{\left(2\pi\right)^{n}}\int e^{i\left(y-\tilde{y}\right)\eta}q\left(y,\tilde{y};\eta\right)d\eta d\tilde{y}.
\end{align*}
Analogously to the right quantization map used in \S\ref{subsec:Formal-adjoints},
to properly interpret this map we decompose $S_{\mathcal{S}}^{m}\left(\mathbb{R}^{2n};\mathbb{R}^{n}\right)=\mathcal{S}\left(\mathbb{R}^{n}\right)\hat{\otimes}\mathcal{S}\left(\mathbb{R}^{n}\right)\hat{\otimes}\mathcal{A}^{-m}\left(\overline{\mathbb{R}^{n}}\right)$
and then we observe that the map
\[
\left(a\left(y\right),b\left(\tilde{y}\right),q\left(\eta\right),u\left(y\right)\right)\mapsto a\left(y\right)\frac{1}{\left(2\pi\right)^{n}}\int e^{iy\eta}\widehat{b\cdot u}\left(\eta\right)d\eta
\]
is well-defined and continuous multilinear from $\mathcal{S}\left(\mathbb{R}^{n}\right)\times\mathcal{S}\left(\mathbb{R}^{n}\right)\times\mathcal{A}^{-m}\left(\overline{\mathbb{R}^{n}}\right)\times\mathcal{S}\left(\mathbb{R}^{n}\right)$
to $\mathcal{S}\left(\mathbb{R}^{n}\right)$, so from the universal
property of $\hat{\otimes}$ it induces a continuous linear map \[S_{\mathcal{S}}^{m}\left(\mathbb{R}^{2n};\mathbb{R}^{n}\right)\to\Op\left(\mathbb{R}^{n}\right.\]
Again, one can prove that $\Op\left(S_{\mathcal{S}}^{m}\left(\mathbb{R}^{2n};\mathbb{R}^{n}\right)\right)$
coincides with $\Psi_{\mathcal{S}}^{m}\left(\mathbb{R}^{n}\right)$,
i.e.\ for every full symbol $q\left(y,\tilde{y};\eta\right)$ in $S_{\mathcal{S}}^{m}\left(\mathbb{R}^{2n};\mathbb{R}^{n}\right)$
there exists a unique left-reduced symbol $q_{L}\left(y;\eta\right)$ in the space $S_{\mathcal{S}}^{m}\left(\mathbb{R}^{n};\mathbb{R}^{n}\right)$
such that $\Op\left(q\right)=\Op_{L}\left(q_{L}\right)$. Now, if
$A\in\Psi_{\mathcal{S}}^{m}\left(\mathbb{R}^{n}\right)$ and $B\in\Psi_{\mathcal{S}}^{m'}\left(\mathbb{R}^{n}\right)$,
we can write $A=\Op_{L}\left(a\right)$ and $B=\Op_{R}\left(b\right)$
for a left-reduced symbol $a\left(y;\eta\right)$ and a right-reduced
symbol $b\left(\tilde{y};\eta\right)$. By definition of $\Op_{L}$
and $\Op_{R}$, we have
\begin{align*}
	ABu & =\frac{1}{\left(2\pi\right)^{n}}\int e^{iy\eta}a\left(y;\eta\right)\widehat{Bu}\left(\eta\right)d\eta\\
	& =\frac{1}{\left(2\pi\right)^{n}}\int e^{i\left(y-\tilde{y}\right)\eta}a\left(y;\eta\right)b\left(\tilde{y};\eta\right)u\left(\tilde{y}\right)d\tilde{y}d\eta.
\end{align*}
The proof is concluded if we can show that $a\left(y;\eta\right)b\left(\tilde{y};\eta\right)\in S_{\mathcal{S}}^{m+m'}\left(\mathbb{R}^{2n};\mathbb{R}^{n}\right)$:
using the definition $S_{\mathcal{S}}^{m}\left(\mathbb{R}^{k};\mathbb{R}^{n}\right)=\mathcal{S}\left(\mathbb{R}^{k}\right)\hat{\otimes}\mathcal{A}^{-m}\left(\overline{\mathbb{R}^{n}}\right)$,
this follows immediately by observing that the multiplication map
\begin{align*}
	\mathcal{A}^{-m}\left(\overline{\mathbb{R}^{n}}\right)\times\mathcal{A}^{-m'}\left(\overline{\mathbb{R}^{n}}\right) & \to\mathcal{A}^{-\left(m+m'\right)}\left(\overline{\mathbb{R}^{n}}\right)\\
	\left(a\left(\eta\right),b\left(\eta\right)\right) & \mapsto a\left(\eta\right)b\left(\eta\right)
\end{align*}
is well-defined and continuous. It is often useful to express all
the operators $A,B$ and $C=AB$ as left quantizations: if $A=\Op_{L}\left(a\right)$
and $B=\Op_{L}\left(b\right)$, then the composition $C=AB$ is expressed
as a left quantization $C=\Op_{L}\left(c\right)$, where the symbol
$c\left(y;\eta\right)$ is uniquely determined modulo $S_{\mathcal{S}}^{-\infty}\left(\mathbb{R}^{n};\mathbb{R}^{n}\right)$
as an asymptotic sum
\[
c\left(y;\eta\right)\sim\left[\sum_{\alpha}\frac{1}{\alpha!}\left(\partial_{\eta}^{\alpha}a\right)\left(D_{\tilde{y}}^{\alpha}b\right)\right]_{|\tilde{y}=y}.
\]

\subsubsection{\label{subsubsec:Local-composition-theorems-untwisted}Local composition
	theorems for un-twisted operators}

Let us first prove the local versions of the composition results which
do not involve twisted symbolic $0$-trace, $0$-Poisson and boundary
operators. We prove three theorems:
\begin{thm}
	\label{thm:compositions-involving-boundary}Let $\mathcal{E}=(\mathcal{E}_{\of},\mathcal{E}_{\ff})$,
	$\mathcal{F}=(\mathcal{F}_{\of},\mathcal{F}_{\ff})$ and
	$\mathcal{G}$ be index sets. Let
	\begin{equation*}
		\begin{aligned}
			A &\in \hat{\Psi}_{0\tr,\mathcal{S}}^{-\infty,\mathcal{F}}
			\left(\mathbb{R}_{1}^{n+1},\mathbb{R}^{n}\right),\\
			B &\in \hat{\Psi}_{0\po,\mathcal{S}}^{-\infty,\mathcal{E}}
			\left(\mathbb{R}^{n},\mathbb{R}_{1}^{n+1}\right),\\
			Q &\in \Psi_{\phg,\mathcal{S}}^{-\mathcal{G}}
			\left(\mathbb{R}^{n}\right).
		\end{aligned}
	\end{equation*}
	Then:
	\begin{enumerate}
		\item $BQ$ is well-defined and in $\hat{\Psi}_{0\po,\mathcal{S}}^{-\infty,\left(\mathcal{E}_{\of},\mathcal{E}_{\ff}+\mathcal{G}\right)}\left(\mathbb{R}^{n},\mathbb{R}_{1}^{n+1}\right)$;
		\item $QA$ is well-defined and in $\hat{\Psi}_{0\tr,\mathcal{S}}^{-\infty,\left(\mathcal{F}_{\of},\mathcal{F}_{\ff}+\mathcal{G}\right)}\left(\mathbb{R}_{1}^{n+1},\mathbb{R}^{n}\right)$;
		\item $BA$ is well-defined and in $\hat{\Psi}_{0,\mathcal{S}}^{-\infty,\left(\mathcal{E}_{\of},\mathcal{F}_{\of},\mathcal{E}_{\ff}+\mathcal{F}_{\ff}\right)}\left(\mathbb{R}_{1}^{n+1}\right)$;
		\item if $\Re\left(\mathcal{F}_{\of}+\mathcal{E}_{\of}\right)>-1$, then
		$AB$ is well-defined and in $\Psi_{\phg,\mathcal{S}}^{-\left(\mathcal{E}_{\ff}+\mathcal{F}_{\ff}\right)}\left(\mathbb{R}^{n}\right)$.
	\end{enumerate}
\end{thm}

\begin{thm}
	\label{thm:compositions-mixed}Let $\mathcal{E}=\left(\mathcal{E}_{\lf},\mathcal{E}_{\rf},\mathcal{E}_{\ff_{b}},\mathcal{E}_{\ff_{0}}\right)$
	and $\mathcal{F}=\left(\mathcal{F}_{\of},\mathcal{F}_{\ff}\right)$
	be index sets. Define also $\mathcal{E}'=\left(\mathcal{E}_{\lf},\mathcal{E}_{\rf},\mathcal{E}_{\ff_{0}}\right)$.
	Let $A\in\hat{\Psi}_{0\tr,\mathcal{S}}^{-\infty,\mathcal{F}}\left(\mathbb{R}_{1}^{n+1},\mathbb{R}^{n}\right)$,
	$B\in\hat{\Psi}_{0\po,\mathcal{S}}^{-\infty,\mathcal{F}}\left(\mathbb{R}^{n},\mathbb{R}_{1}^{n+1}\right)$,
	$P_{0}\in\hat{\Psi}_{0,\mathcal{S}}^{-\infty,\mathcal{E}'}\left(\mathbb{R}_{1}^{n+1}\right)$
	and $P_{0b}\in\hat{\Psi}_{0b,\mathcal{S}}^{-\infty,\mathcal{E}}\left(\mathbb{R}_{1}^{n+1}\right)$.
	\begin{enumerate}
		\item \textup{\emph{If $\Re\left(\mathcal{E}_{\rf}+\mathcal{F}_{\of}\right)>-1$,
				then $P_{0}B$ is well-defined and in \[\hat{\Psi}_{0\po,\mathcal{S}}^{-\infty,\left(\mathcal{E}_{\lf},\mathcal{F}_{\ff}+\mathcal{E}_{\ff_{0}}\right)}\left(\mathbb{R}^{n},\mathbb{R}_{1}^{n+1}\right).\]}}
		\item \textup{\emph{If $\Re\left(\mathcal{E}_{\rf}+\mathcal{F}_{\of}\right)>-1$,
				then $P_{0b}B$ is well-defined and in \[\hat{\Psi}_{0\po,\mathcal{S}}^{-\infty,\left(\mathcal{E}_{\lf}\eunion\left(\mathcal{F}_{\of}+\mathcal{E}_{\ff_{b}}\right),\mathcal{F}_{\ff}+\mathcal{E}_{\ff_{0}}\right)}\left(\mathbb{R}^{n},\mathbb{R}_{1}^{n+1}\right).\]}}
		\item \textup{\emph{If $\Re\left(\mathcal{F}_{\of}+\mathcal{E}_{\lf}\right)>-1$,
				then $AP_{0}$ is well-defined and in \[\hat{\Psi}_{0\tr,\mathcal{S}}^{-\infty,\left(\mathcal{E}_{\rf},\mathcal{F}_{\ff}+\mathcal{E}_{\ff_{0}}\right)}\left(\mathbb{R}_{1}^{n+1},\mathbb{R}^{n}\right).\]}}
		\item \textup{\emph{If $\Re\left(\mathcal{F}_{\of}+\mathcal{E}_{\lf}\right)>-1$,
				then $AP_{0b}$ is well-defined and in \[\hat{\Psi}_{0\tr,\mathcal{S}}^{-\infty,\left(\mathcal{E}_{\rf}\eunion\left(\mathcal{F}_{\of}+\mathcal{E}_{\ff_{b}}\right),\mathcal{F}_{\ff}+\mathcal{E}_{\ff_{0}}\right)}\left(\mathbb{R}_{1}^{n+1},\mathbb{R}^{n}\right).\]}}
	\end{enumerate}
\end{thm}

\begin{thm}
	\label{thm:compositions-involving-interior}Let $\mathcal{E}=\left(\mathcal{E}_{\lf},\mathcal{E}_{\rf},\mathcal{E}_{\ff_{b}},\mathcal{E}_{\ff_{0}}\right)$\textup{\emph{
			and}} $\mathcal{F}=\left(\mathcal{F}_{\lf},\mathcal{F}_{\rf},\mathcal{F}_{\ff_{b}},\mathcal{F}_{\ff_{0}}\right)$\textup{\emph{
			be index sets. Define also }}$\mathcal{E}'=\left(\mathcal{E}_{\lf},\mathcal{E}_{\rf},\mathcal{E}_{\ff_{0}}\right)$
	and $\mathcal{F}'=\left(\mathcal{F}_{\lf},\mathcal{F}_{\rf},\mathcal{F}_{\ff_{0}}\right)$.
	\textup{\emph{Define $\mathcal{G}=\left(\mathcal{G}_{\lf},\mathcal{G}_{\rf},\mathcal{G}_{\ff_{b}},\mathcal{G}_{\ff_{0}}\right)$
			and $\mathcal{G}'=\left(\mathcal{G}_{\lf},\mathcal{G}_{\rf},\mathcal{G}_{\ff_{0}}\right)$,
			with
			\begin{align*}
				\mathcal{G}_{\lf} & =\mathcal{E}_{\lf}\eunion\left(\mathcal{F}_{\lf}+\mathcal{E}_{\ff_{b}}\right)\\
				\mathcal{G}_{\rf} & =\mathcal{F}_{\rf}\eunion\left(\mathcal{E}_{\rf}+\mathcal{F}_{\ff_{b}}\right)\\
				\mathcal{G}_{\ff_{b}} & =\left(\mathcal{E}_{\ff_{b}}+\mathcal{F}_{\ff_{b}}\right)\eunion\left(\mathcal{E}_{\lf}+\mathcal{F}_{\rf}+1\right)\\
				\mathcal{G}_{\ff_{0}} & =\mathcal{E}_{\ff_{0}}+\mathcal{F}_{\ff_{0}}.
			\end{align*}
	}}Let $P_{0}\in\hat{\Psi}_{0,\mathcal{S}}^{-\infty,\mathcal{E}'}\left(\mathbb{R}_{1}^{n+1}\right)$,
	$P_{0b}\in\hat{\Psi}_{0b,\mathcal{S}}^{-\infty,\mathcal{E}}\left(\mathbb{R}_{1}^{n+1}\right)$,
	$Q_{0}\in\hat{\Psi}_{0,\mathcal{S}}^{-\infty,\mathcal{F}'}\left(\mathbb{R}_{1}^{n+1}\right)$,
	$Q_{0b}\in\hat{\Psi}_{0b,\mathcal{S}}^{-\infty,\mathcal{F}}\left(\mathbb{R}_{1}^{n+1}\right)$.\textup{\emph{
			If $\Re\left(\mathcal{E}_{\rf}+\mathcal{F}_{\lf}\right)>-1$, then:}}
	\begin{enumerate}
		\item \textup{\emph{$P_{0b}Q_{0b}$ is well-defined and in $\hat{\Psi}_{0b,\mathcal{S}}^{-\infty,\mathcal{G}}\left(\mathbb{R}_{1}^{n+1}\right)$;}}
		\item \textup{\emph{$P_{0b}Q_{0}$ is well-defined and in $\hat{\Psi}_{0,\mathcal{S}}^{-\infty,\left(\mathcal{G}_{\lf},\mathcal{F}_{\rf},\mathcal{G}_{\ff_{0}}\right)}\left(\mathbb{R}_{1}^{n+1}\right)$;}}
		\item \textup{\emph{$P_{0}Q_{0b}$ is well-defined and in $\hat{\Psi}_{0,\mathcal{S}}^{-\infty,\left(\mathcal{E}_{\lf},\mathcal{G}_{\rf},\mathcal{G}_{\ff_{0}}\right)}\left(\mathbb{R}_{1}^{n+1}\right)$;}}
		\item \textup{\emph{$P_{0}Q_{0}$ is well-defined and in $\hat{\Psi}_{0,\mathcal{S}}^{-\infty,\left(\mathcal{E}_{\lf},\mathcal{F}_{\rf},\mathcal{G}_{\ff_{0}}\right)}\left(\mathbb{R}_{1}^{n+1}\right)$.}}
	\end{enumerate}
\end{thm}

The basic strategy to prove all these results is the same as in the
standard case. Analogously to the case of symbols in $\mathbb{R}^{n}$,
for each operator type (Poisson, trace, interior) we can define a
``full quantization map'': formally, these full quantization maps
are defined as
\begin{align*}
	\Op^{\tr}\left(a\left(y,\tilde{y};\tilde{x},\eta\right)\right) & =\left[\frac{1}{\left(2\pi\right)^{n}}\int e^{i\left(y-\tilde{y}\right)\eta}a\left(y,\tilde{y};\tilde{x},\eta\right)d\eta\right]d\tilde{x}d\tilde{y}\\
	\Op^{\po}\left(p\left(y,\tilde{y};x,\eta\right)\right) & =\left[\frac{1}{\left(2\pi\right)^{n}}\int e^{i\left(y-\tilde{y}\right)\eta}b\left(y,\tilde{y};x,\eta\right)d\eta\right]d\tilde{y}\\
	\Op^{\inte}\left(p\left(y,\tilde{y};x,\tilde{x},\eta\right)\right) & =\left[\frac{1}{\left(2\pi\right)^{n}}\int e^{i\left(y-\tilde{y}\right)\eta}p\left(y,\tilde{y};x,\tilde{x},\eta\right)d\eta\right]d\tilde{x}d\tilde{y}.
\end{align*}
One can prove, exactly as in the standard case, that the space of
quantizations of full symbols (of a given type, and with given index
sets) coincides with the space of quantizations of left-reduced or
right-reduced symbols (with the same type and index sets). Now, given
two operators $Q_{1},Q_{2}$ in the classes above, we write $Q_{1}=\Op_{L}^{\bullet}\left(q_{1}\right)$
and $Q_{2}=\Op_{R}^{\bullet}\left(q_{2}\right)$ for appropriate left
and right reduced symbols of the appropriate type. Then, if the composition
$Q_{1}Q_{2}$ is well-defined, $Q_{1}Q_{2}$ is the quantization of
the full symbol $q_{1}*q_{2}$. Here $*$ is a bilinear continuous
``contraction map'', which depends on the symbol types of $q_{1}$
and $q_{2}$. For example, if $Q\in\Psi_{\phg,\mathcal{S}}^{-\mathcal{G}}\left(\mathbb{R}^{n}\right)$
and $A\in\hat{\Psi}_{0\tr,\mathcal{S}}^{-\infty,\mathcal{F}}\left(\mathbb{R}^{n}\right)$,
then $Q=\Op_{L}\left(q\left(y;\eta\right)\right)$, $A=\Op_{R}^{\tr}\left(a\left(\tilde{y};\tilde{x},\eta\right)d\tilde{x}\right)$,
and $q*a$ is simply the product $q\left(y;\eta\right)a\left(\tilde{y};\tilde{x},\eta\right)$.
If instead $A\in\hat{\Psi}_{0\tr,\mathcal{S}}^{-\infty,\mathcal{F}}\left(\mathbb{R}^{n}\right)$
and $B\in\hat{\Psi}_{0\po,\mathcal{S}}^{-\infty,\mathcal{E}}\left(\mathbb{R}^{n}\right)$,
then $A=\Op_{L}^{\tr}\left(a\left(y;\tilde{x},\eta\right)\right)$,
$B=\Op_{R}^{\po}\left(b\left(\tilde{y};x,\eta\right)\right)$, and
$a*b$ is the contraction map
\[
c\left(y,\tilde{y};\eta\right)=\int a\left(y;\tilde{x},\eta\right)b\left(\tilde{y};x,\eta\right)d\tilde{x}.
\]
Of course in this case $a*b$ is not always well-defined, but the
integrability condition for the integral above is precisely the integrability
condition that ensures (by the mapping properties proved in \S\ref{subsec:Basic-mapping-properties})
that the composition $AB$ is well-defined. Finally, exactly as in
the standard case, one can prove that if $Q_{1}=\Op_{L}^{\bullet}\left(q_{1}\right)$
and $Q_{2}=\Op_{L}^{\bullet}\left(q_{2}\right)$, and the composition
$Q=Q_{1}Q_{2}$ is well-defined, then the left symbol $q$ for which
$Q=\Op_{L}^{\bullet}\left(q\right)$ is determined up to residual
symbols (of the appropriate type) as an asymptotic sum
\[
q\sim\left[\sum_{\alpha}\frac{1}{\alpha!}\left(\partial_{\eta}^{\alpha}q_{1}\right)*\left(D_{\tilde{y}}^{\alpha}q_{2}\right)\right]_{|\tilde{y}=y}.
\]
Summarizing, in each case we are reduced to prove the well-definedness
of various bilinear ``contraction maps'' between appropriate symbol
spaces. As in the standard case, we can further simplify the argument
by considering symbols with ``no coefficients''.
\begin{proof}
	(Proof of Theorems \ref{thm:compositions-involving-boundary} and
	\ref{thm:compositions-mixed})\textbf{ ($0$-Poisson $\circ$ boundary)
	}First of all, if $B$ belongs to $\hat{\Psi}_{0\po,\mathcal{S}}^{-\infty,\mathcal{E}}\left(\mathbb{R}^{n},\mathbb{R}_{1}^{n+1}\right)$
	and $Q$ is in $\Psi_{\phg,\mathcal{S}}^{-\mathcal{G}}\left(\mathbb{R}^{n}\right)$,
	then we know that $Q$ induces a continuous linear map $Q:\mathcal{S}\left(\mathbb{R}^{n}\right)\to\mathcal{S}\left(\mathbb{R}^{n}\right)$,
	and by the proof of Theorem \ref{thm:mapping-properties-on-phg-untwisted}
	$B$ induces a continuous linear map $\mathcal{S}\left(\mathbb{R}^{n}\right)\to\mathcal{A}_{\phg}^{\left(\mathcal{E}_{\of},\infty,\infty\right)}\left(\overline{\mathbb{R}_{1}^{1}}\times\overline{\mathbb{R}^{n}}\right)$.
	Therefore, the composition $BQ$ is well-defined. Moreover, if $B=\Op_{L}^{\po}\left(b\right)$
	and $Q=\Op_{R}\left(q\right)$, then for every $u\in\mathcal{S}\left(\mathbb{R}^{n}\right)$
	we have
	\begin{align*}
		\left(BQu\right)\left(x,y\right) & =\frac{1}{\left(2\pi\right)^{n}}\int e^{iy\eta}b\left(y;x,\eta\right)\widehat{Qu}\left(\eta\right)d\eta\\
		& =\frac{1}{\left(2\pi\right)^{n}}\int e^{i\left(y-\tilde{y}\right)\eta}b\left(y;x,\eta\right)q\left(\tilde{y};\eta\right)u\left(\tilde{y}\right)d\tilde{y}d\eta.
	\end{align*}
	Therefore, we need to show that $b\left(y;x,\eta\right)q\left(\tilde{y};\eta\right)\in S_{0\po,\mathcal{S}}^{-\infty,\left(\mathcal{E}_{\of},\mathcal{E}_{\ff}+\mathcal{G}\right)}\left(\mathbb{R}^{2n},\mathbb{R}_{1}^{n+1}\right)$.
	In order to do so, it suffices to prove that scalar multiplication
	in the interior determines a continuous bilinear map
	\begin{align}
		\mathcal{A}_{\phg}^{\left(\mathcal{E}_{\of},\mathcal{E}_{\ff},\infty,\infty\right)}\left(\hat{P}_{0}^{2}\right)\times\mathcal{A}_{\phg}^{-\mathcal{G}}\left(\overline{\mathbb{R}^{n}}\right) & \to\mathcal{A}_{\phg}^{\left(\mathcal{E}_{\of},\mathcal{E}_{\ff}+\mathcal{G},\infty,\infty\right)}\left(\hat{P}_{0}^{2}\right)\nonumber \\
		\left(b\left(x,\eta\right),q\left(\eta\right)\right) & \mapsto b\left(x,\eta\right)q\left(\eta\right).\label{eq:Poisson*boundary}
	\end{align}
	The well-definedness and continuity of \ref{eq:Poisson*boundary}
	is an easy application of the pull-back Theorem. Indeed, the lifted
	projection $\beta_{\po,R}:\hat{P}_{0}^{2}\to\overline{\mathbb{R}^{n}}$
	lifts $\left\langle \eta\right\rangle ^{-1}$ to a product $r_{\ff}r_{\iif_{\eta}}$,
	where $r_{\ff}$ and $r_{\iif_{\eta}}$ are boundary defining functions
	for the faces $\ff$ and $\iif_{\eta}$ of $\hat{P}_{0}^{2}$. Therefore,
	\begin{align*}
		\beta_{\po}^{*}\left(b\left(x,\eta\right)q\left(\eta\right)\right) & =\beta_{\po}^{*}b\cdot\beta_{\po,R}^{*}q\\
		& \in\mathcal{A}_{\phg}^{\left(\mathcal{E}_{\of},\mathcal{E}_{\ff}+\mathcal{G},\infty,\infty\right)}\left(\hat{P}_{0}^{2}\right)
	\end{align*}
	as claimed.
	
	\textbf{(boundary $\circ$ $0$-trace) }The claim follows from the
	composition rule we just proved by taking formal adjoints.
	
	\textbf{($0$-Poisson $\circ$ $0$-trace) }Arguing as above, it suffices
	to prove that pointwise multiplication in the interior determines
	a continuous bilinear map
	\begin{align*}
		\mathcal{A}_{\phg}^{\left(\mathcal{E}_{\of},\mathcal{E}_{\ff},\infty,\infty\right)}\left(\hat{P}_{0}^{2}\right)\times\mathcal{A}_{\phg}^{\left(\mathcal{F}_{\of},\mathcal{F}_{\ff}-1,\infty,\infty\right)}\left(\hat{T}_{0}^{2}\right) & \to\mathcal{A}_{\phg}^{\left(\mathcal{E}_{\of},\mathcal{F}_{\of},\mathcal{E}_{\ff}+\mathcal{F}_{\ff}-1,\infty,\infty\right)}\left(\hat{M}_{0}^{2}\right)\\
		\left(b\left(x,\eta\right),a\left(\tilde{y},\eta\right)\right) & \mapsto b\left(x,\eta\right)a\left(\tilde{y},\eta\right).
	\end{align*}
	This result is proved using an argument similar to the one used in
	the proof of Theorem \ref{thm:mapping-properties-on-phg-untwisted}.
	Define the space $Z=\left[\hat{M}_{0}^{2};\iif_{\eta}\cap\left(\lf\cap\rf\right)\right]$
	and the $b$-fibrations $\gamma_{L}:Z\to\hat{P}_{0}^{2}$ and $\gamma_{R}:Z\to\hat{T}_{0}^{2}$.
	The space $Z$ is the blow-down of the space $Z_{b}$ used in the
	proof of Theorem \ref{thm:mapping-properties-on-phg-untwisted}. Accordingly,
	denote the faces of $Z$ by $\lf$, $\rf$, $\ff_{0}$, $\ef_{x}$,
	$\ef_{\tilde{x}}$, $\iif_{\eta}$, $\iif_{x}$, $\iif_{\tilde{x}}$.
	The ``extra faces'' $\ef_{x},\ef_{\tilde{x}}$ are obtained from
	the blow-ups of the corners $\iif_{\eta}\cap\lf$ and $\iif_{\eta}\cap\rf$
	respectively, while the other faces are lifted from $\hat{M}_{0}^{2}$.
	The graphical representation of $Z$ and $\gamma_{L},\gamma_{R}$
	is given in Figure \ref{fig:Z}\begin{figure}
		\centering
		\caption{$Z$}
		\label{fig:Z}
		
		%%%%%%%%%%%%%%%%%%%%%%%%%%%%%%%
		%%%%%%%%%% double space
		%%%%%%%%%%%%%%%%%%%%%%%%%%%%%%%
		
		\tikzmath{\L = 5;\R = \L /2.5;\u=180/7;}
		\tdplotsetmaincoords{60}{160}
		\begin{tikzpicture}[tdplot_main_coords, scale = 0.9]
			\centering
			%%% axes
			%\coordinate (O) at (0,0,0) ;
			%\coordinate (X) at (1,0,0) ;
			%\coordinate (Y) at (0,1,0) ;
			%\coordinate (Z) at (0,0,1) ;
			%\draw[->] (O) -- (X) node {$x$};
			%\draw[->] (O) -- (Y) node {$y$};
			%\draw[->] (O) -- (Z) node {$z$};
			
			%%% points
			\coordinate (A) at (-\L,0,0) ;
			\coordinate (A1) at (-\L+\R,0,0) ;
			\coordinate (A2) at (-\L,\R,0) ;
			\coordinate (A3) at (-\L,0,\R) ;
			\coordinate (A4) at ({-\L+\R*cos(\u)},{\R*sin(\u)},0) ;
			\coordinate (A5) at ({-\L+\R*cos(\u)},0,{\R*sin(\u)}) ;
			\coordinate (A6) at (-\L,{\R*cos(\u)},{\R*sin(\u)}) ;
			\coordinate (A7) at (-\L,{\R*cos(90-\u)},{\R*sin(90-\u)}) ;
			\coordinate (A8) at ({-\L+\R*cos(90-\u)},{\R*sin(90-\u)},0) ;
			\coordinate (A9) at ({-\L+\R*sin(\u)},0,{\R*cos(\u)}) ;

			\coordinate (B) at (\L,0,0) ;
			\coordinate (B1) at (\L-\R,0,0) ;
			\coordinate (B2) at (\L,\R,0) ;
			\coordinate (B3) at (\L,0,\R) ;
			\coordinate (B4) at ({\L-\R*cos(\u)},{\R*sin(\u)},0) ;
			\coordinate (B5) at ({\L-\R*cos(\u)},0,{\R*sin(\u)}) ;
			\coordinate (B6) at (\L,{\R*cos(\u)},{\R*sin(\u)}) ;
			\coordinate (B7) at (\L,{\R*cos(90-\u)},{\R*sin(90-\u)}) ;
			\coordinate (B8) at ({\L-\R*cos(90-\u)},{\R*sin(90-\u)},0) ;
			\coordinate (B9) at ({\L-\R*sin(\u)},0,{\R*cos(\u)}) ;

			\coordinate (C) at (-\L,\L,0) ;
			\coordinate (C1) at ({-\L+\R*cos(90-\u)},\L,0) ;
			\coordinate (C2) at (-\L,\L,{\R*sin(\u)}) ;
			
			\coordinate (D) at (\L,\L,0) ;
			\coordinate (D1) at ({\L-\R*cos(90-\u)},\L,0) ;
			\coordinate (D2) at (\L,\L,{\R*sin(\u)}) ;
			
			\coordinate (E) at (-\L,0,\L) ;
			\coordinate (E1) at (-\L,{\R*cos(90-\u)},\L) ;
			\coordinate (E2) at ({-\L+\R*sin(\u)},0,\L) ;

			\coordinate (F) at (\L,0,\L) ;
			\coordinate (F1) at (\L,{\R*cos(90-\u)},\L) ;
			\coordinate (F2) at ({\L-\R*sin(\u)},0,\L) ;
			
			\coordinate (G) at (-\L,\L,\L);
			\coordinate (H) at (\L,\L,\L);
			
			%%% segments
			\draw
			(B1) -- (A1)
			(A6) -- (C2)
			(A7) -- (E1)
			(A8) -- (C1)
			(A9) -- (E2)
			(B7) -- (F1)
			(B9) -- (F2)
			(B6) -- (D2)
			(B8) -- (D1)
			
			;
			
			%%% dashed segments
			\draw[]
			(F1) -- (H)
			(H) -- (D2)
			(E1) -- (G)
			(G) -- (C2)
			(F2) -- (E2)
			(D1) -- (C1)
			(G) -- (H)
			;
			
			%%% arcs
			\begin{scope}[canvas is zy plane at x=0]
				\draw (B7) arc(\u:{90-\u}:\R);
				\draw (A7) arc(\u:{90-\u}:\R);	
			\end{scope}
			\begin{scope}[canvas is xy plane at z=0]
				\draw (B8) arc({90+\u}:{180}:\R);
				\draw (A8) arc({90-\u}:{0}:\R);
				\draw (A9) arc(0:90:{\R*sin(\u)});
				\draw[] (E2) arc(0:90:{\R*sin(\u)});
				\draw (B7) arc(90:180:{\R*sin(\u)});
				\draw[] (F1) arc(90:180:{\R*sin(\u)});
			\end{scope}
			\begin{scope}[canvas is xz plane at y=0]
				\draw (B9) arc({90+\u}:{180}:\R);
				\draw (A9) arc({90-\u}:{0}:\R);
				\draw (A8) arc(0:90:{\R*sin(\u)});
				\draw[] (C1) arc(0:90:{\R*sin(\u)});
				\draw (B6) arc(90:180:{\R*sin(\u)});
				\draw[] (D2) arc(90:180:{\R*sin(\u)});
			\end{scope}
			
			%%% text
			
			% RIGHT FACE
			\node at (0,{2/3*\L},0) {$\text{rf}$};
			
			% LEFT FACE
			\node at (0,0,{2/3*\L}) {$\text{lf}$};
			
			% FF0 FACE
			\node at ({\L-\L/6},{\L/6},{\L/6}) {$\text{ff}_0$};
			\node at ({-\L+\L/6},{\L/6},{\L/6}) {$\text{ff}_0$};
			
			% EFx FACE
			\node at ({-\L+\L/12},{\L/12},{2/3*\L}) {$\text{ef}_x$};
			\node at ({\L-\L/8},{\L/8},{2/3*\L}) {$\text{ef}_x$};
			
			% EF\tildex FACE
			\node at ({-\L+\L/12},{2/3*\L},{\L/12}) {$\text{ef}_{\tilde{x}}$};
			\node at ({\L-\L/8},{2/3*\L},{\L/8}) {$\text{ef}_{\tilde{x}}$};
			
			% \eta INFINITY FACE
			%\node at ({-\L},{5/6 * \L},{5/6 * \L}) {$\text{if}_\eta$};
			%\node at ({\L},{5/6 * \L},{5/6 * \L}) {$\text{if}_\eta$};
			
			%%%%%%%%%%%%%%%%%%%%%%%%%%%%%%%
			%%%%%%%%%% Poisson space
			%%%%%%%%%%%%%%%%%%%%%%%%%%%%%%%
			
			\begin{scope}[shift = {(0,0,{2*\L})}]
				%%% axes
				%\coordinate (O) at (0,0,0) ;
				%\coordinate (X) at (1,0,0) ;
				%\coordinate (Y) at (0,1,0) ;
				%\coordinate (Z) at (0,0,1) ;
				%\draw[->] (O) -- (X) node {$x$};
				%\draw[->] (O) -- (Y) node {$y$};
				%\draw[->] (O) -- (Z) node {$z$};
				
				%%% points
				\coordinate (A) at (-\L,0,0) ;
				\coordinate (A1) at (-\L+\R/2,0,0) ;
				\coordinate (A2) at (-\L,\R/2,0) ;
				
				\coordinate (B) at (\L,0,0) ;
				\coordinate (B1) at (\L-\R/2,0,0) ;
				\coordinate (B2) at (\L,\R/2,0) ;
				
				\coordinate (C) at (-\L,\L,0) ;
				\coordinate (D) at (\L,\L,0) ;
				
				%%% segments
				\draw 
				(B2) -- (D) 
				(B1) -- (A1) 
				(A2) -- (C) 
				;
				
				%%% dashed segments
				\draw[]
				(D) -- (C)
				;
				
				%%% arcs
				\begin{scope}[canvas is xy plane at z=0]
					\draw (B2) arc(90:180:\R/2);
					\draw (A2) arc(90:0:\R/2);
				\end{scope}
				
				%%% text
				
				% ORIGINAL FACE
				\node at (0,{\L/9},0) {$\text{of}$};
				
				% FRONT FACE
				\node at ({\L-\L/4},{\L/4},0) {$\text{ff}$};
				\node at ({-\L+\L/4},{\L/4},0) {$\text{ff}$};
				
				% \eta INFINITY FACE
				%\node at ({\L-\L/12},{\L*2/3},0) {$\text{if}_\eta$};
				%\node at ({-\L+\L/12},{\L*2/3},0) {$\text{if}_\eta$};
				
			\end{scope}
			
			%%%%%%%%%%%%%%%%%%%%%%%%%%%%%%%
			%%%%%%%%%% Trace space
			%%%%%%%%%%%%%%%%%%%%%%%%%%%%%%%
			
			\begin{scope}[shift = {(0,{2.5*\L},0)}]
				%%% axes
				%\coordinate (O) at (0,0,0) ;
				%\coordinate (X) at (1,0,0) ;
				%\coordinate (Y) at (0,1,0) ;
				%\coordinate (Z) at (0,0,1) ;
				%\draw[->] (O) -- (X) node {$x$};
				%\draw[->] (O) -- (Y) node {$y$};
				%\draw[->] (O) -- (Z) node {$z$};
				
				%%% points
				\coordinate (A) at (-\L,0,0) ;
				\coordinate (A1) at (-\L+\R/2,0,0) ;
				\coordinate (A3) at (-\L,0,\R/2) ;
				
				\coordinate (B) at (\L,0,0) ;
				\coordinate (B1) at (\L-\R/2,0,0) ;
				\coordinate (B3) at (\L,0,\R/2) ;
				
				\coordinate (E) at (-\L,0,\L) ;
				\coordinate (F) at (\L,0,\L) ;
				
				%%% segments
				\draw 
				(B3) -- (F) 
				(B1) -- (A1) 
				(A3) -- (E)
				;
				
				%%% dashed segments
				\draw[]
				(F) -- (E)
				;
				
				%%% arcs
				\begin{scope}[canvas is xz plane at y=0]
					\draw (B3) arc(90:180:\R/2);
					\draw (A1) arc(0:90:\R/2);
				\end{scope}
				
				%%% text
				
				% ORIGINAL FACE
				\node at (0,0,{\L/9}) {$\text{of}$};
				
				% FRONT FACE
				\node at ({\L-\L/4},0,{\L/4}) {$\text{ff}$};
				\node at ({-\L+\L/4},0,{\L/4}) {$\text{ff}$};
				
				% \eta INFINITY FACE
				%\node at ({\L-\L/12},0,{\L*2/3}) {$\text{if}_\eta$};
				%\node at ({-\L+\L/12},0,{\L*2/3}) {$\text{if}_\eta$};
				
			\end{scope}
			
			%%%%%%%%%%%%%%%%%%%%%%%%%%%%%%%
			%%%%%%%%%% blow-down maps
			%%%%%%%%%%%%%%%%%%%%%%%%%%%%%%%
			
			\draw[-latex] (0,0,{\L+0.1*\L}) -- (0,0,{3/2*\L}) node[midway, right] {$\gamma_L$};
			\draw[-latex] (0,{\L+0.1*\L},0) -- (0,{3/2*\L},0) node[midway, right] {$\gamma_R$};
			
		\end{tikzpicture}
		
	\end{figure}. Choose boundary defining functions $\rho_{\of},\rho_{\ff},\rho_{\iif_{\eta}},\rho_{\iif_{x}}$
	for $\hat{P}_{0}^{2}$, and $\tilde{\rho}_{\of},\tilde{\rho}_{\ff},\tilde{\rho}_{\iif_{\eta}},\tilde{\rho}_{\iif_{\tilde{x}}}$
	for $\hat{T}_{0}^{2}$. Then we have
	\[
	\begin{array}{ll}
		\gamma_{L}^{*}\rho_{\of}=r_{\lf} & \gamma_{R}^{*}\tilde{\rho}_{\of}=r_{\rf}\\
		\gamma_{L}^{*}\rho_{\ff}=r_{\ff_{0}}r_{\ef_{x}} & \gamma_{R}^{*}\tilde{\rho}_{\ff}=r_{\ff_{0}}r_{\ef_{\tilde{x}}}\\
		\gamma_{L}^{*}\rho_{\iif_{\eta}}=r_{\iif_{\eta}}r_{\ef_{\tilde{x}}} & \gamma_{R}^{*}\tilde{\rho}_{\iif_{\eta}}=r_{\iif_{\eta}}r_{\ef_{x}}\\
		\gamma_{L}^{*}\rho_{\iif_{x}}=r_{\iif_{x}} & \gamma_{R}^{*}\tilde{\rho}_{\iif_{\tilde{x}}}=r_{\iif_{\tilde{x}}}
	\end{array}
	\]
	where $r_{\lf},r_{\rf},r_{\ff_{0}},r_{\ef_{x}},r_{\ef_{\tilde{x}}},r_{\iif_{\eta}},r_{\iif_{x}},r_{\iif_{\tilde{x}}}$
	are boundary defining functions for $Z$. The index
	sets of $\gamma_{L}^{*}b$ are \[\left(\mathcal{E}_{\of},0,\mathcal{E}_{\ff},\mathcal{E}_{\ff},\infty,\infty,\infty,\infty\right)\]
	and the index sets of $\gamma_{R}^{*}a$ are \[\left(0,\mathcal{F}_{\of},\mathcal{F}_{\ff}-1,\infty,\mathcal{F}_{\ff}-1,\infty,\infty,\infty\right).\]
	Therefore, the index sets of the product are \[\left(\mathcal{E}_{\of},\mathcal{F}_{\of},\mathcal{E}_{\ff}+\mathcal{F}_{\ff}-1,\infty,\infty,\infty,\infty,\infty\right).\]
	In particular, observe that the index sets at both $\ef_{x}$ and
	$\ef_{\tilde{x}}$ are $\infty$, and therefore $\gamma_{L}^{*}b\cdot\gamma_{R}^{*}a$
	is polyhomogeneous on the blow-down $\hat{M}_{0}^{2}$, with index
	sets \[\left(\mathcal{E}_{\of},\mathcal{F}_{\of},\mathcal{E}_{\ff}+\mathcal{F}_{\ff}-1,\infty,\infty,\infty\right).\]
	This concludes the proof.
	
	\textbf{($0$-trace $\circ$ $0$-Poisson)} If $A\in\hat{\Psi}_{0\tr,\mathcal{S}}^{-\infty,\mathcal{F}}\left(\mathbb{R}_{1}^{n+1},\mathbb{R}^{n}\right)$
	and $B\in\hat{\Psi}_{0\po,\mathcal{S}}^{-\infty,\mathcal{E}}\left(\mathbb{R}^{n},\mathbb{R}_{1}^{n+1}\right)$,
	then by the proof of Theorem \ref{thm:mapping-properties-on-phg-untwisted}
	$B$ defines a continuous linear map $B:\mathcal{S}\left(\mathbb{R}^{n}\right)\to\mathcal{A}_{\phg}^{\left(\mathcal{E}_{\of},\infty,\infty\right)}\left(\overline{\mathbb{R}_{1}^{1}}\times\overline{\mathbb{R}^{n}}\right)$,
	and since $\Re\left(\mathcal{F}_{\of}+\mathcal{E}_{\of}\right)>-1$,
	$A$ induces a continuous linear map $A:\mathcal{A}_{\phg}^{\left(\mathcal{E}_{\of},\infty,\infty\right)}\left(\overline{\mathbb{R}_{1}^{1}}\times\overline{\mathbb{R}^{n}}\right)\to\mathcal{S}\left(\mathbb{R}^{n}\right)$.
	Therefore, the composition is well-defined; if $A=\Op_{L}^{\tr}\left(a\left(y;\tilde{x},\eta\right)d\tilde{x}\right)$
	and $B=\Op_{R}^{\po}\left(b\left(\tilde{y};x,\eta\right)\right)$,
	then $AB=\Op\left(c\left(y,\tilde{y};\eta\right)\right)$ where
	\[
	c\left(y,\tilde{y};\eta\right)=\int a\left(y;\tilde{x},\eta\right)b\left(\tilde{y};\tilde{x},\eta\right)d\tilde{x}.
	\]
	Arguing as above, it suffices to prove that integration by $\tilde{x}$
	defines a continuous linear map
	\begin{align*}
		\mathcal{A}_{\phg}^{\left(\mathcal{E}_{\of},\mathcal{E}_{\ff}-1,\infty,\infty\right)}\left(\hat{T}_{0}^{2}\right)\times\mathcal{A}_{\phg}^{\left(\mathcal{F}_{\of},\mathcal{F}_{\ff},\infty,\infty\right)}\left(\hat{T}_{0}^{2}\right) & \to\mathcal{A}_{\phg}^{\mathcal{E}_{\ff}+\mathcal{F}_{\ff}}\left(\overline{\mathbb{R}^{n}}\right)\\
		\left(a\left(\tilde{x},\eta\right),b\left(\tilde{x},\eta\right)\right) & \mapsto\int a\left(\tilde{x},\eta\right)b\left(\tilde{x},\eta\right)d\tilde{x}.
	\end{align*}
	The multiplication $a\left(\tilde{x},\eta\right)b\left(\tilde{x},\eta\right)$
	is clearly in $\mathcal{A}_{\phg}^{\left(\mathcal{E}_{\of}+\mathcal{F}_{\of},\mathcal{E}_{\ff}+\mathcal{F}_{\ff}-1,\infty,\infty\right)}\left(\hat{T}_{0}^{2}\right)$.
	The integral on the right is simply the push-forward $\left(\beta_{\tr,R}\right)_{*}\left(a\left(\tilde{x},\eta\right)b\left(\tilde{x},\eta\right)d\tilde{x}\right)$
	via the $b$-fibration $\beta_{\tr,R}:\hat{T}_{0}^{2}\to\overline{\mathbb{R}^{n}}$
	obtained by lifting the canonical projection $\left(\tilde{x},\eta\right)\mapsto\eta$.
	We know that the pull-back of $\left\langle \eta\right\rangle ^{-1}$
	is of the form $r_{\iif_{\eta}}r_{\ff}$, where $r_{\iif_{\eta}},r_{\ff}$
	are boundary defining functions for the hyperfaces $\iif_{\eta},\ff$
	of $\hat{T}_{0}^{2}$. Moreover, calling $\beta_{\tr,L}:\hat{T}_{0}^{2}\to\overline{\mathbb{R}_{1}^{1}}$
	the other lifted projection, we have
	\begin{align*}
		\beta_{\tr,L}^{*}\mathcal{D}_{\overline{\mathbb{R}_{1}^{1}}}^{1} & =\beta_{\tr}^{*}\mathcal{D}_{\hat{T}^{2}}^{1}\otimes\beta_{\tr,R}^{*}\mathcal{D}_{\overline{\mathbb{R}^{n}}}^{-1}\\
		& =r_{\ff}\mathcal{D}_{\hat{T}_{0}^{2}}^{1}\otimes\beta_{\tr,R}^{*}\mathcal{D}_{\overline{\mathbb{R}^{n}}}^{-1}\\
		& =r_{\of}r_{\ff}^{2}\cdot{^{b}\mathcal{D}_{\hat{T}_{0}^{2}}^{1}}\otimes\beta_{\tr,R}^{*}\mathcal{D}_{\overline{\mathbb{R}^{n}}}^{-1}
	\end{align*}
	where $r_{\of}$ is a boundary defining function for $\of$. This
	computation is a simple application of Lemma C.4 of \cite{EpsteinMelroseShrinking}.
	Therefore, $a\left(\tilde{x},\eta\right)b\left(\tilde{x},\eta\right)$
	lifts to a polyhomogeneous section of the bundle ${^{b}\mathcal{D}_{\hat{T}_{0}^{2}}^{1}}\otimes\beta_{\tr,R}^{*}\mathcal{D}_{\overline{\mathbb{R}^{n}}}^{-1}$
	over $\hat{T}_{0}^{2}$, with index sets $\mathcal{E}_{\of}+\mathcal{F}_{\of}+1,\mathcal{E}_{\ff}+\mathcal{F}_{\ff}+1,\infty,\infty$.
	By the push-forward Theorem, since $\Re\left(\mathcal{E}_{\of}+\mathcal{F}_{\of}\right)>-1$,
	$\left(\beta_{\tr,R}^{*}\right)_{*}\left(a\left(\tilde{x},\eta\right)b\left(\tilde{x},\eta\right)d\tilde{x}\right)$
	is well-defined and a polyhomogeneous section of $^{b}\mathcal{D}_{\overline{\mathbb{R}^{n}}}^{1}\otimes\mathcal{D}_{\overline{\mathbb{R}^{n}}}^{-1}$
	with index set $\mathcal{E}_{\ff}+\mathcal{F}_{\ff}+1$. Finally,
	we have $^{b}\mathcal{D}_{\overline{\mathbb{R}^{n}}}^{1}=\left\langle \eta\right\rangle \mathcal{D}_{\overline{\mathbb{R}^{n}}}^{1}$,
	so in fact $\left(\beta_{\tr,R}^{*}\right)_{*}\left(a\left(\tilde{x},\eta\right)b\left(\tilde{x},\eta\right)d\tilde{x}\right)$
	is just a polyhomogeneous function with index set $\mathcal{E}_{\ff}+\mathcal{F}_{\ff}$,
	as claimed.
	
	\textbf{(interior $\circ$ $0$-Poisson) }For the $0b$-interior case,
	the proof is essentially equal to the proof of point 2 of Theorem
	\ref{thm:mapping-properties-on-phg-untwisted}. In the $0$-interior
	case, it suffices to use the simpler blow-up $Z$ used to discuss
	the \textbf{$0$}-Poisson\textbf{ $\circ$ $0$}-trace composition
	rule. We omit the details, which can be easily filled by decorating
	Figures \ref{fig:Z} and \ref{fig:Zb} with the appropriate index
	sets.
	
	\textbf{($0$-trace $\circ$ interior) }These results follow immediately
	from the previous cases by taking adjoints.
\end{proof}
\begin{proof}
	(Proof of Theorem \ref{thm:compositions-involving-interior}) Let
	us first prove the \textbf{$0$}-interior\textbf{ $\circ$ $0$}-interior\textbf{
	}case. As in the previous proofs, it is sufficient to prove that if
	\begin{align*}
		a\left(x,\tilde{x},\eta\right) & \in\mathcal{A}_{\phg}^{\left(\mathcal{E}_{\lf},\mathcal{E}_{\rf},\mathcal{E}_{\ff_{0}}-1,\infty,\infty,\infty\right)}\left(\hat{M}_{0}^{2}\right)\\
		b\left(x,\tilde{x},\eta\right) & \in\mathcal{A}_{\phg}^{\left(\mathcal{F}_{\lf},\mathcal{F}_{\rf},\mathcal{F}_{\ff_{0}}-1,\infty,\infty,\infty\right)}\left(\hat{M}_{0}^{2}\right)
	\end{align*}
	and $\Re\left(\mathcal{E}_{\rf}+\mathcal{F}_{\lf}\right)>-1$, then
	the distribution
	\[
	c\left(x,\tilde{x},\eta\right)=\int a\left(x,x',\eta\right)b\left(x',\tilde{x},\eta\right)dx'
	\]
	is well-defined and in $\mathcal{A}_{\phg}^{\left(\mathcal{E}_{\lf},\mathcal{F}_{\rf},\mathcal{E}_{\ff_{0}}+\mathcal{F}_{\ff_{0}}-1,\infty,\infty,\infty\right)}\left(\hat{M}_{0}^{2}\right)$.
	
	To prove this, it is convenient to introduce an alternative compactification
	of the locus $\hat{M}_{0}^{2}\backslash\left(\iif_{\eta}\cup\iif_{x}\cup\iif_{\tilde{x}}\right)$.
	Consider the functions
	\[
	\tau=x\left\langle \eta\right\rangle ,\tilde{\tau}=\tilde{x}\left\langle \eta\right\rangle .
	\]
	These functions are polyhomogeneous on $\hat{M}_{0}^{2}$. More precisely,
	recalling that the order of the faces of $\hat{M}_{0}^{2}$ is $\lf$,
	$\rf$, $\ff_{0}$, $\iif_{\eta}$, $\iif_{x}$, $\iif_{\tilde{x}}$,
	the index sets are:
	\begin{enumerate}
		\item $\left(1,0,1,0,-1,0\right)$ for $x$;
		\item $\left(0,1,1,0,0,-1\right)$ for $\tilde{x}$;
		\item $\left(0,0,-1,-1,0,0\right)$ for $\left\langle \eta\right\rangle $;
	\end{enumerate}
	therefore, the index sets of $\tau$ are $\left(1,0,0,-1,-1,0\right)$,
	while those of $\tilde{\tau}$ are $\left(0,1,0,-1,-1,0\right)$.
	In particular, on $\hat{M}_{0}^{2}\backslash\left(\iif_{\eta}\cup\iif_{x}\cup\iif_{\tilde{x}}\right)$,
	$\tau$ is a boundary defining function for $\lf$ and $\tilde{\tau}$
	is a boundary defining function for $\rf$, while $\left\langle \eta\right\rangle ^{-1}$
	is a boundary defining function for $\ff_{0}$. Now, consider the
	diffeomorphism
	\begin{align*}
		\mathbb{R}_{2}^{2}\times\mathbb{R}^{n} & \to\hat{M}^{2}\backslash\left(\iif_{\eta}\cup\iif_{x}\cup\iif_{\tilde{x}}\right)\\
		\left(\tau,\tilde{\tau},\eta\right) & \mapsto\left(\tau\left\langle \eta\right\rangle ^{-1},\tilde{\tau}\left\langle \eta\right\rangle ^{-1},\eta\right).
	\end{align*}
	This diffeomorphism lifts to a diffeomorphism
	\begin{align*}
		\varphi:\mathbb{R}_{2}^{2}\times\overline{\mathbb{R}^{n}} & \to\hat{M}_{0}^{2}\backslash\left(\iif_{\eta}\cup\iif_{x}\cup\iif_{\tilde{x}}\right),
	\end{align*}
	so that the coordinates $\tau,\tilde{\tau}$ restricted to the locus
	$\left\langle \eta\right\rangle =\infty$ parametrize $\ff_{0}\backslash\iif_{\eta}$
	and identify it with $\mathbb{R}_{2}^{2}$. We define $\hat{D}_{0}^{2}$
	as the compactification $\overline{\mathbb{R}_{2}^{2}}\times\overline{\mathbb{R}^{n}}$.
	This space is represented in Figure \ref{fig:Dhat}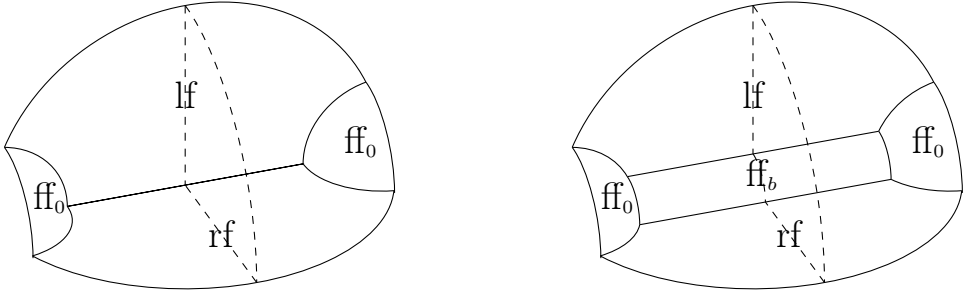
\begin{figure}
		\centering
		\caption{$\hat{D}^2_0$ and $\hat{D}^2_{0b}$}
		\label{fig:Dhat}
		
		\begin{minipage}{.49\textwidth}
			\centering
			\tikzmath{
				\L = 5;
				\u = 0;
				\R0 = \L/2.5;
				\Rb = \R0 * sin(\u);
				\v = asin((1/2 *sqrt(4* \R0^2 - \R0^4 / \L^2))/\R0);
				\w = asin((1/2 *sqrt(4* \R0^2 - \R0^4 / \L^2))/\L);
			}
			\tdplotsetmaincoords{60}{160}
			\begin{tikzpicture}[tdplot_main_coords, scale = \myscale]
				
				%%% points
				\coordinate (A) at (-\L,0,0);
				\coordinate (A1) at ({-\L + \R0*cos(\v)},0,{\R0*sin(\v)});
				\coordinate (A2) at ({-\L + \R0*cos(\u)},0,{\R0*sin(\u)});
				\coordinate (A3) at ({-\L + \R0*cos(\v)},{\R0*sin(\v)},0);
				\coordinate (A4) at ({-\L + \R0*cos(\u)},{\R0*sin(\u)},0);
				
				\coordinate (B) at (\L,0,0);
				\coordinate (B1) at ({\L - \R0*cos(\v)},0,{\R0*sin(\v)});
				\coordinate (B2) at ({\L - \R0*cos(\u)},0,{\R0*sin(\u)});
				\coordinate (B3) at ({\L - \R0*cos(\v)},{\R0*sin(\v)},0);
				\coordinate (B4) at ({\L - \R0*cos(\u)},{\R0*sin(\u)},0);
				
				%%% segments
				\draw
				(A2) -- (B2)
				(A4) -- (B4)
				;
				\draw[dashed] (0,0,\L) -- (0,0,\Rb);
				\draw[dashed] (0,\L,0) -- (0,\Rb,0);
				
				%%% arcs
				\begin{scope}[canvas is xz plane at y=0]
					\draw[] (B1) arc ({\w}:{180-\w}:\L);
					\draw (A2) arc ({\u}:{\v}:\R0);
					\draw (B2) arc ({180-\u}:{180-\v}:\R0);
				\end{scope}
				\begin{scope}[canvas is xy plane at z=0]
					\draw[] (B3) arc ({\w}:{180-\w}:\L);
					\draw (A4) arc ({\u}:{\v}:\R0);
					\draw (B4) arc ({180-\u}:{180-\v}:\R0);
				\end{scope}
				\begin{scope}[canvas is zy plane at x=0]
					\draw (B1) arc(0:90:\R0);
					\draw (B2) arc(0:90:\Rb);
					\draw (A2) arc(0:90:\Rb);
					\draw (A1) arc(0:90:\R0);
					\draw[dashed] (\L,0) arc(0:90:\L);
					\draw[dashed] (\Rb,0) arc(0:90:\Rb);
				\end{scope}
				
				%%% text
				\node at (0,0,{\L/2}) {$\text{\Large lf}$};
				\node at (0,{\L/2},0) {$\text{\Large rf}$};
				\node at ({\L-\L/4},{\L/6},{\L/6}) {$\text{\Large ff}_0$};
				\node at ({\L/6-\L},{\L/6},{\L/6}) {$\text{\Large ff}_0$};
				
			\end{tikzpicture}
			
		\end{minipage}
		\begin{minipage}{.49\textwidth}
			\centering
			\tikzmath{
				\L = 5;
				\u = 180/7;
				\R0 = \L/2.5;
				\Rb = \R0 * sin(\u);
				\v = asin((1/2 *sqrt(4* \R0^2 - \R0^4 / \L^2))/\R0);
				\w = asin((1/2 *sqrt(4* \R0^2 - \R0^4 / \L^2))/\L);
			}
			\tdplotsetmaincoords{60}{160}
			\begin{tikzpicture}[tdplot_main_coords, scale = \myscale]
				
				%%% points
				\coordinate (A) at (-\L,0,0);
				\coordinate (A1) at ({-\L + \R0*cos(\v)},0,{\R0*sin(\v)});
				\coordinate (A2) at ({-\L + \R0*cos(\u)},0,{\R0*sin(\u)});
				\coordinate (A3) at ({-\L + \R0*cos(\v)},{\R0*sin(\v)},0);
				\coordinate (A4) at ({-\L + \R0*cos(\u)},{\R0*sin(\u)},0);
				
				\coordinate (B) at (\L,0,0);
				\coordinate (B1) at ({\L - \R0*cos(\v)},0,{\R0*sin(\v)});
				\coordinate (B2) at ({\L - \R0*cos(\u)},0,{\R0*sin(\u)});
				\coordinate (B3) at ({\L - \R0*cos(\v)},{\R0*sin(\v)},0);
				\coordinate (B4) at ({\L - \R0*cos(\u)},{\R0*sin(\u)},0);
				
				%%% segments
				\draw
				(A2) -- (B2)
				(A4) -- (B4)
				;
				\draw[dashed] (0,0,\L) -- (0,0,\Rb);
				\draw[dashed] (0,\L,0) -- (0,\Rb,0);
				
				%%% arcs
				\begin{scope}[canvas is xz plane at y=0]
					\draw[] (B1) arc ({\w}:{180-\w}:\L);
					\draw (A2) arc ({\u}:{\v}:\R0);
					\draw (B2) arc ({180-\u}:{180-\v}:\R0);
				\end{scope}
				\begin{scope}[canvas is xy plane at z=0]
					\draw[] (B3) arc ({\w}:{180-\w}:\L);
					\draw (A4) arc ({\u}:{\v}:\R0);
					\draw (B4) arc ({180-\u}:{180-\v}:\R0);
				\end{scope}
				\begin{scope}[canvas is zy plane at x=0]
					\draw (B1) arc(0:90:\R0);
					\draw (B2) arc(0:90:\Rb);
					\draw (A2) arc(0:90:\Rb);
					\draw (A1) arc(0:90:\R0);
					\draw[dashed] (\L,0) arc(0:90:\L);
					\draw[dashed] (\Rb,0) arc(0:90:\Rb);
				\end{scope}
				
				%%% text
				\node at (0,0,{\L/2}) {$\text{\Large lf}$};
				\node at (0,{\L/2},0) {$\text{\Large rf}$};
				\node at ({\L-\L/4},{\L/6},{\L/6}) {$\text{\Large ff}_0$};
				\node at ({\L/6-\L},{\L/6},{\L/6}) {$\text{\Large ff}_0$};
				\node at (0,{\L/8},{\L/8}) {$\text{\Large ff}_b$};
				
			\end{tikzpicture}
			
		\end{minipage}
		
	\end{figure}, and it has four hyperfaces $\lf$, $\rf$, $\ff_{0}$ and $\iif$:
	the infinity face $\iif$ is the locus $\tau^{2}+\tilde{\tau}^{2}=+\infty$
	produced by radially compactifying the $\mathbb{R}_{2}^{2}$ factor.
	Clearly, the manifolds with corners $\hat{D}_{0}^{2}$ and $\hat{M}_{0}^{2}$
	are not diffeomorphic, and in fact the map $\varphi:\hat{D}_{0}^{2}\backslash\iif\to\hat{M}_{0}^{2}\backslash\left(\iif_{\eta}\cup\iif_{x}\cup\iif_{\tilde{x}}\right)$
	defined above does not even extend to a $b$-map $\hat{D}_{0}^{2}\to\hat{M}_{0}^{2}$.
	Nevertheless, $\varphi$ does lift canonically to a diffeomorphism
	from the iterated blow-up $\left[\hat{D}_{0}^{2};\ff_{0}\cap\iif;\left(\iif\cap\lf\right)\cup\left(\iif\cap\rf\right)\right]$
	of $\hat{D}_{0}^{2}$ to the simple blow-up $\left[\hat{M}_{0}^{2};\iif_{x}\cap\iif_{\tilde{x}}\right]$
	of $\hat{M}_{0}^{2}$. We need this only to observe that $\varphi$
	determines (by pull-back) an identification between the space of extendible
	distributions on $\hat{M}_{0}^{2}$ and the space of extendible distributions
	on $\hat{D}_{0}^{2}$. In particular, an extendible distribution $a\left(x,\tilde{x},\eta\right)$
	on $\hat{M}_{0}^{2}$ vanishes to infinite order at $\iif_{\eta}\cup\iif_{x}\cup\iif_{\tilde{x}}$
	if and only if its pull-back $\left(\varphi^{*}a\right)\left(\tau,\tilde{\tau},\eta\right)=a\left(\tau\left\langle \eta\right\rangle ^{-1},\tilde{\tau}\left\langle \eta\right\rangle ^{-1},\eta\right)$,
	seen as an extendible distribution on $\hat{D}_{0}^{2}$, vanishes
	to infinite order at $\iif$. Moreover, from the index sets computations
	above, a distribution $a\left(x,\tilde{x},\eta\right)$ on $\hat{M}_{0}^{2}$
	is in $\mathcal{A}_{\phg}^{\left(\mathcal{E}_{\lf},\mathcal{E}_{\rf},\mathcal{E}_{\ff_{0}}-1,\infty,\infty,\infty\right)}\left(\hat{M}_{0}^{2}\right)$
	if and only if its pull-back $\varphi^{*}a$ is in $\mathcal{A}_{\phg}^{\left(\mathcal{E}_{\lf},\mathcal{E}_{\rf},\mathcal{E}_{\ff_{0}}-1,\infty\right)}\left(\hat{D}_{0}^{2}\right)$.
	
	Why does this help us? Consider again our symbols
	\begin{align*}
		a\left(x,\tilde{x},\eta\right) & \in\mathcal{A}_{\phg}^{\left(\mathcal{E}_{\lf},\mathcal{E}_{\rf},\mathcal{E}_{\ff_{0}}-1,\infty,\infty,\infty\right)}\left(\hat{M}_{0}^{2}\right)\\
		b\left(x,\tilde{x},\eta\right) & \in\mathcal{A}_{\phg}^{\left(\mathcal{F}_{\lf},\mathcal{F}_{\rf},\mathcal{F}_{\ff_{0}}-1,\infty,\infty,\infty\right)}\left(\hat{M}_{0}^{2}\right)
	\end{align*}
	and assume that $\Re\left(\mathcal{E}_{\rf}+\mathcal{F}_{\lf}\right)>-1$.
	We need to prove that the distribution
	\[
	c\left(x,\tilde{x},\eta\right)=\int a\left(x,x',\eta\right)b\left(x',\tilde{x},\eta\right)dx'
	\]
	is well-defined and in $\mathcal{A}_{\phg}^{\left(\mathcal{E}_{\lf},\mathcal{F}_{\rf},\mathcal{E}_{\ff_{0}}+\mathcal{F}_{\ff_{0}}-1,\infty,\infty,\infty\right)}\left(\hat{M}_{0}^{2}\right)$.
	In order to do that, observe that the distribution $c'\left(\tau,\tilde{\tau},\eta\right)$
	given by
	\begin{align*}
		c'\left(\tau,\tilde{\tau},\eta\right) & =\int\left(\varphi^{*}a\right)\left(\tau,\tau',\eta\right)\left(\varphi^{*}b\right)\left(\tau',\tilde{\tau},\eta\right)d\tau'
	\end{align*}
	is well-defined and in $\mathcal{A}_{\phg}^{\left(\mathcal{E}_{\lf},\mathcal{F}_{\rf},\mathcal{E}_{\ff_{0}}+\mathcal{F}_{\ff_{0}}-2,\infty\right)}\left(\hat{D}_{0}^{2}\right)$.
	Indeed, since $\hat{D}_{0}^{2}=\overline{\mathbb{R}_{2}^{2}}\times\overline{\mathbb{R}^{n}}$,
	we have
	\begin{align*}
		\left(\varphi^{*}a\right)\left(\tau,\tau',\eta\right) & \in\mathcal{A}_{\phg}^{\left(\mathcal{E}_{\lf},\mathcal{E}_{\rf},\infty\right)}\left(\overline{\mathbb{R}_{2}^{2}}\right)\hat{\otimes}\mathcal{A}_{\phg}^{\mathcal{E}_{\ff_{0}}-1}\left(\overline{\mathbb{R}^{n}}\right)\\
		\left(\varphi^{*}b\right)\left(\tau,\tau',\eta\right) & \in\mathcal{A}_{\phg}^{\left(\mathcal{F}_{\lf},\mathcal{F}_{\rf},\infty\right)}\left(\overline{\mathbb{R}_{2}^{2}}\right)\hat{\otimes}\mathcal{A}_{\phg}^{\mathcal{F}_{\ff_{0}}-1}\left(\overline{\mathbb{R}^{n}}\right).
	\end{align*}
	Now it suffices to observe that if $f\left(x,\tilde{x}\right)\in\mathcal{A}_{\phg}^{\left(\mathcal{E}_{\lf},\mathcal{E}_{\rf},\infty\right)}\left(\overline{\mathbb{R}_{2}^{2}}\right)$
	and $g\left(x,\tilde{x}\right)\in\mathcal{A}_{\phg}^{\left(\mathcal{F}_{\lf},\mathcal{F}_{\rf},\infty\right)}\left(\overline{\mathbb{R}_{2}^{2}}\right)$,
	then the condition $\Re\left(\mathcal{E}_{\rf}+\mathcal{F}_{\lf}\right)>-1$
	ensures that the integral \[h\left(x,\tilde{x}\right)=\int f\left(x,x'\right)g\left(x',\tilde{x}\right)dx'\]
	is well-defined and in $\mathcal{A}_{\phg}^{\left(\mathcal{E}_{\lf},\mathcal{F}_{\rf},\infty\right)}\left(\overline{\mathbb{R}_{2}^{2}}\right)$.
	Using the universal property of $\hat{\otimes}$, we conclude that
	$c'\left(\tau,\tilde{\tau},\eta\right)$ is well-defined and in $\mathcal{A}_{\phg}^{\left(\mathcal{E}_{\lf},\mathcal{F}_{\rf},\infty\right)}\left(\overline{\mathbb{R}_{2}^{2}}\right)\hat{\otimes}\mathcal{A}_{\phg}^{\mathcal{E}_{\ff_{0}}+\mathcal{F}_{\ff_{0}}-2}\left(\overline{\mathbb{R}^{n}}\right)$.
	Now, observe that
	\begin{align*}
		c\left(x,\tilde{x},\eta\right) & :=c'\left(x\left\langle \eta\right\rangle ,\tilde{x}\left\langle \eta\right\rangle ,\eta\right)\left\langle \eta\right\rangle ^{-1}\\
		& =\int\left(\varphi^{*}a\right)\left(x\left\langle \eta\right\rangle ,\tau',\eta\right)\left(\varphi^{*}b\right)\left(\tau',\tilde{x}\left\langle \eta\right\rangle ,\eta\right)\left\langle \eta\right\rangle ^{-1}d\tau'\\
		& =\int\left(\varphi^{*}a\right)\left(x\left\langle \eta\right\rangle ,x'\left\langle \eta\right\rangle ,\eta\right)\left(\varphi^{*}b\right)\left(x'\left\langle \eta\right\rangle ,\tilde{x}\left\langle \eta\right\rangle ,\eta\right)dx'\\
		& =\int a\left(x,x',\eta\right)b\left(x',\tilde{x},\eta\right)dx'.
	\end{align*}
	Since $c'\left(\tau,\tilde{\tau},\eta\right)$ is in $\mathcal{A}_{\phg}^{\left(\mathcal{E}_{\lf},\mathcal{F}_{\rf},\mathcal{E}_{\ff_{0}}+\mathcal{F}_{\ff_{0}}-2,\infty\right)}\left(\hat{D}_{0}^{2}\right)$
	and $\left\langle \eta\right\rangle ^{-1}$ vanishes simply on $\ff_{0}$, the distribution
	$c'\left(\tau,\tilde{\tau},\eta\right)\left\langle \eta\right\rangle ^{-1}$
	is in $\mathcal{A}_{\phg}^{\left(\mathcal{E}_{\lf},\mathcal{F}_{\rf},\mathcal{E}_{\ff_{0}}+\mathcal{F}_{\ff_{0}}-1,\infty\right)}\left(\hat{D}_{0}^{2}\right)$.
	Moreover, by construction, \[c'\left(\tau,\tilde{\tau},\eta\right)\left\langle \eta\right\rangle ^{-1}=\left(\varphi^{*}c\right)\left(\tau\left\langle \eta\right\rangle ^{-1},\tilde{\tau}\left\langle \eta\right\rangle ^{-1},\eta\right).\]
	This ensures that $c\left(x,\tilde{x},\eta\right)$ is in $\mathcal{A}_{\phg}^{\left(\mathcal{E}_{\lf},\mathcal{F}_{\rf},\mathcal{E}_{\ff_{0}}+\mathcal{F}_{\ff_{0}}-1,\infty,\infty,\infty\right)}\left(\hat{M}_{0}^{2}\right)$
	as claimed.
	
	Now, consider the $0b$-interior $\circ$ $0b$-interior case. The
	proof is essentially unchanged, except that we need to blow $\hat{D}_{0}^{2}$
	up at the corner $\lf\cap\rf$. This blow-up is denoted by $\hat{D}_{0b}^{2}$,
	it has five boundary hyperfaces $\lf$, $\rf$, $\ff_{b}$, $\ff_{0}$,
	$\iif$ with $\ff_{b}$ created from the blow-up, and is represented
	graphically in Figure \ref{fig:Dhat}. Note that $\hat{D}_{0b}^{2}\backslash\iif=\left[\mathbb{R}_{2}^{2};0\right]\times\overline{\mathbb{R}^{n}}$.
	Moreover, the diffeomorphism $\varphi:\mathbb{R}_{2}^{2}\times\overline{\mathbb{R}^{n}}\to\hat{M}_{0}^{2}\backslash\left(\iif_{\eta}\cup\iif_{x}\cup\iif_{\tilde{x}}\right)$
	defined above lifts to a diffeomorphism
	\[
	\left[\mathbb{R}_{2}^{2};0\right]\times\overline{\mathbb{R}^{n}}\to\hat{M}_{0b}^{2}\backslash\left(\iif_{\eta}\cup\iif_{x}\cup\iif_{\tilde{x}}\right),
	\]
	and an extendible distribution $a\left(x,\tilde{x},\eta\right)$ is
	in \[\mathcal{A}_{\phg}^{\left(\mathcal{E}_{\lf},\mathcal{E}_{\rf},\mathcal{E}_{\ff_{b}}-1,\mathcal{E}_{\ff_{0}}-1,\infty,\infty,\infty\right)}\left(\hat{M}_{0b}^{2}\right)\]
	if and only if its pull-back $\varphi^{*}a\left(\tau,\tilde{\tau},\eta\right)=a\left(\tau\left\langle \eta\right\rangle ^{-1},\tilde{\tau}\left\langle \eta\right\rangle ^{-1},\eta\right)$
	lifts from $\hat{D}_{0}^{2}$ to an element of \[\mathcal{A}_{\phg}^{\left(\mathcal{E}_{\lf},\mathcal{E}_{\rf},\mathcal{E}_{\ff_{b}}-1,\mathcal{E}_{\ff_{0}}-1,\infty\right)}\left(\hat{D}_{0b}^{2}\right).\]
	Now, the condition $\Re\left(\mathcal{E}_{\rf}+\mathcal{F}_{\lf}\right)>-1$
	guarantees that the distribution
	\[
	c'\left(\tau,\tilde{\tau},\eta\right):=\int\left(\varphi^{*}a\right)\left(\tau,\tau',\eta\right)\left(\varphi^{*}b\right)\left(\tau',\tilde{\tau},\eta\right)d\tau'
	\]
	is well-defined and in $\mathcal{A}_{\phg}^{\left(\mathcal{G}_{\lf},\mathcal{G}_{\rf},\mathcal{G}_{\ff_{b}}-1,\mathcal{E}_{\ff_{0}}+\mathcal{F}_{\ff_{0}}-2,\infty\right)}\left(\hat{D}_{0}^{2}\right)$,
	where
	\begin{align*}
		\mathcal{G}_{\lf} & =\mathcal{E}_{\lf}\eunion\left(\mathcal{F}_{\lf}+\mathcal{E}_{\ff_{b}}\right)\\
		\mathcal{G}_{\rf} & =\mathcal{F}_{\rf}\eunion\left(\mathcal{E}_{\rf}+\mathcal{F}_{\ff_{b}}\right)\\
		\mathcal{G}_{\ff_{b}} & =\left(\mathcal{E}_{\ff_{b}}+\mathcal{F}_{\ff_{b}}\right)\eunion\left(\mathcal{E}_{\lf}+\mathcal{F}_{\rf}+1\right).
	\end{align*}
	To see this, we use again the fact that $\hat{D}_{0b}^{2}=\left[\overline{\mathbb{R}_{2}^{2}};0\right]\times\overline{\mathbb{R}^{n}}$
	so
	\[
	\mathcal{A}_{\phg}^{\left(\mathcal{E}_{\lf},\mathcal{E}_{\rf},\mathcal{E}_{\ff_{b}}-1,\mathcal{E}_{\ff_{0}}-1,\infty\right)}\left(\hat{D}_{0b}^{2}\right)=\mathcal{A}_{\phg}^{\left(\mathcal{E}_{\lf},\mathcal{E}_{\rf},\mathcal{E}_{\ff_{b}}-1,\infty\right)}\left(\left[\overline{\mathbb{R}_{2}^{2}};0\right]\right)\hat{\otimes}\mathcal{A}_{\phg}^{\mathcal{E}_{\ff_{0}}-1}\left(\overline{\mathbb{R}^{n}}\right).
	\]
	Then the proof then is essentially a consequence of the composition
	theorem for the residual large $b$-calculus on $\mathbb{R}_{1}^{1}$,
	i.e.\ the class $\Psi_{b,\mathcal{S}}^{-\infty,\left(\mathcal{E}_{\lf},\mathcal{E}_{\rf},\mathcal{E}_{\ff_{b}}\right)}\left(\mathbb{R}_{1}^{1}\right)$
	of operators on functions over $\overline{\mathbb{R}_{1}^{1}}$ with
	Schwartz kernels $f\left(x,\tilde{x}\right)d\tilde{x}$, where $f\left(x,\tilde{x}\right)$
	lifts from $\overline{\mathbb{R}_{2}^{2}}$ to an element of the space $\mathcal{A}_{\phg}^{\left(\mathcal{E}_{\lf},\mathcal{E}_{\rf},\mathcal{E}_{\ff_{b}}-1,\infty\right)}\left(\left[\overline{\mathbb{R}_{2}^{2}};0\right]\right)$.
	This composition theorem is proved many times in the literature, and
	we refer for example to \cite{MazzeoEdgeI, Hintz0calculus}. Finally,
	defining again $c\left(x,\tilde{x},\eta\right):=c'\left(x\left\langle \eta\right\rangle ,\tilde{x}\left\langle \eta\right\rangle ,\eta\right)\left\langle \eta\right\rangle ^{-1}$,
	we conclude that \[c \in \mathcal{A}_{\phg}^{\left(\mathcal{G}_{\lf},\mathcal{G}_{\rf},\mathcal{G}_{\ff_{b}}-1,\mathcal{E}_{\ff_{0}}+\mathcal{F}_{\ff_{0}}-1,\infty,\infty,\infty\right)}\left(\hat{M}_{0b}^{2}\right)\]
	as claimed. The statements about compositions between $0b$-interior
	and $0$-interior operators are proved similarly, except that one
	is reduced to a composition theorem between elements of the large
	$b$-calculus on $\mathbb{R}_{1}^{1}$ and very residual operators
	on $\mathbb{R}_{1}^{1}$. This composition theorem is again well-known
	from the literature, and is anyway an easy application of the pull-back
	/ push-forward technique.
\end{proof}

\subsubsection{\label{subsubsec:Local-composition-theorems-twisted}Local composition
	theorems for twisted operators}

We now prove the local versions of the composition results involving
twisted symbolic $0$-trace, $0$-Poisson and boundary operators.
Fix smooth families $\mathfrak{s}:\overline{\mathbb{R}^{n}}\to\mathfrak{gl}\left(M,\mathbb{C}\right)$,
$\mathfrak{t}:\overline{\mathbb{R}^{n}}\to\mathfrak{gl}\left(N,\mathbb{C}\right)$,
$\mathfrak{u}:\overline{\mathbb{R}^{n}}\to\mathfrak{gl}\left(P,\mathbb{C}\right)$
of matrices, each with a constant single eigenvalue.
\begin{thm}
	\label{thm:twisted-compositions-involving-boundary}Let $\mathcal{E}=\left(\mathcal{E}_{\of},\mathcal{E}_{\ff}\right)$,
	$\mathcal{F}=\left(\mathcal{F}_{\of},\mathcal{F}_{\ff}\right)$ and
	$\mathcal{G}$ be index sets. Let
	\begin{equation*}
		\begin{aligned}
			A &\in \hat{\Psi}_{0\tr,\mathcal{S}}^{-\infty,\mathcal{F},\mathfrak{s}}
			\left(\mathbb{R}_{1}^{n+1};\mathbb{R}^{n},\mathbb{C}^{M}\right),\\
			B &\in \hat{\Psi}_{0\po,\mathcal{S}}^{-\infty,\mathcal{E},\mathfrak{t}}
			\left(\mathbb{R}^{n},\mathbb{C}^{N};\mathbb{R}_{1}^{n+1}\right),\\
			Q &\in \Psi_{\phg,\mathcal{S}}^{-\mathcal{G},(\mathfrak{s},\mathfrak{t})}
			\left(\mathbb{R}^{n};\mathbb{C}^{M},\mathbb{C}^{N}\right),\\
			Q' &\in \Psi_{\phg,\mathcal{S}}^{-\mathcal{H},(\mathfrak{t},\mathfrak{u})}
			\left(\mathbb{R}^{n};\mathbb{C}^{N},\mathbb{C}^{P}\right).
		\end{aligned}
	\end{equation*}
	Then:
	\begin{enumerate}
		\item $Q'Q$ is well-defined and in \[\Psi_{\phg,\mathcal{S}}^{-\left(\mathcal{G}+\mathcal{H}\right),\left(\mathfrak{s},\mathfrak{u}\right)}\left(\mathbb{R}^{n};\mathbb{C}^{M},\mathbb{C}^{P}\right);\]
		\item $BQ$ is well-defined and in \[\hat{\Psi}_{0\po,\mathcal{S}}^{-\infty,\left(\mathcal{E}_{\of},\mathcal{H}\right),\mathfrak{s}}\left(\mathbb{R}^{n},\mathbb{C}^{N};\mathbb{R}_{1}^{n+1}\right),\]
		where $\left[\mathcal{H}\right]=\left[\mathcal{E}_{\ff}+\mathcal{G}\right]$;
		\item $QA$ is well-defined and in \[\hat{\Psi}_{0\tr,\mathcal{S}}^{-\infty,\left(\mathcal{F}_{\of},\mathcal{H}\right),\mathfrak{t}}\left(\mathbb{R}_{1}^{n+1};\mathbb{R}^{n},\mathbb{C}^{M}\right),\]
		where $\left[\mathcal{H}\right]=\left[\mathcal{F}_{\ff}+\mathcal{G}\right]$;
		\item assume that $N=M$ and $\mathfrak{s}=\mathfrak{t}$; then $BA$ is
		well-defined and in \[\hat{\Psi}_{0,\mathcal{S}}^{-\infty,\left(\mathcal{E}_{\of},\mathcal{F}_{\of},\mathcal{H}\right)}\left(\mathbb{R}_{1}^{n+1}\right),\]
		where $\left[\mathcal{H}\right]=\left[\mathcal{E}_{\ff}+\mathcal{F}_{\ff}\right]$;
		\item if $\Re\left(\mathcal{F}_{\of}+\mathcal{E}_{\of}\right)>-1$, then
		$AB$ is well-defined and in \[\Psi_{\phg,\mathcal{S}}^{-\mathcal{H},\left(\mathfrak{t},\mathfrak{s}\right)}\left(\mathbb{R}^{n};\mathbb{C}^{N},\mathbb{C}^{M}\right),\]
		where $\left[\mathcal{H}\right]=\left[\mathcal{E}_{\ff}+\mathcal{F}_{\ff}\right]$.
	\end{enumerate}
\end{thm}

\begin{thm}
	\label{thm:twisted-compositions-mixed}Let $\mathcal{E}=\left(\mathcal{E}_{\lf},\mathcal{E}_{\rf},\mathcal{E}_{\ff_{b}},\mathcal{E}_{\ff_{0}}\right)$
	and $\mathcal{F}=\left(\mathcal{F}_{\of},\mathcal{F}_{\ff}\right)$
	be index sets. Define also $\mathcal{E}'=\left(\mathcal{E}_{\lf},\mathcal{E}_{\rf},\mathcal{E}_{\ff_{0}}\right)$.
	Let
	\begin{equation*}
		\begin{aligned}
			A &\in \hat{\Psi}_{0\tr,\mathcal{S}}^{-\infty,\mathcal{F},\mathfrak{s}}
			\left(\mathbb{R}_{1}^{n+1};\mathbb{R}^{n},\mathbb{C}^{M}\right),\\
			B &\in \hat{\Psi}_{0\po,\mathcal{S}}^{-\infty,\mathcal{F},\mathfrak{s}}
			\left(\mathbb{R}^{n},\mathbb{C}^{M};\mathbb{R}_{1}^{n+1}\right),\\
			P_{0} &\in \hat{\Psi}_{0,\mathcal{S}}^{-\infty,\mathcal{E}'}
			\left(\mathbb{R}_{1}^{n+1}\right),\\
			P_{0b} &\in \hat{\Psi}_{0b,\mathcal{S}}^{-\infty,\mathcal{E}}
			\left(\mathbb{R}_{1}^{n+1}\right).
		\end{aligned}
	\end{equation*}
	If \textup{\emph{$\Re\left(\mathcal{E}_{\rf}+\mathcal{F}_{\of}\right)>-1$,
			then:}}
	\begin{enumerate}
		\item \textup{\emph{$P_{0}B$ is well-defined and in \[\hat{\Psi}_{0\po,\mathcal{S}}^{-\infty,\left(\mathcal{E}_{\lf},\mathcal{H}\right),\mathfrak{s}}\left(\mathbb{R}^{n},\mathbb{C}^{M};\mathbb{R}_{1}^{n+1}\right),\]
				where $\left[\mathcal{H}\right]=\left[\mathcal{F}_{\ff}+\mathcal{E}_{\ff_{0}}\right]$;}}
		\item \textup{\emph{$P_{0b}B$ is well-defined and in \[\hat{\Psi}_{0\po,\mathcal{S}}^{-\infty,\left(\mathcal{E}_{\lf}\eunion\left(\mathcal{F}_{\of}+\mathcal{E}_{\ff_{b}}\right),\mathcal{H}\right),\mathfrak{s}}\left(\mathbb{R}^{n},\mathbb{C}^{M};\mathbb{R}_{1}^{n+1}\right),\]
				where $\left[\mathcal{H}\right]=\left[\mathcal{F}_{\ff}+\mathcal{E}_{\ff_{0}}\right]$;}}
		\item \textup{\emph{$AP_{0}$ is well-defined and in \[\hat{\Psi}_{0\tr,\mathcal{S}}^{-\infty,\left(\mathcal{E}_{\rf},\mathcal{H}\right),\mathfrak{s}}\left(\mathbb{R}_{1}^{n+1};\mathbb{R}^{n},\mathbb{C}^{M}\right),\]
				where $\left[\mathcal{H}\right]=\left[\mathcal{F}_{\ff}+\mathcal{E}_{\ff_{0}}\right]$;}}
		\item \textup{\emph{$AP_{0b}$ is well-defined and in \[\hat{\Psi}_{0\tr,\mathcal{S}}^{-\infty,\left(\mathcal{E}_{\rf}\eunion\left(\mathcal{F}_{\of}+\mathcal{E}_{\ff_{b}}\right),\mathcal{H}\right),\mathfrak{s}}\left(\mathbb{R}_{1}^{n+1};\mathbb{R}^{n},\mathbb{C}^{M}\right),\]
				where $\left[\mathcal{H}\right]=\left[\mathcal{F}_{\ff}+\mathcal{E}_{\ff_{0}}\right]$.}}
	\end{enumerate}
\end{thm}

\begin{proof}
	The proof is essentially the same for all the cases above, so to exemplify
	the argument let's prove in detail that if $A\in\hat{\Psi}_{0\tr,\mathcal{S}}^{-\infty,\mathcal{F},\mathfrak{s}}\left(\mathbb{R}_{1}^{n+1};\mathbb{R}^{n},\mathbb{C}^{M}\right)$,
	$B\in\hat{\Psi}_{0\po,\mathcal{S}}^{-\infty,\mathcal{E},\mathfrak{t}}\left(\mathbb{R}^{n},\mathbb{C}^{N};\mathbb{R}_{1}^{n+1}\right)$
	and $\Re\left(\mathcal{F}_{\of}+\mathcal{E}_{\of}\right)>-1$ then
	$AB$ is well-defined and contained in $\Psi_{\phg,\mathcal{S}}^{-\mathcal{H},\left(\mathfrak{t},\mathfrak{s}\right)}\left(\mathbb{R}^{n};\mathbb{C}^{N},\mathbb{C}^{M}\right)$,
	where $\left[\mathcal{H}\right]=\left[\mathcal{E}_{\ff}+\mathcal{F}_{\ff}\right]$.
	Let
	\begin{align*}
		A & =\Op_{L}^{\tr}\left(\left\langle \eta\right\rangle ^{-\mathfrak{s}\left(y\right)}a\left(y;\tilde{x},\eta\right)\right)\\
		B & =\Op_{L}^{\po}\left(b\left(y;x,\eta\right)\left\langle \eta\right\rangle ^{\mathfrak{t}\left(y\right)}\right)
	\end{align*}
	for some symbols
	\begin{align*}
		a & \in S_{0\tr,\mathcal{S}}^{-\infty,\mathcal{F}}\left(\mathbb{R}^{n};\mathbb{R}_{1}^{n+1};\mathbb{C}^{M}\right)\\
		b & \in S_{0\po,\mathcal{S}}^{-\infty,\mathcal{E}}\left(\mathbb{R}^{n};\mathbb{R}_{1}^{n+1};\left(\mathbb{C}^{N}\right)^{*}\right).
	\end{align*}
	Write $\tilde{a}=\left\langle \eta\right\rangle ^{-\mathfrak{s}\left(y\right)}a\left(y;\tilde{x},\eta\right)$
	and $\tilde{b}=b\left(y;x,\eta\right)\left\langle \eta\right\rangle ^{\mathfrak{t}\left(y\right)}$.
	From Lemma \ref{lem:twisting-poisson-trace-are-polyhomogeneous-1},
	we know that 
	\begin{align*}
		\tilde{a} & \in S_{0\tr,\mathcal{S}}^{-\infty,\left(\mathcal{F}_{\of},\tilde{\mathcal{F}}_{\ff}\right)}\left(\mathbb{R}^{n};\mathbb{R}_{1}^{n+1}\right)\\
		\tilde{b} & \in S_{0\po,\mathcal{S}}^{-\infty,\left(\mathcal{E}_{\of},\tilde{\mathcal{E}}_{\ff}\right)}\left(\mathbb{R}^{n};\mathbb{R}_{1}^{n+1}\right)
	\end{align*}
	where $\tilde{\mathcal{E}}_{\ff},\tilde{\mathcal{F}}_{\ff}$ are obtained
	from $\mathcal{E}_{\ff},\mathcal{F}_{\ff}$ by increasing some of
	the logarithmic orders. Now, applying Point 4 of Theorem \ref{thm:compositions-involving-boundary},
	we know that since $\Re\left(\mathcal{F}_{\of}+\mathcal{E}_{\of}\right)>-1$,
	the composition $AB$ is well-defined and in $\Psi_{\phg,\mathcal{S}}^{-\left(\tilde{\mathcal{E}}_{\ff}+\tilde{\mathcal{F}}_{\ff}\right)}\left(\mathbb{R}^{n}\right)$.
	Now, write $C=AB=\Op_{L}\left(\tilde{c}\right)$, where $\tilde{c}\in S_{\mathcal{S}}^{-\left(\tilde{\mathcal{E}}_{\ff}+\tilde{\mathcal{F}}_{\ff}\right)}\left(\mathbb{R}^{n};\mathbb{R}^{n};\hom\left(\mathbb{C}^{N},\mathbb{C}^{M}\right)\right)$
	is determined modulo residual symbols as an asymptotic sum
	\[
	\tilde{c}\left(y;\eta\right)\sim\sum_{\alpha}\frac{1}{\alpha!}\left(D_{\eta}^{\alpha}\tilde{a}\right)*\left(\partial_{y}^{\alpha}\tilde{b}\right),
	\]
	where the contraction map $*$ in this case is
	\[
	\left(D_{\eta}^{\alpha}\tilde{a}\right)*\left(\partial_{y}^{\alpha}\tilde{b}\right)=\int\left(D_{\eta}^{\alpha}\tilde{a}\right)\left(y;\tilde{x},\eta\right)\left(\partial_{y}^{\alpha}\tilde{b}\right)\left(y;\tilde{x},\eta\right)d\tilde{x}.
	\]
	We want to prove that in fact
	\[
	\tilde{c}\in S_{\phg,\mathcal{S}}^{-\mathcal{H},\left(\mathfrak{t},\mathfrak{s}\right)}\left(\mathbb{R}^{n};\mathbb{R}^{n};\hom\left(\mathbb{C}^{N},\mathbb{C}^{M}\right)\right),
	\]
	where $\mathcal{H}$ is an index set with $\left[\mathcal{H}\right]=\left[\mathcal{E}_{\ff}+\mathcal{F}_{\ff}\right]$.
	Since
	\begin{align*}
		\tilde{a}\left(y;\tilde{x},\eta\right) & \in S_{0\tr,\mathcal{S}}^{-\infty,\mathcal{F},\mathfrak{s}}\left(\mathbb{R}^{n};\mathbb{R}_{1}^{n+1};\mathbb{C}^{M}\right)\\
		\tilde{b}\left(y;x,\eta\right) & \in S_{0\po,\mathcal{S}}^{-\infty,\mathcal{E},\mathfrak{t}}\left(\mathbb{R}^{n};\mathbb{R}_{1}^{n+1};\left(\mathbb{C}^{N}\right)^{*}\right),
	\end{align*}
	from Lemma \ref{lem:differentiation-of-twisted-symbols} we obtain
	that
	\begin{align*}
		D_{\eta}^{\alpha}\tilde{a}=\left\langle \eta\right\rangle ^{-\mathfrak{s}}a^{\left(\alpha\right)}, & a^{\left(\alpha\right)}\in S_{0\tr,\mathcal{S}}^{-\infty,\left(\mathcal{F}_{\of},\mathcal{F}_{\ff}+\left|\alpha\right|\right)}\left(\mathbb{R}^{n};\mathbb{R}_{1}^{n+1};\mathbb{C}^{M}\right)\\
		\partial_{y}^{\alpha}\tilde{b}=b^{\left(\alpha\right)}\left\langle \eta\right\rangle ^{\mathfrak{t}}, & b^{\left(\alpha\right)}\in S_{0\tr,\mathcal{S}}^{-\infty,\left(\mathcal{E}_{\of},\tilde{\mathcal{E}}_{\ff}^{\left(\alpha\right)}\right)}\left(\mathbb{R}^{n};\mathbb{R}_{1}^{n+1};\left(\mathbb{C}^{N}\right)^{*}\right),
	\end{align*}
	where $\tilde{\mathcal{E}}_{\ff}^{\left(\alpha\right)}$ is an index
	set obtained from $\mathcal{E}_{\ff}$ by increasing some of the logarithmic
	orders. From the proof of Point 4 of Theorem \ref{thm:compositions-involving-boundary},
	the condition $\Re\left(\mathcal{F}_{\of}+\mathcal{E}_{\of}\right)>-1$
	guarantees that
	\[
	c^{\left(\alpha\right)}:=a^{\left(\alpha\right)}*b^{\left(\alpha\right)}\in S_{\phg,\mathcal{S}}^{-\left(\tilde{\mathcal{E}}_{\ff}+\mathcal{F}_{\ff}+\left|\alpha\right|\right)}\left(\mathbb{R}^{n};\mathbb{R}_{1}^{n+1};\left(\mathbb{C}^{N}\right)^{*}\right).
	\]
	Define
	\[
	\mathcal{H}=\left(\mathcal{E}_{\ff}+\mathcal{F}_{\ff}\right)\cup\bigcup_{\left|\alpha\right|>0}\left(\tilde{\mathcal{E}}_{\ff}^{\left(\alpha\right)}+\mathcal{F}_{\ff}+\left|\alpha\right|\right).
	\]
	Then $\mathcal{H}$ is an index set with $\left[\mathcal{H}\right]=\left[\mathcal{E}_{\ff}+\mathcal{F}_{\ff}\right]$.
	From the formula
	\[
	\tilde{c}\left(y;\eta\right)\sim\sum_{\alpha}\frac{1}{\alpha!}\left\langle \eta\right\rangle ^{-\mathfrak{s}}c^{\left(\alpha\right)}\left(y;\eta\right)\left\langle \eta\right\rangle ^{\mathfrak{t}},
	\]
	it follows by asymptotic completeness that 
	\[
	\tilde{c}\in S_{\phg,\mathcal{S}}^{-\mathcal{H},\left(\mathfrak{t},\mathfrak{s}\right)}\left(\mathbb{R}^{n};\mathbb{R}^{n};\hom\left(\mathbb{C}^{N},\mathbb{C}^{M}\right)\right),
	\]
	concluding the proof.
\end{proof}

\subsubsection{\label{subsubsec:Global-composition-theorems}Global composition
	theorems}

We now discuss the global composition theorems. Unfortunately, we
will only be able to prove global analogues of the local composition
results which do not involve $0b$-interior operators. The author
strongly believes that all the local composition theorems mentioned
above should extend to the global setting, but the proof below does
not work straightforwardly when $0b$-interior operators are involved.
A slightly weaker result, involving the composition $0b$-interior
$\circ$ $0b$-interior, will be proved at the end of the paragraph.

We will now prove the following theorems:
\begin{thm}
	\label{thm:global-compositions-involving-boundary}Let $\mathcal{E}=\left(\mathcal{E}_{\of},\mathcal{E}_{\ff}\right)$,
	$\mathcal{F}=\left(\mathcal{F}_{\of},\mathcal{F}_{\ff}\right)$ and
	$\mathcal{G}$ be index sets. Let $A\in\hat{\Psi}_{0\tr}^{-\infty,\mathcal{F}}\left(X,\partial X\right)$,
	$B\in\hat{\Psi}_{0\po}^{-\infty,\mathcal{E}}\left(\partial X,X\right)$
	and $Q\in\Psi_{\phg}^{-\mathcal{G}}\left(\partial X\right)$. Then:
	\begin{enumerate}
		\item $BQ$ is well-defined and in $\hat{\Psi}_{0\po}^{-\infty,\left(\mathcal{E}_{\of},\mathcal{E}_{\ff}+\mathcal{G}\right)}\left(\partial X,X\right)$;
		\item $QA$ is well-defined and in $\hat{\Psi}_{0\tr}^{-\infty,\left(\mathcal{F}_{\of},\mathcal{F}_{\ff}+\mathcal{G}\right)}\left(X,\partial X\right)$;
		\item $BA$ is well-defined and in $\hat{\Psi}_{0}^{-\infty,\left(\mathcal{E}_{\of},\mathcal{F}_{\of},\mathcal{E}_{\ff}+\mathcal{F}_{\ff}\right)}\left(X\right)$;
		\item if $\Re\left(\mathcal{F}_{\of}+\mathcal{E}_{\of}\right)>-1$, then
		$AB$ is well-defined and in $\Psi_{\phg}^{-\left(\mathcal{E}_{\ff}+\mathcal{F}_{\ff}\right)}\left(\partial X\right)$.
	\end{enumerate}
	Moreover, if $\left[\mathcal{E}_{\ff}\right]=m_{1}$, $\left[\mathcal{F}_{\ff}\right]=m_{2}$
	and $\left[\mathcal{G}\right]=m_{3}$ for $m_{1},m_{2},m_{3}\in\mathbb{C}$,
	then for every $\eta\in T^{*}\partial X\backslash0$, we have
	\begin{align*}
		\hat{N}_{\eta}\left(BQ\right) & =\hat{N}_{\eta}\left(B\right)\sigma_{\eta}\left(Q\right)\\
		\hat{N}_{\eta}\left(QA\right) & =\sigma_{\eta}\left(Q\right)\hat{N}_{\eta}\left(A\right)\\
		\hat{N}_{\eta}\left(BA\right) & =\hat{N}_{\eta}\left(B\right)\hat{N}_{\eta}\left(A\right)\\
		\sigma_{\eta}\left(AB\right) & =\hat{N}_{\eta}\left(A\right)\hat{N}_{\eta}\left(B\right).
	\end{align*}
	
\end{thm}

\begin{thm}
	\label{thm:global-compositions-mixed}Let $\mathcal{E}=\left(\mathcal{E}_{\lf},\mathcal{E}_{\rf},\mathcal{E}_{\ff_{0}}\right)$
	and $\mathcal{F}=\left(\mathcal{F}_{\of},\mathcal{F}_{\ff}\right)$
	be index sets. Let
	\begin{equation*}
		\begin{aligned}
			A &\in \hat{\Psi}_{0\tr}^{-\infty,\mathcal{F}}
			\left(X,\partial X\right),\\
			B &\in \hat{\Psi}_{0\po}^{-\infty,\mathcal{F}}
			\left(\partial X,X\right),\\
			P &\in \hat{\Psi}_{0}^{-\infty,\mathcal{E}}
			\left(X\right).
		\end{aligned}
	\end{equation*}
	\begin{enumerate}
		\item \textup{\emph{If $\Re\left(\mathcal{E}_{\rf}+\mathcal{F}_{\of}\right)>-1$,
				then $PB$ is well-defined and in $\hat{\Psi}_{0\po}^{-\infty,\left(\mathcal{E}_{\lf},\mathcal{F}_{\ff}+\mathcal{E}_{\ff_{0}}\right)}\left(\partial X,X\right)$;}}
		\item \textup{\emph{if $\Re\left(\mathcal{F}_{\of}+\mathcal{E}_{\lf}\right)>-1$,
				then $AP$ is well-defined and in $\hat{\Psi}_{0\tr}^{-\infty,\left(\mathcal{E}_{\rf},\mathcal{F}_{\ff}+\mathcal{E}_{\ff_{0}}\right)}\left(X,\partial X\right)$.}}
	\end{enumerate}
	Moreover, if $\left[\mathcal{E}_{\ff_{0}}\right]=m_{1}$ and $\left[\mathcal{F}_{\ff_{b}}\right]=m_{2}$
	for $m_{1},m_{2}\in\mathbb{C}$, then for every $\eta\in T^{*}\partial X\backslash0$,
	we have
	\begin{align*}
		\hat{N}_{\eta}\left(PB\right) & =\hat{N}_{\eta}\left(P\right)\hat{N}_{\eta}\left(B\right)\\
		\hat{N}_{\eta}\left(AP\right) & =\hat{N}_{\eta}\left(A\right)\hat{N}_{\eta}\left(P\right).
	\end{align*}
	
\end{thm}

\begin{thm}
	\label{thm:compositions-involving-interior-1}Let\textup{\emph{ }}$\mathcal{E}=\left(\mathcal{E}_{\lf},\mathcal{E}_{\rf},\mathcal{E}_{\ff_{0}}\right)$
	and $\mathcal{F}=\left(\mathcal{F}_{\lf},\mathcal{F}_{\rf},\mathcal{F}_{\ff_{0}}\right)$.
	Let $P\in\hat{\Psi}_{0}^{-\infty,\mathcal{E}}\left(X\right)$ and
	$Q\in\hat{\Psi}_{0}^{-\infty,\mathcal{F}}\left(X\right)$.\textup{\emph{
			If $\Re\left(\mathcal{E}_{\rf}+\mathcal{F}_{\lf}\right)>-1$, then
			$PQ$ is well-defined and in $\hat{\Psi}_{0}^{-\infty,\left(\mathcal{E}_{\lf},\mathcal{F}_{\rf},\mathcal{E}_{\ff_{0}}+\mathcal{F}_{\ff_{0}}\right)}\left(X\right)$.
			Moreover, if $\left[\mathcal{E}_{\ff_{0}}\right]=m_{1}$ and $\left[\mathcal{F}_{\ff_{0}}\right]=m_{2}$
			for $m_{1},m_{2}\in\mathbb{C}$, then for every $\eta\in T^{*}\partial X\backslash0$
			we have $\hat{N}_{\eta}\left(PQ\right)=\hat{N}_{\eta}\left(P\right)\hat{N}_{\eta}\left(Q\right)$.}}
\end{thm}

Fix $\boldsymbol{E}\to\partial X$ and $\boldsymbol{F}\to\partial X$
smooth vector bundles, and let $\boldsymbol{\mathfrak{s}}:\boldsymbol{E}\to\boldsymbol{E}$,
$\boldsymbol{\mathfrak{t}}:\boldsymbol{F}\to\boldsymbol{F}$, and
$\boldsymbol{\mathfrak{u}}:\boldsymbol{W}\to\boldsymbol{W}$ be smooth
endomorphisms with constant eigenvalues.
\begin{thm}
	\label{thm:global-twisted-compositions-involving-boundary}Let $\mathcal{E}=\left(\mathcal{E}_{\of},\mathcal{E}_{\ff}\right)$,
	$\mathcal{F}=\left(\mathcal{F}_{\of},\mathcal{F}_{\ff}\right)$ and
	$\mathcal{G}$ be index sets. Let
	\begin{equation*}
		\begin{aligned}
			\boldsymbol{A} &\in \hat{\Psi}_{0\tr}^{-\infty,
				\left(\mathcal{F}_{\of},\left[\mathcal{F}_{\ff}\right]\right),
				\boldsymbol{\mathfrak{s}}}
			\left(X;\partial X,\boldsymbol{E}\right),\\
			\boldsymbol{B} &\in \hat{\Psi}_{0\po}^{-\infty,
				\left(\mathcal{E}_{\of},\left[\mathcal{E}_{\ff}\right]\right),
				\boldsymbol{\mathfrak{t}}}
			\left(\partial X,\boldsymbol{F};X\right),\\
			\boldsymbol{Q} &\in \Psi_{\phg}^{-\left[\mathcal{G}\right],
				\left(\boldsymbol{\mathfrak{s}},\boldsymbol{\mathfrak{t}}\right)}
			\left(\partial X;\boldsymbol{E},\boldsymbol{F}\right),\\
			\boldsymbol{Q}' &\in \Psi_{\phg}^{-\left[\mathcal{H}\right],
				\left(\boldsymbol{\mathfrak{t}},\boldsymbol{\mathfrak{u}}\right)}
			\left(\partial X;\boldsymbol{F},\boldsymbol{W}\right).
		\end{aligned}
	\end{equation*}
	Then:
	\begin{enumerate}
		\item $\boldsymbol{Q}'\boldsymbol{Q}$ is well-defined and in $\Psi_{\phg}^{-\left[\mathcal{G}+\mathcal{H}\right],\left(\boldsymbol{\mathfrak{s}},\boldsymbol{\mathfrak{u}}\right)}\left(\partial X;\boldsymbol{E},\boldsymbol{W}\right)$;
		\item $\boldsymbol{B}\boldsymbol{Q}$ is well-defined and in $\hat{\Psi}_{0\po}^{-\infty,\left(\mathcal{E}_{\of},\left[\mathcal{E}_{\ff}+\mathcal{G}\right]\right),\boldsymbol{\mathfrak{s}}}\left(\partial X,\boldsymbol{E};X\right)$;
		\item $\boldsymbol{Q}\boldsymbol{A}$ is well-defined and in $\hat{\Psi}_{0\tr}^{-\infty,\left(\mathcal{F}_{\of},\left[\mathcal{F}_{\ff}+\mathcal{G}\right]\right),\boldsymbol{\mathfrak{t}}}\left(X;\partial X,\boldsymbol{F}\right)$;
		\item assume that $\boldsymbol{E}=\boldsymbol{F}$ and $\boldsymbol{\mathfrak{s}}=\boldsymbol{\mathfrak{t}}$;
		then $\boldsymbol{B}\boldsymbol{A}$ is well-defined and in $\hat{\Psi}_{0}^{-\infty,\left(\mathcal{E}_{\of},\mathcal{F}_{\of},\mathcal{H}\right)}\left(X\right)$,
		where $\mathcal{H}$ is an index set with $\left[\mathcal{H}\right]=\left[\mathcal{E}_{\ff}+\mathcal{F}_{\ff}\right]$;
		\item if $\Re\left(\mathcal{F}_{\of}+\mathcal{E}_{\of}\right)>-1$, then
		$\boldsymbol{A}\boldsymbol{B}$ is well-defined and in $\Psi_{\phg}^{-\left[\mathcal{E}_{\ff}+\mathcal{F}_{\ff}\right],\left(\boldsymbol{\mathfrak{t}},\boldsymbol{\mathfrak{s}}\right)}\left(\partial X;\boldsymbol{F},\boldsymbol{E}\right)$.
	\end{enumerate}
	Moreover, if $\left[\mathcal{E}_{\ff}\right]=m_{1}$, $\left[\mathcal{F}_{\ff}\right]=m_{2}$
	and $\left[\mathcal{G}\right]=m_{3}$ for $m_{1},m_{2},m_{3}\in\mathbb{C}$,
	then for every $\eta\in T^{*}\partial X\backslash0$ we have
	\begin{align*}
		\hat{N}_{\eta}\left(\boldsymbol{B}\boldsymbol{Q}\right) & =\hat{N}_{\eta}\left(\boldsymbol{B}\right)\sigma_{\eta}\left(\boldsymbol{Q}\right)\\
		\hat{N}_{\eta}\left(\boldsymbol{Q}\boldsymbol{A}\right) & =\sigma_{\eta}\left(\boldsymbol{Q}\right)\hat{N}_{\eta}\left(\boldsymbol{A}\right)\\
		\hat{N}_{\eta}\left(\boldsymbol{B}\boldsymbol{A}\right) & =\hat{N}_{\eta}\left(\boldsymbol{B}\right)\hat{N}_{\eta}\left(\boldsymbol{A}\right)\\
		\sigma_{\eta}\left(\boldsymbol{A}\boldsymbol{B}\right) & =\hat{N}_{\eta}\left(\boldsymbol{A}\right)\hat{N}_{\eta}\left(\boldsymbol{B}\right).
	\end{align*}
	
\end{thm}

\begin{thm}
	\label{thm:global-twisted-compositions-mixed}Let $\mathcal{E}=\left(\mathcal{E}_{\lf},\mathcal{E}_{\rf},\mathcal{E}_{\ff_{0}}\right)$
	and $\mathcal{F}=\left(\mathcal{F}_{\of},\mathcal{F}_{\ff}\right)$
	be index sets. Let
	\begin{equation*}
		\begin{aligned}
			\boldsymbol{A} &\in \hat{\Psi}_{0\tr}^{-\infty,
				\left(\mathcal{F}_{\of},\left[\mathcal{F}_{\ff}\right]\right),
				\boldsymbol{\mathfrak{s}}}
			\left(X;\partial X,\boldsymbol{E}\right),\\
			\boldsymbol{B} &\in \hat{\Psi}_{0\po}^{-\infty,
				\left(\mathcal{F}_{\of},\left[\mathcal{F}_{\ff}\right]\right),
				\boldsymbol{\mathfrak{s}}}
			\left(\partial X,\boldsymbol{E};X\right),\\
			P &\in \hat{\Psi}_{0}^{-\infty,\mathcal{E}}
			\left(X\right).
		\end{aligned}
	\end{equation*}
	\begin{enumerate}
		\item If \textup{\emph{$\Re\left(\mathcal{E}_{\rf}+\mathcal{F}_{\of}\right)>-1$,
				then $P\boldsymbol{B}$ is well-defined and in $\hat{\Psi}_{0\po}^{-\infty,\left(\mathcal{E}_{\lf},\left[\mathcal{F}_{\ff}+\mathcal{E}_{\ff_{0}}\right]\right),\boldsymbol{\mathfrak{s}}}\left(\partial X,\boldsymbol{E};X\right)$;}}
		\item \textup{\emph{If $\Re\left(\mathcal{F}_{\of}+\mathcal{E}_{\lf}\right)>-1$,
				then $\boldsymbol{A}P$ is well-defined and in $\hat{\Psi}_{0\tr}^{-\infty,\left(\mathcal{E}_{\rf},\left[\mathcal{F}_{\ff}+\mathcal{E}_{\ff_{0}}\right]\right),\boldsymbol{\mathfrak{s}}}\left(X;\partial X,\boldsymbol{E}\right)$.}}
	\end{enumerate}
	Moreover, if $\left[\mathcal{E}_{\ff_{0}}\right]=m_{1}$ and $\left[\mathcal{F}_{\ff}\right]=m_{2}$
	for $m_{1},m_{2}\in\mathbb{C}$, then for every $\eta\in T^{*}\partial X\backslash0$,
	we have
	\begin{align*}
		\hat{N}_{\eta}\left(P\boldsymbol{B}\right) & =\hat{N}_{\eta}\left(P\right)\hat{N}_{\eta}\left(\boldsymbol{B}\right)\\
		\hat{N}_{\eta}\left(\boldsymbol{A}P\right) & =\hat{N}_{\eta}\left(\boldsymbol{A}\right)\hat{N}_{\eta}\left(P\right).
	\end{align*}
	
\end{thm}

\begin{proof}
	Given the local composition Theorems \ref{thm:compositions-involving-boundary},
	\ref{thm:compositions-mixed}, \ref{thm:compositions-involving-interior},
	\ref{thm:twisted-compositions-involving-boundary}, and \ref{thm:twisted-compositions-mixed},
	the proof is almost identical to the standard proof for pseudodifferential
	operators. We exemplify the argument for the case $0$-trace $\circ$
	$0$-interior.
	
	Let $A\in\hat{\Psi}_{0\tr}^{-\infty,\mathcal{F}}\left(X,\partial X\right)$
	and $P\in\hat{\Psi}_{0}^{-\infty,\mathcal{E}}\left(X\right)$, with
	$\mathcal{E}=\left(\mathcal{E}_{\lf},\mathcal{E}_{\rf},\mathcal{E}_{\ff_{0}}\right)$
	and $\mathcal{F}=\left(\mathcal{F}_{\of},\mathcal{F}_{\ff}\right)$.
	Assume that $\Re\left(\mathcal{F}_{\of}+\mathcal{E}_{\lf}\right)>-1$.
	We want to show that $AP\in\hat{\Psi}_{0\tr}^{-\infty,\left(\mathcal{E}_{\rf},\mathcal{F}_{\ff}+\mathcal{E}_{\ff_{0}}\right)}\left(X,\partial X\right)$.
	First of all, we check that the composition $AP$ is well-defined.
	By Points 1 and 3 of Theorem \ref{thm:mapping-properties-on-phg-untwisted},
	$P$ induces a continuous linear map
	\[
	P:\dot{C}^{\infty}\left(X\right)\to\mathcal{A}_{\phg}^{\mathcal{E}_{\lf}}\left(X\right),
	\]
	and since $\Re\left(\mathcal{F}_{\of}+\mathcal{E}_{\lf}\right)>-1$,
	$A$ induces a continuous linear map
	\[
	A:\mathcal{A}_{\phg}^{\mathcal{E}_{\lf}}\left(X\right)\to C^{\infty}\left(\partial X\right).
	\]
	This ensures that $AP$ is well-defined. Now, choose an atlas $\left\{ \left(U_{i},\varphi_{i}\right)\right\} $
	of $\partial X$ such that, for every pair $i,j$ for which $U_{i}\cap U_{j}\not=\emptyset$,
	the union $U_{i}\cup U_{j}$ is contained in the domain $\tilde{U}_{ij}$
	of a chart $\tilde{\varphi}_{ij}$. Choose also a boundary defining
	function $x$ for $X$, compatible with the auxiliary vector field
	$V$ chosen to define the classes $\hat{\Psi}_{0\tr}^{-\infty,\mathcal{F}}\left(X,\partial X\right)$
	and $\hat{\Psi}_{0}^{-\infty,\mathcal{E}}\left(X\right)$. Using the
	vector field $V$ to define a collar $[0,\varepsilon)\times\partial X\to X$,
	call $U_{i}'=[0,\varepsilon/2)\times U_{i}$ interpreted as an open
	subset of $X$. Note that the family $\left\{ U_{i}'\times U_{i}'\right\} $
	covers the locus $\partial\Delta$ in $X^{2}$, and the family $\left\{ U_{i}\times U_{i}'\right\} $
	covers the locus $\partial\Delta$ in $\partial X\times X$. Now,
	decompose
	\begin{align*}
		A & =\sum_{i}A_{i}+R_{A}\\
		P & =\sum_{i}P_{i}+R_{P},
	\end{align*}
	where $A_{i}$ is compactly supported in $U_{i}\times U_{i}'$, $P_{i}$
	is compactly supported in $U_{i}'\times U_{i}'$, $R_{A}$ is residual
	trace, and $R_{P}$ is very residual. More precisely,
	\begin{align*}
		R_{A} & \in\Psi_{\tr}^{-\infty,\mathcal{F}_{\of}}\left(X,\partial X\right)\\
		R_{P} & \in\Psi^{-\infty,\left(\mathcal{E}_{\lf},\mathcal{E}_{\rf}\right)}\left(X\right).
	\end{align*}
	We can write the composition as
	
	\[
	AP=\sum_{ij}A_{i}P_{j}+\sum_{i}\left(A_{i}R_{P}+R_{A}P_{i}\right)+R_{A}R_{P}.
	\]
	
	We now want to show that the term
	\[
	\sum_{i}\left(A_{i}R_{P}+R_{A}P_{i}\right)+R_{A}R_{P}
	\]
	is a residual trace operator in $\Psi_{\tr}^{-\infty,\mathcal{E}_{\rf}}\left(X,\partial X\right)$.
	To see this, decompose
	\begin{align*}
		\Psi^{-\infty,\left(\mathcal{E}_{\lf},\mathcal{E}_{\rf}\right)}\left(X\right) & =\mathcal{A}_{\phg}^{\mathcal{E}_{\lf}}\left(X\right)\hat{\otimes}\mathcal{A}_{\phg}^{\mathcal{E}_{\rf}}\left(X;\mathcal{D}_{X}^{1}\right).
	\end{align*}
	By Point 1 of Theorem \ref{thm:mapping-properties-on-phg-untwisted},
	the condition $\Re\left(\mathcal{F}_{\of}+\mathcal{E}_{\lf}\right)>-1$
	guarantees that $A_{i}$ and $R_{A}$ induce continuous linear maps
	$\mathcal{A}_{\phg}^{\mathcal{E}_{\lf}}\left(X\right)\to C^{\infty}\left(\partial X\right)$.
	Therefore, the actions of $A_{i}$ and $R_{A}$ on the left determine
	continuous linear maps $\Psi^{-\infty,\left(\mathcal{E}_{\lf},\mathcal{E}_{\rf}\right)}\left(X\right)\to C^{\infty}\left(\partial X\right)\hat{\otimes}\mathcal{A}_{\phg}^{\mathcal{E}_{\rf}}\left(X;\mathcal{D}_{X}^{1}\right)$.
	The latter space is precisely $\Psi_{\tr}^{-\infty,\mathcal{E}_{\rf}}\left(X,\partial X\right)$.
	Similarly, to prove the claim for $R_{A}P_{i}$, taking adjoints with
	respect to a smooth positive density $\omega$ on $X$ and a positive
	density $\nu$ on $\partial X$, by Proposition \ref{prop:formal-adjoints}
	we have
	\begin{align*}
		P_{i}^{*} & \in\hat{\Psi}_{0}^{-\infty,\left(\overline{\mathcal{E}_{\rf}},\overline{\mathcal{E}_{\lf}},\overline{\mathcal{E}_{\ff_{0}}}\right)}\left(X\right)\\
		R_{A}^{*} & \in\Psi_{\po}^{-\infty,\overline{\mathcal{F}_{\of}}}\left(\partial X,X\right).
	\end{align*}
	The space $\Psi_{\po}^{-\infty,\overline{\mathcal{F}_{\of}}}\left(\partial X,X\right)$
	decomposes as $\mathcal{A}_{\phg}^{\overline{\mathcal{F}_{\of}}}\left(X\right)\hat{\otimes}C^{\infty}\left(\partial X;\mathcal{D}_{\partial X}^{1}\right)$.
	Now, by Point 3 of Theorem \ref{thm:mapping-properties-on-phg-untwisted},
	the condition $\Re\left(\mathcal{F}_{\of}+\mathcal{E}_{\lf}\right)>-1$
	implies that $P_{i}^{*}$ induces a continuous linear map $\mathcal{A}_{\phg}^{\overline{\mathcal{F}_{\of}}}\left(X\right)\to\mathcal{A}_{\phg}^{\overline{\mathcal{E}_{\rf}}}\left(X\right)$.
	Therefore, the action of $P_{i}^{*}$ on the left determines a continuous
	linear map $\mathcal{A}_{\phg}^{\overline{\mathcal{F}_{\of}}}\left(X\right)\hat{\otimes}C^{\infty}\left(\partial X;\mathcal{D}_{\partial X}^{1}\right)\to\mathcal{A}_{\phg}^{\overline{\mathcal{E}_{\of}}}\left(X\right)\hat{\otimes}C^{\infty}\left(\partial X;\mathcal{D}_{\partial X}^{1}\right)$.
	This shows that $P_{i}^{*}R_{A}^{*}\in\Psi_{\po}^{-\infty,\overline{\mathcal{E}_{\rf}}}\left(\partial X,X\right)$,
	so that taking adjoints again and applying Proposition \ref{prop:formal-adjoints},
	we have $R_{A}P_{i}\in\Psi_{\tr}^{-\infty,\mathcal{E}_{\rf}}\left(X,\partial X\right)$.
	
	It remains to prove that $A_{i}P_{j}\in\hat{\Psi}_{0\tr}^{-\infty,\left(\mathcal{E}_{\rf},\mathcal{F}_{\ff}+\mathcal{E}_{\ff_{0}}\right)}\left(X,\partial X\right)$.
	If $U_{i}\cap U_{j}=\emptyset$, then clearly $U_{i}'\cap U_{j}'=\emptyset$
	and therefore $A_{i}P_{j}=0$. If instead $U_{i}\cap U_{j}\not=\emptyset$,
	we can think of $A_{i},P_{j}$ in terms of the coordinate chart $\left(\tilde{U}_{ij},\tilde{\varphi}_{ij}\right)$
	and a boundary defining function $x$ for $X$ compatible with the
	chosen auxiliary vector field $V$ on the collar, as operators
	\begin{align*}
		A_{i} & \in\hat{\Psi}_{0\tr,\mathcal{S}}^{-\infty,\mathcal{F}}\left(\mathbb{R}_{1}^{n+1},\mathbb{R}^{n}\right)\\
		P_{j} & \in\hat{\Psi}_{0,\mathcal{S}}^{-\infty,\mathcal{E}}\left(\mathbb{R}_{1}^{n+1}\right).
	\end{align*}
	By Point 3 of Theorem \ref{thm:compositions-mixed}, we have $A_{i}P_{j}\in\hat{\Psi}_{0\tr,\mathcal{S}}^{-\infty,\left(\mathcal{E}_{\rf},\mathcal{F}_{\ff}+\mathcal{E}_{\ff_{0}}\right)}\left(\mathbb{R}_{1}^{n+1},\mathbb{R}^{n}\right)$.
	Moreover, $A_{i}P_{j}$ is compactly supported in $\tilde{U}_{ij}\times\tilde{U}_{ij}\times[0,\varepsilon)$,
	so it extends off to an operator in the space $\hat{\Psi}_{0\tr}^{-\infty,\left(\mathcal{E}_{\rf},\mathcal{F}_{\ff}+\mathcal{E}_{\ff_{0}}\right)}\left(X,\partial X\right)$.
	
	Finally, assume that $\left[\mathcal{F}_{\ff}\right]=m_{1}$ and $\left[\mathcal{E}_{\ff_{0}}\right]=m_{2}$.
	Choose coordinates near a point $p\in\partial X$, and write in these
	coordinates
	\begin{align*}
		A & \equiv\Op_{L}^{\tr}\left(a\right)\\
		P & \equiv\Op_{L}^{\inte}\left(p\right)
	\end{align*}
	in a neighborhood of $\left(p,p\right)\in\partial\Delta$, where $a\left(y;\tilde{x},\eta\right)\in S_{0\tr,\mathcal{S}}^{-\infty,\mathcal{F}}\left(\mathbb{R}^{n};\mathbb{R}_{1}^{n+1}\right)$
	and $p\left(y;x,\tilde{x},\eta\right)\in S_{0,\mathcal{S}}^{-\infty,\mathcal{E}}\left(\mathbb{R}^{n};\mathbb{R}_{2}^{n+2}\right)$.
	As explained in \S\ref{subsubsec:The-Bessel-family-of-symbolic-0-trace}
	and \S\ref{subsubsec:The-Bessel-family-of-interior}, in these coordinates
	the Bessel families of $A$ and $P$ at the point $p$ are defined
	as
	\begin{align*}
		\hat{N}_{\eta}\left(A\right) & =\lim_{t\to0^{+}}t^{-m_{1}}a\left(0;t\tilde{x},t^{-1}\eta\right)td\tilde{x}\\
		\hat{N}_{\eta}\left(P\right) & =\lim_{t\to0^{+}}t^{-m_{2}}p\left(0;tx,t\tilde{x},t^{-1}\eta\right)td\tilde{x}.
	\end{align*}
	On the other hand since $AP\in\hat{\Psi}_{0\tr}^{-\infty,\left(\mathcal{E}_{\rf},\mathcal{F}_{\ff}+\mathcal{E}_{\ff_{0}}\right)}\left(X,\partial X\right)$
	and $\left[\mathcal{F}_{\ff}+\mathcal{E}_{\ff_{0}}\right]=m_{1}+m_{2}$,
	we have
	\[
	\hat{N}_{\eta}\left(AP\right)=\lim_{t\to0^{+}}t^{-\left(m_{1}+m_{2}\right)}q\left(0;t\tilde{x},t^{-1}\eta\right)td\tilde{x}
	\]
	where $q\left(y;\tilde{x},\eta\right)\in S_{0\tr,\mathcal{S}}^{-\infty,\left(\mathcal{E}_{\rf},\mathcal{F}_{\ff}+\mathcal{E}_{\ff_{0}}\right)}\left(\mathbb{R}^{n};\mathbb{R}_{1}^{n+1}\right)$
	is a left-reduced $0$-trace symbol, such that $AP\equiv\Op_{L}^{\tr}\left(q\right)$
	in the chosen coordinates near the point $\left(p,p\right)$. Now,
	the symbol $q$ is determined, modulo residual trace symbols, as an
	asymptotic sum
	\[
	q\left(y;\tilde{x},\eta\right)\sim\sum_{\alpha}\frac{1}{\alpha!}\int\left(\partial_{\eta}^{\alpha}a\right)\left(y;x',\eta\right)\left(D_{\tilde{y}}^{\alpha}p\right)\left(y;x',\tilde{x},\eta\right)dx'.
	\]
	Taking the limit of $t^{-\left(m_{1}+m_{2}\right)}q\left(0;t\tilde{x},t^{-1}\eta\right)t$
	as $t\to0^{+}$, all the terms in the summation above go to zero except
	for the first one: therefore,
	\begin{align*}
		\hat{N}_{\eta}\left(AP\right) & =\left[\lim_{t\to0^{+}}t^{-\left(m_{1}+m_{2}\right)}q\left(0;t\tilde{x},t^{-1}\eta\right)t\right]d\tilde{x}\\
		& =\left[\lim_{t\to0^{+}}t^{-\left(m_{1}+m_{2}\right)}\int a\left(0;x',t^{-1}\eta\right)p\left(0;x',t\tilde{x},t^{-1}\eta\right)tdx'\right]d\tilde{x}\\
		& =\left[\lim_{t\to0^{+}}t^{-\left(m_{1}+m_{2}\right)}\int a\left(0;tx',t^{-1}\eta\right)p\left(0;tx',t\tilde{x},t^{-1}\eta\right)t^{2}dx'\right]d\tilde{x}\\
		& =\left[\int\left(\lim_{t\to0^{+}}t^{-m_{1}}a\left(0;tx',t^{-1}\eta\right)t\right)\left(\lim_{t\to0^{+}}t^{-m_{2}}p\left(0;tx',t\tilde{x},t^{-1}\eta\right)t\right)dx'\right]d\tilde{x}\\
		& =\hat{N}_{\eta}\left(A\right)\hat{N}_{\eta}\left(P\right).
	\end{align*}
\end{proof}
Unfortunately, the strategy used in the previous proof does not work
immediately if operators in the class $\hat{\Psi}_{0b}^{-\infty,\bullet}\left(X\right)$
are involved. Consider for example the case $0b$-interior $\circ$
$0b$-interior. If $P\in\hat{\Psi}_{0b}^{-\infty,\mathcal{E}}\left(X\right)$
and $Q\in\hat{\Psi}_{0b}^{-\infty,\mathcal{F}}\left(X\right)$ with
$\mathcal{E}=\left(\mathcal{E}_{\lf},\mathcal{E}_{\rf},\mathcal{E}_{\ff_{b}},\mathcal{E}_{\ff_{0}}\right)$
and $\mathcal{F}=\left(\mathcal{F}_{\lf},\mathcal{F}_{\rf},\mathcal{F}_{\ff_{b}},\mathcal{F}_{\ff_{0}}\right)$,
and $\Re\left(\mathcal{E}_{\rf}+\mathcal{F}_{\lf}\right)>-1$, we
would like to prove that $PQ\in\hat{\Psi}_{0b}^{-\infty,\mathcal{G}}\left(X\right)$,
where $\mathcal{G}=\left(\mathcal{G}_{\lf},\mathcal{G}_{\rf},\mathcal{G}_{\ff_{b}},\mathcal{G}_{\ff_{0}}\right)$
is defined as\emph{
	\begin{align*}
		\mathcal{G}_{\lf} & =\mathcal{E}_{\lf}\eunion\left(\mathcal{F}_{\lf}+\mathcal{E}_{\ff_{b}}\right)\\
		\mathcal{G}_{\rf} & =\mathcal{F}_{\rf}\eunion\left(\mathcal{E}_{\rf}+\mathcal{F}_{\ff_{b}}\right)\\
		\mathcal{G}_{\ff_{b}} & =\left(\mathcal{E}_{\ff_{b}}+\mathcal{F}_{\ff_{b}}\right)\eunion\left(\mathcal{E}_{\lf}+\mathcal{F}_{\rf}+1\right)\\
		\mathcal{G}_{\ff_{0}} & =\mathcal{E}_{\ff_{0}}+\mathcal{F}_{\ff_{0}}
	\end{align*}
}as one would expect from Point 1 of Theorem \ref{thm:compositions-involving-interior}.
Decompose $P=\sum_{i}P_{i}+R_{P}$ and $Q=\sum_{i}Q_{i}+R_{Q}$ as
in the previous proof. Now the residual terms $R_{P},R_{Q}$ are in
the large residual $b$-calculus: more precisely,
\begin{align*}
	R_{P} & \in\Psi_{b}^{-\infty,\left(\mathcal{E}_{\lf},\mathcal{E}_{\rf},\mathcal{E}_{\ff_{b}}\right)}\left(X\right)\\
	R_{Q} & \in\Psi_{b}^{-\infty,\left(\mathcal{F}_{\lf},\mathcal{F}_{\rf},\mathcal{F}_{\ff_{b}}\right)}\left(X\right).
\end{align*}
Writing $PQ=\sum_{ij}P_{i}Q_{j}+\sum_{i}\left(P_{i}R_{Q}+R_{P}Q_{i}\right)+R_{P}R_{Q}$
analogously to the previous proof, Point 1 of Theorem \ref{thm:compositions-involving-interior}
ensures that the terms $P_{i}Q_{j}$ are in $\hat{\Psi}_{0b}^{-\infty,\mathcal{G}}\left(X\right)$.
Moreover, by the composition theorem for the large $b$-calculus proved
in \cite{MazzeoEdgeI}, the term $R_{P}R_{Q}$ is in the space $\Psi_{b}^{-\infty,\left(\mathcal{G}_{\lf},\mathcal{G}_{\rf},\mathcal{G}_{\ff_{b}}\right)}\left(X\right)$
which is the residual part of $\hat{\Psi}_{0b}^{-\infty,\mathcal{G}}\left(X\right)$.
Now, again by Point 1 of Theorem \ref{thm:compositions-involving-interior},
we would expect that the terms $P_{i}R_{Q}$ and $R_{P}Q_{i}$ are
in $\Psi_{b}^{-\infty,\left(\mathcal{G}_{\lf},\mathcal{G}_{\rf},\mathcal{G}_{\ff_{b}}\right)}\left(X\right)$.
However, we cannot use the strategy used in the previous proof, because $\Psi_{b}^{-\infty,\left(\mathcal{G}_{\lf},\mathcal{G}_{\rf},\mathcal{G}_{\ff_{b}}\right)}\left(X\right)$
does not decompose as a completed projective tensor product of spaces
of polyhomogeneous functions on the product $X^{2}$. A similar problem
arises when trying to compose symbolic $0$-trace and $0$-Poisson
operators with operators in $\hat{\Psi}_{0b}^{-\infty,\bullet}\left(X\right)$.

However, we can still prove a slightly weaker version of Point 1 of
Theorem \ref{thm:compositions-involving-interior} using a different
approach.
\begin{thm}
	\label{thm:global-composition-0b-0b}Let $\mathcal{E}=\left(\mathcal{E}_{\lf},\mathcal{E}_{\rf},\mathcal{E}_{\ff_{b}},\mathcal{E}_{\ff_{0}}\right)$\textup{\emph{
			and}} $\mathcal{F}=\left(\mathcal{F}_{\lf},\mathcal{F}_{\rf},\mathcal{F}_{\ff_{b}},\mathcal{F}_{\ff_{0}}\right)$\textup{\emph{
			be index sets. Let }}$P\in\hat{\Psi}_{0b}^{-\infty,\mathcal{E}}\left(X\right)$
	and $Q\in\hat{\Psi}_{0b}^{-\infty,\mathcal{F}}\left(X\right)$. If
	$\Re\left(\mathcal{E}_{\rf}+\mathcal{F}_{\lf}\right)>-1$, then $PQ$
	is well-defined and in $\hat{\Psi}_{0b}^{-\infty,\left(\tilde{\mathcal{G}}_{\lf},\tilde{\mathcal{G}}_{\rf},\mathcal{G}_{\ff_{b}},\mathcal{G}_{\ff_{0}}\right)}\left(X\right)$\textup{\emph{
			, where
			\begin{align*}
				\tilde{\mathcal{G}}_{\lf} & =\mathcal{E}_{\lf}\eunion\left(\mathcal{F}_{\lf}+\mathcal{E}_{\ff_{b}}\right)\eunion\left(\mathcal{F}_{\lf}+\tilde{\mathcal{E}}_{\ff_{0}}\right)\\
				\tilde{\mathcal{G}}_{\rf} & =\mathcal{F}_{\rf}\eunion\left(\mathcal{E}_{\rf}+\mathcal{F}_{\ff_{b}}\right)\eunion\left(\mathcal{E}_{\rf}+\tilde{\mathcal{F}}_{\ff_{0}}\right)\\
				\mathcal{G}_{\ff_{b}} & =\left(\mathcal{E}_{\ff_{b}}+\mathcal{F}_{\ff_{b}}\right)\eunion\left(\mathcal{E}_{\lf}+\mathcal{F}_{\rf}+1\right)\\
				\mathcal{G}_{\ff_{0}} & =\mathcal{E}_{\ff_{0}}+\mathcal{F}_{\ff_{0}}
			\end{align*}
			and $\tilde{\mathcal{E}}_{\ff_{0}}=\mathcal{E}_{\ff_{0}}\eunion\left(\mathcal{E}_{\ff_{b}}+n\right)$,
			$\tilde{\mathcal{F}}_{\ff_{0}}=\mathcal{F}_{\ff_{0}}\eunion\left(\mathcal{F}_{\ff_{b}}+n\right)$.}}
\end{thm}

\begin{proof}
	Decomposing $PQ=\sum_{ij}P_{i}Q_{j}+\sum_{i}\left(P_{i}R_{Q}+R_{P}Q_{i}\right)+R_{P}R_{Q}$
	as in the discussion above, we see that it suffices to prove that
	the terms $P_{i}R_{Q}$ and $R_{P}Q_{i}$ are in $\Psi_{b}^{-\infty,\left(\tilde{\mathcal{G}}_{\lf},\tilde{\mathcal{G}}_{\rf},\mathcal{G}_{\ff_{b}}\right)}\left(X\right)$.
	It suffices to prove that if $P\in\hat{\Psi}_{0b}^{-\infty,\mathcal{E}}\left(X\right)$
	and $Q\in\Psi_{b}^{-\infty,\mathcal{F}}\left(X\right)$ with $\Re\left(\mathcal{E}_{\rf}+\mathcal{F}_{\lf}\right)>-1$,
	then $PQ\in\Psi_{b}^{-\infty,\left(\tilde{\mathcal{G}}_{\lf},\mathcal{G}_{\rf},\mathcal{G}_{\ff_{b}}\right)}\left(X\right)$
	since the result for $QP$ follows by taking adjoints.
	
	From Corollary \ref{cor:0b-interior-are-extended-0-calc-and-viceversa},
	we know that $P\in\Psi_{0b}^{-\infty,\left(\mathcal{E}_{\lf},\mathcal{E}_{\rf},\mathcal{E}_{\ff_{b}},\tilde{\mathcal{E}}_{\ff_{0}}\right)}\left(X\right)$
	where \emph{$\tilde{\mathcal{E}}_{\ff_{0}}=\mathcal{E}_{\ff_{0}}\eunion\left(\mathcal{E}_{\ff_{b}}+n\right)$}.
	We complete the proof using the pull-back / push-forward technique.
	The proof is very similar to the proof of the composition theorem
	for the $b$-calculus proved in \cite{MazzeoEdgeI}, so we will be
	brief. Consider the triple space $X^{3}$, and consider its p-submanifolds
	\begin{align*}
		LC_{0} & =\partial\Delta\times X\\
		LC_{b} & =\partial X\times\partial X\times X\\
		CR_{b} & =X\times\partial X\times\partial X\\
		LR_{b} & =\partial X\times X\times\partial X\\
		T_{b} & =LC_{b}\cap CR_{b}\cap LR_{b}.
	\end{align*}
	We consider the iterated blow-up
	\[
	Z=\left[X^{3};T_{b};LC_{0};LC_{b};CR_{b};LR_{b}\right].
	\]
	We denote by $H_{100},H_{010},H_{001}$ the boundary hyperfaces of
	$Z$ obtained by lifting the hyperfaces $\partial X\times X\times X$,
	$X\times\partial X\times X$, $X\times X\times\partial X$ of $X^{3}$.
	Moreover, we call $\ff_{T},\ff_{LC_{0}},\ff_{LC_{b}},\ff_{CR_{b}},\ff_{LR_{b}}$
	the other boundary hyperfaces of $Z$ obtained from the blow-up. We
	can define three $b$-fibrations
	\begin{align*}
		\gamma_{LC} & :Z\to X_{0b}^{2}\\
		\gamma_{CR} & :Z\to X_{b}^{2}\\
		\gamma_{LR} & :Z\to X_{b}^{2}
	\end{align*}
	as follows. $\gamma_{CR}$ is defined by first blowing down $\ff_{LC_{b}},\ff_{LC_{0}},\ff_{LR_{b}}$
	so that we end up with $\left[X^{3};T_{b};CR_{b}\right]=\left[X^{3};CR_{b};T_{b}\right]$;
	we now blow down $\ff_{T}$, ending up with $\left[X^{3};CR_{b}\right]=X\times X_{b}^{2}$;
	finally, we project to $X_{b}^{2}$. The map $\gamma_{LR}$ is defined
	similarly. Concerning $\gamma_{LC}$, we first blow down $\ff_{CR_{b}},\ff_{LR_{b}}$,
	and then we blow down $\ff_{T}$, ending up with $\left[X^{3};LC_{0};LC_{b}\right]=X_{0b}^{2}\times X$;
	we then project to the factor $X_{0b}^{2}$.
	
	The boundary matrix of $\gamma_{CR}$ (cf.\ \S2 of \cite{MelroseCorners})
	is
	\[
	e_{\gamma_{CR}}=\begin{matrix} & H_{100} & H_{010} & H_{001} & \ff_{T} & \ff_{LC_{0}} & \ff_{LC_{b}} & \ff_{CR_{b}} & \ff_{LR_{b}}\\
		\lf & 0 & 1 & 0 & 0 & 1 & 1 & 0 & 0\\
		\rf & 0 & 0 & 1 & 0 & 0 & 0 & 0 & 1\\
		\ff_{b} & 0 & 0 & 0 & 1 & 0 & 0 & 1 & 0
	\end{matrix}.
	\]
	This is a convenient way to record the index sets of the pull-backs
	via $\gamma_{CR}$ of a triple of boundary defining functions for
	$X_{b}^{2}$. More precisely, if $r_{\lf},r_{\rf},r_{\ff_{b}}$ are
	boundary defining functions for $X_{b}^{2}$, we have
	\begin{align*}
		\gamma_{CR}^{*}r_{\lf} & =\rho_{010}\rho_{LC_{b}}\rho_{LC_{0}}\\
		\gamma_{CR}^{*}r_{\rf} & =\rho_{001}\rho_{LR_{b}}\\
		\gamma_{CR}^{*}r_{\ff_{b}} & =\rho_{T}\rho_{CR_{b}}
	\end{align*}
	where $\rho_{\bullet}$ are various boundary defining functions for
	$Z$. The matrix of boundary exponents is easy to compute, as it is
	a composition of matrices of boundary exponents for blow-down maps
	and canonical projections. Similarly, one computes the boundary matrix
	of $\gamma_{LR}$
	\[
	e_{\gamma_{LR}}=\begin{matrix} & H_{100} & H_{010} & H_{001} & \ff_{T} & \ff_{LC_{0}} & \ff_{LC_{b}} & \ff_{CR_{b}} & \ff_{LR_{b}}\\
		\lf & 1 & 0 & 0 & 0 & 1 & 1 & 0 & 0\\
		\rf & 0 & 0 & 1 & 0 & 0 & 0 & 1 & 0\\
		\ff_{b} & 0 & 0 & 0 & 1 & 0 & 0 & 0 & 1
	\end{matrix}
	\]
	and the boundary matrix of $\gamma_{LC}$
	\[
	e_{\gamma_{LC}}=\begin{matrix} & H_{100} & H_{010} & H_{001} & \ff_{T} & \ff_{LC_{0}} & \ff_{LC_{b}} & \ff_{CR_{b}} & \ff_{LR_{b}}\\
		\lf & 1 & 0 & 0 & 0 & 0 & 0 & 0 & 1\\
		\rf & 0 & 1 & 0 & 0 & 0 & 0 & 1 & 0\\
		\ff_{b} & 0 & 0 & 0 & 1 & 0 & 1 & 0 & 0\\
		\ff_{0} & 0 & 0 & 0 & 0 & 1 & 0 & 0 & 0
	\end{matrix}.
	\]
	Now, the lifted Schwartz kernel $\kappa_{PQ}=\beta_{b}^{*}K_{PQ}$
	is equal to
	\[
	\left(\gamma_{LR}\right)_{*}\left(\gamma_{LC}^{*}\kappa_{P}\cdot\gamma_{CR}^{*}\kappa_{Q}\right).
	\]
	From the pull-back Theorem, to compute the index sets of $\gamma_{LC}^{*}\kappa_{P}$
	we simply multiply the formal row vector $\left(\mathcal{E}_{\lf},\mathcal{E}_{\rf},\mathcal{E}_{\ff_{b}},\tilde{\mathcal{E}}_{\ff_{0}}\right)^{T}$
	on the left by $e_{\gamma_{LC}}$, ending up with
	\[
	\left(\mathcal{E}_{\lf},\mathcal{E}_{\rf},0,\mathcal{E}_{\ff_{b}},\tilde{\mathcal{E}}_{\ff_{0}},\mathcal{E}_{\ff_{b}},\mathcal{E}_{\rf},\mathcal{E}_{\lf}\right).
	\]
	Similarly, the index sets of $\gamma_{CR}^{*}\kappa_{Q}$ are obtained
	by multiplying $\left(\mathcal{F}_{\lf},\mathcal{F}_{\rf},\mathcal{F}_{\ff_{b}}\right)^{T}$
	on the left by $e_{\gamma_{CR}}$, and we obtain
	\[
	\left(0,\mathcal{F}_{\lf},\mathcal{F}_{\rf},\mathcal{F}_{\ff_{b}},\mathcal{F}_{\lf},\mathcal{F}_{\lf},\mathcal{F}_{\ff_{b}},\mathcal{F}_{\rf}\right).
	\]
	The index sets of the product $\gamma_{LC}^{*}\kappa_{P}\cdot\gamma_{CR}^{*}\kappa_{Q}$
	are obtained by adding the two index sets above. The resulting tuple
	refers to the index sets of $\gamma_{LC}^{*}\kappa_{P}\cdot\gamma_{CR}^{*}\kappa_{Q}$
	interpreted as a section of the bundle
	\[
	\gamma_{LC}^{*}\left(r_{\ff_{b}}^{-1}r_{\ff_{0}}^{-n-1}\beta_{0b,R}^{*}\mathcal{D}_{X}^{1}\right)\otimes\gamma_{CR}^{*}\left(r_{\ff_{b}}^{-1}\beta_{b,R}^{*}\mathcal{D}_{X}^{1}\right).
	\]
	Before applying the push-forward Theorem, we need to express it as
	a rescaling of the bundle
	\[
	\gamma_{LR}^{*}\left(r_{\ff_{b}}^{-1}\beta_{b,R}^{*}\mathcal{D}_{X}^{1}\right)\otimes\gamma_{LR}^{*}\left(^{b}\mathcal{D}_{X_{b}^{2}}^{-1}\right)\otimes{^{b}\mathcal{D}_{Z}}.
	\]
	This is a routine computation, which requires the following fundamental
	fact (cf.\ Lemma C.4 of \cite{EpsteinMelroseShrinking}): if $Y$
	is a manifold with corners, $S$ is a p-submanifold of codimension
	$c$, and $\beta:\left[Y;S\right]\to Y$ is the blow-down map, then
	\[
	\beta^{*}\left(C^{\infty}\left(Y;\mathcal{D}_{Y}^{1}\right)\right)=r_{\ff}^{c-1}C^{\infty}\left(\left[Y;S\right];\mathcal{D}_{\left[X;S\right]}^{1}\right).
	\]
	We omit all the elementary computations and only report the results:
	the index sets are
	\begin{align*}
		\mathcal{E}_{\lf} & \text{ at \ensuremath{H_{100}}}\\
		\mathcal{E}_{\rf}+\mathcal{F}_{\lf}+1 & \text{ at \ensuremath{H_{010}}}\\
		\mathcal{F}_{\rf} & \text{ at \ensuremath{H_{001}}}\\
		\mathcal{E}_{\ff_{b}}+\mathcal{F}_{\ff_{b}} & \text{ at \ensuremath{\ff_{T}}}\\
		\mathcal{F}_{\lf}+\tilde{\mathcal{E}}_{\ff_{0}} & \text{ at \ensuremath{\ff_{LC_{0}}}}\\
		\mathcal{F}_{\lf}+\mathcal{E}_{\ff_{b}} & \text{ at \ensuremath{\ff_{LC_{b}}}}\\
		\mathcal{E}_{\rf}+\mathcal{F}_{\ff_{b}} & \text{ at \ensuremath{\ff_{CR_{b}}}}\\
		\mathcal{E}_{\lf}+\mathcal{F}_{\rf}+1 & \text{ at \ensuremath{\ff_{LR_{b}}}}.
	\end{align*}
	Once we have these index sets, the integrability condition for $\left(\gamma_{LR}\right)_{*}\left(\gamma_{LC}^{*}\kappa_{P}\cdot\gamma_{CR}^{*}\kappa_{Q}\right)$
	is that the index set at $H_{010}$ has real part $>0$ ($H_{010}$
	is the only column of $e_{\gamma_{LR}}$ containing only zeros). The
	index sets for the push-forward are obtained by taking the extended
	unions of the index sets corresponding to ones, in each row of $e_{\gamma_{LR}}$.
	The result is exactly the triple of index sets $\tilde{\mathcal{G}}_{\lf},\mathcal{G}_{\rf},\mathcal{G}_{\ff_{0}}$
	defined above.
\end{proof}

\subsection{\label{subsec:Mapping-properties-on-sobolev}Mapping properties on
	Sobolev spaces}

We finally discuss mapping properties on Sobolev spaces.

\subsubsection{$0$-interior and $0b$-interior operators}

Here we discuss mapping properties for elements of $\hat{\Psi}_{0}^{-\infty,\bullet}\left(X\right)$
and $\hat{\Psi}_{0b}^{-\infty,\bullet}\left(X\right)$, on a natural
scale of \emph{$0$-Sobolev spaces} on $X$. Let us first recall the
definition of these spaces.

Given $k\in\mathbb{N}$, we denote by $H_{0}^{k}\left(X\right)$ the
space of $L_{b}^{2}$ functions $u$ on $X$ such that, for any set
$V_{1},...,V_{j}\in\mathcal{V}_{0}\left(X\right)$ with $j\leq k$,
the weak multiple $0$-derivative $V_{1}\cdots V_{j}u$ is still in
$L_{b}^{2}\left(X\right)$. Given $\delta\in\mathbb{R}$, we denote
by $x^{\delta}H_{0}^{k}\left(X\right)$ the space of functions of
the form $x^{\delta}u$, where $u\in H_{0}^{k}\left(X\right)$ and
$x$ is a boundary defining functions. This space does not depend
on the choice of $x$. We can define a natural Banach topology on
$x^{\delta}H_{0}^{k}\left(X\right)$ in the usual way, by means of
a choice of a finite set of vector fields on $X$ which span $^{0}TX$
pointwise. As usual, this topology does not depend on the choices
made.

The following theorem collects the known mapping properties of the
$0$-calculus and the extended $0$-calculus on $0$-Sobolev spaces:
\begin{thm}
	(Theorem 3.25, Corollary 3.23, and Proposition 3.29 of \cite{MazzeoEdgeI},
	Theorem 5.2.2 of \cite{Lauter})
	\begin{enumerate}
		\item If $P\in\Psi_{0}^{m}\left(X\right)$, then $P$ is bounded as a map
		$x^{\delta}H_{0}^{k+m}\left(X\right)\to x^{\delta}H_{0}^{k}\left(X\right)$
		for every $k$ and $\delta$.
		\item If $P\in\Psi_{0}^{-\infty,\mathcal{E}}\left(X\right)$, $\Re\left(\mathcal{E}_{\lf}\right)>\delta$,
		$\Re\left(\mathcal{E}_{\rf}\right)>-\delta-1$, and $\Re\left(\mathcal{E}_{\ff_{0}}\right)\geq0$,
		then $P$ is bounded as a map $x^{\delta}H_{0}^{m_{1}}\left(X\right)\to x^{\delta}H_{0}^{m_{2}}\left(X\right)$
		for every $m_{1},m_{2}$. If moreover $\Re\left(\mathcal{E}_{\ff_{0}}\right)>0$,
		then the map is compact.
		\item If $P\in\Psi_{0b}^{-\infty,\mathcal{E}}\left(X\right)$, $\Re\left(\mathcal{E}_{\lf}\right)>\delta$,
		$\Re\left(\mathcal{E}_{\rf}\right)>-\delta-1$, $\Re\left(\mathcal{E}_{\ff_{b}}\right)>0$,
		and $\Re\left(\mathcal{E}_{\ff_{0}}\right)\geq0$, then $P$ is bounded
		as a map $x^{\delta}H_{0}^{m_{1}}\left(X\right)\to x^{\delta}H_{0}^{m_{2}}\left(X\right)$
		for every $m_{1},m_{2}$. If moreover $\Re\left(\mathcal{E}_{\ff_{0}}\right)>0$,
		the map is compact.
	\end{enumerate}
\end{thm}

Directly by the previous theorem, and by Corollary 5\ref{cor:global-relation-symbolic0b-physical0b},
we get mapping properties for $0$-interior and $0b$-interior operators:
\begin{cor}
	\label{cor:mapping-symbolic-0b-0-sobolev}$ $
	\begin{enumerate}
		\item If $P\in\hat{\Psi}_{0b}^{-\infty,\mathcal{E}}\left(X\right)$, $\Re\left(\mathcal{E}_{\lf}\right)>\delta$,
		$\Re\left(\mathcal{E}_{\rf}\right)>-\delta-1$, $\Re\left(\mathcal{E}_{\ff_{b}}\right)>0$,
		and $\Re\left(\mathcal{E}_{\ff_{0}}\right)\geq0$, then $P$ is bounded
		as a map $x^{\delta}H_{0}^{m_{1}}\left(X\right)\to x^{\delta}H_{0}^{m_{2}}\left(X\right)$
		for every $m_{1},m_{2}$. Moreover, if $\Re\left(\mathcal{E}_{\ff_{0}}\right)>0$,
		then the map is compact.
		\item If $P\in\hat{\Psi}_{0}^{-\infty,\mathcal{E}}\left(X\right)$, $\Re\left(\mathcal{E}_{\lf}\right)>\delta$,
		$\Re\left(\mathcal{E}_{\rf}\right)>-\delta-1$, and $\Re\left(\mathcal{E}_{\ff_{0}}\right)\geq0$,
		then $P$ is bounded as a map $x^{\delta}H_{0}^{m_{1}}\left(X\right)\to x^{\delta}H_{0}^{m_{2}}\left(X\right)$
		for every $m_{1},m_{2}$. Moreover, if $\Re\left(\mathcal{E}_{\ff_{0}}\right)>0$,
		then the map is compact.
	\end{enumerate}
\end{cor}

\subsubsection{Twisted boundary operators, $0$-trace and $0$-Poisson operators}

Here we discuss mapping properties for elements of the twisted boundary
calculus, and for twisted symbolic $0$-trace and $0$-Poisson operators.
While on $X$ we will use the scale of $0$-Sobolev spaces recalled
above, on $\partial X$ we will use Sobolev spaces naturally adapted
to the twisted boundary calculus. These Sobolev spaces were introduced
(in a much more general fashion) in \cite{KrainerMendozaVariableOrders};
we recall the definition and the main properties here, adapting them
to our simplifying assumptions.

Let $\mathfrak{a}:\overline{\mathbb{R}^{n}}\to\hom\left(\mathbb{C}^{N}\right)$
be a smooth family of matrices with a single constant eigenvalue.
We define $H^{\mathfrak{a}}\left(\mathbb{R}^{n};\mathbb{C}^{N}\right)$
as the space of $\mathbb{C}^{N}$-valued tempered distributions $u\in\mathcal{S}'\left(\mathbb{R}^{n};\mathbb{C}^{N}\right)$
such that $\left\langle D\right\rangle ^{\mathfrak{a}}u\in L^{2}\left(\mathbb{R}^{n};\mathbb{C}^{N}\right)$,
where $\left\langle D\right\rangle ^{\mathfrak{a}}=\Op_{L}\left(\left\langle \eta\right\rangle ^{\mathfrak{a}}\right)$.
Observe that if $\mathfrak{a}$ is constant, we can simply define
$H^{\mathfrak{a}}\left(\mathbb{R}^{n};\mathbb{C}^{N}\right)$ as the
space of distributions $u\in\mathcal{S}'\left(\mathbb{R}^{n};\mathbb{C}^{N}\right)$
such that $\left\langle \eta\right\rangle ^{\mathfrak{a}}\hat{u}\left(\eta\right)\in L^{2}\left(\mathbb{R}^{n};\mathbb{C}^{N}\right)$.

We can transpose this local definition to manifolds as usual, using
a partition of unity. More precisely, given a rank $N$ complex vector
bundle $E\to\partial X$, and a smooth endomorphism $\mathfrak{a}:E\to E$
with a single constant eigenvalue, we can define $H^{\mathfrak{a}}\left(\partial X;E\right)$
as the space of measurable sections of $E$ which are locally in $H^{\mathfrak{a}}\left(\mathbb{R}^{n};\mathbb{C}^{N}\right)$
with respect to coordinates on $\partial X$ and a trivialization
for $E$. Finally, if $\boldsymbol{E}\to\partial X$ is a vector bundle,
$\boldsymbol{\mathfrak{a}}:\boldsymbol{E}\to\boldsymbol{E}$ is a
bundle endomorphism with constant eigenvalues, and $\boldsymbol{E}=\bigoplus_{i}E_{i}$
is the decomposition of $\boldsymbol{E}$ into generalized eigenbundles
of $\boldsymbol{\mathfrak{a}}$ (with corresponding decomposition
$\boldsymbol{\mathfrak{a}}=\bigoplus_{i}\mathfrak{a}_{i}$ of $\boldsymbol{\mathfrak{a}}$),
we define $H^{\boldsymbol{\mathfrak{a}}}\left(\partial X;\boldsymbol{E}\right)=\bigoplus_{i}H^{\mathfrak{a}_{i}}\left(\partial X;E_{i}\right)$.

The next lemma relates this class of Sobolev spaces to the usual scale
of real order Sobolev spaces on $\partial X$.
\begin{lem}
	\label{lem:chain-of-inclusion-untwisted-twisted-untwisted-sobolev}Let
	$E\to\partial X$ be a rank $N$ complex vector bundle, and let $\mathfrak{a}:E\to E$
	be a smooth endomorphism of $E$ with a single constant eigenvalue
	$\mu\in\mathbb{C}$. Then, for every $\varepsilon>0$, we have a chain
	of inclusions
	\[
	H^{\Re\left(\mu\right)+\varepsilon}\left(\partial X;E\right)\subset H^{\mathfrak{a}}\left(\partial X;E\right)\subset H^{\Re\left(\mu\right)-\varepsilon}\left(\partial X;E\right).
	\]
\end{lem}

\begin{proof}
	We only prove the first inclusion, since the proof of the second one
	is very similar. Using the usual partition of unity argument, it is
	sufficient to prove the local version of the statement. First of all,
	observe that if $z\in\mathbb{C}$, then $H^{z}\left(\mathbb{R}^{n}\right)=H^{\Re\left(z\right)}\left(\mathbb{R}^{n}\right)$.
	Indeed, $u\in H^{z}\left(\mathbb{R}^{n}\right)$ if and only if $\left\langle \eta\right\rangle ^{z}\hat{u}\left(\eta\right)\in L^{2}\left(\mathbb{R}^{n}\right)$;
	writing $z=a+ib$, we have $\left\langle \eta\right\rangle ^{z}=\left\langle \eta\right\rangle ^{a}e^{ib\log\left\langle \eta\right\rangle }$,
	so $\left|\left\langle \eta\right\rangle ^{z}\right|=\left|\left\langle \eta\right\rangle ^{a}\right|$
	and therefore $\left\langle \eta\right\rangle ^{z}\hat{u}\left(\eta\right)\in L^{2}\left(\mathbb{R}^{n}\right)$
	if and only if $\left\langle \eta\right\rangle ^{a}\hat{u}\left(\eta\right)\in L^{2}\left(\mathbb{R}^{n}\right)$.
	Now, consider a smooth family of matrices $\mathfrak{a}:\overline{\mathbb{R}^{n}}\to\hom\left(\mathbb{C}^{N}\right)$
	with a single constant eigenvalue $\mu\in\mathbb{C}$. Let $u\in H^{\Re\left(\mu\right)+\varepsilon}\left(\mathbb{R}^{n};\mathbb{C}^{N}\right)$.
	Write
	\begin{align*}
		\left\langle D\right\rangle ^{\mathfrak{a}}u & =\frac{1}{\left(2\pi\right)^{n}}\int e^{iy\eta}\left\langle \eta\right\rangle ^{\mathfrak{a}\left(y\right)}\hat{u}\left(\eta\right)d\eta\\
		& =\frac{1}{\left(2\pi\right)^{n}}\int e^{iy\eta}\left\langle \eta\right\rangle ^{\mathfrak{a}\left(y\right)-\mu-\varepsilon}\left\langle \eta\right\rangle ^{\mu+\varepsilon}\hat{u}\left(\eta\right)d\eta\\
		& =\left\langle D\right\rangle ^{\mathfrak{a}-\mu-\varepsilon}\left\langle D\right\rangle ^{\mu+\varepsilon}u.
	\end{align*}
	Since $u\in H^{\mu+\varepsilon}\left(\mathbb{R}^{n};\mathbb{C}^{N}\right)$,
	we have $\left\langle D\right\rangle ^{\mu+\varepsilon}u\in L^{2}\left(\mathbb{R}^{n};\mathbb{C}^{N}\right)$.
	To conclude, it suffices to show that $\left\langle D\right\rangle ^{\mathfrak{a}-\mu-\varepsilon}$
	is a pseudodifferential operator of negative order. To see this, note
	that the matrix $\mathfrak{a}-\mu$ is a smooth family of nilpotent
	matrices; as such, $\left\langle \eta\right\rangle ^{\mathfrak{a}-\mu}$
	is a smooth family of matrices whose entries are linear combinations
	of powers of $\log\left\langle \eta\right\rangle $. This implies
	that $\left\langle \eta\right\rangle ^{\mathfrak{a}-\mu}$ is a smooth
	map $\overline{\mathbb{R}^{n}}\to\mathcal{A}^{-\delta}\left(\overline{\mathbb{R}^{n}}\right)$
	for every $\delta>0$, and therefore $\left\langle \eta\right\rangle ^{-\varepsilon}\left\langle \eta\right\rangle ^{\mathfrak{a}-\mu}$
	is a smooth map $\overline{\mathbb{R}^{n}}\to\mathcal{A}^{\varepsilon'}\left(\overline{\mathbb{R}^{n}}\right)$
	for some $\varepsilon'>0$. In other words, $\left\langle D\right\rangle ^{\mathfrak{a}-\mu-\varepsilon}$
	is a pseudodifferential operator of order $-\varepsilon'$ for some
	$\varepsilon'>0$. As such, $\left\langle D\right\rangle ^{\mathfrak{a}-\mu-\varepsilon}$
	maps $L^{2}\left(\mathbb{R}^{n};\mathbb{C}^{N}\right)$ to $H^{\varepsilon'}\left(\mathbb{R}^{n};\mathbb{C}^{N}\right)$,
	which is (compactly) included in $L^{2}\left(\mathbb{R}^{n};\mathbb{C}^{N}\right)$.
\end{proof}
The following basic mapping property is a direct consequence of the
definitions, but it can also be derived by the more general mapping
properties for the calculus developed in \cite{KrainerMendozaVariableOrders}:
\begin{prop}
	\label{prop:mapping-twisted-boundary-sobolev}Let $\boldsymbol{E},\boldsymbol{F}$
	be vector bundles on $\partial X$ equipped with smooth endomorphisms
	$\boldsymbol{\mathfrak{a}}:\boldsymbol{E}\to\boldsymbol{E}$, $\boldsymbol{\mathfrak{b}}:\boldsymbol{F}\to\boldsymbol{F}$
	with constant eigenvalues. Let $\boldsymbol{Q}\in\Psi_{\phg}^{-\left[0\right],\left(\boldsymbol{\mathfrak{a}},\boldsymbol{\mathfrak{b}}\right)}\left(\partial X;\boldsymbol{E},\boldsymbol{F}\right)$.
	Then $\boldsymbol{Q}$ induces a continuous linear map
	\[
	\boldsymbol{Q}:H^{\delta+\boldsymbol{\mathfrak{a}}}\left(\partial X;\boldsymbol{E}\right)\to H^{\delta+\boldsymbol{\mathfrak{b}}}\left(\partial X;\boldsymbol{F}\right)
	\]
	for every $\delta\in\mathbb{R}$. If $\sigma\left(\boldsymbol{Q}\right)=0$,
	then the map above is compact.
\end{prop}

We will now discuss mapping properties for twisted symbolic $0$-trace
and $0$-Poisson operators.
\begin{thm}
	\label{thm:mapping-twisted-trace-interior-sobolev}Let $\boldsymbol{\mathfrak{s}}:\boldsymbol{E}\to\boldsymbol{E}$
	be a smooth bundle endomorphism with constant eigenvalues.
	\begin{enumerate}
		\item Let $\boldsymbol{A}\in\hat{\Psi}_{0\tr}^{-\infty,\left(\mathcal{E}_{\of},\left[\mathcal{E}_{\ff}\right]\right),\boldsymbol{\mathfrak{s}}}\left(X;\partial X,\boldsymbol{E}\right)$.
		Let $\delta\in\mathbb{R}$, and suppose that
		\begin{align*}
			\Re\left(\mathcal{E}_{\of}\right) & >-\delta-1\\
			\left[\mathcal{E}_{\ff}\right] & =0.
		\end{align*}
		Then $\boldsymbol{A}$ induces a continuous linear map
		\[
		x^{\delta}L_{b}^{2}\left(X\right)\to H^{\delta+\boldsymbol{\mathfrak{s}}}\left(\partial X;\boldsymbol{E}\right).
		\]
		If moreover $\hat{N}\left(\boldsymbol{A}\right)\equiv0$, then the
		map above is compact.
		\item Let $\boldsymbol{B}\in\hat{\Psi}_{0\po}^{-\infty,\left(\mathcal{E}_{\of},\left[\mathcal{E}_{\ff}\right]\right),\boldsymbol{\mathfrak{s}}}\left(\partial X,\boldsymbol{E};X\right)$.
		Let $\delta\in\mathbb{R}$, and suppose that
		\begin{align*}
			\Re\left(\mathcal{E}_{\of}\right) & >\delta\\
			\left[\mathcal{E}_{\ff}\right] & =0.
		\end{align*}
		Then $\boldsymbol{B}$ induces a continuous linear map
		\[
		\boldsymbol{B}:H^{\delta+\boldsymbol{\mathfrak{s}}}\left(\partial X;\boldsymbol{E}\right)\to x^{\delta}L_{b}^{2}\left(X\right).
		\]
		If moreover $\hat{N}\left(\boldsymbol{B}\right)\equiv0$, then the
		map above is compact.
	\end{enumerate}
\end{thm}

\begin{proof}
	(1) We first reduce ourselves to the case $\delta=0$. Given $\boldsymbol{A}\in\hat{\Psi}_{0\tr}^{-\infty,\left(\mathcal{E}_{\of},\left[\mathcal{E}_{\ff}\right]\right),\boldsymbol{\mathfrak{s}}}\left(X;\partial X,\boldsymbol{E}\right)$,
	we have
	\begin{align*}
		\boldsymbol{A}x^{\delta} & \in\hat{\Psi}_{0\tr}^{-\infty,\left(\mathcal{E}_{\of}+\delta,\left[\mathcal{E}_{\ff}+\delta\right]\right),\boldsymbol{\mathfrak{s}}}\left(X;\partial X,\boldsymbol{E}\right)\\
		& =\hat{\Psi}_{0\tr}^{-\infty,\left(\mathcal{E}_{\of}+\delta,\left[\mathcal{E}_{\ff}\right]\right),\delta+\boldsymbol{\mathfrak{s}}}\left(X;\partial X,\boldsymbol{E}\right).
	\end{align*}
	Indeed, suppose that locally the Schwartz kernel of $\boldsymbol{A}$
	is
	\[
	K_{\boldsymbol{A}}\left(y,\tilde{x},\tilde{y}\right)d\tilde{x}d\tilde{y}=\frac{1}{\left(2\pi\right)^{n}}\int e^{i\left(y-\tilde{y}\right)}\left\langle \eta\right\rangle ^{-\boldsymbol{\mathfrak{s}}\left(y\right)}\boldsymbol{a}\left(y;\tilde{x},\eta\right)d\eta d\tilde{x}d\tilde{y}
	\]
	for some $\boldsymbol{a}\in S_{0\tr,\mathcal{S}}^{-\infty,\mathcal{E}}\left(\mathbb{R}^{n};\mathbb{R}_{1}^{n+1};\mathbb{C}^{N}\right)$
	with $N=\mathrm{rank}\left(\boldsymbol{E}\right)$. Then we have
	\begin{align*}
		K_{\boldsymbol{A}x^{\delta}}\left(y,\tilde{x},\tilde{y}\right)d\tilde{x}d\tilde{y} & =\frac{1}{\left(2\pi\right)^{n}}\int e^{i\left(y-\tilde{y}\right)}\left\langle \eta\right\rangle ^{-\boldsymbol{\mathfrak{s}}\left(y\right)}\boldsymbol{a}\left(y;\tilde{x},\eta\right)\tilde{x}^{\delta}d\eta d\tilde{x}d\tilde{y}\\
		& =\frac{1}{\left(2\pi\right)^{n}}\int e^{i\left(y-\tilde{y}\right)}\left\langle \eta\right\rangle ^{-\boldsymbol{\mathfrak{s}}\left(y\right)-\delta}\left\langle \eta\right\rangle ^{\delta}\boldsymbol{a}\left(y;\tilde{x},\eta\right)\tilde{x}^{\delta}d\eta d\tilde{x}d\tilde{y}.
	\end{align*}
	By definition, we have
	\[
	\boldsymbol{a}\left(y;\tilde{x},\eta\right)\in\mathcal{S}\left(\mathbb{R}^{n}\right)\hat{\otimes}\mathcal{A}_{\phg}^{\left(\mathcal{E}_{\of},\mathcal{E}_{\ff}-1,\infty,\infty\right)}\left(\hat{T}_{0}^{2}\right)
	\]
	where $\hat{T}_{0}^{2}$ is the space defined in \S\ref{subsubsec:Symbolic-0-trace-operators};
	now, the lift of $\tilde{x}$ to $\hat{T}_{0}^{2}$ has index sets
	$\left(1,1,0,-1\right)$ (the order of the faces of $\hat{T}_{0}^{2}$
	is $\of,\ff,\iif_{\eta},\iif_{\tilde{x}}$), while the lift of $\left\langle \eta\right\rangle $
	to $\hat{T}_{0}^{2}$ has index sets $\left(0,-1,-1,0\right)$. Therefore,
	we have
	\[
	\left\langle \eta\right\rangle ^{\delta}\boldsymbol{a}\left(y;\tilde{x},\eta\right)\tilde{x}^{\delta}\in\mathcal{S}\left(\mathbb{R}^{n}\right)\hat{\otimes}\mathcal{A}_{\phg}^{\left(\mathcal{E}_{\of}+\delta,\mathcal{E}_{\ff}-1,\infty,\infty\right)}\left(\hat{T}_{0}^{2}\right)
	\]
	and it follows that
	\[
	\boldsymbol{A}x^{\delta}\in\hat{\Psi}_{0\tr}^{-\infty,\left(\mathcal{E}_{\of}+\delta,\left[\mathcal{E}_{\ff}\right]\right),\delta+\boldsymbol{\mathfrak{s}}}\left(X;\partial X,\boldsymbol{E}\right)
	\]
	as claimed. Since $\boldsymbol{A}$ is bounded as a map $x^{\delta}L_{b}^{2}\left(X\right)\to H^{\delta+\boldsymbol{\mathfrak{s}}}\left(\partial X;\boldsymbol{E}\right)$
	if and only if $\boldsymbol{A}x^{\delta}$ is bounded as a map $L_{b}^{2}\left(X\right)\to H^{\delta+\boldsymbol{\mathfrak{s}}}\left(\partial X;\boldsymbol{E}\right)$,
	we can assume without loss of generality that $\delta=0$.
	
	From the discussion above, it suffices to prove that if $\boldsymbol{A}\in\hat{\Psi}_{0\tr}^{-\infty,\left(\mathcal{E}_{\of},\left[\mathcal{E}_{\ff}\right]\right),\boldsymbol{\mathfrak{s}}}\left(X;\partial X,\boldsymbol{E}\right)$
	with $\Re\left(\mathcal{E}_{\of}\right)>-1$ and $\left[\mathcal{E}_{\ff}\right]=0$,
	then $\boldsymbol{A}$ induces a continuous linear map $L_{b}^{2}\left(X\right)\to H^{\boldsymbol{\mathfrak{s}}}\left(\partial X;\boldsymbol{E}\right)$.
	We now reduce ourselves to the untwisted case $\boldsymbol{\mathfrak{s}}=\boldsymbol{0}$.
	Choose an elliptic operator $\boldsymbol{Q}\in\Psi_{\phg}^{-\left[0\right],\left(\boldsymbol{\mathfrak{s}},\boldsymbol{0}\right)}\left(\partial X;\boldsymbol{E}\right)$,
	and let $\boldsymbol{K}\in\Psi_{\phg}^{-\left[0\right],\left(\boldsymbol{0},\boldsymbol{\mathfrak{s}}\right)}\left(\partial X;\boldsymbol{E}\right)$
	be a parametrix for $\boldsymbol{Q}$, so that $\boldsymbol{K}\boldsymbol{Q}=I-\boldsymbol{R}$
	with $\boldsymbol{R}\in\Psi^{-\infty}\left(\partial X;\boldsymbol{E}\right)$.
	Write
	\[
	\boldsymbol{A}=\boldsymbol{K}\boldsymbol{Q}\boldsymbol{A}+\boldsymbol{R}\boldsymbol{A}.
	\]
	Since $\boldsymbol{K}$ is bounded as a map $L^{2}\left(\partial X;\boldsymbol{E}\right)\to H^{\boldsymbol{\mathfrak{s}}}\left(\partial X;\boldsymbol{E}\right)$,
	it suffices to prove that the operator $\boldsymbol{Q}\boldsymbol{A}$ is bounded
	as a map $L_{b}^{2}\left(X\right)\to L^{2}\left(\partial X;\boldsymbol{E}\right)$
	and that $\boldsymbol{R}\boldsymbol{A}$ is bounded as a map $L_{b}^{2}\left(X\right)\to H^{\boldsymbol{\mathfrak{s}}}\left(\partial X;\boldsymbol{E}\right)$.
	The operator $\boldsymbol{R}\boldsymbol{A}$ actually behaves much
	better than that. Indeed, since $\boldsymbol{R}$ has order $-\infty$,
	we can think of $\boldsymbol{R}$ as an element of $\boldsymbol{R}\in\Psi_{\phg}^{-\left[0\right],\left(\boldsymbol{\mathfrak{s}},k1_{\boldsymbol{E}}\right)}\left(\partial X;\boldsymbol{E}\right)$
	for every $k\in\mathbb{N}$. By Theorem \ref{thm:global-twisted-compositions-involving-boundary},
	we then have $\boldsymbol{R}\boldsymbol{A}\in\hat{\Psi}_{0\tr}^{-\infty,\left(\mathcal{E}_{\of},\left[\mathcal{E}_{\ff}\right]\right),k1_{\boldsymbol{E}}}\left(X;\partial X,\boldsymbol{E}\right)$
	for every $k\in\mathbb{N}$. Using the equality $\hat{\Psi}_{0\tr}^{-\infty,\left(\mathcal{E}_{\of},\left[\mathcal{E}_{\ff}\right]\right),k1_{\boldsymbol{E}}}\left(X;\partial X,\boldsymbol{E}\right)=\hat{\Psi}_{0\tr}^{-\infty,\left(\mathcal{E}_{\of},\left[\mathcal{E}_{\ff}+k\right]\right),\boldsymbol{0}}\left(X;\partial X,\boldsymbol{E}\right)$
	noted above, and Corollary \ref{cor:twisted-0-trace-0-poisson-are-untwisted},
	we obtain that $\boldsymbol{R}\boldsymbol{A}\in\hat{\Psi}_{0\tr}^{-\infty,\left(\mathcal{E}_{\of},\tilde{\mathcal{E}}_{\ff}+k\right)}\left(X;\partial X,\boldsymbol{E}\right)$
	for every $k\in\mathbb{N}$, where $\tilde{\mathcal{E}}_{\ff}$ is
	an index set such that $\left[\tilde{\mathcal{E}}_{\ff}\right]=\left[\mathcal{E}_{\ff}\right]=0$.
	It follows that $\boldsymbol{R}\boldsymbol{A}\in\hat{\Psi}_{0\tr}^{-\infty,\left(\mathcal{E}_{\of},\infty\right)}\left(X;\partial X,\boldsymbol{E}\right)=\Psi_{\tr}^{-\infty,\mathcal{E}_{\of}}\left(X;\partial X,\boldsymbol{E}\right)$.
	Using the push-forward Theorem, the integrability condition $\Re\left(\mathcal{E}_{\of}\right)>-1$
	guarantees that $\boldsymbol{R}\boldsymbol{A}$ induces a bounded
	linear map
	\[
	\boldsymbol{R}\boldsymbol{A}:L_{b}^{2}\left(X\right)\to C^{\infty}\left(\partial X;\boldsymbol{E}\right).
	\]
	Therefore, $\boldsymbol{R}\boldsymbol{A}$ is even \emph{compact}
	as a map $L_{b}^{2}\left(X\right)\to H^{\boldsymbol{\mathfrak{s}}}\left(\partial X;\boldsymbol{E}\right)$.
	Concerning the term $\boldsymbol{K}\boldsymbol{Q}\boldsymbol{A}$,
	arguing as above we obtain that
	\[
	\boldsymbol{Q}\boldsymbol{A}\in\hat{\Psi}_{0\tr}^{-\infty,\left(\mathcal{E}_{\of},\tilde{\mathcal{E}}_{\ff}\right)}\left(X;\partial X,\boldsymbol{E}\right)
	\]
	for some index set $\tilde{\mathcal{E}}_{\ff}$ with $\left[\tilde{\mathcal{E}}_{\ff}\right]=\left[\mathcal{E}_{\ff}\right]=0$.
	Since $\boldsymbol{K}$ is bounded as a map $L^{2}\left(\partial X;\boldsymbol{E}\right)\to H^{\boldsymbol{\mathfrak{s}}}\left(\partial X;\boldsymbol{E}\right)$,
	it remains to prove that $\boldsymbol{Q}\boldsymbol{A}$ is bounded
	as a map $L_{b}^{2}\left(X\right)\to L^{2}\left(\partial X;\boldsymbol{E}\right)$.
	
	From the discussion above, it suffices to prove that if $A\in\hat{\Psi}_{0\tr}^{-\infty,\mathcal{E}}\left(X;\partial X\right)$
	with $\Re\left(\mathcal{E}_{\of}\right)>-1$ and $\left[\mathcal{E}_{\ff}\right]=0$,
	then $A$ induces a continuous linear map $L_{b}^{2}\left(X\right)\to L^{2}\left(\partial X\right)$.
	Choose positive smooth densities for $X$ and $\partial X$, inducing
	$L^{2}$ inner products on the spaces above. Let $A^{\dagger}$ be
	the formal adjoint of $A$ with respect to these inner products. By
	\S\ref{thm:global-twisted-compositions-involving-boundary}, we have
	$A^{\dagger}\in\hat{\Psi}_{0\po}^{-\infty,\mathcal{F}}\left(\partial X,X\right)$
	where
	\begin{align*}
		\mathcal{F}_{\of} & :=\overline{\mathcal{E}_{\of}}+1\\
		\mathcal{F}_{\ff} & :=\overline{\mathcal{E}_{\ff}}.
	\end{align*}
	Now, if $u\in\dot{C}^{\infty}\left(X;\boldsymbol{E}\right)$, we have
	\begin{align*}
		\left|\left|Au\right|\right|_{L^{2}}^{2} & =\left(Au,Au\right)_{L^{2}}\\
		& =\left(A^{\dagger}Au,u\right)_{L_{b}^{2}}.
	\end{align*}
	By Theorem \ref{thm:global-compositions-involving-boundary}, the
	operator $A^{\dagger}A$ is well-defined and in $\hat{\Psi}_{0}^{-\infty,\left(\mathcal{F}_{\of},\mathcal{E}_{\of},\mathcal{E}_{\ff}+\mathcal{F}_{\ff}\right)}\left(X\right)$.
	These index sets satisfy
	\begin{align*}
		\Re\left(\mathcal{F}_{\of}\right) & >0\\
		\Re\left(\mathcal{E}_{\of}\right) & >-1\\
		\left[\mathcal{E}_{\ff}+\mathcal{F}_{\ff}\right] & =0,
	\end{align*}
	so by Corollary \ref{cor:mapping-symbolic-0b-0-sobolev} $A^{\dagger}A$
	is bounded as a map $L_{b}^{2}\left(X\right)\to L_{b}^{2}\left(X\right)$.
	Therefore, $\left|\left|Au\right|\right|_{L^{2}}^{2}\leq\left|\left|A^{\dagger}A\right|\right|\left|\left|u\right|\right|_{L_{b}^{2}}^{2}$
	where $\left|\left|A^{\dagger}A\right|\right|$ is the operator norm
	of $A^{\dagger}A$ on $L_{b}^{2}\left(X\right)$. Thus, we can extend
	$A$ by continuity from $\dot{C}^{\infty}\left(X\right)$ to a bounded
	map $L_{b}^{2}\left(X\right)\to L^{2}\left(\partial X\right)$.
	
	Let us now prove the statement about compactness. If $\boldsymbol{A}\in\hat{\Psi}_{0\tr}^{-\infty,\left(\mathcal{E}_{\of},\left[\mathcal{E}_{\ff}\right]\right),\boldsymbol{\mathfrak{s}}}\left(X;\partial X,\boldsymbol{E}\right)$,
	$\Re\left(\mathcal{E}_{\of}\right)>-\delta-1$ and $\left[\mathcal{E}_{\ff}\right]=0$,
	and moreover $\hat{N}\left(\boldsymbol{A}\right)\equiv0$, then we
	can find $\varepsilon>0$ such that $\boldsymbol{A}\in\hat{\Psi}_{0\tr}^{-\infty,\left(\mathcal{E}_{\of},\left[\mathcal{H}\right]\right),\boldsymbol{\mathfrak{s}}}\left(X;\partial X,\boldsymbol{E}\right)$
	for some index set $\mathcal{H}$ with $\left[\mathcal{H}\right]=\varepsilon$.
	Arguing as above, we obtain $\boldsymbol{A}\in\hat{\Psi}_{0\tr}^{-\infty,\left(\mathcal{E}_{\of},\left[\mathcal{H}-\varepsilon\right]\right),\boldsymbol{\mathfrak{s}}+\varepsilon}\left(X;\partial X,\boldsymbol{E}\right)$.
	Now, from what we just proved, $\boldsymbol{A}$ is bounded as a map
	\[
	\boldsymbol{A}:x^{\delta}L_{b}^{2}\left(X\right)\to H^{\boldsymbol{\mathfrak{s}}+\delta+\varepsilon}\left(\partial X;\boldsymbol{E}\right).
	\]
	Since $H^{\boldsymbol{\mathfrak{s}}+\delta+\varepsilon}\left(\partial X;\boldsymbol{E}\right)$
	embeds compactly in $H^{\boldsymbol{\mathfrak{s}}+\delta}\left(\partial X;\boldsymbol{E}\right)$,
	we obtain that $\boldsymbol{A}$ is compact as claimed.
	
	(2) The proof is very similar to the previous one, so we only sketch
	the argument. By observing that $x^{-\delta}\boldsymbol{B}\in\hat{\Psi}_{0\po}^{-\infty,\left(\mathcal{E}_{\of}-\delta,\left[\mathcal{E}_{\ff}\right]\right),\delta+\boldsymbol{\mathfrak{s}}}\left(\partial X,\boldsymbol{E};X\right)$
	and that $\boldsymbol{B}$ is bounded $H^{\boldsymbol{\mathfrak{s}}+\delta}\left(\partial X;\boldsymbol{E}\right)\to x^{\delta}L_{b}^{2}\left(X\right)$
	if and only if $x^{-\delta}\boldsymbol{B}$ is bounded $H^{\boldsymbol{\mathfrak{s}}+\delta}\left(\partial X;\boldsymbol{E}\right)\to L_{b}^{2}\left(X\right)$,
	it suffices to prove the statement in the case $\delta=0$; that is,
	it suffices to prove that if $\boldsymbol{B}\in\hat{\Psi}_{0\po}^{-\infty,\left(\mathcal{E}_{\of},\left[\mathcal{E}_{\ff}\right]\right),\boldsymbol{\mathfrak{s}}}\left(\partial X,\boldsymbol{E};X\right)$
	with $\Re\left(\mathcal{E}_{\of}\right)>0$ and $\left[\mathcal{E}_{\ff}\right]=0$,
	then $\boldsymbol{B}$ is bounded $H^{\boldsymbol{\mathfrak{s}}}\left(\partial X;\boldsymbol{E}\right)\to L_{b}^{2}\left(X\right)$.
	Now we can reduce ourselves to the untwisted case again. Taking $\boldsymbol{Q}\in\Psi_{\phg}^{-\left[0\right],\left(\boldsymbol{0},\boldsymbol{\mathfrak{s}}\right)}\left(\partial X;\boldsymbol{E}\right)$
	elliptic, with parametrix $\boldsymbol{K}\in\Psi_{\phg}^{-\left[0\right],\left(\boldsymbol{\mathfrak{s}},\boldsymbol{0}\right)}\left(\partial X;\boldsymbol{E}\right)$
	satisfying $\boldsymbol{Q}\boldsymbol{K}=I-\boldsymbol{R}$ with $\boldsymbol{R}\in\Psi^{-\infty}\left(\partial X;\boldsymbol{E}\right)$,
	we can write
	\[
	\boldsymbol{B}=\boldsymbol{B}\boldsymbol{Q}\boldsymbol{K}+\boldsymbol{B}\boldsymbol{R}.
	\]
	The operator $\boldsymbol{B}\boldsymbol{R}$ is a very residual Poisson
	operator in $\Psi_{\po}^{-\infty,\mathcal{E}_{\of}}\left(\partial X,\boldsymbol{E};X\right)$,
	and therefore the push-forward Theorem implies that $\boldsymbol{B}\boldsymbol{R}$
	is bounded and compact $H^{m}\left(\partial X;\boldsymbol{E}\right)\to\mathcal{A}_{\phg}^{\mathcal{E}_{\of}}\left(X\right)$
	for every $m\in\mathbb{R}$. Concerning $\boldsymbol{B}\boldsymbol{Q}\boldsymbol{K}$,
	using Theorem \ref{thm:global-twisted-compositions-involving-boundary}
	and Corollary \ref{cor:twisted-0-trace-0-poisson-are-untwisted} again,
	we obtain $\boldsymbol{B}\boldsymbol{Q}\in\hat{\Psi}_{0\po}^{-\infty,\left(\mathcal{E}_{\of},\tilde{\mathcal{E}}_{\ff}\right)}\left(\partial X,\boldsymbol{E};X\right)$
	for some index set $\tilde{\mathcal{E}}_{\ff}$ with $\left[\tilde{\mathcal{E}}_{\ff}\right]=\left[\mathcal{E}_{\ff}\right]=0$;
	moreover, $\boldsymbol{K}$ is bounded $\boldsymbol{K}:H^{\boldsymbol{\mathfrak{s}}}\left(\partial X;\boldsymbol{E}\right)\to L^{2}\left(\partial X;\boldsymbol{E}\right)$
	by construction. Therefore, it suffices to prove that if $B\in\hat{\Psi}_{0\po}^{-\infty,\mathcal{E}}\left(\partial X,X\right)$,
	$\Re\left(\mathcal{E}_{\of}\right)>0$, and $\left[\mathcal{E}_{\ff}\right]=0$,
	then $B$ induces a continuous linear map
	\[
	B:L^{2}\left(\partial X\right)\to L_{b}^{2}\left(X\right).
	\]
	We use the same trick as above: the formal adjoint $B^{\dagger}$
	of $B$ is in $\hat{\Psi}_{0\tr}^{-\infty,\mathcal{F}}\left(X,\partial X\right)$,
	with
	\begin{align*}
		\mathcal{F}{}_{\of} & :=\overline{\mathcal{E}_{\of}}-1\\
		\mathcal{F}_{\ff} & :=\overline{\mathcal{E}_{\ff}},
	\end{align*}
	and since $\Re\left(\mathcal{E}_{\of}+\mathcal{F}_{\of}\right)>-1$,
	we have $B^{\dagger}B\in\Psi_{\phg}^{-\left(\mathcal{E}_{\ff}+\mathcal{F}_{\ff}\right)}\left(\partial X\right)$
	by Theorem \ref{thm:global-compositions-involving-boundary}. Since
	\[\left[\mathcal{E}_{\ff}+\mathcal{F}_{\ff}\right]=0,\] $B^{\dagger}B$
	can be seen as a standard pseudodifferential operator of order $0$,
	and therefore it is bounded as an operator
	\[
	B^{\dagger}B:L^{2}\left(\partial X\right)\to L^{2}\left(\partial X\right).
	\]
	It follows that, for every $u\in C^{\infty}\left(\partial X\right)$,
	we have
	\begin{align*}
		\left|\left|Bu\right|\right|_{L_{b}^{2}}^{2} & =\left(B^{\dagger}Bu,u\right)_{L^{2}}\\
		& \leq\left|\left|B^{\dagger}B\right|\right|\left|\left|u\right|\right|_{L^{2}}^{2},
	\end{align*}
	where $\left|\left|B^{\dagger}B\right|\right|$ is the operator norm
	of $B^{\dagger}B$ on $L^{2}\left(\partial X\right)$. Therefore,
	we can extend $B$ by continuity from $C^{\infty}\left(\partial X\right)$
	to a bounded operator $L^{2}\left(\partial X\right)\to L_{b}^{2}\left(X\right)$.
	The statement about compactness is proved exactly as in the trace
	case.
\end{proof}

\section{\label{sec:Parametrix-construction}Parametrix construction}

In this section, we finally provide a solution to the boundary value
problem outlined in \S\ref{sec:Boundary-value-problems}. Now that
we have all the pieces in place, we can outline again the problem
in a precise way, and state the main results.

Let $L\in\Diff_{0}^{m}\left(X\right)$ be a $0$-elliptic operator
with constant indicial roots, and let $\delta\in\mathbb{R}$ be a
surjective, not injective weight for $L$. Fix an auxiliary smooth
positive density $\omega$ for $X$, a vector field $V$ on $X$ transverse
to $\partial X$ and inward-pointing, and a boundary defining function
$x$ on $X$ compatible with $V$ in the sense that $Vx\equiv1$ near
the boundary. Denote by $P_{1},P_{2},G$ the $x^{\delta}L_{b}^{2}$
orthogonal projectors and generalized inverse for $L$ with respect
to the $x^{\delta}L_{b}^{2}$ inner product determined by the singular
density $x^{-2\delta-1}\omega$ (cf.\ \S\ref{subsubsec:Semi-Fredholmness}).
By Mazzeo's Theorem \ref{thm:(Mazzeo)-generalized-projectors-and-gen-inverse},
we have
\begin{align*}
	P_{1} & \in\Psi_{0}^{-\infty,\mathcal{E}}\left(X\right)\\
	P_{2} & \in\Psi^{-\infty,\mathcal{F}}\left(X\right)\\
	G & \in\Psi_{0}^{-m,\mathcal{H}}\left(X\right),
\end{align*}
where the index sets $\mathcal{E}=\left(\mathcal{E}_{\lf},\mathcal{E}_{\rf},\mathcal{E}_{\ff_{0}}\right)$,
$\mathcal{F}=\left(\mathcal{F}_{\lf},\mathcal{F}_{\rf}\right)$ and
$\mathcal{H}=\left(\mathcal{H}_{\lf},\mathcal{H}_{\rf},\mathcal{H}_{\ff_{0}}\right)$,
are as in Theorem \ref{thm:(Mazzeo)-generalized-projectors-and-gen-inverse}.
In particular, $\mathcal{E}_{\lf}$ is the extended union of the simple
index sets generated by the pairs $\left(\mu,M_{\mu}\right)\in\widetilde{\spec}_{b}\left(L\right)$
such that $\Re\left(\mu\right)>\delta$.

Choose now an injective weight $\overline{\delta}$ for $L$. This
implies that $\overline{\delta}>\delta$. The \emph{critical indicial
	roots} for $L$ relative to $\delta,\overline{\delta}$ are the indicial
roots with real part in $\left(\delta,\overline{\delta}\right)$ (cf.\ \S\ref{subsubsec:Critical-indicial-roots}).
Denote by $\boldsymbol{E}_{L}\to\partial X$ the smooth vector bundle
over $\partial X$ defined as
\[
\boldsymbol{E}_{L}=\bigoplus_{\text{\ensuremath{\Re\left(\mu\right)\in\left(\delta,\overline{\delta}\right)}}}E_{\mu,\tilde{M}_{\mu}},
\]
where $E_{\mu,\tilde{M}_{\mu}}\to\partial X$ is the vector bundle
of rank $\tilde{M}_{\mu}+1$, with $\tilde{M}_{\mu}=\max\left\{ m\in\mathbb{N}:\left(\mu,m\right)\in\mathcal{E}_{\lf}\right\} $,
whose smooth sections are the smooth functions on $N^{+}\partial X$
fibrewise log-homogeneous of degree $\leq\left(\mu,\tilde{M}_{\mu}\right)$
(cf.\ \S\ref{subsubsec:Log-homogeneous-functions}). In terms of
the trivialization $N^{+}\partial X\equiv\partial X\times\mathbb{R}_{1}^{1}$
induced by the choice of $V$, a global smooth frame of $E_{\mu,\tilde{M}_{\mu}}$
is given by the functions $x^{\mu},...,x^{\mu}\left(\log x\right)^{\tilde{M}_{\mu}}$
on $N^{+}\partial X$, where $x$ here is the global coordinate on
$N^{+}\partial X\equiv\partial X\times\mathbb{R}_{1}^{1}$ induced
by $V$.

The trace map $\boldsymbol{A}_{L}$ for $L$ is the operator
\[
\boldsymbol{A}_{L}:x^{\delta}H_{b}^{\infty}\left(X\right)\to C^{\infty}\left(\partial X;\boldsymbol{E}_{L}\right)
\]
sending a function $u\in x^{\delta}H_{b}^{\infty}\left(X\right)$
to the coefficients $\left\{ w_{\mu,l}:\text{\ensuremath{\Re\left(\mu\right)\in\left(\delta,\overline{\delta}\right)}, \ensuremath{l\leq\tilde{M}_{\mu}}}\right\} $
of the expansion of $P_{1}u$,
\[
P_{1}u\sim\sum_{\begin{smallmatrix}\left(\alpha,l\right)\in\mathcal{E}_{\lf}\\
		\Re\left(\alpha\right)>\delta
\end{smallmatrix}}w_{\alpha,l}x^{\alpha}\left(\log x\right)^{l},
\]
computed with respect to the choice of the auxiliary vector field
$V$. Note that, by construction, we have $\boldsymbol{A}_{L}P_{1}=\boldsymbol{A}_{L}$.
Calling $\tr_{\mu,\tilde{M}_{\mu}}^{V}:\mathcal{A}_{\phg}^{\mathcal{E}_{\lf}}\left(X\right)\to C^{\infty}\left(\partial X;E_{\mu,\tilde{M}_{\mu}}\right)$
the map sending a function $u\in\mathcal{A}_{\phg}^{\mathcal{E}_{\lf}}\left(X\right)$
to the log-homogeneous term of degree $\leq\left(\mu,\tilde{M}_{\mu}\right)$
in the expansion of $u$ relative to $V$, and enumerating the critical
indicial roots $\mu_{1},...,\mu_{I}$, we can write in terms of the
decomposition $\boldsymbol{E}_{L}=\bigoplus_{i}E_{\mu_{i},\tilde{M}_{\mu_{i}}}$
\[
\boldsymbol{A}_{L}=\left(\begin{matrix}\tr_{\mu_{1},\tilde{M}_{\mu_{1}}}^{V}\circ P_{1}\\
	\vdots\\
	\tr_{\mu_{I},\tilde{M}_{\mu_{I}}}^{V}\circ P_{1}
\end{matrix}\right).
\]
In \S\ref{subsec:The-trace-map-for-L}, we will prove that $\boldsymbol{A}_{L}$
is a twisted symbolic $0$-trace operator. More precisely, denote
by $\boldsymbol{\mathfrak{s}}_{L}:\boldsymbol{E}_{L}\to\boldsymbol{E}_{L}$
the smooth endomorphism of $\boldsymbol{E}_{L}$ induced by the fibrewise
action of $x\partial_{x}$ (the canonical dilation invariant vector
field on $N^{+}\partial X$ tangent to the fibers) on the space of
log-homogeneous functions on $N^{+}\partial X$. We will prove the
following
\begin{thm*}
	(Theorem \ref{thm:A_L-is-twisted-symbolic-0-trace}) $\boldsymbol{A}_{L}\in\hat{\Psi}_{0\tr}^{-\infty,\left(\mathcal{E}_{\rf},\left[0\right]\right),-\boldsymbol{\mathfrak{s}}_{L}}\left(X;\partial X,\boldsymbol{E}_{L}\right)$.
\end{thm*}
We can now formulate the precise notion of elliptic boundary condition:
\begin{defn}
	\label{def:elliptic-boundary-condition-precise}Let $\boldsymbol{W}\to\partial X$
	be a smooth vector bundle equipped with an endomorphism $\boldsymbol{\mathfrak{t}}:\boldsymbol{W}\to\boldsymbol{W}$
	with constant eigenvalues. An \emph{elliptic boundary condition }for
	$L$ relative to $\delta,\overline{\delta}$ is a twisted boundary
	operator $\boldsymbol{Q}\in\Psi_{\phg}^{-\left[0\right],\left(-\boldsymbol{\mathfrak{s}}_{L},\boldsymbol{\mathfrak{t}}\right)}\left(\partial X;\boldsymbol{E}_{L},\boldsymbol{W}\right)$
	such that, for every $\eta\in T^{*}\partial X\backslash0$, the principal
	symbol $\sigma_{\eta}\left(\boldsymbol{Q}\right):\left(\pi^{*}\boldsymbol{E}_{L}\right)_{\eta}\to\left(\pi^{*}\boldsymbol{W}\right)_{\eta}$
	restricts to an isomorphism $\mathcal{C}_{\eta}\to\left(\pi^{*}\boldsymbol{W}\right)_{\eta}$,
	where $\mathcal{C}\subset\pi^{*}\boldsymbol{E}_{L}$ is the Calderón
	bundle for $L$ relative to $\delta,\overline{\delta}$ (cf.\ \S\ref{subsec:Boundary-conditions}).
\end{defn}

The main result of the paper is the following
\begin{thm}
	\label{thm:main-theorem-parametrix}Let $\boldsymbol{Q}\in\Psi_{\phg}^{-\left[0\right],\left(-\boldsymbol{\mathfrak{s}}_{L},\boldsymbol{\mathfrak{t}}\right)}\left(\partial X;\boldsymbol{E}_{L},\boldsymbol{W}\right)$
	be an elliptic boundary condition for $L$ relative to $\delta,\overline{\delta}$.
	Define $\mathcal{L}=\left(\begin{matrix}L & \boldsymbol{Q}\boldsymbol{A}_{L}\end{matrix}\right)^{T}$.
	Then there exist operators $G_{1},G_{2}\in\Psi_{0}^{-m,\mathcal{H}}\left(X\right)$,
	with
	\begin{align*}
		\Re\left(\mathcal{H}_{\lf}\right) & >\delta\\
		\Re\left(\mathcal{H}_{\rf}\right) & >-\delta-1\\
		\left[\mathcal{H}_{\ff_{0}}\right] & =0,
	\end{align*}
	and twisted symbolic $0$-Poisson operators $\boldsymbol{C}_{1},\boldsymbol{C}_{2}\in\hat{\Psi}_{0\po}^{-\infty,\left(\mathcal{E}_{\lf},\left[0\right]\right),\boldsymbol{\mathfrak{t}}}\left(\partial X,\boldsymbol{W};X\right)$,
	such that:
	\begin{enumerate}
		\item $\mathcal{G}_{1}:=\left(\begin{matrix}G_{1} & \boldsymbol{C}_{1}\end{matrix}\right)$
		is a left parametrix for $\mathcal{L}$, satisfying
		\[
		\mathcal{G}_{1}\mathcal{L}=I-R
		\]
		where $R$ is a very residual operator in $\Psi^{-\infty,\left(\mathcal{E}_{\lf},\mathcal{E}_{\rf}\right)}\left(X\right)$;
		\item $\mathcal{G}_{2}:=\left(\begin{matrix}G_{2} & \boldsymbol{C}_{2}\end{matrix}\right)$
		is a right parametrix for $\mathcal{L}$, satisfying
		\[
		\mathcal{L}\mathcal{G}_{2}=I-\left(\begin{matrix}R'\\
			& \boldsymbol{S}
		\end{matrix}\right),
		\]
		where $R'$ is a very residual operator in $\Psi^{-\infty,\left(\mathcal{F}_{\lf},\mathcal{F}_{\rf}\right)}\left(X\right)$,
		and $\boldsymbol{S}\in\Psi^{-\infty}\left(\partial X;\boldsymbol{W}\right)$
		is a smoothing operator on the boundary.
	\end{enumerate}
\end{thm}

This theorem implies two of the main results of this paper. One is
a Fredholm theorem:
\begin{thm*}
	(Theorem \ref{thm:main-theorem-fredholm}) Let $\boldsymbol{Q}$ be
	an elliptic boundary condition for $L$ relative to $\delta,\overline{\delta}$,
	as in Theorem \ref{thm:main-theorem-parametrix}. Then the map
	\[
	\mathcal{L}=L\oplus\boldsymbol{Q}\boldsymbol{A}_{L}:x^{\delta}H_{0}^{k+m}\left(X\right)\to x^{\delta}H_{0}^{k}\left(X\right)\oplus H^{\delta+\boldsymbol{\mathfrak{t}}}\left(\partial X;\boldsymbol{W}\right)
	\]
	is Fredholm.
\end{thm*}
The other is a regularity result for solutions of $0$-elliptic boundary
value problems.
\begin{cor*}
	(Corollary \ref{cor:main-corollary-regularity}) Consider $\boldsymbol{Q}$, an elliptic boundary condition for $L$ relative to $\delta,\overline{\delta}$
	as in Theorem \ref{thm:main-theorem-parametrix}. Suppose that $u\in x^{\delta}H_{0}^{m}\left(X\right)$
	is a solution of the boundary value problem
	\[
	\begin{cases}
		Lu & =v\\
		\boldsymbol{Q}\boldsymbol{A}_{L}u & =\varphi
	\end{cases},
	\]
	where $v\in\mathcal{A}_{\phg}^{\mathcal{F}}\left(X\right)$ for some
	index set $\mathcal{F}$ with $\Re\left(\mathcal{F}\right)>\delta$,
	and $\varphi\in C^{\infty}\left(\partial X;\boldsymbol{W}\right)$.
	Then $u\in\mathcal{A}_{\phg}^{\tilde{\mathcal{F}}}\left(X\right)$
	for some index set $\mathcal{F}$ with $\Re\left(\tilde{\mathcal{F}}\right)>\delta$.
\end{cor*}

\subsection{\label{subsec:The-orthogonal-projector}The orthogonal projector
	onto the kernel of \texorpdfstring{$L$}{L}}

As mentioned above, from \cite{MazzeoEdgeI} we know that the orthogonal
projector $P_{1}$ onto the $x^{\delta}L_{b}^{2}$ kernel of $L$
is an element of the large residual $0$-calculus. In this sub-section,
we improve this result by showing that $P_{1}$ is in fact a $0$-interior
operator. This property will be important in \S\ref{subsec:The-trace-map-for-L}
to prove that $\boldsymbol{A}_{L}$ is a twisted $0$-trace operator.

The proof relies on the composition Theorem \ref{thm:global-composition-0b-0b},
and on the fact that $P_{1}$, being a projector, is idempotent.
\begin{lem}
	\label{lem:P1-vanishes-index-set-ffb}Let $P$ be an operator in $\Psi_{0}^{-\infty,\mathcal{H}}\left(X\right)$,
	with
	\begin{align*}
		\Re\left(\mathcal{H}_{\lf}\right) & >\delta\\
		\Re\left(\mathcal{H}_{\rf}\right) & >-\delta-1\\
		\left[\mathcal{H}_{\ff_{0}}\right] & =0.
	\end{align*}
	Suppose that the Bessel family $\hat{N}\left(P\right)$ is very residual,
	that is, it is a smooth section of the pull-back bundle $\pi^{*}\Psi^{-\infty,\left(\mathcal{H}_{\lf},\mathcal{H}_{\rf}\right)}\left(N^{+}\partial X\right)$
	over $T^{*}\partial X\backslash0$, fibrewise homogeneous of degree
	$0$. Then $P$ is in $\hat{\Psi}_{0b}^{-\infty,\left(\mathcal{H}_{\lf},\mathcal{H}_{\rf},\mathcal{I}_{\ff_{b}},\mathcal{I}_{\ff_{0}}\right)}\left(X\right)$,
	where $\left[\mathcal{I}_{\ff_{0}}\right]=0$ and $\Re\left(\mathcal{I}_{\ff_{b}}\right)>0$.
\end{lem}

\begin{proof}
	Since $\hat{N}\left(P\right)$ is a section of $\pi^{*}\Psi^{-\infty,\left(\mathcal{H}_{\lf},\mathcal{H}_{\rf}\right)}\left(N^{+}\partial X\right)$
	over $T^{*}\partial X\backslash0$, fibrewise homogeneous of degree
	$0$, we can realize it as the Bessel family of a $0$-interior operator
	$Q\in\hat{\Psi}_{0}^{-\infty,\left(\mathcal{H}_{\lf},\mathcal{H}_{\rf},0\right)}\left(X\right)$.
	Now, by Corollary \ref{cor:global-relation-symbolic0b-physical0b},
	we have $Q\in\Psi_{0}^{-\infty,\left(\mathcal{E}_{\lf},\mathcal{E}_{\rf},\mathcal{I}_{\ff_{0}}'\right)}\left(X\right)$,
	where $\mathcal{I}_{\ff_{0}}'=\mathbb{N}\eunion\left(\mathcal{H}_{\lf}+\mathcal{H}_{\rf}+n+1\right)$.
	Since $\Re\left(\mathcal{H}_{\lf}\right)>\delta$ and $\Re\left(\mathcal{H}_{\rf}\right)>-\delta-1$,
	we have $\mathcal{I}_{\ff_{0}}'=\mathbb{N}\cup\mathcal{I}$ for some
	$\mathcal{I}$ with $\Re\left(\mathcal{I}\right)>n$, and therefore
	$\left[\mathcal{I}_{\ff_{0}}'\right]=0$. Now, by construction, $P$
	and $Q$ have the same Bessel family, and therefore they have the
	same normal family. This implies that $R:=P-Q\in\Psi_{0}^{-\infty,\left(\mathcal{H}_{\lf},\mathcal{H}_{\rf},\mathcal{I}_{\ff_{0}}\backslash\left\{ 0\right\} \right)}\left(X\right)$,
	where $\mathcal{I}_{\ff_{0}}$ is the smallest index set containing
	$\mathcal{I}_{\ff_{0}}'$ and $\mathcal{H}_{\ff_{0}}$. By the pull-back
	Theorem, we then have $R\in\Psi_{0b}^{-\infty,\left(\mathcal{H}_{\lf},\mathcal{H}_{\rf},\mathcal{H}_{\lf}+\mathcal{H}_{\rf}+1,\mathcal{I}_{\ff_{0}}\backslash\left\{ 0\right\} \right)}\left(X\right)$,
	and again by Corollary \ref{cor:global-relation-symbolic0b-physical0b},
	we obtain $R\in\hat{\Psi}_{0b}^{-\infty,\left(\mathcal{H}_{\lf},\mathcal{H}_{\rf},\mathcal{I}_{\ff_{b}},\mathcal{I}_{\ff_{0}}\backslash\left\{ 0\right\} \right)}\left(X\right)$,
	where \[\mathcal{I}_{\ff_{b}}=\left(\mathcal{E}_{\lf}+\mathcal{E}_{\rf}+1\right)\eunion\left(\mathcal{I}_{\ff_{0}}\backslash\left\{ 0\right\} \right).\]
	Note that $\Re\left(\mathcal{I}_{\ff_{b}}\right)>0$. Finally, write
	$P=R+Q$; since \[Q\in\hat{\Psi}_{0}^{-\infty,\left(\mathcal{H}_{\lf},\mathcal{H}_{\rf},0\right)}\left(X\right)\subseteq\hat{\Psi}_{0b}^{-\infty,\left(\mathcal{H}_{\lf},\mathcal{H}_{\rf},\mathcal{H}_{\lf}+\mathcal{H}_{\rf}+1,0\right)}\left(X\right),\]
	we obtain $P\in\hat{\Psi}_{0b}^{-\infty,\left(\mathcal{H}_{\lf},\mathcal{H}_{\rf},\mathcal{I}_{\ff_{b}},\mathcal{I}_{\ff_{0}}\right)}\left(X\right)$.
\end{proof}
\begin{lem}
	\label{lem:conormal-on-M20b-implies-conormal-on-M20}Let $f\left(x,\tilde{x},\eta\right)\in\mathcal{A}^{\left(a,b,a+b,c,\infty,\infty,\infty\right)}\left(\hat{M}_{0b}^{2}\right)$
	for some real numbers $a,b,c$. Then $f\left(x,\tilde{x},\eta\right)$
	is conormal on the blow-down space $\hat{M}_{0}^{2}$, and more precisely
	\[f\left(x,\tilde{x},\eta\right)\in\mathcal{A}^{\left(a,b,c,\infty,\infty,\infty\right)}\left(\hat{M}_{0}^{2}\right).\]
\end{lem}

\begin{proof}
	This is similar to the first part of the proof of Proposition \ref{prop:improved-hintz-lemma-for-0-interior}.
	Let $r_{\lf},r_{\rf},r_{\ff_{0}}$ be boundary defining functions
	for the respective faces of $\hat{M}_{0}^{2}$. Then there are boundary
	defining functions $\tilde{r}_{\lf},\tilde{r}_{\rf},\tilde{r}_{\ff_{b}}$
	for the corresponding faces of $\hat{M}_{0b}^{2}$ such that $r_{\lf}$
	lifts to $\tilde{r}_{\lf}\tilde{r}_{\ff_{b}}$ and $r_{\rf}$ lifts
	to $\tilde{r}_{\rf}\tilde{r}_{\ff_{b}}$. Moreover, $r_{\ff_{0}}$
	lifts to a boundary defining function for $\ff_{0}$ on $\hat{M}_{0b}^{2}$,
	which we still call $r_{\ff_{0}}$. Now, if $f\in\mathcal{A}^{\left(a,b,a+b,c,\infty,\infty,\infty\right)}\left(\hat{M}_{0b}^{2}\right)$,
	then \[r_{\lf}^{-a}r_{\rf}^{-b}r_{\ff_{0}}^{-c}f\in\mathcal{A}^{\left(0,0,0,0,\infty,\infty,\infty\right)}\left(\hat{M}_{0b}^{2}\right)\subseteq L^{\infty}\left(\hat{M}_{0}^{2}\right).\]
	If we hit $r_{\lf}^{-a}r_{\rf}^{-b}r_{\ff_{0}}^{-c}f$ with any number
	of $b$-vector fields on $\hat{M}_{0}^{2}$, this estimate persists
	because (since $\hat{M}_{0b}^{2}$ is obtained from $\hat{M}_{0}^{2}$
	by blowing up a corner) $b$-vector fields on $\hat{M}_{0}^{2}$ lift
	to $b$-vector fields on $\hat{M}_{0b}^{2}$. Therefore, $r_{\lf}^{-a}r_{\rf}^{-b}r_{\ff_{0}}^{-c}f\in\mathcal{A}^{\left(0,0,0,\infty,\infty,\infty\right)}\left(\hat{M}_{0}^{2}\right)$,
	which is equivalent to $f\in\mathcal{A}^{\left(a,b,c,\infty,\infty,\infty\right)}\left(\hat{M}_{0}^{2}\right)$.
\end{proof}
\begin{thm}
	\label{thm:P1_is_0_interior}The projector $P_{1}$ is in $\hat{\Psi}_{0}^{-\infty,\left(\mathcal{E}_{\lf},\mathcal{E}_{\rf},\mathcal{I}_{\ff_{0}}\right)}\left(X\right)$,
	where $\mathcal{I}_{\ff_{0}}$ is an index set with $\left[\mathcal{I}_{\ff_{0}}\right]=0$.
\end{thm}

\begin{proof}
	The idea of the proof is quite similar to the proof of Proposition
	\ref{prop:improved-hintz-lemma-for-0-interior}. First of all, observe
	that by Proposition \ref{prop:The-Bessel-family-of-P1}, we can apply
	Lemma \ref{lem:P1-vanishes-index-set-ffb} to $P_{1}$ and conclude
	that
	\[
	P_{1}\in\hat{\Psi}_{0b}^{-\infty,\left(\mathcal{H}_{\lf},\mathcal{H}_{\rf},\mathcal{I}_{\ff_{b}},\mathcal{I}_{\ff_{0}}\right)}\left(X\right)
	\]
	where $\left[\mathcal{I}_{\ff_{0}}\right]=0$ and $\Re\left(\mathcal{I}_{\ff_{b}}\right)>0$.
	Now, for every $m\in\mathbb{R}$, define
	\[
	Q_{\delta+m}=\prod_{\begin{smallmatrix}\left(\alpha,l\right)\in\mathcal{E}_{\lf}\\
			\Re\left(\alpha\right)\leq\delta+m
	\end{smallmatrix}}\left(xV-\alpha\right)
	\]
	where $V$, $x$ are the auxiliary vector field and boundary defining
	function chosen at the beginning of the section. We will show that,
	for every $m$, we have
	\begin{align*}
		Q_{\delta+m}P_{1} & \in\hat{\Psi}_{0b}^{-\infty,\left(\mathcal{E}_{\lf}\left(\delta+m\right),\mathcal{E}_{\rf},\mathcal{I}_{m},\mathcal{I}_{\ff_{0}}\right)}\left(X\right),
	\end{align*}
	where $\mathcal{E}_{\lf}\left(\delta+m\right)=\left\{ \left(\alpha,l\right)\in\mathcal{E}_{\lf}:\Re\left(\alpha\right)>\delta+m\right\} $,
	$\mathcal{I}_{m}$ is an index set with $\Re\left(\mathcal{I}_{m}\right)>m$,
	and $\left[\mathcal{I}_{\ff_{0}}\right]=0$. Taking adjoints with
	respect to the $x^{\delta}L_{b}^{2}$ inner product induced by $x,\omega$,
	from self-adjointness of $P_{1}$ we get
	\[
	P_{1}Q_{\delta+m}^{\dagger}\in\hat{\Psi}_{0b}^{-\infty,\left(\mathcal{E}_{\lf},\mathcal{E}_{\rf}\left(-\delta-1+m\right),\mathcal{I}_{m},\mathcal{I}_{\ff_{0}}\right)}\left(X\right),
	\]
	where again $\mathcal{E}_{\rf}\left(-\delta-1+m\right)=\left\{ \left(\alpha,l\right)\in\mathcal{E}_{\rf}:\Re\left(\alpha\right)>-\delta-1+m\right\} $
	and $\Re\left(\mathcal{I}_{m}\right)>m$. We will show that this implies
	that $P_{1}\in\hat{\Psi}_{0}^{-\infty,\left(\mathcal{E}_{\lf},\mathcal{E}_{\rf},\mathcal{I}_{\ff_{0}}\right)}\left(X\right)$.
	
	Since $P_{1}\in\Psi_{0}^{-\infty,\left(\mathcal{E}_{\lf},\mathcal{E}_{\rf},\mathcal{E}_{\ff_{0}}\right)}\left(X\right)$,
	we have $Q_{\delta+m}P_{1}\in\Psi_{0}^{-\infty,\left(\mathcal{E}_{\lf}\left(\delta+m\right),\mathcal{E}_{\rf},\mathcal{E}_{\ff_{0}}\right)}\left(X\right)$.
	Indeed, by design, the $0$-differential operator $Q_{\delta+m}$
	lifts to $X_{0}^{2}$ from the left and kills all the terms with $\Re\left(\alpha\right)\leq\delta+m$
	in the expansion of a polyhomogeneous function on $X_{0}^{2}$ with
	index set $\mathcal{E}_{\lf}$ at the left face. Moreover, the Bessel
	family of $Q_{\delta+m}P_{1}$ is $\hat{N}\left(Q_{\delta+m}\right)\hat{N}\left(P_{1}\right)$,
	hence it is again a section of $\pi^{*}\Psi^{-\infty,\left(\mathcal{E}_{\lf}\left(\delta+m\right),\mathcal{E}_{\rf}\right)}\left(N^{+}\partial X\right)$.
	Arguing exactly as in the proof of Lemma \ref{lem:P1-vanishes-index-set-ffb},
	we can prove that $Q_{\delta+m}P_{1}\in\hat{\Psi}_{0b}^{-\infty,\left(\mathcal{E}_{\lf}\left(\delta+m\right),\mathcal{E}_{\rf},\mathcal{I}_{\ff_{b}},\mathcal{E}_{\ff_{0}}\right)}\left(X\right)$
	where $\mathcal{I}_{\ff_{b}}$ is an index set with $\Re\left(\mathcal{I}_{\ff_{b}}\right)>0$.
	Now we use the fact that $P_{1}=P_{1}^{2}$, so $Q_{\delta+m}P_{1}=\left(Q_{\delta+m}P_{1}\right)P_{1}$.
	By the composition Theorem \ref{thm:global-composition-0b-0b}, we
	can improve the index set at $\ff_{b}$ and obtain that $\left(Q_{\delta+m}P_{1}\right)P_{1}\in\hat{\Psi}_{0b}^{-\infty,\left(\mathcal{E}_{\lf}\left(\delta+m\right),\mathcal{E}_{\rf},\mathcal{J},\mathcal{E}_{\ff_{0}}\right)}\left(X\right)$,
	where $\mathcal{J}=\left(\mathcal{I}_{\ff_{b}}+\mathcal{I}_{\ff_{b}}\right)\eunion\left(\mathcal{E}_{\lf}\left(\delta+m\right)+\mathcal{E}_{\rf}+1\right)$.
	This is an improvement in the following sense: since $\Re\left(\mathcal{E}_{\lf}\left(\delta+m\right)\right)>\delta+m$
	and $\Re\left(\mathcal{E}_{\rf}\right)>-\delta-1$, we have $\Re\left(\mathcal{E}_{\lf}\left(\delta+m\right)+\mathcal{E}_{\rf}+1\right)>m$;
	therefore, if $\Re\left(\mathcal{I}_{\ff_{b}}\right)>\varepsilon$,
	we have $\Re\left(\mathcal{J}\right)>\min\left(2\varepsilon,m\right)$.
	But now we can iterate the procedure, until we indeed get $Q_{\delta+m}P_{1}\in\hat{\Psi}_{0b}^{-\infty,\left(\mathcal{E}_{\lf}\left(\delta+m\right),\mathcal{E}_{\rf},\mathcal{I}_{m},\mathcal{E}_{\ff_{0}}\right)}\left(X\right)$
	for some index set $\mathcal{I}_{m}$ with $\Re\left(\mathcal{I}_{m}\right)>m$.
	Taking adjoints, we get \[P_{1}Q_{\delta+m}^{\dagger}\in\hat{\Psi}_{0b}^{-\infty,\left(\mathcal{E}_{\lf},\mathcal{E}_{\rf}\left(-\delta-1+m\right),\mathcal{I}_{m},\mathcal{I}_{\ff_{0}}\right)}\left(X\right)\]
	as well.
	
	Now, choose an atlas $\left(U_{i},\varphi_{i}\right)$ for $\partial X$,
	let $[0,\varepsilon)\times\partial X\hookrightarrow X$ be the collar
	of $\partial X$ in $X$ induced by $V$, and let $U_{i}'=[0,\varepsilon)'\times U_{i}$.
	We can then decompose
	\[
	P_{1}=\sum_{i}P_{1,i}+R,
	\]
	where $P_{1,i}$ is compactly supported on $U_{i}'\times U_{i}'$,
	in the coordinates $\left(x,y\right)$ associated to the chart $\left(U_{i},\varphi_{i}\right)$
	we have
	\[
	P_{1,i}=\Op_{L}^{\inte}\left(p_{i}\left(y;x,\tilde{x},\eta\right)\right)
	\]
	for some $0$-interior symbol
	\[
	p_{i}\left(y;x,\tilde{x},\eta\right)\in S_{0b,\mathcal{S}}^{-\infty,\left(\mathcal{E}_{\lf},\mathcal{E}_{\rf},\mathcal{I}_{\ff_{b}},\mathcal{I}_{\ff_{0}}\right)}\left(\mathbb{R}^{n};\mathbb{R}_{1}^{n+1}\right),
	\]
	and $R$ is compactly supported in the complement of a closed neighborhood
	of $\partial\Delta$ in $X^{2}$. From the decomposition above, a
	priori we would have $R\in\Psi_{b}^{-\infty,\left(\mathcal{E}_{\lf},\mathcal{E}_{\rf},\mathcal{I}_{\ff_{b}}\right)}\left(X\right)$.
	However, we know already that $P_{1}\in\Psi_{0}^{-\infty,\left(\mathcal{E}_{\lf},\mathcal{E}_{\rf},\mathcal{E}_{\ff_{0}}\right)}\left(X\right)$,
	so in fact $R$ agrees with a very residual operator far away from
	$\partial\Delta$. The only problem could a priori arise in the intersection
	between the open set $\bigcup_{i}U_{i}'\times U_{i}'$ and the support
	of $R$, where the expansion of $\kappa_{R}$ at $\ff_{b}$ has to
	make up for the expansions at $\ff_{b}$ of the lifted Schwartz kernels
	$\kappa_{P_{1,i}}$. However, if we can prove that the symbols $p_{i}\left(y;x,\tilde{x},\eta\right)$
	are in fact \emph{$0$-interior }symbols, then we can conclude that
	$P_{1,i}$ coincides with a very residual operator away from $\partial\Delta$
	as well, and therefore $R\in\Psi^{-\infty,\left(\mathcal{E}_{\lf},\mathcal{E}_{\rf}\right)}\left(X\right)$.
	
	Let's prove that $p_{i}\left(y;x,\tilde{x},\eta\right)\in S_{0,\mathcal{S}}^{-\infty,\left(\mathcal{E}_{\lf},\mathcal{E}_{\rf},\mathcal{I}_{\ff_{0}}\right)}\left(\mathbb{R}^{n};\mathbb{R}_{1}^{n+1}\right)$.
	Since \[Q_{\delta+m}P_{1}\in\hat{\Psi}_{0b}^{-\infty,\left(\mathcal{E}_{\lf}\left(\delta+m\right),\mathcal{E}_{\rf},\mathcal{I}_{m},\mathcal{I}_{\ff_{0}}\right)}\left(X\right),\]
	the symbol $Q_{\delta+m}p_{i}$ is in $S_{0b,\mathcal{S}}^{-\infty,\left(\mathcal{E}_{\lf}\left(\delta+m\right),\mathcal{E}_{\rf},\mathcal{I}_{m},\mathcal{I}_{\ff_{0}}\right)}\left(\mathbb{R}^{n};\mathbb{R}_{1}^{n+1}\right)$.
	Since $P_{1,i}$ is supported on $U_{i}'\times U_{i}'$, which is
	contained in the collar induced by $V$, in these coordinates we have
	\begin{align*}
		Q_{\delta+m}p_{i} & =\prod_{\begin{smallmatrix}\left(\alpha,l\right)\in\mathcal{E}_{\lf}\\
				\Re\left(\alpha\right)\leq\delta+m
		\end{smallmatrix}}\left(x\partial_{x}-\alpha\right)p_{i}\\
		& \in S_{0b,\mathcal{S}}^{-\infty,\left(\mathcal{E}_{\lf}\left(\delta+m\right),\mathcal{E}_{\rf},\mathcal{I}_{m},\mathcal{I}_{\ff_{0}}\right)}\left(\mathbb{R}^{n};\mathbb{R}_{1}^{n+1}\right).
	\end{align*}
	Similarly, the fact that $P_{1}Q_{\delta+m}^{\dagger}\in\hat{\Psi}_{0b}^{-\infty,\left(\mathcal{E}_{\lf},\mathcal{E}_{\rf}\left(-\delta-1+m\right),\mathcal{I}_{m},\mathcal{I}_{\ff_{0}}\right)}\left(X\right)$
	implies that
	\begin{align*}
		\tilde{Q}_{-\delta-1+m}p_{i} & =\prod_{\begin{smallmatrix}\left(\alpha,l\right)\in\mathcal{E}_{\rf}\\
				\Re\left(\alpha\right)\leq-\delta-1+m
		\end{smallmatrix}}\left(\tilde{x}\partial_{\tilde{x}}-\alpha\right)p_{i}\\
		& \in S_{0b,\mathcal{S}}^{-\infty,\left(\mathcal{E}_{\lf},\mathcal{E}_{\rf}\left(-\delta-1+m\right),\mathcal{I}_{m},\mathcal{I}_{\ff_{0}}\right)}\left(\mathbb{R}^{n};\mathbb{R}_{1}^{n+1}\right).
	\end{align*}
	The rest of the proof is very similar to that of Proposition \ref{prop:improved-hintz-lemma-for-0-interior}.
	We have
	\begin{align*}
		Q_{\delta+m}p_{i} & \in\mathcal{S}\left(\mathbb{R}^{n}\right)\otimes\mathcal{A}_{\phg}^{\left(\mathcal{E}_{\lf}\left(\delta+m\right),\mathcal{E}_{\rf},\mathcal{I}_{m}-1,\mathcal{I}_{\ff_{0}}-1,\infty,\infty,\infty\right)}\left(\hat{M}_{0b}^{2}\right)\\
		& \subseteq\mathcal{S}\left(\mathbb{R}^{n}\right)\otimes\mathcal{A}^{\left(\delta+m,-\delta-1,m-1,-1,\infty,\infty,\infty\right)}\left(\hat{M}_{0b}^{2}\right)\\
		& \subseteq\mathcal{S}\left(\mathbb{R}^{n}\right)\otimes\mathcal{A}^{\left(\delta+m,-\delta-1,-1,\infty,\infty,\infty\right)}\left(\hat{M}_{0}^{2}\right)
	\end{align*}
	where in the last inclusion we used Lemma \ref{lem:conormal-on-M20b-implies-conormal-on-M20}.
	Similarly, we obtain
	\[
	\tilde{Q}_{-\delta-1+m}p_{i}\in\mathcal{S}\left(\mathbb{R}^{n}\right)\otimes\mathcal{A}^{\left(\delta,-\delta-1+m,-1,\infty,\infty,\infty\right)}\left(\hat{M}_{0}^{2}\right).
	\]
	Finally, choosing a vector field $W$ on $\hat{M}_{0}^{2}$ transverse
	to $\ff_{0}$ and inward-pointing, and tangent to the other faces,
	and calling
	\[
	S_{\lambda}=\prod_{\begin{smallmatrix}\left(\alpha,l\right)\in\mathcal{I}_{\ff_{0}}-1\\
			\Re\left(\alpha\right)\leq\lambda
	\end{smallmatrix}}\left(r_{\ff_{0}}W-\alpha\right),
	\]
	we have
	\[
	S_{\lambda}p_{i}\in\mathcal{S}\left(\mathbb{R}^{n}\right)\otimes\mathcal{A}^{\left(\delta,-\delta-1,\lambda,\infty,\infty,\infty\right)}\left(\hat{M}_{0}^{2}\right).
	\]
	These results together ensure that
	\begin{align*}
		p_{i} & \in\mathcal{S}\left(\mathbb{R}^{n}\right)\hat{\otimes}\mathcal{A}_{\phg}^{\left(\mathcal{E}_{\lf},\mathcal{E}_{\rf},\mathcal{I}_{\ff_{0}}-1,\infty,\infty,\infty\right)}\left(\hat{M}_{0}^{2}\right)\\
		& =S_{0,\mathcal{S}}^{-\infty,\left(\mathcal{E}_{\lf},\mathcal{E}_{\rf},\mathcal{I}_{\ff_{0}}\right)}\left(\mathbb{R}^{n};\mathbb{R}_{1}^{n+1}\right).
	\end{align*}
\end{proof}

\subsection{\label{subsec:The-trace-map-for-L}The trace map for \texorpdfstring{$L$}{L}}

We can now prove that the trace map $\boldsymbol{A}_{L}$ is a twisted
symbolic $0$-trace operator.
\begin{thm}
	\label{thm:A_L-is-twisted-symbolic-0-trace}$\boldsymbol{A}_{L}$
	is a twisted symbolic $0$-trace operator in $\hat{\Psi}_{0\tr}^{-\infty,\left(\mathcal{E}_{\rf},\left[0\right]\right),-\boldsymbol{\mathfrak{s}}_{L}}\left(X;\partial X,\boldsymbol{E}_{L}\right)$.
	Moreover, $\hat{N}\left(\boldsymbol{A}_{L}\right)\equiv\hat{\boldsymbol{a}}_{L}$,
	where $\hat{\boldsymbol{a}}_{L}$ is the Bessel trace map (cf.\ Definition
	\ref{def:construction-of-theBessel-trace-map}).
\end{thm}

\begin{proof}
	Enumerating the critical indicial roots $\mu_{1},...,\mu_{I}$, write
	\[
	\boldsymbol{A}_{L}=\left(\begin{matrix}\tr_{\mu_{1},\tilde{M}_{\mu_{1}}}^{V}\\
		\vdots\\
		\tr_{\mu_{I},\tilde{M}_{\mu_{I}}}^{V}
	\end{matrix}\right)P_{1}.
	\]
	Given this decomposition, we can assume without loss of generality
	that there is only one critical indicial root $\mu$, so that $\tilde{M}_{\mu}=M_{\mu}$
	with $M_{\mu}+1$ equal to the multiplicity of $\mu$, $A_{L}=\tr_{\mu,M_{\mu}}^{V}\circ P_{1}$,
	$\boldsymbol{E}_{L}=E_{L}=E_{\mu,M_{\mu}}$, and $\boldsymbol{\mathfrak{s}}_{L}=\mu+\mathfrak{N}$
	with $\mathfrak{N}=\mathfrak{N}_{M_{\mu}}$.
	
	Now, choose an atlas $\left(U_{i},\varphi_{i}\right)$ for $\partial X$,
	let $[0,\varepsilon)\times\partial X\hookrightarrow X$ be the collar
	of $\partial X$ in $X$ induced by $V$, and let $U_{i}'=[0,\varepsilon)'\times U_{i}$.
	From Theorem \ref{thm:P1_is_0_interior}, we have $P_{1}\in\hat{\Psi}_{0}^{-\infty,\left(\mathcal{E}_{\lf},\mathcal{E}_{\rf},\mathcal{I}_{\ff_{0}}\right)}\left(X\right)$
	with $\left[\mathcal{I}_{\ff_{0}}\right]=0$. We can then decompose
	\[
	P_{1}=\sum_{i}P_{1,i}+R,
	\]
	where $P_{1,i}$ is compactly supported on $U_{i}'\times U_{i}'$,
	in the coordinates $\left(x,y\right)$ associated to the chart $\left(U_{i},\varphi_{i}\right)$
	we have
	\[
	P_{1,i}=\Op_{L}^{\inte}\left(p_{i}\left(y;x,\tilde{x},\eta\right)\right)
	\]
	for some $0$-interior symbol
	\[
	p_{i}\left(y;x,\tilde{x},\eta\right)\in S_{0,\mathcal{S}}^{-\infty,\left(\mathcal{E}_{\lf},\mathcal{E}_{\rf},\mathcal{I}_{\ff_{0}}\right)}\left(\mathbb{R}^{n};\mathbb{R}_{1}^{n+1}\right),
	\]
	and $R$ is very residual in $\Psi^{-\infty,\left(\mathcal{E}_{\lf},\mathcal{E}_{\rf}\right)}\left(X\right)$.
	
	The composition $\tr_{\mu,M_{\mu}}^{V}\circ R$ has Schwartz kernel
	$\tr_{\mu,M_{\mu}}^{\pi_{L}^{*}V}K_{R}$, where $\pi_{L}^{*}V$ is
	the lift from the left of $V$ to $X^{2}$. Since $K_{R}\in\mathcal{A}_{\phg}^{\left(\mathcal{E}_{\lf},\mathcal{E}_{\rf}\right)}\left(X^{2};\pi_{R}^{*}\mathcal{D}_{X}^{1}\right)$,
	we have \[\tr_{\mu,M_{\mu}}^{\pi_{L}^{*}V}K_{R}\in\mathcal{A}_{\phg}^{\mathcal{E}_{\rf}}\left(\partial X\times X;\pi_{L}^{*}E_{L}\otimes\pi_{R}^{*}\mathcal{D}_{X}^{1}\right),\]
	which means that $\tr_{\mu,M_{\mu}}^{V}\circ R$ is very residual
	trace in $\Psi_{\tr}^{-\infty,\mathcal{E}_{\of}}\left(X;\partial X,\boldsymbol{E}_{L}\right)$.
	Here we are tacitly identifying the inward-pointing normal bundle
	$N^{+}\left(\partial X\times X\right)$ of the left face of $X^{2}$
	with the pull-back $\pi_{L}^{*}N^{+}\partial X$ of the inward-pointing
	normal bundle of $\partial X$ in $X$ via the projection $\pi_{L}:X^{2}\to X$.
	
	Now let's consider the terms $\tr_{\mu,M_{\mu}}^{V}\circ P_{1,i}$.
	It is convenient to consider the $p_{i}\left(y;x,\tilde{x},\eta\right)$
	as polyhomogeneous functions on $\overline{\mathbb{R}^{n}}\times\hat{D}_{0}^{2}$,
	where $\hat{D}_{0}^{2}\equiv\overline{\mathbb{R}_{2}^{2}}\times\overline{\mathbb{R}^{n}}$
	is the alternative compactification of $\hat{M}_{0}^{2}\backslash\left(\iif_{\eta}\cup\iif_{x}\cup\iif_{\tilde{x}}\right)$
	used in the proof of Theorem \ref{thm:compositions-involving-interior}
	and depicted in Figure \ref{fig:Dhat}. There, we observed that introducing
	the variables $\tau=x\left\langle \eta\right\rangle $ and $\tilde{\tau}=\tilde{x}\left\langle \eta\right\rangle $,
	the map
	\begin{align*}
		\varphi:\mathbb{R}_{2}^{2}\times\mathbb{R}^{n} & \to\mathbb{R}_{2}^{2}\times\mathbb{R}^{n}\\
		\left(\tau,\tilde{\tau},\eta\right) & \mapsto\left(\left\langle \eta\right\rangle ^{-1}\tau,\left\langle \eta\right\rangle ^{-1}\tilde{\tau},\eta\right)
	\end{align*}
	induces (by pull-back) an identification
	\[
	\varphi^{*}:\mathcal{A}_{\phg}^{\left(\mathcal{E}_{\lf},\mathcal{E}_{\rf},\mathcal{I}_{\ff_{0}}-1,\infty,\infty,\infty\right)}\left(\hat{M}_{0}^{2}\right)\to\mathcal{A}_{\phg}^{\left(\mathcal{E}_{\lf},\mathcal{E}_{\rf},\mathcal{I}_{\ff_{0}}-1,\infty\right)}\left(\hat{D}_{0}^{2}\right).
	\]
	In other words, if $p\left(x,\tilde{x},\eta\right)\in\mathcal{A}_{\phg}^{\left(\mathcal{E}_{\lf},\mathcal{E}_{\rf},\mathcal{I}_{\ff_{0}}-1,\infty,\infty,\infty\right)}\left(\hat{M}_{0}^{2}\right)$,
	then calling $p':=\varphi^{*}p$ we have 
	\[
	p'\left(\tau,\tilde{\tau},\eta\right)\in\mathcal{A}_{\phg}^{\left(\mathcal{E}_{\lf},\mathcal{E}_{\rf},\mathcal{I}_{\ff_{0}}-1,\infty\right)}\left(\hat{D}_{0}^{2}\right)=\mathcal{A}_{\phg}^{\left(\mathcal{E}_{\lf},\mathcal{E}_{\rf},\infty\right)}\left(\overline{\mathbb{R}_{2}^{2}}\right)\hat{\otimes}\mathcal{A}_{\phg}^{\mathcal{I}_{\ff_{0}}-1}\left(\overline{\mathbb{R}^{n}}\right).
	\]
	Therefore, we have
	\[
	a'\left(\tilde{\tau},\eta\right):=\left(\tr_{\mu,M_{\mu}}^{\partial_{\tau}}p'\right)\left(\tilde{\tau},\eta\right)\in\mathcal{A}_{\phg}^{\left(\mathcal{E}_{\rf},\infty\right)}\left(\overline{\mathbb{R}_{1}^{1}}\right)\hat{\otimes}\mathcal{A}_{\phg}^{\mathcal{I}_{\ff_{0}}-1}\left(\overline{\mathbb{R}^{n}}\right).
	\]
	Note that the left face of $\hat{D}_{0}^{2}$, identified with $\overline{\mathbb{R}_{1}^{1}}\times\overline{\mathbb{R}^{n}}$
	by the coordinates $\tau,\eta$, can be seen as an alternative compactification
	of $\hat{T}_{0}^{2}\backslash\left(\iif_{\eta}\cup\iif_{\tilde{x}}\right)$;
	moreover, the restriction $\varphi_{\lf}$ of $\varphi$ to the left
	faces determines an identification
	\[
	\varphi_{\lf}^{*}:\mathcal{A}_{\phg}^{\left(\mathcal{E}_{\rf},\mathcal{I}_{\ff_{0}}-1,\infty,\infty\right)}\left(\hat{T}_{0}^{2}\right)\to\mathcal{A}_{\phg}^{\left(\mathcal{E}_{\rf},\mathcal{I}_{\ff_{0}}-1,\infty\right)}\left(\lf\left(\hat{D}_{0}^{2}\right)\right).
	\]
	Thus, the distribution $a\left(\tilde{x},\eta\right):=a'\left(\tilde{x}\left\langle \eta\right\rangle ,\eta\right)$
	lifts to an element of $\mathcal{A}_{\phg}^{\left(\mathcal{E}_{\rf},\mathcal{I}_{\ff_{0}}-1,\infty,\infty\right)}\left(\hat{T}_{0}^{2}\right)$.
	Now, if $u\left(\tilde{x}\right)\in\dot{C}^{\infty}\left(\overline{\mathbb{R}_{1}^{1}}\right)$,
	using Lemma \ref{lem:trace-and-dilations} we have
	\begin{align*}
		& \tr_{\mu,M_{\mu}}^{\partial_{x}}\left(\int p\left(x,\tilde{x},\eta\right)u\left(\tilde{x}\right)d\tilde{x}\right)\\
		= & \tr_{\mu,M_{\mu}}^{\partial_{x}}\int p\left(x,\left\langle \eta\right\rangle ^{-1}\tilde{\tau},\eta\right)u\left(\left\langle \eta\right\rangle ^{-1}\tilde{\tau}\right)\left\langle \eta\right\rangle ^{-1}d\tilde{\tau}\\
		= & \int\tr_{\mu,M_{\mu}}^{\partial_{x}}\left(p'\left(\left\langle \eta\right\rangle x,\tilde{\tau},\eta\right)\right)u\left(\left\langle \eta\right\rangle ^{-1}\tilde{\tau}\right)\left\langle \eta\right\rangle ^{-1}d\tilde{\tau}\\
		= & \int\tr_{\mu,M_{\mu}}^{\partial_{x}}\left(\left(\lambda_{\left\langle \eta\right\rangle }^{*}p'\right)\left(x,\tilde{\tau},\eta\right)\right)u\left(\left\langle \eta\right\rangle ^{-1}\tilde{\tau}\right)\left\langle \eta\right\rangle ^{-1}d\tilde{\tau}\\
		= & \int\left\langle \eta\right\rangle ^{\mu+\mathfrak{N}}\left(\tr_{\mu,M_{\mu}}^{\partial_{\tau}}\left(p'\left(\tau,\tilde{\tau},\eta\right)\right)\right)u\left(\left\langle \eta\right\rangle ^{-1}\tilde{\tau}\right)\left\langle \eta\right\rangle ^{-1}d\tilde{\tau}\\
		= & \int\left\langle \eta\right\rangle ^{\mu+\mathfrak{N}}a'\left(\tilde{\tau},\eta\right)u\left(\left\langle \eta\right\rangle ^{-1}\tilde{\tau}\right)\left\langle \eta\right\rangle ^{-1}d\tilde{\tau}\\
		= & \int\left\langle \eta\right\rangle ^{\mu+\mathfrak{N}}a\left(\tilde{x},\eta\right)u\left(\tilde{x}\right)d\tilde{x}.
	\end{align*}
	Applying this argument to $P_{1,i}$ in the coordinates induced by
	the chosen chart, we see that for every $u\left(x,y\right)\in\dot{C}^{\infty}\left(\overline{\mathbb{R}_{1}^{1}}\times\overline{\mathbb{R}^{n}}\right)$
	we have
	\begin{align*}
		\tr_{\mu,M_{\mu}}^{\partial_{x}}\left(P_{1,i}u\right) & =\frac{1}{\left(2\pi\right)^{n}}\int e^{iy\eta}\left\langle \eta\right\rangle ^{\mu+\mathfrak{N}}a\left(y;\tilde{x},\eta\right)\hat{u}\left(\tilde{x},\eta\right)d\tilde{x}d\eta\\
		& =\Op_{L}^{\tr}\left(\left\langle \eta\right\rangle ^{\mu+\mathfrak{N}}a\left(y;\tilde{x},\eta\right)\right)u,
	\end{align*}
	where
	\begin{align*}
		a\left(y;\tilde{x},\eta\right) & \in\mathcal{S}\left(\mathbb{R}^{n}\right)\hat{\otimes}\mathcal{A}_{\phg}^{\left(\mathcal{E}_{\rf},\mathcal{I}_{\ff_{0}}-1,\infty,\infty\right)}\left(\hat{T}_{0}^{2};\mathbb{C}^{M_{\mu}+1}\right)\\
		& =S_{0\tr,\mathcal{S}}^{-\infty,\left(\mathcal{E}_{\rf},\mathcal{I}_{\ff_{0}}\right)}\left(\mathbb{R}^{n};\mathbb{R}_{1}^{n+1};\mathbb{C}^{M_{\mu}+1}\right).
	\end{align*}
	It follows that $\tr_{\mu,M_{\mu}}^{\partial_{x}}\circ P_{1,i}\in\hat{\Psi}_{0\tr,\mathcal{S}}^{-\infty,\left(\mathcal{E}_{\rf},\mathcal{I}_{\ff_{0}}\right),-\left(\mu+\mathfrak{N}\right)}\left(\mathbb{R}_{1}^{n+1};\mathbb{R}^{n},\mathbb{C}^{M_{\mu}+1}\right)$.
	This operator is compactly supported in $U_{i}\times U_{i}'$, and
	extending invariantly off this open set to the whole of $\partial X\times X$,
	we obtain $\tr_{\mu,M_{\mu}}^{V}\circ P_{1,i}\in\hat{\Psi}_{0\tr}^{-\infty,\left(\mathcal{E}_{\rf},\left[\mathcal{I}_{\ff_{0}}\right]\right),-\left(\mu+\mathfrak{N}\right)}\left(X;\partial X,E_{L}\right)$.
	Therefore, $\tr_{\mu,M_{\mu}}^{V}\circ P_{1}\in\hat{\Psi}_{0\tr}^{-\infty,\left(\mathcal{E}_{\rf},\left[0\right]\right),-\left(\mu+\mathfrak{N}\right)}\left(X;\partial X,E_{L}\right)$
	as claimed.
\end{proof}
\begin{rem}
	\label{rem:untwisted-when-simple-and-no-resonances}The previous result
	implies that, when the critical indicial roots $\mu_{1},...,\mu_{I}$
	are all simple and there is no resonance, we can think of the components
	of $\boldsymbol{A}_{L}$ as\emph{ untwisted }symbolic $0$-trace maps.
	Specifically, in this case:
	\begin{enumerate}
		\item we have $\tilde{M}_{\mu_{i}}=0$ for every $i$;
		\item the bundle $E_{i}:=E_{\mu_{i},0}$ is the line bundle of fibrewise
		homogeneous functions of degree $\mu_{i}$ on $N^{+}\partial X$;
		\item the component $A_{L}^{i}$ taking values in $E_{i}$ is in $\hat{\Psi}_{0\tr}^{-\infty,\left(\mathcal{E}_{\rf},\mathcal{I}-\mu_{i}\right)}\left(X;\partial X,E_{\mu_{i},0}\right)$
		for some index set $\mathcal{I}$ with $\left[\mathcal{I}\right]=0$.
	\end{enumerate}
	In this case, we have $\boldsymbol{\mathfrak{s}}_{L}=\bigoplus_{i}\mu_{i}\id_{E_{i}}$.
	It is then possible to formulate boundary conditions $\boldsymbol{Q}\in\Psi_{\phg}^{-\left[0\right],\left(-\boldsymbol{\mathfrak{s}}_{L},\boldsymbol{\mathfrak{t}}\right)}\left(\partial X;\boldsymbol{E}_{L},\boldsymbol{W}\right)$
	where $\boldsymbol{W}=\bigoplus_{j}W_{j}$ and $\boldsymbol{\mathfrak{t}}:\boldsymbol{W}\to\boldsymbol{W}$
	is also of the form $\boldsymbol{\mathfrak{t}}=\bigoplus_{j}\nu_{j}\id_{W_{j}}$.
	In other words, in this case we can impose boundary conditions belonging
	to the Douglis--Nirenberg calculus. We remark however that many geometrically
	interesting examples have resonant indicial roots (cf.\ \S\ref{subsec:The-guiding-example-=0000A77})
	or non-simple indicial roots (an example is the bi-Laplacian on conformally
	compact manifolds).
\end{rem}

\subsection{\label{subsec:Proof-of-the-main-theorem}Proof of the main theorem}

In this section, we prove Theorem \ref{thm:main-theorem-parametrix}.
In fact, without any more effort, we can prove a slightly more general
result. Let \[\boldsymbol{A}\in\hat{\Psi}_{0\tr}^{-\infty,\left(\mathcal{E}_{\rf},\left[0\right]\right),-\boldsymbol{\mathfrak{s}}_{L}}\left(X;\partial X,\boldsymbol{E}_{L}\right)\]
be a symbolic $0$-trace operator with the following two properties:
\begin{enumerate}
	\item $\boldsymbol{A}=\boldsymbol{A}P_{1}$;
	\item $\hat{N}\left(\boldsymbol{A}\right)$ equals the Bessel trace map
	$\hat{\boldsymbol{a}}_{L}$ (cf.\ Definition \ref{def:construction-of-theBessel-trace-map}).
\end{enumerate}
We have just seen that $\boldsymbol{A}_{L}$ is an example of such
a map. Let $\boldsymbol{Q}\in\Psi_{\phg}^{-\left[0\right],\left(-\boldsymbol{\mathfrak{s}}_{L},\boldsymbol{\mathfrak{t}}\right)}\left(\partial X;\boldsymbol{E}_{L},\boldsymbol{W}\right)$
be an elliptic boundary condition. We will prove that the operator
$\mathcal{L}=\left(\begin{matrix}L & \boldsymbol{Q}\boldsymbol{A}\end{matrix}\right)^{T}$
has good left and right parametrices, as in the statement of Theorem
\ref{thm:main-theorem-parametrix}.

The construction of both left and right parametrices relies on the
construction of two fundamental objects:
\begin{enumerate}
	\item an operator in the twisted boundary calculus \[\boldsymbol{K}_{0}\in\Psi_{\phg}^{-\left[0\right],\left(\boldsymbol{\mathfrak{t}},-\boldsymbol{\mathfrak{s}}_{L}\right)}\left(\partial X;\boldsymbol{W},\boldsymbol{E}_{L}\right),\]
	whose principal symbol $\sigma_{\eta}\left(\boldsymbol{K}_{0}\right)$
	is such that $\sigma_{\eta}\left(\boldsymbol{K}_{0}\right)\sigma_{\eta}\left(\boldsymbol{Q}\right)$
	is a projector $\left(\pi^{*}\boldsymbol{E}_{L}\right)_{\eta}\to\mathcal{C}_{\eta}$;
	\item a twisted symbolic $0$-Poisson operator \[\boldsymbol{B}_{0}\in\hat{\Psi}_{0\po}^{-\infty,\left(\mathcal{E}_{\lf},\left[0\right]\right),-\boldsymbol{\mathfrak{s}}_{L}}\left(\partial X,\boldsymbol{E}_{L};X\right),\]
	whose Bessel map $\hat{N}_{\eta}\left(\boldsymbol{B}_{0}\right)$
	satisfies $\hat{N}_{\eta}\left(\boldsymbol{A}\right)\hat{N}_{\eta}\left(\boldsymbol{B}_{0}\right)=\sigma_{\eta}\left(\boldsymbol{K}_{0}\boldsymbol{Q}\right)$,
	and with the property that $L\boldsymbol{B}_{0}\equiv0$.
\end{enumerate}
The existence of an operator $\boldsymbol{K}_{0}$ as in point 1 is
guaranteed by the fact that $\boldsymbol{Q}$ is an elliptic boundary
condition. Indeed, for every $\eta\in T^{*}\partial X\backslash0$,
define $\sigma_{\eta}\left(\boldsymbol{K}_{0}\right):\left(\pi^{*}\boldsymbol{W}\right)_{\eta}\to\left(\pi^{*}\boldsymbol{E}_{L}\right)_{\eta}$
as the inverse of $\sigma_{\eta}\left(\boldsymbol{Q}\right)_{|\mathcal{C}_{\eta}}:\mathcal{C}_{\eta}\to\left(\pi^{*}\boldsymbol{W}\right)_{\eta}$.
In other words, given $w\in\left(\pi^{*}\boldsymbol{W}\right)_{\eta}$,
define $\sigma_{\eta}\left(\boldsymbol{K}_{0}\right)w=u_{1}$ where
$u_{1}$ is the unique element in $\mathcal{C}_{\eta}$ such that
$\sigma_{\eta}\left(\boldsymbol{Q}\right)u_{1}=w$. If $u\in\left(\pi^{*}\boldsymbol{E}_{L}\right)_{\eta}$,
we can decompose canonically $u=u_{1}+u_{2}$ with $u_{1}\in\mathcal{C}_{\eta}$
and $u_{2}\in\ker\sigma_{\eta}\left(\boldsymbol{Q}\right)$, and by
construction $\sigma_{\eta}\left(\boldsymbol{K}_{0}\right)\sigma_{\eta}\left(\boldsymbol{Q}\right)u=u_{1}$;
therefore, $\sigma_{\eta}\left(\boldsymbol{K}_{0}\right)\sigma_{\eta}\left(\boldsymbol{Q}\right)$
is a projector $\left(\pi^{*}\boldsymbol{E}_{L}\right)_{\eta}\to\mathcal{C}_{\eta}$.
Moreover, again by construction, $\sigma_{\eta}\left(\boldsymbol{Q}\right)\sigma_{\eta}\left(\boldsymbol{K}_{0}\right)$
is the identity on $\left(\pi^{*}\boldsymbol{W}\right)_{\eta}$. Finally,
since $\boldsymbol{Q}\in\Psi_{\phg}^{-\left[0\right],\left(-\boldsymbol{\mathfrak{s}}_{L},\boldsymbol{\mathfrak{t}}\right)}\left(\partial X;\boldsymbol{E}_{L},\boldsymbol{W}\right)$,
we have for every $\lambda\in\mathbb{R}^{+}$
\[
\sigma_{\lambda\eta}\left(\boldsymbol{Q}\right)=\lambda^{-\boldsymbol{\mathfrak{t}}}\sigma_{\eta}\left(\boldsymbol{Q}\right)\lambda^{-\boldsymbol{\mathfrak{s}}_{L}},
\]
and therefore
\[
\sigma_{\lambda\eta}\left(\boldsymbol{K}_{0}\right)=\lambda^{\boldsymbol{\mathfrak{s}}_{L}}\sigma_{\eta}\left(\boldsymbol{K}_{0}\right)\lambda^{\boldsymbol{\mathfrak{t}}}.
\]
In other words, we constructed $\sigma\left(\boldsymbol{K}_{0}\right)$
as a section of $\pi^{*}\hom\left(\boldsymbol{W},\boldsymbol{E}_{L}\right)$,
twisted homogeneous of degree $\left(\boldsymbol{\mathfrak{t}},-\boldsymbol{\mathfrak{s}}_{L}\right)$.
Therefore, we can realize it as the principal symbol of an operator
$\boldsymbol{K}_{0}\in\Psi_{\phg}^{-\left[0\right],\left(\boldsymbol{\mathfrak{t}},-\boldsymbol{\mathfrak{s}}_{L}\right)}\left(\partial X;\boldsymbol{W},\boldsymbol{E}_{L}\right)$.

Concerning the construction of $\boldsymbol{B}_{0}$ as in point 2,
in \S\ref{subsec:Solving-the-model} we constructed the \emph{Bessel
	Poisson map} $\hat{\boldsymbol{b}}_{\eta}$ as the unique map $\hat{\boldsymbol{b}}_{\eta}:\left(\pi^{*}\boldsymbol{E}_{L}\right)_{\eta}\to x^{\delta}L_{b}^{2}\left(N_{\pi\left(\eta\right)}^{+}\partial X\right)$
sending a vector $\varphi\in\left(\pi^{*}\boldsymbol{E}_{L}\right)_{\eta}$
to the unique $x^{\delta}L_{b}^{2}$ solution of 
\[
\begin{cases}
	\hat{N}_{\eta}\left(L\right)u & =0\\
	\hat{N}_{\eta}\left(\boldsymbol{A}\right)u & =\sigma_{\eta}\left(\boldsymbol{K}_{0}\boldsymbol{Q}\right)\varphi.
\end{cases}
\]
We proved in \S\ref{subsec:Solving-the-model} that $\eta\mapsto\hat{\boldsymbol{b}}_{\eta}$
is a section of $\pi^{*}\Psi_{\po}^{-\infty,\mathcal{E}_{\lf}}\left(N^{+}\partial X;\boldsymbol{E}_{L}^{*}\right)$,
twisted homogeneous of degree $-\boldsymbol{\mathfrak{s}}_{L}$. Therefore,
we can realize $\hat{\boldsymbol{b}}_{\eta}$ as the Bessel family
of a twisted symbolic $0$-Poisson operator $\boldsymbol{B}_{0}'\in\hat{\Psi}_{0\po}^{-\infty,\left(\mathcal{E}_{\lf},\left[0\right]\right),-\boldsymbol{\mathfrak{s}}_{L}}\left(\partial X,\boldsymbol{E}_{L};X\right)$.
By construction, we have \[\hat{N}_{\eta}\left(\boldsymbol{A}\right)\hat{N}_{\eta}\left(\boldsymbol{B}_{0}'\right)=\sigma_{\eta}\left(\boldsymbol{K}_{0}\boldsymbol{Q}\right).\]
We improve $\boldsymbol{B}_{0}'$ by precomposing it with the projector
$P_{1}$: note that, since \[P_{1}\in\hat{\Psi}_{0}^{-\infty,\left(\mathcal{E}_{\lf},\mathcal{E}_{\rf},\mathcal{H}_{\ff_{0}}\right)}\left(X\right)\]
with $\left[\mathcal{H}_{\ff_{0}}\right]=0$, and \[\boldsymbol{B}_{0}'\in\hat{\Psi}_{0\po}^{-\infty,\left(\mathcal{E}_{\lf},\left[0\right]\right),-\boldsymbol{\mathfrak{s}}_{L}}\left(\partial X,\boldsymbol{E}_{L};X\right),\]
the condition $\Re\left(\mathcal{E}_{\rf}+\mathcal{E}_{\lf}\right)>-1$
implies from Theorem \ref{thm:global-twisted-compositions-involving-boundary}
that the composition $\boldsymbol{B}_{0}:=P_{1}\boldsymbol{B}_{0}'$
is well-defined and in $\hat{\Psi}_{0\po}^{-\infty,\left(\mathcal{E}_{\lf},\left[0\right]\right),-\boldsymbol{\mathfrak{s}}_{L}}\left(\partial X,\boldsymbol{E}_{L};X\right)$.
Moreover, we have \[\hat{N}_{\eta}\left(\boldsymbol{B}_{0}\right)\equiv\hat{N}_{\eta}\left(P_{1}\right)\hat{\boldsymbol{b}}_{\eta},\]
and since $\hat{\boldsymbol{b}}_{\eta}$ maps into the $x^{\delta}L_{b}^{2}$
kernel of $\hat{N}_{\eta}\left(L\right)$ by construction, we have
\[\hat{N}_{\eta}\left(P_{1}\right)\hat{N}_{\eta}\left(\boldsymbol{B}_{0}\right)\equiv\hat{N}_{\eta}\left(\boldsymbol{B}_{0}\right).\]
\begin{lem}
	$ $
	\begin{enumerate}
		\item The composition $\boldsymbol{A}\boldsymbol{B}_{0}$ is well-defined
		and in $\Psi_{\phg}^{-\left[0\right],\left(-\boldsymbol{\mathfrak{s}}_{L},-\boldsymbol{\mathfrak{s}}_{L}\right)}\left(\partial X;\boldsymbol{E}_{L},\boldsymbol{E}_{L}\right)$;
		moreover, for every $\eta\in T^{*}\partial X\backslash0$, we have
		$\sigma_{\eta}\left(\boldsymbol{A}\boldsymbol{B}_{0}\right)=\sigma_{\eta}\left(\boldsymbol{K}_{0}\boldsymbol{Q}\right)$.
		\item The composition $\boldsymbol{B}_{0}\boldsymbol{A}$ is well-defined
		and in $\hat{\Psi}_{0}^{-\infty,\left(\mathcal{E}_{\lf},\mathcal{E}_{\rf},\mathcal{I}_{\ff_{0}}\right)}\left(X\right)$
		for some index set $\mathcal{I}_{\ff_{0}}$ with $\left[\mathcal{I}_{\ff_{0}}\right]=0$;
		moreover, for every $\eta\in T^{*}\partial X\backslash0$, we have
		$\hat{N}_{\eta}\left(\boldsymbol{B}_{0}\boldsymbol{A}\right)=\hat{N}_{\eta}\left(P_{1}\right)$.
	\end{enumerate}
\end{lem}

\begin{proof}
	We have $\boldsymbol{A}\in\hat{\Psi}_{0\tr}^{-\infty,\left(\mathcal{E}_{\rf},\left[0\right]\right),-\boldsymbol{\mathfrak{s}}_{L}}\left(X;\partial X,\boldsymbol{E}_{L}\right)$
	and $\boldsymbol{B}_{0}\in\hat{\Psi}_{0\po}^{-\infty,\left(\mathcal{E}_{\lf},\left[0\right]\right),-\boldsymbol{\mathfrak{s}}_{L}}\left(\partial X,\boldsymbol{E}_{L};X\right)$;
	since $\Re\left(\mathcal{E}_{\lf}\right)>\delta$ and $\Re\left(\mathcal{E}_{\rf}\right)>-\delta-1$,
	$\Re\left(\mathcal{E}_{\rf}+\mathcal{E}_{\lf}\right)>-1$ and therefore
	we can apply Theorem \ref{thm:global-twisted-compositions-involving-boundary}
	and conclude that $\boldsymbol{A}\boldsymbol{B}_{0}\in\Psi_{\phg}^{-\left[0\right],\left(-\boldsymbol{\mathfrak{s}}_{L},-\boldsymbol{\mathfrak{s}}_{L}\right)}\left(\partial X;\boldsymbol{E}_{L},\boldsymbol{E}_{L}\right)$
	and $\sigma_{\eta}\left(\boldsymbol{A}\boldsymbol{B}_{0}\right)=\hat{N}_{\eta}\left(\boldsymbol{A}\right)\hat{N}_{\eta}\left(\boldsymbol{B}_{0}\right)$.
	The fact that $\hat{N}_{\eta}\left(\boldsymbol{A}\right)\hat{N}_{\eta}\left(\boldsymbol{B}_{0}\right)=\sigma_{\eta}\left(\boldsymbol{K}_{0}\boldsymbol{Q}\right)$
	is essentially the definition of the Bessel Poisson map $\hat{N}_{\eta}\left(\boldsymbol{B}_{0}\right)$.
	Now, again by Theorem \ref{thm:global-twisted-compositions-involving-boundary},
	we have indeed $\boldsymbol{B}_{0}\boldsymbol{A}\in\hat{\Psi}_{0}^{-\infty,\left(\mathcal{E}_{\lf},\mathcal{E}_{\rf},\mathcal{I}_{\ff_{0}}\right)}\left(X\right)$
	with $\left[\mathcal{I}_{\ff_{0}}\right]=0$, and moreover $\hat{N}_{\eta}\left(\boldsymbol{B}_{0}\boldsymbol{A}\right)=\hat{N}_{\eta}\left(\boldsymbol{B}_{0}\right)\hat{N}_{\eta}\left(\boldsymbol{A}\right)$.
	To prove that $\hat{N}_{\eta}\left(\boldsymbol{B}_{0}\boldsymbol{A}\right)=\hat{N}_{\eta}\left(P_{1}\right)$,
	first of all observe that by definition of the Bessel Poisson map
	$\hat{N}_{\eta}\left(\boldsymbol{B}_{0}\right)$, $\hat{N}_{\eta}\left(\boldsymbol{B}_{0}\boldsymbol{A}\right)$
	maps into the $x^{\delta}L_{b}^{2}$ kernel of $\hat{N}_{\eta}\left(L\right)$;
	since the Bessel trace map $\hat{N}_{\eta}\left(\boldsymbol{A}\right)$
	is injective on this kernel, it suffices to prove that $\hat{N}_{\eta}\left(\boldsymbol{A}\right)\left(\hat{N}_{\eta}\left(\boldsymbol{B}_{0}\boldsymbol{A}\right)-\hat{N}_{\eta}\left(P_{1}\right)\right)=0$.
	But by definition of $\hat{N}_{\eta}\left(\boldsymbol{A}\right)$,
	we have $\hat{N}_{\eta}\left(\boldsymbol{A}\right)\hat{N}_{\eta}\left(P_{1}\right)=\hat{N}_{\eta}\left(\boldsymbol{A}\right)$,
	and moreover from the previous point we have 
	\begin{align*}
		\hat{N}_{\eta}\left(\boldsymbol{A}\right)\hat{N}_{\eta}\left(\boldsymbol{B}_{0}\boldsymbol{A}\right) & =\sigma_{\eta}\left(\boldsymbol{A}\boldsymbol{B}_{0}\right)\hat{N}_{\eta}\left(\boldsymbol{A}\right)\\
		& =\sigma_{\eta}\left(\boldsymbol{K}_{0}\boldsymbol{Q}\right)\hat{N}_{\eta}\left(\boldsymbol{A}\right).
	\end{align*}
	By definition, $\hat{N}_{\eta}\left(\boldsymbol{A}\right)$ maps onto
	the Calderón space $\mathcal{C}_{\eta}$, and by construction $\sigma_{\eta}\left(\boldsymbol{K}_{0}\boldsymbol{Q}\right)$
	is the identity on $\mathcal{C}_{\eta}$, so we have $\sigma_{\eta}\left(\boldsymbol{K}_{0}\boldsymbol{Q}\right)\hat{N}_{\eta}\left(\boldsymbol{A}\right)=\hat{N}_{\eta}\left(\boldsymbol{A}\right)$.
	This concludes the proof.
\end{proof}
Define now $\boldsymbol{C}_{0}:=\boldsymbol{B}_{0}\boldsymbol{K}_{0}$.
This operator will allow us to construct good left and right parametrices
for $\mathcal{L}$.
\begin{lem}
	$ $
	\begin{enumerate}
		\item The operator $\boldsymbol{C}_{0}$ is well-defined and in $\hat{\Psi}_{0\po}^{-\infty,\left(\mathcal{E}_{\lf},\left[0\right]\right),\boldsymbol{\mathfrak{t}}}\left(\partial X,\boldsymbol{W};X\right)$.
		\item The composition $\boldsymbol{Q}\boldsymbol{A}\boldsymbol{C}_{0}$
		is well-defined and in $\Psi_{\phg}^{-\left[0\right],\left(\boldsymbol{\mathfrak{t}},\boldsymbol{\mathfrak{t}}\right)}\left(\partial X;\boldsymbol{W},\boldsymbol{W}\right)$.
		\item The principal symbol $\sigma\left(\boldsymbol{Q}\boldsymbol{A}\boldsymbol{C}_{0}\right)$
		is the identity on $\pi^{*}\boldsymbol{W}$.
	\end{enumerate}
\end{lem}

\begin{proof}
	(1) We have \[\boldsymbol{B}_{0}\in\hat{\Psi}_{0\po}^{-\infty,\left(\mathcal{E}_{\lf},\left[0\right]\right),-\boldsymbol{\mathfrak{s}}_{L}}\left(\partial X,\boldsymbol{E}_{L};X\right),\]
	and \[\boldsymbol{K}_{0}\in\Psi_{\phg}^{-\left[0\right],\left(\boldsymbol{\mathfrak{t}},-\boldsymbol{\mathfrak{s}}_{L}\right)}\left(\partial X;\boldsymbol{W},\boldsymbol{E}_{L}\right).\]
	By Theorem \ref{thm:global-twisted-compositions-involving-boundary},
	we have \[\boldsymbol{C}_{0}=\boldsymbol{B}_{0}\boldsymbol{K}_{0}\in\hat{\Psi}_{0\po}^{-\infty,\left(\mathcal{E}_{\lf},\left[0\right]\right),\boldsymbol{\mathfrak{t}}}\left(\partial X,\boldsymbol{W};X\right).\]
	(2) We have \[\boldsymbol{Q}\in\Psi_{\phg}^{-\left[0\right],\left(-\boldsymbol{\mathfrak{s}}_{L},\boldsymbol{\mathfrak{t}}\right)}\left(\partial X;\boldsymbol{E}_{L},\boldsymbol{W}\right)\]
	and \[\boldsymbol{A}\in\hat{\Psi}_{0\tr}^{-\infty,\left(\mathcal{E}_{\rf},\left[0\right]\right),-\boldsymbol{\mathfrak{s}}_{L}}\left(X;\partial X,\boldsymbol{E}_{L}\right).\]
	By Theorem \ref{thm:global-twisted-compositions-involving-boundary},
	we have \[\boldsymbol{Q}\boldsymbol{A}\in\hat{\Psi}_{0\tr}^{-\infty,\left(\mathcal{E}_{\rf},\left[0\right]\right),\boldsymbol{\mathfrak{t}}}\left(X;\partial X,\boldsymbol{W}\right);\]
	moreover, using $\Re\left(\mathcal{E}_{\rf}+\mathcal{E}_{\lf}\right)>-1$
	and Theorem \ref{thm:global-twisted-compositions-involving-boundary},
	we have \[\boldsymbol{Q}\boldsymbol{A}\boldsymbol{C}_{0}\in\Psi_{\phg}^{-\left[0\right],\left(\boldsymbol{\mathfrak{t}},\boldsymbol{\mathfrak{t}}\right)}\left(\partial X;\boldsymbol{W},\boldsymbol{W}\right).\]
	(3) We have
	\begin{align*}
		\sigma_{\eta}\left(\boldsymbol{Q}\boldsymbol{A}\boldsymbol{C}_{0}\right) & =\sigma_{\eta}\left(\boldsymbol{Q}\right)\hat{N}_{\eta}\left(\boldsymbol{A}\right)\hat{N}_{\eta}\left(\boldsymbol{B}_{0}\boldsymbol{K}_{0}\right)\\
		& =\sigma_{\eta}\left(\boldsymbol{Q}\right)\sigma_{\eta}\left(\boldsymbol{A}\boldsymbol{B}_{0}\right)\sigma_{\eta}\left(\boldsymbol{K}_{0}\right)\\
		& =\sigma_{\eta}\left(\boldsymbol{Q}\right)\sigma_{\eta}\left(\boldsymbol{K}_{0}\boldsymbol{Q}\right)\sigma_{\eta}\left(\boldsymbol{K}_{0}\right)\\
		& =\sigma_{\eta}\left(\boldsymbol{Q}\boldsymbol{K}_{0}\right)\sigma_{\eta}\left(\boldsymbol{Q}\boldsymbol{K}_{0}\right)\\
		& =1_{\left(\pi^{*}\boldsymbol{W}\right)_{\eta}}.
	\end{align*}
\end{proof}
Define $\mathcal{G}_{0}=\left(\begin{matrix}G & \boldsymbol{C}_{0}\end{matrix}\right)$.
We will now compute the remainders $I-\mathcal{L}\mathcal{G}_{0}$
and $I-\mathcal{G}_{0}\mathcal{L}$. We have

\begin{align*}
	\mathcal{L}\mathcal{G}_{0} & =\left(\begin{matrix}L\\
		\boldsymbol{Q}\boldsymbol{A}
	\end{matrix}\right)\left(\begin{matrix}G & \boldsymbol{C}_{0}\end{matrix}\right)\\
	& =\left(\begin{matrix}LG & L\boldsymbol{C}_{0}\\
		\boldsymbol{Q}\boldsymbol{A}G & \boldsymbol{Q}\boldsymbol{A}\boldsymbol{C}_{0}
	\end{matrix}\right)\\
	& =\left(\begin{matrix}I-P_{2} & 0\\
		0 & \boldsymbol{Q}\boldsymbol{A}\boldsymbol{C}_{0}
	\end{matrix}\right)\\
	& =I-\left(\begin{matrix}P_{2}\\
		& \boldsymbol{R}_{0}
	\end{matrix}\right),
\end{align*}
where $P_{2}\in\Psi^{-\infty,\left(\mathcal{F}_{\lf},\mathcal{F}_{\rf}\right)}\left(X\right)$
is the very residual orthogonal projector onto the $x^{\delta}L_{b}^{2}$
orthogonal complement of the range of $L$, and $\boldsymbol{R}_{0}=I-\boldsymbol{Q}\boldsymbol{A}\boldsymbol{C}_{0}$.
The operator $L\boldsymbol{C}_{0}$ vanishes because, by construction,
we have $\boldsymbol{C}_{0}=P_{1}\boldsymbol{C}_{0}$ and $LP_{1}=0$;
similarly, the operator $\boldsymbol{Q}\boldsymbol{A}G$ vanishes
because, by definition of $\boldsymbol{A}$, we have $\boldsymbol{A}P_{1}=\boldsymbol{A}$
and by definition of generalized inverse, $G$ maps onto the $x^{\delta}L_{b}^{2}$
orthogonal complement of the kernel of $L$ so $P_{1}G=0$. Finally,
a priori we have $\boldsymbol{R}_{0}\in\Psi_{\phg}^{-\left[0\right],\left(\boldsymbol{\mathfrak{t}},\boldsymbol{\mathfrak{t}}\right)}\left(\partial X;\boldsymbol{W},\boldsymbol{W}\right)$,
but in fact $\sigma\left(\boldsymbol{Q}\boldsymbol{A}\boldsymbol{C}_{0}\right)=\sigma\left(I\right)$,
and therefore $\sigma\left(\boldsymbol{R}_{0}\right)=0$, so we have
$\boldsymbol{R}_{0}\in\Psi_{\phg}^{-\left[\mathcal{I}\right],\left(\boldsymbol{\mathfrak{t}},\boldsymbol{\mathfrak{t}}\right)}\left(\partial X;\boldsymbol{W},\boldsymbol{W}\right)$
for some index set $\mathcal{I}$ with $\Re\left(\mathcal{I}\right)\geq1$.

This allows us to promote $\mathcal{G}_{0}$ to a good right parametrix
using a familiar argument. Let $\boldsymbol{S}$ be an asymptotic
sum for
\[
\boldsymbol{S}\sim\sum_{k=0}^{\infty}\boldsymbol{R}_{0}^{k}.
\]
Since $\Re\left(\mathcal{I}\right)\geq1$, for $k>0$ we have $\boldsymbol{R}_{0}^{k}\in\Psi_{\phg}^{-\left[\mathcal{I}^{\left(k\right)}\right],\left(\boldsymbol{\mathfrak{t}},\boldsymbol{\mathfrak{t}}\right)}\left(\partial X;\boldsymbol{W},\boldsymbol{W}\right)$
with $\Re\left(\mathcal{I}^{\left(k\right)}\right)\geq k$. Therefore,
we have $\boldsymbol{S}\in\Psi_{\phg}^{-\left[0\right],\left(\boldsymbol{\mathfrak{t}},\boldsymbol{\mathfrak{t}}\right)}\left(\partial X;\boldsymbol{W},\boldsymbol{W}\right)$
and
\[
\left(I-\boldsymbol{R}_{0}\right)\boldsymbol{S}\equiv I\mod\Psi^{-\infty}\left(\partial X;\boldsymbol{W},\boldsymbol{W}\right).
\]
Define $G_{2}=G$ and $\boldsymbol{C}_{2}=\boldsymbol{C}_{0}\boldsymbol{S}$.
By Theorem \ref{thm:global-twisted-compositions-involving-boundary},
we have again $\boldsymbol{C}_{2}\in\hat{\Psi}_{0\po}^{-\infty,\left(\mathcal{E}_{\lf},\left[0\right]\right),\boldsymbol{\mathfrak{t}}}\left(\partial X,\boldsymbol{W};X\right)$,
and calling $\mathcal{G}_{2}=\left(\begin{matrix}G & \boldsymbol{C}_{2}\end{matrix}\right)$,
we obtain
\[
\mathcal{L}\mathcal{G}_{2}=I-\left(\begin{matrix}P_{2}\\
	& \boldsymbol{R}
\end{matrix}\right)
\]
with $\boldsymbol{R}\in\Psi^{-\infty}\left(\partial X;\boldsymbol{W},\boldsymbol{W}\right)$,
as claimed.

Let's now take care of the left remainder.
\begin{lem}
	$ $
	\begin{enumerate}
		\item The operator $\boldsymbol{C}_{0}\boldsymbol{Q}\boldsymbol{A}$ is
		well-defined and in $\hat{\Psi}_{0}^{-\infty,\left(\mathcal{E}_{\lf},\mathcal{E}_{\rf},\mathcal{I}\right)}\left(X\right)$,
		with $\left[\mathcal{I}\right]=0$.
		\item $\hat{N}\left(\boldsymbol{C}_{0}\boldsymbol{Q}\boldsymbol{A}\right)=\hat{N}\left(P_{1}\right)$.
	\end{enumerate}
\end{lem}

\begin{proof}
	(1) We have $\boldsymbol{C}_{0}\in\hat{\Psi}_{0\po}^{-\infty,\left(\mathcal{E}_{\lf},\left[0\right]\right),\boldsymbol{\mathfrak{t}}}\left(\partial X,\boldsymbol{W};X\right)$
	and $\boldsymbol{Q}\boldsymbol{A}\in\hat{\Psi}_{0\tr}^{-\infty,\left(\mathcal{E}_{\rf},\left[0\right]\right),\boldsymbol{\mathfrak{t}}}\left(X;\partial X,\boldsymbol{W}\right)$,
	so by $\Re\left(\mathcal{E}_{\lf}+\mathcal{E}_{\rf}\right)>-1$ and
	Theorem \ref{thm:global-twisted-compositions-involving-boundary}
	we obtain $\boldsymbol{C}_{0}\boldsymbol{Q}\boldsymbol{A}\in\hat{\Psi}_{0}^{-\infty,\left(\mathcal{E}_{\lf},\mathcal{E}_{\rf},\mathcal{I}\right)}\left(X\right)$.
	(2) We have
	\begin{align*}
		\hat{N}_{\eta}\left(\boldsymbol{C}_{0}\boldsymbol{Q}\boldsymbol{A}\right) & =\hat{N}_{\eta}\left(\boldsymbol{B}_{0}\boldsymbol{K}_{0}\boldsymbol{Q}\boldsymbol{A}\right)\\
		& =\hat{N}_{\eta}\left(\boldsymbol{B}_{0}\right)\sigma_{\eta}\left(\boldsymbol{K}_{0}\boldsymbol{Q}\right)\hat{N}_{\eta}\left(\boldsymbol{A}\right).
	\end{align*}
	Now, $\sigma_{\eta}\left(\boldsymbol{K}_{0}\boldsymbol{Q}\right)$
	is the identity on $\mathcal{C}_{\eta}$, and $\hat{N}_{\eta}\left(\boldsymbol{A}\right)$
	takes values in $\mathcal{C}_{\eta}$, so
	\begin{align*}
		\hat{N}_{\eta}\left(\boldsymbol{C}_{0}\boldsymbol{Q}\boldsymbol{A}\right) & =\hat{N}_{\eta}\left(\boldsymbol{B}_{0}\boldsymbol{A}\right)=\hat{N}_{\eta}\left(P_{1}\right).
	\end{align*}
\end{proof}
Now, let's compute the left remainder $I-\mathcal{G}_{0}\mathcal{L}$.
We have
\begin{align*}
	\mathcal{G}_{0}\mathcal{L} & =\left(\begin{matrix}G & \boldsymbol{C}_{0}\end{matrix}\right)\left(\begin{matrix}L\\
		\boldsymbol{Q}\boldsymbol{A}
	\end{matrix}\right)\\
	& =GL+\boldsymbol{C}_{0}\boldsymbol{Q}\boldsymbol{A}\\
	& =I-\left(P_{1}-\boldsymbol{C}_{0}\boldsymbol{Q}\boldsymbol{A}\right)\\
	& =I-T_{0},
\end{align*}
where by the previous lemma we have $T_{0}\in\hat{\Psi}_{0}^{-\infty,\left(\mathcal{E}_{\lf},\mathcal{E}_{\rf},\mathcal{I}\right)}\left(X\right)$
for some index set $\mathcal{I}$ with $\Re\left(\mathcal{I}\right)\geq1$.
We can now improve the parametrix using the same argument as above.
Let $U$ be an asymptotic sum for
\[
U\sim\sum_{k=1}^{\infty}T_{0}^{k}.
\]
Since $\Re\left(\mathcal{E}_{\lf}+\mathcal{E}_{\rf}\right)>-1$, using
Theorem \ref{thm:compositions-involving-interior-1} we have
\[
T_{0}^{k}\in\hat{\Psi}_{0}^{-\infty,\left(\mathcal{E}_{\lf},\mathcal{E}_{\rf},k\mathcal{I}\right)}\left(X\right)
\]
for every $k>0$. Therefore, $U\in\hat{\Psi}_{0}^{-\infty,\left(\mathcal{E}_{\lf},\mathcal{E}_{\rf},\mathcal{J}\right)}\left(X\right)$
for some index set with $\Re\left(\mathcal{J}\right)\geq1$, and

\[
\left(I+U\right)\left(I-T_{0}\right)\equiv I\mod\Psi^{-\infty,\left(\mathcal{E}_{\lf},\mathcal{E}_{\rf}\right)}\left(X\right).
\]
Define $G_{1}=\left(I+U\right)G$ and $\boldsymbol{C}_{1}=\left(I+U\right)\boldsymbol{C}_{0}$.
Again by Theorem \ref{thm:global-twisted-compositions-involving-boundary},
we have \[\boldsymbol{C}_{1}\in\hat{\Psi}_{0\po}^{-\infty,\left(\mathcal{E}_{\lf},\left[0\right]\right),\boldsymbol{\mathfrak{t}}}\left(\partial X,\boldsymbol{W};X\right),\]
with $\hat{N}\left(\boldsymbol{C}_{1}\right)=\hat{N}\left(\boldsymbol{C}_{0}\right)$.
Now, we know from \cite{MazzeoEdgeI} that $G\in\Psi_{0}^{-m,\mathcal{H}}\left(X\right)$
where \[\mathcal{H}=\left(\mathcal{H}_{\lf},\mathcal{H}_{\rf},\mathcal{H}_{\ff_{0}}\right)\]
satisfies the properties listed in Theorem \ref{thm:main-theorem-parametrix}.
Moreover, by Corollary \ref{cor:global-relation-symbolic0b-physical0b},
$U$ is in the $0$-calculus, and more precisely $U\in\Psi_{0}^{-\infty,\left(\mathcal{E}_{\lf},\mathcal{E}_{\rf},\tilde{\mathcal{J}}\right)}\left(X\right)$
where $\tilde{\mathcal{J}}$ satisfies again $\Re\left(\tilde{\mathcal{J}}\right)\geq1$.
Therefore, by the composition theorem for the $0$-calculus proved
in \cite{MazzeoEdgeI}, we have again $UG\in\Psi_{0}^{-m,\tilde{\mathcal{H}}}\left(X\right)$
where $\left(\tilde{\mathcal{H}}_{\lf},\tilde{\mathcal{H}}_{\rf},\tilde{\mathcal{H}}_{\ff_{0}}\right)$
satisfies again $\Re\left(\tilde{\mathcal{H}}_{\lf}\right)>\delta$,
$\Re\left(\tilde{\mathcal{H}}_{\rf}\right)>-\delta-1$, and $\left[\tilde{\mathcal{H}}_{\ff_{0}}\right]=0$.
This concludes the proof of Theorem \ref{thm:main-theorem-parametrix}.

\subsection{\label{subsec:corollaries}Corollaries}

The first consequence of the main theorem is the Fredholm result mentioned
in the introduction:
\begin{thm}
	\label{thm:main-theorem-fredholm}Let $\boldsymbol{Q}$ be an elliptic
	boundary condition for $L$ relative to $\delta,\overline{\delta}$,
	as in Theorem \ref{thm:main-theorem-parametrix}. Then the map
	\[
	\mathcal{L}=\left(\begin{matrix}L\\
		\boldsymbol{Q}\boldsymbol{A}_{L}
	\end{matrix}\right):x^{\delta}H_{0}^{k+m}\left(X\right)\to x^{\delta}H_{0}^{k}\left(X\right)\oplus H^{\delta+\boldsymbol{\mathfrak{t}}}\left(\partial X;\boldsymbol{W}\right)
	\]
	is Fredholm.
\end{thm}

\begin{proof}
	First of all, let's check that the map $\mathcal{L}$ is bounded between
	the spaces above. Since $L$ is a $0$-differential operator of order
	$m$, $L$ is bounded $x^{\delta}H_{0}^{k+m}\left(X\right)\to x^{\delta}H_{0}^{k}\left(X\right)$.
	Now, we know that $\boldsymbol{A}_{L}$ is a twisted symbolic $0$-trace
	operator in $\hat{\Psi}_{0\tr}^{-\infty,\left(\mathcal{E}_{\rf},\left[0\right]\right),-\boldsymbol{\mathfrak{s}}_{L}}\left(X;\partial X,\boldsymbol{E}_{L}\right)$.
	Since $\Re\left(\mathcal{E}_{\rf}\right)>-\delta-1$, by Theorem \ref{thm:main-theorem-parametrix}
	$\boldsymbol{A}_{L}$ induces a bounded map
	\[
	\boldsymbol{A}_{L}:x^{\delta}H_{0}^{k+m}\left(X\right)\to H^{\delta-\boldsymbol{\mathfrak{s}}_{L}}\left(\partial X;\boldsymbol{E}_{L}\right).
	\]
	Now, since $\boldsymbol{Q}\in\Psi_{\phg}^{-\left[0\right],\left(-\boldsymbol{\mathfrak{s}}_{L},\boldsymbol{\mathfrak{t}}\right)}\left(\partial X;\boldsymbol{E}_{L},\boldsymbol{W}\right)$,
	by Proposition \ref{prop:mapping-twisted-boundary-sobolev} $\boldsymbol{Q}$
	induces a bounded map
	\[
	\boldsymbol{Q}:H^{\delta-\boldsymbol{\mathfrak{s}}_{L}}\left(\partial X;\boldsymbol{E}_{L}\right)\to H^{\delta+\boldsymbol{\mathfrak{t}}}\left(\partial X;\boldsymbol{W}\right).
	\]
	Therefore, $\mathcal{L}$ is well-defined. Now, to prove that $\mathcal{L}$
	is Fredholm, let $\mathcal{G}_{1},\mathcal{G}_{2}$ be the left and
	right parametrices for $\mathcal{L}$ constructed in Theorem \ref{thm:main-theorem-parametrix}.
	We have
	\[
	\mathcal{G}_{1}\mathcal{L}=I-R
	\]
	where $R$ is a very residual operator in $\Psi^{-\infty,\left(\mathcal{E}_{\lf},\mathcal{E}_{\rf}\right)}\left(X\right)$,
	and
	\[
	\mathcal{L}\mathcal{G}_{2}=I-\left(\begin{matrix}P_{2}\\
		& \boldsymbol{S}
	\end{matrix}\right),
	\]
	where $P_{2}\in\Psi^{-\infty,\left(\mathcal{F}_{\lf},\mathcal{F}_{\rf}\right)}\left(X\right)$
	is a very residual operator and $\boldsymbol{S}\in\Psi^{-\infty}\left(\partial X;\boldsymbol{W}\right)$.
	Since $R$ is a very residual operator and $\Re\left(\mathcal{E}_{\lf}\right)>\delta$,
	$\Re\left(\mathcal{E}_{\rf}\right)>-\delta-1$, $R$ is compact as
	an operator on $x^{\delta}H_{0}^{k+m}\left(X\right)$. Similarly,
	$P_{2}$ is compact as an operator on $x^{\delta}H_{0}^{k}\left(X\right)$.
	Finally, $\boldsymbol{S}$ is smoothing, and therefore it is bounded
	as a map $H^{\delta+\boldsymbol{\mathfrak{t}}}\left(\partial X;\boldsymbol{W}\right)\to H^{\delta+\boldsymbol{\mathfrak{t}}+\epsilon}\left(\partial X;\boldsymbol{W}\right)$
	for every $\epsilon>0$, hence it is compact as a map on $H^{\delta+\boldsymbol{\mathfrak{t}}}\left(\partial X;\boldsymbol{W}\right)$.
\end{proof}
We also obtain a regularity result for solutions of $0$-elliptic
boundary value problems.
\begin{cor}
	\label{cor:main-corollary-regularity}Let $\boldsymbol{Q}$ be an
	elliptic boundary condition for $L$ relative to $\delta,\overline{\delta}$,
	as in Theorem \ref{thm:main-theorem-parametrix}. Suppose that $u\in x^{\delta}H_{0}^{m}\left(X\right)$
	is a solution of the boundary value problem
	\[
	\begin{cases}
		Lu & =v\\
		\boldsymbol{Q}\boldsymbol{A}_{L}u & =\varphi
	\end{cases},
	\]
	where $v\in\mathcal{A}_{\phg}^{\mathcal{F}}\left(X\right)$ for some
	index set $\mathcal{F}$ with $\Re\left(\mathcal{F}\right)>\delta$,
	and $\varphi\in C^{\infty}\left(\partial X;\boldsymbol{W}\right)$.
	Then $u\in\mathcal{A}_{\phg}^{\tilde{\mathcal{F}}}\left(X\right)$
	for some index set $\tilde{\mathcal{F}}$ with $\Re\left(\tilde{\mathcal{F}}\right)>\delta$. 
\end{cor}

\begin{proof}
	Let $\mathcal{G}_{1}=\left(\begin{matrix}G_{1} & \boldsymbol{C}_{1}\end{matrix}\right)$
	be the left parametrix for $\mathcal{L}$ constructed in Theorem \ref{thm:main-theorem-parametrix}.
	We have
	\[
	G_{1}\in\Psi_{0}^{-m,\mathcal{H}}\left(X\right)
	\]
	with $\Re\left(\mathcal{H}_{\lf}\right)>\delta$, $\Re\left(\mathcal{H}_{\rf}\right)>-\delta-1$
	and $\left[\mathcal{H}_{\ff_{0}}\right]=0$, and 
	\[
	\boldsymbol{C}_{1}\in\hat{\Psi}_{0\po}^{-\infty,\left(\mathcal{E}_{\lf},\left[0\right]\right),\boldsymbol{\mathfrak{t}}}\left(\partial X,\boldsymbol{W};X\right).
	\]
	Writing $u=\mathcal{G}_{1}\mathcal{L}u+Ru$, we get
	\[
	u=Ru+G_{1}v+\boldsymbol{C}_{1}\varphi.
	\]
	Since $R$ is very residual in $\Psi^{-\infty,\left(\mathcal{E}_{\lf},\mathcal{E}_{\rf}\right)}\left(X\right)$,
	$R$ maps $x^{\delta}L_{b}^{2}\left(X\right)$ to $\mathcal{A}_{\phg}^{\mathcal{E}_{\lf}}\left(X\right)$.
	Moreover, from the mapping properties for $\boldsymbol{C}_{1}$ proved
	in \S\ref{subsec:Basic-mapping-properties}, we know that
	\[
	\boldsymbol{C}_{1}:C^{\infty}\left(\partial X;\boldsymbol{W}\right)\to\mathcal{A}_{\phg}^{\mathcal{E}_{\lf}}\left(X\right)
	\]
	so $\boldsymbol{C}_{1}\varphi\in\mathcal{A}_{\phg}^{\mathcal{E}_{\lf}}\left(X\right)$
	as well. Now, applying the mapping properties of elements of the large
	$0$-calculus on polyhomogeneous functions (cf.\ \cite{MazzeoEdgeI}),
	we obtain from the inequality $\Re\left(\mathcal{H}_{\rf}+\mathcal{F}\right)>-1$,
	that the operator $G_{1}$ induces a continuous linear map
	\[
	G_{1}:\mathcal{A}_{\phg}^{\mathcal{F}}\left(X\right)\to\mathcal{A}_{\phg}^{\mathcal{H}_{\lf}\eunion\left(\mathcal{F}+\mathcal{H}_{\ff_{0}}\right)}\left(X\right).
	\]
	It follows that $G_{1}v\in\mathcal{A}_{\phg}^{\mathcal{H}_{\lf}\eunion\left(\mathcal{F}+\mathcal{H}_{\ff_{0}}\right)}\left(X\right)$.
	Therefore, $u$ is polyhomogeneous with index set $\tilde{\mathcal{F}}=\mathcal{E}_{\lf}\cup\left(\mathcal{H}_{\lf}\eunion\left(\mathcal{F}+\mathcal{H}_{\ff_{0}}\right)\right)$.
	Since $\Re\left(\mathcal{E}_{\lf}\right)>\delta$, $\Re\left(\mathcal{H}_{\lf}\right)>\delta$,
	$\Re\left(\mathcal{F}\right)>\delta$ and $\left[\mathcal{H}_{\ff_{0}}\right]=0$,
	it follows that $\Re\left(\tilde{\mathcal{F}}\right)>\delta$ as claimed.
\end{proof}

\subsection{\label{subsec:The-guiding-example-=0000A76}The guiding example (continued)}

We continue from \S\ref{subsec:The-guiding-example-=0000A73} our
analysis of the scalar Laplacian on the asymptotically hyperbolic
manifold $\left(X^{n+1},g\right)$.

The analysis of the Bessel family of $\Delta$ carried out in \S\ref{subsec:The-guiding-example-=0000A72}
implies that, for every $w\in\left(0,n\right)$, the map
\[
\Delta:x^{w}H_{0}^{2}\left(X\right)\to x^{w}L_{b}^{2}\left(X\right)
\]
is Fredholm. In fact, Theorem \ref{thm:(Mazzeo)-generalized-projectors-and-gen-inverse}
implies that the $x^{w}L_{b}^{2}$ kernel of $\Delta$ consists of
polyhomogeneous functions; since the only indicial root $>w$ is the
simple indicial root $n$, the $x^{w}L_{b}^{2}$ kernel coincides
with the kernel on $x^{n}C^{\infty}\left(X\right)$. Moreover, the
range is orthogonally complemented by the $x^{w}L_{b}^{2}$ kernel
of $\Delta^{\dagger}=x^{-\left(n-2w\right)}\Delta x^{n-2w}$, which
is again isomorphic to the $x^{n}C^{\infty}$ kernel of $\Delta$
because $n-w\in\left(0,n\right)$ as well. Integrating by parts, we
see that this kernel is actually trivial, so the map above is actually
invertible for every $w\in\left(0,n\right)$. We can now prove the
following result, without need for the symbolic $0$-calculus:
\begin{thm}
	Let $\varphi\in C^{\infty}\left(\partial X\right)$. Let $\mathcal{E}_{\lf}=\mathbb{N}\eunion\left(\mathbb{N}+n\right)$.
	Then there exists a unique solution $u\in\mathcal{A}_{\phg}^{\mathcal{E}_{\lf}}$
	of $\Delta u=0$ such that $u_{|\partial X}=\varphi$.
\end{thm}

\begin{proof}
	Choose an arbitrary function $u_{0}\in C^{\infty}\left(X\right)$
	such that $u_{0|\partial X}=\varphi$. Since $0$ is an indicial root
	of $\Delta$, we have $v:=\Delta u_{0}\in xC^{\infty}\left(X\right)$.
	Now, choose $w\in\left(0,1\right)$. Then $xC^{\infty}\left(X\right)\subset x^{w}L_{b}^{2}\left(X\right)$.
	Therefore, taking the inverse $G$ of $\Delta$ as a map $x^{w}H_{0}^{2}\left(X\right)\to x^{w}L_{b}^{2}\left(X\right)$,
	we can define $u_{1}=-Gv$. We have $\Delta u_{1}=-\Delta Gv=-v$,
	so $\Delta\left(u_{0}+u_{1}\right)=0$. To conclude we must show that
	$u=u_{0}+u_{1}\in\mathcal{A}_{\phg}^{\mathcal{E}_{\lf}}\left(X\right)$
	and that $u_{|\partial X}=\varphi$. By Theorem \ref{thm:(Mazzeo)-generalized-projectors-and-gen-inverse},
	we know that $G\in\Psi_{0}^{-m,\mathcal{H}}\left(X\right)$ where
	$\Re\left(\mathcal{H}_{\lf}\right)>w$, $\Re\left(\mathcal{H}_{\rf}\right)>-w-1$,
	and $\left[\mathcal{H}_{\ff_{0}}\right]=0$. Since $v\in xC^{\infty}\left(X\right)$,
	by the mapping properties proved in \cite{MazzeoEdgeI}, we have $Gv\in\mathcal{A}_{\phg}^{\mathcal{H}_{\lf}\eunion\left(\mathcal{H}_{\ff_{0}}+1\right)}\left(X\right)$.
	It follows that $u_{1}$ is polyhomogeneous, and since $\Re\left(\mathcal{H}_{\lf}\eunion\left(\mathcal{H}_{\ff_{0}}+1\right)\right)>w>0$,
	it vanishes on the boundary. This implies that $u_{|\partial X}=\left(u_{0}\right)_{|\partial X}=\varphi$.
	Since the indicial roots of $\Delta$ are $0$ and $n$, both simple,
	we can improve the index set to $\mathcal{E}_{\lf}$ by plugging the
	expansion of $u$ in the equation $\Delta u=0$ and solving term by
	term.
\end{proof}
Now, we would like to generalize this result by allowing the boundary
value to live in a Sobolev space. The expected regularity is suggested
by Theorem \ref{thm:sharp-regularity-of-traces}. To formulate it,
choose a weight $\delta<0$ and denote by $\tr:\ker_{x^{\delta}L_{b}^{2}}\Delta\to C^{-\infty}\left(\partial X\right)$
the map associating to a weak solution $u\in x^{\delta}L_{b}^{2}\left(X\right)$
of $\Delta u=0$ the (distributional) leading term in the weak expansion
of $u$ (cf.\ \S\ref{subsec:Weak-polyhomogeneous-expansions}).
\begin{thm}
	\label{thm:dirichlet-sharp-for-laplacian}Let $\varphi\in H^{\delta}\left(\partial X\right)$.
	Then there exists a unique solution $u\in x^{\delta}L_{b}^{2}\left(X\right)$
	of $\Delta u=0$ such that $\tr u=\varphi$. The solution is in $x^{\delta}H_{0}^{k}\left(X\right)$
	for every $k\in\mathbb{N}$.
\end{thm}

\begin{rem}
	Before starting the proof, we remark that the argument used in the
	previous proof does \emph{not }immediately work in this case. The
	problem is that, if $\varphi\in H^{\delta}\left(\partial X\right)$,
	it is not obvious how to extend $\varphi$ to a function $u_{0}\in x^{\delta}L_{b}^{2}\left(X\right)$.
	Moreover, even if we are able to do that, it is not immediately obvious
	how to construct $u_{0}$ in such a way that $v=\Delta u_{0}$ is
	in a space where we can invert $\Delta$. This is essentially what
	an appropriately constructed symbolic $0$-Poisson operator does for
	us.
\end{rem}

\begin{proof}
	Choose an arbitrary weight $\overline{\delta}\in\left(0,n\right)$,
	and denote by $A_{\Delta}$ the trace map for $\Delta$ associated
	to $\delta,\overline{\delta}$ (and the auxiliary choices), defined
	as in \S\ref{subsec:The-guiding-example-=0000A73}. If $u\in x^{\delta}L_{b}^{2}\left(X\right)$
	and $\Delta u=0$, then $A_{\Delta}u=\tr u$. Now, by Theorem \ref{thm:main-theorem-parametrix}
	(more specifically, by its proof) there exists a symbolic $0$-Poisson
	operator $B\in\hat{\Psi}_{0\po}^{-\infty,\left(\mathcal{E}_{\lf},\left[0\right]\right)}\left(\partial X,X\right)$
	such that $P_{1}B=B$ and $A_{\Delta}B=\id-S$ for some $S\in\Psi^{-\infty}\left(\partial X\right)$.
	By Theorem \ref{thm:mapping-twisted-trace-interior-sobolev}, $B$
	induces a continuous linear map
	\[
	B:H^{\delta}\left(\partial X\right)\to x^{\delta}L_{b}^{2}\left(X\right).
	\]
	Therefore, we can define $u_{0}=B\varphi$. By construction, we have
	$\Delta u_{0}=0$ and $\tr u_{0}=A_{\Delta}u_{0}=\varphi-S\varphi$.
	Since $S$ has order $-\infty$, we have $S\varphi\in C^{\infty}\left(\partial X\right)$.
	By the previous theorem, we can then find a unique $\tilde{u}\in\mathcal{A}_{\phg}^{\mathcal{E}_{\lf}}\left(X\right)$
	such that $\Delta\tilde{u}=0$ and $\tilde{u}_{|\partial X}=S\varphi$.
	The function $u:=u_{0}+\tilde{u}$ is then in $x^{\delta}L_{b}^{2}\left(X\right)$
	and solves $\Delta u=0$ and $\tr u=A_{\Delta}u=\varphi$. The solution
	is in $x^{\delta}H_{0}^{k}\left(X\right)$, because $P_{1}u=u$ and
	$P_{1}$ is bounded as a map $x^{\delta}L_{b}^{2}\left(X\right)\to x^{\delta}H_{0}^{k}\left(X\right)$
	for every $k$.
\end{proof}
In fact, the previous proof leads to the following bonus result:
\begin{thm}
	\label{thm:poisson-of-Laplacian-is-symbolic-0-poisson}Denote by $B_{\Delta}:C^{\infty}\left(\partial X\right)\to\mathcal{A}_{\phg}^{\mathcal{E}_{\lf}}\left(X\right)$
	the Poisson operator of $\Delta$, namely the operator mapping $\varphi\in C^{\infty}\left(\partial X\right)$
	to the unique solution $u\in\mathcal{A}_{\phg}^{\mathcal{E}_{\lf}}\left(X\right)$
	of $\Delta u=0$ and $u_{|\partial X}=\varphi$. Then $B_{\Delta}\in\hat{\Psi}_{0\po}^{-\infty,\left(\mathcal{E}_{\lf},\left[0\right]\right)}\left(\partial X,X\right)$.
\end{thm}

\begin{proof}
	Let $B$ and $S$ be as in the previous proof. Then we have $B_{\Delta}=B+B_{\Delta}S$.
	Now, since $S$ is smoothing, it defines a continuous linear map $C^{-\infty}\left(\partial X\right)\to C^{\infty}\left(\partial X\right)$.
	The map $B_{\Delta}:C^{\infty}\left(\partial X\right)\to\mathcal{A}_{\phg}^{\mathcal{E}_{\lf}}\left(X\right)$
	is continuous as well: if the sequence of functions $\varphi_{k}\in C^{\infty}\left(\partial X\right)$
	converges to $0$, then $B_{\Delta}\varphi_{k}$ converges to a solution
	of $\Delta u=0$ with vanishing trace, hence converges to $0$. It
	follows that $B_{\Delta}S$ is a continuous linear map $C^{-\infty}\left(\partial X\right)\to\mathcal{A}_{\phg}^{\mathcal{E}_{\lf}}\left(X\right)$.
	This is an equivalent alternative characterization of the residual
	Poisson operators $\Psi_{\po}^{-\infty,\mathcal{E}_{\lf}}\left(\partial X,X\right)$.
	Since $B\in\hat{\Psi}_{0\po}^{-\infty,\left(\mathcal{E}_{\lf},\left[0\right]\right)}\left(\partial X,X\right)$
	and $\Psi_{\po}^{-\infty,\mathcal{E}_{\lf}}\left(\partial X,X\right)$
	is contained in this space, the proof is concluded.
\end{proof}
The analysis of this example continues in \S\ref{subsec:The-guiding-example-=0000A77}.

\section{\label{sec:sharp-regularity-of-traces}A regularity result}

In this section, we prove a sharp regularity result for the coefficients
in the expansions of weak $x^{\delta}L_{b}^{2}$ solutions of $0$-elliptic
equations.

\subsection{\label{subsec:Weak-polyhomogeneous-expansions}Weak polyhomogeneous
	expansions and twisted Sobolev regularity}

Fix a $0$-elliptic operator $L\in\Diff_{0}^{m}\left(X\right)$ with
constant indicial roots. Fix an auxiliary vector field $V$ on $X$
transverse to $\partial X$ and inward-pointing, and a boundary defining
function $x$ such that $Vx\equiv1$ near $\partial X$. We also choose
a collar $\partial X\times[0,\varepsilon)\to X$ induced by the flow
of $V$ starting at $\partial X$, with $\varepsilon>0$ small enough
for $Vx\equiv1$ to hold on the collar. All the polyhomogeneous expansions
in this sub-section are computed with respect to $V$.

As discussed extensively in this paper, if $u$ is a $x^{\delta}H_{b}^{\infty}$
solution of $Lu=0$, then $u$ is polyhomogeneous with expansion
\[
u\sim\sum_{\left(\alpha,l\right)\in\mathcal{E}_{\lf}}u_{\alpha,l}x^{\alpha}\left(\log x\right)^{l},
\]
where $\mathcal{E}_{\lf}$ is the extended union of the simple index
sets generated by the pairs $\left(\mu,M_{\mu}\right)\in\widetilde{\spec}_{b}\left(L\right)$
with $\Re\left(\mu\right)>\delta$. Now, as noted in \S7 of \cite{MazzeoEdgeI},
if $u$ is only in $x^{\delta}L_{b}^{2}\left(X\right)$ then $u$
is only \emph{weakly }polyhomogeneous. More precisely, think of $u$
as a function on $\partial X\times[0,\varepsilon)$; given a smooth
density $\nu$ on $\partial X$, the contraction $u\left[\nu\right]$
is polyhomogeneous on $[0,\varepsilon)$ with expansion
\[
u\left[\nu\right]\sim\sum_{\left(\alpha,l\right)\in\mathcal{E}_{\lf}}u_{\alpha,l}\left[\nu\right]x^{\alpha}\left(\log x\right)^{l}.
\]
The coefficients $u_{\alpha,l}$ therefore make sense as distributions.

If $\delta$ is a surjective weight for $L$, we can use \S\ref{subsec:The-trace-map-for-L}
to understand better the regularity of the coefficients $u_{\alpha,l}$.
More precisely, for every $\alpha\in\mathcal{E}_{\lf}$, call $\tilde{M}_{\alpha}=\max\left\{ l\in\mathbb{N}:\left(\alpha,l\right)\in\mathcal{E}_{\lf}\right\} $.
Then we have the following
\begin{thm}
	\label{thm:sharp-regularity-of-traces}Denote by $\mathfrak{N}=\mathfrak{N}_{\tilde{M}_{\alpha}}$
	the endomorphism of $E_{\alpha,\tilde{M}_{\alpha}}$ induced by the
	action of $x\partial_{x}-\alpha$. Then $\tr_{\alpha,\tilde{M}_{\alpha}}^{V}u$
	is a section of $H^{\delta-\alpha-\mathfrak{N}}\left(\partial X;E_{\alpha,\tilde{M}_{\alpha}}\right)$.
\end{thm}

\begin{proof}
	First of all, if $\delta$ is also injective, then the
	$x^{\delta}L_{b}^{2}$ kernel is a finite-dimensional subspace of
	$\mathcal{A}_{\phg}^{\mathcal{E}_{\lf}}\left(X\right)$; in this case,
	we have $\tr_{\alpha,\tilde{M}_{\alpha}}^{V}u\in C^{\infty}\left(\partial X;E_{\alpha,\tilde{M}_{\alpha}}\right)$.
	We can therefore assume that $\delta$ is not injective. Choose an
	injective weight $\overline{\delta}>\Re\left(\alpha\right)$ and a
	smooth positive density $\omega$ for $X$. We can then define the
	trace map $\boldsymbol{A}_{L}$ relative to $\delta,\overline{\delta}$
	as in \S\ref{subsubsec:The-trace-map}. By Theorem \ref{thm:A_L-is-twisted-symbolic-0-trace},
	we have $\boldsymbol{A}_{L}\in\hat{\Psi}_{0\tr}^{-\infty,\left(\mathcal{E}_{\rf},\left[0\right]\right),-\boldsymbol{\mathfrak{s}}_{L}}\left(X;\partial X,\boldsymbol{E}_{L}\right)$
	where
	\[
	\boldsymbol{E}_{L}=\bigoplus_{\begin{smallmatrix}\mu\in\spec_{b}\left(L\right)\\
			\Re\left(\mu\right)\in\left(\delta,\overline{\delta}\right)
	\end{smallmatrix}}E_{\mu,\tilde{M}_{\mu}},
	\]
	$\boldsymbol{\mathfrak{s}}_{L}:\boldsymbol{E}_{L}\to\boldsymbol{E}_{L}$
	is the endomorphism induced by the action of $x\partial_{x}$, and
	$\mathcal{E}_{\rf}=\overline{\mathcal{E}_{\lf}}-2\delta-1$. By construction,
	the restriction of $\boldsymbol{\mathfrak{s}}_{L}$ to $E_{\mu,\tilde{M}_{\mu}}$
	is precisely $\mu+\mathfrak{N}$. Therefore, calling $A_{L}^{\mu}$
	the component of $\boldsymbol{A}_{L}$ taking values in $E_{\mu,\tilde{M}_{\mu}}$,
	we have $A_{L}^{\mu}\in\hat{\Psi}_{0\tr}^{-\infty,\left(\mathcal{E}_{\rf},\left[0\right]\right),-\left(\mu+\mathfrak{N}\right)}\left(X;\partial X,E_{\mu,\tilde{M}_{\mu}}\right)$.
	Now, since $\Re\left(\mathcal{E}_{\rf}\right)>-\delta-1$, by Theorem
	\ref{thm:mapping-twisted-trace-interior-sobolev} $A_{L}^{\mu}$ is
	bounded as a map
	\[
	A_{L}^{\mu}:x^{\delta}L_{b}^{2}\left(X\right)\to H^{\delta-\mu-\mathfrak{N}}\left(\partial X;E_{\mu,\tilde{M}_{\mu}}\right).
	\]
	Since $Lu=0$, we have $P_{1}u=u$ and therefore $A_{L}^{\mu}u=\tr_{\mu,\tilde{M}_{\mu}}^{V}u$.
	Therefore, we conclude that $\tr_{\mu,\tilde{M}_{\mu}}^{V}u\in H^{\delta-\mu-\mathfrak{N}}\left(\partial X;E_{\mu,\tilde{M}_{\mu}}\right)$.
	Now, since $\left(\alpha,0\right)\in\mathcal{E}_{\lf}$, there exists
	a \emph{unique} indicial root $\mu\in\spec_{b}\left(L\right)$ such
	that $\Re\left(\mu\right)>\delta$ , $\alpha=\mu+k$ for some $k\in\mathbb{N}$,
	and $\tilde{M}_{\alpha}=\tilde{M}_{\mu}$. Plugging the expansion
	of $u$ in the equation $Lu=0$ and solving term by term, it is immediate
	to see that $\tr_{\alpha,\tilde{M}_{\alpha}}^{V}u=D_{k}\tr_{\mu,\tilde{M}_{\mu}}^{V}u$
	for some operator $D_{k}\in\Diff^{k}\left(\partial X;E_{\mu,\tilde{M}_{\mu}},E_{\alpha,\tilde{M}_{\alpha}}\right)$
	dependent on $L$ and $V$. The claim then follows from the fact that
	$D_{k}$ induces a bounded map
	\[
	D_{k}:H^{\delta-\mu-\mathfrak{N}}\left(\partial X;E_{\mu,\tilde{M}_{\mu}}\right)\to H^{\delta-\mu-k-\mathfrak{N}}\left(\partial X;E_{\mu+k,\tilde{M}_{\mu}}\right).
	\]
	This is a general feature of Sobolev spaces of twisted order; to see
	this, using standard arguments in pseudodifferential theory, it suffices
	to prove that if $\mathfrak{a}:\mathbb{C}^{N}\to\mathbb{C}^{N}$,
	$u\in H^{\mathfrak{a}}\left(\mathbb{R}^{n};\mathbb{C}^{N}\right)$
	and $\gamma\in\mathbb{N}^{n}$ then $\partial_{y}^{\gamma}u\in H^{\mathfrak{a}-\left|\gamma\right|}\left(\mathbb{R}^{n};\mathbb{C}^{N}\right)$.
	This is done as follows: writing $\left\langle \eta\right\rangle ^{\mathfrak{a}-\left|\gamma\right|}\widehat{\partial_{y}^{\gamma}u}=\left\langle \eta\right\rangle ^{\mathfrak{a}}\left\langle \eta\right\rangle ^{-\left|\gamma\right|}\left(i\eta\right)^{\gamma}\hat{u}\left(\eta\right)$,
	we have that $\left\langle \eta\right\rangle ^{-\left|\gamma\right|}\left(i\eta\right)^{\gamma}$
	is bounded conormal on $\overline{\mathbb{R}^{n}}$, and moreover $\left\langle \eta\right\rangle ^{\mathfrak{a}}\hat{u}\left(\eta\right)\in L^{2}\left(\mathbb{R}^{n};\mathbb{C}^{N}\right)$
	since $u\in H^{\mathfrak{a}}\left(\mathbb{R}^{n};\mathbb{C}^{N}\right)$;
	we then get $\left\langle \eta\right\rangle ^{\mathfrak{a}-\left|\gamma\right|}\widehat{\partial_{y}^{\gamma}u}\in L^{2}\left(\mathbb{R}^{n};\mathbb{C}^{N}\right)$,
	that is, $\partial_{y}^{\gamma}u\in H^{\mathfrak{a}-\left|\gamma\right|}\left(\mathbb{R}^{n};\mathbb{C}^{N}\right)$.
\end{proof}

\subsection{\label{subsec:Individual-regularity-of}Individual regularity of
	the coefficients}

The previous result justifies a deeper discussion on the twisted order
Sobolev spaces $H^{\alpha-\mathfrak{N}_{M}}\left(\partial X;\mathbb{C}^{M+1}\right)$,
where $\alpha\in\mathbb{C}$ and
\[
\mathfrak{N}_{M}=\left(\begin{matrix}0 & 1\\
	& 0 & \ddots\\
	&  & \ddots & M\\
	&  &  & 0
\end{matrix}\right).
\]
To better understand this space, we introduce a scale of Sobolev spaces
on $\partial X$ with ``log-shifted'' order. We first define a variant
of the Japanese bracket, $\left[\eta\right]=\sqrt{1+\left\langle \eta\right\rangle ^{2}}$,
designed so that $\left[\eta\right]$ is nowhere zero and $\left[\eta\right]\sim\left\langle \eta\right\rangle $
as $\left|\eta\right|\to+\infty$. Given $m\in\mathbb{R}$ and $k\in\mathbb{Z}$,
define $H_{\log}^{m,k}\left(\mathbb{R}^{n}\right)$ as the space of
tempered distributions $u\in\mathcal{S}'\left(\mathbb{R}^{n}\right)$
such that \[\left\langle \eta\right\rangle ^{m}\left(\log\left[\eta\right]\right)^{k}\hat{u}\left(\eta\right)\in L^{2}\left(\mathbb{R}^{n}\right).\]
We can easily extend this definition to $\partial X$, using the usual
partition of unity arguments. Note that $H_{\log}^{m,0}\left(\partial X\right)=H^{m}\left(\partial X\right)$,
and that we have compact inclusions
\[
H^{m+\varepsilon}\left(\partial X\right)\subset H_{\log}^{m,k}\left(\partial X\right)\subset H_{\log}^{m,k-1}\left(\partial X\right)\subset H^{m-\varepsilon}\left(\partial X\right)
\]
for every $m\in\mathbb{R}$, $\varepsilon>0$ and $k\in\mathbb{Z}$.
\begin{prop}
	\label{prop:log-sobolev-spaces-and-sub-diagonal-twisting}Let $\alpha\in\mathbb{C}$
	and $M\in\mathbb{N}$. Let $u=\left(u^{0},...,u^{M}\right)\in H^{\alpha-\mathfrak{N}_{M}}\left(\partial X;\mathbb{C}^{M+1}\right)$.
	\begin{enumerate}
		\item For every $j\in\left\{ 0,...,M\right\} $, we have $u^{j}\in H_{\log}^{\Re\left(\alpha\right),j-M}\left(\partial X\right)$.
		\item Let $0\leq M'\leq M$. If $u^{M'+1}=\cdots=u^{M}=0$, then \[u'=\left(u^{0},...,u^{M'}\right)\in H^{\alpha-\mathfrak{N}_{M'}}\left(\partial X;\mathbb{C}^{M'+1}\right).\]
	\end{enumerate}
\end{prop}

\begin{proof}
	It suffices to prove the analogous statements on $\mathbb{R}^{n}$.
	Let \[u=\left(u^{0},...,u^{M}\right)\in H^{\alpha-\mathfrak{N}_{M}}\left(\mathbb{R}^{n};\mathbb{C}^{M+1}\right).\]
	To prove that $u^{j}\in H_{\log}^{\Re\left(\alpha\right),j-M}\left(\mathbb{R}^{n}\right)$
	for every $j\in\left\{ 0,...,M\right\} $, we need to prove that \[\left\langle \eta\right\rangle ^{\alpha}\left(\log\left[\eta\right]\right)^{j-M}\hat{u}^{j}\left(\eta\right)\in L^{2}\left(\mathbb{R}^{n}\right).\]
	In matrix form, writing
	\[
	Z=\left(\begin{matrix}\left(\log\left[\eta\right]\right)^{-M}\\
		& \left(\log\left[\eta\right]\right)^{-M+1}\\
		&  & \ddots\\
		&  &  & 1
	\end{matrix}\right),
	\]
	we need to prove that
	\[
	\left\langle \eta\right\rangle ^{\alpha}Z\hat{u}\left(\eta\right)\in L^{2}\left(\mathbb{R}^{n};\mathbb{C}^{M+1}\right).
	\]
	Write $\left\langle \eta\right\rangle ^{\alpha}Z\hat{u}\left(\eta\right)=Z\left\langle \eta\right\rangle ^{\mathfrak{N}_{M}}\left\langle \eta\right\rangle ^{\alpha-\mathfrak{N}_{M}}\hat{u}\left(\eta\right)$.
	Since $u\in H^{\alpha-\mathfrak{N}_{M}}\left(\mathbb{R}^{n};\mathbb{C}^{M+1}\right)$,
	we have \[\left\langle \eta\right\rangle ^{\alpha-\mathfrak{N}_{M}}\hat{u}\left(\eta\right)\in L^{2}\left(\mathbb{R}^{n};\mathbb{C}^{M+1}\right),\]
	so it suffices to prove that the entries of $Z\left\langle \eta\right\rangle ^{\mathfrak{N}_{M}}$
	are bounded. Arguing as in \S\ref{subsubsec:Log-homogeneous-functions},
	given $j,k\in\left\{ 0,...,M\right\} $ with $k\geq j$, the $\left(j,k\right)$-th
	entry of $\left\langle \eta\right\rangle ^{\mathfrak{N}_{M}}$ is
	\[
	\left(\left\langle \eta\right\rangle ^{\mathfrak{N}_{M}}\right)_{j}^{k}=\binom{k}{j}\left(\log\left\langle \eta\right\rangle \right)^{k-j}.
	\]
	Therefore, the $\left(j,k\right)$-th entry of $Z\left\langle \eta\right\rangle ^{\mathfrak{N}_{M}}$
	is
	\[
	\left(Z\left\langle \eta\right\rangle ^{\mathfrak{N}_{M}}\right)_{j}^{k}=\binom{k}{j}\frac{\left(\log\left\langle \eta\right\rangle \right)^{k-j}}{\left(\log\left[\eta\right]\right)^{M-j}}.
	\]
	Since $\left\langle \eta\right\rangle \sim\left[\eta\right]$ as $\left|\eta\right|\to+\infty$,
	we have
	\[
	\frac{\left(\log\left\langle \eta\right\rangle \right)^{k-j}}{\left(\log\left[\eta\right]\right)^{M-j}}\sim\left(\log\left|\eta\right|\right)^{k-M}
	\]
	which is bounded away from the origin because $k\leq M$. As $\left|\eta\right|\to0$,
	\[\left(\log\left[\eta\right]\right)^{M-j}\to\left(\log\sqrt{2}\right)^{M-j}\]
	and either $\left(\log\left\langle \eta\right\rangle \right)^{k-j}\to0$
	if $k>j$, or $\left(\log\left\langle \eta\right\rangle \right)^{k-j}\to1$
	if $k=j$. It follows that the entries of $Z\left\langle \eta\right\rangle ^{\mathfrak{N}_{M}}$
	are bounded as claimed; this concludes the proof of the first point.
	To prove the second point, observe that we can write $\left\langle \eta\right\rangle ^{-\mathfrak{N}_{M}}$
	as a block matrix
	\[
	\left\langle \eta\right\rangle ^{-\mathfrak{N}_{M}}=\left(\begin{matrix}\left\langle \eta\right\rangle ^{-\mathfrak{N}_{M'}} & A\\
		0 & B
	\end{matrix}\right).
	\]
	If $u\in H^{\alpha-\mathfrak{N}_{M}}\left(\mathbb{R}^{n};\mathbb{C}^{M+1}\right)$
	and the coefficients $u^{M'+1},...,u^{M}$ vanish, then we can write
	in block form
	\[
	\left\langle \eta\right\rangle ^{\alpha-\mathfrak{N}_{M}}\hat{u}\left(\eta\right)=\left(\begin{matrix}\left\langle \eta\right\rangle ^{\alpha-\mathfrak{N}_{M'}}\hat{u}'\left(\eta\right)\\
		0
	\end{matrix}\right).
	\]
	Therefore, $u'\in H^{\alpha-\mathfrak{N}_{M'}}\left(\mathbb{R}^{n};\mathbb{C}^{M'+1}\right)$.
\end{proof}
From Theorem \ref{thm:sharp-regularity-of-traces} and Proposition
\ref{prop:log-sobolev-spaces-and-sub-diagonal-twisting}, we immediately
obtain the following
\begin{cor}
	\label{cor:Log-loss-of-regularity}Let $L$, $\delta$, $V$, $u$
	be as in Theorem \ref{thm:sharp-regularity-of-traces}. Then:
	\begin{enumerate}
		\item $u_{\alpha,l}\in H_{\log}^{\delta-\Re\left(\alpha\right),-\left(\tilde{M}_{\alpha}-l\right)}\left(\partial X\right)$
		for every $l\in\left\{ 0,...,\tilde{M}_{\alpha}\right\} $; in particular,
		\[u_{\alpha,\tilde{M}_{\alpha}}\in H^{\delta-\Re\left(\alpha\right)}\left(\partial X\right).\]
		\item If $0\leq N\leq\tilde{M}_{\alpha}$ and $u_{\alpha,N+1}=\cdots=u_{\alpha,\tilde{M}_{\alpha}}=0$,
		then $u_{\alpha,l}\in H_{\log}^{\delta-\Re\left(\alpha\right),-\left(N-l\right)}\left(\partial X\right)$
		for every $l\in\left\{ 0,...,N\right\} $; in particular, \[u_{\alpha,N}\in H^{\delta-\Re\left(\alpha\right)}\left(\partial X\right).\]
	\end{enumerate}
\end{cor}

\begin{rem}
	The corollary above corrects Theorem 7.3 of \cite{MazzeoEdgeI}, where
	the logarithmic loss of regularity is missing. In \S\ref{subsec:The-guiding-example-=0000A77},
	we shall see a concrete example where this loss does actually occur.
\end{rem}

\begin{rem}
	(cf.\ Remark \ref{rem:rescaled-x^mL} and the end of Remark \ref{rem:why-not-reduced-trace-map})
	Let $\tilde{L}\in\Diff^{m}\left(X\right)$ be elliptic in the standard
	sense, and let $x$ be a boundary defining function on $X$. Then
	$L:=x^{m}\tilde{L}$ is $0$-elliptic, with constant indicial roots
	equal to $0,...,m-1$, all simple. In this case, every weight $\delta<0$
	is surjective. Fixing such a $\delta$, the index set $\mathcal{E}_{\lf}$
	is the extended union of the simple index sets $k+\mathbb{N}$ for
	$k\in\left\{ 0,...,m-1\right\} $; equivalently, $\mathcal{E}_{\lf}$
	is generated by the pairs $\left(0,0\right),\left(1,1\right),...,\left(m-1,m-1\right)$.
	However, by standard elliptic theory, the expansion of a $x^{\delta}H_{b}^{\infty}$
	solution $u$ of $Lu=0$ does not actually have log terms: it is a
	simple Taylor expansion of the form
	\[
	u\sim\sum_{k=0}^{\infty}u_{k}x^{k}.
	\]
	When $u\in x^{\delta}L_{b}^{2}\left(X\right)$, by the previous corollary,
	we obtain $u_{k}\in H^{\delta-k}\left(\partial X\right)$ for every
	$k$. In particular, using the equality $L^{2}\left(X\right)=x^{-\frac{1}{2}}L_{b}^{2}\left(X\right)$,
	we recover the well-known fact that if $u$ is a weak solution of
	$\tilde{L}u=0$ in $L^{2}\left(X\right)$, then its boundary trace
	is in $H^{-\frac{1}{2}}\left(\partial X\right)$.
\end{rem}

\subsection{\label{subsec:The-guiding-example-=0000A77}The guiding example (continued)}

We continue from \S\ref{subsec:The-guiding-example-=0000A76}, and
conclude, our analysis of the scalar Laplacian on the asymptotically
hyperbolic manifold $\left(X^{n+1},g\right)$. We now provide
an example where the logarithmic loss of regularity described in Corollary
\ref{cor:Log-loss-of-regularity} actually occurs.

Let $\delta<0$. Choose a function
\[
\varphi\in H^{\delta}\left(\partial X\right)\backslash H_{\log}^{\delta,1}\left(\partial X\right).
\]
By the result proved in \S\ref{subsec:The-guiding-example-=0000A76},
we know that there exists a unique function $u\in x^{\delta}L_{b}^{2}\left(X\right)$
such that $\Delta u=0$ and $\tr u=\varphi$, where 
\[
\tr:\ker_{x^{\delta}L_{b}^{2}}\Delta\to H^{\delta}\left(\partial X\right)
\]
is the map assigning to a weak solution $u\in x^{\delta}L_{b}^{2}\left(X\right)$
of $\Delta u=0$ the leading coefficient in its weak expansion. Now,
consider the log-homogeneous term of degree $\leq\left(n,1\right)$
in the weak expansion of $u$:
\[
\tr_{n,1}^{\partial_{x}}u=u_{n,0}x^{n}+u_{n,1}x^{n}\log x
\]
(here, as in \S\ref{subsec:The-guiding-example-=0000A72}, $\partial_{x}$
is the gradient of $x$ with respect to the metric $x^{2}g$). By
Theorem \ref{thm:sharp-regularity-of-traces}, we have $\tr_{n,1}^{V}u\in H^{\delta-n-\mathfrak{N}}\left(\partial X;E_{n,1}\right)$
where $\mathfrak{N}:E_{n,1}\to E_{n,1}$ is the endomorphism induced
by $x\partial_{x}-n$. In particular, Corollary \ref{cor:Log-loss-of-regularity}
guarantees that
\begin{align*}
	u_{n,0} & \in H_{\log}^{\delta-n,-1}\left(\partial X\right)\\
	u_{n,1} & \in H^{\delta-n}\left(\partial X\right).
\end{align*}
Whether we \emph{actually} have the loss in the log-regularity of
$u_{n,0}$, depends on $g$ and on the parity of $n$.

A purely formal analysis of the equation $\Delta u=0$ shows that
the coefficient $u_{n,1}$ is determined by the coefficient $u_{0,0}=\tr_{0,0}u=\varphi$
by an equation of the form
\[
u_{n,1}=P_{n}u_{0,0},
\]
where $P_{n}\in\Diff^{n}\left(\partial X\right)$ is a natural differential
operator dependent on $g$ and on the choice of $h_{0}\in\mathfrak{c}_{\infty}\left(g\right)$.

The \emph{principal symbol} of $P_{n}$ can be recovered from the
analysis of the Bessel operator of $\Delta$: more precisely, given
$p\in\partial X$ and $\eta\in T_{p}^{*}\partial X$, the principal
symbol $\sigma_{\eta}\left(P_{n}\right)$ is the coefficient such
that, if $v$ is a solution of $\hat{N}_{\eta}\left(\Delta\right)v=0$
decaying at $x\to+\infty$ and with $v_{|x=0}=v_{0}$, then $\sigma_{\eta}\left(P_{n}\right)v_{0}$
is the coefficient of $x^{n}\log x$ in the expansion of $v$. We
have seen in \S\ref{subsec:The-guiding-example-=0000A72} that $v$
must take the form $v=Cx^{\frac{n}{2}}K_{\frac{n}{2}}\left(\left|\eta\right|x\right)$.
We must then distinguish two cases.

If $n$ is odd, then $x^{\frac{n}{2}}K_{\frac{n}{2}}\left(x\right)$
is smooth on $\mathbb{R}_{1}^{1}$; therefore there is no log term.
This implies that $\sigma_{\eta}\left(P_{n}\right)=0$, and therefore
$P_{n}$ is actually of order $\leq n-1$. If $n$ is even, then the
leading coefficient of $x^{\frac{n}{2}}K_{\frac{n}{2}}\left(\left|\eta\right|x\right)$
is 
\[
\gamma_{0,0}:=2^{\frac{n}{2}-1}\left|\eta\right|^{-\frac{n}{2}}\Gamma\left(\frac{n}{2}\right),
\]
while the coefficient of $x^{n}\log x$ is
\[
-\left(-2\right)^{-\frac{n}{2}}\Gamma\left(\frac{n}{2}+1\right)^{-1}\left|\eta\right|^{\frac{n}{2}}.
\]
Therefore, we have
\[
\sigma_{\eta}\left(P_{n}\right)=\frac{\left(-1\right)^{\frac{n}{2}+1}}{2^{n-1}\Gamma\left(\frac{n}{2}\right)\Gamma\left(\frac{n}{2}+1\right)}\left|\eta\right|^{n}.
\]
This formula shows that, in the $n$ even case, $P_{n}$ has the same
principal symbol of the (rescaled) power of the Laplacian $\Delta_{h_{0}}^{\frac{n}{2}}$;
in particular, it is elliptic.
\begin{rem}
	If $g$ is \emph{Poincaré--Einstein }(that is, $g$ satisfies the
	equation $\mathrm{Ric}_{0}\left(g\right)=-ng$), then $P_{n}$ depends
	only on the conformal infinity $\mathfrak{c}_{\infty}\left(g\right)$
	and on the representative $h_{0}\in\mathfrak{c}_{\infty}\left(g\right)$.
	Specifically $P_{n}$ is, up to a multiplicative constant one of the
	\emph{GJMS operators }introduced by Graham--Jenne--Mason--Sparling
	(cf.\ \cite{GJMS}). In this case, when $n$ is odd, the operator
	$P_{n}$ actually vanishes.
\end{rem}

Assume that $n$ is even. In this case, we can use the ellipticity
of $P_{n}$ to prove that 
\[
u_{n,0}\in H_{\log}^{\delta-n,-1}\left(\partial X\right)\backslash H^{\delta-n}\left(\partial X\right).
\]
We first need the following
\begin{lem}
	Let
	\[
	\mathfrak{N}=\left(\begin{matrix}0 & 1\\
		0 & 0
	\end{matrix}\right)
	\]
	and let $\left(v^{0},v^{1}\right)\in H^{\alpha-\mathfrak{N}}\left(\partial X;\mathbb{C}^{2}\right)$.
	If $v^{0}\in H^{\alpha}\left(\partial X\right)$, then $v^{1}\in H_{\log}^{\alpha,1}\left(\partial X\right)$.
\end{lem}

\begin{proof}
	It suffices to prove this statement on $\mathbb{R}^{n}$. The fact
	that $\left(v^{0},v^{1}\right)\in H^{\alpha-\mathfrak{N}}\left(\mathbb{R}^{n};\mathbb{C}^{2}\right)$
	means that
	\[
	\left\langle \eta\right\rangle ^{\alpha}\left(\begin{matrix}1 & -\log\left\langle \eta\right\rangle \\
		0 & 1
	\end{matrix}\right)\left(\begin{matrix}\hat{v}^{0}\\
		\hat{v}^{1}
	\end{matrix}\right)\in L^{2}\left(\mathbb{R}^{n};\mathbb{C}^{2}\right).
	\]
	The first condition is $\left\langle \eta\right\rangle ^{\alpha}\hat{v}^{0}\left(\eta\right)-\left\langle \eta\right\rangle ^{\alpha}\log\left\langle \eta\right\rangle \hat{v}^{1}\left(\eta\right)\in L^{2}\left(\mathbb{R}^{n}\right)$.
	Now, since $v^{0}\in H^{\alpha}\left(\mathbb{R}^{n}\right)$, we have
	$\left\langle \eta\right\rangle ^{\alpha}\hat{v}^{0}\left(\eta\right)\in L^{2}\left(\mathbb{R}^{n}\right)$,
	and therefore $\left\langle \eta\right\rangle ^{\alpha}\log\left\langle \eta\right\rangle \hat{v}^{1}\left(\eta\right)\in L^{2}\left(\mathbb{R}^{n}\right)$.
	We know that $v^{1}\in H^{\alpha}\left(\mathbb{R}^{n}\right)$, and
	therefore $\left\langle \eta\right\rangle ^{\alpha}\hat{v}^{1}\left(\eta\right)\in L^{2}\left(\mathbb{R}^{n}\right)$;
	it follows that $\left\langle \eta\right\rangle ^{\alpha}\log\left[\eta\right]\hat{v}^{1}\left(\eta\right)\in L^{2}\left(\mathbb{R}^{n}\right)$,
	that is, $v^{1}\in H_{\log}^{\alpha,1}\left(\mathbb{R}^{n}\right)$
	as claimed.
\end{proof}
Let us go back to our proof of the fact that $u_{n,0}\not\in H^{\delta-n}\left(\partial X\right)$.
Suppose for the sake of a contradiction that $u_{n,0}\in H^{\delta-n}\left(\partial X\right)$.
Since $\tr_{n,1}^{\partial_{x}}u\in H^{\delta-n-\mathfrak{N}}\left(\partial X;E_{n,1}\right)$,
the previous lemma implies that $u_{n,1}\in H_{\log}^{\delta-n,1}\left(\partial X\right)$.
Write now $u_{n,1}=P_{n}\varphi$, and let $Q\in\Psi^{-n}\left(\partial X\right)$
be a parametrix for $P_{n}$. Then we have
\[
\varphi=Qu_{n,1}+R\varphi,
\]
where $R\in\Psi^{-\infty}\left(\partial X\right)$. The term $R\varphi$
is smooth, so it is certainly in $H_{\log}^{\delta,1}\left(\partial X\right)$.
Since $Q\in\Psi^{-n}\left(\partial X\right)$ and $u_{n,1}\in H_{\log}^{\delta-n,1}\left(\partial X\right)$,
we have $Qu_{n,1}\in H_{\log}^{\delta,1}\left(\partial X\right)$.
We then obtain $\varphi\in H_{\log}^{\delta,1}\left(\partial X\right)$,
in contradiction with our initial hypothesis on $\varphi$.

\appendix

\section{Appendix}

\subsection{\label{subsec:Conormality,-polyhomogeneity,projective-tensor-products}Conormality,
	polyhomogeneity, projective tensor products}

Let $X$ be a compact manifold with corners. We call $\mathcal{M}_{1}\left(X\right)$
the set of boundary hyperfaces of $X$. We denote by $\mathcal{V}_{b}\left(X\right)$
the space of vector fields on $X$ tangent to every $H\in\mathcal{M}_{1}\left(X\right)$.
The elements of $\mathcal{V}_{b}\left(X\right)$ are called \emph{$b$-vector
	fields}.

A \emph{multi-weight} for $X$ is a family $\mathfrak{m}=\left\{ \mathfrak{m}_{H}\right\} _{H\in\mathcal{M}_{1}\left(X\right)}$,
where $\mathfrak{m}_{H}\in\mathbb{R}$. For every $H\in\mathcal{M}_{1}\left(X\right)$,
choose an auxiliary boundary defining function $r_{H}$ for $H$ in
$X$, and call $r=\prod_{H\in\mathcal{M}_{1}\left(X\right)}r_{H}$
a total boundary defining function for $X$. 
\begin{defn}
	We denote by $\mathcal{A}^{\mathfrak{m}}\left(X\right)$ the space
	$r^{-\mathfrak{m}}\mathcal{A}^{0}\left(X\right)$, where $\mathcal{A}^{0}\left(X\right)$
	is the space of extendible distributions $u$ on $X$ such that $V_{1}\cdots V_{k}u\in L^{\infty}\left(X\right)$
	for every $V_{1},...,V_{k}\in\mathcal{V}_{b}\left(X\right)$. The
	elements of $\mathcal{A}^{\mathfrak{m}}\left(X\right)$ are called
	\emph{conormal functions of order $\mathfrak{m}$.}
\end{defn}

The space $\mathcal{A}^{0}\left(X\right)$ is equipped with a natural
Fréchet topology, defined as follows. We choose a finite family of $b$-vector
fields $V_{1},...,V_{K}$ which generate $\mathcal{V}_{b}\left(X\right)$
pointwise (such a finite family exists because $X$ is compact). Then
$\mathcal{A}^{0}\left(X\right)$ is topologized by the seminorms
\[
\sup\left|V_{i_{1}}\cdots V_{i_{l}}u\right|.
\]
The space $\mathcal{A}^{\mathfrak{m}}\left(X\right)$ is topologized
by the isomorphism $\mathcal{A}^{\mathfrak{m}}\left(X\right)\to\mathcal{A}^{0}\left(X\right)$
given by $u\mapsto r^{-\mathfrak{m}}u$. One can prove that the Fréchet
topology on $\mathcal{A}^{\mathfrak{m}}\left(X\right)$ does not depend
on the choice of the vector fields $V_{1},...,V_{K}$ and the boundary
defining functions $r_{H}$.
\begin{defn}
	\label{def:index-set}An \emph{index set} is a subset $\mathcal{E}\subseteq\mathbb{C}\times\mathbb{N}$
	with the following properties:
	\begin{enumerate}
		\item if $\left(\alpha,l\right)\in\mathcal{E}$, then $\left(\alpha+k,l'\right)\in\mathcal{E}$
		for every $k\in\mathbb{N}$ and every $l'\in\left\{ 0,...,l\right\} $;
		\item for every $M\in\mathbb{R}$, the set $\left\{ \left(\alpha,l\right)\in\mathcal{E}:\Re\left(\alpha\right)<M\right\} $
		is finite.
	\end{enumerate}
	An \emph{index family} for $X$ is a family $\mathcal{E}=\left\{ \mathcal{E}_{H}\right\} _{H\in\mathcal{M}_{1}\left(X\right)}$
	of index sets.
\end{defn}

\begin{defn}
	\label{def:leading-set-index-set}Let $\mathcal{I}$ be an index set.
	The \emph{leading set }of $\mathcal{I}$ is the space of pairs $\left(\alpha,l\right)\in\mathcal{I}$
	with $\inf\Re\left(\mathcal{I}\right)\leq\Re\left(\alpha\right)<\inf\Re\left(\mathcal{I}\right)+1$.
\end{defn}

\begin{rem}
	Often, when discussing index sets, we denote by $\alpha$ the pairs
	$\left(\alpha,0\right)$. In particular,
	we will often write expressions such as $\left[\mathcal{E}\right]=\mu$
	(meaning $\left[\mathcal{E}\right]=\left\{ \left(\mu,0\right)\right\} $)
	and $\mathcal{E}\backslash\mu$ (meaning $\mathcal{E}\backslash\left\{ \left(\mu,0\right)\right\} $).
\end{rem}

Fix an index family $\mathcal{E}$ for $X$. Now, choose for every
boundary hyperface $H$ a vector field $V_{H}$ transverse to $H$,
inward-pointing, and tangent to all the other boundary hyperfaces.
For every $M\in\mathbb{R}$, define
\[
P_{H,M}=\prod_{\begin{smallmatrix}\left(\alpha,l\right)\in\mathcal{E}_{H}\\
		\Re\left(\alpha\right)\leq M
\end{smallmatrix}}\left(r_{H}V_{H}-\alpha\right).
\]

\begin{defn}
	\label{def:polyhomogeneous-functions}We denote by $\mathcal{A}_{\phg}^{\mathcal{E}}\left(X\right)$
	the space of extendible distributions $u$ on $X$ for which there
	exists a multi-weight $\mathfrak{m}$ for $X$ such that:
	\begin{enumerate}
		\item $u\in\mathcal{A}^{\mathfrak{m}}\left(X\right)$;
		\item for every $M\in\mathbb{R}$, we have $P_{H,M}u\in\mathcal{A}^{\mathfrak{m}\left(H,M\right)}\left(X\right)$
		where $\mathfrak{m}\left(H,M\right)$ is the new multi-weight obtained
		from $\mathfrak{m}$ obtained by replacing $\mathfrak{m}_{H}$ with
		$M$.
	\end{enumerate}
	The elements of $\mathcal{A}_{\phg}^{\mathcal{E}}\left(X\right)$
	are called \emph{polyhomogeneous functions with index sets} $\mathcal{E}$\emph{.}
\end{defn}

The space $\mathcal{A}_{\phg}^{\mathcal{E}}\left(X\right)$ is equipped
with a natural Fréchet topology, induced by the seminorms $\left|\left|P_{H,N}u\right|\right|$
for $H\in\mathcal{M}_{1}\left(X\right)$ and $N\in\mathbb{N}$, where
$\left|\left|\cdot\right|\right|$ ranges among a family of seminorms
for $\mathcal{A}^{\mathfrak{m}\left(H,N\right)}\left(X\right)$. One
can check that the topology does not depend on the auxiliary choices
made.

Given a Fréchet space $F$, we can define spaces $\mathcal{A}^{\mathfrak{m}}\left(X;F\right)$,
$\mathcal{A}_{\phg}^{\mathcal{E}}\left(X;F\right)$ of $F$-valued
conormal and polyhomogeneous functions. The definitions are exactly
as above, except that we replace the norm $\left|\cdot\right|$ in
$\mathbb{R}$ with a family of seminorms for $F$. Given another compact
manifold with corners $Y$, one can then easily check that we have
canonical identifications
\begin{align*}
	\mathcal{A}_{\phg}^{\left(\mathcal{E},\mathcal{F}\right)}\left(X\times Y\right) & =\mathcal{A}_{\phg}^{\mathcal{E}}\left(X;\mathcal{A}_{\phg}^{\mathcal{F}}\left(Y\right)\right)\\
	\mathcal{B}^{\left(\mathcal{E},\mathfrak{m}\right)}\left(X\times Y\right) & =\mathcal{A}_{\phg}^{\mathcal{E}}\left(X;\mathcal{A}^{\mathfrak{m}}\left(Y\right)\right).
\end{align*}
In the second line, the space $\mathcal{B}^{\left(\mathcal{E},\mathfrak{m}\right)}\left(X\times Y\right)$
is the hybrid space of conormal functions on $X\times Y$, polyhomogeneous
with index set $\mathcal{E}_{H}$ at $H\times Y$ for every $H\in\mathcal{M}_{1}\left(X\right)$,
and conormal of order $\mathfrak{m}_{H}$ at $X\times H$ for every
$H\in\mathcal{M}_{1}\left(Y\right)$. Now, given two Fréchet spaces
$F,G$, we denote by $F\otimes G$ their projective tensor product
(cf.\ \S45 of \cite{Treves}). The topology on $F\otimes G$ is
the strongest topology that makes the canonical bilinear map $F\times G\to F\otimes G$
continuous. We call $F\hat{\otimes}G$ the completed projective tensor
product: every element of $F\hat{\otimes}G$ can be written as a convergent
sum $\sum_{j=0}^{\infty}\lambda_{j}f_{j}g_{j}$, where $\left\{ \lambda_{j}\right\} $
is an absolutely summable scalar sequence, and $\left\{ f_{j}\right\} ,\left\{ g_{j}\right\} $
are null sequences in $F$ and $G$. The space $F\hat{\otimes}G$
satisfies the following universal property: every continuous bilinear
map $F\times G\to H$ factors uniquely through a continuous linear
map $F\hat{\otimes}G\to H$. The proof of Proposition 5.7 of \cite{LauterSeiler}
implies that there are canonical isomorphisms
\begin{align*}
	\mathcal{A}_{\phg}^{\mathcal{E}}\left(X;\mathcal{A}_{\phg}^{\mathcal{F}}\left(Y\right)\right) & =\mathcal{A}_{\phg}^{\mathcal{E}}\left(X\right)\hat{\otimes}\mathcal{A}_{\phg}^{\mathcal{F}}\left(Y\right)\\
	\mathcal{A}_{\phg}^{\mathcal{E}}\left(X;\mathcal{A}^{\mathfrak{m}}\left(Y\right)\right) & =\mathcal{A}_{\phg}^{\mathcal{E}}\left(X\right)\hat{\otimes}\mathcal{A}^{\mathcal{\mathfrak{m}}}\left(Y\right).
\end{align*}

\subsection{\label{subsec:Pull-back-and-Push-forward}Pull-back and push-forward
	Theorems}

Let $X,Y$ be two manifolds with corners. Fix boundary defining functions
$\left\{ r_{H}\right\} _{H\in\mathcal{M}_{1}\left(X\right)}$ and
$\left\{ \rho_{G}\right\} _{G\in\mathcal{M}_{1}\left(Y\right)}$.
A smooth map $f:X\to Y$ between compact manifolds with corners is
an \emph{interior} \emph{$b$-map} if, for every $G\in\mathcal{M}_{1}\left(Y\right)$,
we have
\[
f^{*}\rho_{G}=a_{G}\prod_{H\in\mathcal{M}_{1}\left(X\right)}r_{H}^{e_{f}\left(G,H\right)}
\]
for some $a_{G}\in C^{\infty}\left(X\right)$ positive and $e_{f}\left(G,H\right)\in\mathbb{N}$.
These numbers are called the \emph{boundary exponents} of $f$, and
$e_{f}$ is called the \emph{boundary matrix} of $f$.
\begin{thm}
	(Melrose's pull-back Theorem) Let $f:X\to Y$ be an interior $b$-map,
	and let \emph{$\mathcal{F}$} be an index family for $Y$. Define
	an index family $f^{\flat}\mathcal{F}$ on $X$ by
	\[
	\left(f^{\flat}\mathcal{F}\right)_{H}=\mathbb{N}+\sum_{f\left(H\right)\subseteq G}e_{f}\left(G,H\right)\mathcal{F}_{G}.
	\]
	Then the pull-back map $f^{*}:\dot{C}^{\infty}\left(X\right)\to C^{\infty}\left(Y\right)$
	extends by continuity to a continuous linear map
	\[
	f^{*}:\mathcal{A}_{\phg}^{\mathcal{F}}\left(Y\right)\to\mathcal{A}_{\phg}^{f^{\flat}\mathcal{F}}\left(X\right).
	\]
\end{thm}

The space $\mathcal{V}_{b}\left(X\right)$ of $b$-vector fields determines,
via the Serre--Swan Theorem, a vector bundle $^{b}TX$ of which $\mathcal{V}_{b}\left(X\right)$
is the module of sections. The inclusion $\mathcal{V}_{b}\left(X\right)\to\mathcal{V}\left(X\right)$
determines a bundle map $\#:{^{b}TX}\to TX$, which is a natural isomorphism
on the interior. If $f:X\to Y$ is an interior $b$-map, then the
differential of $f$ extends to a map $^{b}df:{^{b}TX}\to{^{b}TY}$
called the $b$\emph{-differential}. $f$ is called a $b$\emph{-submersion}
if $^{b}df$ is surjective. $f$ is called a \emph{$b$-fibration}
if:
\begin{enumerate}
	\item it is a surjective $b$-submersion;
	\item it does not map any boundary hyperface of $X$ to a corner of $Y$.
\end{enumerate}
It is often easier to check that a surjective interior $b$-map $f$
is a $b$-fibration via the following equivalent characterization:
\begin{enumerate}
	\item if $K$ is any boundary face of $X$, $\text{codim}f\left(K\right)\leq\text{codim}K$;
	\item the restriction of $f$ to $K^{\circ}$ for any boundary face, is
	a fibration onto $f\left(K\right)^{\circ}$.
\end{enumerate}
Here a \emph{boundary face }of $X$ is a disjoint union of connected
components of intersections of boundary hyperfaces.

Denote by $^{b}\mathcal{D}_{X}^{1}$ the bundle of singular densities
of the form $r^{-1}\omega$, where $\omega$ is a smooth density on
$X$ and $r$ is a total boundary defining function for $X$.
\begin{thm}
	(Melrose's push-forward Theorem) Let $f:X\to Y$ be a $b$-fibration,
	and let $\mathcal{E}$ be an index family for $X$. Define an index
	family $f_{\flat}\mathcal{E}$ on $Y$ by
	\[
	\left(f_{\flat}\mathcal{E}\right)_{G}=\overline{\bigcup}_{f\left(H\right)\subseteq G}\frac{\mathcal{E}_{H}}{e\left(G,H\right)}.
	\]
	Here the operation $\eunion$ is called \emph{extended union}: for
	a family $\mathcal{E}_{1},...,\mathcal{E}_{k}$ of index sets, $\overline{\bigcup}_{i}\mathcal{E}_{i}$
	is the index set generated by the union $\bigcup_{i}\mathcal{E}_{i}$
	and the pairs $\left(z,p_{1}+\cdots+p_{k}+1\right)$ such that $\left(z,p_{i}\right)\in\mathcal{E}_{i}$.
	If, for every $H\in\mathcal{M}_{1}\left(X\right)$ such that $f\left(H\right)\cap Y^{\circ}\not=\emptyset$
	we have $\Re\left(\mathcal{E}_{H}\right)>0$, then the push-forward
	map $f_{*}:\dot{C}^{\infty}\left(X;{^{b}\mathcal{D}_{X}^{1}}\right)\to C^{-\infty}\left(Y;{^{b}\mathcal{D}_{Y}^{1}}\right)$
	extends to a continuous linear map
	\[
	f_{*}:\mathcal{A}_{\phg}^{\mathcal{E}}\left(X;{^{b}\mathcal{D}_{X}^{1}}\right)\to\mathcal{A}_{\phg}^{f_{\flat}\mathcal{E}}\left(Y;{^{b}\mathcal{D}_{Y}^{1}}\right).
	\]
\end{thm}

\subsection{\label{subsec:Technical-details-on-proofs}Invariance under changes
	of coordinates}

In this last sub-section, we discuss coordinate invariance of the
local calculus introduced in \S\ref{sec:The-symbolic-0-calculus}.
For simplicity, we only discuss in detail the $0$-Poisson case, since
the other cases are similar.

The proof is very similar to the standard coordinate invariance proof
for pseudodifferential operators on $\mathbb{R}^{n}$. We sketch the
argument following closely the proof of Theorem 5.2 of \cite{HintzMicrolocal}.
Given a bounded open subset $\Omega$ of $\mathbb{R}^{n}$, denote
by $\hat{\Psi}_{0\po,c}^{-\infty,\mathcal{E}}\left(\Omega,[0,\varepsilon)\times\Omega\right)$
the subspace of $\hat{\Psi}_{0\po,\mathcal{S}}^{-\infty,\mathcal{E}}\left(\mathbb{R}^{n},\mathbb{R}_{1}^{n+1}\right)$
consisting of operators compactly supported on $[0,\varepsilon)\times\Omega\times\Omega$.
We similarly define $\Psi_{\po,c}^{-\infty,\mathcal{E}_{\of}}\left(\Omega,[0,\varepsilon)\times\Omega\right)$.
\begin{prop}
	\label{prop:poisson-diffeo-invariance}Let $\Omega,\Omega'$ be bounded
	open subsets of $\mathbb{R}^{n}$, and let $\varphi:\Omega'\to\Omega$
	be a diffeomorphism. Let $B\in\hat{\Psi}_{0\po,c}^{-\infty,\mathcal{E}}\left(\Omega,[0,\varepsilon)\times\Omega\right)$,
	and let $B_{\varphi}$ be the operator defined by $\left(\id\times\varphi\right)^{*}\circ B\circ\left(\varphi^{-1}\right)^{*}$.
	Then $B_{\varphi}\in\hat{\Psi}_{0\po,c}^{-\infty,\mathcal{E}}\left(\Omega',[0,\varepsilon)\times\Omega'\right)$.
	Moreover, if $B=\Op_{L}^{\po}\left(b\right)$ for $b\in S_{0\po,\mathcal{S}}^{-\infty,\mathcal{E}}\left(\mathbb{R}^{n};\mathbb{R}_{1}^{n+1}\right)$,
	then $B_{\varphi}=\Op_{L}^{\po}\left(b_{\varphi}\right)$ where $b_{\varphi}\in S_{0\po,\mathcal{S}}^{-\infty,\mathcal{E}}\left(\mathbb{R}^{n};\mathbb{R}_{1}^{n+1}\right)$
	is determined modulo $S_{\po,\mathcal{S}}^{-\infty,\mathcal{E}_{\of}}\left(\mathbb{R}^{n};\mathbb{R}_{1}^{n+1}\right)$
	by the asymptotic expansion
	\[
	b_{\varphi}\sim\sum_{\alpha}\frac{1}{\alpha!}\left(D_{\eta}^{\alpha}b\right)\left(\varphi\left(y\right);x,\J\left(\varphi^{-1}\right)^{T}\left(\varphi\left(y\right)\right)\eta\right)\Psi_{\alpha}\left(y;\eta\right),
	\]
	where $\Psi_{\alpha}\left(y;\eta\right)$ is a polynomial in $\eta$
	with coefficients in $C^{\infty}\left(\Omega'\right)$ dependent only
	on $\varphi$, and $\Psi_{0}=1$.
\end{prop}

\begin{proof}
	(sketch) Fix a cutoff function $\chi_{\epsilon}$ supported on $\left|Y\right|<2\epsilon$
	and equal to $1$ on $\left|Y\right|<\epsilon$. Define 
	\[
	b_{\epsilon}=\chi_{\epsilon}\left(y-\tilde{y}\right)b\left(y;x,\eta\right)\in S_{0\po,\mathcal{S}}^{-\infty,\mathcal{E}}\left(\mathbb{R}^{2n};\mathbb{R}_{1}^{n+1}\right).
	\]
	From the proof of the previous proposition, we know that
	\[
	B-B_{\epsilon}=:R_{\epsilon}\in\Psi_{\po,c}^{-\infty,\mathcal{E}_{\of}}\left(\mathbb{R}^{n},\mathbb{R}_{1}^{n+1}\right).
	\]
	Let $R_{\epsilon}=R_{\epsilon}\left(x,y,\tilde{y}\right)d\tilde{y}$.
	Then $R_{\epsilon}$ has compact support in $\mathbb{R}_{1}^{1}\times\Omega\times\Omega$
	again, and
	\[
	\left(R_{\epsilon}\right)_{\varphi}=R_{\epsilon}\left(x,\varphi\left(y\right),\varphi\left(\tilde{y}\right)\right)\left|\det\J\left(\varphi\right)\left(\tilde{y}\right)\right|d\tilde{y}
	\]
	where $\J\left(\varphi\right)$ is the Jacobian matrix of $\varphi$.
	Since $\varphi$ is a diffeomorphism, $\det\J\left(\varphi\right)\left(\tilde{y}\right)$
	does not change sign and therefore $\left|\det\J\left(\varphi\right)\left(\tilde{y}\right)\right|$
	is a smooth factor. Moreover, $R_{\epsilon}\left(x,\varphi\left(y\right),\varphi\left(\tilde{y}\right)\right)$
	is still smooth, compactly supported on $\mathbb{R}_{1}^{1}\times\Omega'\times\Omega'$,
	and polyhomogeneous at $x=0$ with index set $\mathcal{E}_{\of}$.
	It follows that $\left(R_{\epsilon}\right)_{\varphi}\in\Psi_{\po,c}^{-\infty,\mathcal{E}_{\of}}\left(\Omega,[0,\varepsilon)\times\Omega\right)$.
	Therefore, it suffices to prove that $\left(B_{\epsilon}\right)_{\varphi}\in\hat{\Psi}_{0\po,c}^{-\infty,\mathcal{E}}\left(\Omega,[0,\varepsilon)\times\Omega\right)$.
	The appropriate choice of $\epsilon$ is determined by the properties
	of $\varphi$, as explained below.
	
	The Schwartz kernel of $B$ is $K\left(y;x,y-\tilde{y}\right)d\tilde{y}$,
	where $K\left(y;x,Y\right)$ is the inverse Fourier transform of $b\left(y;x,\eta\right)$.
	Therefore, the Schwartz kernel of $\left(B_{\epsilon}\right)_{\varphi}$
	is
	\begin{align*}
		\chi_{\epsilon}\left(\varphi\left(y\right)-\varphi\left(\tilde{y}\right)\right)K\left(\varphi\left(y\right);x,\varphi\left(y\right)-\varphi\left(\tilde{y}\right)\right)\left|\det\J\left(\varphi\right)\left(\tilde{y}\right)\right|d\tilde{y}\\
		=\frac{1}{\left(2\pi\right)^{n}}\int e^{i\left(\varphi\left(y\right)-\varphi\left(\tilde{y}\right)\right)\eta}b_{\epsilon}\left(\varphi\left(y\right),\varphi\left(\tilde{y}\right);x,\eta\right)\left|\det\J\left(\varphi\right)\left(\tilde{y}\right)\right|d\eta d\tilde{y} & .
	\end{align*}
	Now, by Taylor's Theorem, we can write
	\[
	\varphi\left(y\right)-\varphi\left(\tilde{y}\right)=\Phi\left(y,\tilde{y}\right)\left(y-\tilde{y}\right),
	\]
	where $\Phi\left(y,\tilde{y}\right)$ is a smooth $n\times n$ matrix
	defined on $\Omega'\times\Omega'$, and we are thinking of $y-\tilde{y}$
	as a column vector. We have
	\[
	\left(B_{\epsilon}\right)_{\varphi}=\left[\frac{1}{\left(2\pi\right)^{n}}\int e^{i\left(y-\tilde{y}\right)\cdot\Phi^{T}\left(y,\tilde{y}\right)\eta}b_{\epsilon}\left(\varphi\left(y\right),\varphi\left(\tilde{y}\right);x,\eta\right)\left|\det\J\left(\varphi\right)\left(\tilde{y}\right)\right|d\eta\right]d\tilde{y}.
	\]
	Since $\Phi\left(y,y\right)=\J\left(\varphi\right)\left(y\right)$
	and $\varphi$ is a diffeomorphism $\Omega'\to\Omega$, $J\left(\varphi\right)\left(y\right)$
	is invertible for every $y\in\Omega'$. Let $Z$ be a compact subset
	of $\Omega$ such that the support of $B$ is contained in $\mathbb{R}_{1}^{1}\times Z\times Z$,
	and call $Z'=\varphi^{-1}\left(Z\right)$. Then $\Phi^{T}\left(y,\tilde{y}\right)$
	has a smooth inverse in $Z'\times Z'\cap\left|y-\tilde{y}\right|<\epsilon$
	for $\epsilon>0$ sufficiently small, and we can write
	\[
	\left(B_{\epsilon}\right)_{\varphi}=\Op^{\po}\left(\left(b_{\epsilon}\right)_{\varphi}\right)
	\]
	with
	\[
	\left(b_{\epsilon}\right)_{\varphi}=b_{\epsilon}\left(\varphi\left(y\right),\varphi\left(\tilde{y}\right);x,\left(\Phi^{T}\left(y,\tilde{y}\right)\right)^{-1}\eta\right)\left|\det\J\left(\varphi\right)\left(\tilde{y}\right)\right|\left|\det\left(\Phi^{T}\left(y,\tilde{y}\right)\right)^{-1}\right|.
	\]
	We now need to prove that $\left(b_{\epsilon}\right)_{\varphi}\in S_{0\po,\mathcal{S}}^{-\infty,\mathcal{E}}\left(\mathbb{R}^{2n};\mathbb{R}_{1}^{n+1}\right)$.
	Since $b_{\epsilon}\equiv0$ for $y,\tilde{y}$ away from a compact
	set, we can without loss of generality assume that $\varphi:\Omega'\to\Omega$
	extends to a diffeomorphism $\overline{\mathbb{R}^{n}}\to\overline{\mathbb{R}^{n}}$.
	Similarly, we can assume that $\left(\Phi^{T}\left(y,\tilde{y}\right)\right)^{-1}$
	extends to a smooth map $G:\overline{\mathbb{R}^{n}}\times\overline{\mathbb{R}^{n}}\to\GL\left(n,\mathbb{R}^{n}\right)$.
	It is then sufficient to prove that the multilinear map
	\begin{align*}
		\mathcal{S}\left(\mathbb{R}^{n}\right)\times\mathcal{S}\left(\mathbb{R}^{n}\right)\times\mathcal{A}_{\phg}^{\left(\mathcal{E}_{\of},\mathcal{E}_{\ff},\infty,\infty\right)}\left(\hat{P}_{0}^{2}\right) & \to S_{0\po,\mathcal{S}}^{-\infty,\mathcal{E}}\left(\mathbb{R}^{2n};\mathbb{R}_{1}^{n+1}\right)\\
		\left(c_{1}\left(y\right),c_{2}\left(\tilde{y}\right),b\left(x,\eta\right)\right) & \mapsto c_{1}\left(\varphi\left(y\right)\right)c_{2}\left(\varphi\left(\tilde{y}\right)\right)b\left(x,G\left(y,\tilde{y}\right)\eta\right)
	\end{align*}
	is well-defined and continuous. It suffices to prove that the map
	\begin{align*}
		\mathcal{A}_{\phg}^{\left(\mathcal{E}_{\of},\mathcal{E}_{\ff},\infty,\infty\right)}\left(\hat{P}_{0}^{2}\right) & \to C^{\infty}\left(\overline{\mathbb{R}^{n}}\times\overline{\mathbb{R}^{n}}\right)\hat{\otimes}\mathcal{A}_{\phg}^{\left(\mathcal{E}_{\of},\mathcal{E}_{\ff},\infty,\infty\right)}\left(\hat{P}_{0}^{2}\right)\\
		b\left(x,\eta\right) & \mapsto b\left(x,G\left(y,\tilde{y}\right)\eta\right)
	\end{align*}
	is well-defined and continuous. It is proved in \cite{MelroseCorners}
	that every $A\in\GL\left(n,\mathbb{R}\right)$ extends from $\mathbb{R}^{n}$
	to a diffeomorphism of $\overline{\mathbb{R}^{n}}$. It follows that
	the map
	\begin{align*}
		\gamma:\overline{\mathbb{R}^{n}}\times\overline{\mathbb{R}^{n}}\times\hat{P}^{2} & \to\hat{P}^{2}\\
		\left(y,\tilde{y},x,\eta\right) & \mapsto\left(x,G\left(y,\tilde{y}\right)\eta\right)
	\end{align*}
	is a smooth $b$-map of manifolds with corners. The map $\gamma$
	sends the locus $\left\{ x=0,\left|\eta\right|=\infty\right\} $ in
	$\overline{\mathbb{R}^{n}}\times\overline{\mathbb{R}^{n}}\times\hat{P}^{2}$
	to the locus $\left\{ x=0,\left|\eta\right|=\infty\right\} $ in $\hat{P}^{2}$,
	so it lifts to a smooth $b$-map
	\[
	\gamma:\overline{\mathbb{R}^{n}}\times\overline{\mathbb{R}^{n}}\times\hat{P}_{0}^{2}\to\hat{P}_{0}^{2}.
	\]
	The function $b\left(x,G\left(y,\tilde{y}\right)\eta\right)$ is then
	simply the pull-back $\left(\gamma^{*}b\right)\left(y,\tilde{y},x,\eta\right)$,
	and by the pull-back Theorem $\gamma^{*}$ is continuous as a map
	\[
	\gamma^{*}:\mathcal{A}_{\phg}^{\left(\mathcal{E}_{\of},\mathcal{E}_{\ff},\infty,\infty\right)}\left(\hat{P}_{0}^{2}\right)\to C^{\infty}\left(\overline{\mathbb{R}^{n}}\times\overline{\mathbb{R}^{n}}\right)\hat{\otimes}\mathcal{A}_{\phg}^{\left(\mathcal{E}_{\of},\mathcal{E}_{\ff},\infty,\infty\right)}\left(\hat{P}_{0}^{2}\right).
	\]
	The statement concerning the asymptotic expansion follows by taking
	the left reduction of $\left(b_{\epsilon}\right)_{\varphi}$ and its
	asymptotic expansion.
\end{proof}
\bibliography{allpapers}{}

@article{GJMS,
	author = {Graham, C. Robin and Jenne, Ralph and Mason, Lionel J. and Sparling, George A. J.},
	title = {Conformally invariant powers of the {Laplacian}. {I}: {Existence}},
	fjournal = {Journal of the London Mathematical Society. Second Series},
	journal = {J. Lond. Math. Soc., II. Ser.},
	issn = {0024-6107},
	volume = {46},
	number = {3},
	pages = {557--565},
	year = {1992},
	language = {English},
	doi = {10.1112/jlms/s2-46.3.557},
	zbMATH = {4198939},
	Zbl = {0726.53010}
}

@misc{UsulaNahmPoles,
	author = {Marco Usula},
	title = {Nahm {Poles} and 0-{Instantons}},
	year = {2026},
	howpublished = {Preprint, {arXiv}:2603.13805 [math.{DG}] (2026)},
	url = {https://arxiv.org/abs/2603.13805},
	arXiv = {arXiv:2603.13805}
}

@book {MazzeoPhD,
	AUTHOR = {Mazzeo, Rafe R.},
	TITLE = {{Hodge cohomology of negatively curved manifolds}},
	NOTE = {Thesis (Ph.D.)--Massachusetts Institute of Technology},
	PUBLISHER = {ProQuest LLC, Ann Arbor, MI},
	YEAR = {1986},
	PAGES = {(no paging)},
	MRCLASS = {Thesis},
	MRNUMBER = {2941112},
	URL =
	{http://gateway.proquest.com/openurl?url_ver=Z39.88-2004&rft_val_fmt=info:ofi/fmt:kev:mtx:dissertation&res_dat=xri:pqdiss&rft_dat=xri:pqdiss:0374406},
}

@article{MazzeoHodge,
	author = {Mazzeo, Rafe},
	title = {The {Hodge} cohomology of a conformally compact metric},
	fjournal = {Journal of Differential Geometry},
	journal = {J. Differ. Geom.},
	issn = {0022-040X},
	volume = {28},
	number = {2},
	pages = {309--339},
	year = {1988},
	language = {English},
	doi = {10.4310/jdg/1214442281},
	zbMATH = {4071836},
	Zbl = {0656.53042}
}

@article{MazzeoEdgeI,
	author = {Mazzeo, Rafe},
	title = {Elliptic theory of differential edge operators. {I}},
	fjournal = {Communications in Partial Differential Equations},
	journal = {Commun. Partial Differ. Equations},
	issn = {0360-5302},
	volume = {16},
	number = {10},
	pages = {1615--1664},
	year = {1991},
	language = {English},
	doi = {10.1080/03605309108820815},
	zbMATH = {28820},
	Zbl = {0745.58045}
}

@article{MazzeoMelroseResolvent,
	author = {Mazzeo, Rafe R. and Melrose, Richard B.},
	title = {Meromorphic extension of the resolvent on complete spaces with asymptotically constant negative curvature},
	fjournal = {Journal of Functional Analysis},
	journal = {J. Funct. Anal.},
	issn = {0022-1236},
	volume = {75},
	pages = {260--310},
	year = {1987},
	language = {English},
	doi = {10.1016/0022-1236(87)90097-8},
	zbMATH = {4036774},
	Zbl = {0636.58034}
}

@article{MazzeoEdgeII,
	author = {Mazzeo, Rafe and Vertman, Boris},
	title = {Elliptic theory of differential edge operators. {II}: {Boundary} value problems},
	fjournal = {Indiana University Mathematics Journal},
	journal = {Indiana Univ. Math. J.},
	issn = {0022-2518},
	volume = {63},
	number = {6},
	pages = {1911--1955},
	year = {2014},
	language = {English},
	doi = {10.1512/iumj.2014.63.5435},
	url = {www.iumj.indiana.edu/IUMJ/ABS/2014/5435},
	zbMATH = {6405945},
	Zbl = {1311.58014}
}

@book{MelroseCorners,
	title={{Differential analysis on manifolds with corners}},
	author={Melrose, Richard B.},
	year={1996},
	publisher={In preparation}
}

@article{EpsteinMelroseShrinking,
	title={{Shrinking tubes and the $\bar \partial$-{N}eumann problem}},
	author={Epstein, Charles L. and Melrose, Richard B.},
	journal={preprint},
	year={1990}
}

@book{MelroseMendozaTotallyCharacteristic,
	title={Elliptic operators of totally characteristic type},
	author={Melrose, Richard and Mendoza, Gerardo},
	year={1983},
	publisher={Mathematical Sciences Research Institute New York}
}

@book{MelroseAPS,
	author = {Melrose, Richard B.},
	title = {The {Atiyah}-{Patodi}-{Singer} index theorem},
	fseries = {Research Notes in Mathematics},
	series = {Res. Notes Math.},
	volume = {4},
	isbn = {1-56881-002-4},
	year = {1993},
	publisher = {Wellesley, MA: A. K. Peters, Ltd.},
	language = {English},
	zbMATH = {496294},
	Zbl = {0796.58050}
}

@book{MelroseMicrolocal,
	title={{An introduction to microlocal analysis}},
	author={Melrose, Richard B. and Uhlmann, Gunther},
	year={2008},
	publisher={Department of Mathematics, Massachusetts Institute of Technology}
}

@book{HintzMicrolocal,
	author = {Hintz, Peter},
	title = {An introduction to microlocal analysis},
	fseries = {Graduate Texts in Mathematics},
	series = {Grad. Texts Math.},
	issn = {0072-5285},
	volume = {304},
	isbn = {978-3-031-90705-0; 978-3-031-90708-1; 978-3-031-90706-7},
	year = {2025},
	publisher = {Cham: Springer},
	language = {English},
	doi = {10.1007/978-3-031-90706-7},
	zbMATH = {8014458}
}

@article{KrainerMendozaVariableOrders,
	author = {Krainer, Thomas and Mendoza, Gerardo A.},
	title = {Elliptic systems of variable order},
	fjournal = {Revista Matem{\'a}tica Iberoamericana},
	journal = {Rev. Mat. Iberoam.},
	issn = {0213-2230},
	volume = {31},
	number = {1},
	pages = {127--160},
	year = {2015},
	language = {English},
	doi = {10.4171/RMI/829},
	zbMATH = {6427271},
	Zbl = {1317.58028}
}

@misc{HintzDilations,
	author = {Peter Hintz},
	title = {Microlocal analysis of operators with asymptotic translation- and dilation-invariances},
	year = {2023},
	howpublished = {Preprint, {arXiv}:2302.13803 [math.{AP}] (2023)},
	url = {https://arxiv.org/abs/2302.13803},
	arXiv = {arXiv:2302.13803}
}

@article{Hintz0calculus,
	author = {Hintz, Peter},
	title = {Elliptic parametrices in the 0-calculus of {Mazzeo} and {Melrose}},
	fjournal = {Pure and Applied Analysis},
	journal = {Pure Appl. Anal.},
	issn = {2578-5893},
	volume = {5},
	number = {3},
	pages = {729--766},
	year = {2023},
	language = {English},
	doi = {10.2140/paa.2023.5.729},
	zbMATH = {7741457},
	Zbl = {1522.35267}
}

@book{Lauter,
	author = {Lauter, Robert},
	title = {Pseudodifferential analysis on conformally compact spaces},
	fseries = {Memoirs of the American Mathematical Society},
	series = {Mem. Am. Math. Soc.},
	issn = {0065-9266},
	volume = {777},
	isbn = {978-0-8218-3272-1; 978-1-4704-0375-1},
	year = {2003},
	publisher = {Providence, RI: American Mathematical Society (AMS)},
	language = {English},
	doi = {10.1090/memo/0777},
	zbMATH = {1911384},
	Zbl = {1024.58009}
}

@incollection{LauterSeiler,
	author = {Lauter, R. and Seiler, J.},
	title = {Pseudodifferential analysis on manifolds with boundary -- a comparison of {{\(b\)}}-calculus and cone algebra},
	booktitle = {Approaches to singular analysis. Based on the workshop, Berlin, Germany, April 8--10, 1999},
	isbn = {3-7643-6518-8},
	pages = {131--166},
	year = {2001},
	publisher = {Basel: Birkh{\"a}user},
	language = {English},
	zbMATH = {1617653},
	Zbl = {0998.58018}
}

@book{UsulaPhD,
	title={{0-elliptic boundary value problems}},
	author={Usula, Marco},
	year={2022},
	MRCLASS = {Thesis},
	NOTE = {Thesis (Ph.D.) KU Leuven, Faculty of Science}
}

@book{Schulze1991PseudoDifferentialOO,
	author = {Schulze, Bert-Wolfgang},
	title = {Pseudo-differential operators on manifolds with singularities},
	fseries = {Studies in Mathematics and its Applications},
	series = {Stud. Math. Appl.},
	issn = {0168-2024},
	volume = {24},
	isbn = {0-444-88137-9},
	year = {1991},
	publisher = {Amsterdam etc.: North-Holland},
	language = {English},
	zbMATH = {51573},
	Zbl = {0747.58003}
}

@incollection{schultze1994pseudo,
	author = {Schulze, Bert-Wolfgang},
	title = {Pseudo-differential operators, ellipticity, and asymptotics on manifolds with edges},
	booktitle = {Partial differential equations. Models in physics and biology. Contributions to the conference, held in Han-sur-Lesse, Belgium, in December 1993},
	isbn = {3-05-501657-2},
	pages = {290--328},
	year = {1994},
	publisher = {Berlin: Akademie Verlag},
	language = {English},
	zbMATH = {739065},
	Zbl = {0827.58053}
}

@article{schulze2002operators,
	author = {Schulze, Bert-Wolfgang},
	title = {Operators with symbol hierarchies and iterated asymptotics.},
	fjournal = {Publications of the Research Institute for Mathematical Sciences, Kyoto University},
	journal = {Publ. Res. Inst. Math. Sci.},
	issn = {0034-5318},
	volume = {38},
	number = {4},
	pages = {735--802},
	year = {2002},
	language = {English},
	doi = {10.2977/prims/1145476196},
	zbMATH = {1902209},
	Zbl = {1051.58011}
}

@article{krainer2016boundary,
	author = {Krainer, Thomas and Mendoza, Gerardo A.},
	title = {Boundary value problems for first order elliptic wedge operators},
	fjournal = {American Journal of Mathematics},
	journal = {Am. J. Math.},
	issn = {0002-9327},
	volume = {138},
	number = {3},
	pages = {585--656},
	year = {2016},
	language = {English},
	doi = {10.1353/ajm.2016.0022},
	url = {muse.jhu.edu/article/619397/pdf},
	zbMATH = {6603602},
	Zbl = {1346.58009}
}

@incollection{KrainerMendozaSurvey,
	author = {Krainer, Thomas and Mendoza, Gerardo A.},
	title = {Boundary value problems for elliptic wedge operators: the first-order case},
	booktitle = {Elliptic and parabolic equations. Selected papers based on the presentations at the workshop, Hannover, Germany, September 10--12, 2013},
	isbn = {978-3-319-12546-6; 978-3-319-12547-3},
	pages = {209--232},
	year = {2015},
	publisher = {Cham: Springer},
	language = {English},
	doi = {10.1007/978-3-319-12547-3_9},
	url = {hdl.handle.net/20.500.12613/5844},
	zbMATH = {6499121},
	Zbl = {1329.35138}
}

@book{Biquard_metriques,
	author = {Biquard, Olivier},
	title = {M{\'e}triques d'{Einstein} asymptotiquement sym{\'e}triques},
	fseries = {Ast{\'e}risque},
	series = {Ast{\'e}risque},
	issn = {0303-1179},
	volume = {265},
	year = {2000},
	publisher = {Paris: Soci{\'e}t{\'e} Math{\'e}matique de France},
	language = {French},
	zbMATH = {1450829},
	Zbl = {0967.53030}
}

@article{GrahamLee,
	author = {Graham, C. Robin and Lee, John M.},
	title = {Einstein metrics with prescribed conformal infinity on the ball},
	fjournal = {Advances in Mathematics},
	journal = {Adv. Math.},
	issn = {0001-8708},
	volume = {87},
	number = {2},
	pages = {186--225},
	year = {1991},
	language = {English},
	doi = {10.1016/0001-8708(91)90071-E},
	zbMATH = {14470},
	Zbl = {0765.53034}
}

@book{LeeFredholm,
	author = {Lee, John M.},
	title = {Fredholm operators and {Einstein} metrics on conformally compact manifolds},
	fseries = {Memoirs of the American Mathematical Society},
	series = {Mem. Am. Math. Soc.},
	issn = {0065-9266},
	volume = {864},
	isbn = {978-0-8218-3915-7; 978-1-4704-0468-0},
	year = {2006},
	publisher = {Providence, RI: American Mathematical Society (AMS)},
	language = {English},
	doi = {10.1090/memo/0864},
	zbMATH = {5063217},
	Zbl = {1112.53002}
}

@book{FeffermanGraham,
	author = {Fefferman, Charles and Graham, C. Robin},
	title = {The ambient metric.},
	fseries = {Annals of Mathematics Studies},
	series = {Ann. Math. Stud.},
	volume = {178},
	isbn = {978-0-691-15314-8; 978-0-691-15313-1; 978-1-400-84058-8},
	year = {2012},
	publisher = {Princeton, NJ: Princeton University Press},
	language = {English},
	doi = {10.1515/9781400840588},
	zbMATH = {5987929},
	Zbl = {1243.53004}
}

@book{BiquardAdS,
	editor = {Biquard, Olivier},
	title = {{AdS}/{CFT} correspondence: {Einstein} metrics and their conformal boundaries. 73rd meeting of the theoretical physicists and mathematicians, {Strasbourg}, {France}, {September} 11--13, 2003},
	fseries = {IRMA Lectures in Mathematics and Theoretical Physics},
	series = {IRMA Lect. Math. Theor. Phys.},
	issn = {2523-5133},
	volume = {8},
	isbn = {3-03719-013-2},
	year = {2005},
	publisher = {Z{\"u}rich: European Mathematical Society (EMS)},
	language = {English},
	doi = {10.4171/013},
	zbMATH = {2172481},
	Zbl = {1062.81002}
}

@Book{Treves,
	Author = {Tr{\`e}ves, Fran{\c{c}}ois},
	Title = {Topological vector spaces, distributions and kernels},
	FSeries = {Pure and Applied Mathematics (Academic Press)},
	Series = {Pure Appl. Math., Academic Press},
	ISSN = {0079-8169},
	Volume = {25},
	Year = {1967},
	Publisher = {New York-London: Academic Press},
	Language = {English},
	zbMATH = {3272562},
	Zbl = {0171.10402}
}
\bibliographystyle{alpha}
\end{document}